\PassOptionsToPackage{table}{xcolor} 
\documentclass[11pt]{article} 
\usepackage[utf8]{inputenc} 
\usepackage[margin=1in]{geometry} 
\usepackage{graphicx} 

\usepackage{booktabs} 
\usepackage{array} 
\usepackage{paralist} 
\usepackage{verbatim} 
\usepackage{mathrsfs}
\usepackage{amssymb}
\usepackage{amsthm}
\usepackage{amsmath,amsfonts,amssymb}
\usepackage{esint}
\usepackage{graphics}
\usepackage{enumerate}
\usepackage{mathtools}
\usepackage{xfrac}
\usepackage{nicefrac}
\usepackage{subcaption}
\usepackage[normalem]{ulem}
\usepackage{cancel}
\usepackage{enumitem}
\usepackage{bm}
\usepackage{xparse}
\usepackage{stmaryrd}

\let\olddiamond\diamond
\let\oldsquare\square 

\usepackage{mathabx}

\renewcommand{\square}{\oldsquare}
\renewcommand{\diamond}{\olddiamond}

\usepackage[dvipsnames]{xcolor}
\usepackage[colorlinks=true, pdfstartview=FitV, linkcolor=blue, citecolor=blue, urlcolor=blue]{hyperref}
\usepackage[normalem]{ulem}

\usepackage{tikz, makecell,forest}
\usetikzlibrary{shapes,arrows,positioning,calc}

\numberwithin{equation}{section}
\numberwithin{figure}{section}

\newtheorem{theorem}{Theorem}[section]

\newtheorem{assumption}[theorem]{Assumption}
\newtheorem{corollary}[theorem]{Corollary}
\newtheorem{proposition}[theorem]{Proposition}
\newtheorem{lemma}[theorem]{Lemma}
\newtheorem{theoremA}{Theorem}

\newtheorem{conjectureA}[theoremA]{Conjecture}

\theoremstyle{definition}

\newtheorem{remark}[theorem]{Remark}

\newcommand*{\supp}{\ensuremath{\mathrm{supp\,}}}
\newcommand*{\Id}{\ensuremath{\mathrm{I}_d}}

\newcommand*{\Itwod}{\ensuremath{\mathrm{I}_{2d}}}

\newcommand*{\N}{\ensuremath{\mathbb{N}}}

\newcommand*{\Z}{\ensuremath{\mathbb{Z}}}

\newcommand*{\R}{\ensuremath{\mathbb{R}}}
\newcommand*{\Zd}{\ensuremath{\mathbb{Z}^d}}
\newcommand*{\Rd}{\ensuremath{\mathbb{R}^d}}

\newcommand{\eps}{\varepsilon}
\renewcommand*{\tilde}{\widetilde}
\renewcommand{\P}{\ensuremath{\mathbb{P}}}
\renewcommand{\O}{\ensuremath{\mathcal{O}}}
\newcommand{\X}{\ensuremath{\mathcal{X}}}

\renewcommand{\b}{\ensuremath{\mathbf{b}}}

\newcommand{\qand}{\quad \mbox{and} \quad }
\newcommand{\sym}{\mathrm{sym}}

\newcommand{\f}{\mathbf{f}}

\newcommand{\g}{\mathbf{g}}
\newcommand{\h}{\mathbf{h}}
\newcommand{\hshell}{\mathbf{h}}
\newcommand{\s}{\mathbf{s}}
\newcommand{\pot}{\mathrm{pot}}
\newcommand{\sol}{\mathrm{sol}}
\NewDocumentCommand{\bfs}{e{^_}}{{\boldsymbol{\sigma}}\IfValueT{#1}{^{#1}}\IfValueT{#2}{_{\!#2}}}

\newcommand{\ep}{\eps}

\newcommand{\A}{\mathcal{A}}

\renewcommand{\S}{\mathcal{S}}
\newcommand{\G}{\mathbf{G}}
\DeclareMathOperator{\dist}{dist}

\DeclareMathOperator*{\var}{var}
\DeclareMathOperator*{\cov}{cov}

\newcommand{\E}{\mathbb{E}}

\newcommand{\nf}{\nicefrac}

\DeclareSymbolFont{boldoperators}{OT1}{cmr}{bx}{n}
\SetSymbolFont{boldoperators}{bold}{OT1}{cmr}{bx}{n}
\usepackage{accents}
\newcommand\thickbar[1]{\accentset{\rule{.45em}{.6pt}}{#1}}
\renewcommand{\bar}{\thickbar}
\renewcommand{\a}{\mathbf{a}}
\renewcommand{\k}{\mathbf{k}}
\newcommand{\kfield}{\mathbf{k}}
\newcommand{\kcg}{\mathsf{k}}

\newcommand{\ahom}{\bar{\a}}

\newcommand{\shom}{\bar{\mathbf{s}}}
\newcommand{\khom}{\bar{\mathbf{k}}}

\newcommand{\pert}{\boldsymbol{\delta}}

\newcommand{\Lop}{\mathfrak{L}}
\NewDocumentCommand{\auxmat}{e{^_}}{{\boldsymbol{d}}\IfValueT{#1}{^{\raisebox{0.8pt}{$\;\!\scriptstyle{#1}$}}}\IfValueT{#2}{_{#2}}}

\newcommand{\expon}{\rho}
\newcommand{\gammafun}{\mathbf{\Gamma}}

\newcommand{\Lsolo}{L^2_{\mathrm{sol,0}}}

\newcommand{\Lpoto}{L^2_{\mathrm{pot,0}}}

\newcommand{\bfA}{\mathbf{A}}
\newcommand{\bfAhom}{\overline{\mathbf{A}}}
\newcommand{\bfB}{\mathbf{B}}

\newcommand{\bfJ}{\mathbf{J}}

\newcommand{\bfE}{\mathbf{E}}
\newcommand{\uhom}{u_{\mathrm{hom}}}
\newcommand{\vhom}{v_{\mathrm{hom}}}

\newcommand{\linear}{\boldsymbol{\ell}}
\newcommand{\cstar}{c_*}
\newcommand{\nondegconst}{\breve{A}}
\newcommand{\CE}{\mathcal{E}^\dagger}
\newcommand{\FE}{\mathcal{E}^\star}

\makeatletter
\newcommand{\negphantom}{\v@true\h@true\negph@nt}
\newcommand{\neghphantom}{\v@false\h@true\negph@nt}
\newcommand{\negph@nt}{\ifmmode\expandafter\mathpalette
  \expandafter\mathnegph@nt\else\expandafter\makenegph@nt\fi}
\newcommand{\makenegph@nt}[1]{%
  \setbox\z@\hbox{\color@begingroup#1\color@endgroup}\finnegph@nt}
\newcommand{\finnegph@nt}{%
  \setbox\tw@\null
  \ifv@ \ht\tw@\ht\z@\dp\tw@\dp\z@\fi \ifh@\wd\tw@-\wd\z@\fi\box\tw@}
\newcommand{\mathnegph@nt}[2]{%
  \setbox\z@\hbox{$\m@th #1{#2}$}\finnegph@nt}
\makeatother

\newcommand{\Hminushat}[1]{\hat{\phantom{H}}\negphantom{H}{H}^{#1}}

\newcommand{\Hminusul}{\hat{\phantom{H}}\negphantom{H}\underline{H}^{-1}}
\newcommand{\Hminusulh}{\hat{\phantom{H}}\negphantom{H}\underline{H}^{-\nf12}}

\newcommand{\Wminusul}[2]{\hat{\phantom{W}}\negphantom{W}\underline{W}^{#1,#2}}

\newcommand{\css}[1]{\mathfrak{c}_{#1}}

\def\Xint#1{\mathchoice
{\XXint\displaystyle\textstyle{#1}}%
{\XXint\textstyle\scriptstyle{#1}}%
{\XXint\scriptstyle\scriptscriptstyle{#1}}%
{\XXint\scriptscriptstyle\scriptscriptstyle{#1}}%
\!\int}
\def\XXint#1#2#3{{\setbox0=\hbox{$#1{#2#3}{\int}$}
\vcenter{\hbox{$#2#3$}}\kern-.5\wd0}}

\def\fint{\Xint-}

\newcommand{\avsum}{\mathop{\mathpalette\avsuminner\relax}\displaylimits}

\makeatletter
\newcommand\avsuminner[2]{%
  {\sbox0{$\m@th#1\sum$}%
   \vphantom{\usebox0}%
   \ooalign{%
     \hidewidth
     \smash{\,\rule[.23em]{8.8pt}{1.1pt} \relax}%
     \hidewidth\cr
     $\m@th#1\sum$\cr
   }%
  }%
}
\makeatother

\makeatletter
\newcommand\avsuminnerr[2]{%
  {\sbox0{$\m@th#1\sum$}%
   \vphantom{\usebox0}%
   \ooalign{%
     \hidewidth
     \smash{\,\rule[.23em]{6pt}{0.7pt} \relax}%
     \hidewidth\cr
     $\m@th#1\sum$\cr
   }%
  }%
}
\makeatother

\let\originalleft\left
\let\originalright\right
\renewcommand{\left}{\mathopen{}\mathclose\bgroup\originalleft}
\renewcommand{\right}{\aftergroup\egroup\originalright}

\newcommand{\cu}{\square}

\newcommand{\indc}{{\boldsymbol{1}}}
\renewcommand{\hat}{\widehat}

\usepackage{titlesec}

\newcommand{\addperiod}[1]{#1.}
\titleformat{\section}
   {\centering\normalfont\Large}{\thesection.}{0.5em}{}
\titleformat{\subsection}[runin]
  {\normalfont\bfseries}
  {\thesubsection.}
  {0.5em}
  {\addperiod}
\titleformat{\subsubsection}[runin]
  {\normalfont\bfseries}
  {\thesubsubsection.}
  {0.5em}
  {\addperiod}
\titleformat*{\subsubsection}{\normalfont\itshape}
\titleformat*{\paragraph}{\bfseries}
\titleformat*{\subparagraph}{\large\bfseries}

\title{Superdiffusive central limit theorem for a Brownian particle in a critically-correlated incompressible random drift}

\author{Scott Armstrong
\thanks{Courant Institute of Mathematical Sciences, New York University.{\footnotesize \href{mailto:scotta@cims.nyu.edu}{scotta@cims.nyu.edu}.}
}
\and
Ahmed Bou-Rabee
\thanks{Courant Institute of Mathematical Sciences, New York University.{\footnotesize \href{mailto:ahmedmb@cims.nyu.edu}{ahmedmb@cims.nyu.edu}.}
}
\and
Tuomo Kuusi
\thanks{Department of Mathematics and Statistics, University of Helsinki.{\footnotesize \href{mailto:tuomo.kuusi@helsinki.fi}{tuomo.kuusi@helsinki.fi}.}
}
}

\date{\today }

\usepackage[nottoc,notlot,notlof]{tocbibind}

\begin{document}

\maketitle

\begin{abstract}
We consider the long-time behavior of a diffusion process on~$\Rd$ advected by a stationary random vector field which is assumed to be divergence-free, dihedrally symmetric in law and have a log-correlated potential. A special case includes~$\nabla^\perp$ of the Gaussian free field in two dimensions. We show the variance and the second moment of the displacement at a large time~$t$ behave like~$2d\cstar^{\nf12} t (\log t)^{\nicefrac 12}$, in a quenched sense and with a precisely determined prefactor~$\cstar=\cstar(\P)>0$, independent of~$\nu$. We also prove a quenched invariance principle under this superdiffusive scaling. The proof is based on a rigorous renormalization group argument in which we inductively analyze coarse-grained diffusivities, scale-by-scale. Our analysis leads to sharp homogenization and large-scale regularity estimates on the infinitesimal generator, which are subsequently transferred into quantitative information on the process.
\end{abstract}

\setcounter{tocdepth}{1}
\renewcommand{\baselinestretch}{0.75}\normalsize
\tableofcontents
\renewcommand{\baselinestretch}{1.0}\normalsize

\section{Introduction}
\subsection{Superdiffusive central limit theorem}
We consider the long-time behavior of a Brownian particle advected by a random, divergence-free velocity field in~$\Rd$. This is described by the stochastic differential equation
\begin{equation}
	\label{e.sde.intro}
	\left\{
	\begin{aligned}
		& dX_t = \f (X_t) \,dt +  \sqrt{2\nu} dW_t  \,, \\
		& X_0 = x_0 \in\Rd \,,
	\end{aligned}
	\right.
\end{equation}
where~$\nu\in (0,1]$\footnote{The assumption that~$\nu \leq 1$ is made for convenience and without loss of generality. Indeed, if~$\nu >1$, we may rescale the equation and apply the result.} is a given positive parameter called the \emph{molecular diffusivity},~$\{ W_t \}$ is a standard Brownian motion on~$\Rd$ and the vector field~$\f:\Rd \to \Rd$ is a stationary random field with law~$\P$. We assume that the vector field~$\f$ is locally Lipschitz, dihedrally symmetric in law, and satisfies the incompressibility condition
\begin{equation}
	\nabla \cdot \f = 0 \quad \mbox{in} \ \Rd \,.
\end{equation}
It is assumed to behave like a Gaussian field with Hurst parameter~$-1$; roughly,
\begin{equation}
	\label{e.cov.critical}
	\cov\bigl[ \f(x) , \f(y) \bigr] \simeq |x-y|^{-2}\,, \quad \mbox{for} \ |x-y| \gg 1\,.
\end{equation}
See Section~\ref{ss.intro.statements} below for the precise assumptions on~$\f$.

\smallskip

A special case of this setup is in two dimensions with~$\f = \nabla^\perp (H \ast \eta)$, where~$H$ is a standard Gaussian free field (GFF) and~$\eta\in C^\infty_c(B_1)$ is a fixed nonnegative radial mollifier satisfying~$\int_{\R^2}\eta=1$. This and closely related models were studied heuristically beginning with~\cite{FNS}, followed by~\cite{FFQSS, KLY, Fisher, AN,BCGLD,Fann}, and rigorously in the more recent works~\cite{TV,CHT,CMOW}.

\smallskip

It is well-known that advection by an incompressible vector field enhances the diffusivity of a Brownian particle. If the covariance (and thus the power spectrum) of~$\f$ decays very quickly, then the vector field will increase the effective diffusivity from~$\nu$ to some constant~$\nu_{\mathrm{eff}} > \nu$, and the displacement variance~$\mathbf{E}^0[ |X_t - \mathbf{E}^0[X_t]|^2 ]$ will grow at the linear rate~$2d\nu_{\mathrm{eff}} t + o(t)$ for large times~$t$. Here we denote by~$\mathbf{P}^{x_0}$ the law of the process~\eqref{e.sde.intro} conditioned on~$\f$, and by~$\mathbf{E}^{x_0}$ its corresponding expectation. Conversely, if the covariances of~$\f$ have slow decay, then the low frequencies of~$\f$ have larger amplitudes and the diffusivity enhancements may accumulate, as the particle sees more and more of these larger length scales. In this situation we expect \emph{superdiffusivity}, that is,~$\mathbf{E}^0[ |X_t - \mathbf{E}^0[X_t]|^2 ]$ should grow \emph{superlinearly} in~$t$ for~$t\gg1$, due to the effect of diffusivity enhancements across an infinite range of length scales.

\begin{figure}
	\begin{center}
		\includegraphics[width=0.8\textwidth]{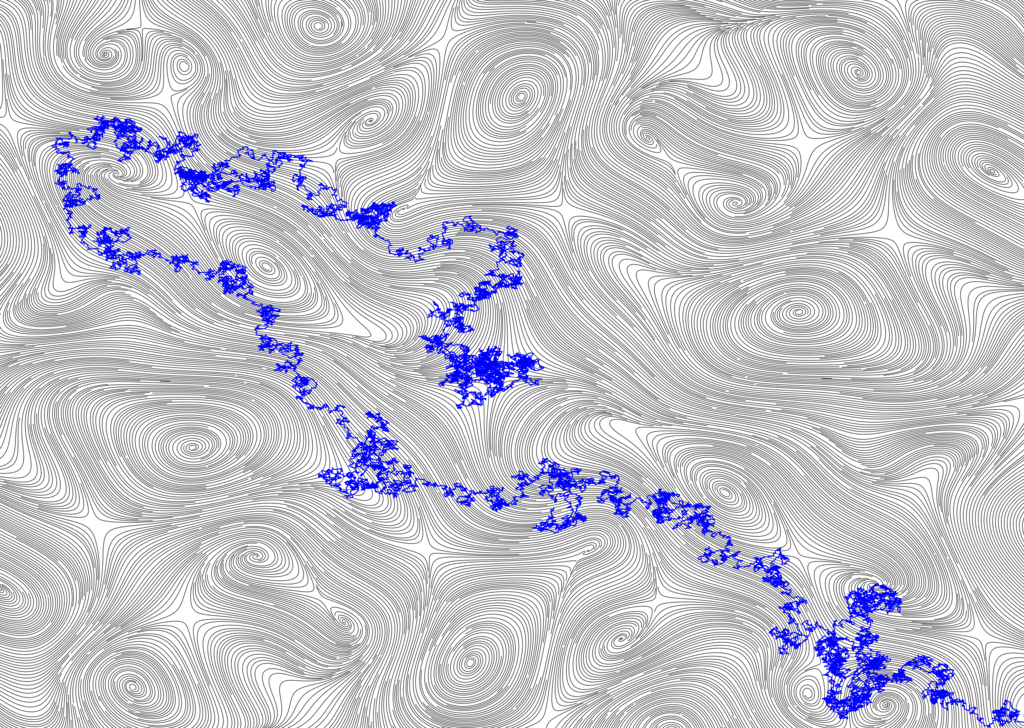}
	\end{center}
	\caption{A Brownian motion  subject to a random incompressible drift.}
\end{figure}

\smallskip

The decay rate in~\eqref{e.cov.critical} is critical: it is situated exactly on the boundary between the diffusive regime (in which the right side of~\eqref{e.cov.critical} is replaced by~$|x-y|^{-\xi}$ for~$\xi>2$) and the expected superdiffusive regime ($\xi<2$). Predictions in~\cite{FNS,FFQSS, KLY, Fisher, AN,BCGLD,BG,Fann}, which were based on heuristic renormalization group arguments, are that we should see logarithmic-type superdiffusivity in the critical case of~\eqref{e.cov.critical}: more precisely,~$\mathbf{E}^0[ |X_t - \mathbf{E}^0[X_t]|^2 ]$ is expected to grow at the rate of~$t (\log t)^{\nf12}$ as~$t\to \infty$.\footnote{The dihedral symmetry assumption is crucial for this result; heuristic arguments give~$t (\log t)^{\nicefrac23}$ for certain non-isotropic models, see~\cite{TV}.}

\smallskip

In this article we formalize these heuristic renormalization group arguments using methods from elliptic homogenization. In particular, we are able to sharply characterize the superdiffusivity of the process~$\{ X_t\}$. One of our main results is a \emph{superdiffusive central limit theorem}: it states that, for large times~$t$, the process~$\{ X_t\}$ resembles a Brownian motion with diffusivity equal to~$\cstar^{\nf12}\sqrt{\log t}$ for a deterministic constant~$\cstar=\cstar(\P)>0$, independent of~$\nu$.

\begin{theoremA}[Quenched superdiffusive invariance principle]
\label{t.A}
Let~$d\geq2$ and~$\nu\in(0,1]$, and suppose that Assumption~\ref{a.multiscale.stream} holds with constants~$\cstar,\nondegconst\in(0,\infty)$. Let~$X$ be the diffusion in~\eqref{e.sde.intro} with any fixed initial point. Then, for~$\P$--a.e.~realization of the vector field~$\f$,
\begin{equation}
	\label{e.convinlaw}
	\left| \log \ep^2 \right| ^{-\nicefrac14} \ep X_{\nicefrac{t}{\ep^2}}
	\Rightarrow
	\sqrt{2\cstar^{\nf12}} W_t
	\quad
	\mbox{as} \ \ep \to 0\,,
\end{equation}
where~$\{ W_t \}_{t\geq0}$ is a standard Brownian motion on~$\Rd$ and the convergence in~\eqref{e.convinlaw} is in law with respect to the topology of locally uniform convergence on~$C([0,\infty);\Rd)$. Moreover, if~$X_0=0$, then, for each~$\delta\in (0,\nicefrac14)$ and~$\beta \in (0,4\delta)$, there exists a constant~$C(\beta,\delta,\cstar,\nu,\nondegconst,d)<\infty$ such that, for every~$t \in [10,\infty)$,
\begin{equation}
	\label{e.Dt.exp}
	\P \biggl[
	\Bigl| \frac1t
	\mathbf{E}^0 \bigl[ \bigl| X_t \bigr|^2 \bigr] -
	2d\cstar^{\nf12} (\log t)^{\nf12}
	\Bigr|
	+
	\frac1t
	\bigl| \mathbf{E}^0 \bigl[ X_t \bigr] \bigr|^2
	>
	C \bigl( \log t \bigr)^{\nf14+\delta}
	\biggr]
	\leq
	C \exp \bigl( - C^{-1} (\log t)^{\beta} \bigr)
\end{equation}
and, for every exponent~$p\in[1,\infty)$, allowing~$C$ to depend also on~$p$, we have
\begin{equation}
	\label{e.annealed.Dt}
	\E \biggl[ \Bigl| \frac{1}{t} \bigl[ \mathbf{E}^0[ |X_t|^2 ]  \bigr] - 2d\cstar^{\nf12}  (\log t)^{\nicefrac 12} \Bigr|^p
	+
	\Bigl( \frac1t
	\bigl| \mathbf{E}^0 \bigl[ X_t \bigr] \bigr|^2 \Bigr)^{p}
	\biggr]^{\nicefrac1p}
	\leq
	C (\log t)^{\nicefrac 14+\delta}
	\,.
\end{equation}
\end{theoremA}

The constant~$\cstar$ in Theorem~\ref{t.A} is semi-explicit and can be computed in many cases of interest, as we will show. For instance, in the two-dimensional special case mentioned above, in which~$\f$ is~$\nabla^{\perp}$ of a mollified~GFF, we have that~$\cstar = \frac{1}{4\pi}$.

\smallskip

A natural generalization of this two-dimensional example to~$d>2$, to which the conclusions of Theorem~\ref{t.A} apply after the deterministic normalization and time rescaling described in Appendix~\ref{s.LGF}, is to consider~$\f = \nabla \cdot (\mathbf{H} \ast \eta)$, where~$\eta\in C^\infty_c(B_1)$ is a fixed nonnegative radial mollifier satisfying~$\int_{\Rd}\eta=1$ and~$\mathbf{H}$ is an anti-symmetric matrix-valued random field,\footnote{We use the convention that the divergence of a matrix-valued function~$A=(A_{ij})$ is the vector field with~$j$th coordinate given by~$(\nabla \cdot A)_j = \sum_{i=1}^d \partial_{x_i} A_{ij}$.} the entries of which are obtained from independent copies of the standard log-correlated Gaussian field (LGF) on~$\Rd$. In this case, the constant~$\cstar$ is given by the formula
\begin{equation}
	\label{e.c.star.for.d}
	\cstar=
	\frac{(d-1) 2^{1-d}}{d\,\pi^{\nicefrac d2}\gammafun(\nicefrac d2) }
	\,.
\end{equation}
Note that this example coincides with~$\nabla^{\perp}$ of the GFF in~$d=2$, since in two dimensions the LGF and the GFF coincide and anti-symmetric matrices have only one degree of freedom and therefore may be identified with scalars. In three dimensions, the divergence of an anti-symmetric matrix can be identified with~$\nabla\times$ of a vector field, and the~$\f$ above can be written in~$d=3$ as~$\f= \nabla \times ((H_1,H_2,H_3)\ast\eta)$, where the~$H_i$ are independent realizations of the~$d=3$~LGF. The higher-dimensional case was analyzed heuristically in~\cite{BCGLD,BCGLD.erratum}, where the prefactor in~\eqref{e.c.star.for.d} was predicted; the erratum does not alter this divergence-free prediction.

\smallskip

Prior to this paper, several authors~\cite{TV,CHT,CMOW} have pursued rigorous results in the direction of Theorem~\ref{t.A} in the special case in which~$\f$ is~$\nabla^\perp$ of a mollified~GFF in two dimensions. The first rigorous result demonstrating (qualitative) superdiffusivity was obtained in~\cite{TV}, as well as (a time-averaged version of) the annealed estimate
\begin{equation}
	\label{e.Valko}
	C^{-1} \log \log t \leq \frac{1}{t}\E \bigl[ \mathbf{E}^0[ |X_t|^2 ]  \bigr] \leq C \log t\,.
\end{equation}
This bound was significantly improved in~\cite{CHT}, who obtained the predicted growth rate of~$t ( \log t)^{\nf12}$ up to double logarithmic corrections: precisely, they showed that, for every~$\ep \in (0,1]$,
\begin{equation}
	\label{e.Toninelli}
	C_\ep^{-1} \frac{( \log t)^{\nf12}}{(\log \log t)^{1+\ep}} \leq \frac{1}{t}\E \bigl[ \mathbf{E}^0[ |X_t|^2 ]  \bigr] \leq C_\ep ( \log t)^{\nf12} ( \log\log t)^{1+\ep} \,.
\end{equation}
The double logarithmic corrections were removed in~\cite{CMOW}, who obtained the estimate, for general~$\nu>0$,
\begin{equation}
	\label{e.Otto}
	C_\nu^{-1}
	\leq
	\liminf_{t \to \infty}
	\frac{\E \bigl[ \mathbf{E}^0[ |X_t|^2 ]  \bigr]}{t( \log t)^{\nf12} }
	\leq
	\limsup_{t\to \infty}
	\frac{\E \bigl[ \mathbf{E}^0[ |X_t|^2 ]  \bigr] }{t( \log t)^{\nf12} }
	\leq C_\nu  \,.
\end{equation}
Annealed results in the superdiffusive regime in which the covariances of the field decay slower than exponent~$2$ were obtained in~\cite{KO}. We refer to~\cite{Toth} for further references and more on the history of the problem.

\smallskip

Theorem~\ref{t.A} improves on these works by obtaining the precise rate of superdiffusivity with nearly the next-order correction, by establishing superdiffusivity in the sense of the quenched displacement \emph{variance} (as opposed to second moment), by proving that the diffusion process~$\{X_t\}$ has a scaling limit which is a Brownian motion, and by treating all dimensions~$d\geq 2$. Its statement also asserts that the asymptotic rate of superdiffusivity is independent of the molecular diffusivity~$\nu$, a phenomenon characteristic of anomalous diffusion.

\smallskip

Furthermore, the estimates presented in Theorem~\ref{t.A} and throughout this paper are the first \emph{quenched} statements about the behavior of the process. In fact, to our knowledge, Theorem~\ref{t.A} is the first result establishing even the qualitative quenched superlinear growth of the second moment, 
\begin{equation*}
	\P \biggl[ \limsup_{t \to \infty} \frac{1}{t}\mathbf{E}^0[|X_t|^2] = +\infty\biggr]  > 0\,.
\end{equation*}
In contrast, the results in~\cite{TV,CHT,CMOW} are \emph{annealed}, which means that they are valid only after averaging over all realizations of the vector field~$\f$. We emphasize that there is a substantial increase in difficulty when proving a quenched estimate rather than an annealed one. In the problem we are considering, annealed estimates can often be obtained by relatively soft arguments, using ergodicity and stationarity. It appears, however, that quenched information requires a quantitative, scale-by-scale approach in which the random fluctuations of all quantities are carefully controlled.\footnote{The first arXiv version of~\cite{CMOW} contains quenched estimates for infrared cutoffs of the vector field. However, these estimates are not shown to be valid on scales below that of the cutoff, and as such do not provide quenched information about the original process.} The payoff is that a quenched analysis yields more precise estimates, even for annealed quantities.

\smallskip

It is certainly unusual to obtain convergence to a limit process under a  scaling which is different from the one which leaves the limiting process invariant. Notably, a superdiffusive invariance principle, with the same~$\sqrt{\log}$-scaling as the one encountered here, has been previously obtained for the periodic Lorentz gas~\cite{MT,SV}.

\smallskip

The large-time behavior of the process~$X_t$ solving the SDE~\eqref{e.sde.intro} has also been studied in the case~$\f=\nabla H$, and the ``unconstrained'' case in which~$\f$ has both potential and solenoidal parts. In the potential case, the critical scaling for the covariance is the same as in~\eqref{e.cov.critical}---so~$\f$ is essentially the gradient of the LGF mollified on the unit scale. This model was analyzed heuristically in~\cite{FFQSS,KLY,Fisher, BCGLD}, with the prediction of subdiffusive behavior due to trapping:~$t^{-1}\E \bigl[ \mathbf{E}^0[ |X_t|^2 ]  \bigr] \simeq t^{-\theta_{\rm trap}}$ for some~$\theta_{\rm trap}>0$. An earlier rigorous subdiffusive result was proved in~\cite{BDG}, where it was shown that the exit time from a large ball of a two-dimensional discrete version of this process has a subdiffusive scaling. For the unconstrained model, in which~$\f$ is a vector white noise mollified on the unit scale, the prediction in~\cite{FFQSS,KLY,Fisher,BCGLD} depends on dimension: in~$d>2$, one expects a diffusive scaling, as well as in~$d=2$ for weak disorder (when~$\f$ is also small). These predictions were rigorously confirmed in the case~$d>2$ for weak disorder in~\cite{BK,SZ}.

\smallskip

In the following two subsections, we introduce the precise assumptions and state our main results. In Section~\ref{ss.heuristic}, we present an informal, heuristic argument for the superdiffusive scaling observed in Theorem~\ref{t.A} and, in Section~\ref{ss.proofoutline}, we give a detailed outline of the rigorous arguments we use in this paper to formalize these heuristics.

\subsection{Quantitative homogenization}
\label{ss.intro.statements}

The diffusion process described by~\eqref{e.sde.intro} is \emph{Feller}, as we show in Appendix~\ref{app.Feller}. Feller processes are uniquely characterized by their infinitesimal generators (see~\cite[Lemma 19.5]{Kallenberg}), and therefore any statement concerning a family of Feller processes can \emph{in principle} be translated into an equivalent statement about their generators. For example, a sequence of Feller processes converges to a limit if and only if their corresponding generators converge: see~\cite[Theorem 19.25]{Kallenberg} for the precise statement. This allows us to rephrase the invariance principle asserted in Theorem~\ref{t.A} as a statement about  \emph{homogenization} of elliptic partial differential operators.

\smallskip

The infinitesimal generator of the diffusion process described by~\eqref{e.sde.intro} is the elliptic operator~$\Lop$ given by
\begin{equation}
	\label{e.L.gen}
	\Lop u = \nu \Delta u + \mathbf{f} \cdot \nabla u \,.
\end{equation}
We start by writing~$\Lop$ as a purely second-order operator in divergence form,
\begin{equation*}
	\Lop u = \nabla \cdot \bigl( \nu \Id + \k \bigr) \nabla u\,,
\end{equation*}
where~$\k$ is the matrix potential or \emph{stream matrix} for~$\f$. That is,~$\k$ is an anti-symmetric~$d$-by-$d$ matrix satisfying~$\nabla \cdot \k = \f$. In our setup, the matrix~$\k$ is a random field which is not stationary: like the log-correlated Gaussian field, only its gradient is stationary. We define the matrix-valued random field~$\a(x)$ by
\begin{equation}
	\label{e.a.superdiff}
	\a(x) = \nu \Id + \k(x)\,,
\end{equation}
so that we can simply write~$\Lop  = \nabla \cdot \a \nabla$. The first statement in Theorem~\ref{t.A} is a consequence of a stronger, more quantitative statement about the large-scale behavior of~$\Lop$, which is stated in Theorem~\ref{t.superdiffusivity} below.

\smallskip

If the stream matrix~$\k$ were a stationary function with finite second moment, then classical homogenization theory would immediately imply that~$\Lop$ homogenizes in the large scale limit to a deterministic and constant-coefficient elliptic operator of the form~$\overline{\Lop } = \nabla \cdot \ahom \nabla$. We would consequently obtain an invariance principle for the process~$\{ X_t \}$ with diffusivity given by (twice) the symmetric part of~$\ahom$ and, in particular, we would observe the usual diffusive scaling~$\mathbf{E}^0[ |X_t|^2 ]\simeq t$. It is because the stream matrix~$\k$ is \emph{not} bounded that superdiffusivity is possible; however, this unboundedness brings the problem outside the realm of standard homogenization theory.

\smallskip

Nevertheless, in this article we demonstrate that this superdiffusivity can be analyzed using \emph{iterative quantitative homogenization}. We prove that, on length scales of order~$3^m$, the operator~$\Lop$ can be approximated (in the sense of homogenization, with estimates that are quenched) by the operator~$\overline{\Lop}_m = \shom_m \Delta$, where the \emph{renormalized diffusivity}~$\shom_m>0$ is close to~$( 2 \cstar \log 3^m)^{\nf12}$ for a special constant~$\cstar>0$.

\smallskip

To prove this, we must directly confront the fact that the oscillation of the stream matrix diverges as a function of the length scale. This necessitates flexible analytic arguments which do not degenerate at each step: the conclusions we obtain at one scale must be just as strong as our assumption at the previous, smaller scale. This requires ideas and results from the quantitative homogenization theory developed for high contrast coefficient fields in our companion paper~\cite{AK.HC}.

\smallskip

We write~$\cu_m$ for the axis-aligned cube centered at the origin with side length~$3^m$,
\begin{equation*}
	\cu_m \coloneqq  \Bigl( -\frac123^m , \frac123^m\Bigr)^d\,.
\end{equation*}
We next introduce the general assumptions we work with throughout the paper.
\begin{assumption}[Multiscale stream matrix]
\label{a.multiscale.stream}
We take the stream matrix~$\k$, not the vector field~$\f$, as the given random field. We assume that it is given by the formal sum
\begin{equation}
	\label{e.k.sum.def}
	\k(x) = \sum_{n=0}^\infty
	\mathbf{j}_n (x)\,,
\end{equation}
where~$\{ \mathbf{j}_0,\mathbf{j}_1,\mathbf{j}_2, \ldots,\}$ is a sequence of~$\Rd$--stationary random fields\footnote{The~$\Rd$--stationary assumption can be replaced by~$\Zd$--stationary with minimal notational changes.}, valued in the set~$\R^{d\times d}_{\mathrm{skew}}$ of anti-symmetric~$d$-by-$d$ matrices with real entries, with law~$\P$ and corresponding expectation denoted by~$\E$, satisfying the following assumptions:

\begin{enumerate}[label=(\textbf{J\arabic*})]
\setcounter{enumi}{0}

\item
\label{a.j.frd} The field~$\mathbf{j}_n$ has range of dependence~$3^n\sqrt{d}$.

\item
\label{a.j.indy}
For each~$A,B \subseteq \N$ with~$A \cap B = \varnothing$, the subcollections~$\{ \mathbf{j}_n \, :\, n \in A \}$ and~$\{ \mathbf{j}_n \, :\, n \in B \}$ are independent.

\item
\label{a.j.reg}
Local regularity: with probability one,~$\mathbf{j}_n$ belongs to the space~$C^{1,1}_{\mathrm{loc}} (\Rd; \R^{d\times d}_{\mathrm{skew}})$ of anti-symmetric matrices with entries which belong to~$C^{1,1}(B_r)$ for every~$r>0$, and
\begin{equation}
	\label{e.k(n).reg}
	\P \Bigl[
	\| \mathbf{j}_n \|_{L^{\infty}(\cu_n)}
	+
	\sqrt{d} 3^n \| \nabla \mathbf{j}_n \|_{L^{\infty}(\cu_n)}
	+
	d3^{2n} \| \nabla^2 \mathbf{j}_n  \|_{L^{\infty}(\cu_n)}
	> t
	\bigr]
	\leq \exp(-t^2)
	\quad \forall t \in [1, \infty) \,.
\end{equation}

\item
\label{a.j.iso} Dihedral symmetry: the joint law of~$\{ \mathbf{j}_n\}_{n\in\N}$ is invariant under negation, reflections and permutations across the coordinate planes. That is, for every matrix~$R$ with exactly one~$\pm 1$ in each row and column and~$0$s elsewhere, the law of the conjugated sequence~$\{ R^t \mathbf{j}_n(R \cdot) R\}_{n\in\N}$ is the same as that of~$\{ \mathbf{j}_n \}_{n\in\N}$, and the law of the sequence~$\{ -\mathbf{j}_n\}_{n\in\N}$ is the same as that of~$\{ \mathbf{j}_n \}_{n\in\N}$.

\item
\label{a.j.nondeg}
Non-degeneracy: there exist constants~$\cstar,\nondegconst \in (0,\infty)$ such that, for every~$m,n \in \N$ with~$n<m$ and~$e\in \partial B_1$,
\begin{equation}
	\label{e.nondeg}
	\Biggl|
	\E \Biggl[
	\biggl| \nabla \Delta^{-1}
	\biggl(\nabla \cdot
	\sum_{l=n+1}^m
	\mathbf{j}_l  e
	\biggr) (0)
	\biggr|^2
	\Biggr]
	-
	\cstar(\log 3) (m-n)
	\Biggr|
	\leq \nondegconst \,.
\end{equation}
\end{enumerate}
\end{assumption}
The~$C^{1,1}$ assumption in~\ref{a.j.reg} is used for local classical regularity and generator-domain arguments; the variational estimates below are stable under smooth approximation.

\smallskip

Throughout, we denote the law of~$\{ \mathbf{j}_n \}_{n\in\N}$ by~$\P$ and the corresponding expectation by~$\E$.

\smallskip

Strictly speaking, the sum on the right side of~\eqref{e.k.sum.def} is divergent. To be more precise, we should define instead
\begin{equation}
	\label{e.good.k}
	\k(x) - \k(y) = \sum_{n=0}^\infty
	\bigl( \mathbf{j}_n (x) - \mathbf{j}_n (y) \bigr)  \,,
\end{equation}
which does converge and is stationary. Indeed, by~\ref{a.j.reg} and the Borel-Cantelli lemma,~$|\mathbf{j}_n (x) - \mathbf{j}_n(y)|$ is almost surely summable in~$n$. Similarly,~$\nabla \k$ is a well-defined stationary field.

\smallskip

Notice that, by~\eqref{e.nondeg},~\eqref{e.k(n).reg},~\ref{a.j.indy}, the negation symmetry in~\ref{a.j.iso}, and the~$L^2$ contraction of~$\nabla\Delta^{-1}\nabla\cdot$, for every~$n\in\N$ and~$e\in\partial B_1$,
\begin{align}
	\cstar\log 3
	&=
	\lim_{N\to\infty}\frac1N
	\sum_{l=n+1}^{n+N}
	\E\Bigl[
	\bigl|
	\nabla\Delta^{-1}\nabla\cdot(\mathbf j_l e)(0)
	\bigr|^2
	\Bigr]
	\notag \\
	&\leq
	\limsup_{N\to\infty}\frac1N
	\sum_{l=n+1}^{n+N}
	\E\bigl[|\mathbf j_l(0)e|^2\bigr]
	\leq
	1+e^{-1}
	\quad \implies \quad
	\cstar \leq 2
	\,.
	\label{e.cstar.bound}
\end{align}

\smallskip

We introduce the divergence-free vector field
\begin{equation}
	\label{e.ff}
	\f(x)  \coloneqq  (\nabla \cdot \k )(x)
	= \sum_{n=0}^\infty
	(\nabla \cdot \mathbf{j}_n )(x)
	\,,
\end{equation}
which is also unambiguously defined. Notice that the incompressibility condition~$\nabla \cdot \f=0$ is immediate from the assumed anti-symmetry of each~$\mathbf{j}_n$.

\smallskip

As mentioned above, the main example we have in mind is of an anti-symmetric matrix with entries obtained from independent copies of a log-correlated Gaussian field (LGF), mollified on the unit scale. In dimensions~$d\geq 2$, the LGF is the (generalized) Gaussian random field on~$\Rd$ characterized by
\begin{equation}
	\label{e.intro.cov}
	\cov\bigl[ (h,\phi_1) , (h, \phi_2) \bigr]
	=
	\int_{\R^d}\int_{\R^d}
	-\frac{d|B_1|}{(2\pi)^d}
	\log|x-y|
	\phi_1(x)
	\phi_2(y)\,dx\,dy
	\,,
\end{equation}
where~$\phi_1, \phi_2$ are compactly supported smooth functions with~$\int_{\Rd} \phi_1 = \int_{\Rd} \phi_2 = 0$. In the two-dimensional case, this coincides with the Gaussian free field. That the LGF acts only on mean-zero test functions is a reflection of the fact that it is defined only modulo additive constants. For the reader's convenience, Appendix~\ref{s.LGF} gives an explicit construction of the log-correlated Gaussian field and its finite-range decomposition. We also verify there that the conclusions of Theorem~\ref{t.A} apply after the deterministic normalization and time rescaling, with~$\cstar$ for the original field given by~\eqref{e.c.star.for.d}.\footnote{Conversely, the assumptions~\ref{a.j.frd} and~\ref{a.j.reg} imply that the correlations of~$\k$ are at most log-correlated, and the non-degeneracy condition~\ref{a.j.nondeg} can be seen as a weak, two-sided log-correlation bound. We work with these general assumptions to make it clear that we do not use Gaussianity, nor the precise covariance structure of the log-correlated Gaussian field. Furthermore, it will become clear from the proof below that the right side of~\eqref{e.k(n).reg} can be relaxed to allow for distributions with heavier tails, and the assumption~\eqref{e.nondeg} can be tailored to other situations with stronger or weaker correlations (resulting in faster or slower superdiffusivity). }

\smallskip

The limit in~\eqref{e.convinlaw} says that the rescaled process~$\{ X^\ep_t \}$, defined for each~$\ep > 0$ by
\begin{equation}
	\label{e.Xep.t.def}
	X^\ep_t  \coloneqq
	\ep X_{\frac{t}{\ep^2 ( 8 \cstar \left| \log \ep \right|  )^{1/2}}} \,,
\end{equation}
converges in law, as~$\ep\to 0$, to a Brownian motion. The infinitesimal generator of~\eqref{e.Xep.t.def} is the operator~$\Lop^\ep$ of~$X^\ep_t$ given by
\begin{equation}
	\label{e.A.ep}
	\Lop^\ep =
	\frac12
	\bigl( 2 \cstar \left| \log \ep \right|  \bigr)^{-\nf12}
	\nabla
	\cdot \bigl( \nu\Id + \k^\ep \bigr) \nabla
	\,, \quad \mbox{where} \quad
	\k^\ep(x)  \coloneqq  \k \bigl( \nicefrac x \ep \bigr)
	\,.
\end{equation}
Since the infinitesimal generator of Brownian motion is~$\frac12\Delta$, we should expect the limit in~\eqref{e.convinlaw} to be equivalent to an appropriate statement concerning the convergence of~$\Lop^\ep$ to~$\frac12\Delta$, in the limit as~$\ep\to 0$. The following is a quantitative version of such a statement with an essentially sharp bound on the error.

\begin{theoremA}[Quantitative homogenization]
\label{t.superdiffusivity}
Let~$\alpha,\beta \in (0,1]$ with~$\beta+2\alpha<1$ and~$U \subseteq \Rd$ be a smooth, bounded domain. There exists a random variable~$\mathcal{Z}$ and a constant~$C < \infty$, each depending on~$(U,\beta,\alpha, \cstar, \nu, \nondegconst, d)$, satisfying
\begin{equation}
	\label{e.Z.integrability}
	\P \bigl[ \mathcal{Z} \geq \xi \bigr]
	\leq
	C \exp \bigl( - C^{-1}( \log \xi)^{\beta} \bigr)\,, \quad \forall \xi \in [1,\infty)\,,
\end{equation}
such that, for every~$\ep\in (0,\nicefrac12]$ with~$\ep^{-1} \geq \mathcal{Z}$ and functions~$f \in L^{\infty}(U)$ and~$g \in W^{1,\infty}(U)$ if we let~$u^\ep, \uhom \in H^1(U)$ denote the solutions of the boundary value problems
\begin{equation}
	\label{e.BVPs}
	\left\{
	\begin{aligned}
		& -\bigl( 2 \cstar \left| \log \ep \right|  \bigr)^{-\nf12}
		\nabla
		\cdot \bigl( \nu\Id + \k^\ep \bigr) \nabla u^\ep =  f & \mbox{in} & \ U \,, \\
		& u^\ep = g & \mbox{on} & \ \partial U \,,
	\end{aligned}
	\right.
	\qand
	\left\{
	\begin{aligned}
		& - \Delta \uhom   = f & \mbox{in} & \ U \,, \\
		& \uhom = g & \mbox{on} & \ \partial U \,,
	\end{aligned}
	\right.
\end{equation}
then we have the estimate
\begin{multline}
	\label{e.homogenization.error}
	\bigl\| u^\ep - \uhom\bigr\|_{{L}^{\infty}(U)}
	+
	\bigl[
	\nabla u^\ep - \nabla \uhom
	\bigr]_{{H}^{-1}(U)}
	+
	\biggl[
	\frac{\nu \Id + \k^\ep-(\k^\ep)_{U}}{( 2 \cstar \left| \log \ep \right|  )^{\nf12}} \nabla u^\ep -  \nabla \uhom
	\biggr]_{{H}^{-1}(U)}
	\\
	\leq
	C\left| \log \ep \right| ^{-\alpha}
	\bigl( \| \nabla g \|_{{L}^\infty(U)}+
	\| f \|_{{L}^\infty(U)}
	\bigr)
	\,.
\end{multline}
\end{theoremA}

We will show in Section~\ref{s.invariance.principle} that Theorem~\ref{t.superdiffusivity} implies each of the assertions in Theorem~\ref{t.A}.

\smallskip

The estimates for the latter two quantities on the left side of~\eqref{e.homogenization.error} quantify the \emph{weak convergence} in~$L^2(U)$, as~$\ep \to 0$, of the gradients and fluxes of the solutions~$u^\ep$ to their homogenized limits.

\smallskip

The random variable~$\mathcal{Z}$ in Theorem~\ref{t.superdiffusivity} quantifies the lower bound on the scale separation ratio required for homogenization. Observe that the estimate~\eqref{e.Z.integrability} on~$\mathcal{Z}$ is very weak: it does not give any positive moments for~$\mathcal{Z}$, since~$\beta < 1$. This differs substantially from the typical case in quantitative homogenization theory, where one obtains very strong, exponential moments for minimal scales (see for instance~\cite{AK.Book}). The very weak estimate we observe here is not an artifact of our proof, but is an intrinsic feature of the model. In fact, we expect that the stochastic integrability in~\eqref{e.Z.integrability} is essentially optimal.

\smallskip

We also expect that the scaling of the error in~\eqref{e.homogenization.error}, namely~$\left| \log \ep \right| ^{-\frac12+}$, is nearly sharp. To see this, observe first that no convergence rate better than~$\left| \log \ep \right| ^{-1}$ is possible, due to scaling (just compare the solutions of the first problem in~\eqref{e.BVPs} with~$\ep$ and~$\nicefrac\ep2$). To see that~$\left| \log \ep \right| ^{-\frac12}$ should be sharp, observe that resampling the fields~$\mathbf{j}_{k}$ for~$k \geq | \log_3\ep|-10$ will perturb the flux or the gradient of the solution~$u^\ep$ by at least~$O(\left| \log \ep \right| ^{-\frac12})$.

\smallskip

Theorem~\ref{t.superdiffusivity} is the first \emph{iterative} homogenization result in the random setting. All previous results to our knowledge, either qualitative or quantitative, observe homogenization after essentially a finite scale separation. In contrast, the emergence of superdiffusivity is the result of an infinite cascade of enhancements of the diffusivity due to advection. The proof therefore requires quantitative homogenization machinery to be invoked repeatedly across an infinite number of geometric scales, as explained in more detail in Subsections~\ref{ss.heuristic} and~\ref{ss.proofoutline} below. This kind of phenomenon is expected to occur in many problems in statistical physics, including field theory and turbulence, in which one observes the divergence of correlation length scales.

\smallskip

We also mention recent related works~\cite{AV,BSW} in which multiscale, incompressible vector fields are constructed from periodic ingredients and analyzed by iterative quantitative homogenization. Armstrong and Vicol~\cite{AV} obtain anomalous diffusion, while Burczak, Sz\'ekelyhidi, and Wu~\cite{BSW} establish anomalous dissipation and enhanced pair dispersion. These analyses use the periodic structure to obtain stronger quantitative homogenization estimates than are available in the random setting. The tradeoff is that the vector fields in these explicit constructions are less generic, in a statistical sense, than the one considered here.

\subsection{Liouville theorems and large-scale regularity}

The process estimates are obtained from large-scale regularity estimates for the~$\Lop$-harmonic functions associated with its infinitesimal generator~$\Lop$. We next present large-scale~$C^{0,\gamma}$ and~$C^{1,\gamma}$ estimates exhibiting the superdiffusive scaling, together with Liouville-type classifications of~$\Lop$-harmonic functions with strictly sublinear and, respectively, subquadratic growth at infinity.

\smallskip

To see why some regularity estimate is necessary for the invariance principle, note that a quenched invariance principle requires pointwise homogenization in~$L^\infty$. Since homogenization estimates are naturally obtained with~$L^2$ spatial integrability, an~$L^\infty$ estimate like the one in Theorem~\ref{t.superdiffusivity} is typically proved by combining an~$L^2$ estimate with some uniform equicontinuity of the sequence~$\{ u^\ep \}_{\ep>0}$\,. Conversely, the strengthened invariance principle with convergence from every sequence of initial points~$x_\ep\to x$ implies the corresponding equicontinuity by~\cite[Theorem 19.25]{Kallenberg}. This kind of regularity estimate is not true deterministically and cannot be obtained by quoting classical elliptic theory since the equation for~$u^\ep$ becomes highly singular and degenerate as~$\ep \to 0$ (e.g.,~the De Giorgi-Nash estimates do not apply).

\smallskip

The importance of the regularity theory is not limited to its use as a technical tool to obtain Theorem~\ref{t.A} from Theorem~\ref{t.superdiffusivity}. In fact, the regularity theory presented here plays a central role in the paper, and without it we would be unable to prove even an~$L^2$ version of Theorem~\ref{t.superdiffusivity}, or to obtain the sharp asymptotic growth of the renormalized diffusivities. Indeed, it is the engine that allows us to formalize the heuristic renormalization group argument outlined below in Section~\ref{ss.heuristic}.

\smallskip

The first regularity statement we present is a large-scale H\"older~$C^{0,\gamma}$-type estimate, for~$\gamma<1$. It states that the oscillation of a solution in a ball decays at a rate consistent with H\"older continuity, as one decreases the radius of the ball from a large macroscopic scale down to a random minimal scale. This is paired with a Liouville-type result, which asserts that solutions satisfying the sublinear growth condition in~\eqref{e.Liouville.Calpha} are constants. The latter should be regarded as a soft, infinite-volume, qualitative version of the former.

\smallskip

Here and throughout the paper we denote volume-normalized integrals and~$L^p$ norms by
\begin{equation*}
	\fint_U f(x) \,dx  \coloneqq  \frac{1}{|U|} \int_U f(x)\,dx
	\qquad \mbox{and} \qquad
	\| f \|_{\underline{L}^p(U)} \coloneqq  \Bigl( \fint_U |f(x)|^p \,dx \Bigr)^{\nicefrac1p}
	\,.
\end{equation*}
We write~$\A(U)$ for the set of~$\a$-harmonic functions in~$U$, and put
\begin{equation*}
	2_* \coloneqq \frac{2d}{d+2}\,.
\end{equation*}

\begin{theoremA}[Large-scale~$C^{0,\gamma}$ estimate]
\label{t.large.scale.Holder}
For each~$\gamma ,\sigma \in (0,1)$, there exists~$C(\gamma,\sigma,\nu,\cstar,\nondegconst,d)<\infty$ and a random variable~$\X\geq2$ satisfying
\begin{equation}
	\label{e.large.scale.Holder.X}
	\P \bigl[ \X > t \bigr]
	\leq
	C \exp\bigl( - C^{-1}  ( \log t)^\sigma \bigr)
	\,, \quad \forall t\in[2,\infty)\,,
\end{equation}
such that the following statements are valid.

\begin{enumerate}
\item \underline{Liouville theorem.} Almost surely with respect to~$\P$, if~$u\in H^1_{\mathrm{loc}}(\Rd)$ satisfies
\begin{equation}
	\label{e.Liouville.Calpha}
	-\nabla \cdot \a \nabla u = 0 \quad \mbox{in} \ \Rd
	\qquad \mbox{and} \qquad
	\liminf_{r\to \infty} r^{-\gamma} \| u \|_{\underline{L}^2(B_r)} =0\,,
\end{equation}
then~$u$ is constant.

\item \underline{Large-scale~$C^{0,\gamma}$ estimate.} For every~$R \geq \X$,~$f \in L^\infty(B_R)$ and solution~$u\in H^1(B_R)$ of the equation
\begin{equation}
	\label{e.LSH.pde}
	-\nabla \cdot \a \nabla u = f \quad \mbox{in} \ B_R\,,
\end{equation}
we have, for every~$r \in [\X,R/2]$, the estimate
\begin{equation}
	\label{e.large.scale.Holder}
	\| u - (u)_{B_r} \|_{L^\infty(B_r)}
	\leq C \Bigl( \frac{r}{R} \Bigr)^{\!\gamma} \Bigl( \| u - (u)_{B_R} \|_{\underline{L}^2(B_R)} + (\log R)^{-\nf12} R^2 \| f \|_{L^\infty(B_R)} \Bigr)
	\,.
\end{equation}
\end{enumerate}
\end{theoremA}

The restriction~$\gamma<1$ is sharp in the sense that the statement for~$\gamma=1$ is false. As explained below, a large-scale~$C^{0,1}$ estimate of this type would be inconsistent with superdiffusivity, and the best that one could hope for in our setting is a~$\sqrt[4]{\log}$-Lipschitz type estimate (see Conjecture~\ref{conj.growth} below). This is in contrast to the case of uniformly elliptic homogenization, in which a large-scale~$C^{0,1}$ estimate is valid (cf.~\cite[Theorem 1.21]{AK.Book}). See the introduction of~\cite{AV} for more discussion on the necessary trade-off between regularity of the solutions and the strength of superdiffusivity.

\smallskip

The factor~$(\log R)^{-\nf12} R^2$ multiplying the norm of~$f$ in~\eqref{e.large.scale.Holder} differs from the usual factor of~$R^2$ found in elliptic estimates. The extra factor of~$(\log R)^{-\nf12}$ reflects the superdiffusive scaling, and ensures that the rescaled equation in~\eqref{e.BVPs} will have uniform bounds in~$\ep$. As part of our analysis we will coarse-grain the Caccioppoli and Poincar\'e inequalities and obtain roughly that, for some random minimal scale~$\X$ and all solutions~$u$ of~\eqref{e.LSH.pde} on scales~$r\geq\X$, we have
\begin{equation*}
	\begin{aligned}
		&  \| u - (u)_{B_r} \|_{\underline{L}^2(B_r)}  \lesssim
		\frac{ r\nu^{\nf12}}{(\log r)^{\nicefrac 14 }}
		\| \nabla u \|_{\underline{L}^2(B_r)}
		+ r^2 \| f \|_{\underline{L}^{2}(B_r)}
		\,,
		&& \mbox{(superdiffusive Poincar\'e)}
		\\
		&
		\frac{ r \nu^{\nf12} }{(\log r)^{\nicefrac 14 }} \| \nabla u \|_{\underline{L}^2(B_r)}
		\lesssim  \| u - (u)_{B_{2r}} \|_{\underline{L}^2(B_{2r})}
		+ \frac{r^2}{ (\log r)^{\nicefrac 12}} \| f \|_{\underline{L}^2(B_{2r})} \,.
		&& \mbox{(superdiffusive Caccioppoli)}
	\end{aligned}
\end{equation*}
See Lemmas~\ref{l.Dirichlet.Cacc.blackbox} and~\ref{l.superdiffusive.Poincare.with.rhs} for the precise statements. In particular, the bound~\eqref{e.large.scale.Holder} implies
\begin{equation}
	\label{e.large.scale.Holder.nabla}
	\nu^{\nf12}\| \nabla u \|_{\underline{L}^2(B_r)}
	\leq C \Bigl( \frac{r}{R} \Bigr)^{\!\gamma-1} \Bigl( \nu^{\nf12}\| \nabla u\|_{\underline{L}^2(B_R)} +
	(\log R)^{-\nicefrac14} R \| f \|_{L^\infty(B_R)} \Bigr)
	\,,
\end{equation}
where we have removed the logarithmic factors by slightly adjusting the parameter~$\gamma$. This estimate does appear to have the usual elliptic scaling, but the superdiffusive scaling is hiding in these coarse-grained (``superdiffusive'') Caccioppoli and Poincar\'e inequalities.

\smallskip

The estimate~\eqref{e.large.scale.Holder} in Theorem~\ref{t.large.scale.Holder} is not strong enough to imply an estimate on the actual (pointwise) H\"older seminorm of a solution. While we do not know if such an estimate is true, we suspect not. Pointwise bounds are typically obtained from large-scale regularity estimates by covering the domains with balls of the form~$B_{\X(y)}(y)$, where~$\X(y)$ is the minimal scale for the environment centered at~$y$, and then applying local estimates (such as Schauder or De Giorgi-Nash) in each of these balls. Carrying this out requires an estimate of the maximum of the~$\X(y)$'s, which is typically obtained by a union bound. The difficulty in our setting is that the stochastic integrability of the minimal scale~$\X$ in~\eqref{e.large.scale.Holder.X} is very weak: since~$\sigma<1$, we have no positive moment bounds on~$\X$, only stretched exponential moments for~$\log \X$. This restricts our ability to make use of union bounds.

\smallskip

Theorem~\ref{t.large.scale.Holder} gives a H\"older estimate down to a mesoscopic scale. This estimate implies uniform equicontinuity for sequences of solutions~$u^\ep$ of~\eqref{e.BVPs}, which suffices for the pointwise~$L^\infty$ estimate in Theorem~\ref{t.superdiffusivity} and hence for the invariance principle. As explained in more detail below---see the discussion leading to~\eqref{e.Cgamma.uep.eta.yep}---a solution~$u^\ep$ of the first problem in~\eqref{e.BVPs} satisfies, for all~$s>0$ and~$V$ with~$\overline{V} \subseteq U$ and small~$\ep>0$ (smaller than a random~$\ep_0$),
\begin{equation*}
	\| u^\ep \ast \eta_{\left| \log \ep \right| ^{-s}} \|_{C^{0,\gamma}(V)}
	\leq
	C \bigl( \| g \|_{L^\infty(U)} + \| f \|_{L^\infty(U)} \bigr)\,,
\end{equation*}
where the constant~$C$ depends on the subdomain~$V$ in addition to~$(U,\gamma,\cstar,\nu,\nondegconst,d)$ and~$\{ \eta_r \}_{r>0}$ is the standard mollifier. We do not possess continuity estimates on the solutions~$u^\ep$ on scales smaller than a power of~$\left| \log \ep \right| ^{-1}$, and we in fact expect the solution to have wild, ``intermittent'' behavior on such scales which can be highly localized in space (and will be localized in the regions corresponding to large values of the random variable~$\X$ in Theorem~\ref{t.large.scale.Holder}).

\smallskip

Our next main result is a large-scale~$C^{1,\gamma}$ regularity estimate, which can be considered as a next-order analogue of Theorem~\ref{t.large.scale.Holder}. We denote, for~$\gamma\in (0,1)$, the linear subspace of~$\A(\Rd)$ consisting of functions which grow like~$O(|x|^{1+\gamma})$ as~$|x|\to\infty$ by
\begin{equation}
	\label{e.harmonic.coordinates}
	\mathcal{A}^{1+\gamma} (\Rd)
	\coloneqq
	\Bigl\{
	u \in \mathcal{A}(\Rd)
	\,:\,
	\limsup_{r \to \infty} r^{-(1+\gamma)} \|u\|_{\underline{L}^2(B_r)} < \infty
	\Bigr\} \,.
\end{equation}
The next theorem asserts that~$\mathcal{A}^{1+\gamma}(\Rd)$ has dimension~$1+d$, the dimension of the affine functions. It also states that, for large~$R$, an arbitrary element of~$\A(B_R)$ can be approximated on every smaller ball above a random minimal scale by an element of~$\mathcal{A}^{1+\gamma}(\Rd)$, with the same precision as the approximation of a harmonic function by an affine function.

\begin{theoremA}[Large-scale~$C^{1,\gamma}$ estimate]
\label{t.C1beta}
Let~$\gamma,\sigma \in (0,1)$ and~$0 < \alpha < \frac12(1-\sigma)$. There exists a constant~$C(\alpha,\gamma,\sigma,\nu, \cstar,\nondegconst, d)<\infty$ and a random variable~$\X\geq2$ satisfying
\begin{equation}
	\label{e.Xgamma.Xi}
	\P \bigl[ \X > t \bigr]
	\leq
	C \exp\bigl( - C^{-1}  ( \log t)^\sigma \bigr)
	\,, \quad \forall t\in[2,\infty)
\end{equation}
such that the following statements are valid.

\begin{enumerate}
\item \underline{Liouville theorem.} Almost surely with respect to~$\P$, the space~$\mathcal{A}^{1+\gamma} (\Rd)$ has dimension~$1+d$ and does not depend on~$\gamma$.

\item \underline{Flatness at every scale.} For every~$\phi\in\A^{1+\gamma}(\Rd)$ and~$r \geq \X$, we have
\begin{equation}
	\label{e.flatness.at.every.scale}
	\inf_{e\in\Rd} \| \phi - \linear_e - (\phi)_{B_r} \|_{\underline{L}^2(B_r)}
	\leq
	C (\log r)^{-\alpha}  \| \phi \|_{\underline{L}^2(B_r)} \,.
\end{equation}

\item \underline{Large-scale~$C^{1,\gamma}$ estimate.} For each~$R\in [\X,\infty)$ and~$u \in \mathcal{A}(B_R)$, there exists~$\phi \in \mathcal{A}^{1+\gamma} (\Rd)$ such that
\begin{equation}
	\label{e.largescaleC1gamma}
	\| \nabla u - \nabla \phi\|_{\underline{L}^2(B_r)} \leq C \Bigl( \frac{r}{R} \Bigr)^{\!\gamma} \| \nabla u\|_{\underline{L}^2(B_R)}
	\,, \quad \forall r \in [\X,R)\,.
\end{equation}
\end{enumerate}
\end{theoremA}

In light of the first statement of the theorem, we henceforth drop the~$\gamma$ and denote~$\mathcal{A}^{1+\gamma}(\Rd)$ by~$\mathcal{A}^1(\Rd)$.

\smallskip

In the case of stationary uniformly elliptic equations, a large-scale~$C^{1,\gamma}$ estimate similar to the statement of Theorem~\ref{t.C1beta} is valid (see~\cite{GNO.reg} or~\cite[Theorem 1.21]{AK.Book}) and plays a central role in the theory of quantitative homogenization. In that case, the space~$\A^{1}(\Rd)$ consists of solutions which can be written as an affine function plus a stationary corrector which is strictly sublinear at infinity. In particular, each of these solutions grows like a particular affine function and has a gradient which is stationary; unlike in our setting, the ``flatness at every scale'' assertion~\eqref{e.flatness.at.every.scale} is valid with a slope~$e$ that is independent of the scale~$r$. Consequently there exists a canonical linear isomorphism from~$\A^{1}(\Rd) / \R$ to~$\Rd$, that is, the isomorphism respects stationarity.

\smallskip

The space~$\A^{1}(\Rd)$ in our setting has very different behavior. The vector~$e$ attaining the infimum in~\eqref{e.flatness.at.every.scale} cannot be selected independently of the scale~$r$. That is, while the elements of~$\A^{1}(\Rd)$ will indeed be flat at every scale as~\eqref{e.flatness.at.every.scale} asserts, they will have different slopes at different scales. As a consequence, the individual elements of~$\A^{1}(\Rd)$ do not have stationary gradients and there is no canonical (stationary) isomorphism from~$\A^{1}(\Rd)/\R$ to~$\Rd$.

\smallskip

Indeed, if~$e_r[\phi] \in\Rd$ attains the infimum in~\eqref{e.flatness.at.every.scale}, then we should expect~$e_r[\phi]$ to be related to the energy of~$\phi$ in~$B_r$ and the effective diffusivity~$\tilde{\s}_r$ at scale~$r$ by
\begin{equation}
	\label{e.not.stationary}
	\nu \| \nabla \phi \|_{\underline{L}^2(B_r)} ^2
	\approx
	\tilde{\s}_r | e_r[\phi]|^2\,,
\end{equation}
up to a small error. This type of relation is made precise in the proof of Theorem~\ref{t.C1beta}. Meanwhile, the main results in the present work assert that~$\tilde{\s}_r \approx (2\cstar\log r)^{\nf12}$. Now, if~$\nabla \phi$ was a stationary field with bounded second moment, then~$e_r[\phi]$ would converge as~$r\to \infty$ to the expectation of~$\nabla \phi(0)$, by the ergodic theorem, and the left side of~\eqref{e.not.stationary} would also be independent of~$r$ in expectation. This is clearly inconsistent with~$\tilde{\s}_r \simeq (\log r)^{\nf12}$.

\smallskip

This reasoning suggests that the size of the ``slope'' of an element~$\phi \in \A^{1}(\Rd)$ in a ball~$B_r$ should scale with the radius~$r$ like~$|e_r[\phi]| \simeq \tilde{\s}_r^{-\nf12} \simeq ( 2\cstar \log r)^{-\nicefrac 14}$, leading us to the following conjecture.

\begin{conjectureA}
\label{conj.growth}
For every element~$\phi \in \A^{1}(\Rd)$, the energy density~$\nu |\nabla \phi|^2$ is a stationary field. Moreover, if~$e_r[\phi] \in\Rd$ attains the infimum in~\eqref{e.flatness.at.every.scale}, then
\begin{equation}
	\lim_{r\to \infty} \frac{| e_r[\phi]| }{(2\cstar \log r)^{-\nicefrac14}} = \E \bigl[ \nu |\nabla \phi |^2 \bigr]^{\nf12}
	\,.
\end{equation}
\end{conjectureA}

If Conjecture~\ref{conj.growth} is true, then it is immediate from~\eqref{e.flatness.at.every.scale} that  every nonconstant element~$\phi$ of~$\A^{1}(\Rd)$ grows like~$|x| (\log |x|)^{-\nicefrac14}$. Combining this with the first statement of Theorem~\ref{t.C1beta}, we would deduce that the equation has no solutions which grow at least linearly but at most like~$O(|x|^{1+\gamma})$ for~$\gamma<1$. It can also be combined with the third statement of Theorem~\ref{t.C1beta} to obtain a large-scale~$\sqrt[4]{\log}$-Lipschitz estimate, which would evidently be the sharp estimate (in view of the second statement of the conjecture).

\smallskip

We  offer a second but related conjecture about the scaling limit of the vector space~$\A^{1}(\Rd)$.

\begin{conjectureA}
\label{conj.scaling.limit}
Given~$e\in \Rd$ with~$|e|=1$ and~$r \geq \X$, let~$\phi_{e,r}$ denote the unique element of~$\mathcal{A}^1(\Rd)$ satisfying~$e_{r} [\phi_{e,r}] = e$ and~$(\phi_{e,r})_{B_r}=0$. Then we have that the vector field
\begin{equation*}
	(2\cstar \log r)^{\nicefrac 12} (\nabla \phi_{e,r} - e ) ( r\cdot)
\end{equation*}
converges in law (with respect to the topology of distributions), jointly with respect to~$e$, to the random field~$\nabla \Phi_e$, where~$\Phi_e=\Delta^{-1} (\nabla\cdot \k e)$ is chosen so that its average vanishes in~$B_1$.
\end{conjectureA}

\subsection{Heuristic arguments for the superdiffusive scaling}
\label{ss.heuristic}

The logarithmic divergence of the diffusivity in Theorems~\ref{t.A} and~\ref{t.superdiffusivity} is a result of an accumulation of smaller enhancements of diffusivity due to advection, iterated across many length scales. We formalize this by tracking the change in the effective diffusivity between successive scales, as a function of the scale. This subsection contains an informal description of this strategy and a more detailed outline appears in Section~\ref{ss.proofoutline} below.

\smallskip

The vector field~$\f$ in~\eqref{e.ff} is the sum of vector fields~$\nabla \cdot \mathbf{j}_n$, each of which has size~$3^{-n}$ and inverse frequencies (wavelengths) of order~$3^n$. It follows that, for times~$t \ll 3^{2n}$, the vector field~$\nabla \cdot \mathbf{j}_n$ has a small relative effect on the position~$X_t$ of the particle. As time grows, the particle ``sees'' more and more of the terms~$\nabla \cdot \mathbf{j}_n$. Each such term enhances the effective diffusivity of the particle, which should thus be viewed as a function of the scale.

\smallskip

This story can be told in the language of homogenization, in terms of the infinitesimal generator of the diffusion process, which is
\begin{equation}
	\Lop = \nabla \cdot \bigl( \nu \Id + \k \bigr) \nabla  \,.
\end{equation}
For each~$n\in\N$, on length scales much smaller than~$3^n$, the matrix~$\mathbf{j}_n$ is nearly a constant, by~\eqref{e.k(n).reg}. Since adding a constant anti-symmetric matrix to the divergence in~$\Lop$ leaves~$\Lop$ unchanged, we deduce that
\begin{equation}
	\label{e.L.infrared.approx}
	\Lop \approx \nabla \cdot \Bigl( \nu \Id + \sum_{k=0}^n \mathbf{j}_k \Bigr) \nabla  =: \Lop_n
	\quad
	\mbox{on length scales smaller than~$3^n$.}
\end{equation}
The operator~$\Lop_n$ is called the \emph{infrared cutoff} of~$\Lop$ on scale~$3^n$, and it has range of dependence of order~$3^n$, by assumption~\ref{a.j.frd}. By classical homogenization theory, the operator~$\Lop_n$ is close to a constant-coefficient, deterministic operator~$\overline{\Lop}_n$ on length scales much larger than~$3^n$:
\begin{equation}
	\label{e.Ln.homog.approx}
	\Lop_n \approx \overline{\Lop}_n =: \shom_n \Delta
	\quad
	\mbox{on length scales larger than~$3^n$.}
\end{equation}
The dihedral symmetry assumption~\ref{a.j.iso} ensures that~$\overline{\Lop}_n$ is a multiple of the Laplacian, and we define the constant~$\shom_n>0$ to be this multiple. It is natural to expect coarse-graining on a smaller scale to (approximately) commute with taking an infrared cutoff on a larger scale. Indeed, adding a constant anti-symmetric matrix commutes with homogenization, and the larger scale matrices are approximately constant on smaller scales. This leads to the ansatz
\begin{equation}
	\label{e.L.homog.approx}
	\Lop \approx
	\nabla \cdot \Bigl( \shom_n \Id + \sum_{k=n+1}^\infty \mathbf{j}_k \Bigr) \nabla
	\quad
	\mbox{on length scales larger than~$3^n$.}
\end{equation}
In effect, the oscillations in the vector field on scales smaller than~$3^n$ have been integrated out, \emph{becoming part of the diffusion matrix}, which has increased its value from~$\nu$ to~$\shom_n$. This process can be iterated, resulting in a \emph{reverse cascade of homogenization}. Quantifying the increase in diffusivity between successive scales leads to an approximate recurrence relation for~$\shom_n$ which allows us to compute, to leading order, its growth rate.

\smallskip
This heuristic is not new. It is a renormalization group argument that is present in some form in the papers~\cite{FNS,FFQSS, KLY, Fisher, AN,BCGLD,Fann}, with~\cite{Fann} being the closest to our discussion here. It is also similar to the heuristics presented in~\cite{CMOW} and in~\cite{AV}, which considers a different but related problem.

\smallskip

There are however several difficulties in passing from~\eqref{e.L.infrared.approx} and~\eqref{e.Ln.homog.approx} to the conclusion~\eqref{e.L.homog.approx}, stemming from a \emph{lack of scale separation} and a \emph{large ellipticity contrast} in the diffusion matrices.\footnote{These difficulties can, to a certain extent, be circumvented if one is after \emph{annealed} as opposed to \emph{quenched} estimates. Our discussion here is oriented toward the latter.}

\smallskip

An assumption of scale separation was implicit in passing from~\eqref{e.L.infrared.approx}~\&~\eqref{e.Ln.homog.approx} to~\eqref{e.L.homog.approx}. Indeed, this implication relies on the ``macroscopic'' part of the vector field, represented by~$\sum_{k=n+1}^\infty \mathbf{j}_k$, not affecting the homogenization of the ``microscopic'' scales, i.e., the range of dependence of~$\Lop_n$. However, this is clearly a tenuous assumption, since~$\mathbf{j}_{n+1}$ is active on scale~$3^{n+1}$, which is only a factor of three more than the microscopic scale.

\smallskip

This difficulty is compounded by the large ellipticity contrast of the operator~$\Lop_n$: the ratio of its ellipticity constants in a cube of size~$3^n$ is typically of order~$\nu^{-2} n^2$.

\smallskip

This is where we crucially rely on the high contrast quantitative homogenization theory developed in~\cite{AK.HC}. The results in~\cite{AK.HC} assert that, for the high-contrast fields used here, the relative homogenization error becomes small after approximately~$O(\log^2 n)$ many geometric scales. In our setting, this says that~$\Lop_n$ should homogenize by at most length scale~$3^{n + C(\log n)^2}$. We can therefore update~\eqref{e.Ln.homog.approx} to the more precise claim that
\begin{equation}
	\label{e.Ln.homog.approx.polylog}
	\Lop_n \approx \overline{\Lop}_n =: \shom_n \Delta
	\quad
	\mbox{on length scales larger than~$3^{n + C(\log n)^2}$.}
\end{equation}

\smallskip

The problem of scale separation will be fixed by arguing that the error due to the overlapping scales near~$3^n$ is much smaller than the size of~$\shom_n$. Indeed, if we can show that~$\shom_n\to \infty$ as~$n\to \infty$ at a sufficiently fast rate, then we can neglect the contribution of~$\mathbf{j}_k$ for~$k$ in a large interval and make a small relative error. Using this idea, we will eventually show that the homogenization approximation in~\eqref{e.Ln.homog.approx.polylog} is actually valid on scales \emph{below}~$3^n$:
\begin{equation}
	\label{e.Ln.homog.approx.betterer}
	\Lop_n \approx \overline{\Lop}_n =: \shom_n \Delta
	\quad
	\mbox{on length scales larger than~$3^{n - C n^\delta}$, for every~$\delta \in (0, \nicefrac12)$.}
\end{equation}
This requires a lower bound on~$\shom_n$ of~$n^{\nf12} (\log n)^{-\xi}$ for some exponent~$\xi <\infty$, which is achieved by a crude and less precise version of the argument which is to follow.

\smallskip

To obtain the recurrence relation, we use~\eqref{e.Ln.homog.approx.betterer} to obtain roughly that
\begin{align}
	\label{e.heuristic.step}
	\overline{\Lop}_{n+k} \approx \Lop_{n+k}
	& \approx
	\nabla \cdot \Bigl( \shom_n \Id + (\mathbf{j}_{n+1}+\cdots+ \mathbf{j}_{n+k}) \Bigr) \nabla
	\notag \\ &
	=
	\shom_n \nabla \cdot \Bigl( \Id + \shom_n^{-1}(\mathbf{j}_{n+1}+\cdots+ \mathbf{j}_{n+k}) \Bigr) \nabla \,.
\end{align}
Since~$\shom_n \gtrsim n^{\nf12} (\log n)^{-\xi}$, for fixed~$k \ll n^{\delta}$ for~$\delta\in (0,\nicefrac12)$ and large enough~$n$ we have
\begin{equation}
	\label{e.size.of.perturb}
	\shom_n^{-1} (\mathbf{j}_{n+1}+\cdots+ \mathbf{j}_{n+k}) \lesssim k^{\nf12} \shom_n^{-1}
	\simeq
	k^{\nf12} n^{-\nf12} (\log n)^{\xi}
	\ll
	n^{-\frac14 }
	\,.
\end{equation}
The operator on the right side of~\eqref{e.heuristic.step} can therefore be analyzed perturbatively. For the normalized fresh-shell perturbation~$\pert\coloneqq\shom_n^{-1}(\mathbf{j}_{n+1}+\cdots+\mathbf{j}_{n+k})$, whose constituent shells are independent and whose projected energy satisfies~\ref{a.j.nondeg}, the argument in Section~\ref{s.sharp.asympt} yields
\begin{equation*}
	\nabla \cdot \bigl( \Id + \pert \bigr) \nabla
	\quad
	\mbox{homogenizes to} \quad
	\shom \Delta \,,
	\quad \mbox{where} \quad
	\shom = 1 + \E \bigl[ |\nabla \Delta^{-1} \nabla\cdot \pert e_1|^2(0) \bigr]
	+
	O \bigl( \shom_n^{-4}k^2 \bigr)
	\,.
\end{equation*}
After multiplying the perturbative expansion by~$\shom_n$, the quartic term contributes a lower-order error under the smallness condition~\eqref{e.size.of.perturb}. The assumption~\ref{a.j.nondeg} tells us that
\begin{equation*}
	\Bigl|  \E \bigl[ |\nabla \Delta^{-1} \nabla\cdot \pert e_1|^2(0) \bigr] - \shom_n^{-2} k \cstar \log 3 \Bigr|
	\leq
	\shom_n^{-2}\nondegconst
	\,.
\end{equation*}
We deduce that, for~$k \ll n^{\delta}$ with~$\delta \in (0,\nicefrac12)$,
\begin{equation}
	\label{e.intro.recurrence}
	\shom_{n+k} = \shom_n \bigl( 1 + k(\cstar \log 3 ) \shom_n^{-2}  \bigr) +
	O\bigl((\log^2 n+\nondegconst)\shom_n^{-1}\bigr)
	\,.
\end{equation}
A simple analysis of this recurrence gives that
\begin{equation}
	\label{e.intro.asymps}
	\frac{\shom_n}{\sqrt{2\cstar \log 3^n}} \rightarrow 1 \quad \mbox{as}  \ n\to\infty\,.
\end{equation}
To obtain this precise growth rate, we need to make~$k$, the number of scales in our recurrence step, large enough that the increment in the recurrence~\eqref{e.intro.recurrence} dominates the error on the right side. By quantifying this idea we are able to obtain a convergence rate for this limit, which is stated below in Theorem~\ref{t.sstar.sharp.bounds}.

\smallskip

With the sharp asymptotics of~$\shom_n$, we can return to~\eqref{e.L.infrared.approx} and~\eqref{e.Ln.homog.approx.betterer} to conclude that
\begin{equation*}
	\Lop \approx \sqrt{2\cstar \log 3^n} \Delta \quad \mbox{on length scales between~$3^{n - C n^\delta}$ and~$3^{n}$.}
\end{equation*}
This tells us that the original process~$\{ X_t \}$ behaves, on length scales between~$3^{n - C n^\delta}$ and~$3^{n}$, like a Brownian motion with covariance~$\sqrt{8\cstar \log 3^n}$. In other words, the rescaled process~$X^\ep_t$ defined in~\eqref{e.Xep.t.def} should be close to a Brownian motion.

\subsection{An outline of the rigorous proof}
\label{ss.proofoutline}

In this subsection we explain how the informal heuristics above are formalized and give a detailed overview of the structure of the proofs of Theorems~\ref{t.A} and~\ref{t.superdiffusivity}. Rather than attempt to directly iterate homogenization of the operators~$\Lop_n$ defined in~\eqref{e.L.infrared.approx}, we instead analyze the asymptotics of certain (quenched) quantities which we regard as representing the effective diffusion matrix at different length scales. We call these quantities the \emph{coarse-grained diffusion matrices} of the diffusion matrix~$\nu\Id+\k$. These objects were introduced in quantitative homogenization; see~\cite[Chapter~5]{AK.Book} for the precise variational formulation used here.

\smallskip

As recalled in Section~\ref{s.coarse.graining}, below, we define, for each cube~$\cu\subseteq\Rd$ and coefficient field~$\a(x)$, a dual pair of symmetric matrices~$\s(\cu)$ and~$\s_*(\cu)$, and another matrix~$\k(\cu)$. We think of~$\s(\cu)$ and~$\s_*(\cu)$ as representing two competing notions of the ``symmetric part of the effective diffusion matrix in~$\cu$'' and~$\k(\cu)$ as representing the anti-symmetric part. The two symmetric matrices are ordered: they satisfy~$\s_*(\cu) \leq \s(\cu)$, and the error in certain ``coarse-graining estimates'' will become small when the difference~$\s(\cu) - \s_*(\cu)$ is small. We therefore regard~$[\s_*(\cu)+\k(\cu), \s(\cu)+\k(\cu)]$ as a confidence interval for the effective diffusion matrix.\footnote{These coarse-grained diffusion matrices are related to, but different from, the quantity that Fannjiang proposes to analyze in~\cite{Fann} which he calls \emph{box diffusivity}. In fact, his quantity lies in the interval~$[\s_*(U),\s(U)]$.} There are also \emph{annealed} versions of these quantities defined below, which we denote by~$\shom(\cu)$,~$\shom_*(\cu)$ and~$\khom(\cu)$. As we will see in~\eqref{e.homs.defs.U} below, in our context we have, by symmetry, that~$\khom(\cu)=0$.

\smallskip

The coarse grained diffusion matrices can be thought of as the coefficients of the corresponding elliptic operator in a wavelet-type expansion. They organize and compress the information in the full elliptic operator into discrete multiscale representatives. Indeed, quantitative estimates on the coarse-grained matrices can be translated into estimates on the solutions. For instance, we will deduce Theorem~\ref{t.superdiffusivity} as a consequence of the quantitative convergence of the coarse-grained matrices, which roughly states that, for the coefficient field~$\a(x) = \nu\Id + \k(x)$, we have that, for every~$\alpha>0$,
\begin{equation}
	\label{e.intro.wts}
	\bigl| \s(\cu_m) -
	( 2 \cstar (\log 3) m )^{\nf12}  \Id
	\bigr|
	+
	\bigl| \s_*(\cu_m) -
	( 2 \cstar (\log 3) m )^{\nf12}  \Id
	\bigr|
	+
	\bigl| \k(\cu_m) - (\k)_{\cu_m} \bigr|
	\lesssim
	m^{\alpha}
	\,.
\end{equation}
Recall that~$\cu_m$ denotes the axis-aligned cube centered at the origin with side length~$3^m$,
\begin{equation*}
	\cu_m \coloneqq  \Bigl( -\frac123^m , \frac123^m\Bigr)^d\,.
\end{equation*}
As discussed above, the field~$\nu\Id + \k(x)$ is only defined modulo a constant anti-symmetric matrix. The coarse-grained matrices inherit this property, and in fact the matrices~$\s(U)$ and~$\s_*(U)$ do not depend on the choice of the anti-symmetric matrix, while the matrix~$\k(U)$ transforms by addition of the same constant anti-symmetric matrix. The expressions on the left side of~\eqref{e.intro.wts} are therefore unambiguously defined.

\smallskip

\begin{figure}
	\begin{center}

\begin{tikzpicture}[auto, node distance=0.6cm,>=latex,block/.style={draw, fill=white, rectangle}, minimum height=1.2em, minimum width=6em]

		\node[block, text width = 4.25 cm, align = center] (A0) {\hyperref[ss.localization]{\emph{Localization of~$\s_{L,*}(\cu)$ and~$\k_L(\cu)$}}};

		\node[block, right = of A0, text width = 4.25 cm, align = center] (A1) {\hyperref[s.subopt]{\emph{A suboptimal lower bound for~$\s_{L,*}(\cu)$}}};
		\node[block,  below = of A1, text width = 4.25 cm, align = center](A2){
		\hyperref[s.homog.below.cutoff]{\emph{Homogenization below infrared cutoff}.}
		};
		\node[block, below = of A2,  text width = 4.25 cm, align = center](A3){
		\hyperref[s.sharp.C.one]{\emph{Large-scale regularity}}
		};
		\node[block, right = of A3, text width = 4.25 cm, align = center](A4){
		\hyperref[s.Dirichlet]{\emph{Weak and pointwise Dirichlet estimates}}
		};
		\node[block, below =  2cm of A2, text width = 4.25 cm, align = center](A5){
		\hyperref[s.sharp.asympt]{\emph{Recurrence relation and sharp superdiffusivity}}
		};
		\node[block, below = 4cm of A2, text width = 4.25 cm, align = center](A6){
		\hyperref[s.Dirichlet]{\emph{Homogenization in~$L^{\infty}$}}
		};
		\node[block, right = of A6, text width = 4.25 cm, align = center](A7){
			\hyperref[s.invariance.principle]{\emph{Consequences for the diffusion~$X_t$}}
		};

	\draw[->] (A0) -- (A1); 	\draw[->] (A0.south)  |- (0,-1.60cm) -- ($(A2.west) + (0,0.1cm)$);
	\draw[->] (A1) -- (A2);
	\draw[->] (A2) -- (A3);
	\draw[->] (A2.east) -- ($(A2.east)+(0.35cm,0)$) |- ($(A4.north)+(0,0.2cm)$) -- (A4.north);
	\draw[->] (A3) -- (A4);
	\draw[->] (A1.east) -- ($(A1.east)+(5.4cm,0)$) |- ($(A5.east)+(0.45cm,0)$) -- (A5.east);
	\draw[->] ($(A2.west) - (0,0.4cm)$) |- ($(A2.west) - (2cm,0.4cm)$) |- ($(A5.west) - (2cm,0)$) -- (A5.west);
	\draw[->] (A5.north) -- (A3.south);
	\draw[->] (A5.north east) -- (A4.south west);
	\draw[->] (A4.south east) |- ($(A6.north east)+(0,0.45cm)$) -- (A6.north east);
	\draw[->] (A5) -- (A6);
	\draw[->] (A6) -- (A7);

	\end{tikzpicture}

	\end{center}
	\caption{Outline of the proof.}
	\label{fig.proof.steps}
\end{figure}
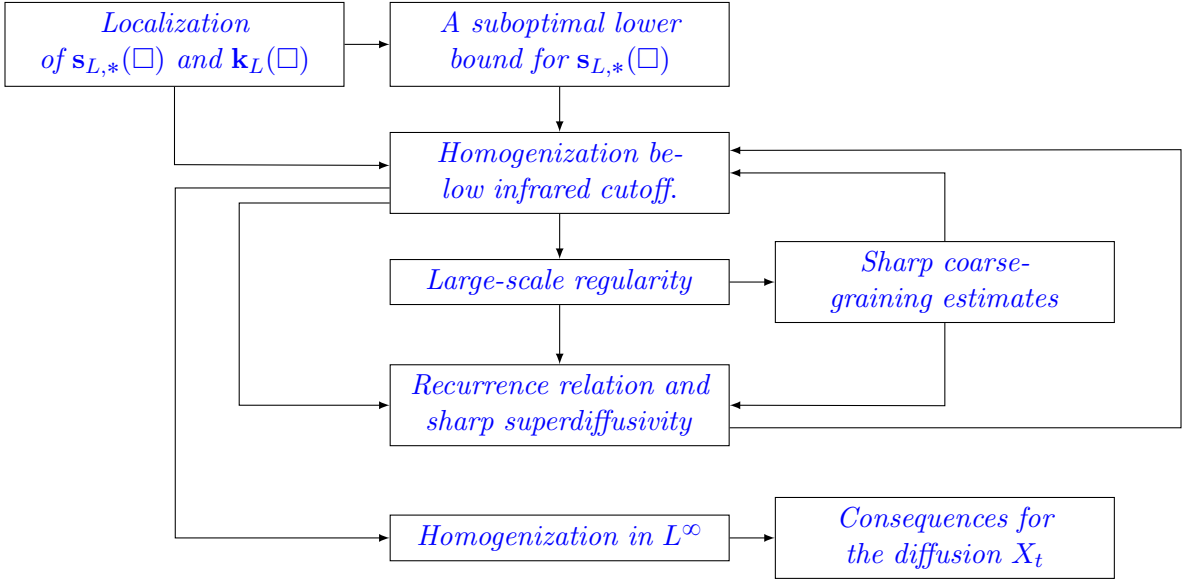

The quantitative bound~\eqref{e.intro.wts} is a very precise estimate which will not be proved until the end of Section~\ref{s.sharp.asympt}. The rest of this discussion is an outline of its proof, illustrated in Figure~\ref{fig.proof.steps}, which we organize into seven main steps.

\smallskip

It is natural to work with infrared cutoffs of the field~$\k(\cdot)$ and~$\a(\cdot)$, which are defined at scale~$3^L$ with~$L\in\N$ by
\begin{equation}
	\label{e.infrared.cutoff.def}
	\k_L(x)  \coloneqq  \sum_{n=0}^L \mathbf{j}_n (x)\,,
	\quad \mbox{and} \quad
	\a_L (x)  \coloneqq  \nu\Id+\k_L(x)
	\,.
\end{equation}
We let~$\s_L(\cu)$,~$\s_{L,*}(\cu)$ and~$\k_L(\cu)$ be the associated coarse-grained matrices. As mentioned in the informal heuristics section above (see the discussion below~\eqref{e.Ln.homog.approx}), the infrared cutoff and coarse-graining operations nearly commute with each other, provided that the scale of the cutoff is larger than that of the coarse-graining---this follows from the fact that the coarse-grained matrices commute with the addition of a constant anti-symmetric matrix and depend continuously on the field. Using this we infer that the coarse-grained matrices of the original field (without cutoff) inherit decorrelation properties from those of the cutoff fields (which by definition have finite range dependence). The precise version of this estimate appears in Lemma~\ref{l.localization}, and we refer to this as the \emph{localization} property of the coarse-grained matrices. This localization property is important because it says that the coarse-grained matrices have stronger decorrelation properties than the log-correlated coefficient field, allowing us to obtain quantitative homogenization estimates.

\smallskip

The relative error in the localization estimate involves the ratio of the size of the terms that are removed, multiplied by the coarse-grained matrix~$\s_*^{-1}(\cu)$. This is the term which appeared above in~\eqref{e.size.of.perturb} and, as discussed there, to estimate it effectively we need a lower bound on~$\s_*(\cu)$, which is very close to~$\s_{L,*}(\cu)$ if the size of~$\cu$ is smaller than~$3^L$.

\smallskip

\hyperref[s.subopt]{\emph{Step one}}: \emph{A suboptimal lower bound for~$\s_{L,*}(\cu)$.} In Section~\ref{s.subopt}, we prove a quenched lower bound for~$\s_{L,*}(\cu_m)$ which states roughly that, for every~$L,m\in\N$ with~$L \geq m \geq L/2$ and~$m$ above an explicit logarithmic threshold,
\begin{equation}
	\label{e.sstar.lower.bound.quenched.intro}
	\s_{L,*}^{-1} (\cu_m)
	\lesssim
	m^{- \nicefrac12} (\log m )^5 \Id
	\,.
\end{equation}
The precise version of~\eqref{e.sstar.lower.bound.quenched.intro} is stated in Proposition~\ref{p.sstar.lower.bound}. This estimate is a suboptimal lower bound on the diffusivity which, like the annealed estimate~\eqref{e.Toninelli}, exhibits a rate of superdiffusivity which is optimal up to a correction which is doubly logarithmic in the length scale.

\smallskip

The proof of~\eqref{e.sstar.lower.bound.quenched.intro} relies on special properties of the matrix quantity~$\s_{L,*}^{-1}(\cu)$. The argument would not work if we attempted to substitute~$\s_{L,*}^{-1}(\cu)$ with other notions of ``diffusivity'', including~$\s_{L}(\cu)$.

\smallskip

The first step in the proof is to observe that~\eqref{e.sstar.lower.bound.quenched.intro} can be reduced to an \emph{annealed} estimate, using the localization of~$\s_*$ and another key property of~$\s_*^{-1}$, namely \emph{subadditivity}, which says that~$\s_*^{-1} (\cu_m)$ can be upper bounded by the sample mean of the~$\s_*^{-1} (z+\cu_n)$'s over the partition of~$\cu_m$ of subcubes of size~$3^n$. Due to the localization property, the sample mean will exhibit stochastic cancellations and, as a result, we deduce strong one-sided control of the fluctuations of~$\s_*^{-1} (\cu_m)$.

\smallskip

It therefore suffices to prove an annealed version of~\eqref{e.sstar.lower.bound.quenched.intro}, which states that
\begin{equation}
	\label{e.sstar.lower.bound.annealed.intro}
	\shom_{L,*}^{-1} (\cu_m)
	\lesssim
	m^{- \nicefrac12} (\log m )^{\nicefrac{9}2}
	\,,
\end{equation}
where~$\shom_{L,*}^{-1}(\cu_m) \coloneqq  \E \bigl[ \s_{L,*}^{-1} (\cu_m)]$ denotes the mean of~$\s_{L,*}^{-1} (\cu_m)$. Note that~$\shom_{L,*}^{-1}(\cu_m)$ is a scalar matrix by the dihedral symmetry assumption, and so by abusing notation slightly we allow~$\shom_{L,*}^{-1}(\cu_m)$ to also denote a scalar.

\smallskip

To prove~\eqref{e.sstar.lower.bound.annealed.intro}, we fix~$L\in\N$ and use a pigeonhole argument to select~$m\in[L/4,L/2]$ such that~$\shom_{L,*}(\cu_k)$ changes little in ratio over a window~$[m-2h,m]$ of length~$2h\ll L$. We then introduce the auxiliary cutoff~$L'=m+2a$, where~$a=\lceil K\log(\nu^{-1}L)\rceil$, so that~$L'-m$ is logarithmic. This selection uses that~$\shom_{L,*}^{-1}(\cu_k)$ is monotone nonincreasing in~$k$ by subadditivity and bounded from above by~$\nu^{-1}$.

\smallskip

After identifying the good window~$[m-2h,m]$, we start from scale~$m-2h$ and show that homogenization occurs before scale~$m$. The near constancy of~$\shom_{L,*}(\cu_k)$ makes the subadditivity inequality nearly an equality, which yields stronger decorrelation estimates. These scales lie below the infrared cutoff. Consequently, the fluctuations of~$\s_{L,*}(\cu_k)$ are small, and the matrices remain close to the deterministic scalar matrix~$\shom_{L,*}(\cu_m)$.

\smallskip

This allows us to commute the influence of the larger scale~$\mathbf{j}_k$'s with the coarse-graining, giving us a rigorous version of~\eqref{e.heuristic.step} on this limited range of scales. We are therefore able to lower bound~$\shom_{L,*}(\cu_m)$ by the expected diffusivity enhancement due to these waves~$\mathbf{j}_k$ on the operator~$\shom_{L,*}(\cu_m)\Delta$. This (roughly) yields the bound
\begin{equation}
	\label{e.enhance.subopt.intro}
	\shom_{L,*}(\cu_m) \gtrsim \shom_{L,*}^{-1}(\cu_m) h \,.
\end{equation}
The actual bound will have additional logarithmic factors, see~\eqref{e.crude.enhance.h.bnd}. Since~$h$ can be taken to be of order~$m/\log(\nu^{-1}m)$, we obtain, up to these logarithmic factors,
\begin{equation*}
	\shom_{L,*}^2 (\cu_m) \gtrsim m  \,.
\end{equation*}
This results in the bound~\eqref{e.sstar.lower.bound.annealed.intro}, and the choice of~$m$ can be removed by using the monotonicity of~$\shom_{L,*} (\cu_m)$ in~$m$, due to the subadditivity of~$\shom_{L,*}^{-1}$.

\smallskip

Implementing this argument requires the quantitative homogenization estimates for high-contrast coefficient fields established in~\cite{AK.HC}, applied to the infrared cutoff field~$\a_L$. The argument above does not directly yield~\eqref{e.enhance.subopt.intro}. Instead, it gives
\begin{equation}
	\label{e.enhance.subopt.intro.truth}
	\shom_{L}(\cu_m) \gtrsim \shom_{L,*}^{-1}(\cu_m) h \,.
\end{equation}
The estimates in~\cite{AK.HC} state that the relative homogenization error becomes small after approximately~$O( \log^2 m)$ many geometric scales. This tells us that, if we impose the constraint~$h \gtrsim \log^2 m$ on the width of the interval of scales, then we can ensure that~$\shom_{L}(\cu_m) \leq 2\shom_{L,*}(\cu_m)$. We then obtain~\eqref{e.enhance.subopt.intro} from~\eqref{e.enhance.subopt.intro.truth}. The constraint~$h \gtrsim \log^2 m$ is responsible for the additional logarithmic factors mentioned above.

\smallskip

\hyperref[s.homog.below.cutoff]{\emph{Step two}}: \emph{Quantitative homogenization on scales below the infrared cutoff}. With a lower bound for~$\s_{L,*}$ in hand, we turn our attention to the problem of lack of scale separation in Section~\ref{s.homog.below.cutoff} with the goal of formalizing the heuristic statement~\eqref{e.Ln.homog.approx.betterer}. That is, we want to prove that, across a range of scales strictly smaller than~$L$, the operator~$\nabla\cdot \a_L \nabla$ is close to~$\shom_L\Delta$, where~$\shom_L \coloneqq  \lim_{m\to \infty} \shom_{L,*}(\cu_m)$.

\smallskip

The precise statement formalizing the closeness of the operators~$\nabla\cdot \a_L \nabla$ to~$\shom_L\Delta$ is presented in Proposition~\ref{p.minimal.scales}. It says roughly that the coarse-grained matrices satisfy
\begin{equation}
	\label{e.rough.Prop41}
	\shom_m^{-1} \bigl| \s_L(\cu_m) - \shom_m  \bigr| + \shom_m^{-1} \bigl| \s_{L,*}(\cu_m) - \shom_m  \bigr|
	+
	\shom_m^{-1} \bigl| \k_L(\cu_m)-(\k_L)_{\cu_m} \bigr|
	\lesssim
	m^{-\nicefrac14 + \delta} \,.
\end{equation}
This statement does not take the form of a traditional homogenization estimate, in which solutions corresponding to the two operators are compared to one another. We can certainly obtain such a statement as a \emph{consequence}, and we will do so, but it is much more efficient to encode information about the operator~$\nabla\cdot \a_L \nabla$ in the coarse-grained matrices themselves. The exponent~$\nicefrac14$ appearing in~\eqref{e.rough.Prop41} is not optimal, and will be improved later.

\smallskip

Since, for each fixed~$L\in\N$, the field~$\a_L$ has a finite range of dependence, the quantitative homogenization results of~\cite{AK.HC} are applicable (see~\cite[Theorem B]{AK.HC}). However, since the range of dependence of~$\a_L$ is of order~$3^L$, these results will give us homogenization only on scales \emph{larger} than~$3^L$. On the other hand, we learn from the localization estimates in Section~\ref{ss.localization} that the operator~$\nabla \cdot\a_L\nabla$ is a very good approximation of~$\nabla \cdot \a\nabla$ on scales a bit smaller than~$3^L$. It is therefore on length scales strictly below the infrared cutoff that quantitative homogenization estimates would be most useful. Indeed, the main interest in an estimate like~\eqref{e.rough.Prop41} is when the scale parameter~$m$ lies in a sub-cutoff range~$m\in [L-C L^\alpha,L]$ with~$\alpha<\nicefrac12$.

\smallskip

The main difficulty in proving homogenization estimates below the scale of the infrared cutoff is due to the lack of good quantitative ergodicity of~$\a_L$ on these scales. On scales below~$L$, the field~$\a_L$ has essentially the same logarithmic correlations as the field without cutoff. The strategy is to use the lower bound for~$\shom_{L,*}$ in Proposition~\ref{p.sstar.lower.bound} and the special structure of the coarse-grained coefficients to argue that, notwithstanding the lack of correlation decay of~$\a_L$, the coarse-grained matrices still possess sufficient correlation decay on scales larger than~$3^{L - o(L)}$. In other words, \emph{coarse-graining reveals a hidden decorrelation structure.} This will allow us to apply the results of~\cite{AK.HC} to obtain the desired homogenization estimates, since these results are applicable under such weak mixing conditions.

\smallskip

\hyperref[s.sharp.asympt]{\emph{Step three}}: \emph{The approximate recurrence relation and sharp superdiffusivity}. The main goal of Section~\ref{s.sharp.asympt} is to obtain the approximate recurrence relation informally stated in~\eqref{e.intro.recurrence}. The proof proceeds by isolating the fresh shell~$\k_{m+h}-\k_m$, introducing the associated Dirichlet and Neumann response fields, localizing the block variational problem for the coarse-grained matrices, and then reducing the principal contribution from the new scale to the old scale. We evaluate the resulting deterministic quadratic form using the Neumann and Dirichlet response fields to obtain the upper and lower bounds, respectively. Together these two evaluations yield the approximate recurrence relation and, after iteration, the sharp asymptotics for~$\shom_m$.

\smallskip

\hyperref[s.sharp.C.one]{\emph{Step four}}: \emph{Large-scale regularity theory}. A bootstrap technique that is present in many works in quantitative homogenization is to use a weak or suboptimal homogenization estimate to gain an improvement of regularity, which is then used to obtain stronger (and often optimal) homogenization estimates. This idea is crucial to the present work, as explained in the next step.

\smallskip

In Section~\ref{s.sharp.C.one}, we combine the multiscale homogenization estimate of Proposition~\ref{p.minimal.scales} with the large-scale Caccioppoli and Poincar\'e inequalities to prove the sharp large-scale~$C^{1,\gamma}$ estimate, Theorem~\ref{t.C1beta}. The large-scale H\"older estimate is then proved in Section~\ref{s.Dirichlet}, where the right-hand side estimates are also incorporated.

\smallskip

The overall strategy for obtaining regularity statements from homogenization estimates is reviewed in~\cite[Chapter 3]{AKMBook}. Briefly, the idea is to approximate an arbitrary solution by harmonic functions at each scale, and thereby deduce that the solution enjoys the same oscillation decay estimates as harmonic functions do, up to the homogenization error. These estimates can then be iterated if the homogenization error is sufficiently small. To obtain H\"older~$C^{0,\gamma}$-type estimates, the iteration can be closed if the homogenization error is a small enough constant, which is ensured by the smallness assumption on the homogenization error.

\smallskip

The~$C^{1,\gamma}$ estimate is the critical one for the bootstrap argument in the paper. It is more involved than the H\"older estimate, and it differs substantially from the usual argument in uniformly elliptic homogenization theory. There one uses the homogenization estimates to obtain bounds on the flatness of corrected affines on all sufficiently large scales. Then, working with a modified notion of excess\footnote{The \emph{excess} of a given function~$u$ is usually defined as~$E(u,r)  \coloneqq  \inf_{\linear} \| u - \linear \|_{\underline{L}^2(B_r)}$, where the infimum is over all affines~$\linear$. It measures the relative distance between a given function and the nearest affine.} in which affines are replaced by corrected affines, we obtain excess decay estimates for general solutions from those of harmonic functions, once again up to the homogenization error. These can then be iterated to give the desired~$C^{1,\gamma}$-type bound.

\smallskip

Unlike the case of uniformly elliptic equations, we do not begin with stationary infinite-volume correctors whose slope is fixed across scales. We instead construct corrected affines in finite volume, and the main additional difficulty is to show that they are compatible across different scales. In other words, each finite-volume corrector must be close to an affine at every scale, with a scale-dependent slope. This estimate, which we call \emph{flatness at every scale}, is proved as part of an induction argument which gives a finite-volume version of the~$C^{1,\gamma}$ estimate. The finite-volume estimate then implies the Liouville theorem and allows us to construct the infinite-volume corrected-affine space used in the final~$C^{1,\gamma}$ statement.

\smallskip

\hyperref[s.Dirichlet]{\emph{Step five}}: \emph{Improved weak-flux and homogenization estimates}. The homogenization error obtained in Step~2 is \emph{larger} than the diffusivity enhancement observed in the recurrence relation~\eqref{e.intro.recurrence}. This is not surprising, because the localization estimate used in Step~2 essentially discards the wavelengths responsible for the enhancement. The recurrence and sharp asymptotics from Step~3 allow us to sharpen the homogenization estimates needed in the remainder of the proof.

\smallskip

The homogenization error is captured by the sizes of the differences~$|(\s-\s_*)(\cu)|$ as well as~$| \k (\cu) - \khom|$ and~$| \s_*(\cu) - \shom|$. However, the real bottleneck which limits the convergence rate is the first difference~$|(\s-\s_*)(\cu)|$, which represents the size of the ``confidence interval'' for the coarse-graining operation in~$\cu$. We call this the \emph{coarse-graining error}, and it is distinct from the other two errors, whose sum we call the \emph{fluctuation error}, which measures how close the coarse-grained matrices are to~$\ahom$.

\smallskip

The coarse-graining error is usually much smaller than the fluctuation error, in fact it is typically the square of the latter. The role of the coarse-graining error is fundamental in our approach as it is the basis of the rigorous renormalization group argument. We define, for each~$n\in\N$, the \emph{coarse-grained coefficient field~$\hat{\a}_{n}$ at length scale~$3^n$} by
\begin{equation}
	\label{e.CG.field.intro}
	\hat{\a}_{n} (x)
	\coloneqq  \sum_{z\in 3^n\Zd}
	\a_{*}(z+\cu_n) \indc_{z+\cu_n}(x)\,,
\end{equation}
where~$\a_*(\cu) \coloneqq  (\s_* - \k^t)(\cu)$. In fact, if we take a solution~$u$ of the equation and mollify it on scale~$r >3^n$ by considering~$u\ast\eta_r$, then~$u\ast\eta_r$ will be a solution of the coarse-grained equation
\begin{equation*}
	-\nabla \cdot \hat{\a}_n \nabla (u\ast\eta_r) = \nabla \cdot \bigl( \mathrm{Error} \bigr) \,,
\end{equation*}
where the divergence-form error on the right side is controlled---explicitly and deterministically---by the coarse-graining error. Therefore, if the coarse-graining error can be made small, we can \emph{literally} coarse-grain the equation by replacing~$\a(\cdot)$ by~$\hat{\a}_n(\cdot)$. This brings us very close to the informal heuristics in Section~\ref{ss.heuristic} and the recurrence formalized in Step~3.

\smallskip

In Step~4, Proposition~\ref{p.minimal.scales} is combined with the sharp recurrence and diffusivity asymptotics from Step~3 to obtain the estimates collected in Lemma~\ref{l.sharp.scale.inputs}. We use these estimates here to control the flux in a weak norm and obtain harmonic approximation for the Dirichlet problem.

\smallskip

As a result of this analysis, we improve the homogenization error estimate stated roughly in~\eqref{e.rough.Prop41} to the (still roughly stated) estimate
\begin{equation}
	\label{e.rough.Prop61}
	\shom_m^{-1} \bigl| \s_L(\cu_m) - \shom_m  \bigr| + \shom_m^{-1} \bigl| \s_{L,*}(\cu_m) - \shom_m  \bigr|
	+
	\shom_m^{-1} \bigl| \k_L(\cu_m)-(\k_L)_{\cu_m} \bigr|
	\lesssim
	m^{-\nicefrac12 + \delta} \,,
\end{equation}
with the coarse-graining error improved to
\begin{equation}
	\label{e.rough.Prop61.CG}
	\shom_m^{-1} \bigl| \s_L(\cu_m) - \s_{L,*}(\cu_m) \bigr|
	\leq
	m^{-1 + \delta} \,.
\end{equation}
These estimates, which are optimal up to the small~$\delta>0$, are stated in the multiscale form~\eqref{e.minscale.bounds.E} of Proposition~\ref{p.minimal.scales}. We use their consequences from Step~4 to improve the weak homogenization estimates for the Dirichlet problem.

\smallskip

For reference, the recurrence stated precisely in Proposition~\ref{p.one.step.sharp} is roughly
\begin{equation}
	\label{e.approximate.recurrence.intro}
	\bigl| \shom_{m+h} - \shom_{m}
	-
	\cstar (\log 3) \shom_{m} ^{-1} h
	\bigr|
	\lesssim
	\bigl( \log^2 m+\nondegconst \bigr)\shom_m^{-1}
	\,,
\end{equation}
where~$h\in\N\cap[1,\shom_m]$. Upon iteration of~\eqref{e.approximate.recurrence.intro} we obtain a quantitative version of~\eqref{e.intro.asymps}, which says that
\begin{equation*}
	\bigl| \shom_m  - ( 2 \cstar (\log 3) m )^{\nf12}\bigr| \lesssim \cstar^{-1}\bigl( \log^2 m+\nondegconst \bigr) \,.
\end{equation*}
See Theorem~\ref{t.sstar.sharp.bounds} for the precise statement.

\smallskip

\hyperref[s.Dirichlet]{\emph{Step six}}: \emph{Pointwise homogenization estimates}. The main purpose of Section~\ref{s.Dirichlet} is to upgrade the homogenization estimates from~$L^2$ to~$L^\infty$. This will allow us to complete the proof of Theorem~\ref{t.superdiffusivity}. Note that such an upgrade of spatial regularity is also needed in the next step to prove the invariance principle, Theorem~\ref{t.A}.
\smallskip

A common way to obtain pointwise homogenization estimates from~$L^2$ bounds is to obtain uniform bounds in a stronger space---such as~$C^{0,\gamma}$---and then interpolate~$L^\infty$ between~$L^2$ and~$C^{0,\gamma}$. In the case of uniformly elliptic equations, one can directly apply the De Giorgi-Nash H\"older estimate.

\smallskip

The argument is not so simple in our case, since the equation is not uniformly elliptic. We apply the large-scale regularity estimates to obtain a bound on the~$L^2$ oscillation of the solution~$u^\ep$ in every ball larger than~$\left| \log \ep \right| ^{-q}$ for any exponent~$q<\infty$. We take such a large mesoscopic scale because, in order to have such a regularity estimate, we need to cover the domain by a grid of balls~$B_{r/2}(x_i)$ which have radii~$r \geq \ep \X(x_i/\ep)$, where~$\X(z)$ is the minimal scale for the large-scale regularity estimates centered at the point~$z$. These estimates are similar to Theorem~\ref{t.large.scale.Holder}, but have~$L^2$ type spatial integrability rather than~$L^\infty$. We use a union bound and a bound on~$\X$ similar to~\eqref{e.large.scale.Holder.X} to estimate the probability, for a given~$r$, that such a covering is possible. The stochastic integrability of~\eqref{e.large.scale.Holder.X} is however very weak. There will be~$r^{-d}$ many balls, so for the union to yield something useful, we require that~$r^{-d} \P[\X > \ep^{-1} r ] \ll 1$. This is only possible if~$r\gg \exp( -( c \log \ep^{-1})^\sigma )$ for~$\sigma< 1$. Taking then~$r_\ep  \coloneqq  \left| \log \ep \right| ^{-q}$ is clearly fine, but we cannot take it to be as small as~$\ep^{0.999}$.

\smallskip

The result of this argument is an estimate of the form
\begin{equation}
	\label{e.Cgamma.uep.eta.yep}
	\bigl[ u^\ep \ast \eta_{r_\ep} \bigr]_{C^{0,\gamma}} \lesssim 1\,,
\end{equation}
where~$\{ \eta_r\}_{r>0}$ is the standard mollifier. In fact, we have~\eqref{e.Cgamma.uep.eta.yep} for~$\gamma=1$ for such~$r_\ep$, as follows from the large-scale Lipschitz estimate in Proposition~\ref{p.Dirichlet.C01.blackbox}, which states roughly that, in every ball~$B_r(x_i)$ with~$r \in [r_\ep,1]$,
\begin{equation*}
	\| u^\ep - (u^\ep)_{B_r(x_i)} \|_{\underline{L}^2(B_r(x_i))}
	\lesssim
	r
	\,.
\end{equation*}
We then use the De Giorgi-Nash~$L^\infty$-$L^2$ estimate with explicit prefactor depending on the ellipticity ratio\footnote{For symmetric coefficients in dimensions~$d\geq4$, Bella and Sch\"affner~\cite{BellaSchaff.Lipschitz} proved the optimal prefactor~$C(d)\Lambda^{\nf{d-1}4}$, where~$\Lambda$ is the ellipticity ratio; their result in dimension three is near optimal. We only use the standard prefactor~$C(d)\Lambda^{\nf d4}$, since the precise exponent is immaterial here.} to take care of the smaller scales. Since the ellipticity of our equation is at most of order~$\nu^{-1} \left| \log \ep \right| ^2$ (by another union bound), we obtain
\begin{equation*}
	\| u^\ep - (u^\ep)_{B_r(x_i)} \|_{L^\infty(B_{r/2}(x_i))}
	\lesssim
	\bigl( \nu^{-1} \left| \log \ep \right| ^2 \bigr)^{\frac{d}4}
	\| u^\ep - (u^\ep)_{B_r(x_i)} \|_{\underline{L}^2(B_r(x_i))}  \,.
\end{equation*}
When combined with the previous display, we obtain, in light of the choice of~$r$, that
\begin{equation}
	\label{e.Linfty.uep.yep}
	\bigl\| u^\ep \ast \eta_{r_\ep} - u^\ep \bigr\|_{L^\infty}
	\lesssim
	\bigl( \nu^{-1} \left| \log \ep \right| ^2 \bigr)^{\frac{d}4} r_\ep
	\leq
	\nu^{-\frac d4}
	\left| \log \ep \right| ^{-q + \frac d2}
	\lesssim
	\left| \log \ep \right| ^{-1000}
	\,,
\end{equation}
after taking~$q>\frac d2+1000$. The combination of~\eqref{e.Cgamma.uep.eta.yep} and~\eqref{e.Linfty.uep.yep} gives us the uniform equicontinuity of~$\{ u^\ep\}_{\ep>0}$ and allows us to upgrade the homogenization estimates from~$L^2$ to~$L^\infty$ and complete the proofs of Theorem~\ref{t.superdiffusivity} and~\ref{t.large.scale.Holder}.

\smallskip

\hyperref[s.invariance.principle]{\emph{Step seven}}: \emph{Consequences for the diffusion process~$\{ X_t\}$}. It is a basic fact that convergence of a sequence of Feller processes is equivalent to convergence of the corresponding infinitesimal generators: see for instance~\cite[Theorem 19.25]{Kallenberg}. More precisely, if~$X_t^\ep$ is a sequence of Feller processes with infinitesimal generators~$\Lop^\ep$ and~$X_t$ is a Feller process with generator~$\Lop$,
\begin{equation*}
	X_t^\ep \Rightarrow  X_t
	\iff
	\Lop^\ep \to \Lop
\end{equation*}
where, informally,~$\Lop^\ep \to \Lop$ means that, for every~$u$ in a core for~$\Lop$, there exists a sequence~$u^\ep$ in the domain of~$\Lop^\ep$ such that
\begin{equation*}
	u^\ep \to u \qand \Lop^\ep u^\ep \to \Lop u \,,
\end{equation*}
in the uniform norm on~$C_0(\Rd)$. In our setting, the latter statement can be deduced from Theorem~\ref{t.superdiffusivity}. Specifically, we let~$u^\ep$ be the solution of the problem
\begin{equation*}
	\Lop^\ep  u^{\ep} = \Lop u  \quad \text{in} \ \Rd
\end{equation*}
and obtain convergence of~$u^{\ep}$ to~$u$ by approximation via the Dirichlet problem with zero boundary conditions on a very large domain and applying Theorem~\ref{t.superdiffusivity}.

\smallskip

The remaining step needed in the proof of the first part of Theorem~\ref{t.A} is to verify that~$\{ X^\ep_t\}$ is Feller. This is, however, an immediate corollary of the generalized Nash-Aronson type upper bound established in the appendix (see  Corollary~\ref{c.yes.Feller}).

\smallskip

To prove the convergence of the diffusivity~$t^{-1}\mathbf{E}^0[|X_t|^2]$ stated in the second part of Theorem~\ref{t.A}, we use Theorem~\ref{t.superdiffusivity} to find a solution~$u^\ep$ of the equation
\begin{equation*}
	\Lop^\ep u^\ep = 1
\end{equation*}
which is very close to the quadratic function~$Q(x) = \frac1d |x|^{2}$. This allows us to compute, making a small error due to the difference~$\| u^\ep - Q\|_{L^\infty}$,
\begin{equation*}
	\frac1d \partial_t \mathbf{E}^0\bigl[ |X_{t}^\ep|^2 \bigr]
	=
	\partial_t \mathbf{E}^0\bigl[ Q (X_{t}^\ep) \bigr]
	\approx
	\partial_t \mathbf{E}^0\bigl[ u^\ep (X_{t}^\ep) \bigr]
	=
	\mathbf{E}^0 \bigl[ \Lop^\ep u^\ep (X_{t}^\ep) \bigr]
	=1\,.
\end{equation*}
Therefore, after integration, we deduce that
\begin{equation*}
	\frac1d \mathbf{E}^0\bigl[ |X_{t}^\ep|^2 \bigr]
	\approx t\,.
\end{equation*}
After undoing the scaling in~\eqref{e.Xep.t.def}, this gives the original-process estimate
\begin{equation*}
	\mathbf{E}^0\bigl[ |X_t|^2 \bigr]
	\approx
	2d \cstar^{\nf12} t(\log t)^{\nf12}
	\,.
\end{equation*}
The first-moment estimate is obtained from the same computation with the quadratic function~$Q$ replaced by an affine function. To make these arguments precise, we must work in a bounded domain. We do so by defining~$u^\ep$ to be the solution of the Dirichlet problem
\begin{equation}
	\label{e.DP.for.uep}
	\left\{
	\begin{aligned}
		& -\Lop^\ep u^\ep =  -1 & \mbox{in} & \ B_1 \,, \\
		& u^\ep = Q & \mbox{on} & \ \partial B_1 \,,
	\end{aligned}
	\right.
\end{equation}
and using a stopping time~$T^\ep_{B_1}$,  the first exit time of the process~$\{ X^\ep_t\}$ from the domain~$B_1$:
\begin{equation*}
	\partial_t \mathbf{E}^0 \bigl[ u^\ep (X_{t\wedge T^\ep_{B_1}}^\ep) \bigr]
	= \mathbf{E}^0 \bigl[ \Lop^\ep u^\ep (X_{t\wedge T^\ep_{B_1}}^\ep) \indc_{\{T^\ep_{B_1} >t\}}\bigr]
	=  \mathbf{P}^0  \bigl[ T^\ep_{B_1} >t\bigr]\,.
\end{equation*}
The previous computation can then be repeated as long as the first exit time~$T^\ep_{B_1}$ is larger than~$t$ with high probability, and we should expect this to hold for times~$t \ll 1$. Upon undoing the scaling, we obtain the estimates~\eqref{e.Dt.exp} and~\eqref{e.annealed.Dt} stated in Theorem~\ref{t.A} for the second moment of the displacement for all times~$t$. The estimates for the expectation of the displacement are obtained similarly. 

\smallskip

The exit-time estimate follows by converting the Dirichlet homogenization bounds into resolvent estimates and then into parabolic estimates. Iterating the parabolic bounds shows that the probability that~$X^\ep_t$ exits~$B_1$ before time~$(\log t)^{-\delta}$ is extremely small; see~\eqref{e.exit.time.crushed}.

\subsection{Notation}
The set of natural numbers is~$\N \coloneqq \{0, 1, 2, \ldots\}$, including zero. The Euclidean norm on~$\R^m$ is denoted by~$|\cdot|$. We sometimes write~$r\wedge s \coloneqq  \min\{ r,s\}$ and~$r\vee s \coloneqq  \max\{ r,s\}$. The H\"older conjugate exponent of an exponent~$p\in[1,\infty]$ is denoted by~$p'$, where~$p'=\infty$ if~$p=1$,~$p' \coloneqq  p(p-1)^{-1}$ if~$p\in(1,\infty)$ and~$p'=1$ if~$p=\infty$. We let~$\linear_e(x) = e \cdot x$ denote the linear function with slope~$e\in\Rd$. The distance between subsets~$A,B\subseteq\Rd$ is denoted by~$\dist(A,B) \coloneqq  \inf_{x\in A, y\in B} |x-y|$. The set of~$m$-by-$n$ matrices with real entries is denoted by~$\R^{m\times n}$. If~$B\in \R^{m\times n}$, then~$B^t$ is the transpose of~$B$. The~$n$-by-$n$ identity matrix is written~$\mathrm{I}_n$. The symmetric and anti-symmetric~$n$-by-$n$ matrices are denoted respectively by~$\R^{n\times n}_{\mathrm{sym}}$ and~$\R^{n\times n}_{\mathrm{skew}}$. We denote the Loewner ordering on~$\R^{n\times n}_{\mathrm{sym}}$ by~$\leq$; that is, if~$A,B\in \R^{n\times n}_{\mathrm{sym}}$ then~$A\leq B$ means that~$B-A$ has nonnegative eigenvalues. We let~$\R^{n\times n}_{\mathrm{sym},+}$ denote the set of~$n$-by-$n$ symmetric positive definite matrices. Unless otherwise indicated, the norm we use for~$\R^{m\times n}$, denoted by~$|A|$, is the square root of the largest eigenvalue of~$A^tA$. As a convention,~$L$ is usually reserved for an infrared cutoff,~$m,n,k$ for triadic spatial scales, and~$h$ for scale increments; local exceptions are declared where they occur. The Lebesgue measure of a (measurable) subset~$U \subseteq\Rd$ is denoted by~$|U|$. If~$V$ is a subset of~$\Rd$ of codimension~$1$, such as the boundary~$\partial U$ of a nice domain~$U$,  then~$|V|$ refers instead to the~$d-1$ dimensional Hausdorff measure of~$V$. We denote volume-normalized integrals and~$L^p$ norms for~$p\in[1,\infty)$ by
\begin{equation*}
	(f)_U  \coloneqq
	\fint_U f(x) \,dx  \coloneqq  \frac{1}{|U|} \int_U f(x)\,dx
	\qquad \mbox{and} \qquad
	\| f \|_{\underline{L}^p(U)} \coloneqq  \Bigl( \fint_U |f(x)|^p \,dx \Bigr)^{\nicefrac1p}
	\,.
\end{equation*}
We also put a slash through the sum symbol~$\sum$ to denote the average of a finite sequence. We denote by~$|A|$ the cardinality of a finite set~$A$ and, for every~$f:A \to \R$,
\begin{equation*}
	\avsum_{a \in A} f(a)  \coloneqq  \frac{1}{|A|} \sum_{ a\in A} f(a)\,.
\end{equation*}
We denote indicator functions---both for events and for subsets of~$\Rd$---using the symbol~$\indc$.

\smallskip

The function spaces we use include the standard H\"older spaces~$C^{k,\alpha}(U)$ for~$k\in\N$,~$\alpha \in (0,1]$ and a domain~$U\subseteq\Rd$, as well as Sobolev spaces, which are denoted by~$W^{s,p}(U)$ for~$s\in \R$ and~$p\in[1,\infty]$. The fractional Sobolev spaces are defined in~\cite[Appendix B]{AKMBook}, and we use the notation from this appendix (which we do not repeat here). When~$p = 2$ and~$s > 0$, we denote~$W^{s, p}$ by~$H^s$.  For most of the paper, we use the classical space~$W^{1,p}(U)$. For~$p\in[1,\infty)$, the norm is defined by
\begin{equation*}
	\| f \|_{{W}^{1,p}(U)}
	\coloneqq
	\Bigl( \| \nabla f \|_{{L}^p(U)}^p+ \| f \|_{{L}^p(U)}^p \Bigr)^{\frac1p} \,.
\end{equation*}
For~$p=\infty$, we use the usual norm~$\|f\|_{W^{1,\infty}(U)}\coloneqq \|\nabla f\|_{L^\infty(U)}+\|f\|_{L^\infty(U)}$. If~$|U|<\infty$ and~$p\in[1,\infty)$, then the volume-normalized norm~$\| f \|_{\underline{W}^{1,p}(U)}$ is defined by
\begin{equation*}
	\| f \|_{\underline{W}^{1,p}(U)}
	\coloneqq
	\Bigl( \| \nabla f \|_{\underline{L}^p(U)}^p+ |U|^{-\frac pd} \| f \|_{\underline{L}^p(U)}^p \Bigr)^{\frac1p} \,.
\end{equation*}
We also write~$[ f ]_{\underline{W}^{1,p}(U)}\coloneqq \|\nabla f\|_{\underline{L}^p(U)}$ for the normalized first-order seminorm.

\smallskip

The negative, dual seminorms are defined by
\begin{equation*}
	\bigl[ f \bigr]_{\underline{W}^{-1,p'}(U)} \coloneqq
	\sup
	\biggl\{
	\fint_U f(x) g(x) \,dx \,:\,
	g\in C^\infty_c(U),\ [ g ]_{\underline{W}^{1,p}(U)} \leq 1 \biggr\}
	\,,
\end{equation*}
and
\begin{equation*}
	\bigl[ f \bigr]_{\Wminusul{-1}{p'}(U)} \coloneqq
	\sup
	\biggl\{
	\fint_U f(x) g(x) \,dx \,:\,
	g\in C^\infty(U),\ [g]_{\underline{W}^{1,p}(U)} \leq 1\,, \ (g)_U=0 \biggr\}\,.
\end{equation*}

If~$p=p'=2$, then we write~$H^{-1}$ in place of~$W^{-1,p}$. For vector- and matrix-valued fields, these definitions are understood componentwise. For instance,~$\|\mathbf F\|_{\Hminusul(U)}$ is the supremum of~$\fint_U \mathbf F\cdot\mathbf g$ over~$\mathbf g\in C^\infty(U;\Rd)$ with~$(\mathbf g)_U=0$ and~$\|\nabla\mathbf g\|_{\underline L^2(U)}\leq1$. We let~$W^{1,p}_0(U)$ denote the closure of~$C^\infty_c(U)$ in~$W^{1,p}(U)$ with respect to the norm~$\| \cdot \|_{W^{1,p}(U)}$. If~$X(V)$ is a function space defined for every domain~$V\subseteq\Rd$, then~$X_{\mathrm{loc}}(U)$ denotes the set of functions~$v$ on~$U$ such that~$v|_V\in X(V)$ for every domain~$V\Subset U$. We let~$C_0(\Rd)$ denote the space of continuous functions~$u:\Rd\to\R$ such that~$\lim_{|x|\to \infty} u(x)=0$, and~$C_c^k(\Rd)$ denotes the subspace of~$C^k(\Rd)$ with compact support in~$\Rd$.

\smallskip

We keep track of the stochastic integrability of our random variables with the~$\O_\Psi(\cdot)$ notation defined in Section~\ref{ss.bigO}. Throughout, for~$\sigma \in (0,\infty)$ we denote~$\Gamma_\sigma(t) \coloneqq  \exp (t^\sigma)$ as defined in~\eqref{e.Gamma.sigma}. The bold symbol~$\gammafun$ is used to denote the gamma function~$\gammafun(s)  \coloneqq  \int_0^\infty t^{s-1} \exp(-t) \,dt$.

\smallskip

For quick reference, the most frequently used symbols are:
\begin{center}
\begin{tabular}{ll}
$\k$ & the stream matrix field \\
$\a=\nu\Id+\k$ & the coefficient field \\
$\s(U),\s_*(U),\k(U)$ & coarse-grained matrices in a domain~$U$ \\
$\shom_m$ & annealed scalar diffusivity at infrared cutoff~$m$ \\
$\mathcal E_{s,q},\CE_s,\FE_s$ & multiscale homogenization-error quantities \\
$\X,\mathcal Z$ & random minimal scales \\
$\O_\Psi(\theta)$ & stochastic bound with Orlicz function~$\Psi$
\end{tabular}
\end{center}

\section{Coarse-graining estimates}
\label{s.coarse.graining}

In this section we introduce the main objects in our approach to renormalization, namely the coarse-grained coefficient fields. These quantities are not new, and have been used extensively in the theory of quantitative homogenization; see~\cite{AKMBook} for the large-scale regularity background and~\cite[Chapter~5]{AK.Book} for the variational formulation used here.

\smallskip

In a very recent paper~\cite{AK.HC}, two of the authors developed a formalization of the renormalization group in the context of elliptic homogenization, with an analysis based on these coarse-grained coefficients. In particular, in that paper we measure ellipticity in a negative regularity space, and this allows for the renormalization of the ellipticity ratio. Concretely, for solutions on large scales, we are able to make use of elliptic estimates on a coarse-grained level---we can use the renormalized diffusion matrices rather than the microscopic diffusivity. The analysis in the present paper makes critical use of this idea.

\smallskip

We begin our discussion in the next subsection by introducing the reader to the coarse-grained matrices for very general coefficient fields, specializing to the setting of Theorem~\ref{t.superdiffusivity} in Subsection~\ref{ss.infrared}.

\smallskip

There are many equivalent ways of defining and of thinking about coarse-grained fields, and they have many interesting algebraic properties. These are presented in a complete and self-contained way in~\cite{AK.Book,AK.HC}. To avoid repetition, here we will summarize the properties that are needed while referring to those papers for many of the proofs.

\subsection{Definition of the coarse-grained matrices}
\label{ss.bfA.def}

Consider a general coefficient field~$\a:\Rd \to \R^{d \times d}$ and write the symmetric and anti-symmetric parts of~$\a$, respectively, as
\begin{equation*}
	\s \coloneqq  \frac12 (\a+\a^t) \qquad \mbox{and} \qquad \k \coloneqq  \frac12 ( \a - \a^t)\,,
\end{equation*}
where~$\a^t$ denotes the transpose of~$\a$. Although the coarse-grained matrices can be defined for more general coefficient fields (see~\cite{AK.HC}), in this paper we work under the assumption that~$\a(\cdot)$ is qualitatively uniformly elliptic on bounded subsets of~$\Rd$. This means that~$\s$ is valued in the set of positive definite matrices and~$\s^{-1},\s,\k\in L^\infty_{\mathrm{loc}}(\Rd)$. Associated to the field~$\a$ is another field~$\bfA$ which is valued in the set~$\R^{2d\times 2d}_{\mathrm{sym}}$ of~$2d$-by-$2d$ symmetric matrices and given by
\begin{equation}
	\label{e.bfA.def}
	\bfA(x)  \coloneqq
	\begin{pmatrix}
		( \s + \k^t\s^{-1}\k )(x)
		& -(\k^t\s^{-1})(x)
		\\ - ( \s^{-1}\k )(x)
		& \s^{-1}(x)
	\end{pmatrix}
	\,.
\end{equation}
The field~$\bfA$ arises naturally in the variational formulation of the equation~$-\nabla\cdot \a\nabla u= 0$, and the variational perspective is helpful in gaining an intuition for coarse-graining.

\smallskip

For each such~$\a(\cdot)$ and every bounded Lipschitz domain~$U\subseteq\Rd$, we define three matrices which we denote by~$\s(U)$,~$\s_*(U)$ and~$\k(U)$. The matrices~$\s(U)$ and~$\s_*(U)$ are symmetric, and we think of them as representing, respectively, upper and lower bounds for the symmetric part of the coarse-grained matrix. We always have the order~$\s(U) \geq \s_*(U)$, and we think of the gap between~$\s(U)$ and~$\s_*(U)$ as representing ``uncertainty'' in the coarse-graining procedure. The matrix~$\k(U)$ is not necessarily anti-symmetric, but its symmetric part is bounded by the size of~$(\s-\s_*)(U)$ and therefore small if coarse-graining is working well.

\smallskip

We also arrange these three matrices in a pair of~$2d$-by-$2d$ matrices which we denote by
\begin{equation}
	\label{e.bigA.def}
	\bfA(U)
	\coloneqq
	\begin{pmatrix}
		( \s + \k^t\s_*^{-1}\k )(U)
		& -(\k^t\s_*^{-1})(U)
		\\ - ( \s_*^{-1}\k )(U)
		& \s_*^{-1}(U)
		\end{pmatrix}\,,
	\quad
	\bfA_*(U)
	=
	\begin{pmatrix}
		(\s_* + \k \s^{-1}\k^t )(U) & (\k \s^{-1})(U) \\ (\s^{-1}\k^t )(U)
		& \s^{-1}(U)
		\end{pmatrix}\,.
\end{equation}
We consider these matrices to be a coarse-graining of the field~$\bfA$. We have that~$\bfA_*(U) \leq \bfA(U)$ and the difference~$\bfA(U) -  \bfA_*(U)$ is controlled by~$(\s-\s_*)(U)$, once again representing the uncertainty or error in the coarse-graining. We will often need to refer to the top left~$d$-by-$d$ block of~$\bfA(U)$, so we denote this matrix by
\begin{equation}
	\b(U) \coloneqq  ( \s + \k^t\s_*^{-1}\k )(U)
	\,.
\end{equation}
By~\eqref{e.CG.bounds.1}, the matrix~$\b(U)$ satisfies~$\b(U)\geq\s_*(U)>0$, and hence~$\b^{-1}(U)$ is well defined. The matrices~$\bfA(U)$ and~$\bfA_*(U)$ evidently contain the same information as the triple~$(\s,\s_*,\k)(U)$. It is however helpful to have both perspectives in mind. The matrices~$(\s,\s_*,\k)(U)$ are often more intuitive, but on the other hand several important algebraic properties are best written in terms of~$\bfA(U)$. For instance, as shown below,~$\bfA(U)$ is subadditive, while~$\s(U)$ and~$\k(U)$ are not.

\smallskip

There are several equivalent ways to define these coarse-grained matrices. The first way is to define~$\bfA(U)$ by the variational formula
\begin{equation}
	\label{e.J.P0.Dirichlet}
	\frac12 P \cdot \bfA(U) P
	=
	\inf\Bigl\{
	\fint_{U}
	\frac12 (X + P) \cdot \bfA (X + P)
	\, : \,
	X \in  \Lpoto(U) \times  \Lsolo(U)
	\Bigr\}
	\,, \quad P \in \R^{2d}\,,
\end{equation}
where~$\Lpoto(U)$ denotes the set of gradients with zero boundary trace and~$\Lsolo(U)$ denotes the set of solenoidal fields with vanishing normal trace on~$\partial U$; that is,
\begin{equation*}
	\Lpoto(U) \coloneqq  \left\{ \nabla u \,:\,  u \in H^1_0(U) \right\}\,,
	\quad
	\Lsolo(U) \coloneqq  \Bigl\{ \g \in L^2(U;\Rd) \,:\,  \forall \phi\in H^1(U)  \,, \ \int_U \g\cdot \nabla \phi\,dx = 0 \Bigr\}\,.
\end{equation*}
The right side of~\eqref{e.J.P0.Dirichlet} is clearly quadratic in~$P$, and therefore there exists a unique symmetric matrix~$\bfA(U)$ such that the equality in~\eqref{e.J.P0.Dirichlet} holds, and this defines~$\bfA(U)$. Having defined~$\bfA(U)$ in this way, we can define the matrices~$\s_*(U)$,~$\k(U)$ and~$\s(U)$---in that order---by giving names to the various block entries of~$\bfA(U)$.

\smallskip

An alternative way to define the coarse-grained matrices is to define, for each~$p,q\in\Rd$, another variational quantity~$J(U,p,q)$ by
\begin{equation}
	\label{e.J.def}
	J(U,p,q)
	\coloneqq
	\sup_{u \in \A(U)}
	\fint_{U} \Bigl( - \frac12 \nabla u \cdot \s \nabla u - p \cdot \a \nabla u + q \cdot \nabla u \Bigr)
	\,,
\end{equation}
where~$\A(U)$ denotes the set of solutions to the equation~$\nabla\cdot\a\nabla u=0$ in~$U$; that is,
\begin{equation*}
	\A(U)  \coloneqq
	\bigl\{
	u \in H^1_{\mathrm{loc}}(U) \,:\,
	\nabla \cdot \a\nabla u = 0 \ \mbox{in} \ U
	\bigr\}
	\,.
\end{equation*}
We also define the analogue of this quantity for the adjoint equation by
\begin{equation}
	\label{e.Jstar.def}
	J^*(U,p,q)
	\coloneqq
	\sup_{u \in  \A^*(U) }
	\fint_{U} \Bigl( - \frac12 \nabla u \cdot \s \nabla u - p \cdot \a^t \nabla u + q \cdot \nabla u \Bigr)
\end{equation}
where
\begin{equation*}
	\A^*(U)  \coloneqq
	\bigl\{ u \in H^1_{\mathrm{loc}}(U) \, : \, \nabla \cdot \a^t \nabla u = 0 \ \mbox{in} \ U \bigr\}
\end{equation*}
denotes the set of solutions to the adjoint equation in the domain~$U$. The suprema in the variational problems on the right sides of~\eqref{e.J.def} and~\eqref{e.Jstar.def} are achieved, and the maximizers belong to~$H^1(U)$ and are unique up to additive constants; see~\cite[Lemma 5.2]{AK.Book}. Throughout the paper we denote them by~$v(\cdot,U,p,q)$ and~$v^*(\cdot,U,p,q)$, respectively.

\smallskip

Having defined~$J(U,p,q)$, we can then define~$\s(U), \s_*(U)\in \R^{d\times d}_{\mathrm{sym}}$ and~$\k(U)\in \R^{d\times d}$ in such a way that
\begin{equation}
	\label{e.J.mat}
	J(U,p,q) =
	\frac 12p \cdot \s(U)p
	+ \frac 12 (q+\k(U) p) \cdot \s_*^{-1}(U) (q+\k(U) p)
	- p \cdot q \,.
\end{equation}
Duality gives the identity (see~\cite[Lemma~5.2]{AK.Book}):
\begin{equation}
	\label{e.J.mat.star}
	J^*(U,p,q) =
	\frac 12p \cdot \s(U)p
	+ \frac 12 (q-\k(U) p) \cdot \s_*^{-1}(U) (q-\k(U) p)
	- p \cdot q \,.
\end{equation}
If we then define~$\bfA(U)$ by~\eqref{e.bigA.def}, then the above identities become
\begin{equation}
	\label{e.Jaas.matform}
	J(U,p,q)
	=
	\frac 12
	\begin{pmatrix}
		-p \\ q
	\end{pmatrix}
	\cdot \bfA(U)
	\begin{pmatrix}
		-p \\ q
	\end{pmatrix}
	-p\cdot q
	\quad \mbox{and} \quad
	J^*(U,p,q)
	=
	\frac 12
	\begin{pmatrix}
		p \\ q
	\end{pmatrix}
	\cdot \bfA(U)
	\begin{pmatrix}
		p \\ q
	\end{pmatrix}
	-p\cdot q\,.
\end{equation}
This implies that
\begin{equation}
	\label{e.bfA.by.J}
	\begin{pmatrix} p   \\ q\end{pmatrix}\cdot \bfA(U) \begin{pmatrix} p   \\ q\end{pmatrix}
	=
	J(U,p,-q) + J^*(U,p,q)
	\,.
\end{equation}
These two definitions of the coarse-grained matrices given above are equivalent. A proof of this purely algebraic fact can be found for instance in~\cite[Section 5]{AK.Book}.

\smallskip

The quadratic form~$(p,q) \mapsto J(U,p,q)$ is therefore yet another way to represent the coarse-grained matrices. Its usefulness is due to the variational form~\eqref{e.J.def}, from which we can quickly derive many important properties, as we will see in the next subsection.

\smallskip

While the coarse-grained matrices clearly depend on the underlying coefficient field~$\a(\cdot)$, we usually suppress this dependence from the notation. However, in this paper we need to consider several different coefficient fields (for instance infrared cutoffs of~$\a$ in the context of Theorem~\ref{t.superdiffusivity}, as well as renormalized coefficient fields) and so it is necessary to make this dependence explicit in some cases, which we do by writing~$\s(U;\a)$,~$\s_*(U;\a)$,~$\k(U;\a)$,~$\bfA(U;\a)$,~$J(U,p,q;\a)$ and so forth.

\smallskip

One property of the coarse-grained matrices which is very important for the analysis in this paper is the \emph{commutativity of coarse-graining with the addition of constant anti-symmetric matrices}. We first observe that the set of solutions~$\A(U)$ of the equation does not change if we add a constant anti-symmetric matrix~$\k_0 \in \R^{d \times d}$ to the field~$\a$, since~$\nabla \cdot (\a +\k_0) \nabla u = \nabla \cdot \a \nabla u$. That is,~$\A(U;\a) = \A(U;\a+\k_0)$. Indeed, we may even consider that the field~$\a$ is only defined modulo the addition of an anti-symmetric matrix. This invariance is inherited by the coarse-grained matrices. Indeed, using this and~\eqref{e.J.def} we see immediately that
\begin{equation*}
	J(U,p,q;\a+\k_0) = J (U,p,q+\k_0p;\a) \,.
\end{equation*}
In terms of the matrices, we have
\begin{equation}
	\label{e.commute.k0}
	\s(U;\a+\k_0) = \s(U;\a) \,, \quad
	\s_*(U;\a+\k_0) = \s_*(U;\a) \,, \quad \mbox{and} \quad
	\k(U;\a+\k_0) = \k(U;\a) + \k_0\,,
\end{equation}
and
\begin{align}
	\bfA(U;\a+\k_0)
	&=
	\begin{pmatrix}
		( \s + (\k+\k_0)^t\s_*^{-1}(\k+\k_0) )(U)
		& -((\k+\k_0)^t\s_*^{-1})(U)
		\\ - ( \s_*^{-1}(\k+\k_0) )(U)
		& \s_*^{-1}(U)
	\end{pmatrix}
	\notag \\
	&=
	\G_{-\k_0}^t
	\bfA(U;\a)
	\G_{-\k_0}\,,
	\label{e.commute.coarse.grained.k0}
\end{align}
where we define the matrix~$\G_{\h}$ as
\begin{equation}
	\label{e.G}
	\G_{\h}  \coloneqq  \begin{pmatrix} \Id & 0 \\ \h & \Id \end{pmatrix}\,.
\end{equation}
In other words, adding a constant anti-symmetric matrix~$\k_0$ does not change~$\s(U)$ or~$\s_*(U)$, and it merely adds~$\k_0$ to~$\k(U)$.

\smallskip

We combine~$J$ and~$J^*$ into a single quantity by defining
\begin{equation}
	\label{e.bfJ}
	\bfJ
	\biggl(U, \begin{pmatrix} p  \\ q \end{pmatrix}, \begin{pmatrix} q^* \\ p^* \end{pmatrix} \biggr)
	\coloneqq
	\frac12 J\bigl(U,p-p^*,q^*-q\bigr)
	+
	\frac12 J^*\bigl(U,p^*+p,q^*+q\bigr)
	\,.
\end{equation}

\subsection{Basic properties of the coarse-grained matrices}

We list here (without proof) some of the important properties of the coarse-grained matrices. Proofs can be found in~\cite{AK.Book} or~\cite{AK.HC}.

\smallskip

The coarse-grained matrices are bounded by integrals of the field itself: for every bounded Lipschitz domain~$U\subseteq \Rd$, we have that
\begin{equation}
	\label{e.CG.bounds.1}
	\Bigl( \fint_U \s^{-1}(x)\, dx \Bigr)^{\!-1}
	\leq \s_*(U) \leq \s(U)
	\leq
	(\s + \k^t\s_*^{-1} \k )(U)
	\leq
	\fint_U (\s + \k^t\s^{-1} \k )(x)\, dx
	\,.
\end{equation}
The first inequality is the harmonic-mean lower bound, obtained from the dual variational characterization by testing with constant fluxes. More generally,
\begin{equation}
	\label{e.CG.bounds.2}
	\Bigl( \fint_U\bfA^{-1}(x)\,dx \Bigr)^{\!-1}\leq \bfA_*(U)\leq \bfA(U)\leq\fint_U\bfA(x)\,dx\,.
\end{equation}
A more general form of the last inequality in~\eqref{e.CG.bounds.2} is the subadditivity of~$\bfA(U)$, which states that, for every partition~$\{ U_i \}_{i=1}^N$ of~$U$ (up to a Lebesgue null set) we have that
\begin{equation}
	\label{e.subadditivity}
	\bfA(U) \leq
	\sum_{i=1}^N \frac{|U_i|}{|U|} \bfA(U_i) \,.
\end{equation}
Equivalently, the quantity~$J(U,p,q)$ is subadditive. By Young's inequality, if~$\s_1$,~$\s_2$ are symmetric matrices with~$\s_2 \geq 0$ and~$\k$ is another matrix, then
\begin{equation}
	\label{e.how.to.upbound.A}
	\begin{pmatrix}
		\s_1 + \k^t\s_2 \k
		& -\k^t\s_2
		\\ - \s_2\k
		& \s_2
	\end{pmatrix}
	\leq
	\begin{pmatrix}
		\s_1 + 2\k^t\s_2\k
		& 0
		\\ 0
		& 2\s_2
	\end{pmatrix}
	\,.
\end{equation}
We will use this to upper bound matrices like~$\bfA(x)$ and~$\bfA(U)$ by a block diagonal matrix.

\smallskip

The first variation for the optimization problem~\eqref{e.J.def} asserts that, for every~$w\in \A(U)$,
\begin{equation}
	\label{e.firstvar}
	q\cdot \fint_U \nabla w - p \cdot \fint_U \a \nabla w
	=
	\fint_U \nabla w \cdot \s \nabla v(\cdot,U,p,q)\,.
\end{equation}
The second variation, or quadratic response identity, says that, for every~$w\in \A(U)$,
\begin{align}
	\lefteqn{
	J(U,p,q)
	- \fint_U \Bigl  ( -\frac12 \nabla w \cdot \s\nabla w -p\cdot \a\nabla w+ q\cdot \nabla w   \Bigr  )
	} \qquad &
	\notag \\ &
	=
	\fint_U \frac12 \bigl ( \nabla v(\cdot,U,p,q) - \nabla w \bigr )\cdot \s\bigl ( \nabla v(\cdot,U,p,q) - \nabla w \bigr )\,.
	\label{e.secondvar}
\end{align}
It follows that~$J$ can be written as the energy of the maximizer: for every~$p,q\in\Rd$,
\begin{equation}
	\label{e.Jenergy.v}
	J(U,p,q) = \fint_U \frac12 \nabla v(\cdot,U,p,q) \cdot \s \nabla v(\cdot,U,p,q)
	\,.
\end{equation}
Similarly, we have that
\begin{equation}
	\label{e.J.by.lin}
	J(U,p,q) = \frac12 \Bigl( q\cdot \fint_U \nabla v(\cdot,U,p,q) - p \cdot \fint_U \a \nabla v(\cdot,U,p,q) \Bigr) \,.
\end{equation}
By summing~\eqref{e.J.mat} and~\eqref{e.J.mat.star}, we obtain the identity
\begin{align}
	J(U,p,q-h) + J^*(U,p,q+h)
	&=
	p \cdot (\s -\s_*)(U)p
	+
	\bigl| \s_*^{-\nf12}(U)\bigl(q - \s_*(U)p\bigr) \bigr|^2
	\notag \\ & \qquad +
	\bigl| \s_*^{-\nf12}(U)\bigl(h-\k(U)p\bigr) \bigr|^2
	\,.
	\label{e.JJstar1}
\end{align}
In particular,
\begin{equation}
	\label{e.JJstar.gap}
	J(U,e,(\s_*-\k)(U) e ) + J^*(U,e,(\s_*+\k)(U) e)
	=
	e \cdot (\s -\s_*)(U)e\,.
\end{equation}
By~\cite[Lemma 5.2]{AK.Book}, we have that the symmetric part of~$\k$ is controlled, in the Loewner sense, by the gap between~$\s(U)$ and~$\s_*(U)$:
\begin{equation}
	\label{e.symm.part.k}
	\pm(\k+\k^t)(U) \leq (\s-\s_*)(U)
	\,.
\end{equation}

\smallskip

The quantity~$J$ allows us to relate the spatial averages of gradients and fluxes of arbitrary solutions: by~\eqref{e.firstvar}, we have that, for every~$p,q\in\Rd$ and~$w \in \A(U)$,
\begin{equation}
	\label{e.fluxmaps}
	\biggl | \fint_{U} \bigl ( p \cdot \a \nabla w - q \cdot \nabla w \bigr ) \biggr |
	=
	\biggl | \fint_U \nabla w
	\cdot  \s \nabla v\bigl (\cdot, U, p,q \bigr )  \biggr |
	\leq
	(2J \bigl (U, p,q \bigr ) )^{\sfrac12}
	\Bigl( \fint_U \nabla w \cdot \s \nabla w \Bigr)^{\!\sfrac12}
	\,.
\end{equation}
By taking~$q=(\s_*-\k)(U)p$, using~\eqref{e.JJstar.gap}, and then replacing the test vector~$p$ by an arbitrary unit vector, we obtain, for every~$w\in\A(U)$,
\begin{equation}
	\label{e.gradtofluxcg}
	\biggl| \fint_{U} \a \nabla w - (\s_*-\k^t)(U) \fint_U \nabla w \biggr |
	\leq
	2^{\sfrac12} \bigl| \s(U)-\s_*(U)\bigr|^{\sfrac12}
	\Bigl( \fint_U \nabla w \cdot \s \nabla w \Bigr)^{\!\sfrac12}
	\,.
\end{equation}
The coarse-grained matrix~$\s_*(U)$ gives a lower bound for the energy of a solution in terms of its spatially averaged gradient:
\begin{equation}
	\label{e.energymaps.nonsymm}
	\frac12\Bigl( \fint_U \nabla u \Bigr) \cdot \s_*(U) \Bigl( \fint_U \nabla u \Bigr)
	\leq
	\fint_U \frac12 \nabla u \cdot \s\nabla u
	\,, \quad
	\forall u\in \A(U)\,.
\end{equation}
Similarly, the coarse-grained matrix~$\b(U)$ gives a lower bound for the energy of a solution in terms of its spatially averaged flux:
\begin{equation}
	\label{e.energymaps.nonsymm.flux}
	\frac12\Bigl( \fint_U \a \nabla u \Bigr) \cdot \b^{-1} (U) \Bigl( \fint_U \a \nabla u \Bigr)
	\leq
	\fint_U \frac12 \nabla u \cdot \s\nabla u
	\,, \quad
	\forall u\in \A(U)\,.
\end{equation}

\smallskip

For~$p,q \in \R^d$ and a bounded Lipschitz domain~$U \subseteq \Rd$ denote by~$v(\cdot, U, p, q)$ the maximizing solution in~\eqref{e.J.def}. The spatial averages of the gradient and flux can be written explicitly in terms of the coarse-grained matrices: by~\cite[Lemma 5.1]{AK.Book}, we have
\begin{equation}
	\label{e.v.spatial.averages}
	\left\{
	\begin{aligned}
		& \fint_U
		\nabla v(\cdot,U,p,q)
		=
		- p + \s_{*}^{-1}(U) (q + \k(U) p)
		\\ &
		\fint_U
		\a \nabla v(\cdot,U,p,q)
		=
		(\Id -\k^t  \s_{*}^{-1})(U)  q - \b(U) p\,.
	\end{aligned}
	\right.
\end{equation}
We also denote
\begin{equation}
	\label{e.v.oneslot.def}
	v(\cdot,U,e)  \coloneqq  v\bigl(\cdot,U,0, \s_{*} (U) e  \bigr)
	\,,
\end{equation}
so that by~\eqref{e.v.spatial.averages} we have that
\begin{equation}
	\label{e.solution.avg.grad.flux.identity}
	\begin{pmatrix}
		\bigl(\nabla v(\cdot,U,e) \bigr)_{U}  \\
		\bigl( \a   \nabla v(\cdot,U,e) \bigr)_{U}
	\end{pmatrix}
	=
	\begin{pmatrix}
		e  \\ \bigl( \s_*(U) - \k^t(U)  \bigr) e
	\end{pmatrix}
	\,.
\end{equation}
By~\eqref{e.Jenergy.v}, the energy can be expressed as
\begin{equation}
	\label{e.solution.p.q.avg.energy}
	\fint_{U} \frac12 \nabla v(\cdot,U,e) \cdot \s \nabla v(\cdot,U,e)  =  J(U, 0, \s_*(U) e) = \frac 12 e \cdot \s_*(U) e  \,.
\end{equation}
It is clear that the map~$e \mapsto \nabla v(\cdot,U,e)$ is linear. We also define
\begin{equation*}
	\a_*(U)  \coloneqq  \s_*(U) - \k^t(U) \,.
\end{equation*}
Thus~$\a_*(U)e$ is the average flux of the maximizer~$v(\cdot,U,e)$. For each~$n \in \N$, we introduce the coarse-grained coefficient field~$\hat{\a}_{n}$ defined by
\begin{equation}
	\label{e.CG.field}
	\hat{\a}_{n}
	\coloneqq  \sum_{z\in 3^n\Zd}
	\a_{*}(z+\cu_n) \indc_{z+\cu_n}\,.
\end{equation}
This is a piecewise constant, measurable coefficient field; its primal and dual ellipticity bounds are those encoded by the block matrices~$\bfA_*(z+\cu_n)$.

\subsection{Coarse-grained ellipticity constants and functional inequalities}
\label{ss.cg.ellipticity}

We record the multiscale quantities from~\cite{AK.HC} which enter the coarse-grained Poincar\'e and Caccioppoli inequalities. Write~$\css{r}\coloneqq 1-3^{-r}$. Unless an endpoint is explicitly displayed, the Sobolev exponent in this subsection is taken in~$(0,1)$. Given~$s \in (0,1]$,~$q\in [1,\infty)$,~$m \in \Z$, a coefficient field~$\a$ on~$\cu_m$, and a constant matrix~$\a_1 \in \R^{d\times d}$ with positive definite symmetric part, let~$\bfA_1$ denote the constant matrix associated to~$\a_1$ by the formula~\eqref{e.bfA.def}. Although~$\a_1$ need not be symmetric, the associated block matrix~$\bfA_1$ is symmetric positive definite. We then define
\begin{equation}
	\mathcal{E}_{s,q} (\cu_m ; \a, \a_1 )
	\coloneqq
	\biggl(
	\css{sq}
	\sum_{l=-\infty}^m
	\Bigl( 3^{-2s(m-l)} \!\!
	\max_{z \in 3^l\Zd \cap \cu_m}
	\max_{\substack{e\in\R^{2d}\\ |e|=1}}  \bfJ\bigl(z {+} \cu_l,\bfA_1^{-\nf 12} e, \bfA_1^{\nf 12} e \, ; \a \bigr) \Bigr)^{\nf q2}
	\biggr)^{\!\nf1q}
	\,.
	\label{e.mathcal.E.def.new}
\end{equation}
Equivalently, if
\begin{equation*}
	\bfA_1^{-\nf 12} e = \begin{pmatrix} p \\ q \end{pmatrix}
	\qquad\mbox{and}\qquad
	\bfA_1^{\nf 12} e = \begin{pmatrix} q^* \\ p^* \end{pmatrix}\,,
\end{equation*}
then by~\eqref{e.bfJ},
\begin{equation}
	\label{e.mathcal.E.bfJ.decomp}
	\bfJ\bigl(U,\bfA_1^{-\nf 12} e,\bfA_1^{\nf 12} e \, ; \a \bigr)
	=
	\frac12 J\bigl(U,p-p^*,q^*-q \, ; \a \bigr)
	+
	\frac12 J^*\bigl(U,p^*+p,q^*+q \, ; \a\bigr)
	\,.
\end{equation}
In particular, if~$\a_1=\shom$ for some symmetric positive definite matrix~$\shom$, then
\begin{equation*}
	\bfA_1= \begin{pmatrix} \shom & 0 \\ 0 & \shom^{-1} \end{pmatrix}
\end{equation*}
and, for every~$e_1,e_2\in\Rd$ with~$|e_1|^2+|e_2|^2=1$, setting
\begin{equation*}
	e \coloneqq \frac1{\sqrt2}
	\begin{pmatrix}
		e_1+e_2
		\\
		e_2-e_1
		\end{pmatrix}\,,
\end{equation*}
we have~$|e|=1$ and
\begin{equation}
	\label{e.mathcal.E.bfJ.scalar}
	\bfJ\bigl(U,\bfA_1^{-\nf 12} e,\bfA_1^{\nf 12} e  \, ; \a \bigr)
	=
	J\bigl(U,\shom^{-\nf 12}e_1,\shom^{\nf 12}e_1 \, ; \a  \bigr)
	+
	J^*\bigl(U,\shom^{-\nf 12}e_2,\shom^{\nf 12}e_2 \, ; \a \bigr)
	\,.
\end{equation}
Consequently,
\begin{equation}
	\label{e.mathcal.E.bfJ.scalar.max}
	\max_{|\eta|=1}
	\bfJ\bigl(U,\bfA_1^{-\nf 12} \eta,\bfA_1^{\nf 12} \eta \, ; \a \bigr)
	=
	\max
	\biggl\{
	\max_{|e|=1} J\bigl(U,\shom^{-\nf 12}e,\shom^{\nf 12}e \, ; \a\bigr)\,,
	\max_{|e|=1} J^*\bigl(U,\shom^{-\nf 12}e,\shom^{\nf 12}e \, ; \a\bigr)
	\biggr\}
	\,.
\end{equation}
If, in addition,~$\shom=\lambda\Id$ for some~$\lambda>0$, then the rescaling in the variational problems defining~$J$ and~$J^*$ gives, for every~$e\in\Rd$,
\begin{equation*}
	\begin{aligned}
		J\bigl(U,\shom^{-\nf 12}e,\shom^{\nf 12}e \, ; \a\bigr)
		&=
		J\bigl(U,e,e \, ; \shom^{-1}\a\bigr)\,,
		\\
		J^*\bigl(U,\shom^{-\nf 12}e,\shom^{\nf 12}e \, ; \a\bigr)
		&=
		J^*\bigl(U,e,e \, ; \shom^{-1}\a\bigr)\,.
	\end{aligned}
\end{equation*}
Combining this with~\eqref{e.commute.k0}, if~$\shom=\lambda\Id$ for some~$\lambda>0$ and~$\khom$ is a constant skew-symmetric matrix, then
\begin{equation}
	\label{e.mathcal.E.scalar.skew.normalize}
	\mathcal{E}_{s,q}(\cu_m;\a,\shom+\khom)
	=
	\mathcal{E}_{s,q}\bigl(\cu_m;\shom^{-1}(\a-\khom),\Id\bigr)\,.
\end{equation}
We extend~\eqref{e.mathcal.E.def.new} to~$q=\infty$ and to translations~$y+\cu_m$ of~$\cu_m$ in the obvious way. We also record the monotonicity of~$\mathcal E_{s,q}$ in the exponent. The elementary observation is that, for every nondecreasing sequence~$\{\alpha_n\}_{n\in\N}$,
\begin{equation}
	\label{e.deep.fact}
	r\mapsto
	\css{r}\sum_{n=0}^\infty 3^{-rn}\alpha_n
	\quad \mbox{is nonincreasing on~$(0,\infty)$}\,.
\end{equation}
Indeed, after writing~$r=sq$ in~\eqref{e.mathcal.E.def.new} and setting
\begin{equation*}
	\alpha_n
	\coloneqq
	\max_{z\in3^{m-n}\Zd\cap\cu_m}
	\max_{|\eta|=1}
	\bfJ\bigl(z+\cu_{m-n},\bfA_1^{-\nf12}\eta,\bfA_1^{\nf12}\eta\,;\a\bigr)^{\nf q2}
	\,,
\end{equation*}
the subadditivity~\eqref{e.subadditivity}, applied through~\eqref{e.bfJ}, implies that~$\{\alpha_n\}_{n\in\N}$ is nondecreasing. Thus, for every finite~$q\in[1,\infty)$ and~$0<t<s$,
\begin{equation}
	\label{e.mathcal.E.monotone.s}
	\mathcal E_{s,q}(\cu_m;\a,\a_1)
	\leq
	\mathcal E_{t,q}(\cu_m;\a,\a_1)
	\,.
\end{equation}
The case~$q=\infty$ is immediate from the definition, and the same monotonicity holds on translated cubes.

\smallskip

We also introduce the following related coarse-grained quantities. Given any coefficient field~$\a$, a deterministic matrix~$\ahom \in\R^{d\times d}$ with symmetric part~$\shom$ and antisymmetric part~$\khom$, an exponent~$s \in (0,1)$ and~$m\in\N$, we define
\begin{equation}
	\label{e.CE.def}
	\CE_s(\cu_m;\a,\shom)
	\coloneqq
	\biggl( \css{2s} \sum_{k=-\infty}^{m} 3^{2s(k-m)} \max_{z \in 3^k \Zd \cap \cu_m}
	\bigl| \shom^{-1} (\s - \s_*)(z +  \cu_k;\a)
	\bigr| \biggr)^{\nf12}
	\,
\end{equation}
and
\begin{align}
	\label{e.FE.def}
	\FE_s(\cu_m;\a,\ahom)
	&\coloneqq
	\biggl( \css{2s} \sum_{k=-\infty}^{m} 3^{2s(k-m)}
	\max_{z \in 3^k \Zd \cap \cu_m}
	\bigl| \s_{*}^{-\nf12}(z+\cu_k;\a)(\s_* (z+\cu_k;\a) - \shom ) \shom^{-\nf12} \bigr|^2
	\biggr)^{\nf12}
	\notag \\ &\qquad +
	\biggl( \css{2s} \sum_{k=-\infty}^{m} 3^{2s(k-m)}
	\max_{z \in 3^k \Zd \cap \cu_m}
	\bigl| \s_{*}^{-\nf12}(z+\cu_k;\a)(\k (z+\cu_k;\a) - \khom ) \shom^{-\nf12} \bigr|^2
	\biggr)^{\nf12}
	\,.
\end{align}
When~$\shom = \lambda \Id$ for some~$\lambda > 0$, the constant-skew normalization~\eqref{e.mathcal.E.scalar.skew.normalize} allows us to center the coefficient by~$\khom$, and~\eqref{e.J.mat},~\eqref{e.J.mat.star} together with~\eqref{e.mathcal.E.bfJ.scalar.max} then give
\begin{equation}
	\label{e.sum.CE.FE}
	\mathcal{E}_{s,2} (\cu_m \, ; \a, \ahom )
	\leq
	\CE_s(\cu_m;\a,\shom)  +
	\FE_s(\cu_m;\a,\ahom)
	\leq
	5 \, \mathcal{E}_{s,2} (\cu_m \, ; \a, \ahom )
	\,.
\end{equation}
The random variables in~\eqref{e.CE.def} and~\eqref{e.FE.def} measure, respectively, the coarse-graining error and the difference between the coarse-grained matrices and the given deterministic matrix~$\ahom$, with a geometric discount for smaller scales.

\smallskip

Next, for every~$s\in(0,\infty)$,~$q\in[1,\infty)$,~$m\in\Z$ and coefficient field~$\a$ on~$\cu_m$, we define
\begin{equation}
	\label{e.coarse.grained.ellipticity}
	\left\{
	\begin{aligned}
		& \Lambda_{s,q}(\cu_m\,;\a)
		\coloneqq
		\biggl(
		\css{sq} \sum_{k=-\infty}^{m}
		3^{-sq(m-k)}
		\max_{z\in 3^k\Zd \cap \cu_m}
		\bigl| \b(z+\cu_k; \a) \bigr|^{\nf q2}
		\biggr)^{\!\nf{2}{q}}\,,
		\\
		& \lambda_{s,q}(\cu_m;\a)
		\coloneqq
		\biggl(
		\css{sq} \sum_{k=-\infty}^{m}
		3^{-sq(m-k)}
		\max_{z\in 3^k\Zd \cap \cu_m}
		\bigl| \s_{*}^{-1}(z+\cu_k; \a) \bigr|^{\nf q2}
		\biggr)^{\!-\nf{2}{q}} \,.
	\end{aligned}
	\right.
\end{equation}
When the exponent~$q$ is omitted, it is understood to be~$1$. We extend these definitions to translations~$y+\cu_m$ in the obvious way. For every finite~$q\in[1,\infty)$ and~$0<t<s$, we use the elementary ordering
\begin{equation}
	\label{e.ellipticities.monotone.ordered}
	|\s_*(\cu_m;\a)|
	\leq
	|\b(\cu_m;\a)|
	\leq
	\Lambda_{s,q}(\cu_m; \a)
	\leq
	\Lambda_{t,q}(\cu_m; \a)
	\qand
	\lambda_{s,q}^{-1}(\cu_m;\a)
	\leq
	\lambda_{t,q}^{-1}(\cu_m;\a)\,.
\end{equation}
By Cauchy--Schwarz in the scale sum, we also have, for every~$s\in(0,1]$,
\begin{equation}
	\label{e.ellipticities.qone.by.qtwo}
	\Lambda_{s,1}(\cu_m;\a)
	\leq
	\Lambda_{\nf s2,2}(\cu_m;\a)
	\qand
	\lambda_{s,1}^{-1}(\cu_m;\a)
	\leq
	\lambda_{\nf s2,2}^{-1}(\cu_m;\a)
	\,.
\end{equation}
Finally,~\cite[Lemma~5.9]{AK.HC} gives the comparison between the homogenization error and the coarse-grained ellipticity constants: for every~$q\in[1,\infty)$ and every symmetric positive definite matrix~$\s_1$,
\begin{equation}
	\label{e.bound.Lambdas.by.Es}
	\max\bigl\{
	|\s_1|^{-1}\Lambda_{s,q}(\cu_m;\a)\,,
	|\s_1^{-1}|^{-1}\lambda_{s,q}^{-1}(\cu_m;\a)
	\bigr\}
	\leq
	1+2\mathcal{E}_{s,q}(\cu_m;\a,\s_1)^2
	+2\mathcal{E}_{s,q}(\cu_m;\a,\s_1)\,.
\end{equation}

\smallskip

Next, for~$s\in(0,1)$,~$p,q\in[1,\infty)$ and~$n\in\Z$, we define the normalized Besov seminorm in~$\cu_n$ by
\begin{equation*}
	[g]_{\underline{B}_{p,q}^{s}(\cu_n)}
	\coloneqq
	\biggl(
	\sum_{k=-\infty}^n
	3^{-sqk}
	\biggl(
	\avsum_{z\in3^{k-1}\Zd\,:\,z+\cu_k\subseteq\cu_n}
	\|g-(g)_{z+\cu_k}\|_{\underline L^p(z+\cu_k)}^p
	\biggr)^{\!\nf{q}{p}}
	\biggr)^{\!\nf1q}
	\,.
\end{equation*}
For~$q=\infty$, we set
\begin{equation*}
	[g]_{\underline{B}_{p,\infty}^{s}(\cu_n)}
	\coloneqq
	\sup_{k\in\Z\cap(-\infty,n]}
	3^{-sk}
	\biggl(
	\avsum_{z\in3^{k-1}\Zd\,:\,z+\cu_k\subseteq\cu_n}
	\|g-(g)_{z+\cu_k}\|_{\underline L^p(z+\cu_k)}^p
	\biggr)^{\!\nf1p}
	\,.
\end{equation*}
The corresponding normalized norm is
\begin{equation*}
	\|g\|_{\underline{B}_{p,q}^{s}(\cu_n)}
	\coloneqq
	3^{-sn}|(g)_{\cu_n}|+[g]_{\underline{B}_{p,q}^{s}(\cu_n)}\,,
\end{equation*}
and the definition on translated cubes is obtained by translation. We use Besov embeddings and the equivalence of~$\underline H^s$ with~$\underline B_{2,2}^{s}$ from~\cite[Appendix~A]{AK.HC}.

\smallskip

We then recall the deterministic estimates from~\cite{AK.HC}. The first two are~\cite[Lemma~2.3 and Proposition~2.5]{AK.HC}; the displayed consequences follow from~\eqref{e.ellipticities.qone.by.qtwo} and~\eqref{e.bound.Lambdas.by.Es}. The third is~\cite[Proposition~5.3]{AK.HC}.

\begin{lemma}[Coarse-grained Poincar\'e inequality]
\label{l.Poincare.largescale}
There exists a constant~$C(d) < \infty$ such that, for every~$s\in [0,1)$,~$n \in \N$ and~$u \in \A(\cu_n;\a)$,
\begin{equation}
	\label{e.Poink}
	\left\| u- (u)_{\cu_n}  \right\|_{\underline{B}_{2,\infty}^{s}(\cu_n)}
	\leq
	C3^{(1-s)n}
	\lambda_{1-s,1}^{-\nf 12}(\cu_n\,;\a)
	\|\s^{\nf12} \nabla u \|_{\underline{L}^2(\cu_n)}
	\,.
\end{equation}
In particular, if~$\s_1 \coloneqq \lambda \Id$ for~$\lambda>0$ and~$\mathcal{E}_{\frac12(1-s),2} (\cu_n\, ; \a , \s_1 )\leq 1$, then, using~\eqref{e.ellipticities.qone.by.qtwo} and~\eqref{e.bound.Lambdas.by.Es},
\begin{equation}
	\label{e.Poinked}
	\| u- (u)_{\cu_n} \|_{\underline{B}_{2,\infty}^{s}(\cu_n)}
	\leq
	C 3^{(1-s)n} \lambda^{-\nf 12}
	\|\s^{\nf12} \nabla u \|_{\underline{L}^2(\cu_n)}
	\,.
\end{equation}
\end{lemma}

\begin{proposition}[Coarse-grained Caccioppoli inequality]
\label{p.coarse.grained.Caccioppoli}
There exists a constant~$C(d)<\infty$ such that, for every~$s,t\in (0,1)$ with~$s+t < 1$ and for every~$n \in \Z$ and~$u\in\A(\cu_n; \a)$, we have the estimate
\begin{equation}
	\label{e.coarse.grained.Caccioppoli}
	\| \s^{\nf12}  \nabla u \|_{\underline{L}^2(\cu_{n-1})}^2
	\leq
	\biggl( \frac{C}{1-s-t} \biggr)^{\! 2 + \frac{4s}{1-s-t}}
	\biggl( \frac{\Lambda_{s,1}(\cu_n; \a)}{\lambda_{t,1}(\cu_n; \a)} \biggr)^{\!\frac{s}{1-s-t}}
	{\Lambda}_{s,1}(\cu_n; \a)
	3^{-2n}
	\| u \|_{\underline{L}^2(\cu_{n})}^2
	\,.
\end{equation}
The dependence on the ellipticity contrast enters through the ratio~$\Lambda_{s,1}/\lambda_{t,1}$, which also appears in Lemma~\ref{l.localization.average}. In particular, if~$\s_1 \coloneqq \lambda \Id$ for~$\lambda>0$ and~$\mathcal{E}_{\frac18,2} (\cu_n\, ; \a , \s_1 )\leq 1$, then, using the preceding estimate with~$s=t=\nicefrac14$ together with~\eqref{e.ellipticities.qone.by.qtwo} and~\eqref{e.bound.Lambdas.by.Es},
\begin{equation}
	\label{e.coarse.grained.Caccioppoli.inpar}
	\| \s^{\nf12}  \nabla u \|_{\underline{L}^2(\cu_{n-1})}
	\leq
	C \lambda^{\nf 12}
	3^{-n}
	\| u \|_{\underline{L}^2(\cu_{n})}
	\,.
\end{equation}
\end{proposition}

In the next proposition,~$\Omega(\cu_0)$ denotes the local coefficient class used in~\cite{AK.HC}. The spaces~$H^1_{\a}(U)$ and~$H^1_{\a,0}(U)$ are the corresponding finite-energy space and zero-boundary finite-energy space.

\begin{proposition}[Quantitative homogenization estimate]
\label{p.big.black.box}
Assume that~$s\in (0,\nf12)$,~$U \subseteq \cu_0$ is a Lipschitz domain,~$\a \in \Omega(\cu_0)$ and~$\a_1 \in \R^{d\times d}$ is a matrix satisfying the primal and dual ellipticity bounds, for some constants~$0<\lambda < \Lambda < \infty$,
\begin{equation}
	\lambda|\xi|^2 \leq \xi \cdot \a_1 \xi
	\quad \mbox{and} \quad
	\Lambda^{-1}|\xi|^2 \leq \xi \cdot \a_1^{-1} \xi
	\,.
	\label{e.a0.ue}
\end{equation}
Then there exists a constant~$C(U,s,\lambda,\Lambda,d)<\infty$ such that the following hold.
\begin{itemize}

\item \emph{Homogenization error estimate.} For every~$h \in H^{1-s}(U)$ and~$u\in H^1_{\a}(U)$ satisfying
\begin{equation*}
	\left\{
	\begin{aligned}
		& -\nabla \cdot \a\nabla u = 0 & \mbox{in} & \ U \,,
		\\
		& -\nabla \cdot \a_1\nabla h = 0 & \mbox{in} & \ U \,,
		\\
		&
		u - h \in H^{1-s}_0(U)
		\,,
	\end{aligned}
	\right.
\end{equation*}
we have the estimate
\begin{equation}
	\|  \nabla u - \nabla h \|_{\Hminushat{-s} (U)}
	+
	\| \a \nabla u - \a_1\nabla h \|_{\Hminushat{-s} (U)}
	\leq
	C
	\mathcal{E}_ {s,2} (\cu_0;\a,\a_1)
	\|\s^{\nf12} \nabla u \|_{{L}^2(U)}
	\,.
	\label{e.bbb.homogenization}
\end{equation}

\item \emph{Attainability of~$H^{\nf12+s}(\partial U)$ boundary data.} For every~$g \in W^{1,\infty}(U)\cap H^{1+s}(U)$, the unique solution~$u\in g+H^{1}_{\a,0}(U)$ of the Dirichlet problem
\begin{equation}
	\left\{
	\begin{aligned}
		& -\nabla \cdot \a\nabla u = 0 & \mbox{in} & \ U \,,
		\\
		& u = g & \mbox{on} & \ \partial U
		\,,
	\end{aligned}
	\right.
	\label{e.bbb.Dirichlet}
\end{equation}
satisfies the estimate
\begin{equation}
	\| \s^{\nf12} \nabla u \|_{\underline{L}^2(U)}
	\leq
	C
	\bigl(
	1+
	\mathcal{E}_ {s,2} (\cu_0;\a,\a_1)
	\bigr)
	\| \nabla g \|_{H^s(U)}
	\,.
	\label{e.attainability.estimate}
\end{equation}

\end{itemize}
\end{proposition}

\subsection{Orlicz notation for tail bounds}
\label{ss.bigO}

Throughout the paper we track tail bounds of random variables as follows: for~$A > 0$ and an increasing function~$\Psi: \R_+ \to [1, \infty)$ satisfying~$\Psi(t)\to\infty$ as~$t\to\infty$ and
\begin{equation}
	\label{e.Psi.int}
	\int_1^\infty
	\frac{t}{\Psi(t)}\,dt  < \infty
	\,,
\end{equation}
and a random variable~$X$, we write
\begin{equation}
	\label{e.orlicz.tail}
	X \leq \O_{\Psi}(A)
\end{equation}
to mean
\begin{equation}
	\label{e.orlicztailbound.bound}
	\P[X > t A] \leq \frac{1}{\Psi(t)} \, , \quad \forall t \in [1,\infty) \,.
\end{equation}
As we will see just below, this induces a particularly useful algebra which allows us to essentially multiply and add the right sides of~\eqref{e.orlicz.tail}. More generally, we write
\begin{equation*}
	X \leq \O_{\Psi_1}(A_1) + \ldots +  \O_{\Psi_n}(A_n)
\end{equation*}
to mean that
\begin{equation*}
	X \leq X_1 + \cdots + X_n \,, \qquad \mbox{where} \ X_i = \O_{\Psi_i}(A_i)\,, \quad \forall i\in\{1,\ldots,n\}\,.
\end{equation*}
For nonnegative~$X$, we also write~$X = \O_\Psi(A)$ as a synonym for~$X\leq \O_\Psi(A)$. For a signed random variable~$X$, we write~$X = \O_\Psi(A)$ to mean~$|X|\leq \O_\Psi(A)$, i.e.,~$\P[\,|X| > tA\,]\leq 1/\Psi(t)$ for every~$t\in[1,\infty)$.

\smallskip

Although we introduce this notation for a general increasing function~$\Psi$ satisfying~\eqref{e.Psi.int}, for most of this paper we will use~$\Psi = \Gamma_{\sigma}$, where for~$\sigma \in (0, \infty)$,
\begin{equation}
	\label{e.Gamma.sigma}
	\Gamma_{\sigma}(t)  \coloneqq  \exp(t^{\sigma}) \,.
\end{equation}
This describes random variables with stretched exponential tails. The important case~$\sigma = 2$ specifies Gaussian tails and~\eqref{e.k(n).reg} can be rewritten as
\begin{equation}
	\label{e.k(n).reg.with.gamma}
	\| \mathbf{j}_n \|_{L^{\infty}(\cu_n)}
	+
	3^n \| \nabla \mathbf{j}_n \|_{L^{\infty}(\cu_n)}
	+
	3^{2n} \| \nabla^2 \mathbf{j}_n  \|_{L^{\infty}(\cu_n)}
	\leq \O_{\Gamma_2}(1) \,.
\end{equation}
We recall some basic properties and refer to~\cite[Appendix C]{AK.HC} and~\cite[Appendix A]{AKMBook} for an in-depth discussion.
\begin{lemma}[Generalized triangle inequality]
\label{l.Gamma.sigma.triangle}
There exists a universal constant~$C<\infty$ such that the following holds. Let~$\sigma \in (0,\infty)$ and let~$\{a_k\}_{k\in\N}$ be a nonnegative sequence with~$\sum_k a_k<\infty$. Suppose that the random variables~$\{X_k\}_{k\in\N}$ satisfy~$X_k=\O_{\Gamma_\sigma}(a_k)$ for every~$k\in\N$. Then
\begin{equation}
	\label{e.Gamma.sigma.triangle}
	\sum_{k \in \N} X_k \leq \O_{\Gamma_\sigma}\biggl( C\bigl(1 + \sigma^{-\nicefrac1\sigma}\indc_{\{\sigma < 1\}} \bigr) \sum_{k \in \N} a_k \biggr)
	\,.
\end{equation}
\end{lemma}
\begin{proof}
It is enough to prove the estimate for finite sums and then pass to the limit over partial sums; in the signed case the summability of~$\sum_k a_k$ also gives almost-sure absolute convergence. The inequality for~$\sigma \geq 1$ is proved in~\cite[Lemma A.4]{AKMBook}. Let~$0<\sigma<1$. Tail integration, the gamma-function bound, and Minkowski's inequality give, for every~$p\geq1$ and~$A\coloneqq\sum_k a_k$,
\begin{align*}
	\biggl\|\sum_k|X_k|\biggr\|_{L^p(\Omega)}
	&\leq
	\bigl(1+\gammafun(1+p/\sigma)\bigr)^{\nicefrac1p}A
	\leq
	C\sigma^{-\nicefrac1\sigma}p^{\nicefrac1\sigma}A\,,
	\\
	\P\biggl[\sum_k|X_k|>eC\sigma^{-\nicefrac1\sigma}tA\biggr]
	&\leq
	\exp(-t^\sigma)\,,
	\qquad t\geq1\,,
\end{align*}
where the second line follows from the first by taking~$p=t^\sigma$ and applying Markov's inequality. This is the asserted estimate.
\end{proof}

\begin{lemma}[Multiplication property]
\label{l.o.gamma2.mult}
For every~$\sigma_1, \sigma_2 \in (0, \infty)$ if~$X_1$,~$X_2$ are positive random variables, then
\begin{equation}
	\label{e.multGammasig}
	X_1 \leq \O_{\Gamma_{\sigma_1}}(A_1) \qand X_2 \leq \O_{\Gamma_{\sigma_2}}(A_2) \implies X_1 X_2 \leq \O_{\Gamma_{\frac{\sigma_1 \sigma_2}{\sigma_1 + \sigma_2}}}(C A_1 A_2) \,.
\end{equation}
In particular, for every~$\sigma, p, K \in (0, \infty)$ and positive random variable~$X$,
\begin{equation}
	\label{e.powerofGammasigma}
	X \leq \O_{\Gamma_{\sigma}}(K)
	\iff
	X^{p} \leq \O_{\Gamma_{\sigma/p}} (K^p)\,.
\end{equation}
\end{lemma}
\begin{proof}
This is~\cite[Lemma A.3]{AKMBook}.
\end{proof}

\begin{lemma}[Maximum of~$\O_{\Gamma_\sigma}$ random variables]
\label{l.maximums.Gamma.s}
Suppose that~$\sigma,A>0$ and~$X_1,\ldots,X_N$ is a sequence of random variables satisfying~$X_i = \O_{\Gamma_\sigma}(A)$ for~$N \geq 2$. Then
\begin{equation}
	\label{e.maxy.bound}
	\max_{1\leq i \leq N} X_i = \O_{\Gamma_\sigma} \bigl( (3 \log N)^{\nicefrac{1}{\sigma}} A\bigr)\,.
\end{equation}
\end{lemma}
\begin{proof}
For every~$t \geq 1$ we use a union bound to estimate
\begin{align*}
	\P \Bigl[ \max_{1\leq i \leq N} X_i > A (3 \log N)^{\nicefrac{1}{\sigma}} t  \Bigr]
	&\leq \sum_{i=1}^N \P \bigl[ X_i > A (3 \log N)^{\nicefrac{1}{\sigma}} t\bigr] \\
	&\leq N \exp\Bigl( -3 t^{\sigma} \log N \Bigr)
	\leq \exp\bigl(-t^{\sigma} (3 \log N - \log N)  \bigr)
	\leq  \exp(-t^\sigma) \,,
\end{align*}
where in the last line we used~$2 \log N \geq 2\log 2 > 1$. The lower tail follows from
\begin{equation*}
	\P \Bigl[ \max_{1\leq i \leq N} X_i < - A (3 \log N)^{\nicefrac{1}{\sigma}} t  \Bigr]
	\leq
	\P \bigl[ X_1 < - A (3 \log N)^{\nicefrac{1}{\sigma}} t\bigr]
	\leq
	\exp(-t^\sigma)
	\,.
\end{equation*}
\end{proof}

The indicator function~$\indc_E$ of an event~$E$ with~$0<\P[E] \leq e^{-1}$ satisfies, for every~$\sigma\in (0,\infty)$,
\begin{equation}
	\label{e.indc.O.sigma}
	\indc_{E} \leq \O_{\Gamma_\sigma} \bigl( \bigl| \log \P[E] \bigr|^{-\nicefrac{1}{\sigma}} \bigr)
	\,.
\end{equation}
This is immediate from the definition of~\eqref{e.orlicz.tail} in~\eqref{e.orlicztailbound.bound}.

\begin{lemma}
\label{l.moments.gamma.psi}
For every~$A \in (0,\infty)$ and random variable~$X$ satisfying~$|X| \leq \O_{\Gamma_2}(A)$,
\begin{equation}
	\label{e.moments.OGamma2}
	\E[|X|^k] \leq
	A^k \bigl( 1 + \gammafun(\nicefrac{k}{2}+1) \bigr) \,, \quad \forall k\in\N\,,
\end{equation}
and, for every~$N \geq 0$,
\begin{equation}
	\label{e.moments.lognormal}
	\E \bigl[ \exp( NX) - 1 \bigr]
	\leq
	3 \max \bigl\{ A N \exp(A N) , A^2N^2 \exp( A^2N^2) \bigr\}
	\,.
\end{equation}
\end{lemma}
\begin{proof}
To obtain~\eqref{e.moments.OGamma2}, we assume~$A=1$ without loss of generality and straightforwardly compute
\begin{align*}
	\E[|X|^{k}] = \int_0^{\infty} \P[ |X|^{k} > t] \,dt
	&\leq 1 +  \int_1^{\infty} \P[|X| > t] k t^{k-1} \,dt \\
	& \leq
	1 + \frac12\int_1^\infty k t^{k-2} \exp ( -t^2)\, 2t dt
	\leq 1 + \gammafun(\nicefrac{k}{2}+1)\,.
\end{align*}
We next use~\eqref{e.moments.OGamma2} to estimate
\begin{equation*}
	\E[ \exp(N X) ] - 1
	\leq
	\sum_{k=1}^{\infty}  \frac{N^k }{k!}\E\bigl[|X|^k \bigr]
	\leq
	\sum_{k=1}^{\infty}  \frac{A^k N^k }{k!}
	+
	\sum_{k=1}^{\infty}  \frac{A^k N^k \gammafun(\nicefrac{k}{2}+1)}{k!}
	\,.
\end{equation*}
The first term on the right side is bounded from above by~$A N \exp(A N)$ and the second term is bounded using the identity
\begin{equation*}
	\gammafun(k+\nicefrac{1}{2}) = \frac{\sqrt{\pi} (2k)!}{4^k k!} \,, \quad k\in\N\,,
\end{equation*}
as follows:
\begin{align*}
	\sum_{k=1}^{\infty}  \frac{A^k N^k \gammafun(\nicefrac{k}{2}+1)}{k!}
	&
	=
	\sum_{k=1}^{\infty}
	\frac{(A N)^{2k} \gammafun(k+1)}{(2k)!}
	+
	\sum_{k=0}^{\infty}
	\frac{(A N)^{2k+1} \gammafun((k+1)+\nicefrac{1}{2})}{(2k+1)!}
	\notag \\ &
	=
	\sum_{k=1}^{\infty}
	\frac{(A N)^{2k} }{k!} \frac{(k!)^2}{(2k)!}
	+
	\sum_{k=0}^\infty
	\frac{(A N)^{2k+1} }{k!}\frac{\sqrt{\pi}} {2\cdot 4^k}
	\notag \\ &
	\leq
	\bigl( \exp( A^2N^2 ) - 1 \bigr)
	+ A N \exp( A^2N^2 )
	\notag \\ &
	\leq
	A^2N^2\exp( A^2N^2 )
	+
	A N \exp( A^2N^2 )
	\,.
\end{align*}
To conclude, we note that
\begin{equation*}
	A N \exp(A N)
	+
	A^2N^2\exp( A^2N^2 )
	+
	A N \exp( A^2N^2 ) \leq 3 \max \{ A N \exp(A N) , A^2N^2 \exp( A^2N^2) \} \,.
\end{equation*}
This completes the proof of~\eqref{e.moments.lognormal}.
\end{proof}

We will use the following concentration inequality for sums of centered, independent random variables with stretched exponential tails. For a proof, see~\cite[Lemma~C.2 \& Corollary C.4]{AK.HC}.

\begin{proposition}[Concentration for~$\O_{\Gamma_\sigma}$]
\label{p.concentration}
For every~$\sigma \in (0,2]$,~$m\in\N$ and finite sequence~$X_1,\ldots,X_m$ of independent random variables satisfying
\begin{equation}
	X_k = \O_{\Gamma_\sigma} (1)
	\quad \mbox{and} \quad
	\E[ X_k ] = 0 \, , \quad \forall k\in \{1,\ldots,m\}\,,
\end{equation}
we have the estimate
\begin{equation}
	\label{e.concentration}
	\sum_{k=1}^m X_k =
	\O_{\Gamma_\sigma}
	\bigl( C_\sigma m^{\nicefrac{1}{2}} \bigr)
	\,.
\end{equation}
Here and below,~$C_\sigma$ is the constant in~\cite[Lemma~C.2 and Corollary~C.4]{AK.HC}; in particular, for~$0<\sigma<1$ we may take~$C_\sigma=2\bigl((8^{2/\sigma}\gammafun(2/\sigma))^2+1\bigr)$.
\end{proposition}

We often apply Proposition~\ref{p.concentration} in the case of sequences which have a finite range of dependence. For instance, we may have random variables~$\{ X_z \}_{z\in 3^n\Zd \cap \cu_m}$ for some~$m,n\in \N$ with~$n<m$, which have the property that~$X_z$ and~$X_{z'}$ are independent provided that the corresponding subcubes do not touch: that is, if~$\dist(z+\cu_n,z'+\cu_n) \neq 0$. In this case, we can simply break into~$3^d$ many subcollections which are independent:
\begin{align*}
	\sum_{z\in 3^n\Zd \cap \cu_m} X_z
	& =
	\sum_{y \in 3^n\Zd \cap \cu_{n+1}}
	\sum_{z\in 3^{n+1}\Zd \cap \cu_m} X_{y+z}
\end{align*}
We then apply~\eqref{e.concentration} to each of the inner sums: assuming~$X_z = \O_{\Gamma_\sigma}(1)$ and~$\E[X_z]=0$, we have
\begin{equation*}
	\sum_{z\in 3^{n+1}\Zd \cap \cu_m} X_{y+z}
	=
	\O_{\Gamma_\sigma}
	\bigl( C_\sigma 3^{\frac d2(m-n-1)} \bigr)
	\,.
\end{equation*}
Summing over~$y\in 3^n\Zd \cap \cu_{n+1}$ and using the triangle inequality then yields
\begin{equation*}
	\sum_{z\in 3^n\Zd \cap \cu_m} X_z
	=
	\O_{\Gamma_\sigma} \bigl(C_\sigma 3^{\frac d2(m-n)} \bigr)
	\,.
\end{equation*}

\subsection{Infrared cutoffs}
\label{ss.infrared}
We now move from the general setting considered above to the particular setting of Theorem~\ref{t.superdiffusivity}, that is,~$\a \coloneqq  \nu \Id + \k$ with~$\k$ given in~\eqref{e.k.sum.def} with the assumptions~\ref{a.j.frd}--\ref{a.j.nondeg} in force.

\smallskip

We let~$\a_m$ and~$\k_m$ be the infrared cutoffs defined in~\eqref{e.infrared.cutoff.def}, and we define also
\begin{equation}
	\label{e.bfA.ell.def}
	\bfA_m(x)  \coloneqq
	\begin{pmatrix}
		( \nu\Id + \nu^{-1}\k_m^t\k_m )(x)
		& -\nu^{-1}\k_m^t(x)
		\\ - \nu^{-1}\k_m (x)
		& \nu^{-1}\Id
	\end{pmatrix}
	\,.
\end{equation}
The divergence-free vector field~$\f_m$ is defined by
\begin{equation*}
	\f_{m}
	\coloneqq
	\nabla \cdot \k_m\,.
\end{equation*}
By~\ref{a.j.frd}, the fields~$\a_m$,~$\k_m$,~$\bfA_m$ and~$\f_m$ are~$\Rd$--stationary and have range of dependence~$\sqrt{d}3^m$. The assumption~\ref{a.j.reg} implies that, for a constant~$C(d)<\infty$,
\begin{equation}
	\label{e.f.to.fell}
	\| \f - \f_{m} \|_{L^\infty(\cu_{m})}
	\leq
	\sum_{n=m+1}^\infty
	\| \nabla \cdot \mathbf{j}_n \|_{L^\infty(\cu_{n})}
	\leq
	\O_{\Gamma_2} \biggl( \sum_{n=m+1}^\infty 3^{-n} \biggr)
	\leq \O_{\Gamma_2}(3^{-m})
\end{equation}
and, indeed, for every~$m,n\in\N$ with~$n < m$,
\begin{equation}
	\label{e.nabla.kmn.Linfty}
	\| \nabla (\k_m-\k_n) \|_{L^\infty(\cu_n)} = \O_{\Gamma_2}\bigl( 3^{-n}\bigr)
	\,.
\end{equation}
In the following lemma we record some more basic estimates on~$\k_m-\k_n$ in~$L^p$ and~$W^{-s,p}$ norms.

\begin{lemma}[Estimates on~$\k_m-\k_n$]
\label{l.ellip.k.scales.estimates}
Let~$s\in (0,1)$. For every~$p\in [1,\infty)$, there exists a constant~$C(s,d,p)<\infty$ such that, for every~$l,m,n\in\N$ with~$n<m\leq l$,
\begin{align}
	\label{e.kmn.Wminussp}
	& \bigl\| \k_m - \k_n \bigr\|_{\Wminusul{-s}{p}(\cu_l)}
	\leq
	\O_{\Gamma_2}\bigl( C3^{sm}\bigr)
	\,,
	\\ &
	\label{e.kmn.Lp}
	\bigl\| \k_m - \k_n \bigr\|_{\underline{L}^{p}(\cu_l)}
	\leq
	C (m-n)^{\nicefrac{1}{2}}+
	\O_{\Gamma_2} \bigl(C (m-n)^{\nicefrac{1}{2}} 3^{-\frac{d}{2p}(l-m)} \bigr)\,,
	\\ &
	\label{e.kmn.Linfty}
	\bigl\| \k_m - \k_n \bigr\|_{{L}^{\infty}(\cu_l)}
	\leq
	\O_{\Gamma_2} (C (m-n)^{\nicefrac{1}{2}}(l-n)^{\nicefrac{1}{2}})
	\,.
	\\ &
	\label{e.kmn.Hs}
	\bigl\| \k_m - \k_n \bigr\|_{\underline{H}^{s}(\cu_l)}
	\leq
	\O_{\Gamma_2} (C 3^{-sn} )
	\,.
\end{align}
In addition, for every~$l,n\in\N$,
\begin{equation}
	\label{e.kmn.Hs.osc}
	\sup_{M\in\N \,:\, M>n}
	[ \k_M - \k_n ]_{\underline{H}^{s}(\cu_l)}
	\leq
	\O_{\Gamma_2} (C 3^{-sn} )
	\,.
\end{equation}
Moreover, for every~$\delta \in (0,1)$ and~$\sigma \in (0,\infty)$, there exists a random variable~$\mathcal{K}_\sigma$ (whose dependence on~$\delta$ is suppressed from the notation) with
\begin{equation}
	\label{e.mathcal.K.int}
	\log \mathcal{K}_\sigma \leq \O_{\Gamma_{2\sigma}}( C_1(s,d) (C_2 \delta^{-1} \sigma^{-1})^{\nicefrac{1}{\sigma}})
\end{equation}
such that, for every~$m\in \N$ with~$3^m \geq \mathcal{K}_\sigma$,
\begin{equation}
	\label{e.kbounds.minscale.form}
	m^{-1} \bigl\| \k - (\k)_{\cu_m} \bigr\|_{L^\infty(\cu_m)}
	+
	3^m \| \nabla \k - \nabla \k_m \|_{L^\infty(\cu_m)}
	+
	3^{-s m} [ \k - (\k)_{\cu_m} ]_{\Hminushat{-s}(\cu_m)}
	\leq
	\delta m^{\sigma}
	\,
\end{equation}
and for every~$A,B \geq1$ and~$n\in\N$ with~$m - \lceil A \log(B  m) \rceil \leq n \leq m$ and~$z \in 3^n \Zd \cap \cu_m$
\begin{equation}
	\label{e.yet.another.minscale.icandoit}
	| (\k)_{z + \cu_n} - (\k)_{\cu_m} | \leq (A \log(B  m)) \delta m^{\sigma} \,.
\end{equation}
Also, for~$\mathcal{K}_\sigma$ satisfying~\eqref{e.mathcal.K.int}, we have, for every~$x\in \Rd$,
\begin{equation}
	\label{e.a.ellipticity.pointwise}
	\bigl| \k(x) - \k(0) \bigr|^2 \leq C \bigl( \log(\mathcal{K}_\sigma^2 + |x|^2)  \bigr)^{2(1+\sigma)}
	\,.
\end{equation}

\end{lemma}
\begin{proof}
The dihedral symmetry assumption~\ref{a.j.iso} and the~$\Rd$-stationarity of~$\mathbf{j}_n$ imply that, for every~$x\in\Rd$
\begin{equation*}
	\E \bigl[ \mathbf{j}_n(x) \bigr] = \E \bigl[ \mathbf{j}_n(0) \bigr] = 0
	\,.
\end{equation*}
The independence and tail assumptions,~\ref{a.j.frd} and~\ref{a.j.reg} respectively, therefore yield, for every~$k\in\N$,~$h\in\Z$, and~$y\in\Rd$,
\begin{equation}
	\label{e.jk.spatialavg}
	\big| (\mathbf{j}_k )_{y+\cu_{h}} \big|
	\leq
	\O_{\Gamma_2} \bigl( C 3^{-\frac d2((h-k)\vee 0) } \bigr) \,.
\end{equation}
Applying~\cite[Lemmas A.1 and A.4]{AK.HC}, followed by the embedding of the scale~$\ell^p$ norm into~$\ell^1$, and then using~\eqref{e.powerofGammasigma} and Lemma~\ref{l.Gamma.sigma.triangle}, for every~$l \geq k$,
\begin{align}
	\label{e.apply.multiscale.Poincare}
	3^{-s l} \bigl\| \mathbf{j}_k \bigr\|_{\Wminusul{-s}{p}(\cu_l)}
	& \leq
	C
	\sum_{h=-\infty}^l
	3^{s(h-l)}
	\biggl(
	\avsum_{y\in 3^{h}\Zd \cap \cu_l}
	\bigl| (  \mathbf{j}_k  )_{y+\cu_{h}} \bigr|^p
	\biggr)^{\nicefrac{1}{p}}
	\leq
	\O_{\Gamma_2}\bigl( C3^{s(k-l)} \bigr)
	\,.
\end{align}
At the endpoint, the scale-$\ell^1$ form of the same multiscale estimate, Jensen's inequality,~\eqref{e.jk.spatialavg},~\eqref{e.powerofGammasigma}, and Lemma~\ref{l.Gamma.sigma.triangle} give, for every~$l>k$,
\begin{equation*}
	3^{-l}\|\mathbf j_k\|_{\Hminusul(\cu_l)} \leq C\sum_{h=-\infty}^l3^{h-l}\biggl(\avsum_{y\in3^h\Zd\cap\cu_l}\bigl|(\mathbf j_k)_{y+\cu_h}\bigr|^2\biggr)^{\nf12}\,.
\end{equation*}
The ranges~$h\leq k$ and~$k<h\leq l$ contribute at most~$\O_{\Gamma_2}(C3^{-(l-k)})$ and~$\O_{\Gamma_2}(C(l-k)3^{-(l-k)})$, respectively. Hence
\begin{equation}
	\label{e.jk.Hminus.endpoint}
	3^{-l}\| \mathbf j_k \|_{\Hminusul(\cu_l)}
	\leq
	\O_{\Gamma_2}\bigl(C(l-k)3^{-(l-k)}\bigr)
	\,.
\end{equation}
The bound~\eqref{e.kmn.Wminussp} now follows from the triangle inequality and Lemma~\ref{l.Gamma.sigma.triangle}:
\begin{equation*}
	3^{-s l}\bigl\| \k_m - \k_n \bigr\|_{\Wminusul{-s}{p}(\cu_l)}
	\leq
	\sum_{k=n+1}^m
	3^{-sl}\bigl\| \mathbf{j}_k \bigr\|_{\Wminusul{-s}{p}(\cu_l)}
	\leq
	\O_{\Gamma_2} \bigl( C 3^{s(m-l)} \bigr)\,.
\end{equation*}
Turning to the proof of~\eqref{e.kmn.Lp} and~\eqref{e.kmn.Linfty}, we use the independence of different~$\mathbf{j}_k$ and the tail assumptions (\ref{a.j.indy} and~\ref{a.j.reg}, respectively) to obtain the existence of~$C(d)<\infty$ such that, for every~$m,n\in\N$ with~$n< m$,
\begin{equation}
	\label{e.k.ell.upscales}
	\bigl| (\k_m - \k_n) (0)\bigr| = \O_{\Gamma_2}\bigl( C(m-n)^{\nicefrac{1}{2}}\bigr) \,.
\end{equation}
Using~\ref{a.j.reg}, we obtain, for every~$m,n\in\N$ with~$n< m$,
\begin{equation}
	\label{e.k.ell.upscales.infty}
	\bigl\| \k_m - \k_n \bigr\|_{L^\infty(\cu_n)}
	\leq
	\sqrt{d}3^n \,\bigl\| \nabla(\k_m - \k_n) \bigr\|_{L^\infty(\cu_n)}
	+\bigl| (\k_m - \k_n) (0)\bigr|
	\leq
	\O_{\Gamma_2}\bigl(C(m-n)^{\nicefrac{1}{2}}\bigr) \,.
\end{equation}
By the assumption of~$\Rd$--stationarity, we obtain, for every~$l\in\Z$ and~$m,n\in\N$ with~$n< m$,
\begin{equation}
	\label{e.kmn.bounds}
	\fint_{\cu_l} \bigl| (\k_m - \k_n) (x)\bigr|^p \,dx
	\leq
	\O_{\Gamma_{\nicefrac{2}{p}}}\bigl( C^p(m-n)^{\nicefrac{p}{2}} \bigr) \,.
\end{equation}
Using the finite range of dependence assumption~\ref{a.j.frd} and Proposition~\ref{p.concentration}, we can improve the previous bound on scales larger than~$3^{m}$: for every~$l,m,n\in \N$ with~$n< m\leq l$,
\begin{equation}
	\label{e.kl.bounds.large}
	\fint_{\cu_l} \bigl| (\k_m - \k_n) (x)\bigr|^p \,dx
	\leq
	C(m-n)^{\nicefrac{p}{2}} +
	\O_{\Gamma_{\nicefrac{2}{p}}}\bigl( C^p (m-n)^{\nicefrac{p}{2}} 3^{-\frac d2(l-m)}\bigr)
	\,.
\end{equation}
After taking the~$p$th root and using~\eqref{e.powerofGammasigma}, this yields~\eqref{e.kmn.Lp}. To obtain~\eqref{e.kmn.Linfty} we return to~\eqref{e.k.ell.upscales.infty} and apply a union bound: for every~$t \geq 1$
\begin{align*}
	&\P\Bigl[ \bigl\| \k_m - \k_n \bigr\|_{L^\infty(\cu_l)} >   C (m-n)^{\nicefrac{1}{2}} (l-n)^{\nicefrac{1}{2}} t\Bigr ]
	\\
	&=
	\P\biggl[
	\sup_{z\in 3^n\Zd \cap \cu_l} \bigl\| \k_m - \k_n \bigr\|_{L^\infty(z+\cu_n)}>   C(m-n)^{\nicefrac{1}{2}} (l-n)^{\nicefrac{1}{2}} t
	\biggr]
	\\
	&
	\leq
	\sum_{z\in 3^n\Zd \cap \cu_l}
	\P\Bigl[ \bigl\|\k_m - \k_n \bigr\|_{L^\infty(z+\cu_n)} >   C(m-n)^{\nicefrac{1}{2}} (l-n)^{\nicefrac{1}{2}} t \Bigr ]
	\\
	&
	\leq
	3^{d(l-n)}
	\exp \bigl( -C t^2 (l-n)  \bigr)
	\leq
	\exp  ( -t^2 )
	\,,
\end{align*}
provided that~$C$ is large enough. Here~$l-n\geq1$ by the standing assumption~$n<m\leq l$, so the constant absorbs the endpoint without adding an extra unit to this factor. This implies~\eqref{e.kmn.Linfty}.

\smallskip

We next prove~\eqref{e.kmn.Hs} and~\eqref{e.kmn.Hs.osc}. We have that
\begin{align*}
	[ \mathbf{j}_k ]_{\underline{H}^s(\cu_l)} ^2
	\leq
	C \fint_{\cu_l} \int_{\cu_l}
	\frac{|\mathbf{j}_k(x) - \mathbf{j}_k(y)|^2}{|x-y|^{d+2s}}
	\, dx \, dy
	\leq
	\O_{\Gamma_1}
	\biggl(
	\underbrace{
	C \fint_{\cu_l} \int_{\cu_l}
	\frac{\min\{ 1 , 3^{-k} |x-y| \}^2}{|x-y|^{d+2s}}
	\, dx \, dy
	}_{
	\leq C 3^{-2sk}
	}
	\biggr)
\end{align*}
for every~$l,k\in\N$. Hence
\begin{equation*}
	\bigl[ \mathbf{j}_k \bigr]_{\underline{H}^s(\cu_l)}
	\leq
	\O_{\Gamma_2} \bigl( C3^{-sk}  \bigr)
\end{equation*}
for every~$l,k\in\N$. The triangle inequality and Lemma~\ref{l.Gamma.sigma.triangle} therefore give
\begin{equation*}
	\sup_{M\in\N \,:\, M>n}
	\bigl[ \k_M - \k_n \bigr]_{\underline{H}^s(\cu_l)}
	\leq
	\sum_{k=n+1}^\infty
	\bigl[ \mathbf{j}_k \bigr]_{\underline{H}^s(\cu_l)}
	\leq
	\O_{\Gamma_2} \bigl( C3^{-sn}  \bigr)\,,
\end{equation*}
which proves~\eqref{e.kmn.Hs.osc}. Also, if~$l\geq k$, we obtain
\begin{equation*}
	\|  \mathbf{j}_k \|_{\underline{H}^s(\cu_l)}
	\leq
	\bigl( 3^{-2sl}
	\|  \mathbf{j}_k \|_{\underline{L}^2(\cu_l)}^2
	+
	[ \mathbf{j}_k ]_{\underline{H}^s(\cu_l)} ^2
	\bigr)^{\nf12}
	\leq
	\O_{\Gamma_2} \bigl( C 3^{-sk}  \bigr)\,.
\end{equation*}
The triangle inequality and the previous display yield
\begin{equation*}
	\bigl\| \k_m - \k_n \bigr\|_{\underline{H}^s(\cu_l)}
	\leq
	\sum_{k=n+1}^m
	\bigl\| \mathbf{j}_k \bigr\|_{\underline{H}^s(\cu_l)}
	\leq
	\O_{\Gamma_2} \bigl( C3^{-sn}  \bigr)\,,
\end{equation*}
which completes the proof of~\eqref{e.kmn.Hs}.

\smallskip

Turning to the proof of~\eqref{e.kbounds.minscale.form} and~\eqref{e.yet.another.minscale.icandoit}, set~$X_0\coloneqq0$ and, for each~$m\in\N$ with~$m\geq1$, define
\begin{align}
	X_m & \coloneqq  m^{-1}
	\bigl\| \k - (\k)_{\cu_m} \bigr\|_{L^\infty(\cu_m)}
	+
	3^m \| \nabla \k - \nabla \k_m \|_{L^\infty(\cu_m)}
	+
	3^{-s m} [ \k - (\k)_{\cu_m} ]_{\Hminushat{-s}(\cu_m)}
	\notag \\ & \qquad
	+ \max_{1 \leq h \leq m} h^{-1} \max_{n \in [m-h, m] \cap \N} \max_{z \in 3^n \Zd \cap \cu_m} |(\k)_{z + \cu_n} - (\k)_{\cu_m}|\,.
	\label{e.Xm.deff}
\end{align}
We claim that there exist constants~$c,C>0$ such that, for every~$m\geq1$ and~$t\geq 1$,
\begin{equation}
	\label{e.Xm.boundiu}
	\P\bigl[ X_m > Ct \bigr] \leq C m \exp(-c t^2) \,.
\end{equation}
The first two terms on the right side of~\eqref{e.Xm.deff} are bounded by~$\O_{\Gamma_2}(C)$ using~\eqref{e.kmn.Linfty} and~\eqref{e.nabla.kmn.Linfty}, applied to finite cutoffs and then passed to the limit. To estimate the negative Sobolev term and the logarithmic-window term, we claim that for every~$n\in\N$ with~$n\leq m$ and every~$z\in 3^n\Zd$ with~$z+\cu_n \subseteq \cu_m$, we have
\begin{equation}
	\label{e.bounding.something.that.is.more.complicated.than.it.seems}
	\bigl|(\k)_{z+\cu_n} - (\k)_{\cu_m} \bigr| \leq \O_{\Gamma_2}(C(m-n)^{\nicefrac 12}) \,.
\end{equation}
Indeed, for every~$M>m$, write
\begin{align*}
	(\k_M)_{z+\cu_n}-(\k_M)_{\cu_m}
	&=
	(\k_n)_{z+\cu_n}-(\k_n)_{\cu_m}
	+
	(\k_m-\k_n)_{z+\cu_n}-(\k_m-\k_n)_{\cu_m}
	\notag \\ &\qquad
	+
	(\k_M-\k_m)_{z+\cu_n}-(\k_M-\k_m)_{\cu_m}\,.
\end{align*}
The first term is~$\O_{\Gamma_2}(C)$ by~\eqref{e.jk.spatialavg} and Lemma~\ref{l.Gamma.sigma.triangle}. The second term is~$\O_{\Gamma_2}(C(m-n)^{\nf12})$ by applying Proposition~\ref{p.concentration} to the independent centered increments~$\mathbf j_{n+1},\ldots,\mathbf j_m$, using~\eqref{e.jk.spatialavg}. For the last term, the two averages differ by at most the oscillation of~$\k_M-\k_m$ on~$\cu_m$, and therefore by~\eqref{e.nabla.kmn.Linfty},
\begin{equation*}
	\bigl|(\k_M-\k_m)_{z+\cu_n}-(\k_M-\k_m)_{\cu_m}\bigr|
	\leq
	C3^m\|\nabla(\k_M-\k_m)\|_{L^\infty(\cu_m)}
	\leq
	\O_{\Gamma_2}(C)\,.
\end{equation*}
Letting~$M\to\infty$ gives~\eqref{e.bounding.something.that.is.more.complicated.than.it.seems}; this uses the local uniform convergence of the centered sums~$\k_M-(\k_M)_{\cu_m}$, which follows from~\eqref{e.good.k},~\eqref{e.nabla.kmn.Linfty}, and the summability of the tail oscillations. For~$n<0$, the trivial bound~$|(\k)_{z+\cu_n}-(\k)_{\cu_m}|\leq \|\k - (\k)_{\cu_m}\|_{L^\infty(\cu_m)}$ together with the first term of~\eqref{e.Xm.deff} gives~$|(\k)_{z+\cu_n}-(\k)_{\cu_m}| \leq \O_{\Gamma_2}(Cm)$, so the corresponding tail of the sum below is bounded by~$Cm\sum_{n=-\infty}^{-1}3^{-s(m-n)}\leq C(s)m\, 3^{-sm}$, which is~$\O_{\Gamma_2}(C(s,d))$ uniformly in~$m\geq 0$. Consequently,
\begin{align*}
	3^{-s m} [\k - (\k)_{\cu_m} ]_{\Hminushat{-s}(\cu_m)}
	&\leq C(s,d) \sum_{n=-\infty}^m 3^{-s (m-n)} \biggl( \avsum_{z \in 3^n \Zd \cap \cu_m} \bigl|(\k)_{z+\cu_n} - (\k)_{\cu_m}\bigr|^2 \biggr)^{\nf12}
	\notag \\ &
	\leq
	\O_{\Gamma_2}(C(s,d))
	\,.
\end{align*}
Also by~\eqref{e.bounding.something.that.is.more.complicated.than.it.seems} and~\eqref{e.maxy.bound}, for every~$h \in \{1,\ldots,m\}$ we have
\begin{equation}
	\label{e.bounding.the.diff.of.k.union}
	\max_{n \in [m-h, m] \cap \N} \max_{z \in 3^n \Zd \cap \cu_m} |(\k)_{z + \cu_n} - (\k)_{\cu_m}| \leq \O_{\Gamma_2}(C h) \,.
\end{equation}
Taking a union bound over~$h\in\{1,\ldots,m\}$ and combining the preceding estimates gives~\eqref{e.Xm.boundiu}.

\smallskip

For each~$\sigma\in(0,\infty)$ we define~$\mathcal{K}_\sigma \coloneqq  3^3 \vee \sup\{3^{m+1}  : m \in \N\,,\, X_m > \frac12 \delta m^\sigma\}$. Set also~$N_\sigma  \coloneqq  \lceil (C \sigma^{-1} \delta^{-1} )^{\nicefrac{1}{\sigma}} \rceil$. By a union bound and~\eqref{e.Xm.boundiu}, together with a straightforward computation, we deduce that, if the constant~$C$ in the definition of~$N_\sigma$ is a large enough universal constant, then
\begin{equation*}
	\P \Bigl[ \log \mathcal{K}_\sigma > (\log 3) N_\sigma m  \Bigr]
	\leq
	\sum_{n= \lfloor N_\sigma m \rfloor}^\infty \P\bigl[ X_{ n} > \frac12 \delta n^{\sigma}  \bigr]
	\leq
	\sum_{n= \lfloor N_\sigma m \rfloor}^\infty Cn \exp \bigl ( -c \delta^2 n^{2\sigma}  \bigr)
	\leq
	\exp(- cm^{2\sigma} )
	\,.
\end{equation*}
If~$3^m \geq \mathcal{K}_\sigma$, then~$m\geq 3$ and~$X_m \leq \frac12 \delta m^\sigma$. This immediately yields~\eqref{e.kbounds.minscale.form}. To get~\eqref{e.yet.another.minscale.icandoit}, fix~$A,B\geq 1$ and set~$h\coloneqq \lceil A\log(Bm)\rceil$. Since~$m\geq 3$, we have~$h \leq 2A\log(Bm)$. If~$h=0$, the assertion is trivial. Otherwise, using the last term in~\eqref{e.Xm.deff} with~$h\wedge m$ gives
\begin{equation*}
	\max_{n \in [m-h, m] \cap \N} \max_{z \in 3^n \Zd \cap \cu_m} |(\k)_{z + \cu_n} - (\k)_{\cu_m}|
	\leq
	\frac12 \delta m^\sigma (h\wedge m)
	\leq
	A\log(Bm) \delta m^\sigma
	\,.
\end{equation*}
This completes the proof of~\eqref{e.kbounds.minscale.form} and~\eqref{e.yet.another.minscale.icandoit}.

\smallskip

Finally,~\eqref{e.a.ellipticity.pointwise} follows from~\eqref{e.kbounds.minscale.form} by a triadic telescoping argument: choose~$r$ with~$3^r\simeq \mathcal K_\sigma+|x|$ and telescope~$\k(x)-\k(0)$ through the averages on the nested cubes between the scale~$\mathcal K_\sigma$ and~$3^r$, using the first term of~\eqref{e.kbounds.minscale.form} for the endpoint oscillations and~\eqref{e.yet.another.minscale.icandoit} for the intermediate averages. The proof is complete.
\end{proof}

We denote the coarse-grained matrices for the infrared cutoff~$\a_m$ with the subscript~$m$; that is,
\begin{equation*}
	\s_m(U) \coloneqq  \s(U;\a_m)\,, \quad
	\s_{m,*}(U) \coloneqq  \s_*(U;\a_m)\,, \quad
	\kcg_m(U)  \coloneqq  \k(U;\a_m)\,, \quad
	\bfA_m(U)  \coloneqq  \bfA (U; \a_m)
	\,
\end{equation*}
Thus~$\k_m$ denotes the cutoff field, while~$\kcg_m(U)$ denotes the coarse-grained matrix associated to~$\a_m$ on~$U$. We also write~$J_m (U,\cdot)  \coloneqq  J (U,\cdot;\a_m )$,~$\A_m(U) \coloneqq  \A(U;\a_m)$,~$\a_{m,*}(U) \coloneqq \a_*(U;\a_m)$, and so forth. We also denote the infrared cutoff coarse-grained coefficient fields~$\hat{\a}_{L,n}$ by
\begin{equation}
	\label{e.CG.field.withcutoff}
	\hat{\a}_{L,n}
	\coloneqq  \sum_{z\in 3^n\Zd}
	\a_{L,*}(z+\cu_n) \indc_{z+\cu_n}\,.
\end{equation}
Associated to these coarse-grained matrices are \emph{annealed}  matrices~$\bfAhom_m(U)$,~$\shom_m(U)$,~$\shom_{m,*}(U),~\khom_m(U)$. These are defined by
\begin{equation}
	\label{e.homs.defs.U.0}
	\bfAhom_m (U)  \coloneqq
	\begin{pmatrix}
		( \shom_m + \khom_m^t\shom_{m,*}^{-1}\khom_m )(U)
		& -(\khom_m^t\shom_{m,*}^{-1})(U)
		\\ - ( \shom_{m,*}^{-1}\khom_m )(U)
		& \shom_{m,*}^{-1}(U)
	\end{pmatrix}
	\coloneqq
	\E \bigl[ \bfA_m(U) \bigr]\,
	\,.
\end{equation}
In fact, using that~$\a_m$ has the same law as~$\a_m^t$, and that
\begin{equation*}
	\k(U;\a^t) = - \k(U;\a)\,,
\end{equation*}
which is immediate from~\eqref{e.J.mat} and~\eqref{e.J.mat.star}, we deduce that~$\kcg_m(U)$ has the same law as~$-\kcg_m(U)$. In particular,~$\khom_m(U) = 0$, which allows us to rewrite the matrix in~\eqref{e.homs.defs.U.0} as
\begin{equation}
	\label{e.homs.defs.U}
	\bfAhom_m (U)  \coloneqq
	\begin{pmatrix}
		\shom_m(U)
		& 0
		\\ 0
		& \shom_{m,*}^{-1}(U)
		\end{pmatrix} \,.
\end{equation}
The dihedral symmetry assumption~\ref{a.j.iso} also implies that each of the diagonal blocks of~$\bfAhom_m(\cu_n)$ is a scalar matrix.

\smallskip

By subadditivity, the matrices~$\bfAhom_m(\cu_n)$ are monotone nonincreasing in~$n$. We also define the deterministic matrices~$\bfAhom_m$ and~$\shom_m$ as the infinite-volume limits of these:
\begin{equation}
	\label{e.homs.defs}
	\bfAhom_m  \coloneqq
	\begin{pmatrix}
		\shom_m
		& 0
		\\ 0
		& \shom_{m}^{-1}
	\end{pmatrix}
	\coloneqq  \lim_{n\to \infty} \E \bigl[ \bfA_m(\cu_n) \bigr]\,.
\end{equation}
The fact that~$\shom_{m,*}^{-1}(\cu_n)$ converges to~$\shom_{m}^{-1}$ as~$n\to\infty$ is a consequence of qualitative homogenization for the field~$\nu\Id + \k_m$ (see~\cite[Proposition D.2 \& Theorem 3.1]{AK.HC}).

\smallskip

Next, by stationarity,~\ref{a.j.indy},~\ref{a.j.reg}, Proposition~\ref{p.concentration}, and~\eqref{e.powerofGammasigma}, for every~$x\in\Rd$,
\begin{equation*}
	|\k_m(x)|^2
	\leq
	\O_{\Gamma_1}(C(1\vee m))\,.
\end{equation*}
Jensen's inequality therefore gives, for every bounded Lipschitz domain~$U$,
\begin{equation}
	\|  \k_m \|_{\underline{L}^2(U)}^2
	\leq
	\O_{\Gamma_1}(C(1\vee m))\,.
	\label{e.km.Ltwo.size}
\end{equation}
We also define, for each~$m\in\N$,
\begin{equation}
	\bfE_m
	\coloneqq
	\begin{pmatrix} \bigl(\nu  +  2 C_{\text{\eqref{e.km.Ltwo.size}}}  \nu^{-1}(1\vee m)\bigr)\Id & 0 \\ 0 & 2 C_{\text{\eqref{e.km.Ltwo.size}}} \nu^{-1} \Id \end{pmatrix}
	\label{e.Enaught.mixing}
	\,.
\end{equation}
The linear dependence on~$m$ reflects the bound on the second moment of the cutoff field~$\k_m$. We use this bound to normalize the estimates by~$\bfE_m$.

\begin{lemma}[{Ellipticity bounds for~$\bfA_m$}]
\label{l.bfAm.ellip}
For every~$m \in \N$ and every bounded Lipschitz domain~$U$,
\begin{equation}
	\label{e.Enaught.vs.A.and.Ahom}
	\bigl| \bfE_m^{-\nf12} \bfA_m(U) \bfE_m^{-\nf12}\bigr|   \leq \O_{\Gamma_1}(1)
\end{equation}
Moreover, for every~$n \in \N$,
\begin{equation*}
	\bfE_m^{-\nf12} \bfAhom_m \bfE_m^{-\nf12}
	\leq
	\bfE_m^{-\nf12} \bfAhom_m(\cu_n) \bfE_m^{-\nf12}
	\leq \mathrm{I}_{2d}
	\,.
\end{equation*}
Moreover, for each~$m\in\N$ and~$\gamma \in (0,1)$, there exists a random minimal scale~$\S_{m,\gamma}$ satisfying
\begin{equation}
	\label{e.Smgamma.integ}
	\S_{m,\gamma} = \O_{\Gamma_\gamma}\biggl(C\exp\biggl(\frac{C|\log\gamma|}{\gamma}\biggr)3^m\biggr)
\end{equation}
such that, for every~$n\in\N$,
\begin{equation}
	\label{e.bfAm.ellip}
	3^n\geq \S_{m,\gamma}
	\implies
	3^{-\gamma(n-l)}\bfE_m^{-\nf12}
	\bfA_m (z+\cu_l)
	\bfE_m^{-\nf12} \leq 2 \Itwod \,, \  \forall l\in \Z \cap (-\infty , n], \ z \in 3^l\Zd\cap\cu_n\,.
\end{equation}
\end{lemma}
\begin{proof}
Fix~$m\in\N$. Using~\eqref{e.CG.bounds.2} and~\eqref{e.how.to.upbound.A}, we compute, for every bounded Lipschitz domain~$U$,
\begin{align*}
	\bfA_m(U)
	\leq
	\fint_U
	\bfA_m(x)\,dx
	&
	\leq
	\fint_U
	\begin{pmatrix}
		\nu\Id + 2\nu^{-1} \k_m^t(x)\k_m(x)
		& 0
		\\ 0
		& 2\nu^{-1} \Id
		\end{pmatrix}\,dx
	\notag \\ &
	\leq
	\begin{pmatrix}
		(\nu + 2\nu^{-1} \|  \k_m \|_{\underline{L}^2(U)}^2 )\Id
		& 0
		\\ 0
		& 2\nu^{-1} \Id
	\end{pmatrix}
	\,.
\end{align*}
We thus deduce~\eqref{e.Enaught.vs.A.and.Ahom} by~\eqref{e.km.Ltwo.size}.  The upper bound on~$\bfAhom_m(\cu_n)$ follows by taking expectations in the previous deterministic block bound with~$U=\cu_n$, while the lower bound follows from the subadditivity of~$\bfAhom_m(\cu_n)$ and the definition of~$\bfAhom_m$ as its infinite-volume limit.

\smallskip

We next prove the local square bound for~$\k_m$. If~$l\geq m$, the estimate follows from~\eqref{e.kmn.Lp}, applied to~$\k_m-\k_0$, together with the unit-scale bound on~$\k_0$. If~$0\leq l<m$, stationarity and~\eqref{e.k.ell.upscales.infty}, applied on translates of~$\cu_l$, give the same estimate without the large-scale averaging factor. If~$l<0$, the same unit-scale bound and the Lipschitz control supplied by~\ref{a.j.reg} give the subunit-scale estimate. Thus, for every~$z\in\Rd$ and~$l\in\Z$,
\begin{equation}
	\label{e.km.square.bound}
	(C(1\vee m))^{-1} \|  \k_m \|_{\underline{L}^2(z+\cu_l)}^2 \leq 1 + \O_{\Gamma_1} (C3^{-(\frac d2 (l-m)\vee 0) } )\,.
\end{equation}

\smallskip

We also deduce from~\eqref{e.km.square.bound} that, for every~$t \geq 1$,~$\gamma\in(0,1)$ and~$n\in\N$ with~$n\geq \max\{ m,l\}$,
\begin{equation*}
	\P \Bigl[ (C(1\vee m))^{-1} \|  \k_m \|_{\underline{L}^2(\cu_l)}^2 > 2t3^{\gamma (n-l)} \Bigr]
	\leq \exp \bigl( -ct 3^{\gamma (n-l)+(\frac d2(l-m)\vee 0)} \bigr) \,.
\end{equation*}
By a union bound and stationarity, for every~$t \geq 1$, for every~$n\in\N$,
\begin{align*}
	&\P \Bigl[
	\exists l\in \Z\cap (-\infty, n]\,, \ \
	\exists z\in 3^l\Zd \cap \cu_n\, , \ \
	\bfE_m^{-\nf12}
	\bfA_m (z+\cu_l)
	\bfE_m^{-\nf12}  \not\leq 2t3^{\gamma (n-l)} \Itwod \Bigr]
	\\
	&
	\leq
	\sum_{l = -\infty}^n
	\sum _{z\in 3^l\Zd \cap \cu_n}
	\P \Bigl[
	\bfE_m^{-\nf12}
	\bfA_m (z+\cu_l)
	\bfE_m^{-\nf12}  \not\leq 2t3^{\gamma (n-l)} \Itwod \Bigr]
	\\
	&
	\leq
	\sum_{l = -\infty}^n
	3^{d(n-l)}
	\P \Bigl[
	\bfE_m^{-\nf12}
	\bfA_m (\cu_l)
	\bfE_m^{-\nf12}  \not\leq 2t3^{\gamma (n-l)} \Itwod \Bigr]
	\notag \\ &
	\leq
	\sum_{l = -\infty}^n
	3^{d(n-l)}
	\exp \bigl( -ct 3^{\gamma (n-l)+(\frac d2(l-m)\vee 0)} \bigr)
	\notag \\ &
	\leq
	\exp\biggl(
	\frac{C\left| \log \gamma \right| }{\gamma}
	\biggr)
	\exp\bigl(
	-c 3^{\gamma (n-m)}
	\bigr)
	\,.
\end{align*}
Taking~$t=1$, we next define
\begin{align*}
	\S_{m,\gamma} &
	\coloneqq
	\sup
	\biggl(
	\biggl\{
	3^m
	\biggr\}
	\cup
	\biggl\{
	3^{n+1} \, : \, n\in \N \cap [m,\infty)
	\ \
	\text{and there exists} \ \
	l\in \Z\cap (-\infty, n]\,, \,
	z\in 3^l\Zd \cap \cu_n\,,
	\\ & \qquad\qquad \qquad\qquad\qquad\qquad\qquad
	\text{such that} \ \
	\bfE_m^{-\nf12}
	\bfA_m (z+\cu_l)
	\bfE_m^{-\nf12}  \not\leq 2\,3^{\gamma (n-l)} \Itwod
	\biggr\}
	\biggr)\,.
\end{align*}
Then~$\S_{m,\gamma}$ satisfies the conclusion of the lemma.
\end{proof}

\subsection{Localization}
\label{ss.localization}
We observe in this subsection that the matrix~$\G_{\h}$ defined in~\eqref{e.G} and the infrared cutoffs of the previous subsection can be used to estimate the cost of localizing, that is, of lowering the infrared cutoff of the coarse-grained matrices. Localization permits independence to produce stochastic cancellations.

\smallskip

We will use the following properties of~$\G_\h$. First, it satisfies
\begin{equation}
	\label{e.G.group}
	\G_{\h_1} \G_{\h_2} = \G_{\h_1+\h_2} \,, \quad \forall \h_1,\h_2\in\R^{d\times d} \,.
\end{equation}
Second, we have that, for any~$\s_j,\s_{j,*} \in\R^{d\times d}_{\sym,+}$ and~$\k_j \in \R^{d\times d}$,~$j \in \{1,2\}$,
\begin{equation}
	\label{e.bfG.factoring}
	\bfA_j
	\coloneqq
	\begin{pmatrix}
		\s_j + \k_j^t\s_{j,*}^{-1}\k_j
		& -\k_j^t\s_{j,*}^{-1}
		\\ - \s_{j,*}^{-1}\k_j
		& \s_{j,*}^{-1}
	\end{pmatrix}
	=
	\G_{-\k_j} ^t
	\begin{pmatrix}
		\s_j
		& 0
		\\ 0
		& \s_{j,*}^{-1}
	\end{pmatrix}
	\G_{-\k_j}
	\,.
\end{equation}
Third, for every~$\bfA,\bfB \in \R^{2d\times 2d}_{\sym,+}$, simultaneous congruence by~$\G_{\h}$ preserves the eigenvalues of the normalized matrix: if we denote~$\bfA_\h  \coloneqq  \G_{\h}^t \bfA \G_{\h}$ and~$\bfB_\h \coloneqq \G_{\h}^t \bfB \G_{\h}$, then
\begin{equation}
	\label{e.bfG.doesnt.mess.up.ratios}
	\bfB_{\h}^{-\nf12} \bfA_{\h} \bfB_{\h}^{-\nf12}
	\quad \mbox{and} \quad
	\bfB^{-\nf12} \bfA \bfB^{-\nf12}
	\quad
	\mbox{have the same eigenvalues.}
\end{equation}
This is because~$\bfB_{\h}^{-1}\bfA_\h$ and~$\bfB^{-1}\bfA$ are similar matrices and thus have the same eigenvalues, and~$\bfB_{\h}^{-\nf12}\bfA_{\h}\bfB_{\h}^{-\nf12}$ and~$\bfB^{-\nf12}\bfA\bfB^{-\nf12}$ have the same eigenvalues as~$\bfB_{\h}^{-1}\bfA_\h$ and~$\bfB^{-1}\bfA$, respectively.

\smallskip

We next compare the coarse-grained matrices corresponding to two different fields.

\begin{lemma}[Localization of~$\bfA$]
\label{l.localization.A}
Let~$\h,\k$ be anti-symmetric matrix fields,~$\s$ a symmetric positive definite matrix field and~$\a \coloneqq \s + \k$. Then, for every bounded Lipschitz domain~$U \subseteq \Rd$, denoting
\begin{equation*}
	D \coloneqq
	\| \s^{-\nf 12}\h \s^{-\nf 12} \|_{L^\infty(U)}
	\bigl(1+\| \s^{-\nf 12}\h \s^{-\nf 12} \|_{L^\infty(U)}\bigr)
	\,,
\end{equation*}
we have the estimate
\begin{equation}
	\label{e.localization.A}
	-D \bfA (U\,;\a)
	\leq
	\bfA (U \, ; \a + \h) - \bfA (U\,;\a)
	\leq
	D \bfA (U\,;\a)
	\,.
\end{equation}
\end{lemma}

\begin{proof}
We abbreviate
\begin{equation*}
	\bfA_0(x) \coloneqq
	\begin{pmatrix}
		\s(x) & 0
		\\
		0 & \s^{-1}(x)
	\end{pmatrix}
	\,.
\end{equation*}
Set
\begin{equation*}
	\bfA(U) \coloneqq \bfA(U\, ; \a)\,, \quad
	\hat{\a} \coloneqq \a + \h
	\qand
	\hat{\bfA}(U) \coloneqq \bfA(U;\hat{\a})\,.
\end{equation*}
For each~$x \in U$, let~$\bfA(x)$ and~$\hat{\bfA}(x)$ denote the block coefficient matrices associated to~$\a$ and~$\hat{\a}$ by~\eqref{e.bfA.def}. Since
\begin{equation*}
	\bfA(x)
	=
	\G_{-\k(x)}^t \bfA_0(x) \G_{-\k(x)}
	\qand
	\hat{\bfA}(x)
	=
	\G_{-\k(x)-\h(x)}^t \bfA_0(x) \G_{-\k(x)-\h(x)}\,,
\end{equation*}
an application of~\eqref{e.G.group} and~\eqref{e.bfG.doesnt.mess.up.ratios} shows that the matrices
\begin{equation*}
	\bfA^{-\nf12}(x)\hat{\bfA}(x)\bfA^{-\nf12}(x)
	\qand
	\bfA_0^{-\nf12}(x)
	\G_{-\h(x)}^{t}
	\bfA_0(x)
	\G_{-\h(x)}
	\bfA_0^{-\nf12}(x)
\end{equation*}
have the same eigenvalues. We compute
\begin{equation*}
	\bfA_0^{-\nf12}(x)
	\G_{-\h(x)}^{t}
	\bfA_0(x)
	\G_{-\h(x)}
	\bfA_0^{-\nf12}(x)
	- \Itwod
	=
	\begin{pmatrix}
		\s^{-\nf 12}\h^t \s^{-1}\h \s^{-\nf 12}
		& -\s^{-\nf 12} \h^t \s^{-\nf 12}
		\\
		- \s^{-\nf 12} \h \s^{-\nf 12}
		& 0
		\end{pmatrix}(x)\,.
\end{equation*}
Taking the norm and using Young's inequality, we obtain
\begin{equation}
	\label{e.ratio.fields}
	\bigl\|
	\bfA^{-\nf12}\hat{\bfA} \bfA^{-\nf12}  - \Itwod
	\bigr\|_{L^\infty(U)}
	\leq
	D
	\,.
\end{equation}
Using the variational characterization~\eqref{e.J.P0.Dirichlet}, we infer that, for every~$P \in \R^{2d}$,
\begin{align}
	\label{e.localize.matrix.bounds.in.lemma}
	P\cdot \hat{\bfA}(U) P
	&
	=
	\inf \biggl\{
	\fint_{U} (X+P) \cdot \hat{\bfA} (X+P) \, : \, X \in \Lpoto(U) \times \Lsolo(U)
	\biggr\}
	\notag \\ &
	\leq
	(1 + D)
	\inf \biggl\{
	\fint_{U} (X+P) \cdot  \bfA (X+P) \, : \, X \in \Lpoto(U) \times \Lsolo(U)
	\biggr\}
	\notag \\ &
	\leq
	\bigl(1 + D \bigr)
	\bigl( P\cdot  \bfA(U) P \bigr)
	\,.
\end{align}
Applying the same argument with~$\a$ and~$\hat{\a}$ interchanged, we also obtain
\begin{equation}
	\label{e.localize.matrix.bounds.in.lemma.flip}
	P\cdot \bfA(U) P
	\leq
	\bigl( 1 + D \bigr)
	\bigl( P\cdot \hat{\bfA}(U) P   \bigr)\,.
\end{equation}
The previous two displays yield
\begin{equation*}
	(1+D)^{-1}\bfA(U)
	\leq
	\hat{\bfA}(U)
	\leq
	(1+D)\bfA(U)
	\,,
\end{equation*}
which implies~\eqref{e.localization.A}.
\end{proof}

\smallskip

The preceding lemma is a deterministic perturbative statement. We now apply it to the infrared cutoffs, where the perturbation is the oscillatory tail~$\k_L-\k_m$ after its average on the test domain has been removed.

\begin{lemma}[Localization]
\label{l.localization}
There exists a constant~$C(d)<\infty$ such that, for every~$m,n,L\in\N$ with~$n \leq m \leq L$ and Lipschitz domain~$U \subseteq \cu_n$, we have the estimates
\begin{equation}
	\label{e.localization.s.star}
	\bigl| \s_{L}^{-\nf12}(U) \s_{m}(U) \s_{L}^{-\nf12}(U)  - \Id \bigr|
	+
	\bigl| \s_{m,*}^{\nf12}(U) \s_{L,*}^{-1}(U) \s_{m,*}^{\nf12}(U) - \Id \bigr|
	\leq
	\O_{\Gamma_1}(C\nu^{-2}3^{-(m-n)})
\end{equation}
and
\begin{equation}
	\label{e.skbounds}
	\bigl| \s_{L,*}^{-\nf12}(U)(\kcg_L(U) - \kcg_m(U) - (\k_L - \k_m)_U) \s_{L}^{-\nf12}(U)       \bigr|
	\leq
	\O_{\Gamma_{1}}(C\nu^{-2}3^{-\frac12(m-n)})
	\,.
\end{equation}
Moreover, for every~$p,q \in \Rd$, we have
\begin{align}
	\lefteqn{
	\bigl\| \nabla v(\cdot,U,p,q\,; \a_L- (\k_L-\k_{m})_{U}) - \nabla v(\cdot,U,p,q\,; \a_{m}  )  \bigr\|_{\underline{L}^2(U)}^2
	} \qquad &
	\notag \\ &
	\leq
	C \nu^{-2} 3^n \bigl\| \nabla(\k_m-\k_{L})\bigr\|_{L^\infty(\cu_n)}
	\notag \\ & \qquad{}
	\times \bigl( J(U,p,q\,;\a_L-(\k_L-\k_{m})_{U}) + J(U,p,q\,;\a_{m}  ) + 2 p\cdot q \bigr)
	\,.
	\label{e.localization.minimizers}
\end{align}
\end{lemma}

\begin{proof}
\emph{Step 1.} We first record the matrix comparison that will be used to prove~\eqref{e.localization.s.star} and~\eqref{e.skbounds}. For convenience we denote
\begin{equation*}
	D \coloneqq  \nu^{-1}\| \k_L - \k_m - (\k_L-\k_m)_U \|_{L^\infty(U)} + \nu^{-2}\| \k_L - \k_m -(\k_L-\k_m)_U \|_{L^\infty(U)}^2  \,.
\end{equation*}
Since~$U \subseteq \cu_n$, we have
\begin{align}
	\label{e.Tsizebounds}
	\| \k_m-\k_L -(\k_m-\k_L)_{U}\|_{L^\infty(U)}
	&
	\leq
	\sup_{x,y\in U}|\k_m(x)-\k_L(x)-\k_m(y)+\k_L(y)|
	\notag \\ &
	\leq
	\sqrt{d}3^{n} \| \nabla (\k_m-\k_L) \|_{L^\infty(\cu_n)}
	\leq
	\O_{\Gamma_2} (C 3^{n-m})
	\,.
\end{align}
In the last line, we used~\eqref{e.nabla.kmn.Linfty}. Therefore
\begin{equation}
	\label{e.mclDsizebounds}
	D
	\leq
	\O_{\Gamma_2}(C\nu^{-1} 3^{n-m})+\O_{\Gamma_1}(C \nu^{-2} 3^{2(n-m)})
	\leq
	\O_{\Gamma_1}(C \nu^{-2}3^{n-m}) \,.
\end{equation}
Applying Lemma~\ref{l.localization.A} with
\begin{equation*}
	\hat\h \coloneqq \k_L-\k_m-(\k_L-\k_m)_U
	\qand
	\hat\a \coloneqq \a_m+\hat\h = \a_L-(\k_L-\k_m)_U
\end{equation*}
and then using~\eqref{e.commute.coarse.grained.k0}, we obtain
\begin{equation}
	\label{e.localization.ml}
	\Bigl|
	\bfA_m^{-\nf12}(U)
	\G_{(\k_L-\k_m)_U}^t \bfA_L (U)
	\G_{(\k_L-\k_m)_U} \bfA_m^{-\nf12}(U)
	-\Itwod\Bigr|
	\leq
	D
	\,.
\end{equation}

\smallskip
To prove~\eqref{e.localization.s.star}, set
\begin{equation*}
	\left\{
	\begin{aligned}
		\hat\a
		&\coloneqq
		\a_L-(\k_L-\k_m)_U \,,
		\\
		\widehat{\bfA}(U)
		&\coloneqq
		\bfA(U;\hat\a)
		=
		\G_{(\k_L-\k_m)_U}^t \bfA_L(U)\G_{(\k_L-\k_m)_U}
		\,,
		\\
		\h_U
		&\coloneqq
		\kcg_L(U) - \kcg_m(U) - (\k_L - \k_m)_U \,.
	\end{aligned}
	\right.
\end{equation*}
The deterministic block comparison to prove is
\begin{equation*}
	\widehat{\bfA}(U) \leq (1+D)\bfA_m(U)
	\qand
	\bfA_m(U)\leq (1+D)\widehat{\bfA}(U)\,.
\end{equation*}
In view of~\eqref{e.commute.coarse.grained.k0} and~\eqref{e.bigA.def}, applying the congruence transformation associated with~$\G_{\kcg_m(U)}$ to the first inequality yields
\begin{equation*}
	\begin{pmatrix}
		\s_L + \h_U^t \s_{L,*}^{-1} \h_U & -\h_U^t \s_{L,*}^{-1}
		\\
		-\s_{L,*}^{-1}\h_U & \s^{-1}_{L,*}
	\end{pmatrix}
	=
	\G_{-\h_U}^t
	\begin{pmatrix}
		\s_L  & 0 \\
		0 & \s^{-1}_{L,*}
	\end{pmatrix}
	\G_{-\h_U}
	\leq
	(1 + D)
	\begin{pmatrix}
		\s_m & 0 \\
		0 & \s^{-1}_{m,*}
	\end{pmatrix}
	\,.
\end{equation*}
Comparing the lower-right and upper-left blocks, we obtain
\begin{equation*}
	\s^{-1}_{L,*} \leq (1+D) \s^{-1}_{m,*}
\end{equation*}
and
\begin{equation}
	\label{e.lh.fs.matrix.cra.bound}
	\s_L \leq \s_L+ \h_U^t \s_{L,*}^{-1} \h_U
	\leq (1+D) \s_m \,,
\end{equation}
with the first inequality due to the fact~$\s_{L,*}^{-1}$ is positive definite and symmetric. An identical argument, using~\eqref{e.localize.matrix.bounds.in.lemma.flip} instead, shows that
\begin{equation*}
	\s^{-1}_{m,*} \leq (1+D) \s^{-1}_{L,*} \qand \s_m \leq (1+D) \s_L \,.
\end{equation*}
The previous three displays imply
\begin{equation*}
	(1+D)^{-1}\Id
	\leq
	\s_L^{- \nicefrac12} \s_m \s_L^{- \nicefrac12}
	\leq
	(1+D)\Id
\end{equation*}
and
\begin{equation*}
	(1+D)^{-1}\Id
	\leq
	\s_{m,*}^{\nf12} \s^{-1}_{L,*} \s_{m,*}^{\nf12}
	\leq
	(1+D)\Id\,.
\end{equation*}
Since the matrices in the previous two displays are symmetric, these bounds and~\eqref{e.mclDsizebounds} yield~\eqref{e.localization.s.star}. Similarly, using~\eqref{e.lh.fs.matrix.cra.bound} we have
\begin{equation*}
	\s_L^{-\nf12} \s_m \s_L^{-\nf12} - \Id \leq D\Id
\end{equation*}
and
\begin{equation*}
	\s_L^{-\nf12} \h_U^t \s_{L,*}^{-1} \h_U \s_L^{-\nf12}
	\leq
	(1+D)\s_L^{- \nicefrac12} \s_m \s_L^{-\nf12} - \Id
	\leq
	D \Id +  (1+D) ( \s_L^{-\nf12} \s_m \s_L^{-\nf12} -\Id )
	\leq D(2+D) \Id
	\,,
\end{equation*}
\begin{equation*}
	\bigl|
	\s_{L,*}^{-\nf12} \h_U \s_L^{-\nf12}
	\bigr|^2
	=
	\bigl|
	\s_L^{-\nf12} \h_U^t \s_{L,*}^{-1} \h_U \s_L^{-\nf12}
	\bigr|
	\leq
	D(2+D)
	\,.
\end{equation*}
Taking square roots and using~\eqref{e.mclDsizebounds}, we obtain~\eqref{e.skbounds}.

\smallskip

	\emph{Step 2.} Proof of~\eqref{e.localization.minimizers}. We prove a deterministic comparison for maximizers. Let~$\a=\s+\k$ be uniformly elliptic in~$U$, let~$\h$ be an anti-symmetric field, and set~$\tilde{\a}\coloneqq \a+\h$. Denote by~$\bfA$ and~$\tilde{\bfA}$ the corresponding block coefficient fields, and set
	\begin{equation*}
		\eta \coloneqq \|\s^{-\nf12}\h\s^{-\nf12}\|_{L^\infty(U)}\,.
	\end{equation*}
		Every step below uses only the variational characterizations of~$J$ and~$\bfA(U)$ and therefore applies directly to bounded measurable uniformly elliptic coefficients.

\smallskip

		Fix~$p,q\in\Rd$ and set~$P\coloneqq (-p,q)$. Let~$X$ and~$\tilde X$ be the minimizers in~\eqref{e.J.P0.Dirichlet} for~$\bfA$ and~$\tilde{\bfA}$, respectively, and write~$Z\coloneqq X+P$ and~$\tilde Z\coloneqq \tilde X+P$. By the equivalence of the block variational formulation with~\eqref{e.J.def} and~\eqref{e.Jstar.def}, there are functions
	\begin{equation*}
		u=v(\cdot,U,p,q;\a),\qquad
		\tilde u=v(\cdot,U,p,q;\tilde\a)\,,
	\end{equation*}
		together with adjoint maximizers~$u^*$ and~$\tilde u^*$, such that
	\begin{equation*}
		Z =
		\frac12
		\begin{pmatrix}
			\nabla u+\nabla u^*\\
			\a\nabla u-\a^t\nabla u^*
		\end{pmatrix}
		\qand
		\tilde Z =
		\frac12
		\begin{pmatrix}
			\nabla\tilde u+\nabla\tilde u^*\\
			\tilde\a\nabla\tilde u-\tilde\a^t\nabla\tilde u^*
			\end{pmatrix}\,.
	\end{equation*}
		Moreover,
	\begin{equation}
		\label{e.minimizers.energy.vs.bfA}
		\left\{
		\begin{aligned}
			& P\cdot\bfA(U;\a)P
			=
			\|\bfA^{\nf12}Z\|_{\underline L^2(U)}^2
			=
			\frac12\Bigl(
			\|\s^{\nf12}\nabla u\|_{\underline L^2(U)}^2
			+
			\|\s^{\nf12}\nabla u^*\|_{\underline L^2(U)}^2
			\Bigr)\,,
			\\
			& P\cdot\bfA(U;\tilde\a)P
			=
			\|\tilde{\bfA}^{\nf12}\tilde Z\|_{\underline L^2(U)}^2
			=
			\frac12\Bigl(
			\|\s^{\nf12}\nabla\tilde u\|_{\underline L^2(U)}^2
			+
			\|\s^{\nf12}\nabla\tilde u^*\|_{\underline L^2(U)}^2
			\Bigr)\,.
		\end{aligned}
		\right.
	\end{equation}
		If~$\eta>\nicefrac14$, the triangle inequality and~\eqref{e.minimizers.energy.vs.bfA} immediately imply
	\begin{equation}
		\label{e.minimizers.large.eta}
		\|\s^{\nf12}(\nabla u-\nabla\tilde u)\|_{\underline L^2(U)}^2 + \|\s^{\nf12}(\nabla u^*-\nabla\tilde u^*)\|_{\underline L^2(U)}^2 \leq C\eta\bigl(P\cdot\bfA(U;\a)P+P\cdot\bfA(U;\tilde\a)P\bigr)\,.
	\end{equation}
		We may therefore assume that~$\eta\leq\nicefrac14$. The pointwise computation in the proof of Lemma~\ref{l.localization.A} gives
	\begin{equation}
		\label{e.minimizers.pointwise.ratio}
		\bigl\|\bfA^{-\nf12}\tilde{\bfA}\bfA^{-\nf12}-\Itwod\bigr\|_{L^\infty(U)}
		+
		\bigl\|\tilde{\bfA}^{-\nf12}\bfA\tilde{\bfA}^{-\nf12}-\Itwod\bigr\|_{L^\infty(U)}
		\leq C\eta \,.
	\end{equation}
		Since~$Z-\tilde Z$ is an admissible zero-boundary variation for both minimization problems, the first variations give
	\begin{equation*}
		\|\bfA^{\nf12}(Z-\tilde Z)\|_{\underline L^2(U)}^2
		=
		\fint_U(\bfA\tilde Z\cdot\tilde Z-\bfA Z\cdot Z)
	\end{equation*}
		and the analogous identity with~$\bfA$ and~$\tilde{\bfA}$ interchanged. Combining the two identities with~\eqref{e.minimizers.pointwise.ratio}, we obtain
	\begin{equation}
		\label{e.minimizers.block.diff}
		\|\bfA^{\nf12}(Z-\tilde Z)\|_{\underline L^2(U)}^2
		+
		\|\tilde{\bfA}^{\nf12}(Z-\tilde Z)\|_{\underline L^2(U)}^2
		\leq
		C\eta\bigl(P\cdot\bfA(U;\a)P+P\cdot\bfA(U;\tilde\a)P\bigr)\,.
	\end{equation}
		Now write
	\begin{equation*}
		\tilde Z
		=
		\underbrace{
		\frac12
		\begin{pmatrix}
			\nabla\tilde u+\nabla\tilde u^*\\
			\a\nabla\tilde u-\a^t\nabla\tilde u^*
			\end{pmatrix}}_{\tilde Z_0}
		+
		\underbrace{
		\frac12
		\begin{pmatrix}
			0\\
			\h(\nabla\tilde u+\nabla\tilde u^*)
			\end{pmatrix}}_{\tilde Z_1}\,.
	\end{equation*}
		For every pair of vectors~$e,e^*\in\Rd$,
	\begin{equation}
		\label{e.iden.AP}
		\begin{pmatrix}e+e^*\\ \a e-\a^t e^*\end{pmatrix}
		\cdot
		\bfA
		\begin{pmatrix}e+e^*\\ \a e-\a^t e^*\end{pmatrix}
		=
		2e\cdot\s e+2e^*\cdot\s e^* \,.
	\end{equation}
		Applying this identity to~$Z-\tilde Z_0$ and using the triangle inequality gives
	\begin{align}
		\label{e.minimizers.gradient.from.block}
		&\|\s^{\nf12}(\nabla u-\nabla\tilde u)\|_{\underline L^2(U)}^2
		+
		\|\s^{\nf12}(\nabla u^*-\nabla\tilde u^*)\|_{\underline L^2(U)}^2
		\notag\\
		&\qquad\leq
		C\|\bfA^{\nf12}(Z-\tilde Z)\|_{\underline L^2(U)}^2
		+
		C\|\bfA^{\nf12}\tilde Z_1\|_{\underline L^2(U)}^2 \,.
	\end{align}
		By the definition of~$\eta$,
	\begin{align*}
		\|\bfA^{\nf12}\tilde Z_1\|_{\underline L^2(U)}^2
		&\leq
		C\|\s^{-\nf12}\h(\nabla\tilde u+\nabla\tilde u^*)\|_{\underline L^2(U)}^2
		\\
		&\leq
		C\eta^2
		\Bigl(
		\|\s^{\nf12}\nabla\tilde u\|_{\underline L^2(U)}^2
		+
		\|\s^{\nf12}\nabla\tilde u^*\|_{\underline L^2(U)}^2
		\Bigr)
		\\
		&\leq
		C\eta^2 P\cdot\bfA(U;\tilde\a)P \,.
	\end{align*}
Combining this with~\eqref{e.minimizers.block.diff},~\eqref{e.minimizers.gradient.from.block}, and~$\eta\leq\nicefrac14$, so that the~$C\eta^2$ contribution is absorbed into~$C\eta$, we get
	\begin{equation}
		\label{e.minimizers.deterministic}
		\|\s^{\nf12}(\nabla u-\nabla\tilde u)\|_{\underline L^2(U)}^2
		\leq
		C\eta\bigl(P\cdot\bfA(U;\a)P+P\cdot\bfA(U;\tilde\a)P\bigr)\,.
	\end{equation}
		Together with~\eqref{e.minimizers.large.eta}, this estimate is valid for all~$\eta\geq0$. By~\eqref{e.Jaas.matform}, we have
	\begin{equation*}
		P\cdot\bfA(U;\a)P + P\cdot\bfA(U;\tilde\a)P
		=
		2\bigl(J(U,p,q;\a)+J(U,p,q;\tilde\a)+2p\cdot q\bigr)\,,
	\end{equation*}
		and this factor~$2$ is absorbed into~$C$. We finally apply this estimate with~$\a=\a_m$ and~$\h=\k_L-\k_m-(\k_L-\k_m)_U$, so that~$\tilde\a=\a_L-(\k_L-\k_m)_U$ and~$\s=\nu\Id$. Since~$U\subseteq\cu_n$,
	\begin{equation*}
		\eta
		=
		\nu^{-1}\|\k_L-\k_m-(\k_L-\k_m)_U\|_{L^\infty(U)}
		\leq
		C\nu^{-1}3^n\|\nabla(\k_L-\k_m)\|_{L^\infty(\cu_n)}\,.
	\end{equation*}
		Dividing~\eqref{e.minimizers.deterministic} by~$\nu$ gives~\eqref{e.localization.minimizers}. The proof is complete.
\end{proof}

In the following lemma, matrix-valued~$\O_{\Gamma_\sigma}$ bounds are interpreted in operator norm, or equivalently entrywise after changing~$C(d)$; in one-sided matrix inequalities the displayed scalar random variable multiplies~$\Id$.

\begin{lemma}
\label{l.mixing.minscale}
There exists a constant~$C(d)<\infty$  such that, for every~$h,n,L \in \N$ with~$h<n\leq L$, we have
\begin{equation}
	\label{e.sstarL.quenched.lb}
	\s_{L,*}^{-1}(\cu_n)
	\leq
	\shom_{L,*}^{-1}(\cu_{h})
	+
	\O_{\Gamma_2} \bigl(C \nu^{-2}   3^{-\frac14(n-h)} \Id \bigr)
	\,.
\end{equation}
Moreover, for every~$h,n,\ell,L \in \N$ with~$h<n<\ell<L$ we have
\begin{equation}
	\label{e.refined.localization.twoo}
	\bigl| \s_{L,*}^{-\nf12}(\cu_n)  (\k_L - \k_\ell)_{\cu_n}  \s_{L,*}^{-\nf12}(\cu_n) \bigr| \leq \O_{\Gamma_2} \bigl( C  (L-\ell)^{\nf12} \shom_{L,*}^{-1}(\cu_{h})  \bigr) + \O_{\Gamma_{1}} \bigl( C\nu^{-2}   (L-\ell)^{\nf12} 3^{-\frac14(n-h)}  \bigr) \,.
\end{equation}
\end{lemma}
\begin{proof}
We will first show~\eqref{e.sstarL.quenched.lb}. By subadditivity,
\begin{equation*}
	\s_{L,*}^{-1}(\cu_n)
	\leq
	\shom_{L,*}^{-1}(\cu_{h}) + \avsum_{z' \in 3^{h} \Zd \cap \cu_n} \bigl(\s_{L,*}^{-1}(z'+\cu_{h})  - \shom_{L,*}^{-1}(\cu_{h}) \bigr)  \,.
\end{equation*}
To estimate the second term on the right, set~$n'\coloneqq \lceil\frac12(n+h)\rceil$. For~$Q=z'+\cu_h$, the trivial bound from~\eqref{e.CG.bounds.1} and the identity
\begin{equation*}
	\s_{L,*}^{-1}(Q)-\s_{n',*}^{-1}(Q)
	=
	\s_{n',*}^{-\nf12}(Q)
	\Bigl(
	\s_{n',*}^{\nf12}(Q)
	\s_{L,*}^{-1}(Q)
	\s_{n',*}^{\nf12}(Q)
	-\Id
	\Bigr)
	\s_{n',*}^{-\nf12}(Q)
\end{equation*}
give, by interpolation and~\eqref{e.localization.s.star},
\begin{align*}
	| \s_{L,*}^{-1}(Q) - \s_{n',*}^{-1}(Q) |
	&
	\leq
	C\nu^{-1}
	\min\Bigl\{ 1\,,
	\bigl| \s_{n',*}^{\nf12}(Q)
	\s_{L,*}^{-1}(Q)
	\s_{n',*}^{\nf12}(Q)
	-\Id \bigr| \Bigr\}
	\notag \\ &
	\leq
	\O_{\Gamma_2}\bigl( C\nu^{-2} 3^{-\frac12(n'-h)} \bigr)
	\leq
	\O_{\Gamma_2}\bigl( C\nu^{-2} 3^{-\frac14(n-h)} \bigr)
	\,.
\end{align*}
Similarly,
\begin{equation*}
	| \shom_{L,*}^{-1}(\cu_{h}) - \shom_{n',*}^{-1}(\cu_{h}) |
	\leq
	C\nu^{-2} 3^{-\frac14(n -h)} \,.
\end{equation*}
The claim~\eqref{e.sstarL.quenched.lb} now follows by partitioning~$3^h\Zd\cap\cu_n$ into~$O(3^{d(n'-h)})$ sublattices of spacing~$C3^{n'}$ and applying Proposition~\ref{p.concentration} with~$\sigma=2$ on each sublattice, using the bound~$ \s_{n',*}^{-1}(Q) \leq \nu^{-1} \Id$. The factor~$\nu^{-2}$ arises when converting the relative comparison~\eqref{e.localization.s.star} into the absolute bound~\eqref{e.sstarL.quenched.lb} on the inverse matrix. We use the relative comparison again to prove~\eqref{e.refined.localization.twoo}.

\smallskip

We next observe that, by~\eqref{e.jk.spatialavg} and~\ref{a.j.iso}, for every~$k \in\N$ we have
\begin{equation*}
	\bigl| ( \mathbf{j}_k )_{\cu_n} \bigr|
	\leq \O_{\Gamma_2}(1)
	\qand
	\E \bigl[ ( \mathbf{j}_k )_{\cu_n} \bigr] = 0\,,
\end{equation*}
and thus, by Proposition~\ref{p.concentration} with~$\sigma=2$, we obtain
\begin{equation*}
	|(\k_L - \k_\ell)_{\cu_n} |
	=
	\biggl| \sum_{i=\ell+1}^L
	(\mathbf{j}_{i} )_{\cu_n}  \biggr|
	\leq
	\O_{\Gamma_2} \bigl( C  (L-\ell)^{\nf12}  \bigr)
	\,.
\end{equation*}
Since~$\s_{L,*}^{-1}(\cu_n)$ is symmetric and positive definite, the one-sided estimate~\eqref{e.sstarL.quenched.lb}, together with the previous display, gives~\eqref{e.refined.localization.twoo} by the multiplication property of the~$\O_{\Gamma_\sigma}$ notation.
\end{proof}

\smallskip

The previous estimate is local in a single cube. We next average the gauged comparison over many~$3^n$-cubes, which is the form in which independence can later be used to extract cancellations.

\begin{lemma}[Averaged gauged comparison]
\label{l.localization.average}
Let~$m,n,\ell,L \in \N$ satisfy
\begin{equation*}
	n \leq \ell \leq m
	\qquad \mbox{and} \qquad
	L \geq \ell\,.
\end{equation*}
For each~$z \in 3^n\Zd \cap \cu_m$, set~$\h_z \coloneqq (\k_L-\k_\ell)_{z+\cu_n}$. Then, for every~$P \in \R^{2d}$,
\begin{align}
	\lefteqn{
	\biggl|
	\avsum_{z \in 3^n\Zd \cap \cu_m}
	P\cdot
	\bigl(
	\bfA_L(z+\cu_n) - \G_{-\h_z}^t \bfAhom_\ell(\cu_n)\G_{-\h_z}
	\bigr)
	P
	\biggr|
	} \qquad &
	\notag \\ &
	\leq
	\O_{\Gamma_{\nf13}}\Bigl( C \nu^{-3} |P|^2
	\bigl( \indc_{\{L>\ell\}} L 3^{-(\ell-n)}
	+ (1\vee\ell)(1\vee L)(1\vee(m-n)) 3^{-\frac d2(m-\ell)} \bigr) \Bigr)
	\,,
	\label{e.localization.average.oneshot}
\end{align}
for a constant~$C(d)<\infty$.
\end{lemma}

\begin{proof}
We first establish a deterministic comparison. Let
\begin{equation*}
	\h_z \coloneqq
	(\k_L-\k_\ell)_{z+\cu_n}
	\qquad \mbox{for every~$z \in 3^n\Zd \cap \cu_m$\,,}
\end{equation*}
and set
\begin{equation*}
	D_z
	\coloneqq
	\nu^{-1}\| \k_L-\k_\ell-\h_z \|_{L^\infty(z+\cu_n)}
	+
	\nu^{-2}\| \k_L-\k_\ell-\h_z \|_{L^\infty(z+\cu_n)}^2
	\qquad \mbox{for every~$z \in 3^n\Zd \cap \cu_m$\,.}
\end{equation*}
For every~$z \in 3^n\Zd \cap \cu_m$, we have
\begin{equation*}
	\| \k_L-\k_\ell-\h_z \|_{L^\infty(z+\cu_n)}
	\leq
	C3^n \| \nabla (\k_L-\k_\ell) \|_{L^\infty(z+\cu_n)}
\end{equation*}
and therefore
\begin{equation}
	\label{e.localization.average.Dz}
	D_z
	\leq
	C\nu^{-1}3^n \| \nabla (\k_L-\k_\ell) \|_{L^\infty(z+\cu_n)}
	+
	C\nu^{-2}3^{2n} \| \nabla (\k_L-\k_\ell) \|_{L^\infty(z+\cu_n)}^2
	\,.
\end{equation}
Applying Lemma~\ref{l.localization.A} with~$\a = \a_\ell$, perturbation field~$\k_L-\k_\ell-\h_z$, and domain~$z+\cu_n$, and then using~\eqref{e.commute.coarse.grained.k0}, we obtain
\begin{equation*}
	\bigl|
	\bfA_\ell^{-\nf 12}(z+\cu_n)
	\G_{\h_z}^t \bfA_L(z+\cu_n)\G_{\h_z}
	\bfA_\ell^{-\nf 12}(z+\cu_n)
	-
	\Itwod
	\bigr|
	\leq
	D_z
\end{equation*}
for every~$z \in 3^n\Zd \cap \cu_m$. Therefore
\begin{equation*}
	\bigl|
	\G_{-\h_z}P \cdot
	\bigl(
	\G_{\h_z}^t \bfA_L(z+\cu_n)\G_{\h_z} - \bfA_\ell(z+\cu_n)
	\bigr)
	\G_{-\h_z}P
	\bigr|
	\leq
	D_z
	\bigl| \bfA_\ell^{\nf12}(z+\cu_n)\G_{-\h_z}P \bigr|^2
	\,.
\end{equation*}
Adding and subtracting~$\bfAhom_\ell(\cu_n)$, averaging over~$z \in 3^n\Zd \cap \cu_m$, and using that
\begin{equation*}
	\G_{-\h_z}P \cdot \G_{\h_z}^t \bfA_L(z+\cu_n)\G_{\h_z} \G_{-\h_z}P
	=
	P \cdot \bfA_L(z+\cu_n) P\,,
\end{equation*}
we obtain
\begin{equation}
	\label{e.localization.average}
	\biggl|
	\avsum_{z \in 3^n\Zd \cap \cu_m}
	P\cdot
	\bigl(
	\bfA_L(z+\cu_n) - \G_{-\h_z}^t \bfAhom_\ell(\cu_n) \G_{-\h_z}
	\bigr)
	P
	\biggr|
	\leq
	T_1 + T_2
	\,,
\end{equation}
where
\begin{equation*}
	T_1
	\coloneqq
	\avsum_{z \in 3^n\Zd \cap \cu_m}
	D_z
	\bigl| \bfA_\ell^{\nf12}(z+\cu_n)\G_{-\h_z}P \bigr|^2
\end{equation*}
and
\begin{equation*}
	T_2
	\coloneqq
	\biggl|
	\avsum_{z \in 3^n\Zd \cap \cu_m}
	\G_{-\h_z}P \cdot \bigl( \bfA_\ell(z+\cu_n) - \bfAhom_\ell(\cu_n)\bigr)\G_{-\h_z}P
	\biggr|
	\,.
\end{equation*}

\smallskip

Next, denote
\begin{equation*}
	M(P)
	\coloneqq
	\max_{z \in 3^n\Zd \cap \cu_m}
	\bigl| \bfAhom_\ell^{\nf12}(\cu_n)\G_{-\h_z}P \bigr|^2
	\qand
	\mathcal{F}_{>}
	\coloneqq
	\sigma \bigl( \mathbf{j}_r \,:\, r > \ell \bigr)
	\,.
\end{equation*}
If~$M(P)=0$, then~$P=0$ and the estimate~\eqref{e.localization.average.oneshot} is trivial. We may therefore assume that~$M(P)>0$, and define
\begin{equation*}
	X_z
	\coloneqq
	M(P)^{-1}
	\G_{-\h_z}P \cdot
	\bigl( \bfA_\ell(z+\cu_n) - \bfAhom_\ell(\cu_n) \bigr)
	\G_{-\h_z}P
\end{equation*}
for every~$z \in 3^n\Zd \cap \cu_m$. Since~$\h_z$ depends only on the scales~$\{ \mathbf{j}_r \}_{r > \ell}$ while~$\bfA_\ell$ is measurable with respect to~$\{ \mathbf{j}_r \}_{r \leq \ell}$, the sigma algebra~$\mathcal{F}_{>}$ is independent of the field~$\bfA_\ell$. Consequently, conditionally on~$\mathcal{F}_{>}$, the vectors~$\{ \G_{-\h_z}P \}_{z \in 3^n\Zd \cap \cu_m}$ are deterministic and, by stationarity,
\begin{equation*}
	\E[ X_z \mid \mathcal{F}_{>}]
	=
	0
	\qquad \mbox{for every~$z \in 3^n\Zd \cap \cu_m$.}
\end{equation*}
Moreover,
\begin{align*}
	|X_z|
	&\leq
	\bigl|
	\bfAhom_\ell^{\nf12}(\cu_n)\G_{-\h_z}P
	\bigr|^2
	M(P)^{-1}
	\bigl|
	\bfAhom_\ell^{-\nf 12}(\cu_n)
	\bigl( \bfA_\ell(z+\cu_n) - \bfAhom_\ell(\cu_n) \bigr)
	\bfAhom_\ell^{-\nf 12}(\cu_n)
	\bigr|
	\\ &\leq
	\bigl|
	\bfAhom_\ell^{-1}(\cu_n)\bfA_\ell(z+\cu_n)
	\bigr|
	+
	1
	\leq
	\O_{\Gamma_1}(C\nu^{-2}(1\vee\ell))
	\qquad \mbox{conditionally on~$\mathcal{F}_{>}$,}
\end{align*}
where in the last step we used~\eqref{e.Enaught.vs.A.and.Ahom} together with the deterministic comparison
\begin{equation*}
	\bigl|\bfAhom_\ell^{-\nf12}(\cu_n)\bfE_\ell^{\nf12}\bigr|^2
	\leq C\nu^{-2}(1\vee\ell)\,,
\end{equation*}
which follows from~\eqref{e.CG.bounds.1} and the definition of~$\bfE_\ell$. The conditional stochastic bound is valid because the field~$\bfA_\ell$ is independent of~$\mathcal{F}_{>}$. Also, since~$\bfA_\ell$ is determined by the cutoff-$\ell$ environment, the family~$\{ X_z \}_{z \in 3^n\Zd \cap \cu_m}$ has range of dependence~$C 3^\ell$ conditionally on~$\mathcal{F}_{>}$ by~\ref{a.j.frd}. We partition the lattice~$3^n \Zd \cap \cu_m$ into~$c 3^{d(\ell-n)}$ sublattices with spacing at least~$C 3^{\ell}$, so that within each sublattice the variables~$\{ X_z \}$ are independent. Since each sublattice contains~$c 3^{d(m-\ell)}$ points, applying Proposition~\ref{p.concentration} with~$\sigma = 1$ to~$X_z$ and~$-X_z$ on each sublattice conditionally on~$\mathcal{F}_{>}$ and then summing over the finitely many sublattices by the triangle inequality, we obtain
\begin{equation*}
	\P \Biggl[
	\biggl|
	\avsum_{z \in 3^n \Zd \cap \cu_m} X_z
	\biggr|
	>
	t C\nu^{-2}(1\vee\ell) 3^{-\frac d2(m-\ell)}
	\, \bigg| \, \mathcal{F}_{>}
	\Biggr]
	\leq
	\exp(-t)
	\qquad \forall t \geq 1
\end{equation*}
almost surely. That is, conditionally on~$\mathcal{F}_{>}$,
\begin{equation*}
	M(P)^{-1}
	\biggl|
	\avsum_{z \in 3^n\Zd \cap \cu_m}
	\G_{-\h_z}P \cdot
	\bigl( \bfA_\ell(z+\cu_n) - \bfAhom_\ell(\cu_n) \bigr)
	\G_{-\h_z}P
	\biggr|
	\leq
	\O_{\Gamma_1}\bigl( C\nu^{-2}(1\vee\ell) 3^{-\frac d2(m-\ell)} \bigr)\,.
\end{equation*}
Removing the conditioning, or equivalently integrating the conditional tail bound, we also obtain
\begin{equation}
	\label{e.localization.average.mix}
	M(P)^{-1}
	\biggl|
	\avsum_{z \in 3^n\Zd \cap \cu_m}
	\G_{-\h_z}P \cdot
	\bigl( \bfA_\ell(z+\cu_n) - \bfAhom_\ell(\cu_n) \bigr)
	\G_{-\h_z}P
	\biggr|
	\leq
	\O_{\Gamma_1}\bigl( C\nu^{-2}(1\vee\ell) 3^{-\frac d2(m-\ell)} \bigr)
	\,.
\end{equation}

\smallskip

Next, the localization error vanishes identically when~$L=\ell$. If~$L>\ell$, then by~\eqref{e.localization.average.Dz}, stationarity, and~\eqref{e.nabla.kmn.Linfty}, for each fixed~$z$,~$D_z \leq \O_{\Gamma_1}\bigl( C\nu^{-2}3^{-(\ell-n)} \bigr)$. Hence, by~\eqref{e.powerofGammasigma} and the generalized triangle inequality~\eqref{e.Gamma.sigma.triangle},
\begin{equation*}
	\avsum_{z \in 3^n\Zd \cap \cu_m} D_z^2
	\leq
	\indc_{\{L>\ell\}}\,\O_{\Gamma_{\nicefrac12}}\bigl( C\nu^{-4} 3^{-2(\ell-n)} \bigr)
	\,.
\end{equation*}
We next estimate the fourth-power average by separating the lower-scale field from the upper-scale gauge. Set
\begin{equation*}
	Y_z
	\coloneqq
	\bigl| \bfE_\ell^{-\nf12}\bfA_\ell(z+\cu_n)\bfE_\ell^{-\nf12} \bigr|
	\qand
	R_z
	\coloneqq
	\bigl| \bfE_\ell^{\nf12}\G_{-\h_z}P \bigr|^4
	\,.
\end{equation*}
The variables~$Y_z$ are measurable with respect to~$\{ \mathbf{j}_r \}_{r\leq \ell}$ and are therefore independent of~$\mathcal{F}_{>}$, while~$R_z$ is~$\mathcal{F}_{>}$-measurable. By~\eqref{e.Enaught.vs.A.and.Ahom},
\begin{equation*}
	Y_z^2
	\leq
	\O_{\Gamma_{\nicefrac12}}(C)
	\qquad \mbox{conditionally on~$\mathcal{F}_{>}$,}
\end{equation*}
for each fixed~$z$. By~\eqref{e.jk.spatialavg}, Proposition~\ref{p.concentration}, and the definition of~$\h_z$, we have
\begin{equation*}
	|\h_z|
	\leq
	\O_{\Gamma_2}\bigl( C (L-\ell)^{\nf12} \bigr)
	\qquad \mbox{for each fixed~$z \in 3^n\Zd \cap \cu_m$.}
\end{equation*}
In either case, using~\eqref{e.Enaught.mixing},
\begin{equation*}
	R_z
	\leq
	\O_{\Gamma_{\nicefrac12}}\bigl( C\nu^{-2}\bigl((1\vee\ell)^2 + (L-\ell)^2\bigr)|P|^4\bigr)
	\qquad \mbox{for each fixed~$z \in 3^n\Zd \cap \cu_m$.}
\end{equation*}
The generalized triangle inequality therefore gives
\begin{equation*}
	\avsum_{z \in 3^n\Zd \cap \cu_m}
	R_z
	\leq
	\O_{\Gamma_{\nicefrac12}}\bigl( C\nu^{-2}\bigl((1\vee\ell)^2 + (L-\ell)^2\bigr)|P|^4\bigr)
	\,.
\end{equation*}
Since
\begin{equation*}
	\bigl| \bfA_\ell^{\nf12}(z+\cu_n)\G_{-\h_z}P \bigr|^4
	\leq
	Y_z^2 R_z\,,
\end{equation*}
the conditional generalized triangle inequality gives
\begin{equation*}
	\avsum_{z \in 3^n\Zd \cap \cu_m} Y_z^2 R_z
	\leq
	\O_{\Gamma_{\nicefrac12}}\biggl( C\avsum_{z \in 3^n\Zd \cap \cu_m} R_z \biggr)
	\qquad \mbox{conditionally on~$\mathcal{F}_{>}$}\,.
\end{equation*}
Combining the previous two displays with the multiplication property~\eqref{e.multGammasig} gives
\begin{equation*}
	\avsum_{z \in 3^n\Zd \cap \cu_m}
	\bigl| \bfA_\ell^{\nf12}(z+\cu_n)\G_{-\h_z}P \bigr|^4
	\leq
	\O_{\Gamma_{\nicefrac14}}\bigl( C \nu^{-2}\bigl((1\vee\ell)^2 + (L-\ell)^2\bigr) |P|^4 \bigr)
	\,.
\end{equation*}
Applying the Cauchy-Schwarz inequality and~\eqref{e.powerofGammasigma}, the first term satisfies
\begin{equation*}
	T_1
	\leq
	\O_{\Gamma_{\nf13}}\bigl( \indc_{\{L>\ell\}}\, C \nu^{-3}L 3^{-(\ell-n)} |P|^2 \bigr)
\end{equation*}
after enlarging the constant~$C(d)$ if necessary.

\smallskip

For the second term, set
\begin{equation*}
	W_z \coloneqq \bigl|\bfE_\ell^{\nf12}\G_{-\h_z}P\bigr|^2
	\qquad \mbox{and}\qquad
	Z_z \coloneqq \G_{-\h_z}P\cdot\bigl(\bfA_\ell(z+\cu_n)-\bfAhom_\ell(\cu_n)\bigr)\G_{-\h_z}P\,,
\end{equation*}
so that~$W_z$ is~$\mathcal{F}_>$-measurable,~$\E[Z_z\mid\mathcal{F}_>]=0$ by stationarity, and~$T_2 = \bigl|\avsum_{z\in 3^n\Zd\cap\cu_m} Z_z\bigr|$. By~\eqref{e.Enaught.vs.A.and.Ahom}, we have~$|\bfE_\ell^{-\nf12}\bfAhom_\ell(\cu_n)\bfE_\ell^{-\nf12}|\leq 1$ deterministically and~$|\bfE_\ell^{-\nf12}\bfA_\ell(z+\cu_n)\bfE_\ell^{-\nf12}|\leq\O_{\Gamma_1}(C)$ conditionally on~$\mathcal{F}_>$, so the triangle inequality gives
\begin{equation*}
	|Z_z| \leq W_z\cdot\bigl|\bfE_\ell^{-\nf12}\bigl(\bfA_\ell(z+\cu_n)-\bfAhom_\ell(\cu_n)\bigr)\bfE_\ell^{-\nf12}\bigr| \leq W_z\cdot\O_{\Gamma_1}(C)
	\qquad \mbox{conditionally on~$\mathcal{F}_>$.}
\end{equation*}
Set~$\overline W \coloneqq \max_{z\in 3^n\Zd\cap\cu_m} W_z$, which is~$\mathcal{F}_>$-measurable. Apply Proposition~\ref{p.concentration} with~$\sigma=1$ to the conditionally independent centered variables~$\overline W^{-1}Z_z$ on each of the~$\sim 3^{d(\ell-n)}$ sublattices used above (their conditional scale~$|Z_z|/\overline W\leq\O_{\Gamma_1}(C)$ is uniform in~$z$, since~$W_z/\overline W\leq 1$ deterministically given~$\mathcal{F}_>$), and sum over the sublattices via the triangle inequality:
\begin{equation*}
	\biggl|\avsum_{z\in 3^n\Zd\cap\cu_m} Z_z\biggr|
	\leq
	\overline W\cdot\O_{\Gamma_1}\bigl(C\, 3^{-\frac d2(m-\ell)}\bigr)
	\qquad \mbox{conditionally on~$\mathcal{F}_>$.}
\end{equation*}
By~\eqref{e.Enaught.mixing},~$|\bfE_\ell|\leq C\nu^{-1}(1\vee\ell)$, and~\eqref{e.jk.spatialavg} together with Proposition~\ref{p.concentration} and~\eqref{e.powerofGammasigma} gives~$|\h_z|^2 \leq \O_{\Gamma_1}(C(L-\ell))$ for each fixed~$z$. Combining these with~$|\G_{-\h_z}P|^2 \leq C(1+|\h_z|^2)|P|^2$ yields, for each fixed~$z$,
\begin{equation*}
	W_z = \bigl|\bfE_\ell^{\nf12}\G_{-\h_z}P\bigr|^2 \leq \O_{\Gamma_1}\bigl(C\nu^{-1}(1\vee\ell)(1\vee L)\,|P|^2\bigr)\,.
\end{equation*}
Lemma~\ref{l.maximums.Gamma.s} applied to the~$\sim 3^{d(m-n)}$ values~$\{W_z\}$ then gives
\begin{equation*}
	\overline W \leq \O_{\Gamma_1}\bigl(C\nu^{-1}(1\vee\ell)(1\vee L)(1\vee(m-n))\,|P|^2\bigr)\,.
\end{equation*}
Removing the conditioning by integration and combining via the multiplication property~\eqref{e.multGammasig} (which does not require independence of the two factors), we obtain
\begin{equation*}
	T_2 \leq \O_{\Gamma_{\nicefrac12}}\bigl(C\nu^{-1}(1\vee\ell)(1\vee L)(1\vee(m-n))\,3^{-\frac d2(m-\ell)}|P|^2\bigr)\,.
\end{equation*}
Since~$\Gamma_{\nicefrac12}$-tails are stronger than~$\Gamma_{\nf13}$-tails, and since~$\nu^{-1}\leq\nu^{-3}$ under the standing assumption~$\nu\in(0,1]$, the bound for~$T_2$ is also valid at the~$\Gamma_{\nf13}$ rate with the unified~$\nu^{-3}$ prefactor of~\eqref{e.localization.average.oneshot}.

\smallskip

The powers of~$\nu^{-1}$ in the bound on~$T_1$ are precisely the normalization losses from~$\bfE_\ell$ and the relative localization error. Combining the estimates for~$T_1$ and~$T_2$ with~\eqref{e.localization.average} gives~\eqref{e.localization.average.oneshot}.
\end{proof}

\smallskip

The last auxiliary point in this subsection is the passage from finite infrared cutoffs to the full field on bounded domains. This will be used only qualitatively, to justify replacing solutions for~$\a_L$ by solutions for the limiting centered coefficient field.

\begin{lemma}[Cutoff approximation on bounded domains]
\label{l.cutoff.approximation}
Let~$U\subseteq\Rd$ be a bounded Lipschitz domain and suppose that
\begin{equation*}
	\sum_{k=0}^\infty \|\nabla \mathbf{j}_k\|_{L^\infty(U)}<\infty\,.
\end{equation*}
Define
\begin{equation*}
	\k^U\coloneqq \k-(\k)_U
	\coloneqq \sum_{k=0}^\infty\bigl(\mathbf{j}_k-(\mathbf{j}_k)_U\bigr)
	\qand
	\a^U\coloneqq \nu\Id+\k^U\,.
\end{equation*}
Then~$\a^U$ is uniformly elliptic in~$U$. Moreover, for every~$L\in\N$ and every~$u\in \A(U;\a^U)\cap H^1(U)$, there exists~$u_L\in u+H^1_0(U)$ with~$u_L\in\A(U;\a_L)$ such that
\begin{equation}
	\label{e.cutoff.approximation}
	\|\nabla u-\nabla u_L\|_{L^2(U)}
	\leq
	C(U)\nu^{-1}
	\biggl(\sum_{k=L+1}^\infty \|\nabla \mathbf{j}_k\|_{L^\infty(U)}\biggr)
	\|\nabla u\|_{L^2(U)}\,.
\end{equation}
\end{lemma}

\begin{remark}
For every fixed bounded~$U$, the hypothesis is satisfied almost surely by~\ref{a.j.reg} and Borel--Cantelli, after covering~$U$ by a bounded number of translates of the natural-scale cubes~$3^k\cu_0$.
\end{remark}

\begin{proof}
The Lipschitz character of~$U$ gives, for every~$F\in W^{1,\infty}(U)$,
\begin{equation*}
	\|F-(F)_U\|_{L^\infty(U)}
	\leq
	C(U)\|\nabla F\|_{L^\infty(U)}\,.
\end{equation*}
Thus the series defining~$\k^U$ converges in~$L^\infty(U)$, and
\begin{equation*}
	\|\k^U-(\k_L-(\k_L)_U)\|_{L^\infty(U)}
	\leq
	C(U)\sum_{k=L+1}^\infty \|\nabla \mathbf{j}_k\|_{L^\infty(U)}\,.
\end{equation*}
Since~$\k^U$ is anti-symmetric,~$\xi\cdot\a^U\xi=\nu|\xi|^2$, and the previous convergence also gives~$\a^U\in L^\infty(U)$. Let~$u_L\in u+H^1_0(U)$ solve
\begin{equation*}
	-\nabla\cdot\bigl(\nu\Id+\k_L-(\k_L)_U\bigr)\nabla u_L=0
	\qquad \mbox{in } U\,.
\end{equation*}
This is the same equation as the~$\a_L$-equation, since adding a constant anti-symmetric matrix does not change the operator. Testing the equations for~$u$ and~$u_L$ with~$u-u_L$ gives
\begin{equation*}
	\nu\|\nabla u-\nabla u_L\|_{L^2(U)}^2
	\leq
	\|\k^U-(\k_L-(\k_L)_U)\|_{L^\infty(U)}
	\|\nabla u\|_{L^2(U)}
	\|\nabla u-\nabla u_L\|_{L^2(U)}\,.
\end{equation*}
The preceding display implies~\eqref{e.cutoff.approximation}.
\end{proof}

\section{A suboptimal lower bound on the renormalized diffusivities}
\label{s.subopt}

In this section, we prove the precise version of the suboptimal lower bound stated informally in~\eqref{e.sstar.lower.bound.quenched.intro}. The idea of the proof is that, if~$\shom_{L,*}(\cu_m)$ does not grow sufficiently fast in~$m$, then due to its monotonicity, we can find a range of scales across which it does not change much. But this allows us to compare the maximizers of the variational problem in~\eqref{e.J.def} with~$(p,q) = (0,q)$ across this same range of scales and thereby deduce that the maximizer on the largest scale is ``flat'' (close to an affine function). We then argue that the existence of such a flat solution is only possible if the advection term is not contributing much (which is ruled out by assumption~\ref{a.j.nondeg}), or the effective diffusivity~$\shom_{L,*}(\cu_m)$ is large. We deduce therefore that~$\shom_{L,*}(\cu_m)$ must be large.

\begin{proposition}[Suboptimal lower bound estimate]
\label{p.sstar.lower.bound}
There exist~$C(d)\in[1,\infty)$ and~$c(d)\in(0,\nicefrac12]$ such that, for every~$L,m\in\N$ satisfying
\begin{equation}
	\label{e.L.vs.nu}
	L\geq m \geq \frac12 L
	\quad \mbox{and} \quad
	m \geq C\cstar^{-3} \bigl(
	\log^3 (3+\nu^{-1}) \log \log (3+\nu^{-1})
	+ (1+\nondegconst)\log(3+\nu^{-1}+\nondegconst)
	+ q(1+\nondegconst+q)
	\bigr)
	\,,
\end{equation}
where~$q\coloneqq\log(3+\cstar^{-1}+\nu^{-1}+\nondegconst)$, and with~$\shom_L=\lim_{r\to\infty}\shom_{L,*}(\cu_r)$, we have
\begin{equation}
	\label{e.sstar.lower.bound}
	\shom_L \geq \shom_{L,*}(\cu_m)
	\geq
	c \cstar^{\nf32}\nu^2  m^{\nf12} \log^{-\nf {9}2} ( \nu^{-1} m)
	\,.
\end{equation}
Moreover, the following quenched estimate holds:
\begin{equation}
	\label{e.sstar.lower.bound.quenched}
	\s_{L,*}^{-1} (\cu_m)
	\leq
	C \nu^{-2} \cstar^{-\nf32} m^{- \nicefrac12} \log^{\nf {9}2} ( \nu^{-1} m)\,\Id
	+
	\O_{\Gamma_2}(C\nu^{-2} 3^{-\frac 18 m})\,\Id
	\,.
\end{equation}
\end{proposition}

The rest of this section is focused on the proof of Proposition~\ref{p.sstar.lower.bound}. We first prove the deterministic lower bound at a pigeonhole scale. In this part of the proof the letters~$m$ and~$n$ denote auxiliary scales selected below; the prescribed scale in the statement is restored at the end. We fix~$\delta=c_0\cstar^2$, with~$c_0(d)>0$ sufficiently small, and then select~$h\in\N$ satisfying
\begin{equation}
	\label{e.h.restrictions}
	h \geq 10 \lceil K \log^2 (\nu^{-1}L) \rceil
	\quad \mbox{and} \quad
	\frac{C \log (C\nu^{-1}L) }{\delta} h
	\leq L
	\,,
\end{equation}
where~$K$ is a large constant to be selected later. We will also select the parameter~$h$ at the end of the proof; the choice of~$c_0$ ensures in particular that the relative smallness condition~$C\delta^{\nf12}\leq \frac14\cstar$ used below holds. No absolute smallness of~$\delta$ independent of~$\cstar$ is used. We next use the pigeonhole argument in the proof of~\cite[Lemma~3.4]{AK.HC}. That lemma is stated as a dichotomy in~\cite{AK.HC}; here the second alternative is ruled out by the crude deterministic bounds over our finite window. Indeed, set~$A_k\coloneqq\bfAhom_L(\cu_k)$. By subadditivity,~$k\mapsto A_k$ is nonincreasing in the Loewner order. Moreover,~\eqref{e.Enaught.vs.A.and.Ahom} gives~$\bfAhom_L\leq A_k\leq \bfE_L$, while~$\det \bfAhom_L=1$ and~$\det \bfE_L\leq (C\nu^{-1}L)^{2d}$. If the following conclusion failed for every~$m$ in the lattice~$\lfloor L/4\rfloor+2h\N$, with~$m\in[\lfloor L/4\rfloor+2h,\lfloor L/2\rfloor]$, then telescoping the determinant ratios would give
\begin{equation*}
	(1+\delta)^{\lfloor L/(8h)\rfloor-2}
	\leq
	\frac{\det A_{\lfloor L/4\rfloor}}{\det \bfAhom_L}
	\leq
	(C\nu^{-1}L)^{2d}\,,
\end{equation*}
which contradicts the second condition in~\eqref{e.h.restrictions}, using~$\log(1+\delta)\geq\delta/2$. Therefore there exists
\begin{equation*}
	m\in\N \cap [ L/4, L/2]\,,
\end{equation*}
satisfying
\begin{equation}
	\label{e.pigeon.matrix}
	\bigl|
	\bfAhom_L^{-\nf12}(\cu_m) \bfAhom_{L} (\cu_{m-2h}) \bfAhom_L^{-\nf12}(\cu_m) - \Itwod \bigr|
	\leq \delta
	\,.
\end{equation}
Writing~$a\coloneqq \lceil K \log (\nu^{-1}L) \rceil$, we next define scales~$n,\ell,\ell',L'\in\N$ by
\begin{equation}
	\label{e.scale.selection}
	\left\{
	\begin{aligned}
		& L'  \coloneqq  m + 2a\,, \\
		&
		\ell'  \coloneqq  m-h \,,  \\ &
		\ell  \coloneqq  \ell' - a\,, \\ &
		n \coloneqq  \ell - a\,.
	\end{aligned}
	\right.
\end{equation}
For later use, we record the identities
\begin{equation*}
	\begin{gathered}
		\ell-n=a\,,
		\qquad
		\ell'-\ell=a\,,
		\qquad
		\ell'-n=2a\,,
		\\
		L'-m=2a\,,
		\qquad
		L'-\ell=h+3a\,,
		\qquad
		L'-\ell'=m-n=h+2a\,.
	\end{gathered}
\end{equation*}
Note that, by~\eqref{e.h.restrictions} and the lower bound on~$L$ in~\eqref{e.L.vs.nu} (with the constant there taken large relative to~$K$), we can ensure that
\begin{equation}
	\label{e.scales.ordering}
	m-2h < n < \ell < \ell' < m<L' < L
	\,.
\end{equation}
For every~$k \in[m-2h,m]$ and~$z\in 3^n\Zd\cap\cu_m$, the ordering~\eqref{e.scales.ordering} gives~$k\leq m<L'<L$, so the hypotheses of the relative comparison~\eqref{e.localization.s.star} are satisfied on~$z+\cu_k$. Thus
\begin{equation}
	\label{e.localization.s.star.applied}
	\bigl| \s_{L',*}^{\nf12}(z+\cu_k)
	\s_{L,*}^{-1}(z+\cu_k)
	\s_{L',*}^{\nf12}(z+\cu_k) - \Id \bigr|
	\leq
	\O_{\Gamma_1} \bigl( C \nu^{-2} 3^{-(L'-k)} \bigr)
	\leq
	\O_{\Gamma_1} \bigl( C \nu^{-2} 3^{-(L'-m)} \bigr)
	\,.
\end{equation}
The localization error is generated by the tail between the two cutoffs, while the lower-cutoff coarse-grained matrix is measurable with respect to the lower modes; by~\ref{a.j.indy}, the corresponding factors may be separated when this relative comparison is averaged. In the places where we only need an additive comparison, the crude deterministic ellipticity bounds give the same displayed error bound. Set
\begin{equation*}
	\eta_L\coloneqq C\nu^{-5}L3^{-(L'-m)}\,.
\end{equation*}
The extra factor of~$\nu^{-1}$ is harmless and leaves room for the conversion from the relative Loewner comparison to the scalar inverse-block ratio. Taking expectations in the corresponding Loewner inequalities gives an additive comparison of the annealed lower-right blocks. Using the scalar form of those blocks together with the crude bound~$\shom_{L',*}(\cu_k)\leq C\nu^{-2}L$, which follows from~\eqref{e.CG.bounds.1} and~\eqref{e.km.Ltwo.size}, this becomes, for the same range of~$k$,
\begin{equation}
	\label{e.localization.s.star.annealed.applied}
	\bigl|
	\shom_{L',*}(\cu_k)\shom_{L,*}^{-1}(\cu_k)-1
	\bigr|
	\leq
	\eta_L
	\,.
\end{equation}
Combining the annealed comparison~\eqref{e.localization.s.star.annealed.applied} with~\eqref{e.pigeon.matrix}, we deduce that
\begin{equation}
	\label{e.pigeon.scalar}
	\bigl| \shom_{L',*}(\cu_m)
	\shom_{L',*} ^{-1}(\cu_{m-2h}) - 1
	\bigr|
	\leq
	\delta + \eta_L
	\,.
\end{equation}
Indeed, by~\eqref{e.homs.defs.U}, the lower-right block of
\begin{equation*}
	\bfAhom_L^{-\nf12}(\cu_m) \bfAhom_L(\cu_{m-2h}) \bfAhom_L^{-\nf12}(\cu_m)
\end{equation*}
is exactly
\begin{equation*}
	\shom_{L,*}(\cu_m)\shom_{L,*}^{-1}(\cu_{m-2h})\Id \,.
\end{equation*}
Therefore~\eqref{e.pigeon.matrix} gives the analogue of~\eqref{e.pigeon.scalar} with cutoff~$L$. We then use the annealed consequence~\eqref{e.localization.s.star.annealed.applied} at the two scales~$m$ and~$m-2h$ to replace these cutoff-$L$ factors by cutoff-$L'$ at cost~$\eta_L$. Note that, by subadditivity, the quantity~$\shom_{L',*}^{-1}(\cu_k)$ is monotone nonincreasing in~$k$. Combined with~\eqref{e.pigeon.scalar}, this implies that across the range of scales~$k \in [m-2h,m]$ the ratio of any two~$\shom_{L',*}(\cu_k)$ is close to one. In particular, the parameters~$n,\ell$ and~$\ell'$ each represent scales which are within this range, that is,~$n,\ell,\ell' \in [ m-2h,m]$. Since the annealed blocks are scalar multiples of~$\Id$, these comparisons may be used either in the scalar form above or after taking inverses; throughout this section the resulting inverse comparisons are deterministic scalar inequalities.

\smallskip

We next fix a unit vector~$e \in\Rd$ with~$|e|=1$. The subscript in~$J_{L'}$ indicates that the variational problem~\eqref{e.J.def} is formed with the cutoff coefficient~$\a_{L'}$. For each~$y\in\Rd$ and~$k\in\N$, we let~$u_{k,y}$ denote the maximizer in the variational problem defining~$J_{L'}(y+\cu_k,0,\shom_{L',*}^{\nf12}(\cu_n)  e)$; that is,
\begin{equation}
	u_{k,y} \coloneqq  v_{L'}\bigl(\cdot,y+\cu_k,0, \shom_{L',*}^{\nf12}(\cu_n)  e  \bigr)
	\,, \quad y\in\Rd\,, \ k\in\N\,.
	\label{e.u.k.y.def}
\end{equation}
For each~$k\in\N$ we let~$\nabla u_k$ denote the vector field whose restriction in each cube of the form~$z+ \cu_k$ with~$z\in 3^k\Zd$ is equal to~$\nabla u_{k,z}$. That is,
\begin{equation}
	\nabla u_k  \coloneqq  \sum_{z\in 3^k\Zd} \nabla u_{k,z} \indc_{z+\cu_k} \,.
	\label{e.u.k.def}
\end{equation}
This is a slight abuse of notation, since~$\nabla u_k$ is not necessarily a gradient field except on subdomains of a single cube of the form~$z+\cu_k$,~$z\in 3^k\Zd$.

\smallskip

We define the parameters~$p,q \in \Rd$ by
\begin{equation}
	\label{e.Sec3.p.q.def}
	p \coloneqq  \shom_{L',*}^{-\nf12}(\cu_n)  e
	\quad \mbox{and} \quad
	q  \coloneqq  \E[ (\a_\ell \nabla u_n)_{\cu_\ell}]\,.
\end{equation}
Observe that
\begin{equation}
	\label{e.p-bound-crude}
	|p| \leq \nu^{-\nf12}  \,.
\end{equation}
By~\eqref{e.v.spatial.averages}, we have that, for every~$z\in 3^n\Zd$,
\begin{equation}
	\label{e.average.of.vn}
	\E\bigl[ (\nabla u_{n})_{z+\cu_n} \bigr] = \shom_{L',*}^{-1}(\cu_n) \shom_{L',*}^{\nf12}(\cu_n) e = p
	\,.
\end{equation}
We next use the identity~\eqref{e.Jenergy.v} with~$\s = \nu \Id$ and the bound~$\s_{L',*}^{-1} (U) \leq \nu^{-1}\Id$, which is valid in every domain~$U$, to obtain, for every~$z\in 3^k\Zd$, the quenched estimate. This crude ellipticity bound is used only as an a priori energy normalization here; no sharp dependence on~$\nu$ is needed at this point. Observe that
\begin{align}
	\label{e.v.ky.energy}
	\| \nabla u_{k} \|_{\underline{L}^2(z+\cu_k)}^2
	&
	=
	2 \nu^{-1} J_{L'}(z+\cu_k,0,\shom_{L',*}^{\nf12}(\cu_n)  e)
	\notag \\ &
	=
	\nu^{-1}  \shom_{L',*}^{\nf12}(\cu_n)  e   \cdot \s_{L',*}^{-1}(z + \cu_k) \shom_{L',*}^{\nf12}(\cu_n)  e
	\leq
	\nu^{-2} \shom_{L',*}(\cu_n)
	\,.
\end{align}
We next write the equation for~$u_m  \coloneqq  u_{m,0}$ as
\begin{equation}
	\label{e.ellsep}
	-\nabla \cdot \a_{\ell} \nabla u_m
	=
	(\f_{L'}-\f_{\ell})  \cdot \nabla u_m \quad \mbox{in} \ \cu_m\,.
\end{equation}
We let~$w \in H^2(\cu_m)$ be the solution of the Dirichlet problem
\begin{equation}
	\label{e.def.w}
	\left\{
	\begin{aligned}
		& - \Delta w  =
		(\f_{L'} - \f_{\ell'}) \cdot p
		& \mbox{in} & \ \cu_m \,, \\
		&
		w = 0 & \mbox{on} & \ \partial \cu_m \,.
	\end{aligned}
	\right.
\end{equation}
The proof of Proposition~\ref{p.sstar.lower.bound} is based on a comparison of~$w$ to~$u_m$. We proceed by testing the equations~\eqref{e.ellsep} and~\eqref{e.def.w} with~$w$ to get
\begin{equation*}
	\fint_{\cu_m}
	\nabla w \cdot \a_{\ell} \nabla u_m
	=
	\fint_{\cu_m}
	w (\f_{L'} - \f_{\ell}) \cdot \nabla u_m
	=
	p \cdot \fint_{\cu_m} w (\f_{L'} - \f_{\ell})
	+
	\fint_{\cu_m} w (\f_{L'} - \f_{\ell} ) \cdot ( \nabla u_m - p )
\end{equation*}
and
\begin{equation}
	\label{e.w.testing.formula}
	\fint_{\cu_m}
	\bigl| \nabla w \bigr|^2
	=
	p\cdot \fint_{\cu_m} w (\f_{L'} - \f_{\ell'})
	=
	p\cdot \fint_{\cu_m} w (\f_{L'} - \f_{\ell})
	+
	p\cdot \fint_{\cu_m} w (\f_{\ell} - \f_{\ell'})
	\,,
\end{equation}
respectively. Combining these with the vector~$q$ defined in~\eqref{e.Sec3.p.q.def}, we obtain
\begin{align*}
	\fint_{\cu_m}
	\bigl| \nabla w \bigr|^2
	&
	=
	\fint_{\cu_m}
	\nabla w \cdot \a_{\ell} \nabla u_m
	-
	\fint_{\cu_m}
	w (\f_{L'} - \f_{\ell}) \cdot ( \nabla u_m -\, p )
	+
	p\cdot \fint_{\cu_m} w (\f_{\ell} - \f_{\ell'})
	\notag \\ &
	=
	\fint_{\cu_m}
	\nabla w \cdot ( \a_{\ell} \nabla u_m -\, q)
	-
	\fint_{\cu_m} \!\!
	\nabla w \cdot (\k_{\ell} - \k_{L'} )( \nabla u_m -\, p )
	-
	p\cdot\fint_{\cu_m} \!\!
	(\k_{\ell'} - \k_{\ell}) \nabla w
	\,.
\end{align*}
To get the second line in the display above, we integrated by parts, using the anti-symmetry of~$\k_{L'}$ and~$\k_{\ell}$, to remove the divergence from~$\f_{L'} - \f_\ell$ and transfer the derivative to~$w$. We next split the terms on the right side of the previous display involving~$\nabla u_m$ into smaller-scale maximizers and additivity-defect terms. Noting also that~$\a_{L'} = \a_{\ell} + \k_{L'} - \k_{\ell}$, we get
\begin{align}
	\label{e.ellsep.testing}
	\fint_{\cu_m}
	\bigl| \nabla w \bigr|^2
	&
	=
	\fint_{\cu_m}
	\nabla w \cdot ( \a_{\ell} \nabla u_{n} - q)
	+
	\fint_{\cu_m}
	\nabla w \cdot (\k_{L'} - \k_{\ell} ) ( \nabla u_{n} - p )
	\notag \\ & \qquad
	+
	\fint_{\cu_m}
	\nabla w \cdot \a_{L'} ( \nabla u_m - \nabla u_{n} )
	- p\cdot\fint_{\cu_m} (\k_{\ell'} - \k_{\ell}) \nabla w
	\,.
\end{align}
The strategy is to estimate the expectation of the left side of~\eqref{e.ellsep.testing} from below in terms of~$p$, and thus by means of~$\shom_{L,*}^{-1}(\cu_n)$, and then to upper bound the expectation of the right side. Lemma~\ref{l.w.basic.regbounds} first supplies the response norms. Lemma~\ref{l.LHS.term1} gives the lower bound for the left side, Lemmas~\ref{l.RHS.term1}--\ref{l.RHS.term4} estimate the four right-side terms in their displayed order, and Lemma~\ref{l.b.ell.homogenization} controls the remaining coarse ellipticity factor in the third term.

\begin{lemma}[Poisson response fields for stationary fluxes]
\label{l.abstract.response.fields}
Let~$\mathbf F$ be an~$\Rd$--stationary random vector field with~$\E[|\mathbf F(0)|^2]<\infty$, and let~$\nabla \hat w_{\mathbf F}$ be the unique centered stationary potential field satisfying
\begin{equation*}
	-\Delta \hat w_{\mathbf F} = \nabla\cdot \mathbf F
	\quad \mbox{in } \Rd\,,
	\qquad
	\nabla \hat w_{\mathbf F}+\mathbf F
	\quad\mbox{is solenoidal}\,,
	\qquad
	\E[\nabla\hat w_{\mathbf F}(0)]=0\,.
\end{equation*}
For~$M\in\N$, let~$w_{D,\mathbf F}^{(M)}\in H^1_0(\cu_M)$ and the zero-average function~$w_{N,\mathbf F}^{(M)}\in H^1(\cu_M)$ be the solutions of
\begin{equation}
	\label{e.abstract.response.equations}
	\left\{
	\begin{aligned}
		&-\Delta w_{D,\mathbf F}^{(M)} = \nabla\cdot\mathbf F
		&&\mbox{in } \cu_M\,,
		\\
		&w_{D,\mathbf F}^{(M)}=0
		&&\mbox{on } \partial\cu_M\,,
	\end{aligned}
	\right.
	\qquad
	\left\{
	\begin{aligned}
		&-\Delta w_{N,\mathbf F}^{(M)} = \nabla\cdot\mathbf F
		&&\mbox{in } \cu_M\,,
		\\
		&\mathbf n_{\cu_M}\cdot
		\bigl(\nabla w_{N,\mathbf F}^{(M)}+\mathbf F\bigr)=0
		&&\mbox{on } \partial\cu_M\,.
	\end{aligned}
	\right.
\end{equation}
Then
\begin{equation}
	\label{e.energy.comparison.N.D}
	\E\bigl[
	\| \nabla w_{D,\mathbf F}^{(M)} \|_{\underline L^2(\cu_M)}^2
	\bigr]
	\leq
	\E\bigl[ |\nabla \hat w_{\mathbf F}(0)|^2 \bigr]
	\leq
	\E\bigl[
	\| \nabla w_{N,\mathbf F}^{(M)} \|_{\underline L^2(\cu_M)}^2
	\bigr]
	\,,
\end{equation}
and
\begin{equation}
	\label{e.energy.quadratic.N.D}
	\| \nabla w_{N,\mathbf F}^{(M)} \|_{\underline L^2(\cu_M)}^2
	=
	\| \nabla w_{D,\mathbf F}^{(M)} \|_{\underline L^2(\cu_M)}^2
	+
	\| \nabla w_{N,\mathbf F}^{(M)}
	- \nabla w_{D,\mathbf F}^{(M)} \|_{\underline L^2(\cu_M)}^2
	\,.
\end{equation}
Moreover, the following deterministic estimates, which do not use stationarity, hold with a constant~$C(d)<\infty$:
\begin{equation}
	\label{e.abstract.response.L8}
	\| \nabla w_{D,\mathbf F}^{(M)} \|_{\underline L^8(\cu_M)}
	+
	\| \nabla w_{N,\mathbf F}^{(M)} \|_{\underline L^8(\cu_M)}
	\leq
	C \| \mathbf F \|_{\underline L^8(\cu_M)}
	\,,
\end{equation}
\begin{equation}
	\label{e.abstract.response.W28}
	\| \nabla^2 w_{D,\mathbf F}^{(M)} \|_{\underline L^8(\cu_M)}
	+
	\| \nabla^2 w_{N,\mathbf F}^{(M)} \|_{\underline L^8(\cu_M)}
	\leq
	C \| \nabla \mathbf F \|_{\underline L^8(\cu_M)}
	\,,
\end{equation}
and
\begin{equation}
	\label{e.abstract.response.Hhalf}
	\| \nabla w_{D,\mathbf F}^{(M)} \|_{\underline H^{\nf12}(\cu_M)}
	+
	\| \nabla w_{N,\mathbf F}^{(M)} \|_{\underline H^{\nf12}(\cu_M)}
	\leq
	C \| \mathbf F \|_{\underline H^{\nf12}(\cu_M)}
	\,.
\end{equation}
Finally,
\begin{equation}
	\label{e.abstract.response.ND.weak}
	\| \nabla w_{N,\mathbf F}^{(M)}
	- \nabla w_{D,\mathbf F}^{(M)} \|_{\underline L^2(\cu_M)}
	\leq
	C
	\bigl(3^{-M}\| \mathbf F-(\mathbf F)_{\cu_M} \|_{\Hminusul(\cu_M)}\bigr)^{\nf15}
	\| \mathbf F-(\mathbf F)_{\cu_M} \|_{\underline L^4(\cu_M)}^{\nf45}
	+
	C |(\mathbf F)_{\cu_M}|
	\,.
\end{equation}
\end{lemma}

\begin{proof}
Throughout the proof we suppress~$\mathbf F$ from the subscripts.

\smallskip

\emph{Step 1.} Energy comparison. Since~$\hat w-w_D^{(M)}$ is harmonic in~$\cu_M$ and~$w_D^{(M)}\in H^1_0(\cu_M)$,
\begin{equation*}
	\fint_{\cu_M}
	\bigl(\nabla \hat w
	-\nabla w_D^{(M)}\bigr)
	\cdot \nabla w_D^{(M)}
	=0\,.
\end{equation*}
The Cauchy-Schwarz inequality gives
\begin{equation*}
	\| \nabla w_D^{(M)} \|_{\underline L^2(\cu_M)}^2
	=
	\fint_{\cu_M}
	\nabla \hat w\cdot \nabla w_D^{(M)}
	\leq
	\| \nabla \hat w \|_{\underline L^2(\cu_M)}
	\| \nabla w_D^{(M)} \|_{\underline L^2(\cu_M)}\,,
\end{equation*}
and stationarity yields the first inequality in~\eqref{e.energy.comparison.N.D} after taking expectation.

\smallskip

For the second inequality, observe that~$w_N^{(M)}$ minimizes
\begin{equation*}
	u \mapsto
	\fint_{\cu_M}
	\biggl(
	\frac12 |\nabla u|^2+\mathbf F\cdot\nabla u
	\biggr)
\end{equation*}
over~$H^1(\cu_M)$ modulo additive constants. Hence
\begin{equation*}
	-\frac12
	\| \nabla w_N^{(M)} \|_{\underline L^2(\cu_M)}^2
	\leq
	\frac12
	\| \nabla \hat w \|_{\underline L^2(\cu_M)}^2
	+
	\fint_{\cu_M}
	\nabla \hat w\cdot \mathbf F\,.
\end{equation*}
Since~$\nabla \hat w+\mathbf F$ is stationary and solenoidal while~$\nabla \hat w$ is stationary and potential, orthogonality gives
\begin{equation*}
	\E\bigl[
	\bigl(
	(\nabla \hat w+\mathbf F)
	\cdot \nabla \hat w
	\bigr)(0)
	\bigr]=0\,.
\end{equation*}
Taking expectation in the previous variational inequality and using stationarity gives the upper bound in~\eqref{e.energy.comparison.N.D}. Finally, the mixed term below is~$\|\nabla w_D^{(M)}\|_{\underline L^2(\cu_M)}^2$, by testing the equation for~$w_N^{(M)}$ with~$w_D^{(M)}\in H^1_0(\cu_M)$ and then using the equation for~$w_D^{(M)}$. Thus, using the equations of~$w_D^{(M)}$ and~$w_N^{(M)}$,
\begin{align*}
	\| \nabla w_N^{(M)}
	- \nabla w_D^{(M)} \|_{\underline L^2(\cu_M)}^2
	&=
	\| \nabla w_N^{(M)} \|_{\underline L^2(\cu_M)}^2
	+
	\| \nabla w_D^{(M)} \|_{\underline L^2(\cu_M)}^2
	-2\fint_{\cu_M}
	\nabla w_N^{(M)}
	\cdot\nabla w_D^{(M)}
	\\
	&=
	\| \nabla w_N^{(M)} \|_{\underline L^2(\cu_M)}^2
	-
	\| \nabla w_D^{(M)} \|_{\underline L^2(\cu_M)}^2\,,
\end{align*}
which proves~\eqref{e.energy.quadratic.N.D}.

\smallskip

\emph{Step 2.} A priori estimates. The estimates~\eqref{e.abstract.response.L8} and~\eqref{e.abstract.response.W28} are the constant-coefficient Dirichlet and prescribed-flux Neumann estimates in a cube, obtained for instance by reflection and Fourier multiplier estimates. For the Dirichlet problem one uses odd reflection across the faces; for the prescribed-flux Neumann problem one reflects the normal component of~$\mathbf F$ oddly and the tangential components evenly. The fractional estimate~\eqref{e.abstract.response.Hhalf} follows similarly from the constant-coefficient estimate for the Laplacian, for instance by interpolating the corresponding~$L^2$ and~$H^1$ estimates.

\smallskip

\emph{Step 3.} Harmonic interpolation. We show that, if~$v$ is harmonic in~$\cu_M$, then
\begin{equation}
	\label{e.abstract.harmonic.interpolation}
	\| \nabla v \|_{\underline L^2(\cu_M)}
	\leq
	C
	\bigl( 3^{-M}\| v-(v)_{\cu_M} \|_{\underline L^2(\cu_M)} \bigr)^{\nf15}
	\| \nabla v \|_{\underline L^4(\cu_M)}^{\nf45}
	\,.
\end{equation}
Fix~$\ep\in(0,\frac12)$ and take~$\varphi\in C_{\mathrm c}^\infty(\cu_M)$ such that~$0\leq\varphi\leq 1$,~$\varphi=1$ on the set of points whose distance from~$\partial \cu_M$ is larger than~$\ep 3^M$, and~$\| \nabla \varphi \|_{L^\infty}\leq C\ep^{-1}3^{-M}$. Testing~$\Delta v=0$ against~$(v-(v)_{\cu_M})\varphi^2$ gives
\begin{equation*}
	\| \varphi\nabla v \|_{\underline L^2(\cu_M)}^2
	\leq
	4\| (v-(v)_{\cu_M})\nabla\varphi \|_{\underline L^2(\cu_M)}^2
	\,.
\end{equation*}
Splitting the integral into the interior and the boundary layer, and applying H\"older's inequality on the second piece, we obtain
\begin{align*}
	\| \nabla v \|_{\underline L^2(\cu_M)}^2
	&\leq
	4\| (v-(v)_{\cu_M})\nabla\varphi \|_{\underline L^2(\cu_M)}^2
	+
	\|1-\varphi^2\|_{\underline L^2(\cu_M)}
	\| \nabla v \|_{\underline L^4(\cu_M)}^2
	\\
	&\leq
	C\ep^{-2}
	\bigl( 3^{-M}\| v-(v)_{\cu_M} \|_{\underline L^2(\cu_M)} \bigr)^2
	+
	C\ep^{\nf12}
	\| \nabla v \|_{\underline L^4(\cu_M)}^2
	\,.
\end{align*}
Writing~$A\coloneqq 3^{-M}\| v-(v)_{\cu_M} \|_{\underline L^2(\cu_M)}$ and~$B\coloneqq \| \nabla v \|_{\underline L^4(\cu_M)}$, the cases~$A=0$ or~$B=0$ are trivial. Otherwise, we balance the two terms by taking~$\ep\simeq (A/B)^{\nf45}$ when~$A\leq B$. If~$A>B$, then the same conclusion follows from~$\|\nabla v\|_{\underline L^2(\cu_M)}\leq B$. This proves~\eqref{e.abstract.harmonic.interpolation}.

\smallskip

\emph{Step 4.} Conclusion. Let~$\mathbf F_0\coloneqq \mathbf F-(\mathbf F)_{\cu_M}$. Since~$\cu_M$ is centered, the Dirichlet response for~$\mathbf F_0$ is~$w_D^{(M)}$, while the Neumann response for~$\mathbf F_0$ is~$w_N^{(M)}+(\mathbf F)_{\cu_M}\cdot x$, since adding this linear function changes precisely the constant normal flux on~$\partial\cu_M$. We apply~\eqref{e.abstract.harmonic.interpolation} to~$v\coloneqq w_N^{(M)}+(\mathbf F)_{\cu_M}\cdot x-w_D^{(M)}$, which is harmonic in~$\cu_M$, and get
\begin{equation*}
	\| \nabla v \|_{\underline L^2(\cu_M)}
	\leq
	C
	\bigl( 3^{-M}\| v-(v)_{\cu_M} \|_{\underline L^2(\cu_M)} \bigr)^{\nf15}
	\| \nabla v \|_{\underline L^4(\cu_M)}^{\nf45}
	\,.
\end{equation*}
Duality, the identity~$(\mathbf F_0)_{\cu_M}=0$, and the normalized~$W^{2,2}$ estimates for the Dirichlet and Neumann Poisson problems give
\begin{equation*}
	3^{-M}\| v-(v)_{\cu_M} \|_{\underline L^2(\cu_M)}
	\leq
	C3^{-M}\| \mathbf F_0 \|_{\Hminusul(\cu_M)}
	\qand
	\| \nabla v \|_{\underline L^4(\cu_M)}
	\leq
	C\| \mathbf F_0 \|_{\underline L^4(\cu_M)}\,.
\end{equation*}
Combining the previous two displays and using
\begin{equation*}
	\nabla w_N^{(M)}-\nabla w_D^{(M)}
	=
	\nabla v-(\mathbf F)_{\cu_M}
	\,,
\end{equation*}
yields~\eqref{e.abstract.response.ND.weak}.
\end{proof}

\smallskip

The first estimates concern the auxiliary response~$w$ which tests the energy generated by the stream tail~$\k_{L'}-\k_{\ell'}$. The stationary response-field comparison above gives upper and lower bounds used in the four-term decomposition~\eqref{e.ellsep.testing}.

\begin{lemma}
\label{l.w.basic.regbounds}
There exists a constant~$C(d)<\infty$ such that
\begin{equation}
	\label{e.nablaw.Lt}
	h^{-\nf 12}\| \nabla w \|_{\underline{L}^8(\cu_m)}
	+
	3^{\ell'} \| \nabla^2 w \|_{\underline{L}^8(\cu_m)}
	+
	3^{\frac12\ell'} \| \nabla w \|_{\underline{H}^{\nf12}(\cu_m)}
	\leq \O_{\Gamma_2} \bigl( C |p| \bigr)
	\,.
\end{equation}
\end{lemma}
\begin{proof}
Set~$\mathbf F\coloneqq(\k_{L'}-\k_{\ell'})p$. Since~$\nabla\cdot\mathbf F=(\f_{L'}-\f_{\ell'})\cdot p$, the function~$w$ is the Dirichlet response~$w_{D,\mathbf F}^{(m)}$ from Lemma~\ref{l.abstract.response.fields}. The Dirichlet estimates in~\eqref{e.abstract.response.L8},~\eqref{e.abstract.response.W28} and~\eqref{e.abstract.response.Hhalf}, with~$\mathbf F$ replaced by~$\mathbf F-(\mathbf F)_{\cu_m}$ where constants are harmless under the divergence, give
\begin{align*}
	\lefteqn{
	h^{-\nf 12}\| \nabla w \|_{\underline{L}^8(\cu_m)}
	+
	3^{\ell'} \| \nabla^2 w \|_{\underline{L}^8(\cu_m)}
	+
	3^{\frac12\ell'} \| \nabla w \|_{\underline{H}^{\nf12}(\cu_m)}
	} \qquad &
	\\ &
	\leq
	C|p| \bigl(
	h^{-\nf 12}
	\| \k_{L'}-\k_{\ell'}-(\k_{L'}-\k_{\ell'})_{\cu_m}
	\|_{\underline L^8(\cu_m)}
	\\ &\qquad
	+
	3^{\ell'}
	\| \nabla(\k_{L'}-\k_{\ell'}) \|_{\underline L^8(\cu_m)}
	+
	3^{\frac12\ell'}
	[ \k_{L'}-\k_{\ell'} ]_{\underline H^{\nf12}(\cu_m)}
	\bigr)\,.
\end{align*}
The three terms in parentheses are bounded by~$\O_{\Gamma_2}(C)$. Indeed, for the first term we split at scale~$m$, using~\eqref{e.kmn.Lp} with~$p=8$ and~$l=m$ for the lower shell, and~\eqref{e.nabla.kmn.Linfty} for the upper shell:
\begin{align*}
	\lefteqn{
	\| \k_{L'}-\k_{\ell'}-(\k_{L'}-\k_{\ell'})_{\cu_m}
	\|_{\underline L^8(\cu_m)}
	} \qquad &
	\\ &
	\leq
	\biggl\| \sum_{k=\ell'+1}^{m-1}\mathbf j_k
	\biggr\|_{\underline L^8(\cu_m)}
	+
	\sum_{k=m}^{L'}3^m\|\nabla\mathbf j_k\|_{L^\infty(\cu_m)}
	\leq
	\O_{\Gamma_2}(C h^{\nf12})\,.
\end{align*}
The second term follows from stationarity,~\ref{a.j.reg}, and Lemma~\ref{l.Gamma.sigma.triangle}, since
\begin{equation*}
	\| \nabla(\k_{L'}-\k_{\ell'}) \|_{\underline L^8(\cu_m)}
	\leq
	\sum_{k=\ell'+1}^{L'}
	\| \nabla\mathbf j_k \|_{\underline L^8(\cu_m)}
	\leq
	\O_{\Gamma_2}(C3^{-\ell'})\,,
\end{equation*}
and the third follows from~\eqref{e.kmn.Hs.osc} with~$s=\nf12$. The~$L^2$ part of the full~$\underline H^{\nf12}$ norm is controlled by the first bound above, since
\begin{equation*}
	3^{\frac12\ell'}3^{-\frac12m}
	\| \k_{L'}-\k_{\ell'}-(\k_{L'}-\k_{\ell'})_{\cu_m}
	\|_{\underline L^2(\cu_m)}
	\leq
	\O_{\Gamma_2}(C)\,.
\end{equation*}
The identity~$m-\ell'=h$ from~\eqref{e.scale.selection} completes the proof.
\end{proof}

We next prove a lower bound estimate for the expectation of the left side of~\eqref{e.ellsep.testing}, which matches the upper bound in~\eqref{e.nablaw.Lt}.

\begin{lemma}
\label{l.LHS.term1}
There exists~$C(d) <\infty$ such that
\begin{equation}
	\label{e.nabla.w.lower.bound}
	\biggl|
	\E \biggl[ \fint_{\cu_m}
	\bigl| \nabla w \bigr|^2 \biggr] - (\log 3) \cstar (m-n) |p|^2
	\biggr|
	\leq
	\bigl(C(1+L'-m) + \nondegconst\bigr) |p|^2
	\,.
\end{equation}
\end{lemma}
\begin{proof}
Set
\begin{equation*}
	\mathbf F\coloneqq(\k_{L'}-\k_{\ell'})p
	=\sum_{r=\ell'+1}^{L'} \mathbf j_rp
	\,.
\end{equation*}
Let~$\hat w \coloneqq \hat w_{\mathbf F}$,~$w \coloneqq w_{D,\mathbf F}^{(m)}$ and~$\tilde w \coloneqq w_{N,\mathbf F}^{(m)}$ be the response fields associated with~$\mathbf F$ in Lemma~\ref{l.abstract.response.fields} with~$M=m$; this Dirichlet response is the function~$w$ from~\eqref{e.def.w}. Then~\eqref{e.energy.comparison.N.D} gives
\begin{equation*}
	\E\bigl[ \| \nabla w \|_{\underline L^2(\cu_m)}^2 \bigr]
	\leq
	\E\bigl[ |\nabla\hat w(0)|^2 \bigr]
	\leq
	\E\bigl[ \| \nabla \tilde{w} \|_{\underline L^2(\cu_m)}^2 \bigr]
	\,.
\end{equation*}
Moreover, let~$w_r\coloneqq w_{D,\mathbf j_rp}^{(m)}$ and~$\tilde w_r\coloneqq w_{N,\mathbf j_rp}^{(m)}$. By linearity,
\begin{equation*}
	w=\sum_{r=\ell'+1}^{L'}w_r
	\qand
	\tilde w=\sum_{r=\ell'+1}^{L'}\tilde w_r
	\,.
\end{equation*}

We estimate the weak norms in~$\cu_m$ scale-by-scale. If~$r<m$, then~\eqref{e.jk.Hminus.endpoint} gives
\begin{equation*}
	3^{-m}\| \mathbf j_rp \|_{\Hminusul(\cu_m)}
	\leq
	\O_{\Gamma_2}\bigl(C|p|(m-r)3^{-(m-r)}\bigr)\,,
\end{equation*}
while for~$r\geq m$, Poincar\'e's inequality and~\ref{a.j.reg} give
\begin{equation*}
	3^{-m}\| \mathbf j_rp-(\mathbf j_rp)_{\cu_m} \|_{\Hminusul(\cu_m)}
	\leq
	C|p|3^m\|\nabla \mathbf j_r\|_{L^\infty(\cu_m)}
	\leq
	\O_{\Gamma_2}\bigl(C|p|3^{-(r-m)}\bigr)
	\,.
\end{equation*}
The following~$\underline L^4$ bound follows from~\eqref{e.kmn.Lp} with~$p=4$ for~$r<m$, and from~\ref{a.j.reg} and Poincar\'e's inequality for~$r\geq m$: for every~$r$,
\begin{equation*}
	\| \mathbf j_rp-(\mathbf j_rp)_{\cu_m} \|_{\underline L^4(\cu_m)}
	\leq
	\O_{\Gamma_2}\bigl(C|p|\bigr)
	\,.
\end{equation*}
Applying~\eqref{e.abstract.response.ND.weak} and using~\eqref{e.jk.spatialavg} to absorb the average term for~$r<m$ (which decays faster than the displayed right side) yields
\begin{equation}
	\| \nabla \tilde w_r-\nabla w_r \|_{\underline L^2(\cu_m)}
	\leq
	\O_{\Gamma_2}\bigl(C|p|(m-r)^{\nf15}3^{-\frac15(m-r)}\bigr)
	\quad \mbox{if } r < m
	\,.
	\label{e.wNr.vs.wDr.r.small}
\end{equation}
For~$r\geq m$, set
\begin{equation*}
	\tilde v_r
	\coloneqq
	\tilde w_r+(\mathbf j_rp)_{\cu_m}\cdot x
	\,.
\end{equation*}
Since~$\cu_m$ is centered,~$(\tilde v_r)_{\cu_m}=0$. Moreover,~$\tilde v_r$ is the Neumann response in~$\cu_m$ associated with the centered flux~$\mathbf j_rp-(\mathbf j_rp)_{\cu_m}$, while~$w_r$ is also the Dirichlet response for this centered flux, since constants disappear under the divergence. Applying~\eqref{e.abstract.response.ND.weak} to this centered flux gives
\begin{equation*}
	\| \nabla \tilde v_r-\nabla w_r \|_{\underline L^2(\cu_m)}
	\leq
	\O_{\Gamma_2}\bigl(C|p|3^{-\frac15(r-m)}\bigr)
	\,.
\end{equation*}
Since~$\nabla \tilde w_r=\nabla\tilde v_r-(\mathbf j_rp)_{\cu_m}$, we obtain
\begin{equation*}
	\| \nabla \tilde w_r-\nabla w_r \|_{\underline L^2(\cu_m)}
	\leq
	\O_{\Gamma_2}\bigl(C|p|3^{-\frac15(r-m)}\bigr)
	+
	\bigl|(\mathbf j_rp)_{\cu_m}\bigr|
	\,.
\end{equation*}

\smallskip

Summing the preceding estimates, and retaining only the average part of the scales~$r\geq m$, we obtain
\begin{equation}
	\label{e.wN.wD.with.average.error}
	\| \nabla \tilde w-\nabla w \|_{\underline L^2(\cu_m)}
	\leq
	\O_{\Gamma_2}(C|p|)
	+
	\biggl| \sum_{r= m}^{L'} (\mathbf{j}_r p)_{\cu_m} \biggr|
	\,.
\end{equation}
By independence of the scales~\ref{a.j.indy}, the zero-mean consequence of the dihedral symmetry assumption~\ref{a.j.iso}, and~\eqref{e.jk.spatialavg},
\begin{equation}
	\label{e.average.error.energy}
	\E\biggl[
	\biggl| \sum_{r= m}^{L'} (\mathbf{j}_r p)_{\cu_m} \biggr|^2
	\biggr]
	\leq
	C\bigl(1+L'-m\bigr)|p|^2
	\,.
\end{equation}
Combining~\eqref{e.wN.wD.with.average.error} and~\eqref{e.average.error.energy} with~\eqref{e.energy.comparison.N.D} and~\eqref{e.energy.quadratic.N.D} yields
\begin{equation*}
	\Bigl|
	\E\bigl[ \| \nabla w \|_{\underline L^2(\cu_m)}^2 \bigr]
	-
	\E\bigl[ |\nabla\hat w(0)|^2 \bigr]
	\Bigr|
	\leq
	C\bigl(1+L'-m\bigr)|p|^2
	\,.
\end{equation*}
The scale identities~\eqref{e.scale.selection} give~$L'-\ell'=m-n$, and assumption~\ref{a.j.nondeg} gives
\begin{equation}
	\label{e.use.nondeg.ass}
	\Bigl|
	\E \bigl[ |\nabla\hat w(0)|^2 \bigr]
	- \cstar(\log 3) (m-n) |p|^2\Bigr|
	\leq \nondegconst |p|^2\,.
\end{equation}
The last two displays prove~\eqref{e.nabla.w.lower.bound}.
\end{proof}

We next estimate the expectation of the first term on the right side of~\eqref{e.ellsep.testing}.

\begin{lemma}
\label{l.RHS.term1}
There exists a constant~$C(d)<\infty$ such that
\begin{equation}
	\label{e.RHS.term1}
	\left|
	\E \Biggl[
	\fint_{\cu_m}
	\nabla w \cdot ( \a_{\ell} \nabla u_{n} - q)
	\Biggr]
	\right|
	\leq
	C \nu^{-3} \ell^2 h \bigl( 3^{-\frac12(\ell-n)} + 3^{-(\ell'-\ell)} \bigr)
	\,.
\end{equation}
\end{lemma}
\begin{proof}
The idea is to find stochastic cancellations between scales~$3^\ell$ and~$3^m$. Since~$u_n$ is defined using the cutoff~$L'$ rather than~$\ell$, the field~$\a_{\ell}\nabla u_n$ does not have range of dependence~$O(3^\ell)$, and we therefore introduce a localized approximation and estimate it separately from the localization error.  For each~$z \in 3^n\Zd$, let~$\tilde{u}_{n,z} \coloneqq  v_{\ell}(\cdot,z+\cu_n,0,\shom_{L',*}^{\nf12}(\cu_n)e)$ denote the maximizer of~$J_\ell(z+\cu_n,0,\shom_{L',*}^{\nf12}(\cu_n)e)$. This uses the same variational direction as~$u_{n,z}$, but with cutoff~$\ell$. Let~$\tilde{q}  \coloneqq  \E[(\a_\ell \nabla \tilde{u}_n)_{\cu_\ell}]$. Since~$\tilde{u}_{n,z}$ depends only on~$\a_\ell$, the centered flux field~$\a_\ell \nabla \tilde{u}_n-\tilde q$ has range of dependence~$C3^\ell$. Define~$\tilde{u}_n$ as in~\eqref{e.u.k.def} using~$\tilde{u}_{n,z}$ instead of~$u_{n,z}$. We estimate using triangle inequality and Cauchy-Schwarz as follows:
\begin{align}
	\lefteqn{
	\biggl| \fint_{\cu_m}
	\nabla w \cdot ( \a_{\ell} \nabla u_{n} - q) \biggr|
	} \qquad &
	\notag \\ &
	\leq
	\biggl| \fint_{\cu_m}
	\nabla w \cdot ( \a_{\ell} \nabla \tilde{u}_{n} - \tilde q) \biggr|
	+
	\| \nabla w \|_{\underline{L}^2(\cu_m)} \bigl(
	\| \a_{\ell} (\nabla u_{n} - \nabla \tilde{u}_n) \|_{\underline{L}^2(\cu_m)}
	+
	| \tilde q - q| \bigr)
	\,.
	\label{e.decompose.flux.u.n}
\end{align}
We estimate the two terms on the right side separately.

\smallskip

\emph{Step 1.} We show that
\begin{equation}
	\E\Bigl[ \| \nabla w \|_{\underline{L}^2(\cu_m)} \bigl(
	\| \a_{\ell} (\nabla u_{n} - \nabla \tilde{u}_n) \|_{\underline{L}^2(\cu_m)}
	+
	| \tilde q - q| \bigr) \Bigr]
	\leq
	C \nu^{-\nf 32} \ell h^{\nf 12} 3^{-\frac12(\ell-n)}
	\,.
	\label{e.decompose.flux.u.n.first}
\end{equation}
First, by~\eqref{e.commute.k0}, we have, for every~$Q \in \Rd$,
\begin{equation*}
	J(z'+\cu_n,0,Q \, ; \a_{L'}) = J(z'+\cu_n,0,Q \, ; \a_{L'} - (\k_{L'} - \k_{\ell})_{z'+\cu_n})
	= \frac12 Q \cdot \s_{L',*}^{-1}(z'+\cu_{n}) Q
\end{equation*}
and
\begin{equation*}
	v_{L'}(\cdot,z'+\cu_n,0,Q ) = v(\cdot,z'+\cu_n,0,Q \, ; \a_{L'} - (\k_{L'} - \k_{\ell})_{z'+\cu_n})
	\,.
\end{equation*}
By a similar computation for~$\a_{\ell}$ instead of~$\a_{L'}$, we thus obtain
\begin{equation}
	\label{e.un.tilde.un.energy}
	\| \nabla u_n \|_{\underline{L}^2(\cu_m)}^2
	+
	\| \nabla \tilde{u}_n \|_{\underline{L}^2(\cu_m)}^2
	\leq
	2\nu^{-2}  \shom_{L',*}(\cu_n)
	\,.
\end{equation}
Then~\eqref{e.localization.minimizers} yields that
\begin{align}
	\lefteqn{
	\bigl\| \nabla u_{n,z'} - \nabla \tilde{u}_{n,z'}   \bigr\|_{\underline{L}^2(z'+\cu_n)}^2
	} \qquad &
	\notag \\ &
	\leq
	C \nu^{-2} 3^n \bigl\| \nabla(\k_{L'}-\k_{\ell})\bigr\|_{L^\infty(z'+\cu_n)}
	\shom_{L',*}(\cu_n) \bigl| \s_{L',*}^{-1}(z'+\cu_n) +  \s_{\ell,*}^{-1}(z'+\cu_n)  \bigr|
	\notag \\ &
	\leq
	C \nu^{-3} 3^n \bigl\| \nabla(\k_{L'}-\k_{\ell})\bigr\|_{L^\infty(z'+\cu_n)} \shom_{L',*}(\cu_n)
	\,.
	\label{e.localization.minimizers.applied}
\end{align}
Here the translated version of~\eqref{e.localization.minimizers} is used on~$z'+\cu_n$: the oscillation of~$\k_{L'}-\k_{\ell}$ around its average is bounded by
\begin{equation*}
	\| \k_{L'}-\k_{\ell}-(\k_{L'}-\k_{\ell})_{z'+\cu_n}\|_{L^\infty(z'+\cu_n)}
	\leq
	C3^n \bigl\| \nabla(\k_{L'}-\k_{\ell})\bigr\|_{L^\infty(z'+\cu_n)}
	\,,
\end{equation*}
and the variational energy factor in~\eqref{e.localization.minimizers} is bounded by the preceding identities for~$J$ and the scalar inverse matrices. More explicitly, for~$Q=\shom_{L',*}^{\nf12}(\cu_n)e$, both relevant energies are bounded by~$\nu^{-1}|Q|^2=\nu^{-1}\shom_{L',*}(\cu_n)$. Thus, by~\eqref{e.nabla.kmn.Linfty},~\eqref{e.kmn.Linfty} and stationarity
\begin{align*}
	\lefteqn{
	\E \Bigl[ \bigl\|  \a_\ell ( \nabla u_{n} - \nabla \tilde{u}_{n} )   \bigr\|_{\underline{L}^2(\cu_m)}^2 \Bigr]
	} \qquad &
	\notag \\ &
	=
	\avsum_{z \in 3^n \Zd \cap \cu_m}
	\E \Bigl[ \bigl\|  \a_\ell ( \nabla u_{n} - \nabla \tilde{u}_{n} )   \bigr\|_{\underline{L}^2(z+\cu_n)}^2 \Bigr]
	\notag \\ &
	\leq
	C \nu^{-3}  \E\bigl[ \| \a_\ell \|_{L^\infty(\cu_\ell)}^4 \bigr]^{\nf 12}
	\shom_{L',*}(\cu_n)
	\avsum_{z \in 3^n \Zd \cap \cu_m}
	3^{n}
	\E\bigl[ \bigl\| \nabla(\k_{L'}-\k_{\ell})\bigr\|_{L^\infty(z+\cu_n)}^2 \bigr]^{\nf 12}
	\notag \\ &
	\leq
	C \nu^{-3} \ell^2 3^{-(\ell-n)} \shom_{L',*}(\cu_n)
	\,.
\end{align*}
We therefore also get
\begin{equation*}
	| \tilde q - q|^2 \leq C \nu^{-3} \ell^2 3^{-(\ell-n)} \shom_{L',*}(\cu_n)
	\,.
\end{equation*}
Moreover, we have by~\eqref{e.nablaw.Lt} that
\begin{equation*}
	\E \Bigl[ \bigl\|  \nabla w  \bigr\|_{\underline{L}^2(\cu_m)}^2 \Bigr]
	\leq
	C h |p|^2
	=
	C h \shom_{L',*}^{-1}(\cu_n)
	\,.
\end{equation*}
Combining the above three estimates using the Cauchy--Schwarz inequality yields~\eqref{e.decompose.flux.u.n.first}.

\smallskip

\emph{Step 2.} We show that
\begin{equation}
	\E\Biggl[ \biggl| \fint_{\cu_m}
	\nabla w \cdot ( \a_{\ell} \nabla \tilde{u}_{n} - \tilde q) \biggr| \Biggr]
	\leq
	C \nu^{-1} \bigl(\ell m(m- \ell)\bigr)^{\nf12} 3^{-(\ell'-\ell)}
	\,.
	\label{e.decompose.flux.u.n.second}
\end{equation}
Since~$w\in H^1_0(\cu_m)$, we have~$(\nabla w)_{\cu_m}=0$, so the hatted dual norm applies. We compute
\begin{equation*}
	\biggl| \fint_{\cu_m}
	\nabla w \cdot ( \a_{\ell} \nabla \tilde{u}_{n} - \tilde q) \biggr|
	\leq
	\| \nabla w \|_{\underline{H}^1(\cu_m)}
	\| \a_{\ell} \nabla \tilde{u}_{n} -  \tilde q \|_{\Hminusul(\cu_m)}
	\,.
\end{equation*}
By~\eqref{e.Sec3.p.q.def} and~\eqref{e.nablaw.Lt}, we have
\begin{equation*}
	\E \Bigl[ \| \nabla w \|_{\underline{H}^1(\cu_m)}^2\Bigr]
	\leq
	C3^{-2\ell'}  \shom_{L',*}^{-1} (\cu_n)
	\,.
\end{equation*}
We next estimate~$\| \a_\ell \nabla \tilde u_n - \tilde q \|_{\Hminusul(\cu_m)}$. Since the localized field~$\a_\ell \nabla \tilde{u}_n - \tilde q$ has range of dependence~$O(3^\ell)$ and is centered so that~$\E[(\a_\ell \nabla \tilde{u}_n - \tilde q)_{\cu_\ell}] = 0$, we obtain by~\eqref{e.kmn.Linfty} and~\eqref{e.un.tilde.un.energy} that
\begin{equation*}
	\E\bigl[ \bigl\| \a_\ell \nabla \tilde{u}_n - \tilde q  \bigr\|_{\underline{L}^2(\cu_m)}^2 \bigr]
	\leq
	C \E\bigl[ \| \a_\ell \|_{L^\infty(\cu_m)}^4 \bigr]^{\nf 12}
	\E\bigl[ \| \nabla \tilde{u}_n \|_{\underline{L}^2(\cu_m)}^4 \bigr]^{\nf 12}
	\leq C \nu^{-2} \ell m  \shom_{L',*} (\cu_n)
	\,.
\end{equation*}
For~$k\geq\ell$, each~$3^k$--cube is decomposed into aligned~$3^\ell$--blocks. After splitting these blocks into a bounded number of sublattices, the~$O(3^\ell)$ range of dependence allows us to apply Proposition~\ref{p.concentration}, and the definition of~$\tilde q$ centers every such block by stationarity. Thus, decomposing triadically over scales~$k\in\Z$ with~$k\leq m$, we obtain
\begin{align*}
	\E\bigl[ \| \a_\ell \nabla \tilde{u}_n - \tilde q \|_{\Hminusul(\cu_m)}^2 \bigr]
	&
	\leq
	C \sum_{k=-\infty}^m  3^{2k}
	\avsum_{z \in 3^k \Zd \cap \cu_m} \E\bigl[ \bigl| ( \a_\ell \nabla \tilde{u}_n - \tilde q)_{z+\cu_k} \bigr|^2 \bigr]
	\notag \\ &
	\leq
	C \E\bigl[ \bigl\| \a_\ell \nabla \tilde{u}_n - \tilde q  \bigr\|_{\underline{L}^2(\cu_m)}^2 \bigr] \sum_{k=-\infty}^m  3^{2k} 3^{-d(k-\ell)_+}
	\notag \\ &
	\leq
	C  \nu^{-2} \ell m \shom_{L',*} (\cu_n) \sum_{k=-\infty}^m  3^{2k} 3^{-d(k-\ell)_+}
	\notag \\ &
	\leq
	C  \nu^{-2} \ell m (m- \ell) 3^{2\ell}  \shom_{L',*} (\cu_n)
	\,.
\end{align*}
The above four displays yield~\eqref{e.decompose.flux.u.n.second}; the coarser form needed in~\eqref{e.RHS.term1} follows from~$m \leq C\ell$ and~$m-\ell\leq 2h$ for the selected scales.

\smallskip

\emph{Step 3.} Conclusion. We combine~\eqref{e.decompose.flux.u.n.first} and~\eqref{e.decompose.flux.u.n.second} with~\eqref{e.decompose.flux.u.n} to obtain the result.
\end{proof}

We next estimate the expectation of the second term on the right side of~\eqref{e.ellsep.testing}.

\begin{lemma}
\label{l.RHS.term2}
There exists~$C(d)<\infty$ such that, for the scales selected in~\eqref{e.scale.selection},
\begin{equation}
	\label{e.RHS.term2}
	\left|
	\E \Biggl[
	\avsum_{z\in 3^n\Zd\cap \cu_m}
	\fint_{z+\cu_n}
	\nabla w \cdot (\k_{L'} - \k_{\ell} ) ( \nabla u_{n,z} - p )
	\Biggr]
	\right|
	\leq
	C\nu^{-3} (L')^2 3^{-\frac12(\ell-n)}
	\,.
\end{equation}
\end{lemma}
\begin{proof}
We use the same localization device as in the proof of Lemma~\ref{l.RHS.term1}. Let
\begin{equation*}
	R \coloneqq (\k_{L'}-\k_\ell)^t \nabla w \,.
\end{equation*}
With this convention,~$\nabla w\cdot(\k_{L'}-\k_\ell)(\nabla u_{n,z}-p)=R\cdot(\nabla u_{n,z}-p)$. For each~$z \in 3^n\Zd$, let
\begin{equation*}
	\tilde{u}_{n,z}
	\coloneqq
	v_\ell(\cdot,z+\cu_n,0,\shom_{L',*}^{\nf12}(\cu_n)e)
\end{equation*}
and define~$\nabla \tilde{u}_n$ by~\eqref{e.u.k.def}. Set
\begin{equation*}
	\tilde p \coloneqq \E\bigl[ (\nabla\tilde u_n)_{\cu_n}\bigr]\,.
\end{equation*}
By stationarity,~\eqref{e.k.ell.upscales},~\ref{a.j.reg} and H\"older's inequality,
\begin{equation*}
	\E\bigl[
	\|\k_{L'}-\k_\ell\|_{\underline L^4(\cu_m)}^4
	\bigr]^{\nf14}
	\leq C(L'-\ell)^{\nf12}\,,
	\qquad
	3^\ell\E\bigl[
	\|\nabla(\k_{L'}-\k_\ell)\|_{\underline L^4(\cu_m)}^4
	\bigr]^{\nf14}
	\leq C\,.
\end{equation*}
Combining this with~\eqref{e.nablaw.Lt} and the product rule gives
\begin{equation}
	\label{e.RHS.term2.R.bounds}
	\E\bigl[\|R\|_{\underline L^2(\cu_m)}^2\bigr]^{\nf12}
	+ 3^\ell \E\bigl[\|\nabla R\|_{\underline L^2(\cu_m)}^2\bigr]^{\nf12}
	\leq
	C L' |p| \,.
\end{equation}
The localization estimate for minimizers, used exactly as in~\eqref{e.localization.minimizers.applied}, gives
\begin{equation}
	\label{e.RHS.term2.proxy.error}
	\E\bigl[
	\|\nabla u_n-\nabla\tilde u_n\|_{\underline L^2(\cu_m)}^2
	\bigr]^{\nf12}
	+ |\tilde p-p|
	\leq
	C\nu^{-\nf32} L' 3^{-\frac12(\ell-n)}
	\shom_{L',*}^{\nf12}(\cu_n)\,.
\end{equation}
Moreover, by~\eqref{e.un.tilde.un.energy},
\begin{equation}
	\label{e.RHS.term2.proxy.energy}
	\E\bigl[
	\|\nabla\tilde u_n-\tilde p\|_{\underline L^2(\cu_m)}^2
	\bigr]^{\nf12}
	\leq
	C\nu^{-1}\shom_{L',*}^{\nf12}(\cu_n)\,.
\end{equation}

\smallskip

We decompose the desired expectation into the localization error and the localized approximation:
\begin{align*}
	\lefteqn{
	\biggl|
	\E \Biggl[
	\fint_{\cu_m} R\cdot(\nabla u_n-p)
	\Biggr]
	\biggr|
	} \qquad &
	\notag \\ &
	\leq
	\E\biggl[
	\fint_{\cu_m}|R|\,|\nabla u_n-\nabla\tilde u_n|
	\biggr]
	+ |\tilde p-p|\,\E\bigl[\|R\|_{\underline L^2(\cu_m)}^2\bigr]^{\nf12}
	+
	\Biggl|
	\E \Biggl[
	\fint_{\cu_m} R\cdot(\nabla\tilde u_n-\tilde p)
	\Biggr]
	\Biggr|\,.
\end{align*}
The first two terms are bounded by~$C\nu^{-3}(L')^2 3^{-\frac12(\ell-n)}$, using~\eqref{e.RHS.term2.R.bounds},~\eqref{e.RHS.term2.proxy.error} and~$|p|\,\shom_{L',*}^{\nf12}(\cu_n)=1$.

\smallskip

It remains to estimate the localized approximation term. Let
\begin{equation*}
	\mathcal F_> \coloneqq \sigma(\mathbf j_r\,:\, r>\ell)\,.
\end{equation*}
The field~$R$ is~$\mathcal F_>$--measurable, while~$\nabla\tilde u_n$ is measurable with respect to~$\sigma(\mathbf j_r:r\leq\ell)$. Thus, for every~$z\in3^n\Zd$,
\begin{equation*}
	\E\bigl[
	(R)_{z+\cu_n}\cdot
	\bigl((\nabla\tilde u_{n,z})_{z+\cu_n}-\tilde p\bigr)
	\bigr]
	=0\,.
\end{equation*}
Subtracting~$(R)_{z+\cu_n}$ on each cube, applying Poincar\'e's inequality on~$z+\cu_n$, and then using~\eqref{e.RHS.term2.R.bounds} and~\eqref{e.RHS.term2.proxy.energy}, we get
\begin{align*}
	\biggl|
	\E \Biggl[
	\fint_{\cu_m} R\cdot(\nabla\tilde u_n-\tilde p)
	\Biggr]
	\biggr|
	&\leq
	C3^n
	\E\bigl[\|\nabla R\|_{\underline L^2(\cu_m)}^2\bigr]^{\nf12}
	\E\bigl[\|\nabla\tilde u_n-\tilde p\|_{\underline L^2(\cu_m)}^2\bigr]^{\nf12}
	\leq
	C\nu^{-1}L' 3^{-(\ell-n)}
	\,.
\end{align*}
Combining the previous displays, and using~$\nu\leq1$, gives~\eqref{e.RHS.term2}.
\end{proof}

The estimate of the third term on the right in~\eqref{e.ellsep.testing} is the most involved one, and requires some coarse-graining ideas. We first record a consequence of the localization estimates in (the proof of) Lemma~\ref{l.localization}.
\begin{lemma}
For every~$m,n,\ell \in \N$ with~$n < m < \ell$ and every cube~$U=x+\cu_n$, with~$x\in\Rd$, and~$\ep \in (0,1]$,
\begin{equation}
	\label{e.blupbounds}
	\bigl| \bigl( \b_{\ell} - \b_m \bigr)  (U) \bigr|
	\leq
	\ep \bigl| \b_m(U) \bigr| +  (1 +  \ep^{-1}) \bigl| \s_{m,*}^{-\nf 12}(U) (\k_m - \k_{\ell})_{U} \bigr|^2
	+ \O_{\Gamma_{\nicefrac12}}\bigl(
	C \nu^{-3} \ell 3^{n-m}   \bigr)
	\,.
\end{equation}
\end{lemma}
\begin{proof}
Set
\begin{equation*}
	\h \coloneqq \k_\ell-\k_m
	\qand
	\h_U \coloneqq (\k_\ell-\k_m)_U
	\,.
\end{equation*}
We start from the decomposition
\begin{equation*}
	\bfA_\ell(U)-\bfA_m(U)
	=
	\G_{-\h_U}^{t}\bfA_m(U)\G_{-\h_U}-\bfA_m(U)
	+
	\G_{-\h_U}^{t}
	\bigl(\G_{\h_U}^{t}\bfA_\ell(U)\G_{\h_U}-\bfA_m(U)\bigr)
	\G_{-\h_U}\,.
\end{equation*}
Since~$\a_\ell-\h_U = \a_m +(\h-\h_U)$, Lemma~\ref{l.localization.A}, applied with~$\a=\a_m$, perturbation field~$\h-\h_U$ and domain~$U$, gives
\begin{equation*}
	-D\bfA_m(U)
	\leq
	\bfA(U;\a_\ell-\h_U)-\bfA_m(U)
	\leq
	D\bfA_m(U)\,,
\end{equation*}
where, by the same estimate as in~\eqref{e.Tsizebounds}--\eqref{e.mclDsizebounds},
\begin{equation*}
	D
	\coloneqq
	\nu^{-1}\|\h-\h_U\|_{L^\infty(U)}
	+
	\nu^{-2}\|\h-\h_U\|_{L^\infty(U)}^2
	\leq
	\O_{\Gamma_1}(C\nu^{-2}3^{n-m})\,.
\end{equation*}
Here and below, the estimates from Subsection~\ref{ss.localization} are used on the translated cube~$U$; stationarity gives the same~$\O_{\Gamma_\sigma}$ bounds as on~$\cu_n$. By~\eqref{e.commute.coarse.grained.k0},
\begin{equation*}
	\bfA(U;\a_\ell-\h_U)
	=
	\G_{\h_U}^{t}\bfA_\ell(U)\G_{\h_U}\,.
\end{equation*}
Thus, for every~$P\in\R^{2d}$ and~$Q\coloneqq \G_{-\h_U}P$,
\begin{equation}
	\label{e.blupbounds.localized.error}
	\bigl|
	P\cdot
	\G_{-\h_U}^{t}
	\bigl(\G_{\h_U}^{t}\bfA_\ell(U)\G_{\h_U}-\bfA_m(U)\bigr)
	\G_{-\h_U}
	P
	\bigr|
	\leq
	D\, Q\cdot\bfA_m(U)Q\,.
\end{equation}
For~$P=(e,0)$ with~$|e|=1$, we have~$Q=(e,-\h_U e)$. By~\eqref{e.kmn.bounds},
\begin{equation*}
	|\h_U|^2
	\leq
	\O_{\Gamma_1}\bigl(C(\ell-m)\bigr)\,.
\end{equation*}
Indeed, this is~\eqref{e.kmn.bounds} with~$p=2$, applied to~$\h=\k_\ell-\k_m$, together with~\eqref{e.powerofGammasigma}. By the block form~\eqref{e.bigA.def} and Cauchy-Schwarz,
\begin{align*}
	Q\cdot\bfA_m(U)Q
	&=
	e\cdot\b_m(U)e
	+
	2e\cdot\h_U^t\s_{m,*}^{-1}(U)\kcg_m(U)e
	+
	e\cdot\h_U^t\s_{m,*}^{-1}(U)\h_U e
	\notag \\
	&\leq
	2e\cdot\b_m(U)e
	+
	2e\cdot\h_U^t\s_{m,*}^{-1}(U)\h_U e
	\notag \\
	&\leq
	2 | \b_m(U)| + 2|\h_U|^2 | \s_{m,*}^{-1}(U)|\,.
\end{align*}
In the Cauchy-Schwarz step we used that the upper-left block satisfies~$\kcg_m^t(U)\s_{m,*}^{-1}(U)\kcg_m(U)\leq \b_m(U)$ by~\eqref{e.bigA.def}. By~\eqref{e.CG.bounds.1}, since the symmetric part of~$\a_m$ is~$\nu\Id$,~$|\s_{m,*}^{-1}(U)| \leq \nu^{-1}$. By~\eqref{e.Enaught.mixing} and~\eqref{e.Enaught.vs.A.and.Ahom},
\begin{equation*}
	|\b_m(U)|
	\leq
	|\bfA_m(U)|
	\leq
	\O_{\Gamma_1}\bigl(C\nu^{-1}m\bigr)\,.
\end{equation*}
Using these bounds and Lemma~\ref{l.Gamma.sigma.triangle}, we obtain
\begin{equation*}
	Q\cdot\bfA_m(U)Q
	\leq
	2 | \b_m(U)| + 2|\h_U|^2 | \s_{m,*}^{-1}(U)|
	\leq
	\O_{\Gamma_{1}}\bigl(C\nu^{-1}\ell \bigr)\,.
\end{equation*}
Combining this with~\eqref{e.blupbounds.localized.error} and Lemma~\ref{l.o.gamma2.mult} yields
\begin{equation}
	\label{e.blupbounds.remainder}
	\bigl|
	P\cdot
	\G_{-\h_U}^{t}
	\bigl(\G_{\h_U}^{t}\bfA_\ell(U)\G_{\h_U}-\bfA_m(U)\bigr)
	\G_{-\h_U}
	P
	\bigr|
	\leq
	\O_{\Gamma_{\nicefrac12}}\bigl(C\nu^{-3}\ell 3^{n-m}\bigr)\,.
\end{equation}

\smallskip

It remains to estimate the first term on the right side of the initial decomposition. For~$P=(e,0)$, this term satisfies
\begin{align*}
	P\cdot
	\bigl(
	\G_{-\h_U}^{t}\bfA_m(U)\G_{-\h_U}
	-\bfA_m(U)
	\bigr)P
	&=
	e\cdot \h_U^t\s_{m,*}^{-1}(U)\h_U e
	{}+
	2e\cdot\h_U^t\s_{m,*}^{-1}(U)\kcg_m(U)e
	\,.
\end{align*}
Young's inequality gives
\begin{align*}
	2\bigl|e\cdot\h_U^t\s_{m,*}^{-1}(U)\kcg_m(U)e\bigr|
	&\leq
	\ep\bigl|\s_{m,*}^{-\nf12}(U)\kcg_m(U)e\bigr|^2
	+
	\ep^{-1}\bigl|\s_{m,*}^{-\nf12}(U)\h_U e\bigr|^2
	\\
	&\leq
	\ep\, e\cdot\b_m(U)e
	+
	\ep^{-1}e\cdot\h_U^t\s_{m,*}^{-1}(U)\h_U e\,.
\end{align*}
Since the upper-left block of~$\bfA_\ell(U)-\bfA_m(U)$ is the symmetric matrix~$(\b_\ell-\b_m)(U)$ and
\begin{equation*}
	\h_U^t\s_{m,*}^{-1}(U)\h_U
	=
	(\k_m-\k_\ell)_U^t\s_{m,*}^{-1}(U)(\k_m-\k_\ell)_U\,,
\end{equation*}
taking the supremum over~$|e|=1$ and using~\eqref{e.blupbounds.remainder} gives~\eqref{e.blupbounds}.
\end{proof}

\smallskip

With this localization estimate available, we can treat the third term in~\eqref{e.ellsep.testing}. This is the only term in which additivity of the coarse-grained matrices enters directly, so the proof is organized around separating the small-cube oscillations from the coarse-grained average.

\begin{lemma}
\label{l.RHS.term3}
There exists a constant~$C(d)<\infty$ such that
\begin{align}
	\label{e.RHS.term3}
	\E \biggl[
	\fint_{\cu_m} \!\!\!\!
	\nabla w \cdot \a_{L'}  ( \nabla u_m - \nabla u_{n} )
	\biggr]
	&
	\leq
	C \bigl( \delta + \eta_L \bigr)^{\nf12} \Bigl(
	(L'-\ell)^2 \shom_{L',*}^{-1}(\cu_n)^2
	+
	\shom_{\ell}(\cu_n)^2
	\Bigr)^{\nf12}
	\notag \\ &\qquad
	+
	C \nu^{-4}  (L')^4  \bigl( 3^{-\frac 12 (\ell'-n)}  + 3^{-\frac18(\ell-n)} + 3^{-\frac 14(\ell'-\ell)} +  3^{-\frac 18 h} \bigr) \,.
\end{align}
\end{lemma}

\begin{proof}
We split the term on the left in~\eqref{e.RHS.term3} into two parts as follows:
\begin{align}
	\label{e.additivity.defect.splitting}
	\E \Biggl[
	\fint_{\cu_m}
	\nabla w \cdot \a_{L'}  ( \nabla u_m - \nabla u_{n} )
	\Biggr]
	&
	=
	\E \Biggl[
	\avsum_{z\in 3^n\Zd\cap \cu_m}
	\fint_{z+\cu_n}
	\bigl( \nabla w - (\nabla w )_{z+\cu_n}  \bigr) \cdot \a_{L'}  ( \nabla u_m - \nabla u_{n,z} )
	\Biggr]
	\notag \\ &\qquad
	+
	\E \Biggl[
	\avsum_{z\in 3^n\Zd\cap \cu_m}
	\bigl(\nabla w \bigr)_{z+\cu_n}
	\cdot \bigl( \a_{L'}  ( \nabla u_m - \nabla u_{n,z} )
	\bigr)_{z+\cu_n}
	\Biggr]
	\,.
\end{align}
We estimate the second term on the right side by coarse-graining, while the first term will be estimated directly using the separation of scales between the small cubes~$z+\cu_n$ and the scales on which~$w$ varies, which are at least of order~$3^{\ell'}$.  The two terms on the right side of~\eqref{e.additivity.defect.splitting} are estimated in the following four steps, in~\eqref{e.RHS.term3.B} and~\eqref{e.RHS.term3.A}, respectively. Together, these inequalities imply~\eqref{e.RHS.term3}. Throughout the proof,~$\b(U)\coloneqq(\s+\k^t\s_*^{-1}\k)(U)$ denotes the upper-left~$d$-by-$d$ block of~$\bfA(U)$.

\smallskip

\emph{Step 1.} We estimate the first term on the right side of~\eqref{e.additivity.defect.splitting}. The claim is that
\begin{multline}
	\label{e.RHS.term3.B}
	\E \Biggl[
	\avsum_{z\in 3^n\Zd\cap \cu_m}
	\fint_{z+\cu_n}
	\bigl( \nabla w - (\nabla w )_{z+\cu_n}  \bigr) \cdot
	\a_{L'}  ( \nabla u_m - \nabla u_{n,z} )
	\Biggr]
	\\
	\leq
	C\nu^{-\nicefrac32} \bigl( \delta + \eta_L \bigr)^{\nf12} (L')^{\nf12} 3^{-(\ell' -n)}
	\,.
\end{multline}
Here we use a crude estimate, since scale separation is available. Apply Cauchy--Schwarz and cube Poincar\'e, using the centering of~$\nabla w$, and then invoke~\eqref{e.nablaw.Lt}. We obtain
\begin{align*}
	\lefteqn{
	\fint_{z+\cu_n}
	\bigl( \nabla w - (\nabla w )_{z+\cu_n}  \bigr) \cdot \a_{L'}  ( \nabla u_m - \nabla u_{n,z} )
	} \qquad &
	\notag \\ &
	\leq
	\bigl[ \nabla w - (\nabla w )_{z+\cu_n}   \bigr]_{\underline{H}^1(z+\cu_n)}
	\notag \\ & \qquad{}
	\times
	\|  \a_{L'}  ( \nabla u_m - \nabla u_{n,z} ) \|_{\Hminusul(z+\cu_n)}
	\,.
\end{align*}
By~\eqref{e.nablaw.Lt} we have that
\begin{equation*}
	\avsum_{z\in 3^n \Zd\cap \cu_m} \E\Bigl[
	\bigl[ \nabla w - (\nabla w )_{z+\cu_n}   \bigr]_{\underline{H}^1(z+\cu_n)}^3
	\Bigr]
	\leq
	C
	\avsum_{z\in 3^n \Zd\cap \cu_m}
	\E\Bigl[
	\| \nabla^2 w  \|_{\underline{L}^3(z+\cu_n)}^3
	\Bigr]
	\leq
	C\nu^{-\nicefrac32} 3^{-3\ell'}
\end{equation*}
and, by the multiscale Poincar\'e inequality, H\"older's inequality, and~\eqref{e.energymaps.nonsymm.flux}, by denoting~$\s \coloneqq \nu \Id$,
\begin{align*}
	\lefteqn{
	\E\Bigl[ \|  \a_{L'}  ( \nabla u_m - \nabla u_{n,z} ) \|_{\Hminusul(z+\cu_n)}^{\nicefrac 32} \Bigr]
	} \qquad &
	\notag \\ &
	\leq
	\E\Biggl[
	\biggl(
	\sum_{j = -\infty}^n 3^{j} \biggl( \avsum_{z' \in z + 3^j \Z^d \cap \cu_n} \bigl| \bigl( \a_{L'}  ( \nabla u_m - \nabla u_{n,z} ) \bigr)_{z'+\cu_j} \bigr|^2   \biggr)^{\! \nicefrac12} \biggr)^{\! \nicefrac 32} \Biggr]
	\notag \\ &
	\leq
	C 3^{\frac32 n} \E\biggl[
	\| \s^{\nf12} ( \nabla u_m - \nabla u_{n,z} ) \|_{\underline{L}^2(z+\cu_n)}^{\nicefrac 32}
	\sum_{j = -\infty}^n 3^{j-n} \max_{z' \in z + 3^j \Z^d \cap \cu_n} \bigl| \b_{L'}(z'+\cu_j) \bigr|^{ \nicefrac 34}  \biggr]
	\notag \\ &
	\leq
	C 3^{\frac32 n} (L'\nu^{-1})^{\nicefrac 34}
	\E\Bigl[ \| \s^{\nf12} ( \nabla u_m - \nabla u_{n,z} ) \|_{\underline{L}^2(z+\cu_n)}^2 \Bigr]^{\nicefrac 34} \,.
\end{align*}
The last factor on the right side can be estimated by~\eqref{e.secondvar} and~\eqref{e.pigeon.scalar}:
\begin{align}
	\label{e.additivity.error.superdiff}
	\avsum_{z\in 3^n\Zd\cap \cu_m}
	\E\Bigl[ \| \s^{\nf12} ( \nabla u_m {-} \nabla u_{n,z} ) \|_{\underline{L}^2(z+\cu_n)}^2 \Bigr]
	\leq
	C\bigl| 1-\shom_{L',*}(\cu_n)
	\shom_{L',*} ^{-1}(\cu_m)
	\bigr|
	\leq
	C\bigl(\delta {+} \eta_L\bigr)
	\,.
\end{align}
Combining the above displays with H\"older's inequality yields~\eqref{e.RHS.term3.B}.

\smallskip

\emph{Step 2.} In this step we start to estimate the second term on the right side of~\eqref{e.additivity.defect.splitting}. Let~$k = \lceil \frac{4}{d+4} \ell' + \frac{d}{d+4} \ell \rceil$. The claim is that there exists~$C(d)<\infty$ such that
\begin{align}
	\label{e.RHS.term3.A}
	\lefteqn{
	\E \Biggl[
	\avsum_{z\in 3^n\Zd\cap \cu_m}
	\bigl(\nabla w \bigr)_{z+\cu_n}
	\cdot \bigl(
	\a_{L'}  ( \nabla u_m - \nabla u_{n,z} )
	\bigr)_{z+\cu_n}
	\Biggr]
	} \qquad &
	\notag \\ &
	\leq
	\E \Biggl[ \avsum_{z'\in 3^k\Zd\cap \cu_m}
	\avsum_{z\in z' + 3^n\Zd\cap \cu_k}
	\left|\b_{L'}^{\nf12} (z+\cu_n) \bigl(\nabla w \bigr)_{z'+\cu_k}\right|^2
	\Biggr]^{\nf12}
	(\delta + C \eta_L)^{\nf12} \notag \\
	&\qquad +  C 3^{-\frac{1}{4}(\ell'-\ell)} (L')^2 \nu^{-\nicefrac52} \,.
\end{align}
Thus~$\b_{L'}$ and~$\b_{\ell}$ denote the corresponding coarse-grained matrices for the cutoff fields~$\a_{L'}$ and~$\a_\ell$, respectively. We first record that, by~\eqref{e.energymaps.nonsymm.flux} and~\eqref{e.additivity.error.superdiff},
\begin{align}
	\label{e.flux-additivity-estimate-in-an-lemma}
	\lefteqn{
	\E \Biggl[
	\avsum_{z\in 3^n\Zd\cap \cu_m}
	\!\! \!\!
	\bigl| \b_{L'}^{-\nf12} (z+\cu_n)
	\bigl( \a_{L'}  \nabla ( u_m - u_{n,z} )
	\bigr)_{z+\cu_n} \bigr|^2
	\Biggr]
	} \qquad &
	\notag \\ &
	\leq
	\E \Biggl[
	\avsum_{z\in 3^n\Zd\cap \cu_m}
	\!\!
	\fint_{z+\cu_n}
	\nabla ( u_m - u_{n,z} )
	\cdot
	\s \nabla ( u_m - u_{n,z} )
	\Biggr]
	\leq
	\delta + C \eta_L
	\,.
\end{align}
Decompose the left side of~\eqref{e.RHS.term3.A} as
\begin{align}
	\label{e.w-flux-indepen-decomp}
	\lefteqn{\avsum_{z\in 3^n\Zd\cap \cu_m}
	\bigl(\nabla w \bigr)_{z+\cu_n}
	\cdot \bigl(
	\a_{L'}  ( \nabla u_m - \nabla u_{n,z} )
	\bigr)_{z+\cu_n}} \qquad &
	\notag \\ &
	= \avsum_{z'\in 3^k\Zd\cap \cu_m}
	\avsum_{z\in z' + 3^n\Zd\cap \cu_k}
	\bigl(\nabla w \bigr)_{z'+\cu_k}
	\cdot \bigl(
	\a_{L'}  ( \nabla u_m - \nabla u_{n,z} )
	\bigr)_{z+\cu_n} \notag \\ & \qquad
	+\! \!
	\avsum_{z'\in 3^k\Zd\cap \cu_m}
	\avsum_{z\in z' + 3^n\Zd\cap \cu_k}
	\left(\bigl(\nabla w \bigr)_{z+\cu_n} - \bigl(\nabla w \bigr)_{z'+\cu_k} \right)
	\cdot \bigl(
	\a_{L'}  ( \nabla u_m - \nabla u_{n,z} )
	\bigr)_{z+\cu_n} \,.
\end{align}
The first term on the right of~\eqref{e.w-flux-indepen-decomp} is estimated using Cauchy-Schwarz and~\eqref{e.flux-additivity-estimate-in-an-lemma} as
\begin{align}
	\lefteqn{
	\E \Biggl[
	\avsum_{z'\in 3^k\Zd\cap \cu_m}
	\avsum_{z\in z' + 3^n\Zd\cap \cu_k}
	\bigl(\nabla w \bigr)_{z'+\cu_k}
	\cdot \bigl(
	\a_{L'}  ( \nabla u_m - \nabla u_{n,z} )
	\bigr)_{z+\cu_n}
	\Biggr] } \qquad &
	\notag \\ &
	=
	\E \Biggl[ \avsum_{z'\in 3^k\Zd\cap \cu_m}
	\avsum_{z\in z' + 3^n\Zd\cap \cu_k}
	\b_{L'}^{\nf12} (z+\cu_n) \bigl(\nabla w \bigr)_{z'+\cu_k}
	\notag \\ & \qquad\qquad \times \bigl( \b_{L'}^{-\nf12} (z+\cu_n)
	\a_{L'}  ( \nabla u_m - \nabla u_{n,z} )
	\bigr)_{z+\cu_n} \Biggr]\notag \\
	&\leq
	\E \Biggl[ \avsum_{z'\in 3^k\Zd\cap \cu_m}
	\avsum_{z\in z' + 3^n\Zd\cap \cu_k}
	\left|\b_{L'}^{\nf12} (z+\cu_n) \bigl(\nabla w \bigr)_{z'+\cu_k}\right|^2
	\Biggr]^{\nf12} \notag \\
	& \quad \times
	\E \Biggl[ \avsum_{z\in 3^n\Zd\cap \cu_m}
	\left| \bigl( \b_{L'}^{-\nf12} (z+\cu_n)
	\a_{L'}  ( \nabla u_m - \nabla u_{n,z} ) \bigr)_{z+\cu_n}  \right|^2
	\Biggr]^{\nf12} \notag  \\
	&\leq
	\E \Biggl[ \avsum_{z'\in 3^k\Zd\cap \cu_m}
	\avsum_{z\in z' + 3^n\Zd\cap \cu_k}
	\left|\b_{L'}^{\nf12} (z+\cu_n) \bigl(\nabla w \bigr)_{z'+\cu_k}\right|^2
	\Biggr]^{\nf12}
	\notag \\ & \qquad \times
	(\delta + C \eta_L)^{\nf12} \,.
\end{align}
Similarly, we have
\begin{align}
	\lefteqn{
	\E \Biggl[ \avsum_{z'\in 3^k\Zd\cap \cu_m}
	\avsum_{z\in z' + 3^n\Zd\cap \cu_k}
	\left(\bigl(\nabla w \bigr)_{z+\cu_n} - \bigl(\nabla w \bigr)_{z'+\cu_k} \right)
	\cdot \bigl(
	\a_{L'}  ( \nabla u_m - \nabla u_{n,z} )
	\bigr)_{z+\cu_n}
	\Biggr]
	} \qquad &
	\notag \\
	&\leq
	C \E \Biggl[ \avsum_{z'\in 3^k\Zd\cap \cu_m}
	\avsum_{z\in z' + 3^n\Zd\cap \cu_k}
	\left|\b_{L'}^{\nf12}(z+\cu_n)\left(\bigl(\nabla w \bigr)_{z+\cu_n} - \bigl(\nabla w \bigr)_{z'+\cu_k}\right) \right|^2
	\Biggr]^{\nf12} \notag \\
	&\leq
	C\E \Biggl[ \avsum_{z'\in 3^k\Zd\cap \cu_m} \avsum_{z\in z' + 3^n\Zd\cap \cu_k}
	\left| \bigl(\nabla w \bigr)_{z+\cu_n} - \bigl(\nabla w \bigr)_{z'+\cu_k} \right|^4
	\Biggr]^{\nicefrac14}
	\E\Biggl[\left |\b_{L'}(\cu_n)\right|^2\Biggr]^{\nicefrac14} \notag \\
	&\leq C  (\nu^{-1} L')^{\nf12}  \E \Biggl[ \avsum_{z'\in 3^k\Zd\cap \cu_m} \avsum_{z\in z' + 3^n\Zd\cap \cu_k}
	\left| \bigl(\nabla w \bigr)_{z+\cu_n} - \bigl(\nabla w \bigr)_{z'+\cu_k} \right|^4
	\Biggr]^{\nicefrac14}
	\notag \\ &
	\leq C 3^{k}  (\nu^{-1} L')^{\nf12} \E \bigl[
	\|\nabla^2 w  \|_{\underline{L}^4(\cu_m)}^4
	\bigr]^{\nicefrac14}  \notag \\
	&\leq C 3^{k-\ell'} (\nu^{-1} L')^{\nf12} \nu^{-\nf12}
	\leq C 3^{-\frac{d}{d+4}(\ell'-\ell)} (L')^2 \nu^{-\nicefrac52} \,,
\end{align}
where in the last inequality we used~\eqref{e.nablaw.Lt} and~\eqref{e.p-bound-crude}. Combining the above displays yields~\eqref{e.RHS.term3.A}.

\smallskip

\emph{Step 3.} We are left to estimate the expectation on the right in~\eqref{e.RHS.term3.A}. For this we use the finite-range dependence of the cutoff-$\ell$ coarse-grained variables and the regularity of~$w$. The claimed estimate is that there exists a constant~$C(d)<\infty$ such that
\begin{align}
	\label{e.bL.to.bhomell}
	\lefteqn{
	\E \Biggl[ \avsum_{z'\in 3^k\Zd\cap \cu_m}
	\avsum_{z\in z' + 3^n\Zd\cap \cu_k}
	\left|\b_{L'}^{\nf12} (z+\cu_n) \bigl(\nabla w \bigr)_{z'+\cu_k}\right|^2
	\Biggr]
	} \qquad &
	\notag \\ &
	\leq
	C \left( (L'-\ell)^2 \shom_{L',*}^{-1}(\cu_n)^2
	+
	\shom_{\ell}(\cu_n)^2
	+
	\nu^{-4} (L')^4 \bigl( 3^{-\frac14(\ell-n)} + 3^{-\frac 12(\ell'-\ell)} + 3^{-\frac18 h}\bigr)
	\right)
	\,.
\end{align}
Recall the choice~$k = \lceil \frac{4}{d+4} \ell' + \frac{d}{d+4} \ell \rceil$ from Step~2. To prove~\eqref{e.bL.to.bhomell}, we first decompose
\begin{align*}
	\left| \b_{L'}(z+\cu_n) \right| &\leq \left| \b_{L'}(z+\cu_n) - \b_{\ell}(z+\cu_n)\right|
	+ \left| \b_{\ell}(z+\cu_n)\right|
\end{align*}
and bound the first term using~\eqref{e.blupbounds}, with the fixed choice~$\ep=1$, as
\begin{align*}
	\left| \b_{L'}(z+\cu_n) - \b_{\ell}(z+\cu_n)\right|
	&\leq
	C | (\k_{L'}-\k_{\ell})_{z+\cu_n}^t \s_{\ell,*}^{-1}(z+\cu_n) (\k_{L'}-\k_{\ell})_{z+\cu_n}|
	\notag \\ & \qquad
	+
	C |\b_{\ell}(z+\cu_n)|
	+
	\O_{\Gamma_{\nf13}}\bigl(C \nu^{-3} L' 3^{-(\ell-n)} \bigr) \,,
\end{align*}
and then continue to decompose,
\begin{align*}
	\left| \b_{\ell}(z+\cu_n) \right| \leq
	d \shom_{\ell}(\cu_n)  +   \sum_{i=1}^d \left| e_i \cdot \left( \b_{\ell}(z+\cu_n) - \shom_{\ell}(\cu_n) \right) e_i \right|
	\,.
\end{align*}
By combining the above three displays we have,
\begin{align}
	\label{e.bL.to.bhomell.pre0zz}
	\left| \b_{L'}(z+\cu_n) \right|
	&\leq  C \shom_{\ell}(\cu_n) + C | (\k_{L'}-\k_{\ell})_{z+\cu_n}^t \s_{\ell,*}^{-1}(z+\cu_n) (\k_{L'}-\k_{\ell})_{z+\cu_n}| \notag
	\\
	&\qquad + C \sum_{i=1}^d \left| e_i \cdot \left( \b_{\ell}(z+\cu_n) - \shom_{\ell}(\cu_n) \right) e_i \right|
	+
	\O_{\Gamma_{\nf13}}\bigl(C \nu^{-3} L' 3^{-(\ell-n)} \bigr) \,.
\end{align}
By the relative comparison~\eqref{e.localization.s.star}, the quadratic forms associated with~$\s_{\ell,*}^{-1}(z+\cu_n)$ and~$\s_{L',*}^{-1}(z+\cu_n)$ differ by a factor of~$1 + \O_{\Gamma_1}(C\nu^{-2}3^{-(\ell-n)})$. More explicitly, after multiplying by the slowly varying weight coming from~$w$ and averaging, this replacement error satisfies
\begin{multline*}
	\E \Biggl[ \avsum_{z'\in 3^k\Zd\cap \cu_m} |\bigl(\nabla w \bigr)_{z'+\cu_k}|^2
	\avsum_{z\in z' + 3^n\Zd\cap \cu_k}
	\bigl|
	(\k_{L'}-\k_{\ell})_{z+\cu_n}^t
	\bigl(\s_{\ell,*}^{-1}-\s_{L',*}^{-1}\bigr)(z+\cu_n)
	(\k_{L'}-\k_{\ell})_{z+\cu_n}
	\bigr|
	\Biggr]
	\\
	\leq
	C\nu^{-4}(L')^4 3^{-\frac14(\ell-n)}
	\,.
\end{multline*}
Indeed, this follows from~\eqref{e.localization.s.star},~\eqref{e.jk.spatialavg} and~\eqref{e.nablaw.Lt}, after lowering the exponential rate and absorbing the polynomial factors into the bound~$\nu^{-4}(L')^4$. Thus we may replace~$\s_{\ell,*}^{-1}$ by~$\s_{L',*}^{-1}$ in the quadratic term of~\eqref{e.bL.to.bhomell.pre0zz}. We now estimate using H\"older's inequality and~\eqref{e.bL.to.bhomell.pre0zz},
\begin{align}
	\notag
	\lefteqn{
	\E \Biggl[ \avsum_{z'\in 3^k\Zd\cap \cu_m} |\bigl(\nabla w \bigr)_{z'+\cu_k}|^2
	\avsum_{z\in z' + 3^n\Zd\cap \cu_k}
	\left|\b_{L'} (z+\cu_n) \right|
	\Biggr]
	} \qquad &
	\notag \\ &
	\leq
	C \shom_{\ell}(\cu_n)
	\E \Bigl[ \bigl\| \nabla w  \bigr\|_{\underline{L}^2(\cu_{m})}^2
	\Bigr]
	\notag \\ &
	\qquad
	+
	C \sum_{i=1}^d \E \Biggl[ \biggl| \avsum_{z\in 3^n\Zd\cap \cu_k} e_i \cdot \bigl(\b_{\ell}(z+\cu_n)  - \shom_{\ell}(\cu_n)\bigr)e_i \biggr|^2 \Biggr]^{\nf12}
	\E \Bigl[ \bigl\| \nabla w  \bigr\|_{\underline{L}^4(\cu_{m})}^4
	\Bigr]^{\nf12}
	\notag \\ &
	\qquad +
	C \E \Biggl[ \biggl| \avsum_{z\in 3^n\Zd\cap \cu_k} (\k_\ell-\k_{L'})_{z+\cu_n}^t \s_{L',*}^{-1}(z+\cu_n) (\k_\ell-\k_{L'})_{z+\cu_n} \biggr|^2 \Biggr]^{\nf12}
	\E \Bigl[ \bigl\| \nabla w  \bigr\|_{\underline{L}^4(\cu_{m})}^4
	\Bigr]^{\nf12}
	\notag \\ &
	\qquad +C \nu^{-3} L' 3^{-(\ell-n)} \E \Bigl[ \bigl\| \nabla w  \bigr\|_{\underline{L}^4(\cu_{m})}^4 \Bigr]^{\nf12} \,.
\end{align}
We estimate these terms separately. First,~\eqref{e.nablaw.Lt} gives the sharper factor~$h|p|^2$, and since~$h\leq L'-\ell$, we have
\begin{equation*}
	\shom_{\ell}(\cu_n)
	\E \Bigl[ \bigl\| \nabla w  \bigr\|_{\underline{L}^2(\cu_{m})}^2
	\Bigr]  \leq C \shom_{\ell}(\cu_n) (L'-\ell) \shom_{L',*}^{-1}(\cu_{n}) \,.
\end{equation*}
Next, for each coordinate vector~$e_i$, the centered scalar variables~$Y_z^{(i)}  \coloneqq e_i\cdot(\b_{\ell}(z+\cu_n) - \shom_{\ell}(\cu_n))e_i$ for~$z\in 3^n\Zd\cap\cu_k$ satisfy~$Y_z^{(i)} = \O_{\Gamma_1}(C\ell\nu^{-1})$. Before applying concentration, we note explicitly that each~$Y_z^{(i)}$ is a function only of the cutoff-$\ell$ environment in a~$C3^\ell$ neighborhood of~$z+\cu_n$; this follows from the definition of the localized coarse-grained matrix and~\ref{a.j.frd}. Hence the family has range of dependence~$C3^\ell$. We partition the lattice~$3^n\Zd\cap\cu_k$ into~$O(3^{d(\ell-n)})$ sublattices with spacing at least~$C3^\ell$, so that within each sublattice the variables~$\{Y_z^{(i)}\}$ are independent. Applying Proposition~\ref{p.concentration} to each sublattice (each of which contains~$O(3^{d(k-\ell)})$ independent terms) and summing over the finitely many sublattices via the triangle inequality, we obtain
\begin{equation*}
	\E \Biggl[ \biggl| \avsum_{z\in 3^n\Zd\cap \cu_k} Y_z^{(i)} \biggr|^2 \Biggr]^{\nf12}
	\leq
	C \ell\nu^{-1} 3^{-\frac{d}{2} (k-\ell)}\,.
\end{equation*}
Combining this with~\eqref{e.nablaw.Lt}, we get
\begin{align*}
	\sum_{i=1}^d
	\E \Biggl[ \biggl| \avsum_{z\in 3^n\Zd\cap \cu_k} Y_z^{(i)} \biggr|^2 \Biggr]^{\nf12}
	\E \Bigl[ \bigl\| \nabla w  \bigr\|_{\underline{L}^4(\cu_{m})}^4
	\Bigr]^{\nf12} &\leq
	\frac{C}{\nu^2} (m-\ell') \ell 3^{-\frac{d}{2} (k-\ell)} \leq \frac{C}{\nu^2} (m-\ell') L' 3^{-\frac{2d}{d+4} (\ell'-\ell)} \,.
\end{align*}
Next we estimate the averaged quadratic tail. The scale identities in~\eqref{e.scale.selection} and the lower bound on~$h$ in~\eqref{e.h.restrictions} give~$h>2a$, after increasing the lower threshold, and hence~$m-2h<n<\ell<L'$. Thus the one-sided part of Lemma~\ref{l.mixing.minscale}, applied on each translate with local lower scale~$m-2h$ and local cube scale~$n$, gives
\begin{equation*}
	\s_{L',*}^{-1}(z+\cu_n)
	\leq
	\shom_{L',*}^{-1}(\cu_{m-2h})\Id
	+
	\O_{\Gamma_2}\bigl(C\nu^{-2}3^{-\frac14(n-(m-2h))}\Id\bigr)
	\,.
\end{equation*}
The second estimate of the same lemma controls the averaged stream tail~$\k_\ell-\k_{L'}$ in the metric~$\s_{L',*}^{-1}$, while~\eqref{e.nablaw.Lt} gives~$\E[\|\nabla w\|_{\underline{L}^4(\cu_m)}^4]^{1/2}\leq Ch|p|^2\leq C(L'-\ell)|p|^2$. Since~$n-(m-2h)=h-2\lceil K\log(\nu^{-1}L)\rceil$, we obtain
\begin{align}
	\lefteqn{
	\E \Biggl[ \biggl| \avsum_{z\in 3^n\Zd\cap \cu_k} (\k_\ell-\k_{L'})_{z+\cu_n}^t \s_{L',*}^{-1}(z+\cu_n) (\k_\ell-\k_{L'})_{z+\cu_n} \biggr|^2 \Biggr]^{\nf12}
	\E \Bigl[ \bigl\| \nabla w  \bigr\|_{\underline{L}^4(\cu_{m})}^4
	\Bigr]^{\nf12}
	} \qquad &
	\notag \\
	&\leq   C  (L'-\ell)^2 \shom_{L',*}^{-1}(\cu_{m-2h}) |p|^2  +
	C\nu^{-3}   (L'-\ell)^2 3^{-\frac14(h - 2\lceil K\log(\nu^{-1}L)\rceil)}   \notag \\
	&\leq C  (L'-\ell)^2 \shom_{L',*}^{-1}(\cu_{n})^2  +
	C\nu^{-3}   (L'-\ell)^2 3^{-\frac18 h} \,.
\end{align}
Here the last line uses
\begin{equation*}
	|p|^2\leq\shom_{L',*}^{-1}(\cu_n)\,,
\end{equation*}
and the pigeonhole comparability~\eqref{e.pigeon.scalar} to replace~$\shom_{L',*}^{-1}(\cu_{m-2h})$ by~$\shom_{L',*}^{-1}(\cu_n)$. We also use~\eqref{e.nablaw.Lt} again to see that
\begin{equation*}
	L' 3^{-(\ell-n)} \E \Bigl[ \bigl\| \nabla w  \bigr\|_{\underline{L}^4(\cu_{m})}^4 \Bigr]^{\nf12}
	\leq
	C \nu^{-1} (L'-\ell)  L' 3^{-(\ell-n)} \,.
\end{equation*}
Combining the above four displays with Young's inequality yields the desired estimate in~\eqref{e.bL.to.bhomell}; in this last step we only use the crude bound~$(L'-\ell)\leq L'$, the bound~$\nu\leq1$, and absorb the resulting polynomial factors into~$\nu^{-4}(L')^4$.
\smallskip

\emph{Step 4.} Combining~\eqref{e.RHS.term3.B},~\eqref{e.RHS.term3.A} and~\eqref{e.bL.to.bhomell} concludes the proof.
\end{proof}

We estimate the fourth and final term on the right side of~\eqref{e.ellsep.testing}.
\begin{lemma}
\label{l.RHS.term4}
There exists a constant~$C(d)<\infty$ such that
\begin{equation}
	\label{e.RHS.term4}
	\left|
	\E \biggl[
	p\cdot\fint_{\cu_m} (\k_{\ell'} - \k_{\ell}) \nabla w
	\biggr]
	\right|
	\leq
	C \shom_{L',*}^{-1}(\cu_n)
	\,.
\end{equation}
\end{lemma}
\begin{proof}
By~\eqref{e.kmn.Wminussp} and the fractional part of~\eqref{e.nablaw.Lt}, we find
\begin{align*}
	\left|
	\E \biggl[
	p\cdot\fint_{\cu_m} (\k_{\ell'} - \k_{\ell}) \nabla w \biggr]
	\right|
	&
	\leq
	\E \Bigl[ \| (\k_{\ell'} - \k_{\ell}) p \|_{\Hminusulh(\cu_m)}
	\| \nabla w \|_{\underline{H}^{\nf12}(\cu_m)}
	\Bigr]
	\notag \\ &
	\leq
	|p| \E \Bigl[ \| (\k_{\ell'} - \k_{\ell}) \|_{\Hminusulh(\cu_m)}^2 \Bigr]^{\nf12}
	\E \Bigl[
	\| \nabla w \|_{\underline{H}^{\nf12}(\cu_m)}^2
	\Bigr]^{\nf12}
	\notag \\ &
	\leq
	C|p| 3^{\frac12 \ell' } \cdot |p|  3^{-\frac12 \ell'}
	=
	C|p|^2
	=
	C\shom_{L',*}^{-1}(\cu_n)
	\,,
\end{align*}
where we used~\eqref{e.Sec3.p.q.def} to get the last equality. This is~\eqref{e.RHS.term4}.
\end{proof}

We are almost done with the proof of~Proposition~\ref{p.sstar.lower.bound} except for an estimate of the factor~$\shom_{\ell}(\cu_n)$ in~\eqref{e.RHS.term3}. To estimate this factor, we first localize, using Lemma~\ref{l.localization.A}, and replace the factor by~$\shom_{k}(\cu_n)$ for a parameter~$k \ll n$. The replacement is a coarse-grained matrix coming from a high-contrast \emph{local} equation at the scale~$3^k$. Since the local equation homogenizes, by~\cite{AK.HC}, the ratio~$\shom_{k}(\cu_n) \shom_{k,*}^{-1}(\cu_n)$ is bounded. This argument leads to the following estimate.

\begin{lemma}
\label{l.b.ell.homogenization}
There exists a constant~$C(d)<\infty$ such that, if~$n \geq C \log^2 (\nu^{-1}L)$, then we have that, for every~$\ell\in\N$ with~$n<\ell\leq L$,
\begin{equation}
	\label{e.bell.vs.starell}
	\shom_\ell(\cu_n)   \leq C \nu^{-4}  \bigl( \ell - n + \log^2(\nu^{-1}\ell) \bigr)^{\!4}\, \shom_{\ell,*}(\cu_{n})\,.
\end{equation}
\end{lemma}
\begin{proof}
Let~$C_0(d)<\infty$ be the constant in~\cite[Theorem~3.1]{AK.HC}. After increasing the constant in the hypothesis, if necessary, we may fix~$k<n$ such that
\begin{equation*}
	n-k \geq C_0 \log^2(\nu^{-1} k)
	\qand
	\ell-k \leq C \bigl( \ell-n+\log^2(\nu^{-1}\ell) \bigr)\,.
\end{equation*}
For instance, after increasing the constant in the hypothesis so that~$n>C_0\log^2(\nu^{-1}\ell)+1$, take~$k \coloneqq n-\lceil C_0\log^2(\nu^{-1}\ell)\rceil$, with~$C_0$ enlarged if necessary. Since~$k\leq \ell$, this choice gives~$n-k\geq C_0\log^2(\nu^{-1}k)$. The cutoff field~$\a_k$ is stationary, has range of dependence~$C3^k$, and satisfies the high-contrast ellipticity bound required in~\cite[Theorem~3.1]{AK.HC} with contrast bounded by a dimensional polynomial in~$\nu^{-1}k$; the choice of~$n-k$ is precisely the corresponding scale-separation condition. Set
\begin{equation*}
	\bar D
	\coloneqq
	\nu^{-1}\| \k_\ell-\k_{k} \|_{L^\infty(\cu_n)}
	+
	\nu^{-2}\| \k_\ell-\k_{k} \|_{L^\infty(\cu_n)}^2
	\,.
\end{equation*}
Applying Lemma~\ref{l.localization.A} with~$\a=\a_k$, perturbation field~$\k_\ell-\k_k$, and~$U=\cu_n$, we obtain
\begin{equation*}
	\bfA_\ell(\cu_n)
	\leq
	(1+\bar D)\bfA_{k}(\cu_n)\,.
\end{equation*}
Since~$\bar D$ is measurable with respect to~$\{ \mathbf{j}_r \}_{k<r\leq \ell}$ while~$\bfA_k(\cu_n)$ is measurable with respect to~$\{ \mathbf{j}_r \}_{r\leq k}$, the two random variables are independent. Taking expectation therefore gives
\begin{equation*}
	\bfAhom_\ell(\cu_n) \leq \bigl( 1 + \E[\bar D] \bigr)\bfAhom_k(\cu_n)
	\,.
\end{equation*}
By the block-diagonal form~\eqref{e.homs.defs.U}, which follows from~$\khom_m(U)=0$, the upper-left block of~$\bfAhom_m(U)$ is~$\shom_m(U)$. Taking the upper-left and lower-right blocks, respectively, we infer that
\begin{equation*}
	\shom_\ell(\cu_n)
	\leq
	(1+\E[\bar D])\shom_k(\cu_n)
	\qand
	\shom_{\ell,*}^{-1}(\cu_n)
	\leq
	(1+\E[\bar D])\shom_{k,*}^{-1}(\cu_n)
	\,.
\end{equation*}
By~\cite[Theorem~3.1]{AK.HC} and our choice of~$k$,
\begin{equation*}
	\shom_k(\cu_n) \leq 2 \shom_{k,*}(\cu_n)\,.
\end{equation*}
Combining the previous two displays, and inverting the scalar lower-right block inequality, yields
\begin{equation*}
	\shom_\ell(\cu_n)
	\leq
	2(1+\E[\bar D])^2\shom_{\ell,*}(\cu_n)\,.
\end{equation*}
By~\eqref{e.kmn.Linfty}, using~$\cu_n\subseteq\cu_\ell$ to dominate the~$L^\infty$ norm, we have
\begin{equation*}
	\| \k_\ell-\k_k \|_{L^\infty(\cu_n)}
	\leq
	\O_{\Gamma_2}\bigl( C(\ell-k)\bigr)\,.
\end{equation*}
Therefore
\begin{equation*}
	1+\E[\bar D]
	\leq
	C\nu^{-1}(\ell-k)+C\nu^{-2}(\ell-k)^2
	\leq
	C\nu^{-2}(\ell-k)^2
	\leq
	C\nu^{-2}\bigl( \ell-n+\log^2(\nu^{-1}\ell) \bigr)^2\,,
\end{equation*}
Combining this with the preceding bound on~$\shom_\ell(\cu_n)$ gives~\eqref{e.bell.vs.starell}.
\end{proof}

The four terms in~\eqref{e.ellsep.testing} are now controlled, including the remaining ellipticity factor in the third term. We return to the proof of Proposition~\ref{p.sstar.lower.bound}.

\begin{proof}[Proof of Proposition~\ref{p.sstar.lower.bound}]
Choose the parameter~$K$ sufficiently large, depending only on~$d$, for the scale-separation estimates below. With the final choice of~$h$ made below, and after increasing the constant in~\eqref{e.L.vs.nu} if necessary, the selected scales satisfy
\begin{equation*}
	\begin{gathered}
		h \geq 100(\ell -n)\,,
		\qquad
		\cstar h >
		C\bigl(1+\nondegconst+K \log (\nu^{-1}L)\bigr)\,,
		\qquad
		C\delta^{\nf12}\leq \frac14\cstar\,,
		\\
		C (\nu^{-1} L)^{9}  \bigl( 3^{-\frac 12 (\ell'-n)}  + 3^{-\frac18(\ell-n)} + 3^{-\frac 14(\ell'-\ell)} +  3^{-\frac 18 h} \bigr) \,   \leq \nu^2 L^{-1000}\,.
		\\
		\eta_L\leq L^{-1000}\,.
	\end{gathered}
\end{equation*}
We also use the crude bound, following from~\eqref{e.CG.bounds.1} and~\eqref{e.km.Ltwo.size},
\begin{equation*}
	\shom_{L',*}(\cu_n)\leq C\nu^{-2}L
	\qquad \mbox{and hence} \qquad
	\shom_{L',*}^{-1}(\cu_n)\geq c\nu^2L^{-1}\,.
\end{equation*}
Therefore the absolute errors bounded by~$\nu^2 L^{-1000}$ in the preceding scale condition are negligible compared with~$\shom_{L',*}^{-1}(\cu_n)$: by the crude bound~$\shom_{L',*}^{-1}(\cu_n)\geq c\nu^2L^{-1}$ above, their ratio is at most~$c^{-1}L^{-999}$, independently of~$\nu$, after increasing the threshold in~\eqref{e.L.vs.nu}. Since~$L'-\ell=h+3a$ and~$h\geq100(\ell-n)=100a$, the contribution~$(L'-\ell)\shom_{L',*}^{-1}(\cu_n)$ from~\eqref{e.RHS.term3} is absorbed into~$h\shom_{L',*}^{-1}(\cu_n)$. Combining Lemma~\ref{l.LHS.term1} and~\eqref{e.ellsep.testing} with~\eqref{e.RHS.term1},~\eqref{e.RHS.term2},~\eqref{e.RHS.term3} and~\eqref{e.RHS.term4}, taking note of~\eqref{e.Sec3.p.q.def} and~\eqref{e.homs.defs.U}, we see that
\begin{align*}
	\cstar h \shom_{L',*}^{-1}(\cu_n)
	&
	\leq
	C (\delta + L^{-1000})^{\nf12} \bigl( \shom_{\ell}(\cu_n) +  h  \shom_{L',*}^{-1}(\cu_n)  \bigr)
	+
	C\bigl(1+\nondegconst+K \log (\nu^{-1}L)\bigr)  \shom_{L',*}^{-1}(\cu_n)
	\,.
\end{align*}
The condition~$C\delta^{\nf12}\leq \frac14\cstar$ absorbs the term proportional to~$h\shom_{L',*}^{-1}(\cu_n)$, while the condition~$\cstar h > C(1+\nondegconst+K\log(\nu^{-1}L))$ absorbs the lower-order term. We deduce that
\begin{equation*}
	\cstar h \shom_{L',*}^{-1}(\cu_n)  \leq
	C\shom_{\ell}(\cu_n)
	\,.
\end{equation*}
Combining this with~\eqref{e.bell.vs.starell}, we obtain
\begin{equation}
	\label{e.crude.enhance.h.bnd}
	\shom_{L',*}^{-1}(\cu_n)  \leq
	\frac{C \log^{8} (\nu^{-1}L) }{\cstar h \nu^4}
	\shom_{\ell,*}(\cu_n)
	\,.
\end{equation}
To optimize this inequality, we take~$h$ as large as we are permitted. By~\eqref{e.h.restrictions}, this means choosing~$h  \coloneqq  \lfloor c\delta L / \log(C\nu^{-1}L) \rfloor$ for a sufficiently small constant~$c(d)>0$. Since~$\delta=c_0\cstar^2$, this gives
\begin{equation}
	\label{e.h.optimized.size}
	h \geq \frac{c\cstar^2 L}{\log(C\nu^{-1}L)}
	\,,
\end{equation}
for all sufficiently large~$L$. (The lower bound~$h \geq 10\lceil K\log^2(\nu^{-1}L)\rceil$ is satisfied under~\eqref{e.L.vs.nu}.) By subadditivity of~$\shom_{L',*}^{-1}$ and~\eqref{e.pigeon.scalar}, the quantities~$\shom_{L',*}(\cu_k)$ are comparable for~$k \in [m-2h,m]$. Applying the relative comparison~\eqref{e.localization.s.star} to the cutoff pairs~$(\ell,L')$ and~$(L',L)$ on~$\cu_n$ and also to~$(L',L)$ on~$\cu_m$, then taking expectations in the Loewner inequalities as above, gives
\begin{equation*}
	\bigl|\shom_{\ell,*}(\cu_n)\shom_{L',*}^{-1}(\cu_n)-1\bigr|
	\leq L^{-1000}
	\qand
	\bigl|\shom_{L',*}(\cu_n)\shom_{L,*}^{-1}(\cu_n)-1\bigr|
	\leq L^{-1000}
\end{equation*}
after increasing~$K(d)$ if necessary. Therefore
\begin{equation*}
	\frac12 \shom_{L',*}(\cu_n) \leq \shom_{\ell,*}(\cu_n)\leq 2\shom_{L',*}(\cu_n)
	\qand
	\frac14 \shom_{L',*}(\cu_n)\leq \shom_{L,*}(\cu_m)\leq 4\shom_{L',*}(\cu_n)
	\,,
\end{equation*}
for all sufficiently large~$L$. Combining~\eqref{e.crude.enhance.h.bnd} with~$\shom_{\ell,*}(\cu_n)\leq2\shom_{L',*}(\cu_n)$ yields
\begin{equation*}
	\bigl(\shom_{L',*}(\cu_n)\bigr)^2
	\geq
	\frac{c\cstar h\nu^4}{\log^8(\nu^{-1}L)}
	\geq
	c\cstar^3\nu^4 L\log^{-9}(\nu^{-1}L)\,,
\end{equation*}
where the last step uses~\eqref{e.h.optimized.size}. Combining this with~$\shom_{L,*}(\cu_m)\geq \frac14\shom_{L',*}(\cu_n)$ gives
\begin{equation*}
	\shom_{L,*}(\cu_m)
	\geq
	c\cstar^{\nf32}\nu^2 L^{\nf12}\log^{-\nicefrac92}(\nu^{-1}L)
	\geq
	c\cstar^{\nf32}\nu^2 m^{\nf12}\log^{-\nicefrac92}(\nu^{-1}m)
	\,.
\end{equation*}
This proves~\eqref{e.sstar.lower.bound} for the pigeonhole scale~$m\in[L/4,L/2]$.

\smallskip

It remains only to remove this harmless choice of scale. Let~$\bar m\in[L/2,L]$ be prescribed. Since~$\shom_{L,*}^{-1}$ is nonincreasing in the spatial scale and~$m\leq \bar m$, we have
\begin{equation*}
	\shom_{L,*}^{-1}(\cu_{\bar m})
	\leq
	\shom_{L,*}^{-1}(\cu_m)\,.
\end{equation*}
The lower bound at the pigeonhole scale therefore gives the same lower bound at~$\bar m$. Since~$m,\bar m\in[L/4,L]$, the factors~$m^{\nf12}$,~$\bar m^{\nf12}$ and~$L^{\nf12}$ are comparable, as are the corresponding logarithmic factors. This gives~\eqref{e.sstar.lower.bound} for the prescribed~$\bar m$, and hence for the~$m$ in the proposition.

\smallskip

We finally prove the quenched estimate~\eqref{e.sstar.lower.bound.quenched}. Set~$n_0\coloneqq\lfloor m/2\rfloor$ and~$R\coloneqq 2n_0$. After increasing the lower threshold in~\eqref{e.L.vs.nu}, we may apply the deterministic estimate just proved with cutoff~$R$ and spatial scale~$n_0$, obtaining
\begin{equation*}
	\shom_{R,*}^{-1}(\cu_{n_0})
	\leq
	C \nu^{-2} \cstar^{-\nf32} n_0^{-\nf12}
	\log^{\nf {9}2} ( \nu^{-1} n_0)\,.
\end{equation*}
Since~$R-n_0=n_0$, the relative localization comparison~\eqref{e.localization.s.star}, after taking expectations in the associated Loewner inequalities and converting the additive inverse-block comparison to the scalar ratio as in~\eqref{e.localization.s.star.annealed.applied}, gives
\begin{equation*}
	\bigl|\shom_{R,*}(\cu_{n_0})\shom_{L,*}^{-1}(\cu_{n_0})-1\bigr|
	\leq
	C\nu^{-5}L3^{-n_0}\,.
\end{equation*}
After increasing the lower threshold in~\eqref{e.L.vs.nu}, this factor is absorbed into the constant, so the deterministic bound transfers from cutoff~$R$ to cutoff~$L$:
\begin{equation*}
	\shom_{L,*}^{-1}(\cu_{n_0})
	\leq
	C \nu^{-2} \cstar^{-\nf32} m^{-\nf12}
	\log^{\nf {9}2} ( \nu^{-1} m)\,.
\end{equation*}
Applying~\eqref{e.sstarL.quenched.lb} with~$h=n_0$ and~$n=m$ yields
\begin{equation*}
	\s_{L,*}^{-1}(\cu_m)
	\leq
	\shom_{L,*}^{-1}(\cu_{n_0})
	+
	\O_{\Gamma_2}\bigl(C\nu^{-2}3^{-\frac14(m-n_0)}\Id\bigr)
	\,.
\end{equation*}
As~$m-n_0\geq m/2$, this proves~\eqref{e.sstar.lower.bound.quenched} and completes the proof.
\end{proof}

\clearpage
\section{Homogenization on scales below the infrared cutoff}
\label{s.homog.below.cutoff}

In this section we prove homogenization estimates for the infrared cutoff field~$\a_L$ defined in~\eqref{e.infrared.cutoff.def} on a range of length scales which include scales smaller than~$3^L$. The goal, as described in the introduction, is to show that the operator~$\nabla \cdot \a_L\nabla$ is close to the operator~$\nabla \cdot \shom_L \nabla$ on length scales of order~$3^m$ with~$m\in [ L - C L^{\alpha},L]$ for~$\alpha < \nf12$. In the intermediate estimates below we sometimes normalize instead by~$\shom_m$ when the spatial averaging scale is~$3^m$; the comparison back to the effective diffusivity~$\shom_L$ is supplied by Lemma~\ref{l.shomm.vs.shomell}. Throughout the section, the multiscale homogenization error~$\mathcal{E}_{s,q}$, the coarse-grained ellipticity constants~$\Lambda_{s,q}$ and~$\lambda_{s,q}$, and the coarse-grained Poincar\'e and Caccioppoli inequalities are those collected in Subsection~\ref{ss.cg.ellipticity}.

\smallskip

In what follows, we will need a lower bound on the length scale with strength of the bound depending on other parameters. For this purpose, we introduce, for every~$\alpha \in [0,1)$ and~$M \in [1,\infty)$, a constant~$L_0$ defined by
\begin{equation}
	\label{e.Lnaught.def}
	L_0(M,\alpha,\cstar,\nu) \coloneqq  \biggl( \frac{C (M+1+\nondegconst)\cstar^{-3}}{ (1-\alpha)^{12} \nu^4 } \log^{12} \Bigl( 2+\frac{(M+1+\nondegconst)\cstar^{-3}}{\nu (1-\alpha)}\Bigr) \biggr)^{\frac 1{1-\alpha}}
	\,,
\end{equation}
where the dependence on~$\nondegconst$ is suppressed in the argument list throughout, the additive constant in the previous threshold has been absorbed into~$\cstar^{-3}$ using~\eqref{e.cstar.bound}, and the universal constant~$C$ is chosen to be so large that
\begin{equation*}
	L \geq L_0 \implies
	L^{1-\alpha} \geq M \cstar^{-3} \nu^{-4} \log^{12} L
	\quad \mbox{and} \quad
	\frac{L^{1-\alpha}}{\log(3+L)}
	\geq
	(M+1+\nondegconst)\cstar^{-3}
	\,.
\end{equation*}
We have defined~$L_0$ in such a way that Proposition~\ref{p.sstar.lower.bound} yields
\begin{equation}
	\label{e.L.vs.Lnaught}
	L \geq L_0(M,\alpha,\cstar,\nu)
	\implies
	\inf_{h \in \N \cap [\frac12 L,L]}\shom_{L,*}^{\,2}(\cu_h)
	\geq
	c(d) M L^\alpha \log^{3} L
	\,.
\end{equation}
Indeed, Proposition~\ref{p.sstar.lower.bound} gives~$\shom_{L,*}^{\,2}(\cu_h) \geq c(d)\cstar^3\nu^4 L\log^{-9}L$ for~$h\in[\frac12 L,L]$, and the defining threshold in~\eqref{e.Lnaught.def} makes the latter at least~$c(d)ML^\alpha\log^3L$:
\begin{equation*}
	L^{1-\alpha}\geq C M\cstar^{-3}\nu^{-4}\log^{12}L
	\implies
	c\cstar^3\nu^4 L\log^{-9}L
	\geq
	cM L^\alpha\log^3L\,.
\end{equation*}
The following proposition is the main result of this section. It states that, on sufficiently large scales and for~$L$ sufficiently large, the infrared cutoff~$\a_L$ has small multiscale homogenization error, in the sense of~\eqref{e.mathcal.E.def.new}, in a family of cubes with sizes which may be smaller than~$L$. Recall that~$\shom_m$ is the infinite-volume limit defined in~\eqref{e.homs.defs}.

\begin{proposition}
\label{p.minimal.scales}
There exists a constant~$C(d)<\infty$ such that the following holds. For every~$\expon \in (0,\nf12)$,~$\delta,s\in(0,1]$ and~$M \in [1,\infty)$, set
\begin{equation*}
	\hat L_2 \coloneqq
	L_0\bigl(C \expon^{-1}\delta^{-2}s^{-4}M^2,1-\expon,\cstar,\nu\bigr)
	\vee
	L_0\bigl(C M s^{-2},\nf12+\expon,\cstar,\nu\bigr)\,.
\end{equation*}
Then there exists a minimal scale~$\X$ satisfying
\begin{equation}
	\label{e.minscale.bound}
	\log \X = \O_{\Gamma_{2\expon}}(\hat L_2)
	\,.
\end{equation}
With
\begin{equation*}
	h  \coloneqq  \lceil CMs^{-1}\log m\rceil\,,
\end{equation*}
for every~$m,n \in\N$ satisfying
\begin{equation}
	m \geq \hat L_2\,, \quad
	3^m\geq \X
	\qand
	m- h\leq n \leq m
	\,,
\end{equation}
and every coefficient field
\begin{equation*}
	\a \in
	\bigl\{\nu \Id + \k_L-(\k_L)_{\cu_m}\bigr\}_{L\geq m}
	\cup
	\bigl\{\nu \Id  + \k -(\k)_{\cu_m}\bigr\}\,,
\end{equation*}
where, on~$\cu_m$,~$\k-(\k)_{\cu_m}$ denotes the local uniform limit of~$\k_L-(\k_L)_{\cu_m}$, we have the estimate
\begin{equation}
	\label{e.minscale.bounds.E}
	\max_{z\in 3^{n-3}\Zd\,:\, z+\cu_n\subseteq\cu_m}
	\mathcal{E}_{s,2}
	\bigl(
	z+\cu_n\,;\,
	\a\,,
	\shom_m
	\bigr)
	\leq
	\delta \shom_m^{-1}m^{\expon}\log m\,.
\end{equation}
\end{proposition}

The proof is given below after the required mixing estimates. The exponent~$2\expon$ in~\eqref{e.minscale.bound} comes from the tail estimate~\eqref{e.new.mixing.minscale.one.scale}. We continue in the next subsection with a presentation of the needed mixing estimates for the coarse-grained matrices.

\subsection{Mixing properties of the coarse-grained coefficients}
In this subsection we prove a mixing property for the infrared cutoff coarse-grained matrices with a strength depending on the separation between the cutoff scale and the observation scale. The basic point is that, after the nearly constant antisymmetric tail has been gauged out, the remaining block contributions can be averaged as almost independent objects. These estimates on~$\bfA$ will be converted into bounds on~$\mathcal{E}_{s,2}$ from Subsection~\ref{ss.cg.ellipticity}, as required by Proposition~\ref{p.minimal.scales}.

\begin{proposition}[Mixing below the infrared cutoff]
\label{p.mixing.P.three.prime}
There exists a constant~$C(d)<\infty$ such that, for every~$m,n\in\N$ and~$L\in\N$ with~$L\geq1$ satisfying
\begin{equation}
	\label{e.mixing.gaps}
	\begin{aligned}
		&L - n \leq C^{-1} \shom_{L,*}^{\,2}(\cu_n)
		\qand
		m < 2n\,,
		\\
		&n \leq m - \lceil C \log\bigl(\nu^{-1}(L\vee n)\bigr) \rceil
		\qand
		n \geq \lceil C\log(\nu^{-1}L)\rceil
	\end{aligned}
	\,,
\end{equation}
we have
\begin{align}
	\label{e.main.mixing.estimate}
	\lefteqn{
	\biggl|
	\bfAhom_L^{-\nf12}(\cu_n) \avsum_{z \in 3^n \Zd \cap \cu_m}
	\bigl(\bfA_L(z+\cu_n) - \bfAhom_L(\cu_n) \bigr)
	\bfAhom_L^{-\nf12}(\cu_n)
	\biggr| \indc_{\{ m \leq L + C \log(\nu^{-1} L) \}}
	} \qquad &
	\notag \\ &
	\leq
	\O_{\Gamma_2} \bigl( C (L-n)_+^{\nf12} \shom_{L,*}^{\,-1} (\cu_n) \bigr)
	+
	\O_{\Gamma_1} \bigl( C (L-n)_+ \shom_{L,*}^{\,-2} (\cu_n) \bigr)
	+
	\O_{\Gamma_{\nf13}} (m^{-3000} )
	\,,
\end{align}
and
\begin{equation*}
	\biggl|
	\bfAhom_L^{-\nf12}(\cu_n) \avsum_{z \in 3^n \Zd \cap \cu_m}
	\bigl(\bfA_L(z+\cu_n) - \bfAhom_L(\cu_n) \bigr)
	\bfAhom_L^{-\nf12}(\cu_n)
	\biggr| \indc_{\{ m > L + C \log(\nu^{-1} L) \}}
	\leq
	\O_{\Gamma_{1}} (m^{-3000} )
	\,.
\end{equation*}
\end{proposition}

We will also use the following crude comparison, valid for every~$n\in\N$ and~$L\in\N$ with~$L\geq1$:
\begin{equation}
	\label{e.Enaught.vs.Ahom.L.crude}
	| \bfE_L \bfAhom_L^{-1}(\cu_n) |
	\leq
	4(\nu + C \nu^{-1} L)\shom_L^{-1}(\cu_n)
	+
	4 \nu^{-1}\shom_{L,*}(\cu_n)
	\leq C \nu^{-2} L \,.
\end{equation}
This is just the block form of~\eqref{e.Enaught.vs.A.and.Ahom}, together with~$\shom_L(\cu_n)\geq\nu$ and the bound~$\shom_{L,*}(\cu_n)\leq\shom_L(\cu_n)\leq C\nu^{-1}L$, which follows from Jensen's inequality, the ordering~$\s_*\leq\s$ in~\eqref{e.CG.bounds.1}, and~$\bfAhom_L(\cu_n)\leq\bfE_L$ from Lemma~\ref{l.bfAm.ellip}.

\begin{proof}[Proof of Proposition~\ref{p.mixing.P.three.prime}]
Let~$L,m,n\in\N$ be such that
\begin{equation}
	\label{e.sqrts.by.shomL}
	L - n \leq K^{-1} \shom_{L,*}^{\,2}(\cu_n)
	\,, \quad
	m < 2n
	\,, \quad
	n \leq m - \lceil K\log\bigl(\nu^{-1}(L\vee n)\bigr) \rceil
	\,, \quad
	n \geq \lceil K\log(\nu^{-1} L)\rceil
	\,,
\end{equation}
where~$K\geq1$ is a large constant to be selected below, depending only on~$d$. After fixing~$K$, the constant~$C$ in the statement, and in particular the cutoff constant denoted later by~$C_{\mathrm{mix}}$, is chosen large enough compared with~$K$. We first consider the case that~$m \leq L +  \frac K2 \log (\nu^{-1} L)$. The argument in the case~$m > L + \frac K2 \log (\nu^{-1} L )$ is much simpler, as we explain at the end of the proof. In the first case, we may suppose without loss of generality that
\begin{equation}
	\label{e.mn.gap.bdd}
	m-n \in  [K \log(\nu^{-1} L), 10 K  \log(\nu^{-1} L)] \,.
\end{equation}
Indeed, if~$n < m - 10 K \log(\nu^{-1} L)$, we may split the sum into smaller subcubes and then apply the result in each of those subcubes. This reduction is justified by induction on~$m-n$: choosing the intermediate scale~$\tilde m\coloneqq n+\lceil 10K\log(\nu^{-1}L)\rceil$, applying the already proved smaller-gap statement on the subcubes of scale~$\tilde m$, and averaging gives the original estimate. The case assumption and the lower bound in~\eqref{e.mn.gap.bdd} imply that~$L-n\geq \frac K2\log(\nu^{-1}L)$. In particular,~$n\leq L$ and~$L\vee n=L$ throughout this case, so the following auxiliary scale exists.

\smallskip

Fix an auxiliary scale~$\ell\in\N$ satisfying
\begin{equation}
	\label{e.good.gaps}
	\min\bigl \{ L-\ell, m-\ell, \ell-n, n \bigr\} \geq \frac{K}{10} \log ( \nu^{-1}L)  \,.
\end{equation}
This is the only property of~$\ell$ used below. We begin the proof with the following decomposition. Set~$\h_z \coloneqq (\k_L-\k_\ell)_{z+\cu_n}$ for~$z \in 3^n\Zd \cap \cu_m$ and then
\begin{align}
	\label{e.decompose.AL.minus.Aell}
	\avsum_{z \in 3^n\Zd \cap \cu_m}
	\bigl(\bfA_L(z+\cu_n)- \bfAhom_\ell(\cu_n)\bigr)
	&
	=
	\avsum_{z \in 3^n\Zd \cap \cu_m}
	\bigl(\bfA_L(z+\cu_n)-\G_{-\h_z}^t \bfAhom_\ell(\cu_n)\G_{-\h_z}\bigr)
	\notag \\ & \qquad
	+
	\avsum_{z \in 3^n\Zd \cap \cu_m}
	\bigl(\G_{-\h_z}^t \bfAhom_\ell(\cu_n)\G_{-\h_z}-\bfAhom_\ell(\cu_n)\bigr)
	\,.
\end{align}
We estimate the two terms on the right side of~\eqref{e.decompose.AL.minus.Aell}. For the first one, Lemma~\ref{l.localization.average} gives, after testing with~$P=\bfAhom_\ell^{-\nf 12}(\cu_n)Q$, taking the supremum over a finite net of~$|Q|=1$, and using~\eqref{e.Enaught.vs.Ahom.L.crude} together with~\eqref{e.good.gaps},
\begin{equation}
	\biggl|
	\bfAhom_\ell^{-\nf 12}(\cu_n)
	\avsum_{z \in 3^n\Zd \cap \cu_m}
	\bigl(\bfA_L(z+\cu_n)-\G_{-\h_z}^t \bfAhom_\ell(\cu_n)\G_{-\h_z}\bigr)
	\bfAhom_\ell^{-\nf 12}(\cu_n)
	\biggr|
	\leq
	\O_{\Gamma_{\nf13}}(m^{-5000})
	\,,
\end{equation}
provided~$K(d)$ is chosen sufficiently large. For the second term, since~$\bfAhom_\ell(\cu_n)$ is block diagonal by~\eqref{e.homs.defs.U}, we have
\begin{equation*}
	\G_{-\h_z}^t \bfAhom_\ell(\cu_n)\G_{-\h_z}
	-
	\bfAhom_\ell(\cu_n)
	=
	\begin{pmatrix}
		\shom_{\ell,*}^{-1}(\cu_n)\h_z^t\h_z
		&
		-\shom_{\ell,*}^{-1}(\cu_n)\h_z^t
		\\
		-\shom_{\ell,*}^{-1}(\cu_n)\h_z
		&
		0
	\end{pmatrix}
	\,.
\end{equation*}
By the relative comparison~\eqref{e.localization.s.star}, the quadratic forms~$\s_{\ell,*}^{-1}(\cu_n)$ and~$\s_{L,*}^{-1}(\cu_n)$ differ by a factor of~$1+\O_{\Gamma_1}(m^{-7000})$, after increasing~$K(d)$ and using~\eqref{e.good.gaps}; indeed~$\ell-n\geq (K/10)\log(\nu^{-1}L)$ absorbs the polynomial powers of~$\nu^{-1}$ into~$m^{-7000}$. Taking expectations in the Loewner form, integrating the~$\Gamma_1$ tail, and using~\eqref{e.CG.bounds.1}, we get
\begin{equation*}
	\shom_{\ell,*}^{-1}(\cu_n)
	\leq
	\shom_{L,*}^{-1}(\cu_n)+ m^{-6000}\,.
\end{equation*}
Therefore, in the~$\bfAhom_\ell(\cu_n)$-normalized norm,
\begin{align}
	\lefteqn{
	\biggl|
	\bfAhom_\ell^{-\nf 12}(\cu_n)
	\avsum_{z \in 3^n\Zd \cap \cu_m}
	\bigl(\G_{-\h_z}^t \bfAhom_\ell(\cu_n)\G_{-\h_z}-\bfAhom_\ell(\cu_n)\bigr)
	\bfAhom_\ell^{-\nf 12}(\cu_n)
	\biggr|
	} \qquad &
	\notag \\ &
	\leq
	C \shom_{L,*}^{-1}(\cu_n)
	\avsum_{z \in 3^n\Zd \cap \cu_m} |\h_z|
	+
	C \shom_{L,*}^{-2}(\cu_n)
	\avsum_{z \in 3^n\Zd \cap \cu_m} |\h_z|^2
	+
	\O_{\Gamma_{\nf13}}(m^{-5000})
	\notag \\ &
	\leq
	C \shom_{L,*}^{-1}(\cu_n) \| \k_L-\k_\ell \|_{\underline{L}^2(\cu_m)}
	+
	C \shom_{L,*}^{-2}(\cu_n) \| \k_L-\k_\ell \|_{\underline{L}^2(\cu_m)}^2
	+
	\O_{\Gamma_{\nf13}}(m^{-5000})\,,
\end{align}
where the last line follows from Cauchy--Schwarz, Jensen's inequality for the local averages defining~$\h_z$, and the definition of~$\h_z$. Combining the previous estimates with the case~$p=2$ of~\eqref{e.kmn.bounds}, we obtain
\begin{align*}
	\lefteqn{
	\biggl|
	\bfAhom_\ell^{-\nf 12}(\cu_n)
	\avsum_{z \in 3^n\Zd \cap \cu_m}
	\bigl(\bfA_L(z+\cu_n)-\bfAhom_\ell(\cu_n)\bigr)
	\bfAhom_\ell^{-\nf 12}(\cu_n)
	\biggr|
	} \qquad &
	\notag \\ &
	\leq
	\O_{\Gamma_2}\bigl( C (L-\ell)^{\nf12}\shom_{L,*}^{-1}(\cu_n)\bigr)
	+
	\O_{\Gamma_1}\bigl( C (L-\ell)\shom_{L,*}^{-2}(\cu_n)\bigr)
	+
	\O_{\Gamma_{\nf13}}(m^{-4500})
	\,.
\end{align*}
Taking expectations in the preceding display and using stationarity gives
\begin{equation*}
	\bigl|
	\bfAhom_\ell^{-\nf 12}(\cu_n)
	\bfAhom_L(\cu_n)
	\bfAhom_\ell^{-\nf 12}(\cu_n) - \Itwod
	\bigr|
	\leq
	C (L-\ell)^{\nf12}\shom_{L,*}^{-1}(\cu_n)
	+
	C (L-\ell)\shom_{L,*}^{-2}(\cu_n)
	+Cm^{-4500}
	\,.
\end{equation*}
Set
\begin{equation*}
	B\coloneqq
	\bfAhom_\ell^{-\nf12}(\cu_n)
	\bfAhom_L(\cu_n)
	\bfAhom_\ell^{-\nf12}(\cu_n)\,,
\end{equation*}
so the previous display and~\eqref{e.sqrts.by.shomL}, after increasing~$K(d)$ if necessary, give~$|B-\Itwod|\leq \frac12$ and therefore~$|B^{-\nf12}|\leq C$. Combining the previous estimates first gives
\begin{align*}
	\lefteqn{
	\biggl|
	\bfAhom_\ell^{-\nf12}(\cu_n)
	\avsum_{z \in 3^n\Zd \cap \cu_m}
	\bigl(\bfA_L(z+\cu_n)-\bfAhom_L(\cu_n)\bigr)
	\bfAhom_\ell^{-\nf12}(\cu_n)
	\biggr|
	} \qquad &
	\notag \\ &
	\leq
	\O_{\Gamma_2}\bigl( C (L-\ell)^{\nf12}\shom_{L,*}^{-1}(\cu_n)\bigr)
	+
	\O_{\Gamma_1}\bigl( C (L-\ell)\shom_{L,*}^{-2}(\cu_n)\bigr)
	+
	\O_{\Gamma_{\nf13}}(m^{-4500})
	\,.
\end{align*}
Since the homogenized block matrices are diagonal with scalar blocks, they commute, and
\begin{equation*}
	\bfAhom_L^{-\nf12}(\cu_n)
	=
	B^{-\nf12}
	\bfAhom_\ell^{-\nf12}(\cu_n)\,,
\end{equation*}
we obtain
\begin{align}
	\lefteqn{
	\biggl|
	\bfAhom_L^{-\nf12}(\cu_n)
	\avsum_{z \in 3^n\Zd \cap \cu_m}
	\bigl(\bfA_L(z+\cu_n)-\bfAhom_L(\cu_n)\bigr)
	\bfAhom_L^{-\nf12}(\cu_n)
	\biggr|
	} \qquad &
	\notag \\ &
	\leq
	\O_{\Gamma_2}\bigl( C (L-\ell)^{\nf12}\shom_{L,*}^{-1}(\cu_n)\bigr)
	+
	\O_{\Gamma_1}\bigl( C (L-\ell)\shom_{L,*}^{-2}(\cu_n)\bigr)
	+
	\O_{\Gamma_{\nf13}}(m^{-4500})
	\,.
\end{align}
Since~$L-\ell\leq L-n$, the previous display yields~\eqref{e.main.mixing.estimate} in the case~$m \leq L+\frac K2\log(\nu^{-1}L)$.

\smallskip

We conclude with the proof in the (easier) case when the scale separation is large, specifically, when~$m > L + \frac K2 \log(\nu^{-1} L)$. Set~$q\coloneqq L\vee n$. By the variational definition, each matrix~$\bfA_L(z+\cu_n)$ depends only on the cutoff environment in~$z+\cu_n$. Assumption~\ref{a.j.frd} therefore allows us to partition the index set into at most~$C3^{d(q-n)}$ subfamilies, within each of which the matrices are independent and which contain at most~$C3^{d(m-q)}$ elements. Proposition~\ref{p.concentration} with~$\sigma=1$,~\eqref{e.Enaught.vs.A.and.Ahom}, and~\eqref{e.Enaught.vs.Ahom.L.crude} give
\begin{align*}
	\lefteqn{
	\biggl|
	\bfAhom_L^{-\nf12}(\cu_n)
	\avsum_{z \in  3^n \Zd \cap \cu_m}
	\bigl(\bfA_L(z+\cu_n) - \bfAhom_L(\cu_n) \bigr)
	\bfAhom_L^{-\nf12}(\cu_n)
	\biggr|
	} \qquad &
	\notag \\ &\leq
	\O_{\Gamma_1}\Bigl( C \nu^{-2} L 3^{-\frac d2(m-q)} \Bigr)
	\leq \O_{\Gamma_1} ( m^{-3000} )
	\,.
\end{align*}
Indeed, if~$n\leq L$, then the case assumption gives~$m-L>\frac K2\log(\nu^{-1}L)$ and~$m<2n\leq2L$. If~$n>L$, then~\eqref{e.sqrts.by.shomL} gives~$m-n\geq K\log(\nu^{-1}n)$ and~$m<2n$. Thus, in either case, the last inequality holds provided~$K(d)$ is sufficiently large. This completes the proof.
\end{proof}

\subsection{The renormalized ellipticity bound}

In order to homogenize the infrared cutoffs, we require the following rather crude ellipticity-type bound. It is the matrix form of the coarse-grained ellipticity control from Subsection~\ref{ss.cg.ellipticity} needed to verify the high-contrast assumptions. Here ``renormalized'' means that the ellipticity is measured after normalization by the inverse square root of the annealed block matrix, as in the left side of~\eqref{e.ellipticity.Ptwoprime}.

\begin{proposition}
\label{p.ellipticity.Ptwoprime}
There exists~$C(d)<\infty$ such that, for every~$\gamma \in (0,1)$ and~$m,L \in \N$ with~$L\geq1$ and~$m \geq \frac14 L$,  we have the estimate
\begin{equation}
	\label{e.ellipticity.Ptwoprime}
	\sup_{k \in \Z \cap (-\infty,m]}
	3^{\gamma(k-m)}  \max_{z\in 3^{k} \Zd  \cap \cu_m} \bigl|  \bfAhom_L^{-\nf 12}(\cu_m)  \bfA_L(z +\! \cu_k) \bfAhom_L^{-\nf 12}(\cu_m) \bigr|  \leq \O_{\Gamma_1} (C \gamma^{-1} \nu^{-2} L)
	\,.
\end{equation}

\end{proposition}
\begin{proof}
With~$\bfE_L$ defined in~\eqref{e.Enaught.mixing}, we use the crude comparison~\eqref{e.Enaught.vs.Ahom.L.crude}. Let~$K(d)$ be a large constant to be selected just below. Since~$\bfE_L$ and~$\bfAhom_L(\cu_m)$ are block diagonal with scalar blocks,~\eqref{e.Enaught.vs.Ahom.L.crude} gives
\begin{equation*}
	\bigl|\bfAhom_L^{-\nf12}(\cu_m)\bfE_L^{\nf12}\bigr|^2
	\leq
	\bigl|\bfAhom_L^{-1}(\cu_m)\bfE_L\bigr|
	\leq
	C\nu^{-2}L\,.
\end{equation*}
Thus
\begin{equation*}
	\bigl|
	\bfAhom_L^{-\nf12}(\cu_m)\bfA_L(z+\cu_k)\bfAhom_L^{-\nf12}(\cu_m)
	\bigr|
	\leq
	C\nu^{-2}L
	\bigl|\bfE_L^{-\nf12}\bfA_L(z+\cu_k)\bfE_L^{-\nf12}\bigr|\,.
\end{equation*}
Using a union bound and~\eqref{e.Enaught.vs.A.and.Ahom}, we have that, for every~$m \in \N$ with~$m \geq \frac14 L$ and~$t \geq 1$,
\begin{align*}
	\lefteqn{
	\P\Biggl[
	\sup_{k \in \Z \cap (-\infty, m]}
	3^{\gamma(k-m)} \max_{z\in 3^{k} \Zd \cap \cu_m}
	\bigl| \bfAhom_L^{-\nf 12}(\cu_m) \bfA_L(z +\cu_k) \bfAhom_L^{-\nf 12}(\cu_m) \bigr| > K \gamma^{-1} \nu^{-2} L t \Biggr]
	} \quad &
	\\ &
	\leq
	\sum_{k = -\infty}^{m}
	3^{d(m-k)}
	\P\Bigl[
	\bigl| \bfE_L^{-\nf 12}\bfA_L(\cu_k) \bfE_L^{-\nf 12} \bigr| > c K \gamma^{-1} 3^{\gamma(m-k)} t \Bigr]
	\\ & \qquad
	\leq
	\sum_{k = 0}^{\infty}
	3^{dk }
	\exp \Bigl( - cK \gamma^{-1} 3^{\gamma k} t   \Bigr)
	\leq
	\exp(- t)
	\,,
\end{align*}
where the validity of the last inequality was obtained by choosing~$K(d)$ sufficiently large. This concludes the proof.
\end{proof}

\subsection{Homogenization and renormalization}

We first use the preceding estimates to verify the hypotheses of~\cite[Section~6]{AK.HC}. We then apply~\cite[Theorem~6.1]{AK.HC} twice: first at an intermediate scale to compare~$\shom_{L,*}(\cu_n)$ with~$\shom_L$, and then a second time at the target scale~$m$ to obtain homogenization. The quenched estimate is then a short combination of this annealed bound with Proposition~\ref{p.mixing.P.three.prime}.

\begin{lemma}
\label{l.we.can.apply.hc}
There exists a constant~$C(d)$ such that for every~$\gamma \in (0,1)$,~$\alpha \in [0,1)$,~$M \geq 1$, and every~$L \in \N$ with
\begin{equation*}
	L \geq L_0( C M,\alpha,\cstar,\nu) \,,
\end{equation*}
the infrared cutoff field~$\bfA_L$ satisfies the stationarity assumption~(P1) and the isotropy assumption~(P4) from~\cite{AK.HC}. Moreover, in the weaker assumptions~(P2') and~(P3') from~\cite[Section~6]{AK.HC} we may choose
\begin{equation}
	\begin{aligned}
		H  \coloneqq  C_{\eqref{e.ellipticity.Ptwoprime}} \gamma^{-1} \nu^{-2}\,,
		\qquad
		m_2  \coloneqq  \left\lceil \frac14 L\right\rceil\,,
		\qquad
		D  \coloneqq  1\,,
		\\
		\Psi_{\mathcal{S}}  \coloneqq  \Gamma_1\,,
		\qquad
		K_{\Psi_\mathcal{S}}  \coloneqq  2 \exp(2)
		\qand
		p_{\Psi_{\mathcal{S}}}  \coloneqq  2d \,.
	\end{aligned}
	\label{e.checkingp2prime}
\end{equation}
Here~(P2') is supplied by the crude ellipticity estimate Proposition~\ref{p.ellipticity.Ptwoprime}, while~(P3') is supplied by the mixing estimate~\eqref{e.main.mixing.estimate}; the constants below are chosen only to match the notation of~\cite[Section~6]{AK.HC}. With~$C_{\mathrm{mix}}$ denoting the constant in the cutoff indicator of~\eqref{e.main.mixing.estimate},
\begin{equation}
	\begin{aligned}
		\beta  \coloneqq  \nf12\,,
		\qquad
		m_3  \coloneqq  \max\bigl\{1,\lceil L-M L^{\alpha} \log^3 L\rceil\bigr\}\,,
		\qquad
		L_1  \coloneqq  4 C_{\eqref{e.mixing.gaps}}\,,
		\qquad
		L_2  \coloneqq  \nu^{-1}\,,
		\\
		\omega_n  \coloneqq  \tilde \omega_n + (1+n)^{-3000}\quad(n\geq m_3)\,,
		\qquad
		\Psi  \coloneqq  \Gamma_{\! \nf13}\,,
		\qquad
		K_{\Psi}  \coloneqq  C
		\qand
		p_{\Psi}  \coloneqq  2d\,,
		\\
		h_n  \coloneqq  n - \bigl\lceil C_{\eqref{e.mixing.gaps}} \log L \bigr\rceil\,,
		\qquad
		\tilde \omega_n  \coloneqq  C (L-n)_+^{\nf12} \shom_{L,*}^{-1} (\cu_{h_n})
		\indc_{\{ n \leq L + C_{\mathrm{mix}} \log L \} } \,.
	\end{aligned}
	\label{e.checkingp3prime}
\end{equation}
The quantities~$\shom_{L,*}(\cu_{h_n})$ are annealed scalar coefficients and are deterministic. The scale~$h_n$ appears here because it is the scale used in the high-contrast iteration theorem; the mixing estimate itself is verified at scale~$n$ below. For~$0\leq n<m_3$, extend the sequence by~$\omega_n\coloneqq\omega_{m_3}$. Subadditivity and the displayed formula show that the resulting positive sequence is nonincreasing on~$\N$ and converges to zero. The functions~$\Gamma_1$ and~$\Gamma_{\nf13}$ satisfy the growth assumptions of~\cite[Section~6]{AK.HC} with the displayed dimensional constants and~$p=2d$. The parameter~$\omega_n$ is used below only for~$n\geq m_3$; after increasing the threshold, this ensures~$h_n\geq L/2$. We have also enlarged the constants so that~$\log(\nu^{-1}L)\leq C\log L$ above the threshold.
\end{lemma}
\begin{proof}
Assume that~$L\in\N$ satisfies
\begin{equation}
	L \geq  L_0( C M,\alpha,\cstar,\nu)   \,,
	\label{e.initial.L.condition}
\end{equation}
where~$C(d)$ is chosen so that~$L-n \leq  M L^{\alpha} \log^3 L$ implies the first condition in~\eqref{e.mixing.gaps}; this is allowed by~\eqref{e.L.vs.Lnaught}. The assumptions~(P1) and~(P4) are immediate from stationarity of the fields~$\mathbf j_n$ and Assumption~\ref{a.j.iso}.

\smallskip

For~(P2'), Proposition~\ref{p.ellipticity.Ptwoprime} gives, for every~$j \geq m_2$,
\begin{equation*}
	\sup_{k \in \Z \cap (-\infty,j]}
	3^{\gamma(k-j)}  \max_{z\in 3^{k} \Zd  \cap \cu_j}
	\bigl|  \bfAhom_L^{-\nf12}(\cu_j)  \bfA_L(z +\! \cu_k) \bfAhom_L^{-\nf12}(\cu_j) \bigr|
	\leq \O_{\Gamma_1} (C_{\eqref{e.ellipticity.Ptwoprime}} \gamma^{-1} \nu^{-2} L)
	\,.
\end{equation*}
Set
\begin{equation*}
	B_{j,k,z}\coloneqq \bfAhom_L^{-\nf12}(\cu_j)\bfA_L(z+\cu_k)\bfAhom_L^{-\nf12}(\cu_j)\,.
\end{equation*}
This is the symmetric positive-definite matrix occurring in~(P2'). Therefore
\begin{equation*}
	\bigl|(B_{j,k,z}-\Itwod)_+\bigr|
	\leq
	|B_{j,k,z}-\Itwod|
	\leq 1+|B_{j,k,z}|\,.
\end{equation*}
Since~$D=1$ and~$j \geq m_2 = \lceil L/4\rceil$, Proposition~\ref{p.ellipticity.Ptwoprime} bounds the right side by~$C Hj^D$. This is exactly~(P2') with the parameters in~\eqref{e.checkingp2prime}, after enlarging~$H$ by a dimensional factor.

\smallskip

For~(P3'), fix~$n,m\in\N$ with~$n \geq m_3$ and~$\beta m < n < m - L_1 \log(L_2 n)$. We claim that these assumptions imply the hypotheses of Proposition~\ref{p.mixing.P.three.prime}. Since~$\beta = \nf12$, we have~$m < 2n$ and, after increasing the threshold in~\eqref{e.initial.L.condition}, the lower bound~$n\geq m_3$ gives~$\log(\nu^{-1}L)\leq 2\log(\nu^{-1}n)$. Therefore
\begin{equation*}
	\log(\nu^{-1} m) < \log(2\nu^{-1} n) \leq 2\log(\nu^{-1} n)
	\qand
	\log(\nu^{-1}L)\leq 2\log(\nu^{-1}n)
\end{equation*}
for~$n$ large enough. Since~$m > n + L_1 \log(L_2 n) = n + 4 C_{\eqref{e.mixing.gaps}} \log(\nu^{-1} n)$, the second gap condition in Proposition~\ref{p.mixing.P.three.prime} follows. The lower bound~$n \geq C_{\eqref{e.mixing.gaps}}\log(\nu^{-1}L)$ follows from~$n\geq m_3$ and the threshold~\eqref{e.Lnaught.def}. The mixing window condition holds because~$n \geq m_3$ gives~$L - n \leq ML^\alpha \log^3 L$: if~$n\leq L$, then~$n\in[\frac12L,L]$ after increasing the threshold and~\eqref{e.L.vs.Lnaught} gives~$\shom_{L,*}^{\,2}(\cu_n) \geq c(d) M L^\alpha \log^3 L$; if~$n>L$, the condition is trivial. Under this condition, the~$\Gamma_1$ term in~\eqref{e.main.mixing.estimate} is the square of the~$\Gamma_2$ term and is absorbed into the~$\Gamma_{\nf13}$ bound after increasing the constant in~$\tilde\omega_n$.

\smallskip

We may therefore apply Proposition~\ref{p.mixing.P.three.prime}, which yields, for every~$n,m\in\N$ with~$n \geq  m_3$ and~$\beta m < n < m - L_1 \log (L_2 n)$,
\begin{equation}
	\biggl|
	\bfAhom_{L}^{-\nf12}(\cu_n)
	\! \! \! \avsum_{z \in  3^n \Zd \cap \cu_m}  \! \! \!
	\bigl(\bfA_{L}(z+\cu_n) - \bfAhom_{L}(\cu_n) \bigr)
	\bfAhom_{L}^{-\nf12}(\cu_n)
	\biggr|
	\leq
	\O_{\Gamma_{\nf13}} (\tilde \omega_n + (1+n)^{-3000} )
	\,.
\end{equation}
This is precisely~(P3') with the parameters in~\eqref{e.checkingp3prime}.
\end{proof}

We next use~\cite[Theorem~6.1]{AK.HC} to prove the following homogenization bound.
\begin{proposition}[Homogenization below the infrared cutoff]
\label{p.homog.below}
There exists a constant~$C(d)<\infty$ such that, for every~$\alpha \in [0,1)$ and~$M \in [1,\infty)$ and for every~$L,m \in \N$ satisfying
\begin{equation}
	L \geq L_0(CM,\alpha,\cstar,\nu)
	\qquad \mbox{and} \qquad
	m \geq L - M L^\alpha \log^{3} L\,,
\end{equation}
we have the estimate
\begin{equation}
	\label{e.ass.valid}
	\bigl| \shom_L^{-1}\shom_L(\cu_m) - \Id \bigr|+\bigl| \shom_L\shom_{L,*}^{-1}(\cu_m) - \Id \bigr|
	\leq C\shom_L^{-2}(L-m + C\log^2 L)_+
	+ Cm^{-3000}
	\,.
\end{equation}
\end{proposition}
\begin{proof}
Assume that~$L\in\N$ satisfies
\begin{equation}
	L \geq L_0(KM,\alpha,\cstar,\nu)
\end{equation}
for a constant~$K \geq C_{\eqref{e.initial.L.condition}}$, depending only on~$d$, to be determined below. It is enough to prove the result under this stronger assumption, since the constant~$C$ in the statement is chosen after the final value of~$K$ is fixed. Select~$\gamma=\nf12$ and fix the parameters from~\cite[Section~6]{AK.HC} as in Lemma~\ref{l.we.can.apply.hc}, but with~$M$ replaced by~$2M$. In particular, when we invoke~\cite[Theorem~6.1]{AK.HC} below, the parameter~$m_3$ is
\begin{equation*}
	m_3 = \max\bigl\{1,\lceil L - 2M L^\alpha \log^3 L\rceil\bigr\}\,.
\end{equation*}
The proof is a double application of~\cite[Theorem~6.1]{AK.HC}. We first choose the initial scale~$m_0$ and verify the smallness assumption. We then apply the theorem from an initial scale~$\tilde m$ up to an intermediate scale~$n = \tilde m + 5m_0$, which yields~$\shom_{L,*}(\cu_n) \simeq \shom_L$. Finally, we choose~$\tilde m'$ so that~$h_{\tilde m'} = n$, insert this comparison into~$\omega_{\tilde m'}$, and apply the theorem a second time from~$\tilde m'$ up to the target scale~$m$. Throughout the proof,~$\Theta_{L,n}$ denotes the renormalized ellipticity ratio from~\cite[Theorem~6.1]{AK.HC} for the coefficient field~$\bfA_L$ on~$\cu_n$.

\smallskip
\emph{Step 1.} Choosing the initial scale. We verify the two numerical hypotheses in~\cite[Theorem~6.1]{AK.HC}. By the definition of~$\bfE_L$ in~\eqref{e.Enaught.mixing}, the initial ellipticity ratio~$\Theta_{L,0}$ satisfies
\begin{equation*}
	\Theta_{L,0} \leq C(d) \nu^{-2} L \,.
\end{equation*}
We compute the constants from~\cite[Theorem~6.1]{AK.HC}. With the parameters in~\eqref{e.checkingp2prime} and~\eqref{e.checkingp3prime}, we have
\begin{equation*}
	\Upsilon_1
	\coloneqq
	\frac{C K_{\Psi}^{4d^2}}{\min\{d+1,p_\Psi\} - d} = C(d)
\end{equation*}
and
\begin{equation*}
	\Upsilon_2
	\coloneqq
	\frac{C}{\min\{3,p_{\Psi_\S}\} - 2}
	\exp \biggl(C \Bigl( L_1 \log L_2 + \frac{D+\log (H +  K_{\Psi_\S})}{1-\gamma} \Bigr)\biggr)
	= C \exp (C \log(\nu^{-1}))
	\,.
\end{equation*}
The initial scale condition in~\cite[Theorem~6.1]{AK.HC} requires
\begin{equation*}
	m_0
	\geq
	\frac{CD}{(1-\beta)(1-\gamma)^2} \log \bigl( (\Upsilon_1 + \Upsilon_2) (m_2 \vee m_3 + m_0) \Theta_{L,0}  \bigr) \log( 3\Theta_{L,0}) \,.
\end{equation*}
Since~$D=1$,~$\beta=\nf12$,~$\gamma=\nf12$,~$\Upsilon_1+\Upsilon_2 \leq C\nu^{-C}$, and~$\Theta_{L,0} \leq C\nu^{-2}L$, this is satisfied by the choice
\begin{equation*}
	m_0  \coloneqq  \lceil C \log^2 L\rceil
\end{equation*}
for a sufficiently large constant~$C(d)$, after increasing~$K$ if necessary. Indeed, the lower bound~$L \geq L_0(KM,\alpha,\cstar,\nu)$ implies~$\log(\nu^{-1}) \leq C(\alpha,\cstar,d)\log L$ by~\eqref{e.Lnaught.def}, so the factor~$\nu^{-C}$ in~$\Upsilon_2$ contributes only through a logarithm and does not create any additional power of~$\nu^{-1}$ in the choice of~$m_0$. We also verify the smallness condition~$\omega_{\tilde m}^2 \leq \Upsilon_1^{-1}$ for every~$\tilde m \geq m_3$. By the definition of~$\omega_n$ in~\eqref{e.checkingp3prime},
\begin{equation*}
	\omega_{\tilde m}^2
	\leq
	C (L - \tilde m)_+ \shom_{L,*}^{-2}(\cu_{h_{\tilde m}})
	\indc_{\{ \tilde m \leq L + C_{\mathrm{mix}} \log L \} }
	+ \tilde m^{-6000} \,.
\end{equation*}
Since~$L - \tilde m \leq L - m_3 \leq 2M L^\alpha \log^3 L$ and, by~\eqref{e.L.vs.Lnaught} with~$M$ replaced by~$2KM$,
\begin{equation*}
	\shom_{L,*}^{\,2}(\cu_{h_{\tilde m}})
	\geq
	c(d) K M L^\alpha \log^3 L \,,
\end{equation*}
we obtain~$\omega_{\tilde m}^2 \leq C(d) K^{-1}$. Since~$\Upsilon_1 = C(d)$, the condition~$\omega_{\tilde m}^2 \leq \Upsilon_1^{-1}$ holds after increasing~$K$ if necessary.

\smallskip
\emph{Step 2.} The first application of~\cite[Theorem~6.1]{AK.HC}. Fix~$\tilde m \in \N$ with~$\tilde m \geq m_3$. Applying~\cite[Theorem~6.1]{AK.HC} with initial scale~$\tilde m$ and final scale~$n$, and with~$\kappa  \coloneqq  \min\{\alpha_*(d), 1-\gamma\} = \min\{\alpha_*(d), \nf12\}$, where~$\alpha_*(d)$ is the exponent from that theorem, gives for every~$n \geq \tilde m + 4m_0$,
\begin{equation}
	\Theta_{L,n} -1
	\leq
	\Upsilon_1 \omega_{\tilde m}^2 + C 3^{-\kappa(n - \tilde m - 4m_0)}
	\,.
\end{equation}
For~$n \geq \tilde m + 5 m_0$, the exponential remainder satisfies~$3^{-\kappa m_0} \leq \tilde m^{-6000}$, since~$m_0 = C\log^2 L$ and~$\tilde m \geq m_3 \geq \frac12 L$, after increasing~$C(d)$ if necessary. Therefore, for every~$\tilde m \geq m_3$ and~$n \geq \tilde m + 5 m_0$,
\begin{equation}
	\label{e.Theta.convergence.prime.simplified}
	\Theta_{L,n} - 1 \leq C(d) \omega_{\tilde m}^2 + C(d) \tilde m^{-6000} \leq \frac14 \,,
\end{equation}
increasing~$K$ if necessary.

\smallskip

Also, by subadditivity and monotonicity of the two scalar diagonal blocks, we obtain for every~$h \in \N$,
\begin{equation}
	\label{e.going.in.a.circle}
	\shom_L  \leq
	\shom_{L}(\cu_{h})
	\leq  \Theta_{L,h} \shom_{L,*}(\cu_{h})
	\leq \Theta_{L,h} \lim_{j \to \infty} \shom_{L,*}(\cu_j) =
	\Theta_{L,h}\, \shom_L
	\,.
\end{equation}
In particular, combining~\eqref{e.Theta.convergence.prime.simplified} and~\eqref{e.going.in.a.circle} with~$n  \coloneqq  \tilde m + 5 m_0$, we obtain for every~$\tilde m \geq m_3$,
\begin{equation}
	\label{e.rats.to.infty}
	\frac12 \shom_L \leq  \shom_{L,*}(\cu_{n}) \leq  \shom_L \,.
\end{equation}
Thus, once~$n$ is at least~$5m_0$ above the initial scale~$\tilde m$, the renormalized ellipticity ratio is close enough to~$1$ to force~$\shom_{L,*}(\cu_n)$ to be comparable to the limiting value~$\shom_L$.

\smallskip
\emph{Step 3.} The second application of~\cite[Theorem~6.1]{AK.HC}. Now fix~$m \in \N$ satisfying the hypothesis of the proposition. We choose the initial scales for the two applications of~\cite[Theorem~6.1]{AK.HC} by setting
\begin{equation*}
	\tilde m  \coloneqq  m - 10m_0 - \lceil C_{\eqref{e.mixing.gaps}} \log L \rceil\,,
	\qquad
	n  \coloneqq  \tilde m + 5m_0\,,
	\qquad
	\tilde m'  \coloneqq  n + \lceil C_{\eqref{e.mixing.gaps}} \log L \rceil = m - 5m_0 \,.
\end{equation*}
These are integers since~$m_0\in\N$. They are positive for sufficiently large~$L$, since~$10m_0+C\log L=O(\log^2L)$ and~$m\geq L-ML^\alpha\log^3L$. Moreover, after increasing~$K$ if necessary, we have
\begin{equation*}
	\tilde m \geq m_3 \geq L - 2M L^\alpha \log^3 L
\end{equation*}
for~$L$ large enough. By~\eqref{e.rats.to.infty} with this choice~$n=\tilde m+5m_0$, we therefore have~$\shom_{L,*}(\cu_n) \geq \frac12 \shom_L$.

\smallskip

These choices are arranged so that
\begin{equation*}
	h_{\tilde m'} = n
	\qand
	m \geq \tilde m' + 5m_0 \,.
\end{equation*}
The first relation lets us insert the intermediate-scale comparison from Step~2 into the definition of~$\omega_{\tilde m'}$, while the second shows that~$m$ lies far enough above~$\tilde m'$ to be the final scale in the second application of~\cite[Theorem~6.1]{AK.HC}. Note also that~$\tilde m' \geq m_3$ after increasing~$K$ if needed. Using~\eqref{e.rats.to.infty} with~$n = h_{\tilde m'}$, the definition of~$\omega_{\tilde m'}$ gives
\begin{equation*}
	\omega_{\tilde m'}^2 \leq C (L - \tilde m')_+ \shom_{L}^{-2} + \tilde m'^{-6000}
	\leq C (L - m + 5m_0)_+ \shom_{L}^{-2} + Cm^{-6000} \,.
\end{equation*}
Since~$m \geq \tilde m' + 5m_0 \geq \tilde m' + 4m_0$, we may apply~\eqref{e.Theta.convergence.prime.simplified} a second time, now with initial scale~$\tilde m'$ and final scale~$m$. Together with~\eqref{e.going.in.a.circle}, this yields
\begin{equation}
	\label{e.Theta.Lm.final.bound}
	\Theta_{L,m} -1
	\leq
	C \shom_{L}^{-2}(L - m + C\log^2 L)_+
	\indc_{\{ m \leq L + C\log^2 L \} }
	+
	C\, m^{-3000}
	\,,
\end{equation}
where we used~$5m_0 \leq C\log^2 L$ to absorb the scale shift into the positive-part term. The indicator can be removed in~\eqref{e.ass.valid}, since the positive part already vanishes unless~$m\leq L+C\log^2L$. To conclude the proof, we spell out how this yields~\eqref{e.ass.valid}.
\smallskip

To see this last implication explicitly, recall from~\eqref{e.homs.defs.U} and the remark following it that both diagonal blocks are scalar matrices. We may therefore write
\begin{equation*}
	\shom_L(\cu_m) = a_m \shom_L \Id
	\qquad\text{and}\qquad
	\shom_{L,*}^{-1}(\cu_m) = b_m \shom_L^{-1} \Id
\end{equation*}
for some numbers~$a_m,b_m \geq 1$, where the lower bounds follow from the monotonicity of the upper-left block and of the inverse lower-right block in~$\bfAhom_L(\cu_m)$. In this scalar block-diagonal setting, the renormalized ellipticity ratio is
\begin{equation*}
	\Theta_{L,m} = a_m b_m \,.
\end{equation*}
If we denote the right-hand side of~\eqref{e.Theta.Lm.final.bound} by~$\varepsilon_m$, then~$\Theta_{L,m}-1 \leq \varepsilon_m$ gives~$a_m b_m \leq 1+\varepsilon_m$. Since~$a_m,b_m \geq 1$, it follows that~$a_m \leq 1+\varepsilon_m$ and~$b_m \leq 1+\varepsilon_m$, and hence
\begin{equation*}
	\bigl| \shom_L^{-1}\shom_L(\cu_m) - \Id \bigr|+\bigl| \shom_L\shom_{L,*}^{-1}(\cu_m) - \Id \bigr|
	=
	(a_m-1) + (b_m-1)
	\leq 2\varepsilon_m \,.
\end{equation*}
After absorbing the factor~$2$ into the constant, this gives~\eqref{e.ass.valid}.
\end{proof}

We record an application of Proposition~\ref{p.homog.below} and Lemma~\ref{l.localization} which enables us to compare the effective diffusivity for different infrared cutoffs.
\begin{lemma}
\label{l.shomm.vs.shomell}
There exists a constant~$C(d)$ such that, for every~$\alpha \in[0,1)$,~$M \in [1,\infty)$ and~$L, \ell \in \N$ satisfying
\begin{equation}
	L \geq \ell  \geq  L_0(CM,\alpha,\cstar,\nu)
	\qquad \mbox{and} \qquad
	\ell \geq L - M  L^\alpha \log^{3} L  \,,
\end{equation}
we have
\begin{equation}
	\label{e.sL.vs.sell}
	\bigl| \shom_\ell^{-1} \shom_L  - \Id \bigr|
	\leq  C M  L^\alpha \shom_L^{-2}  \log^{3} L \leq \frac12
	\,.
\end{equation}
Moreover, there exists a deterministic scale~$L_{\mathrm{g}}(\cstar,\nu,\nondegconst,d)\in\N$ such that, for every~$L\geq L_{\mathrm{g}}$,
\begin{equation}
	\label{e.sL.growth}
	C^{-1 }\cstar^{\nf32} \nu^{2} L^{\nf12} \log^{-\nf{9}{2}}L
	\leq
	\shom_L
	\leq
	C\cstar^{-\nf32} \nu^{-2} L^{\nf12} \log^{\nf{9}{2}}L\,.
\end{equation}
\end{lemma}
\begin{proof}
Assume that~$L, \ell,n \in \N$ satisfy
\begin{equation*}
	L \geq \ell  \geq  L_0(KM,\alpha,\cstar,\nu)
	\qand
	n  \coloneqq  \ell - \lceil 200 \log L \rceil\,,
\end{equation*}
where~$K(d)$ is a large constant to be selected below. By increasing~$K$ and the threshold in~$L_0$, if necessary, we may assume that~$\ell \geq \frac12 L$ and
\begin{equation*}
	200 \log L \leq K M \ell^\alpha \log^3 \ell \leq K M L^\alpha \log^3 L \,.
\end{equation*}
In particular,
\begin{equation*}
	n \geq \ell - K M \ell^\alpha \log^3 \ell
	\qand
	n \geq L - K M L^\alpha \log^3 L \,.
\end{equation*}
According to Proposition~\ref{p.homog.below}, applied with parameter~$K M$ to the two cutoffs~$L$ and~$\ell$,
\begin{equation*}
	| \shom_{L}^{-1} \shom_{L}(\cu_{n}) - \Id |
	+
	| \shom_{\ell}^{-1} \shom_{\ell}(\cu_{n}) - \Id | \leq
	2 K C_{\eqref{e.ass.valid}} M  L^\alpha (\shom_\ell^{-2} + \shom_L^{-2} )  \log^{3}L
\end{equation*}
after absorbing the small polynomial remainders into the constant. By~\eqref{e.L.vs.Lnaught}, with~$M$ replaced by a sufficiently large multiple of~$KM$, the last right side is at most~$\frac18$ after increasing~$K(d)$.

The choice~$\ell-n = \lceil 200\log L\rceil$ makes the localization error smaller than~$L^{-99}$ after taking expectations and using the moment bounds. We therefore compare the full matrices~$\bfA_L(\cu_n)$ and~$\bfA_\ell(\cu_n)$ and then extract their upper-left blocks. Set~$\h\coloneqq(\k_L-\k_\ell)_{\cu_n}$, a constant antisymmetric matrix, and define the normalized error
\begin{equation*}
	\mathbf D\coloneqq
	\bfA_\ell^{-\nf12}(\cu_n)\G_{\h}^t\bfA_L(\cu_n)\G_{\h}\bfA_\ell^{-\nf12}(\cu_n)-\Itwod\,.
\end{equation*}
By~\eqref{e.localization.ml}, we have~$|\mathbf D|\leq\O_{\Gamma_1}(C\nu^{-2}3^{-(\ell-n)})$, and the exact identity is
\begin{equation*}
	\bfA_L(\cu_n)
	=
	\G_{-\h}^t\bfA_\ell^{\nf12}(\cu_n)(\Itwod+\mathbf D)\bfA_\ell^{\nf12}(\cu_n)\G_{-\h}\,.
\end{equation*}
The additive error relative to~$\G_{-\h}^t\bfA_\ell(\cu_n)\G_{-\h}$ is therefore bounded by
\begin{equation*}
	|\mathbf D|\,\bigl|\bfA_\ell^{\nf12}(\cu_n)\G_{-\h}\bigr|^2\,.
\end{equation*}
Its expectation is at most~$CL^{-99}$ by Cauchy--Schwarz, the moment bounds~\eqref{e.Enaught.vs.A.and.Ahom} and~\eqref{e.kmn.bounds}, the identity~$\ell-n=\lceil200\log L\rceil$, and the lower bound on~$L$ in~\eqref{e.L.vs.Lnaught}. Taking the upper-left block and expectations, and expanding~$\G_{-\h}^t\bfA_\ell(\cu_n)\G_{-\h}$ via~\eqref{e.commute.coarse.grained.k0}, we obtain
\begin{equation*}
	\Bigl|\shom_L(\cu_n)-\shom_\ell(\cu_n)
	-\E\bigl[\h^t\s_{\ell,*}^{-1}(\cu_n)\h\bigr]
	-\E\bigl[\kcg_\ell(\cu_n)^t\s_{\ell,*}^{-1}(\cu_n)\h+\h^t\s_{\ell,*}^{-1}(\cu_n)\kcg_\ell(\cu_n)\bigr]\Bigr|
	\leq
	CL^{-99}\,.
\end{equation*}
The matrix~$\h=(\k_L-\k_\ell)_{\cu_n}$ is measurable with respect to the scales in~$(\ell,L]$, while~$\kcg_\ell(\cu_n)$ and~$\s_{\ell,*}(\cu_n)$ are measurable with respect to the scales~$\leq\ell$; by~\ref{a.j.indy} these are independent, and by the negation symmetry in~\ref{a.j.iso}, as in the derivation of~$\khom_\ell(\cu_n)=0$,~$\E[\h]=0$. Conditioning on the scales~$\leq\ell$, the two cross terms vanish in expectation, and, since~$\E[\s_{\ell,*}^{-1}(\cu_n)]=\shom_{\ell,*}^{-1}(\cu_n)$ is a scalar multiple of~$\Id$,
\begin{equation*}
	|\shom_L(\cu_n)-\shom_\ell(\cu_n)|
	\leq
	\shom_{\ell,*}^{-1}(\cu_n)\,\E|\h|^2+CL^{-99}\,.
\end{equation*}
By Jensen's inequality and~\eqref{e.kmn.bounds} with~$p=2$,~$\E|\h|^2\leq C(L-\ell)$, and Proposition~\ref{p.homog.below} applied to the pair~$(\ell,n)$ gives~$\shom_{\ell,*}^{-1}(\cu_n)\leq C\shom_\ell^{-1}$ and~$\shom_\ell(\cu_n)\geq \frac78\shom_\ell$. Hence~$|\shom_L(\cu_n)-\shom_\ell(\cu_n)|\leq C(L-\ell)\shom_\ell^{-1}+CL^{-99}$, and, dividing by~$\shom_\ell(\cu_n)$ and using~$L-\ell\leq M L^\alpha\log^3 L$,
\begin{equation*}
	|\shom_{L}(\cu_n) \shom_{\ell}^{-1}(\cu_{n}) - \Id| \leq  C M L^\alpha\shom_\ell^{-2}\log^3 L+C L^{-99} \leq \frac18\,,
\end{equation*}
after increasing~$K(d)$, by~\eqref{e.L.vs.Lnaught}. Since all three factors below are scalar multiples of~$\Id$, we may write
\begin{equation*}
	\shom_\ell^{-1} \shom_L
	=
	\bigl( \shom_\ell^{-1}\shom_\ell(\cu_n) \bigr)
	\bigl( \shom_L(\cu_n)\shom_\ell^{-1}(\cu_n) \bigr)
	\bigl( \shom_L^{-1}\shom_L(\cu_n) \bigr)^{-1}\,.
\end{equation*}
Combining the previous two displays, we obtain
\begin{equation*}
	\bigl| \shom_\ell^{-1} \shom_L - \Id \bigr|
	\leq
	C M  L^\alpha (\shom_\ell^{-2} + \shom_L^{-2} )  \log^{3}L
	+
	CL^{-99}
	\leq
	\frac12
	\,.
\end{equation*}
In particular,~$\shom_\ell^{-2} \leq \frac94\shom_L^{-2}$, and therefore~\eqref{e.sL.vs.sell} follows after increasing the constant in~\eqref{e.sL.vs.sell} if necessary.

\smallskip

To show~\eqref{e.sL.growth}, fix a large constant~$A \geq C_{\eqref{e.sL.vs.sell}}$ and set
\begin{equation*}
	L_*  \coloneqq  L_0(2C_{\eqref{e.sL.vs.sell}}A,0,\cstar,\nu)\,.
\end{equation*}
By~\eqref{e.Enaught.vs.A.and.Ahom}, we have the crude bound
\begin{equation*}
	\shom_r \leq C\nu^{-1}(1+r)
	\qquad \mbox{for every } r\in \N \,.
\end{equation*}
Choose a deterministic threshold~$L_{**}(\cstar,\nu,\nondegconst,d) \geq 2L_*\vee\lceil \exp(9)\rceil$ so large that
\begin{equation*}
	\lceil A \log^3 t \rceil \leq \frac12 t
	\qquad \mbox{for every } t \geq L_{**}
\end{equation*}
and then choose~$L_{\mathrm{g}}(\cstar,\nu,\nondegconst,d)\geq L_{**}$ so large that
\begin{equation*}
	C\nu^{-1}(1+L_{**})
	\leq
	C \cstar^{-\nf32} \nu^{-2} L^{\nf12} \log^{\nf{9}{2}}L
	\qquad \mbox{for every } L \geq L_{\mathrm{g}}\,.
\end{equation*}
It remains to prove~\eqref{e.sL.growth} for~$L \geq L_{\mathrm{g}}$.

\smallskip

For such~$L$, define recursively an increasing sequence~$\{R_j\}_{j=0}^J$ by setting~$R_J \coloneqq L$ and
\begin{equation*}
	R_j  \coloneqq  R_{j+1} - \lceil A \log^3 R_{j+1} \rceil
	\qquad \mbox{for every } j\in \{0,\ldots,J-1\}\,,
\end{equation*}
where~$J$ is the smallest index such that~$R_0 < L_{**}$. Since then~$R_1 \geq L_{**}$, the choice of~$L_{**}$ gives
\begin{equation*}
	R_0 \geq R_1 - \frac12 R_1 \geq \frac12 L_{**} \geq L_* \,.
\end{equation*}
In particular, for every~$j\in \{0,\ldots,J-1\}$,
\begin{equation*}
	R_{j+1} \geq R_j \geq L_*
	\qand
	R_j \geq R_{j+1} - 2A \log^3 R_{j+1}\,.
\end{equation*}
Applying the first part of the lemma, namely~\eqref{e.sL.vs.sell} with~$\alpha = 0$,~$M = 2A$,~$\ell=R_j$ and~$L=R_{j+1}$, we obtain
\begin{equation*}
	\bigl| \shom_{R_j}^{-1} \shom_{R_{j+1}} - \Id \bigr|
	\leq C A \shom_{R_{j+1}}^{-2} \log^3 R_{j+1} \leq \frac12 \,.
\end{equation*}
Consequently,
\begin{equation*}
	\frac12 \shom_{R_j} \leq \shom_{R_{j+1}} \leq \frac32 \shom_{R_j}
\end{equation*}
and therefore
\begin{equation*}
	\shom_{R_{j+1}} - \shom_{R_j}
	\leq
	C A \shom_{R_{j+1}}^{-1} \log^3 R_{j+1} \,.
\end{equation*}
Using Proposition~\ref{p.sstar.lower.bound} at scale~$R_j$, and the fact that~$R_j \simeq R_{j+1}$, we deduce that
\begin{equation*}
	\shom_{R_{j+1}} - \shom_{R_j}
	\leq
	C \cstar^{-\nf32} \nu^{-2} (R_{j+1}-R_j) R_{j+1}^{-\nf12} \log^{\nf{9}{2}}R_{j+1} \,.
\end{equation*}
Since~$t^{-\nf12}\log^{\nf{9}{2}}t$ is decreasing for~$t\geq L_{**}$, summing over~$j\in \{0,\ldots,J-1\}$ gives
\begin{align*}
	\shom_L
	&
	\leq
	\shom_{R_0}
	+
	C \cstar^{-\nf32} \nu^{-2}
	\sum_{j=0}^{J-1}
	(R_{j+1}-R_j) R_{j+1}^{-\nf12} \log^{\nf{9}{2}}R_{j+1}
	\notag \\ &
	\leq
	\shom_{R_0}
	+
	C \cstar^{-\nf32} \nu^{-2}
	\int_{R_0}^{L} t^{-\nf12} \log^{\nf{9}{2}}t \,dt
	\leq
	\shom_{R_0}
	+
	C \cstar^{-\nf32} \nu^{-2} L^{\nf12} \log^{\nf{9}{2}}L \,.
\end{align*}
Finally, since~$R_0 < L_{**}$, the crude bound above gives~$\shom_{R_0} \leq C\nu^{-1}(1+L_{**})$. By the choice of~$L_{\mathrm{g}}$ and the assumption~$L\geq L_{\mathrm{g}}$, this term is bounded by~$C \cstar^{-\nf32} \nu^{-2} L^{\nf12} \log^{\nf{9}{2}}L$, and therefore the previous display implies the upper bound in~\eqref{e.sL.growth}. The lower bound, on the other hand, is a consequence of Proposition~\ref{p.sstar.lower.bound}.
\end{proof}

Using the previous proposition and the mixing estimates above, we obtain quenched homogenization estimates.

\begin{lemma}[Parameterized mixing estimate]
\label{l.new.mixing.parameterized}
There exists a constant~$C(d)<\infty$ such that, for every~$\alpha \in [0,1)$,~$M,K\in [1,\infty)$ and~$L,m,r \in \N$ satisfying
\begin{equation*}
	L,m \geq L_0(C(M+K),\alpha,\cstar,\nu)\,,
	\qquad
	m \geq L - M L^\alpha \log^3L
	\qand
	|L-r| \leq K\log L\,,
\end{equation*}
there is a nonnegative random variable~$W_{L,m,r}=\O_{\Gamma_{\nf13}}(C)$ such that, for every constant anti-symmetric matrix~$\h_0 \in \R^{d\times d}$,
\begin{align}
	\lefteqn{
	\max_{|\eta|=1}
	\bfJ\bigl(
	\cu_m\,,
	\bfAhom_r^{-\nf 12} \eta \,,
	\bfAhom_r^{\nf 12} \eta
	\,;\,\a_L + \h_0
	\bigr)
	} \qquad &
	\notag \\ &
	\leq
	C\shom_r^{-2}\bigl(|\h_0|^2+ (L-m)_+ + K\log^2L\bigr)
	+
	\O_{\Gamma_1}
	\bigl(
	C\shom_r^{-2}\bigl( (L-m)_+ + K\log L\bigr)
	\bigr)
	\notag
	\\ & \qquad
	+
	m^{-3000}
	\bigl(1+\shom_r^{-1}|\h_0|^2\bigr)W_{L,m,r}\,.
\end{align}
\end{lemma}

\begin{proof}
Fix~$\alpha \in [0,1)$,~$M,K\in [1,\infty)$,~$L,m,r \in \N$ satisfying the hypotheses of the lemma, and a constant anti-symmetric matrix~$\h_0$. Let
\begin{equation*}
	h \coloneqq \lceil K_0 K\log L\rceil
\end{equation*}
for a large constant~$K_0(d)\geq 12000/\log 3$ to be chosen below. If~$m \leq L+h$, then set
\begin{equation*}
	\ell  \coloneqq  m-h
	\qand
	n \coloneqq \ell-h\,.
\end{equation*}
Otherwise, set
\begin{equation*}
	\ell \coloneqq L
	\qand
	n \coloneqq L-h\,.
\end{equation*}
The lower bound for~$K_0$ ensures that the factors~$3^{-h}$ below give errors much smaller than~$m^{-3000}$. By increasing~$K_0$ and the constant in the threshold~$L_0(C(M+K),\alpha,\cstar,\nu)$, if necessary, we may assume that
\begin{equation*}
	n < \ell < m\,,
	\qquad
	\ell \leq L\,,
	\qquad
	\ell \geq \tfrac12 L\,,
	\qquad
	\tfrac12 L \leq r \leq 2L
	\,.
\end{equation*}
Moreover,
\begin{equation*}
	\begin{aligned}
		L-\ell &\leq (L-m)_+ + CK\log L\,,
		\qquad
		\ell-n = h \leq CK\log L\,,
		\\
		r-n &\leq (L-m)_+ + CK\log L\,,
		\qquad
		\min\{\ell,r\}-n \geq 200\log L\,.
	\end{aligned}
\end{equation*}
In particular, Proposition~\ref{p.homog.below} applies to the pairs~$(\ell,n)$ and~$(r,n)$, and~\eqref{e.localization.s.star} applies to the pair of cutoffs~$\ell$ and~$r$ on~$\cu_n$. The parameters~$\alpha$,~$M$ and~$K$ enter only through the threshold~$L_0(C(M+K),\alpha,\cstar,\nu)$, which we suppress in what follows. Also~$\ell\leq L$ in both cases, so Lemma~\ref{l.localization.average} applies with upper cutoff~$L$ and intermediate cutoff~$\ell$.

\smallskip

We first work with a generic vector~$P \in \R^{2d}$ and only specialize to the choices
\begin{equation}
	\qquad
	P_e^- \coloneqq \bfAhom_r^{-\nf 12}
	\begin{pmatrix}
		-e
		\\
		e
	\end{pmatrix}
	\qand
	P_e^+ \coloneqq \bfAhom_r^{-\nf 12}
	\begin{pmatrix}
		e
		\\
		e
	\end{pmatrix}
\end{equation}
below. Letting~$\h_z \coloneqq (\k_L - \k_{\ell})_{z+\cu_n}$, subadditivity,~\eqref{e.commute.coarse.grained.k0}, and Lemma~\ref{l.localization.average}, applied after shifting both coefficient fields by~$\h_0$, yield
\begin{align}
	P \cdot \bfA(\cu_m;\a_L+\h_0) P
	&
	\leq
	\avsum_{z \in 3^n \Zd \cap \cu_m}
	P \cdot \bfA(z + \cu_n;\a_L+\h_0) P
	\notag \\
	&\leq
	T_0(P)
	+
	\avsum_{z \in 3^n \Zd \cap \cu_m}
	\G_{-(\h_z+\h_0)} P \cdot \bfAhom_\ell(\cu_n) \G_{-(\h_z+\h_0)} P
	\label{e.new.mixing.splitting}
\end{align}
By~\eqref{e.commute.coarse.grained.k0} and~\eqref{e.G.group}, adding the constant matrix~$\h_0$ replaces the test vector~$P$ by~$\G_{-\h_0}P$ in the quadratic error from Lemma~\ref{l.localization.average}; the random matrix itself does not change. Apply~\eqref{e.localization.average.oneshot} with~$P$ ranging over a fixed~$\frac14$-net of the unit sphere in~$\R^{2d}$. The quadratic form of a symmetric matrix on such a net controls its operator norm up to a dimensional constant. Lemma~\ref{l.maximums.Gamma.s} therefore gives a nonnegative random variable~$W_{L,m,r}=\O_{\Gamma_{\nf13}}(C)$ such that, simultaneously for every~$P\in\R^{2d}$ and every constant anti-symmetric matrix~$\h_0$,
\begin{align}
	T_0(P)
	&\leq
	\nu^{-4}|\G_{-\h_0}P|^2\indc_{\{L>\ell\}}L^2 3^{-(\ell-n)}W_{L,m,r}
	\notag \\ & \qquad
	+\nu^{-4}|\G_{-\h_0}P|^2\ell L^2(m-n) 3^{-\frac d2(m-\ell)}W_{L,m,r}
	\,.
	\label{e.new.mixing.localization.error}
\end{align}
Indeed, since~$\nu\leq1$ and~$L,\ell,m-n\geq1$, we have~$\nu^{-3}L\leq\nu^{-4}L^2$ and~$\nu^{-3}(1\vee\ell)(1\vee L)(1\vee(m-n))\leq\nu^{-4}\ell L^2(m-n)$. These are exactly the two coefficient bounds used in~\eqref{e.new.mixing.localization.error}.

\smallskip
\emph{Step 1.} Fix~$e \in \Rd$ with~$|e|\leq 1$, let~$\sigma \in \{-,+\}$, and set~$P \coloneqq P_e^\sigma$. By Proposition~\ref{p.homog.below}, applied with cutoff~$r$ at scale~$n$, and using~$r-n \leq (L-m)_+ + CK\log L$ and~$r \simeq L$, we have
\begin{equation*}
	\bigl| \shom_r^{-1}\shom_r(\cu_n) - \Id \bigr|
	+
	\bigl| \shom_r \shom_{r,*}^{-1}(\cu_n) - \Id \bigr|
	\leq
	C\shom_r^{-2}\bigl( (L-m)_+ + K\log^2L\bigr)
	\,.
\end{equation*}
Also, setting~$a \coloneqq \min\{\ell,r\}$ and~$b \coloneqq \max\{\ell,r\}$, the relative comparison~\eqref{e.localization.s.star}, applied to the cutoffs~$a$ and~$b$ on~$\cu_n$, gives
\begin{equation*}
	\bigl| \s_b^{-\nf12}(\cu_n)\s_a(\cu_n)\s_b^{-\nf12}(\cu_n) - \Id \bigr|
	+
	\bigl| \s_{a,*}^{\nf12}(\cu_n)\s_{b,*}^{-1}(\cu_n)\s_{a,*}^{\nf12}(\cu_n) - \Id \bigr|
	\leq
	\O_{\Gamma_1}(CL^{-100})\,.
\end{equation*}
Here~$\min\{\ell,r\}-n\geq200\log L$. The second term in the previous display bounds the lower-right blocks: by~\eqref{e.Enaught.vs.A.and.Ahom} and Cauchy-Schwarz,~$\bigl| \shom_{r,*}(\cu_n)\shom_{\ell,*}^{-1}(\cu_n) - \Id \bigr|\leq CL^{-99}$. For the upper-left blocks, the same comparison as in the proof of Lemma~\ref{l.shomm.vs.shomell}, applied with the cutoffs~$\min\{\ell,r\}$ and~$\max\{\ell,r\}$ on~$\cu_n$, together with Proposition~\ref{p.homog.below} and~$|r-\ell|\leq(L-m)_++CK\log L$, gives
\begin{equation*}
	\bigl| \shom_r(\cu_n)\shom_\ell^{-1}(\cu_n) - \Id \bigr|
	+
	\bigl| \shom_{r,*}(\cu_n)\shom_{\ell,*}^{-1}(\cu_n) - \Id \bigr|
	\leq
	C\shom_r^{-2}\bigl( (L-m)_+ + K\log^2L\bigr)\,.
\end{equation*}
Since all matrices here are scalar multiples of~$\Id$, we may write
\begin{equation*}
	\shom_r^{-1}\shom_\ell(\cu_n)
	=
	\bigl( \shom_r^{-1}\shom_r(\cu_n) \bigr)
	\bigl( \shom_r(\cu_n)^{-1}\shom_\ell(\cu_n) \bigr)
\end{equation*}
and
\begin{equation*}
	\shom_r \shom_{\ell,*}^{-1}(\cu_n)
	=
	\bigl( \shom_r\shom_{r,*}^{-1}(\cu_n) \bigr)
	\bigl( \shom_{r,*}(\cu_n)\shom_{\ell,*}^{-1}(\cu_n) \bigr)\,,
\end{equation*}
and therefore
\begin{equation}
	\label{e.mixing.reference.compare}
	\bigl| \shom_r^{-1}\shom_\ell(\cu_n) - \Id \bigr|
	+
	\bigl| \shom_r \shom_{\ell,*}^{-1}(\cu_n) - \Id \bigr|
	\leq
	C\shom_r^{-2}\bigl( (L-m)_+ + K\log^2L\bigr)
	\,.
\end{equation}
In particular,
\begin{equation*}
	\bigl| \bfAhom_\ell^{\nf 12}(\cu_n)P \bigr|^2
	\leq
	C
	\,.
\end{equation*}
Using Proposition~\ref{p.homog.below} once more with cutoff~$\ell$ at scale~$n$, together with~\eqref{e.mixing.reference.compare} and the factorization
\begin{equation*}
	\shom_r^{-1}\shom_\ell
	=
	\bigl( \shom_r^{-1}\shom_\ell(\cu_n) \bigr)
	\bigl( \shom_\ell^{-1}\shom_\ell(\cu_n) \bigr)^{-1}\,,
\end{equation*}
we obtain
\begin{equation*}
	\bigl| \shom_r^{-1}\shom_\ell - \Id \bigr|
	\leq
	C\shom_r^{-2}\bigl( (L-m)_+ + K\log^2L\bigr)
	\qand
	\shom_\ell^{-2} \leq C\shom_r^{-2}\,.
\end{equation*}
The second bound follows from the first after increasing the threshold so that the right side is bounded by a small dimensional constant. Denote the last term on the right side of~\eqref{e.new.mixing.splitting} by
\begin{equation*}
	T_3
	\coloneqq
	\avsum_{z \in 3^n \Zd \cap \cu_m}
	\G_{-(\h_z+\h_0)} P \cdot \bfAhom_\ell(\cu_n) \G_{-(\h_z+\h_0)} P
	\,.
\end{equation*}
Since~$\bfAhom_\ell(\cu_n)$ is block diagonal, we have
\begin{equation*}
	\bfAhom_\ell(\cu_n)
	=
	\begin{pmatrix}
		\shom_\ell(\cu_n) & 0
		\\
		0 & \shom_{\ell,*}^{-1}(\cu_n)
	\end{pmatrix}
\end{equation*}
and therefore
\begin{equation*}
	P
	=
	\begin{pmatrix}
		\sigma \shom_r^{-\nf12} e
		\\
		\shom_r^{\nf12} e
	\end{pmatrix}
	\qand
	\G_{-(\h_z+\h_0)} P
	=
	\begin{pmatrix}
		\sigma \shom_r^{-\nf12} e
		\\
		-\sigma (\h_z+\h_0) \shom_r^{-\nf12} e + \shom_r^{\nf12} e
	\end{pmatrix}
	\,.
\end{equation*}
Using that~$\h_z+\h_0$ is anti-symmetric and~$\shom_\ell(\cu_n)$,~$\shom_{\ell,*}(\cu_n)$ are scalars, the cross term vanishes because~$e\cdot (\h_z+\h_0) e=0$, and we obtain
\begin{align*}
	\G_{-(\h_z+\h_0)} P \cdot \bfAhom_\ell(\cu_n) \G_{-(\h_z+\h_0)} P
	=\ &
	\shom_r^{-1}\shom_\ell(\cu_n)|e|^2
	+
	\shom_r\shom_{\ell,*}^{-1}(\cu_n)|e|^2
	+
	\shom_r^{-1}\shom_{\ell,*}^{-1}(\cu_n)|(\h_z+\h_0) e|^2
	\\ \leq\ &
	\shom_r^{-1}\shom_\ell(\cu_n)|e|^2
	+
	\shom_r\shom_{\ell,*}^{-1}(\cu_n)|e|^2
	+
	\shom_r^{-1}\shom_{\ell,*}^{-1}(\cu_n)|\h_z+\h_0|^2
	\,.
\end{align*}
Averaging over~$z$ and using~\eqref{e.mixing.reference.compare}, Jensen's inequality, and~$|\h_z+\h_0|^2\leq C|\h_z|^2+C|\h_0|^2$, we deduce that
\begin{equation*}
	T_3
	\leq
	2|e|^2
	+
	C\shom_r^{-2}\bigl( (L-m)_+ + K\log^2L\bigr)
	+
	\shom_r^{-1}\shom_{\ell,*}^{-1}(\cu_n)
	\| \k_L - \k_\ell \|_{\underline{L}^2(\cu_m)}^2
	+
	C\shom_r^{-1}\shom_{\ell,*}^{-1}(\cu_n)|\h_0|^2
	\,.
\end{equation*}
The previous comparison~\eqref{e.mixing.reference.compare} also yields~$\shom_r \shom_{\ell,*}^{-1}(\cu_n) \leq C$, and therefore
\begin{equation*}
	\shom_r^{-1}\shom_{\ell,*}^{-1}(\cu_n) \leq C\shom_r^{-2}\,.
\end{equation*}
Consequently,
\begin{equation*}
	T_3
	\leq
	2|e|^2
	+
	C\shom_r^{-2}\bigl( (L-m)_+ + K\log^2L\bigr)
	+
	C\shom_r^{-2}
	\| \k_L - \k_\ell \|_{\underline{L}^2(\cu_m)}^2
	+
	C\shom_r^{-2}|\h_0|^2
	\,.
\end{equation*}
For~$P=P_e^\sigma$, the anti-symmetry of~$\h_0$ gives~$e\cdot\h_0 e=0$, and therefore~$|\G_{-\h_0}P_e^\sigma|^2=|P_e^\sigma|^2+\shom_r^{-1}|\h_0 e|^2\leq|P_e^\sigma|^2+\shom_r^{-1}|\h_0|^2$. The crude upper bound on~$\shom_r$ and Proposition~\ref{p.sstar.lower.bound} imply that~$|P_e^\sigma|^2$ is bounded by a polynomial in~$L$. Combining this with~$\ell-n=h$,~$m-\ell \geq h$, and~$L-\ell \leq (L-m)_+ + CK\log L$, we may increase~$K_0(d)$ and the threshold in~$L_0(C(M+K),\alpha,\cstar,\nu)$ so that~\eqref{e.new.mixing.localization.error} gives
\begin{equation*}
	T_0(P_e^\sigma)
	\leq
	m^{-3000}\bigl(1+\shom_r^{-1}|\h_0|^2\bigr)W_{L,m,r}
	\,.
\end{equation*}
Here the first term in~\eqref{e.new.mixing.localization.error} is present only in the case~$m\leq L+h$, where~$m\simeq L$ and~$3^{-h}\leq m^{-12000}$; in the case~$\ell=L$, that term vanishes, and the remaining term contains~$3^{-\frac d2(m-L)}$ with~$m-L\geq h$. By~\eqref{e.kmn.bounds} with~$p=2$ and the bound~$L-\ell \leq (L-m)_+ + CK\log L$, we have
\begin{equation*}
	\| \k_L - \k_\ell \|_{\underline{L}^2(\cu_m)}^2
	\leq
	\O_{\Gamma_1}
	\bigl(
	C \bigl( (L-m)_+ + K\log L\bigr)
	\bigr)
	\,.
\end{equation*}
Combining the previous three displays with~\eqref{e.new.mixing.splitting}, we obtain
\begin{align*}
	P \cdot \bfA(\cu_m;\a_L+\h_0) P
	&\leq
	2|e|^2
	+
	C\shom_r^{-2}\bigl( |\h_0|^2+(L-m)_+ + K\log^2L\bigr)
	\notag \\ & \qquad
	+
	\O_{\Gamma_1}
	\bigl(
	C\shom_r^{-2}\bigl( (L-m)_+ + K\log L\bigr)
	\bigr)
	\notag \\ & \qquad
	+
	m^{-3000}\bigl(1+\shom_r^{-1}|\h_0|^2\bigr)W_{L,m,r}
	\,.
\end{align*}
Let
\begin{align*}
	R
	&\coloneqq
	C\shom_r^{-2}\bigl( |\h_0|^2+(L-m)_+ + K\log^2L\bigr)
	+
	\O_{\Gamma_1}
	\bigl(
	C\shom_r^{-2}\bigl( (L-m)_+ + K\log L\bigr)
	\bigr)
	\notag \\ & \qquad
	+
	m^{-3000}\bigl(1+\shom_r^{-1}|\h_0|^2\bigr)W_{L,m,r}
	\,.
\end{align*}
For~$\sigma=-$, by~\eqref{e.Jaas.matform} we have
\begin{equation*}
	J(\cu_m,\shom_r^{-\nf 12}e,\shom_r^{\nf 12}e;\a_L+\h_0)
	=
	\frac12 P_e^- \cdot \bfA(\cu_m;\a_L+\h_0) P_e^- - |e|^2
	\leq
	R\,,
\end{equation*}
while for~$\sigma=+$,
\begin{equation*}
	J^*(\cu_m,\shom_r^{-\nf 12}e,\shom_r^{\nf 12}e;\a_L+\h_0)
	=
	\frac12 P_e^+ \cdot \bfA(\cu_m;\a_L+\h_0) P_e^+ - |e|^2
	\leq
	R\,.
\end{equation*}
Since~$\shom_r$ is scalar, the previous two displays and~\eqref{e.mathcal.E.bfJ.scalar.max} yield
\begin{equation*}
	\max_{|\eta|=1}
	\bfJ\bigl(
	\cu_m\,,
	\bfAhom_r^{-\nf 12}\eta\,,
	\bfAhom_r^{\nf 12}\eta
	\,;\,\a_L+\h_0
	\bigr)
	\leq
	R\,.
\end{equation*}
This proves the claim.
\end{proof}

We next record a localized version of Lemma~\ref{l.new.mixing.parameterized}, in which the infrared cutoff is raised from~$L'$ to~$L$ after subtracting the nearly constant antisymmetric tail~$\k_L-\k_{L'}$. The estimates are stated directly in terms of the homogenization error~\eqref{e.mathcal.E.def.new} from Subsection~\ref{ss.cg.ellipticity}; the passage from the scalar~$J,J^*$ bounds to~$\bfJ$ uses~\eqref{e.mathcal.E.bfJ.scalar.max}.

\begin{proposition}[Bounds for~$\mathcal{E}$]
\label{p.new.mixing.attempt}
There exists a constant~$C(d)<\infty$ such that, for every~$s \in (0,1]$, every~$K \in [C,\infty)$, and every~$m,n,h \in \N$ satisfying
\begin{equation*}
	m \geq L_0(Cs^{-1}K,\nf12,\cstar,\nu)\,,
	\quad
	h \coloneqq \left\lceil K \log m \right\rceil
	\qand
	m-h \leq n \leq m \,,
\end{equation*}
then the following estimate holds for every~$z \in 3^{n-3}\Zd$ such that~$z+\cu_n\subseteq\cu_m$:
\begin{align}
	\lefteqn{
	\sup_{L\in\N \cap [m-Cs^{-1}h,\infty)}
	\mathcal{E}_{s,2}\bigl(z+\cu_n\,; \a_L - (\k_L)_{\cu_m} , \shom_{m}\bigr)
	+
	\mathcal{E}_{s,2}\bigl(z+\cu_n\,; \a - (\k)_{\cu_m} , \shom_{m}\bigr)
	} \qquad &
	\notag \\ &
	\leq
	C s^{-\nf12} K^{\nf12} \shom_{m}^{-1}\log m
	+
	\O_{\Gamma_2}\bigl( C s^{-1} K^{\nf12} \shom_{m}^{-1}\log^{\nf12}m \bigr)
	+
	\O_{\Gamma_{\nf12}}\bigl( C m^{-1000}\bigr) \,,
	\label{e.new.mixing.attempt}
\end{align}
where~$\a - (\k)_{\cu_m}$ denotes the uniform limit, on bounded cubes, of~$\a_L - (\k_L)_{\cu_m}$; this convergence is justified at the end of the proof.
\end{proposition}
\begin{proof}
Let
\begin{equation*}
	\bar h  \coloneqq  \left\lceil C_0 s^{-1}\log m \right\rceil
	\,,
\end{equation*}
for a large constant~$C_0(d)$ to be selected below. We enlarge the constant in the lower threshold~$L_0(Cs^{-1}K,\nf12,\cstar,\nu)$, if necessary, so that~\eqref{e.sL.growth} applies to every scale used below which is comparable to~$m$. Increasing the same threshold once more, we also assume that
\begin{equation*}
	\bar h+C_1s^{-1}h+h \leq \frac1{10}m\,,
\end{equation*}
and that every scale in~$[m-C(\bar h+s^{-1}h),m+h]$ is above the deterministic thresholds invoked below. We will use this enlargement, through~\eqref{e.Lnaught.def}, to absorb the finitely many polynomial factors in~$s^{-1}$,~$K$, and~$\nu^{-1}$ into the displayed~$m^{-3000}$ error bounds. We first prove the estimate for finite cutoffs. Fix~$L\in \N$ satisfying the cutoff range on the left side of~\eqref{e.new.mixing.attempt}. By increasing the dimensional constant in that range, it is enough to consider cutoffs satisfying
\begin{equation*}
	m \leq L + C_1 s^{-1} h\,.
\end{equation*}
Fix~$x\in 3^{n-3}\Zd$ such that~$x+\cu_n\subseteq\cu_m$. We keep~$x$ arbitrary in the centered estimate. Set
\begin{equation*}
	m'  \coloneqq  \min\{L,m+h\}
	\qand
	\tilde \a_{L,m}  \coloneqq  \a_L - (\k_L-\k_{m'})_{\cu_m}\,.
\end{equation*}
Then
\begin{equation*}
	m-C_1 s^{-1} h \leq m' \leq m+h\,.
\end{equation*}
Moreover,
\begin{equation*}
	m \geq L-h \implies m'=L
	\qand
	m < L-h \implies m'=m+h\,.
\end{equation*}
By~\eqref{e.mathcal.E.def.new} with~$q=2$ and recalling that~$\css{2s}=1-3^{-2s}$,
\begin{equation*}
	\mathcal{E}_{s,2}(x+\cu_n\,; \tilde \a_{L,m} , \shom_{m})^2
	=
	\css{2s}
	\sum_{l=-\infty}^{n}
	3^{-2s(n-l)} X_l
\end{equation*}
where
\begin{equation*}
	X_l
	\coloneqq
	\max_{y \in (x+3^l \Zd) \cap (x+\cu_n)}
	\max_{|\eta| =1}
	\bfJ\bigl(y+\cu_l,\bfAhom_{m}^{-\nf 12} \eta,\bfAhom_{m}^{\nf 12} \eta \, ; \tilde \a_{L,m} \bigr)
	\,.
\end{equation*}
We split the sum into the near scales~$l \in \N \cap [n-\bar h,n]$ and the deep scales~$l < n-\bar h$.

\smallskip
\emph{Step 1.} Near scales. In this step we assume that~$m \leq L + C_1 s^{-1} h$ and~$n-\bar h \leq l \leq n$. Fix~$l \in \N \cap [n-\bar h,n]$ and~$y \in (x+3^l \Zd) \cap (x+\cu_n)$. Since
\begin{equation*}
	n-\bar h \leq l \leq n \leq m
	\qand
	m-C_1 s^{-1} h \leq m' \leq m+h\,,
\end{equation*}
all scales~$m'$,~$m$, and~$l$ lie in the interval~$[m-C(\bar h+s^{-1}h),m+h]$. Since~$\bar h+s^{-1}h \leq C s^{-1}K \log m$, after increasing the threshold in~$L_0(Cs^{-1}K,\nf12,\cstar,\nu)$ if necessary, we have
\begin{equation*}
	l \geq m' - C s^{-1}K (m')^{\nf12}\log^3m'
	\qand
	|m'-m| \leq C s^{-1}K\log m'\,.
\end{equation*}
Therefore, by stationarity and Lemma~\ref{l.new.mixing.parameterized}, applied with~$\alpha=\nf12$ and with the lemma parameters~$M_{\rm lem}=K_{\rm lem}=Cs^{-1}K$, and
\begin{equation*}
	(L,m,r)=(m',l,m)\,,
\end{equation*}
and using~$l \simeq m$ in this range, we obtain
\begin{align}
	\label{e.new.mixing.attempt.local.base}
	\lefteqn{
	\max_{|\eta|=1}
	\bfJ\bigl(y+\cu_l,\bfAhom_m^{-\nf 12} \eta,\bfAhom_m^{\nf 12} \eta \, ; \a_{m'} \bigr)
	} \qquad &
	\notag \\ &
	\leq
	C s^{-1}K \shom_{m}^{-2}\log^2 m
	+
	\O_{\Gamma_1}\bigl( C s^{-1}K \shom_{m}^{-2}\log m \bigr)
	+
	\O_{\Gamma_{\nf13}}\bigl( C m^{-3000}\bigr)
	\,.
\end{align}

\smallskip

If~$L=m'$, equivalently~$m \geq L-h$, then~$\tilde \a_{L,m}=\a_{m'}$, and so~\eqref{e.new.mixing.attempt.local.base} already gives the desired bound for~$X_l$. Assume therefore that~$L>m'$, equivalently~$m < L-h$, so that~$m'=m+h$. In particular, in this subcase the reference field~$\a_{m'}=\a_{m+h}$ is fixed, and the dependence on the upper cutoff enters only through the tail~$\k_L-\k_{m+h}$. In this subcase, the following tail estimate is uniform in the upper cutoff. Indeed, by the generalized triangle inequality and~\eqref{e.nabla.kmn.Linfty},
\begin{align*}
	\sup_{L \geq m+h}
	\| \nabla (\k_L-\k_{m+h}) \|_{L^\infty(\cu_{m+3})}
	&\leq
	\sum_{r=m+h+1}^{\infty}
	\| \nabla(\k_r-\k_{r-1}) \|_{L^\infty(\cu_{m+3})}
	\\ &\leq
	\O_{\Gamma_2}\biggl( C\sum_{r=m+h+1}^{\infty} 3^{-r}\biggr)
	\leq
	\O_{\Gamma_2}\bigl( C3^{-(m+h)}\bigr)\,.
\end{align*}
Since~$y+\cu_l\subseteq x+\cu_n\subseteq\cu_m$, the fundamental theorem of calculus for averages yields
\begin{equation*}
	\sup_{L \geq m+ h}
	\| (\k_L-\k_{m'})-(\k_L-\k_{m'})_{\cu_m} \|_{L^\infty(y+\cu_l)}
	\leq
	\O_{\Gamma_2}\bigl( C 3^{-h} \bigr)
	\,,
\end{equation*}
by the previous display and the identity~$m'=m+h$. Therefore, Lemma~\ref{l.localization.A}, applied with~$\a=\a_{m'}$, perturbation field~$(\k_L-\k_{m'})-(\k_L-\k_{m'})_{\cu_m}$, and domain~$y+\cu_l$, yields
\begin{equation*}
	\sup_{L \geq m+h}
	\Bigl|
	\bfA_{m'}^{-\nf12}(y+\cu_l)
	\bfA(y+\cu_l;\tilde \a_{L,m})
	\bfA_{m'}^{-\nf12}(y+\cu_l)
	- \Itwod
	\Bigr|
	\leq
	\O_{\Gamma_2}\bigl( C\nu^{-1} 3^{-h} \bigr)
	+
	\O_{\Gamma_1}\bigl( C\nu^{-2} 3^{-2h} \bigr)\,.
\end{equation*}
For~$e \in \Rd$, set
\begin{equation*}
	Q_e^- \coloneqq \begin{pmatrix} -\shom_m^{-\nf12} e \\ \shom_m^{\nf12} e \end{pmatrix}
	\qand
	Q_e^+ \coloneqq \begin{pmatrix} \shom_m^{-\nf12} e \\ \shom_m^{\nf12} e \end{pmatrix}\,.
\end{equation*}
By Lemma~\ref{l.bfAm.ellip}, stationarity, and the bounds~$m' \leq m+h \leq 2m$ and~\eqref{e.sL.growth},
\begin{equation*}
	\max_{\sigma \in \{-,+\}} \max_{|e|=1}
	\bigl| \bfA_{m'}^{\nf12}(y+\cu_l) Q_e^\sigma \bigr|^2
	\leq
	\O_{\Gamma_1}\bigl( C\nu^{-3} m \bigr)\,.
\end{equation*}
Here Lemma~\ref{l.bfAm.ellip} is used with the block matrix bound~$\bfE_{m'}$, so the two components of~$Q_e^\sigma$ are weighted separately; the growth bound~\eqref{e.sL.growth} then gives the displayed~$C\nu^{-3}m$ bound. This extra polynomial power of~$\nu^{-1}$ is absorbed below by the lower threshold. Since all of the preceding bounds are quadratic in~$e$, the same estimates remain valid if~$|e| \leq \sqrt2$, after adjusting the constant~$C$. Taking the lower bound~$K\geq C(d)$ in the proposition large enough, in particular larger than~$3001/\log 3$, gives~$3^{-h}\leq m^{-3001}$. Enlarging the threshold in~$L_0$ if necessary to absorb the factors~$\nu^{-k}$, the matrix comparison above,~\eqref{e.Jaas.matform}, Lemma~\ref{l.o.gamma2.mult}, and~\eqref{e.powerofGammasigma} yield
\begin{equation*}
	\sup_{L \geq m+h}
	\max_{|e|\leq \sqrt2}
	\Bigl|
	J\bigl(y+\cu_l,\shom_m^{-\nf 12} e,\shom_m^{\nf 12} e \, ; \tilde \a_{L,m} \bigr) - J\bigl(y+\cu_l,\shom_m^{-\nf 12} e,\shom_m^{\nf 12} e \, ; \a_{m'} \bigr)
	\Bigr|
	\leq
	\O_{\Gamma_{\nf13}}\bigl( C m^{-3000}\bigr)\,,
\end{equation*}
and the same bound holds with~$J$ replaced by~$J^*$. Since~$\shom_m$ is scalar and the change of variables in~\eqref{e.mathcal.E.bfJ.scalar} is orthogonal, the previous display and its analogue with~$J^*$ imply
\begin{equation*}
	\sup_{L \geq m+h}
	\max_{|\eta|=1}
	\Bigl|
	\bfJ\bigl(y+\cu_l,\bfAhom_m^{-\nf 12} \eta,\bfAhom_m^{\nf 12} \eta \, ; \tilde \a_{L,m} \bigr)
	-
	\bfJ\bigl(y+\cu_l,\bfAhom_m^{-\nf 12} \eta,\bfAhom_m^{\nf 12} \eta \, ; \a_{m'} \bigr)
	\Bigr|
	\leq
	\O_{\Gamma_{\nf13}}\bigl( C m^{-3000}\bigr)\,.
\end{equation*}
Combined with~\eqref{e.new.mixing.attempt.local.base}, this yields, for the fixed~$y$,
\begin{align*}
	\lefteqn{
	\sup_{L \geq m+h}
	\max_{|\eta|=1}
	\bfJ\bigl(y+\cu_l,\bfAhom_m^{-\nf 12} \eta,\bfAhom_m^{\nf 12} \eta \, ; \tilde \a_{L,m} \bigr)
	} \qquad &
	\notag \\ &\leq
	C s^{-1}K \shom_{m}^{-2}\log^2 m
	+
	\O_{\Gamma_1}\bigl( C s^{-1}K \shom_{m}^{-2}\log m \bigr)
	+
	\O_{\Gamma_{\nf13}}\bigl( C m^{-3000}\bigr)
	\,.
\end{align*}
Together with the fixed-$L$ case~$L=m'$, this gives the same pointwise bound in all cases, with the displayed supremum retained in the subcase~$L>m+h$. Since~$|(x+3^l\Zd) \cap (x+\cu_n)| \leq C 3^{d(n-l)}$ and, in this range,~$0\leq n-l \leq \bar h$, Lemma~\ref{l.maximums.Gamma.s}, together with the direct one-point bound when~$l=n$, implies
\begin{align*}
	\css{2s}3^{-2s(n-l)} X_l
	&\leq
	C \css{2s} 3^{-2s(n-l)} s^{-1}K \shom_{m}^{-2}\log^2 m
	\notag \\ &\qquad
	+
	\O_{\Gamma_1}\bigl( C \css{2s}(1+n-l) 3^{-2s(n-l)} s^{-1}K \shom_{m}^{-2}\log m \bigr)
	\notag \\ &\qquad
	+
	\O_{\Gamma_{\nf13}}\bigl( C \css{2s}(1+n-l)^3 3^{-2s(n-l)} m^{-3000}\bigr)\,.
\end{align*}
Summing over~$l \in \N \cap [n-\bar h,n]$ and using
\begin{equation*}
	\css{2s}\sum_{r=0}^\infty 3^{-2sr}=1\,,
	\qquad
	\css{2s}\sum_{r=0}^\infty (1+r)\,3^{-2sr}\leq C s^{-1}\,,
	\qquad
	\css{2s}\sum_{r=0}^\infty (1+r)^3 3^{-2sr}\leq C s^{-3}\,,
\end{equation*}
we deduce that
\begin{equation}
	\label{e.new.mixing.attempt.near}
	\css{2s}\sum_{l=n-\bar h}^{n} 3^{-2s(n-l)} X_l
	\leq
	C s^{-1}K \shom_{m}^{-2}\log^2 m
	+
	\O_{\Gamma_1}\bigl( C s^{-2}K \shom_{m}^{-2}\log m \bigr)
	+
	\O_{\Gamma_{\nf13}}\bigl( C m^{-2000}\bigr)\,.
\end{equation}
In the subcase~$L>m+h$, the same estimate holds with the left side replaced by its supremum over~$L\geq m+h$. Here we also used~$\css{2s}\geq c s$ and the lower bound~$m \geq L_0(Cs^{-1}K,\nf12,\cstar,\nu)$ to bound the resulting error by the displayed~$m^{-2000}$ remainder.

\smallskip
\emph{Step 2.} Deep scales. In this step we assume that~$m \leq L + C_1 s^{-1} h$ and~$l<n-\bar h$. Fix~$l < n-\bar h$ and~$y \in (x+3^l \Zd) \cap (x+\cu_n)$. In the subcase~$L>m+h$, the perturbative comparison used in Step~1 applies verbatim: the convex hull of~$y+\cu_l$ and~$\cu_m$ is contained in~$\cu_m$, and now~$m'=m+h$ and~$m'-l\geq h$. Thus
\begin{equation*}
	\sup_{L \geq m+h}
	\max_{|\eta|=1}
	\Bigl|
	\bfJ\bigl(y+\cu_l,\bfAhom_m^{-\nf 12} \eta,\bfAhom_m^{\nf 12} \eta \, ; \tilde \a_{L,m} \bigr)
	-
	\bfJ\bigl(y+\cu_l,\bfAhom_m^{-\nf 12} \eta,\bfAhom_m^{\nf 12} \eta \, ; \a_{m'} \bigr)
	\Bigr|
	\leq
	\O_{\Gamma_{\nf13}}\bigl( C m^{-3000}\bigr)
	\,.
\end{equation*}
If instead~$L=m'$, then~$\tilde \a_{L,m}=\a_{m'}$, and the same bound is trivial. We also have
\begin{equation*}
	\max_{|\eta|=1}
	\bfJ\bigl(y+\cu_l,\bfAhom_m^{-\nf 12} \eta,\bfAhom_m^{\nf 12} \eta \, ; \a_{m'} \bigr)
	\leq
	\O_{\Gamma_1}\bigl( C\nu^{-3}m \bigr)\,.
\end{equation*}
Indeed, this follows from Lemma~\ref{l.bfAm.ellip}, stationarity, the Step~1 bound on~$\bigl| \bfA_{m'}^{\nf12}(y+\cu_l) Q_e^\sigma \bigr|^2$, and~\eqref{e.Jaas.matform}--\eqref{e.mathcal.E.bfJ.scalar.max}. Combining the previous two displays, we obtain
\begin{equation*}
	\max_{|\eta|=1}
	\bfJ\bigl(y+\cu_l,\bfAhom_m^{-\nf 12} \eta,\bfAhom_m^{\nf 12} \eta \, ; \tilde \a_{L,m} \bigr)
	\leq
	\O_{\Gamma_1}\bigl( C\nu^{-3}m \bigr)
	+
	\O_{\Gamma_{\nf13}}\bigl( C m^{-3000}\bigr)\,.
\end{equation*}
Applying Lemma~\ref{l.maximums.Gamma.s} to the family over~$y \in (x+3^l \Zd) \cap (x+\cu_n)$, and using~$|(x+3^l\Zd) \cap (x+\cu_n)| \leq C3^{d(n-l)}$, we obtain
\begin{equation*}
	X_l
	\leq
	\O_{\Gamma_1}\bigl( C(n-l)\nu^{-3}m \bigr)
	+
	\O_{\Gamma_{\nf13}}\bigl( C(n-l)^3 m^{-3000}\bigr)\,.
\end{equation*}
In particular, in the subcase~$L>m+h$ we have
\begin{equation*}
	\sup_{L \geq m+h}
	X_l
	\leq
	\O_{\Gamma_1}\bigl( C(n-l)\nu^{-3}m \bigr)
	+
	\O_{\Gamma_{\nf13}}\bigl( C(n-l)^3 m^{-3000}\bigr)\,.
\end{equation*}
Since~$l < n-\bar h$, we have~$n-l \geq \bar h$. Therefore
\begin{equation*}
	\sum_{l=-\infty}^{n-\bar h-1}
	\css{2s}3^{-2s(n-l)} X_l
	\leq
	\O_{\Gamma_1}\bigl( C\nu^{-3}m \css{2s}\sum_{r=\bar h}^{\infty} r\,3^{-2sr} \bigr)
	+
	\O_{\Gamma_{\nf13}}\bigl( C m^{-3000} \css{2s}\sum_{r=\bar h}^{\infty} r^3\,3^{-2sr} \bigr)\,.
\end{equation*}
Moreover,
\begin{equation*}
	\css{2s}\sum_{r=\bar h}^{\infty} r\,3^{-2sr}
	+
	\css{2s}\sum_{r=\bar h}^{\infty} r^3\,3^{-2sr}
	\leq
	C \css{2s}^{-3} \bar h^3 3^{-2s\bar h}\,.
\end{equation*}
Since~$\bar h=\lceil C_0 s^{-1}\log m \rceil$ and~$\css{2s}\geq cs$ for~$s \in (0,1]$, we may choose~$C_0(d)$ sufficiently large. This fixes the exponential power coming from~$3^{-2s\bar h}$; after enlarging the threshold in~$L_0(Cs^{-1}K,\nf12,\cstar,\nu)$ if necessary, the remaining powers of~$s^{-1}$ and~$\nu^{-1}$ are absorbed, and we may assume that
\begin{equation*}
	\css{2s}^{-3}\nu^{-3} m \bar h^3 3^{-2s\bar h}
	+
	\css{2s}^{-3} \bar h^3 3^{-2s\bar h}
	\leq
	m^{-2000}\,.
\end{equation*}
Therefore
\begin{equation}
	\label{e.new.mixing.attempt.deep}
	\css{2s}\sum_{l=-\infty}^{n-\bar h-1}
	3^{-2s(n-l)} X_l
	\leq
	\O_{\Gamma_{\nf13}}\bigl( C m^{-2000}\bigr)\,.
\end{equation}
In the subcase~$L>m+h$, the same estimate holds with the left side replaced by its supremum over~$L\geq m+h$.

\smallskip
\emph{Step 3.} Centering. We pass from the auxiliary field to the fully centered field. Since
\begin{equation*}
	\a_L-(\k_L)_{\cu_m}
	=
	\tilde\a_{L,m}-(\k_{m'})_{\cu_m}\,,
\end{equation*}
we repeat Steps~1 and~2 with the constant shift~$-(\k_{m'})_{\cu_m}$ in both coefficient fields being compared. We first control this shift uniformly over the possible values of~$m'$. For~$m'\leq m$, the triangle inequality and the geometric decay in~\eqref{e.jk.spatialavg} bound all the averages~$(\k_{m'})_{\cu_m}$ by a common~$\O_{\Gamma_2}(C)$ random variable. For~$m\leq m'\leq m+h$, write~$(\k_{m'})_{\cu_m}=(\k_m)_{\cu_m}+\sum_{k=m+1}^{m'}(\mathbf j_k)_{\cu_m}$ and apply Proposition~\ref{p.concentration} and L\'evy's maximal inequality to the independent symmetric increments. Thus
\begin{equation*}
	\sup_{m'\in\N\cap[m-C_1s^{-1}h,m+h]}|(\k_{m'})_{\cu_m}|
	\leq
	\O_{\Gamma_2}\bigl(Cs^{-\nf12}K^{\nf12}\log^{\nf12}m\bigr)\,.
\end{equation*}
The explicit quadratic shift term in Lemma~\ref{l.new.mixing.parameterized} is therefore bounded by~$\O_{\Gamma_1}(Cs^{-1}K\shom_m^{-2}\log m)$, which is absorbed into the second term of the combined estimate below since~$\css{2s}^{-1}\geq1$.

\smallskip

For the near scales~$n-\bar h\leq l\leq n$, apply the translated form of Lemma~\ref{l.new.mixing.parameterized} with~$(L,m,r)=(m',l,m)$ and~$\h_0=-(\k_{m'})_{\cu_m}$, for each~$m'$ in the preceding interval. The lemma holds simultaneously for every~$\h_0$, so this random substitution is legitimate. Its leading~$\Gamma_1$ term is controlled by the squared~$L^2$ norm of a sum of the fields~$\mathbf j_k$, whose endpoints lie in a common interval of at most~$Cs^{-1}K\log m$ indices. Writing each such sum as a difference of two partial sums, L\'evy's maximal inequality and the~$L^2$ case of~\eqref{e.kmn.bounds} give the same bound uniformly over~$m'$. For the deep scales~$l<n-\bar h$, apply the same maximal inequality to~$\mathbf j_k-(\mathbf j_k)_{\cu_m}$ and use the deterministic block bound in the proof of Lemma~\ref{l.bfAm.ellip}. Together with the preceding bound on the averages, this gives the same pointwise~$\O_{\Gamma_1}(C\nu^{-3}m)$ bound as in Step~2, uniformly over~$m'$. Its weighted sum is absorbed as in~\eqref{e.new.mixing.attempt.deep}.

\smallskip

It remains to bound the localization remainders after taking the supremum over cutoffs. There are at most~$C(1+s^{-1}K\log m)$ possible values of~$m'$. By Lemma~\ref{l.maximums.Gamma.s}, the maximum over these cutoffs enlarges a~$\Gamma_{1/3}$ bound by at most~$C\log^3(2+s^{-1}K\log m)$, while the maximum over the cubes~$y+\cu_l$ contributes at most~$C(1+n-l)^3$. The sum of the spatial cubic factors with the weights in~\eqref{e.new.mixing.attempt.near} and~\eqref{e.new.mixing.attempt.deep} is bounded by~$Cs^{-3}$. Starting from the pointwise~$m^{-3000}$ remainders, their maxima and weighted sums are therefore~$\O_{\Gamma_{\nf13}}(Cm^{-2500})$, after increasing the lower threshold. For~$L>m+h$, make the same constant shift in the matrix comparisons of Steps~1 and~2; their tail estimates already hold uniformly over~$L\geq m+h$.

\smallskip

The localization remainders carry the factor~$1+\shom_m^{-1}|(\k_{m'})_{\cu_m}|^2$, which has~$\Gamma_1$ control by the displayed bound on the averages. Lemma~\ref{l.o.gamma2.mult} and~\eqref{e.multGammasig} therefore turn the~$\Gamma_{1/3}$ remainder into~$\O_{\Gamma_{\nf14}}(Cm^{-2000})$ after enlarging the lower threshold. Consequently,~\eqref{e.new.mixing.attempt.near} and~\eqref{e.new.mixing.attempt.deep} hold for the fully centered field~$\a_L-(\k_L)_{\cu_m}$, with their left sides replaced by the corresponding suprema over~$L\geq m-C_1s^{-1}h$; the last random term in both estimates may be bounded by the same~$\O_{\Gamma_{\nf14}}(Cm^{-2000})$ random variable.

\smallskip

Combining~\eqref{e.new.mixing.attempt.near} and~\eqref{e.new.mixing.attempt.deep} for the fully centered field, we obtain
\begin{align*}
	\lefteqn{
	\sup_{L\in\N \cap [m-C_1s^{-1}h,\infty)}
	\mathcal{E}_{s,2}\bigl(x+\cu_n\,;\,\a_L-(\k_L)_{\cu_m},\shom_m\bigr)^2
	} \qquad &
	\notag \\ &\leq
	C s^{-1}K \shom_{m}^{-2}\log^2 m
	+
	\O_{\Gamma_1}\bigl( C \css{2s}^{-1} s^{-1}K \shom_{m}^{-2}\log m \bigr)
	+
	\O_{\Gamma_{\nf14}}\bigl( C m^{-2000}\bigr)
	\,.
\end{align*}
Using~$\sqrt{x+y+z} \leq \sqrt x + \sqrt y + \sqrt z$ for~$x,y,z \geq 0$, together with~\eqref{e.powerofGammasigma}, gives
\begin{align*}
	\lefteqn{
	\sup_{L\in\N \cap [m-C_1s^{-1}h,\infty)}
	\mathcal{E}_{s,2}\bigl(x+\cu_n\,;\,\a_L-(\k_L)_{\cu_m},\shom_m\bigr)
	} \qquad &
	\notag \\ &\leq
	C s^{-\nf12} K^{\nf12} \shom_{m}^{-1}\log m
	+
	\O_{\Gamma_2}\bigl( C \css{2s}^{-\nf12} s^{-\nf12} K^{\nf12} \shom_{m}^{-1}\log^{\nf12}m \bigr)
	+
	\O_{\Gamma_{\nf12}}\bigl( C m^{-1000}\bigr)
	\,.
\end{align*}
Since~$\css{2s}=1-3^{-2s}\geq c s$ for~$s \in (0,1]$, we have~$\css{2s}^{-\nf12}s^{-\nf12}\leq Cs^{-1}$. Thus the last display with arbitrary~$x=z\in 3^{n-3}\Zd$ satisfying~$z+\cu_n\subseteq\cu_m$ gives the finite-cutoff supremum in~\eqref{e.new.mixing.attempt}.

\smallskip

\emph{Step 4.} Limiting coefficient. Fix~$z\in 3^{n-3}\Zd$ with~$z+\cu_n\subseteq\cu_m$. The cutoff approximation Lemma~\ref{l.cutoff.approximation} applies on~$\cu_{m+1}$, and the centering by~$\cu_m$ is harmless. Indeed, by~\ref{a.j.reg},~\eqref{e.k(n).reg}, and Borel-Cantelli,~$\sum_{\ell=0}^\infty\|\nabla\mathbf j_\ell\|_{L^\infty(\cu_{m+1})}<\infty$ almost surely. The fundamental theorem of calculus for averages, applied on the convex hull of~$\cu_m$ and~$\cu_{m+1}$, then gives
\begin{equation*}
	\sum_{\ell=0}^\infty
	\|\mathbf j_\ell-(\mathbf j_\ell)_{\cu_m}\|_{W^{1,\infty}(\cu_{m+1})}
	<\infty
	\,.
\end{equation*}
Hence, almost surely,
\begin{equation*}
	\a_L-(\k_L)_{\cu_m}
	\longrightarrow
	\a-(\k)_{\cu_m}
	\quad\mbox{in }W^{1,\infty}(\cu_{m+1})\,.
\end{equation*}
Every cube appearing in the definition of~$\mathcal E_{s,2}(z+\cu_n;\cdot,\shom_m\Id)$ is contained in~$\cu_{m+1}$. For each such cube, the matrix comparison~\eqref{e.localization.A} and the same energy estimate as in the proof of~\eqref{e.localization.minimizers} give convergence of the corresponding coarse-grained matrices and variational quantities. Applying Fatou's lemma to the nonnegative weighted series in the definition of~$\mathcal E_{s,2}$ gives
\begin{equation*}
	\mathcal E_{s,2}\bigl(z+\cu_n;\a-(\k)_{\cu_m},\shom_m\bigr)
	\leq
	\liminf_{L\to\infty}
	\mathcal E_{s,2}\bigl(z+\cu_n;\a_L-(\k_L)_{\cu_m},\shom_m\bigr)\,.
\end{equation*}
The limiting term is therefore bounded by the same right side as the finite-cutoff supremum. The two terms on the left of~\eqref{e.new.mixing.attempt} are thus controlled by the same deterministic and random bounds; increasing~$C$ and using Lemma~\ref{l.Gamma.sigma.triangle} yields~\eqref{e.new.mixing.attempt}.
\end{proof}

By stationarity, the estimates of Proposition~\ref{p.new.mixing.attempt} also hold with~$\cu_n$ and~$\cu_m$ replaced by any common deterministic translate~$x+\cu_n$ and~$x+\cu_m$.

\begin{proof}[Proof of Proposition~\ref{p.minimal.scales}]
Fix~$\expon \in (0,\nf12)$,~$\delta,s\in(0,1]$ and~$M \in [1,\infty)$, and let~$\hat L_2$ be as in the statement. Let
\begin{equation*}
	K  \coloneqq  C_0 Ms^{-1}
\end{equation*}
for a large constant~$C_0(d)\geq 1$ to be selected below. After enlarging the constant in the definition of~$\hat L_2$, if necessary, we may assume that
\begin{equation}
	\label{e.new.mixing.minscale.thresholds}
	\begin{aligned}
		\hat L_2
		&\geq
		L_0(C\expon^{-1}\delta^{-2}s^{-2}K^2,1-\expon,\cstar,\nu)
		\vee
		L_0(Cs^{-1}K,\nf12+\expon,\cstar,\nu)
		\,,
		\\
		m \geq \hat L_2
		&\implies
		\left\{
		\begin{aligned}
			& Cs^{-\nf12}K^{\nf12} \leq \frac14 \delta m^{\expon}\,,
			\\
			& \delta\shom_m^{-1}m^{\expon}\log m\leq 1\,,
			\\
			& CK\log m \leq c\delta^2s^2K^{-1}m^{2\expon}\,,
			\\
			& \delta^{\nf12}\shom_m^{-\nf12}
			m^{500+\frac12\expon}\log^{\nf12}m
			\geq
			Cm^{100\expon}
		\end{aligned}
		\right.
		\,.
	\end{aligned}
\end{equation}
Here the first lower bound in~\eqref{e.new.mixing.minscale.thresholds} supplies the first and third lines in the implication, while the second supplies, a fortiori, the mixing threshold~$L_0(Cs^{-1}K,\nf12,\cstar,\nu)$ and the power~$m^{\nf12-\expon}$ needed in the second line. After one further deterministic enlargement,~\eqref{e.sL.growth} supplies the fourth line. For each~$j\in\N$, write~$h_j\coloneqq \lceil K\log j\rceil$. For~$j=m$, this agrees with the window size in the statement after fixing the dimensional constant~$C_0$.

Fix~$m\geq \hat L_2$. We first estimate the probability that the desired bound fails at this fixed scale for one of the coefficient fields in the statement. For~$m-h_m\leq n\leq m$ and for~$z\in 3^{n-3}\Zd$ satisfying~$z+\cu_n\subseteq\cu_m$, the centered estimate~\eqref{e.new.mixing.attempt}, together with stationarity, gives the desired bound outside an event with the corresponding stretched-Gaussian tail.

\smallskip

Let~$\mathcal G_m$ be the event that, for every~$n\in\N\cap[m-h_m,m]$, every~$z\in 3^{n-3}\Zd$ such that~$z+\cu_n\subseteq\cu_m$, and every coefficient field~$\a' \in \bigl\{\a_L-(\k_L)_{\cu_m}: L\in\N,\ L\geq m\bigr\} \cup \bigl\{\a-(\k)_{\cu_m}\bigr\}$,
\begin{equation*}
	\mathcal E_{s,2}\bigl(z+\cu_n;\a',\shom_m\bigr)
	\leq
	\delta\shom_m^{-1}m^\expon\log m\,.
\end{equation*}
No union over~$L$ is taken here, since the supremum over the finite cutoffs and the limiting coefficient are already included in~\eqref{e.new.mixing.attempt}. The threshold assumptions~\eqref{e.new.mixing.minscale.thresholds} make the deterministic part of the mixing estimate smaller than a fixed dimensional fraction of the right side. Using also~$h_m\leq CK\log m$, a union bound over the points~$z$ and over~$n\in[m-h_m,m]$ gives, after increasing~$\hat L_2$ so that the~$m^{-1000}$ remainders are absorbed into the right side using~\eqref{e.sL.growth} and the first threshold in~\eqref{e.new.mixing.minscale.thresholds},
\begin{align}
	\label{e.new.mixing.minscale.one.scale}
	\P\bigl[\mathcal G_m^c\bigr]
	&\leq
	C\sum_{n=m-h_m}^{m}3^{d(m-n)}
	\Bigl(\exp\bigl( - c\delta^2 s^2 K^{-1} m^{2\expon}\log m \bigr)
	+ \exp\bigl( - c m^{100\expon}\bigr)\Bigr)
	\notag \\
	&\leq
	\exp\bigl( - c\delta^2 s^2 K^{-1}m^{2\expon}\bigr)
	\,,
\end{align}
where the last line follows from~\eqref{e.new.mixing.minscale.thresholds}. The second exponential in the first line bounds the probability that the~$\O_{\Gamma_{\nf12}}(Cm^{-1000})$ term in~\eqref{e.new.mixing.attempt} exceeds the right side in the definition of~$\mathcal G_m$. The fourth line of~\eqref{e.new.mixing.minscale.thresholds} makes the corresponding~$\Gamma_{1/2}$ exponent at least~$cm^{100\expon}$.

Define~$\mathcal M$ to be the smallest integer~$m_0\geq \hat L_2$ such that~$\mathcal G_k$ holds for every~$k\geq m_0$, and set~$\X\coloneqq 3^{\mathcal M}$. The summability of~\eqref{e.new.mixing.minscale.one.scale} implies that~$\mathcal M<\infty$ almost surely. Moreover, after increasing~$\hat L_2$ if necessary, for every~$m \geq \hat L_2$,
\begin{align*}
	\P\bigl[ \mathcal M > m \bigr]
	&\leq
	\sum_{k=m}^{\infty}
	\exp\bigl( - c\delta^2 s^2 K^{-1} k^{2\expon} \bigr)
	\\ &\leq
	C
	\exp\bigl( - c\delta^2 s^2 K^{-1} m^{2\expon} \bigr)
	\leq
	C \exp\bigl( - c (m/\hat L_2)^{2\expon}\bigr)\,,
\end{align*}
where in the last step we used~$K= C_0Ms^{-1}$ and enlarged the constant in~$\hat L_2$ so that~$\hat L_2^{2\expon} \geq C\delta^{-2}s^{-2}K$. This last enlargement is allowed by the definition~\eqref{e.Lnaught.def}: if~$Y\geq C$, then
\begin{equation*}
	L_0(Y,1-\expon,\cstar,\nu)^{2\expon}
	\geq
	c Y^2\,.
\end{equation*}
Since~$\log\X=(\log3)\mathcal M$, this proves~\eqref{e.minscale.bound}. By the definition of~$\mathcal M$, the event~$\mathcal G_m$ holds for every~$m\geq\mathcal M$; taking the maxima over~$z$,~$n$, and the coefficient fields in the defining event gives~\eqref{e.minscale.bounds.E} whenever~$3^m\geq\X$ and~$m-h_m\leq n\leq m$. This completes the proof.
\end{proof}

\section{Sharp asymptotics for the renormalized diffusivities}
\label{s.sharp.asympt}

In this section, we identify the sharp asymptotic rate of growth of the renormalized diffusivities by showing that
\begin{equation*}
	\lim_{m\to \infty}
	\frac{\shom_m}{\sqrt{2\cstar(\log 3)m}} = 1 \,.
\end{equation*}
Indeed, we prove a rate of convergence for this limit which is given in the following statement.

\begin{theorem}[Asymptotics for the renormalized diffusivities]
\label{t.sstar.sharp.bounds}
There exist constants~$C(d)<\infty$ and~$M(\cstar,\nu,\nondegconst,d)<\infty$ such that, for every~$m\in\N$ with~$m \geq M$,
\begin{equation}
	\label{e.sm.sharp.bounds}
	\bigl| \shom_m -
	\bigl( 2 \cstar (\log 3) m \bigr)^{\nf 12}
	\bigr|
	\leq
	C \cstar^{-1} \bigl(  \log^2 m + \nondegconst \bigr) \,.
\end{equation}
\end{theorem}

Theorem~\ref{t.sstar.sharp.bounds} is a consequence of the following proposition, which gives an approximate recurrence formula for the sequence of renormalized diffusivities.

\begin{proposition}[Approximate recurrence]
\label{p.one.step.sharp}
There exist~$C(d) < \infty$ and~$M(\cstar,\nu,\nondegconst,d)\in\N$ such that, for every~$n \in \N$ with~$n \geq M$ and every~$h\in \N \cap [1, \shom_n]$, we have
\begin{equation}
	\label{e.approximate.recurrence}
	\bigl| \shom_{n+h} - \shom_n -
	\cstar (\log 3) \shom_n^{-1} h
	\bigr|
	\leq
	C \bigl(  \log^2 n + \nondegconst \bigr)\shom_n^{-1} \,.
\end{equation}
\end{proposition}

We first give the proof of Theorem~\ref{t.sstar.sharp.bounds} from Proposition~\ref{p.one.step.sharp}.

\begin{proof}[{Proof of Theorem~\ref{t.sstar.sharp.bounds}} assuming Proposition~\ref{p.one.step.sharp}]
Throughout the proof,~$C<\infty$ and~$c \in (0,1)$ denote constants that depend only on the dimension~$d$ and may change from line to line. The dependence on~$\cstar$,~$\nu$ and~$\nondegconst$ is kept explicit through the iteration. We choose an increasing sequence of thresholds
\begin{equation*}
	M_{\mathrm{prop}}(\cstar,\nu,\nondegconst,d) \in \N\,,
	\qquad
	N_0 \leq N_1 \leq N_2 \leq N_3 \leq N_4 \leq N_5 \leq N_6\,,
\end{equation*}
where~$M_{\mathrm{prop}}$ is the threshold provided by Proposition~\ref{p.one.step.sharp} and depends only on~$(\cstar,\nu,\nondegconst,d)$, and the values~$N_0,\dots,N_6 \in \N$, each depending only on~$(\cstar,\nu,\nondegconst,d)$, are chosen successively below. The threshold in the theorem will be~$M  \coloneqq  N_6$ at the end of the proof.

\smallskip

By~\eqref{e.sstar.lower.bound} and~\eqref{e.sL.growth}, after increasing the threshold so that~$n\geq \nu^{-1}$ and hence~$\log(\nu^{-1}n)\leq C\log n$, there exist~$N_0(\cstar,\nu,\nondegconst,d) \in \N$ and~$C(d)<\infty$ such that, for every~$n \in \N$ with~$n \geq N_0$, we have
\begin{equation}
	\label{e.theorem.shom.bracket}
	C^{-1} \cstar^{\nf32}\nu^2  n^{\nf 12}\log^{-\nf 92}n
	\leq
	\shom_n
	\leq
	C\cstar^{-\nf32} \nu^{-2} n^{\nf 12}\log^{\nf 92} n \,.
\end{equation}
By enlarging the constant~$C(d)$ if necessary, we fix the shorthand
\begin{equation*}
	A  \coloneqq  C \cstar^{-\nf32}\nu^{-2}
\end{equation*}
so that, for every~$n \geq N_0$,
\begin{equation*}
	A^{-1} n^{\nf 12}\log^{-\nf 92}n
	\leq
	\shom_n
	\leq
	A n^{\nf 12}\log^{\nf 92} n \,.
\end{equation*}

\smallskip

\emph{Step 1.} A square-increment estimate. We claim that there exists~$N_1(\nu,\cstar,\nondegconst,d) \in \N$ such that, for every~$n \in \N$ with~$n \geq N_1$ and every~$h \in \N \cap [1,\shom_n]$, we have
\begin{equation}
	\label{e.theorem.square.increment}
	\bigl|
	\shom_{n+h}^2 - \shom_n^2 - 2\cstar(\log 3) h
	\bigr|
	\leq
	C ( \log^2 n + \nondegconst ) \,.
\end{equation}
Indeed, let~$N_1 \geq \max\{N_0,\, M_{\mathrm{prop}}\}$ and fix~$n \geq N_1$ and~$h \in \N \cap [1,\shom_n]$. Since~$N_1 \geq M_{\mathrm{prop}}$, Proposition~\ref{p.one.step.sharp} applies at~$n$, so the quantity~$\delta  \coloneqq  \shom_{n+h} - \shom_n - \cstar(\log 3)\shom_n^{-1} h$ satisfies~$|\delta| \leq C(\log^2 n + \nondegconst)\shom_n^{-1}$. By increasing~$N_1$, if necessary,~\eqref{e.theorem.shom.bracket} yields that~$\shom_n^{-1} \leq 1$. Squaring the identity~$\shom_{n+h} = \shom_n + \cstar(\log 3)\shom_n^{-1} h + \delta$ gives
\begin{equation*}
	\shom_{n+h}^2 - \shom_n^2 - 2\cstar(\log 3) h
	=
	\cstar^2(\log 3)^2 h^2 \shom_n^{-2}
	+ 2\shom_n\, \delta
	+ 2\cstar(\log 3)\shom_n^{-1} h\, \delta
	+ \delta^2 \,.
\end{equation*}
We enlarge~$N_1$ if necessary so that, for every~$n\geq N_1$, the fixed~$\cstar$ factors in the first and third terms below are absorbed into~$C(\log^2 n+\nondegconst)$, the bound~$\cstar\shom_n^{-1}\leq1$ holds, and~$(\log^2 n+\nondegconst)\shom_n^{-2}\leq C$. We bound the four terms on the right-hand side as follows.
\begin{itemize}
\item The first term is bounded by~$\cstar^2(\log 3)^2$, since~$h \leq \shom_n$.
\item The second term is bounded by~$2\shom_n|\delta| \leq C(\log^2 n + \nondegconst)$.
\item For the third term,~$h \leq \shom_n$ gives~$2\cstar(\log 3)\shom_n^{-1} h\,|\delta| \leq 2\cstar(\log 3)|\delta| \leq C (\log^2 n + \nondegconst)$.
\item The fourth term satisfies~$\delta^2 \leq C(\log^2 n + \nondegconst)^2 \shom_n^{-2} \leq C(\log^2 n + \nondegconst)$, after increasing~$N_1$ if necessary and using~\eqref{e.theorem.shom.bracket}.
\end{itemize}
Summing the four bounds, we obtain~\eqref{e.theorem.square.increment} for every~$n \geq N_1$.

\smallskip
\emph{Step 2.} An admissible sequence of indices. Choose~$N_2(\nu,\cstar,\nondegconst,d) \in \N$ to be the smallest integer satisfying~$N_2 \geq N_1$ and such that, for every~$n \in \N$ with~$n \geq N_2$,
\begin{equation}
	\label{e.theorem.N2.cond}
	\shom_n \geq 2 \quad\text{and}\quad \shom_n \leq n\,.
\end{equation}
Such an~$N_2$ exists by~\eqref{e.theorem.shom.bracket}. Define
\begin{equation*}
	n_0  \coloneqq  N_2\,,
	\qquad
	n_{k+1}  \coloneqq  n_k + \lfloor \shom_{n_k} \rfloor
	\quad (k \in \N) \,.
\end{equation*}
By induction~$n_k \geq N_2$ for every~$k \in \N$, so~\eqref{e.theorem.N2.cond} holds at each~$n_k$ and, in particular, the step size satisfies
\begin{equation}
	\label{e.theorem.step.size}
	1 \leq
	\frac 12\shom_{n_k}
	\leq
	\lfloor \shom_{n_k} \rfloor
	\leq
	\shom_{n_k} \,,
\end{equation}
with~$\lfloor \shom_{n_k} \rfloor \in \N$. Since~$n_k \geq N_2 \geq N_1$, the bound~\eqref{e.theorem.square.increment} applies at each step with~$n = n_k$ and~$h = \lfloor \shom_{n_k}\rfloor$. Moreover, the second inequality of~\eqref{e.theorem.N2.cond} gives~$n_{k+1} - n_k = \lfloor \shom_{n_k}\rfloor \leq \shom_{n_k} \leq n_k$, hence
\begin{equation}
	\label{e.theorem.nk.doubling}
	n_{k+1} \leq 2 n_k \qquad \text{for every } k \in \N \,.
\end{equation}
By the lower bound in~\eqref{e.theorem.shom.bracket}, we also have
\begin{equation}
	\label{e.theorem.step.lower}
	n_{k+1} - n_k
	=
	\lfloor \shom_{n_k} \rfloor
	\geq
	\frac12 A^{-1} n_k^{\nf 12}\log^{-\nf 92} n_k \,.
\end{equation}
For any~$N \geq N_2$, let~$K(N)$ denote the largest integer~$k$ with~$n_k \leq N$, so that~$n_{K(N)} \leq N < n_{K(N)+1}$. Telescoping
\begin{equation*}
	2\bigl( n_{k+1}^{\nf 12} - n_k^{\nf 12} \bigr)
	\geq
	\frac{n_{k+1}-n_k}{n_{k+1}^{\nf 12}}
	\geq
	cA^{-1} \log^{-\nf 92} n_{k+1}
\end{equation*}
(where in the last inequality we used~\eqref{e.theorem.step.lower} together with~\eqref{e.theorem.nk.doubling}), we obtain
\begin{equation}
	\label{e.theorem.K.bound}
	K(N)
	\leq
	C A N^{\nf 12} \log^{\nf 92} N \,.
\end{equation}

\smallskip
\emph{Step 3.} Telescoping along~$\{n_k\}$. Set~$B  \coloneqq  \shom_{n_0}^2 - 2\cstar(\log 3)\, n_0$, a fixed finite constant. Telescoping~\eqref{e.theorem.square.increment} along~$\{n_k\}$ and using~$n_{K(N)} - n_0 = \sum_{k=0}^{K(N)-1}\lfloor \shom_{n_k}\rfloor$, we obtain
\begin{align*}
	\bigl|
	\shom_{n_{K(N)}}^2 - 2\cstar(\log 3)\, n_{K(N)} - B
	\bigr|
	&
	\leq
	\sum_{k=0}^{K(N)-1}
	\bigl|
	\shom_{n_{k+1}}^2 - \shom_{n_k}^2
	- 2\cstar(\log 3)\lfloor\shom_{n_k}\rfloor
	\bigr|
	\\
	&
	\leq
	C K(N)\,
	\bigl( \log^2 N + \nondegconst \bigr) \,.
\end{align*}
Combining this with~\eqref{e.theorem.K.bound} yields
\begin{equation}
	\label{e.theorem.telescope}
	\bigl|
	\shom_{n_{K(N)}}^2 - 2\cstar(\log 3)\, n_{K(N)} - B
	\bigr|
	\leq
	CA \bigl( \log^2 N + \nondegconst \bigr) N^{\nf 12}
	\log^{\nf 92}N \,.
\end{equation}

\smallskip
\emph{Step 4.} Extension to general~$n$. Let~$n \geq N_2$ be arbitrary, and set~$k  \coloneqq  K(n)$ so that~$n_k \leq n < n_{k+1}$. Then~$h  \coloneqq  n - n_k$ satisfies~$0 \leq h \leq \lfloor \shom_{n_k}\rfloor \leq \shom_{n_k}$. If~$h \geq 1$, applying~\eqref{e.theorem.square.increment} once at~$n_k$ with this~$h$ (which is permitted since~$n_k \geq N_2 \geq N_1$) gives
\begin{equation*}
	\bigl|
	\shom_n^2 - \shom_{n_k}^2 - 2\cstar(\log 3) h
	\bigr|
	\leq
	C\bigl(\log^2 n + \nondegconst\bigr) \,;
\end{equation*}
this bound is also trivially valid (with left side being zero) when~$h=0$. After increasing~$N_2$ if necessary so that~$A n^{\nf12}\log^{\nf92}n\geq 1$ for every~$n\geq N_2$, the one-step error is absorbed into the right side below. Combining it with~\eqref{e.theorem.telescope} (applied with~$N=n$, so that~$n_{K(N)} = n_k$) and using~$n_k + h = n$, we conclude that, for every~$n \geq N_2$,
\begin{equation}
	\label{e.theorem.square.final}
	\bigl|
	\shom_n^2 - 2\cstar(\log 3)\, n - B
	\bigr|
	\leq
	C A\bigl( \log^2 n + \nondegconst \bigr)\, n^{\nf 12}
	\log^{\nf 92} n \,.
\end{equation}
Since~$B$ is a fixed constant,~$|B|$ can be absorbed into the right-hand side of~\eqref{e.theorem.square.final} as soon as~$|B| \leq CA(\log^2 n + \nondegconst)\, n^{\nf 12} \log^{\nf 92} n$. Choose~$N_3(\nu,\cstar,\nondegconst,d) \in \N$ to be the smallest integer satisfying~$N_3 \geq N_2$ and such that this inequality holds for every~$n \geq N_3$. Then, for every~$n \geq N_3$,
\begin{equation}
	\label{e.theorem.square.noB}
	\bigl|
	\shom_n^2 - 2\cstar(\log 3) n
	\bigr|
	\leq
	CA \bigl( \log^2 n + \nondegconst \bigr) n^{\nf 12}
	\log^{\nf 92}n \,.
\end{equation}

\smallskip
\emph{Step 5.} Explicit coarse comparison. The right-hand side of~\eqref{e.theorem.square.noB} is~$o(n)$ as~$n \to \infty$. We then find an integer~$N_4(\nu,\cstar,\nondegconst,d) \in \N$ satisfying~$N_4 \geq N_3$ such that, for every~$n \geq N_4$,
\begin{equation*}
	C_{\eqref{e.theorem.square.noB}}A\bigl( \log^2 n + \nondegconst \bigr)\, n^{\nf 12}
	\log^{\nf 92}n
	\leq
	\cstar(\log 3) n
	\,.
\end{equation*}
Then, for every~$n \geq N_4$,~\eqref{e.theorem.square.noB} yields
\begin{equation*}
	\cstar(\log 3)\, n
	\leq
	\shom_n^2
	\leq
	3\cstar(\log 3)\, n \,,
\end{equation*}
hence
\begin{equation}
	\label{e.theorem.coarse.compar}
	\bigl(\cstar(\log 3)\bigr)^{\nf 12} n^{\nf 12}
	\leq
	\shom_n
	\leq
	\bigl(3\cstar(\log 3)\bigr)^{\nf 12}  n^{\nf 12}
	\qquad \text{for every } n \geq N_4 \,.
\end{equation}

\smallskip
\emph{Step 6.} Refined iteration with step size of order~$n^{\nf 12}$. Set
\begin{equation*}
	\lambda  \coloneqq  \frac 12\bigl(\cstar(\log 3)\bigr)^{\nf 12} \,.
\end{equation*}
Choose~$N_5(\nu,\cstar,\nondegconst,d) \in \N$ to be the smallest integer satisfying~$N_5 \geq N_4$ and such that, for every~$n \in \N$ with~$n \geq N_5$,
\begin{equation}
	\label{e.theorem.N5.cond}
	\lambda\, n^{\nf 12} \geq 1\,,
	\quad
	\lambda\, n^{\nf 12} \leq n\,,
	\quad
	\cstar^{-\nf 12}\, n^{\nf 12} \geq 1
	\,.
\end{equation}
The conditions in~\eqref{e.theorem.N5.cond} hold as soon as~$n \geq \max\{1, \lambda^{-2}, \cstar\}$ (by~\eqref{e.cstar.bound} we have~$\lambda \leq 1$, so~$n \geq \lambda^{-2}$ also gives~$\lambda\, n^{\nf 12} \leq n$), so such an~$N_5$ exists. Since~$N_5 \geq N_4$, the lower bound in~\eqref{e.theorem.coarse.compar} gives, for every~$n \geq N_5$,
\begin{equation}
	\label{e.theorem.lambda.half}
	2\lambda\, n^{\nf 12}
	=
	\bigl(\cstar(\log 3)\bigr)^{\nf 12}\, n^{\nf 12}
	\leq
	\shom_n\,,
	\qquad\text{i.e.}\qquad
	\lambda\, n^{\nf 12} \leq \frac 12\shom_n\,.
\end{equation}
Define a second sequence by
\begin{equation*}
	m_0  \coloneqq  N_5\,,
	\qquad
	m_{j+1}  \coloneqq  m_j + \lfloor \lambda\, m_j^{\nf 12} \rfloor
	\quad (j \in \N) \,.
\end{equation*}
By induction~$m_j \geq N_5$ for every~$j \in \N$, so each~$m_j$ inherits the properties guaranteed by~$N_5$. In particular, since~$\lambda\, m_j^{\nf 12} \geq 1$ by the first inequality of~\eqref{e.theorem.N5.cond}, the step size~$h_j  \coloneqq  \lfloor \lambda\, m_j^{\nf 12}\rfloor$ is a positive integer; by~\eqref{e.theorem.lambda.half},~$h_j \leq \lambda\, m_j^{\nf 12} \leq \frac 12\shom_{m_j} \leq \shom_{m_j}$, so~\eqref{e.theorem.square.increment} applies at~$n = m_j$ with~$h = h_j$ (note~$m_j \geq N_5 \geq N_1$). Moreover, the second inequality of~\eqref{e.theorem.N5.cond} gives~$m_{j+1} - m_j = h_j \leq \lambda\, m_j^{\nf 12} \leq m_j$, hence
\begin{equation}
	\label{e.theorem.mj.doubling}
	m_{j+1} \leq 2 m_j \qquad \text{for every } j \in \N \,.
\end{equation}
For any~$N \geq N_5$, let~$J(N)$ denote the largest integer~$j$ with~$m_j \leq N$. Telescoping
\begin{equation*}
	2\bigl( m_{j+1}^{\nf 12} - m_j^{\nf 12} \bigr)
	\geq
	\frac{m_{j+1} - m_j}{m_{j+1}^{\nf 12}}
	\geq
	\frac{\lambda\, m_j^{\nf 12}/2}{m_{j+1}^{\nf 12}}
	\geq
	2^{-\nf 32}\lambda
\end{equation*}
(where the second inequality uses~$\lfloor x \rfloor \geq x/2$ for~$x \geq 1$, valid since~$\lambda\, m_j^{\nf 12} \geq 1$, and the last inequality uses~\eqref{e.theorem.mj.doubling}), we obtain
\begin{equation}
	\label{e.theorem.J.bound}
	J(N)
	\leq
	2^{\nf 52} \lambda^{-1} N^{\nf 12}
	\leq
	C \cstar^{-\nf 12} N^{\nf 12} \,.
\end{equation}
Telescoping~\eqref{e.theorem.square.increment} along~$\{m_j\}$ and using~$m_{J(N)} - m_0 = \sum_{j=0}^{J(N)-1} h_j$, we obtain
\begin{align*}
	\bigl|
	\shom_{m_{J(N)}}^2 - 2\cstar(\log 3)\, m_{J(N)} - (\shom_{m_0}^2 - 2\cstar(\log 3)\, m_0)
	\bigr|
	&
	\leq
	\sum_{j=0}^{J(N)-1}
	\bigl|
	\shom_{m_{j+1}}^2 - \shom_{m_j}^2 - 2\cstar(\log 3)\, h_j
	\bigr|
	\\
	&
	\leq
	C\, J(N)\, \bigl( \log^2 N + \nondegconst \bigr) \,.
\end{align*}
Combining with~\eqref{e.theorem.J.bound} yields
\begin{equation}
	\label{e.theorem.square.sharp.mk}
	\bigl|
	\shom_{m_{J(N)}}^2 - 2\cstar(\log 3)\, m_{J(N)}
	\bigr|
	\leq
	C \cstar^{-\nf 12} \bigl( \log^2 N + \nondegconst \bigr)\,
	N^{\nf 12}
	+ \bigl| \shom_{m_0}^2 - 2\cstar(\log 3)\, m_0 \bigr|
	\,.
\end{equation}
Set
\begin{equation*}
	C_0  \coloneqq  \bigl| \shom_{m_0}^2 - 2\cstar(\log 3)\, m_0 \bigr| \,.
\end{equation*}
Since~$m_0 = N_5$ is fixed, the quantity~$C_0$ is finite. It depends only on~$(\nu,\cstar,\nondegconst,d)$. Choose an integer~$N_6 \geq N_5$, depending only on~$(\nu,\cstar,\nondegconst,d)$, such that, for every~$n \geq N_6$,
\begin{equation*}
	C_0 \leq \cstar^{-\nf 12}\bigl(\log^2 n + \nondegconst\bigr)\, n^{\nf 12} \,.
\end{equation*}
Such an~$N_6$ exists since the right-hand side diverges to~$+\infty$ as~$n \to \infty$. For an arbitrary~$n \geq N_6$, set~$j  \coloneqq  J(n)$, so that~$m_j \leq n < m_{j+1}$, and let~$h  \coloneqq  n - m_j$. Then~$0 \leq h \leq h_j \leq \shom_{m_j}$, and one further application of~\eqref{e.theorem.square.increment} at~$m_j$ (trivial if~$h = 0$) yields
\begin{equation*}
	\bigl|
	\shom_n^2 - \shom_{m_j}^2 - 2\cstar(\log 3)\, h
	\bigr|
	\leq
	C\bigl(\log^2 n + \nondegconst\bigr)
	\leq
	C\, \cstar^{-\nf 12}\, \bigl( \log^2 n + \nondegconst \bigr)\,
	n^{\nf 12} \,,
\end{equation*}
where the last inequality holds for every~$n \geq N_5$ by the third inequality of~\eqref{e.theorem.N5.cond}. Combining this with~\eqref{e.theorem.square.sharp.mk} and using~$m_j + h = n$, then absorbing the fixed constant~$C_0$ into the right-hand side by the choice of~$N_6$, we conclude that, for every~$n \geq N_6$,
\begin{equation}
	\label{e.theorem.square.sharp}
	\bigl|
	\shom_n^2 - 2\cstar(\log 3)\, n
	\bigr|
	\leq
	C \cstar^{-\nf 12} \bigl( \log^2 n + \nondegconst \bigr)
	n^{\nf 12} \,.
\end{equation}

\smallskip
\emph{Step 7.} From squares to~$\shom_n$. By the lower bound in~\eqref{e.theorem.coarse.compar} and~$N_5 \geq N_4$,
\begin{equation*}
	\shom_n + \bigl(2\cstar(\log 3)\, n\bigr)^{\nf 12}
	\geq
	c\, \cstar^{\nf 12}\, n^{\nf 12}
	\qquad \text{for every } n \geq N_6 \,.
\end{equation*}
Using the identity
\begin{equation*}
	\bigl| \shom_n - \bigl(2\cstar(\log 3)\, n\bigr)^{\nf 12}\, \bigr|
	=
	\frac{\bigl|\,\shom_n^2 - 2\cstar(\log 3)\, n\,\bigr|}{
	\shom_n + \bigl(2\cstar(\log 3)\, n\bigr)^{\nf 12}}
\end{equation*}
together with~\eqref{e.theorem.square.sharp}, we obtain, for every~$n \geq N_6$,
\begin{equation*}
	\bigl| \shom_n - \bigl(2\cstar(\log 3)\, n\bigr)^{\nf 12}\, \bigr|
	\leq
	\frac{C \cstar^{-\nf 12}\bigl(\log^2 n + \nondegconst\bigr)
	n^{\nf 12}}{c\, \cstar^{\nf 12}\, n^{\nf 12}}
	=
	C \cstar^{-1}\bigl(\log^2 n + \nondegconst\bigr) \,.
\end{equation*}
Setting~$M(\cstar,\nu,\nondegconst,d)  \coloneqq  N_6$, we conclude that~\eqref{e.sm.sharp.bounds} holds for every~$m \geq M$.
\end{proof}

The rest of this section is devoted to the proof of Proposition~\ref{p.one.step.sharp}. The argument uses the fresh shell
\begin{equation*}
	\hshell  \coloneqq  \kfield_m-\kfield_{m-h}\,,
\end{equation*}
and the two response fields obtained by solving constant-coefficient Poisson equations with right-hand sides~$\nabla\cdot(\shom_{m-h}^{-1}\hshell e)$ and~$\nabla\cdot(\shom_{m-h}^{-1}\hshell e')$. We localize the variational problem, reduce the principal contribution from the new scale~$m$ to the old scale~$m-h$, and then evaluate the resulting quadratic form using the Neumann and Dirichlet response fields to obtain the upper and lower bounds, respectively. Together these two evaluations yield the one-sided bounds from which the recurrence formula follows.

\smallskip

We first collect the basic estimates for the Dirichlet and Neumann response fields, including the comparison between their finite-volume energies and the corresponding whole-space projected energy. We then estimate the localization errors and reduce the principal term to a deterministic quadratic form at the old scale. The two one-sided bounds are obtained by evaluating this quadratic form at~$(e, e') = (0,e')$ using the Neumann response, and at~$(e,e') = (e,0)$ using the Dirichlet response.

\smallskip

Throughout the remainder of this section, we fix integers~$m,h,K \in \N$ and define the localization scale
\begin{equation}
	\label{e.n.def.recurrence}
	n  \coloneqq  \bigl\lfloor m-h- 100 \log_3 (\nu^{-1} m)\bigr\rfloor
	\,.
\end{equation}
We think of~$m$ as the new scale,~$m-h$ as the old scale, and~$n$ as the localization scale; this auxiliary~$n$ is local to the proof and should not be confused with the iteration index in the theorem. We impose the standing condition
\begin{equation}
	\label{e.r.def.recurrence}
	1 \leq h \leq \frac{1}{400}\, m \,,
	\qquad
	K \geq 100m \,,
	\qquad
	m  \geq 2\, L_0\!\bigl(C(d),\nf 34,\cstar,\nu\bigr)\,,
\end{equation}
where~$L_0(\cdot,\cdot,\cdot,\cdot)$ is the threshold defined in~\eqref{e.Lnaught.def}. The dimensional constant in this threshold is chosen large enough to absorb the universal constants appearing in Proposition~\ref{p.homog.below} and~\eqref{e.L.vs.Lnaught}. By enlarging the threshold~$L_0(C(d),\nf 34,\cstar,\nu)$ if necessary, with the same permitted parameter dependence, we may further assume that the conditions above guarantee both~$m \geq \nu^{-2}$ and~$n \geq \frac12 m$. Moreover, by choosing~$C(d)$ large enough in the above display, by~\eqref{e.L.vs.Lnaught} we have
\begin{equation}
	\inf_{r \in \N \cap [n,\infty)}\shom_{m-h,*}(\cu_r)
	\geq
	m^{\nf 38} \log^{\nf32} m
	\label{e.L.vs.Lnaught.again}
	\,.
\end{equation}
Indeed,~\eqref{e.L.vs.Lnaught} gives the lower bound at scale~$n$, and the subadditivity of~$\shom_{m-h,*}^{-1}$ extends it to all larger spatial scales (and to the infinite-volume limit). Since~$m-h \geq n \geq \frac12 m \geq L_0(C(d),\nf 34,\cstar,\nu)$, Proposition~\ref{p.homog.below} is at our disposal (applied with~$L \mapsto m-h$,~$m \mapsto n$,~$\alpha = \nf 34$, and~$M = C(d)$). Its lower-scale condition follows from~\eqref{e.n.def.recurrence}, since~$(m-h)-n \leq C\log(\nu^{-1}m) \leq C(m-h)^{\nf34}\log^3(m-h)$ after increasing the threshold; the last inequality is just the logarithmic growth on the left being dominated by the polynomial factor on the right. In particular, by~\eqref{e.L.vs.Lnaught.again},~\eqref{e.n.def.recurrence} and~\eqref{e.ass.valid} we have
\begin{equation}
	\bigl| \shom_{m-h}^{-1}  \shom_{m-h}(\cu_{n}) - \Id \bigr|+\bigl| \shom_{m-h}  \shom^{-1}_{m-h,\ast}(\cu_{n}) - \Id \bigr|
	\leq
	C  \shom_{m-h}^{-2}  \log^2(\nu^{-1}m)
	\leq C m^{-\nf 34}
	\,.
	\label{e.ass.valid.applied}
\end{equation}
Here the final inequality uses~$m\geq\nu^{-2}$ and the lower bound~\eqref{e.L.vs.Lnaught.again}. Finally, by Lemma~\ref{l.bfAm.ellip} and the consequence~$m\geq\nu^{-2}$, we have the crude bound
\begin{equation}
	\shom_{m-h} \leq C\nu^{-1}m \leq C m^2 \,.
	\label{e.crude.shom.bnd}
\end{equation}

\smallskip

For a constant antisymmetric matrix~$\mathbf{k}_0 \in \R^{d\times d}$, we define the gauge factorization matrix
\begin{equation}
	\label{e.setup.G}
	\G_{\mathbf{k}_0}
	\coloneqq
	\begin{pmatrix}
		\Id & 0
		\\
		\mathbf{k}_0 & \Id
		\end{pmatrix}\,.
\end{equation}
In particular,
\begin{equation*}
	\G_{-\mathbf{k}_0}
	\begin{pmatrix}
		p \\ q
	\end{pmatrix}
	=
	\begin{pmatrix}
		p \\ q-\mathbf{k}_0p
		\end{pmatrix}\,.
\end{equation*}
Fix vectors~$e,e' \in \Rd$ with~$|e|,|e'| \leq 1$, and set
\begin{equation}
	\label{e.setup.P}
	P  \coloneqq  \bfAhom_{m-h}^{-\nf 12}
	\begin{pmatrix}
		e' \\ e
	\end{pmatrix}
	\in \R^{2d} \,.
\end{equation}
Let~$w_{D,e}^{(K)} \in H^1_0(\cu_K)$ be the Dirichlet solution and let~$w_{N,e'}^{(K)} \in H^1(\cu_K)$ be the zero-average Neumann solution of
\begin{equation}
	\label{e.setup.w}
	\left\{
	\begin{aligned}
		& -\Delta w_{D,e}^{(K)}
		=
		\nabla \cdot \bigl( \shom_{m-h}^{-1} \hshell e \bigr)
		&& \text{in } \cu_K \,,
		\\
		& w_{D,e}^{(K)}
		= 0
		&& \text{on } \partial \cu_K \,,
	\end{aligned}
	\right.
	\qquad
	\left\{
	\begin{aligned}
		&-\Delta w_{N,e'}^{(K)}
		=
		\nabla \cdot \bigl( \shom_{m-h}^{-1} \hshell e' \bigr)
		&& \text{in } \cu_K \,,
		\\
		& \mathbf{n}_{\cu_K} \cdot
		\bigl(
		\nabla w_{N,e'}^{(K)}
		+
		\shom_{m-h}^{-1} \hshell e'
		\bigr)
		= 0
		&& \text{on } \partial \cu_K \,.
	\end{aligned}
	\right.
\end{equation}
Testing the Dirichlet and Neumann equations with their own solutions gives
\begin{equation}
	\label{e.setup.w-energy}
	\fint_{\cu_K} |\nabla w_{D,e}^{(K)}|^2
	=
	- \shom_{m-h}^{-1}
	\fint_{\cu_K} \hshell e \cdot \nabla w_{D,e}^{(K)}
	=
	\fint_{\cu_K} \shom_{m-h}^{-1} e \cdot \hshell \nabla w_{D,e}^{(K)} \,.
\end{equation}
The corresponding Neumann energy identity is
\begin{equation}
	\label{e.setup.w-energy.neumann}
	\fint_{\cu_K} |\nabla w_{N,e'}^{(K)}|^2
	=
	- \shom_{m-h}^{-1}
	\fint_{\cu_K} \hshell e' \cdot \nabla w_{N,e'}^{(K)}
	=
	\fint_{\cu_K} \shom_{m-h}^{-1} e' \cdot \hshell \nabla w_{N,e'}^{(K)} \,.
\end{equation}
The last equalities use the skew-symmetry of~$\hshell$.

\smallskip

For later use, we also write
\begin{equation}
	\label{e.setup.hbar}
	\bar{\hshell}_z  \coloneqq  \fint_{z+\cu_n} \hshell
	\qquad
	\text{for each } z \in 3^n \Zd \cap \cu_K \,.
\end{equation}

\smallskip

For each~$z \in 3^n \Zd \cap \cu_K$, we define the blockwise slope and fluctuation fields
\begin{equation}
	\label{e.setup.Pz}
	P_z  \coloneqq
	\begin{pmatrix}
		p_z \\ q_z
	\end{pmatrix}
	\coloneqq
	\bfAhom_{m-h}^{-\nf 12}
	\begin{pmatrix}
		e' +  ( \nabla w_{D,e}^{(K)} ) _{z+\cu_n}
		\\
		e + \bigl( \nabla w_{N,e'}^{(K)} + \shom_{m-h}^{-1}\hshell e' \bigr)_{z+\cu_n}
	\end{pmatrix}
	\,,
\end{equation}
and
\begin{equation}
	\mathbf{F}_z
	\coloneqq
	\bfAhom_{m-h}^{-\nf 12}
	\begin{pmatrix}
		e' +
		\nabla w_{D,e}^{(K)} \\
		e+ \nabla w_{N,e'}^{(K)} + \shom_{m-h}^{-1}\hshell e'
	\end{pmatrix}
	- P_z
	\,.
	\label{e.setup.Fz}
\end{equation}
For each~$z \in 3^n \Zd \cap \cu_K$, let
\begin{equation*}
	S_z  \coloneqq
	\operatorname*{arg\,min}
	\left\{
	\fint_{z+\cu_n} \frac12 X \cdot \bfA_m X
	:
	X \in
	P_z + \bigl(L^2_{\pot,0} \times L^2_{\sol,0}\bigr)(z+\cu_n)
	\right\}\,,
\end{equation*}
\begin{equation*}
	X_z
	\coloneqq
	\operatorname*{arg\,min}
	\biggl\{
	\fint_{z+\cu_n} \frac12 X \cdot \bfA_m X
	\;:\;
	X \in
	P_z + \mathbf{F}_z  + \bigl(L^2_{\pot,0} \times L^2_{\sol,0}\bigr)(z+\cu_n)
	\biggr\}
	\,,
\end{equation*}
and
\begin{equation*}
	\widetilde{S}_z  \coloneqq
	\operatorname*{arg\,min}
	\left\{
	\fint_{z+\cu_n} \frac12 X \cdot \bfA_m X
	:
	X \in
	\mathbf{F}_z
	+
	\bigl(L^2_{\pot,0} \times L^2_{\sol,0}\bigr)(z+\cu_n)
	\right\}
	\,.
\end{equation*}
By construction,
\begin{equation}
	\label{e.setup.Xz-split}
	X_z = S_z + \widetilde{S}_z \,.
\end{equation}
To see the global admissibility in the next display, write on each cube~$z+\cu_n$~$X_z=P_z+\mathbf F_z+Z_z$ with~$Z_z\in (L^2_{\pot,0}\times L^2_{\sol,0})(z+\cu_n)$. The zero extensions of the fields~$Z_z$ patch to an element of~$(L^2_{\pot,0}\times L^2_{\sol,0})(\cu_K)$. Moreover, by the definitions of~$P_z$ and~$\mathbf F_z$, the field~$P_z+\mathbf F_z$ agrees on~$z+\cu_n$ with
\begin{equation*}
	P+
	\bfAhom_{m-h}^{-\nf12}
	\begin{pmatrix}
		\nabla w_{D,e}^{(K)}
		\\
		\nabla w_{N,e'}^{(K)}+\shom_{m-h}^{-1}\hshell e'
	\end{pmatrix}
	\,.
\end{equation*}
The two components in the last display belong respectively to~$L^2_{\pot,0}(\cu_K)$ and~$L^2_{\sol,0}(\cu_K)$, by the Dirichlet and Neumann boundary conditions. Thus the patched field
\begin{equation*}
	\sum_{z \in 3^n \Zd \cap \cu_K}
	X_z \,\mathbf{1}_{z+\cu_n}
	\in
	P + \bigl(L^2_{\pot,0} \times L^2_{\sol,0}\bigr)(\cu_K) \,,
\end{equation*}
is an admissible competitor in the variational problem defining~$\bfA_m(\cu_K)$. Therefore
\begin{align}
	\label{e.setup.basic-split}
	P \cdot \bfA_m(\cu_K) P
	&\leq
	\avsum_{z \in 3^n \Zd \cap \cu_K}
	\fint_{z+\cu_n} X_z \cdot \bfA_m X_z
	\notag
	\\
	&=
	\avsum_{z \in 3^n \Zd \cap \cu_K}
	\Biggl(
	P_z \cdot \bfA_m(z+\cu_n) P_z
	+
	\fint_{z+\cu_n} \bigl( 2 P_z + \widetilde{S}_z \bigr) \cdot \bfA_m \widetilde{S}_z
	\Biggr) \,.
\end{align}
The equality in~\eqref{e.setup.basic-split} uses the Euler--Lagrange orthogonality of the local minimizers~$S_z$ and~$\widetilde S_z$, recorded explicitly in the proof of Lemma~\ref{lem.localization}. This is the basic localization identity. The first term is the principal contribution, and the remaining two terms are localization errors. We will therefore estimate
\begin{equation*}
	P_z \cdot \bfA_m(z+\cu_n) P_z
\end{equation*}
after switching from the new scale~$m$ to the old scale~$m-h$, and show that the other two terms are lower order.

\smallskip

We begin by recording the basic estimates for the finite-volume Dirichlet and Neumann response fields and define the corresponding whole-space projected response field. The finite-volume estimates are direct consequences of Lemma~\ref{l.abstract.response.fields}, the shell bounds already available above, and the nondegeneracy assumption~\eqref{e.nondeg} for the energy asymptotic.

\begin{lemma}[Response-field estimates]
\label{lem.response}
In the setup above, there exists a constant~$C(d) < \infty$ such that, for every~$e \in \Rd$ with~$|e|\leq 1$, we have the estimates
\begin{equation}
	\label{e.response.grad.l8}
	\E\Bigl[
	\| \nabla w_{D,e}^{(K)} \|_{\underline L^8(\cu_K)}^8
	+
	\| \nabla w_{N,e}^{(K)} \|_{\underline L^8(\cu_K)}^8
	\Bigr]^{\nf 18}
	\leq
	C \shom_{m-h}^{-1} h^{\nf 12}\,,
\end{equation}
\begin{equation}
	\label{e.response.hess.l8}
	\E\Bigl[
	\| \nabla^2 w_{D,e}^{(K)} \|_{\underline L^8(\cu_K)}^8
	+
	\| \nabla^2 w_{N,e}^{(K)} \|_{\underline L^8(\cu_K)}^8
	\Bigr]^{\nf 18}
	\leq
	C \shom_{m-h}^{-1} 3^{-(m-h)}
	\,,
\end{equation}
and, if in addition~$|e|=1$, then
\begin{multline}
	\label{e.response.energy}
	\max\biggl\{
	\Bigl|
	\E\bigl[
	\| \nabla w_{D,e}^{(K)} \|_{\underline L^2(\cu_K)}^2\bigr]
	-
	\cstar (\log 3) h \shom_{m-h}^{-2}
	\Bigr|
	\,, \,
	\Bigl|
	\E\bigl[
	\| \nabla w_{N,e}^{(K)} \|_{\underline L^2(\cu_K)}^2\bigr]
	-
	\cstar (\log 3) h \shom_{m-h}^{-2}
	\Bigr|
	\biggr\}
	\\
	\leq
	\bigl(C h^{\nf45}(K-m)^{\nf25} 3^{-\frac25(K-m)} + \nondegconst\bigr) \shom_{m-h}^{-2}
	\,.
\end{multline}
\end{lemma}

\begin{proof}
Set
\begin{equation*}
	\mathbf F\coloneqq \shom_{m-h}^{-1}\hshell e\,.
\end{equation*}
Then~$w_{D,e}^{(K)}$ and~$w_{N,e}^{(K)}$ are precisely the Dirichlet and Neumann response fields in Lemma~\ref{l.abstract.response.fields}, with~$M=K$ and flux~$\mathbf F$. Therefore~\eqref{e.abstract.response.L8} and~\eqref{e.abstract.response.W28} give, for~$|e|\leq1$,
\begin{align*}
	\| \nabla w_{D,e}^{(K)} \|_{\underline L^8(\cu_K)}
	+
	\| \nabla w_{N,e}^{(K)} \|_{\underline L^8(\cu_K)}
	&\leq
	C\shom_{m-h}^{-1}
	\|\hshell\|_{\underline L^8(\cu_K)}\,,
	\\
	\| \nabla^2 w_{D,e}^{(K)} \|_{\underline L^8(\cu_K)}
	+
	\| \nabla^2 w_{N,e}^{(K)} \|_{\underline L^8(\cu_K)}
	&\leq
	C\shom_{m-h}^{-1}
	\|\nabla\hshell\|_{\underline L^8(\cu_K)}
	\,.
\end{align*}
The bound~\eqref{e.kmn.Lp}, with~$p=8$ and~$l=K$, implies
\begin{equation*}
	\E\Bigl[
	\| \hshell \|_{\underline L^8(\cu_K)}^8
	\Bigr]^{\nf18}
	\leq
	C h^{\nf12}\,,
\end{equation*}
and stationarity together with~\ref{a.j.reg} gives
\begin{equation*}
	\E\Bigl[
	\| \nabla \hshell \|_{\underline L^8(\cu_K)}^8
	\Bigr]^{\nf18}
	\leq
	\sum_{r=m-h+1}^{m} C3^{-r}
	\leq
	C3^{-(m-h)}\,,
\end{equation*}
which proves~\eqref{e.response.grad.l8} and~\eqref{e.response.hess.l8}.

\smallskip

Assume now that~$|e|=1$. Since~$r<K$ for~$r\in\{m-h+1,\ldots,m\}$, the endpoint estimate~\eqref{e.jk.Hminus.endpoint}, Lemma~\ref{l.Gamma.sigma.triangle}, and the geometric summability of the factors~$(K-r)3^{-(K-r)}$ give the direct weak bound
\begin{equation*}
	3^{-K}\| \mathbf F \|_{\Hminusul(\cu_K)}
	\leq
	\O_{\Gamma_2}\bigl(C\shom_{m-h}^{-1}(K-m)3^{-(K-m)}\bigr)
	\,.
\end{equation*}
Moreover,~\eqref{e.kmn.Lp}, with~$p=4$ and~$l=K$, gives
\begin{equation*}
	\| \mathbf F \|_{\underline L^4(\cu_K)}
	\leq
	\O_{\Gamma_2}\bigl(C\shom_{m-h}^{-1}h^{\nf12}\bigr)
	\,.
\end{equation*}
The spatial-average estimate~\eqref{e.jk.spatialavg} and the geometric summability over~$r\in\{m-h+1,\ldots,m\}$ yield
\begin{equation*}
	|(\mathbf F)_{\cu_K}|
	\leq
	\O_{\Gamma_2}\bigl(C\shom_{m-h}^{-1}3^{-\frac d2(K-m)}\bigr)\,.
\end{equation*}
Applying~\eqref{e.abstract.response.ND.weak} and Lemma~\ref{l.o.gamma2.mult} with the above three displays gives
\begin{equation*}
	\| \nabla w_{N,e}^{(K)}-\nabla w_{D,e}^{(K)} \|_{\underline L^2(\cu_K)}
	\leq
	\O_{\Gamma_2}\bigl(
	C\shom_{m-h}^{-1}h^{\nf25}(K-m)^{\nf15}3^{-\frac15(K-m)}
	\bigr)
	\,.
\end{equation*}
By the moment bound~\eqref{e.moments.OGamma2}, this gives
\begin{equation*}
	\E\Bigl[
	\| \nabla w_{N,e}^{(K)} - \nabla w_{D,e}^{(K)}  \|_{\underline L^2(\cu_K)}^2
	\Bigr]
	\leq
	C \shom_{m-h}^{-2} h^{\nf45}(K-m)^{\nf25}3^{-\frac25(K-m)}
	\,.
\end{equation*}
We define the stationary gradient field~$\nabla \hat w_{e} \coloneqq \shom_{m-h}^{-1} \sum_{\ell=m-h+1}^m \nabla \Delta^{-1}\nabla\cdot(\mathbf{j}_\ell e)$. The nondegeneracy assumption~\eqref{e.nondeg}, applied with~$(n,m) = (m-h,m)$, gives, for~$|e|=1$,
\begin{equation*}
	\bigl|
	\E\bigl[ | \nabla \hat{w}_e(0)|^2 \bigr]
	-
	\cstar(\log 3) h \shom_{m-h}^{-2}
	\bigr|
	\leq
	\nondegconst \shom_{m-h}^{-2}
	\,.
\end{equation*}
By~\eqref{e.energy.comparison.N.D} and~\eqref{e.energy.quadratic.N.D} from Lemma~\ref{l.abstract.response.fields}, applied to the flux~$\mathbf F$,
\begin{equation*}
	\E\bigl[ \| \nabla w_{D,e}^{(K)} \|_{\underline{L}^2(\cu_K)}^2  \bigr]
	\leq
	\E\bigl[ | \nabla \hat{w}_e (0)|^2\bigr]
	\leq
	\E\bigl[ \| \nabla w_{N,e}^{(K)} \|_{\underline{L}^2(\cu_K)}^2  \bigr]
	\,,
\end{equation*}
and
\begin{equation*}
	\| \nabla w_{N,e}^{(K)} \|_{\underline{L}^2(\cu_K)}^2
	=
	\| \nabla w_{D,e}^{(K)} \|_{\underline{L}^2(\cu_K)}^2
	+
	\| \nabla w_{N,e}^{(K)} - \nabla w_{D,e}^{(K)} \|_{\underline{L}^2(\cu_K)}^2
	\,.
\end{equation*}
Combining the comparison with the preceding bound on
\begin{equation*}
	\E\Bigl[
	\| \nabla w_{N,e}^{(K)}-\nabla w_{D,e}^{(K)} \|_{\underline L^2(\cu_K)}^2
	\Bigr]
\end{equation*}
and the nondegeneracy estimate for~$\nabla\hat w_e$ gives the two separate absolute-value bounds in~\eqref{e.response.energy}. Thus the maximum estimate follows for unit vectors.
\end{proof}

We now estimate the two localization-error terms in~\eqref{e.setup.basic-split}. The explicit form of the block matrix field~$\bfA_m$ and the oscillation bounds for the response field allow a direct argument.

\begin{lemma}[Localization bounds]
\label{lem.localization}
Under the setup above and the standing assumptions~\eqref{e.n.def.recurrence}--\eqref{e.r.def.recurrence}, there exists a constant~$C(d) < \infty$ such that
\begin{equation}
	\label{e.localization.error}
	\avsum_{z \in 3^n \Zd \cap \cu_K}
	\E\Biggl[
	\biggl|
	\fint_{z+\cu_n} \bigl( 2 P_z + \widetilde{S}_z \bigr) \cdot \bfA_m \widetilde{S}_z
	\biggr|
	\Biggr]
	\leq
	C m^{-50} \,.
\end{equation}
\end{lemma}

\begin{proof}
For each~$z \in 3^n \Zd \cap \cu_K$, write
\begin{equation*}
	\g_{D,z}  \coloneqq  \nabla w_{D,e}^{(K)} - (\nabla w_{D,e}^{(K)} )_{z+\cu_n}\,,
	\qquad
	\g_{N,z}
	\coloneqq
	\nabla w_{N,e'}^{(K)} + \shom_{m-h}^{-1}\hshell e'
	-
	\bigl(
	\nabla w_{N,e'}^{(K)} + \shom_{m-h}^{-1}\hshell e'
	\bigr)_{z+\cu_n} \,,
\end{equation*}
\begin{equation*}
	e_{D,z}
	\coloneqq e' + \bigl( \nabla w_{D,e}^{(K)} \bigr)_{z+\cu_n}
	\qquad
	e_{N,z}
	\coloneqq
	e +
	\bigl(
	\nabla w_{N,e'}^{(K)} + \shom_{m-h}^{-1}\hshell e'
	\bigr)_{z+\cu_n} \,.
\end{equation*}
Since~$\bfAhom_{m-h}$ is the isotropic block matrix associated with the scalar effective diffusivity~$\shom_{m-h}$, namely
\begin{equation*}
	\bfAhom_{m-h}
	=
	\begin{pmatrix}
		\shom_{m-h}\Id & 0
		\\
		0 & \shom_{m-h}^{-1}\Id
		\end{pmatrix}\,,
\end{equation*}
the definitions~\eqref{e.setup.Pz} and~\eqref{e.setup.Fz} may be rewritten as
\begin{equation*}
	P_z
	=
	\bfAhom_{m-h}^{-\nf 12}
	\begin{pmatrix}
		e_{D,z}
		\\
		e_{N,z}
		\end{pmatrix}\,,
	\qquad
	\mathbf F_z
	=
	\bfAhom_{m-h}^{-\nf 12}
	\begin{pmatrix}
		\g_{D,z}
		\\
		\g_{N,z}
		\end{pmatrix}\,.
\end{equation*}

\smallskip

\emph{Step 1.} We show that
\begin{align}
	\label{e.localization.error.step1}
	\lefteqn{
	\biggl|
	2 P_z \cdot \fint_{z+\cu_n} \bfA_m \widetilde S_z
	\biggr|
	+
	\fint_{z+\cu_n} \widetilde S_z \cdot \bfA_m \widetilde S_z
	} \qquad &
	\notag \\ &
	\leq
	\| \bfA_m^{\nf 12} \mathbf{F}_z \|_{\underline{L}^2(z+\cu_n)}
	\bigl(2 \| \bfA_m^{\nf 12} S_z \|_{\underline{L}^2(z+\cu_n)}  +  \| \bfA_m^{\nf 12} \mathbf{F}_z \|_{\underline{L}^2(z+\cu_n)}  \bigr)
	\,.
\end{align}

\smallskip

We first record the Euler-Lagrange orthogonality relations for the two local minimizers. Since
\begin{equation*}
	S_z - P_z
	\in
	\bigl(L^2_{\pot,0} \times L^2_{\sol,0}\bigr)(z+\cu_n)
	\quad\text{and}\quad
	\widetilde S_z - \mathbf F_z
	\in
	\bigl(L^2_{\pot,0} \times L^2_{\sol,0}\bigr)(z+\cu_n) \,,
\end{equation*}
the strict convexity of the local minimization problems gives, for every
\begin{equation*}
	Y \in \bigl(L^2_{\pot,0} \times L^2_{\sol,0}\bigr)(z+\cu_n) \,,
\end{equation*}
the identities
\begin{equation}
	\label{e.localization.euler}
	\fint_{z+\cu_n} S_z \cdot \bfA_m Y = 0
	\qquad\text{and}\qquad
	\fint_{z+\cu_n} \widetilde S_z \cdot \bfA_m Y = 0 \,.
\end{equation}
Testing the second identity in~\eqref{e.localization.euler} with~$Y = S_z-P_z$ and using the symmetry of~$\bfA_m$ gives
\begin{equation*}
	P_z \cdot \fint_{z+\cu_n} \bfA_m \widetilde S_z
	=
	\fint_{z+\cu_n} S_z \cdot \bfA_m \widetilde S_z
	\,.
\end{equation*}
Testing the first identity in~\eqref{e.localization.euler} with~$Y = \widetilde S_z-\mathbf F_z$ gives
\begin{equation*}
	\fint_{z+\cu_n} S_z \cdot \bfA_m \widetilde S_z
	=
	\fint_{z+\cu_n} S_z \cdot \bfA_m \mathbf F_z
	\,.
\end{equation*}
Combining these two identities, we obtain
\begin{equation}
	\label{e.localization.cross.identity}
	P_z \cdot \fint_{z+\cu_n} \bfA_m \widetilde S_z
	=
	\fint_{z+\cu_n} S_z \cdot \bfA_m \widetilde S_z
	=
	\fint_{z+\cu_n} S_z \cdot \bfA_m \mathbf F_z \,.
\end{equation}
Finally, testing the second identity in~\eqref{e.localization.euler} with~$Y=\widetilde S_z-\mathbf F_z$ gives
\begin{equation}
	\label{e.localization.quad.identity}
	\fint_{z+\cu_n} \widetilde S_z \cdot \bfA_m \widetilde S_z
	=
	\fint_{z+\cu_n} \mathbf F_z \cdot \bfA_m \widetilde S_z \,.
\end{equation}
By Cauchy-Schwarz in the~$\bfA_m$-inner product and the minimality of~$\widetilde S_z$ with competitor~$\mathbf F_z$, the latter identity yields
\begin{equation}
	\label{e.localization.quad.by.F}
	\fint_{z+\cu_n} \widetilde S_z \cdot \bfA_m \widetilde S_z
	\leq
	\fint_{z+\cu_n} \mathbf F_z \cdot \bfA_m \mathbf F_z \,.
\end{equation}
Likewise, for each fixed~$z$, using~\eqref{e.localization.cross.identity} and then~\eqref{e.localization.quad.by.F},
\begin{equation}
	\label{e.localization.cross.by.F}
	\biggl|
	P_z \cdot \fint_{z+\cu_n} \bfA_m \widetilde S_z
	\biggr|
	\leq
	\Biggl(
	\fint_{z+\cu_n} S_z \cdot \bfA_m S_z
	\Biggr)^{\! \nf 12}
	\Biggl(
	\fint_{z+\cu_n} \mathbf F_z \cdot \bfA_m \mathbf F_z
	\Biggr)^{\! \nf 12}
	\,.
\end{equation}
Combining the above displays yields~\eqref{e.localization.error.step1}.

\smallskip

\emph{Step 2.} We show a crude bound
\begin{equation}
	\avsum_{z \in 3^n \Zd \cap \cu_K} \E \Bigl[ \| \bfA_m^{\nf 12} S_z \|_{\underline{L}^2(z+\cu_n)}^2 \Bigr]
	\leq
	C \nu^{-4} m^{4}
	\,.
	\label{e.crude.Sz.bound}
\end{equation}
Using~$\bfE_m$ defined by~\eqref{e.Enaught.mixing}, we obtain
\begin{equation*}
	\| \bfA_m^{\nf 12} S_z \|_{\underline{L}^2(z+\cu_n)}^2
	=
	P_z \cdot \bfA_m(z+\cu_n)P_z
	\leq
	\bigl| \bfE_m^{\nf 12} P_z\bigr|^2
	\bigl| \bfE_m^{-\nf 12} \bfA_m(z+\cu_n) \bfE_m^{-\nf 12} \bigr|
	\,.
\end{equation*}
Since
\begin{align*}
	\bigl| \bfE_m^{\nf 12} P_z\bigr|^2
	&
	\leq
	C \shom_{m-h}^{-1}\bigl(\nu+\nu^{-1}m\bigr) | e_{D,z}|^2
	+ C \shom_{m-h}\nu^{-1} | e_{N,z}|^2
	\\ &
	\leq
	C \nu^{-2} m \Bigl( 1 + \bigl\| \nabla w_{D,e}^{(K)} \bigr\|_{\underline{L}^2(z+\cu_n)}^2
	+  \bigl\| \nabla w_{N,e'}^{(K)} \bigr\|_{\underline{L}^2(z+\cu_n)}^2
	+ \shom_{m-h}^{-2}
	\bigl\| \hshell \bigr\|_{\underline{L}^2(z+\cu_n)}^2
	\Bigr)
	\,,
\end{align*}
where we used~$\shom_{m-h}\leq C\nu^{-1}m$ and~$\shom_{m-h}^{-1}\leq\nu^{-1}$. Averaging over~$z$, and using~\eqref{e.Enaught.vs.A.and.Ahom},~\eqref{e.response.grad.l8} and H\"older's inequality, gives~\eqref{e.crude.Sz.bound}; the additional shell term is absorbed by~\eqref{e.kmn.Lp} with~$p=2$, the bound~$h\leq m$, and~$\shom_{m-h}\geq \nu$. We keep only a crude polynomial bound here, since it is immediately multiplied below by a large negative power of~$\nu^{-1}m$.

\smallskip

\emph{Step 3.} We show another crude estimate
\begin{equation}
	\avsum_{z \in 3^n \Zd \cap \cu_K}
	\E\Bigl[\| \bfA_m^{\nf 12} \mathbf{F}_z \|_{\underline{L}^2(z+\cu_n)}^2 \Bigr]
	\leq
	C\, \nu^{-4}\, m^4\, (\nu^{-1} m)^{-200}
	\,.
	\label{e.crude.Fz.bound}
\end{equation}
We have that
\begin{equation*}
	\| \bfA_m^{\nf 12} \mathbf{F}_z \|_{\underline{L}^2(z+\cu_n)}^2
	\leq
	\| \bfAhom_{m-h}^{-\nf 12}\bfA_m  \bfAhom_{m-h}^{-\nf 12} \|_{\underline{L}^2(z+\cu_n)}
	\Bigl( \bigl\| \g_{D,z}  \bigr\|_{\underline{L}^4(z+\cu_n)}^2 +    \bigl\| \g_{N,z}  \bigr\|_{\underline{L}^4(z+\cu_n)}^2
	\Bigr)
	\,.
\end{equation*}
Thus, by H\"older's inequality,
\begin{align*}
	\lefteqn{
	\avsum_{z \in 3^n \Zd \cap \cu_K}
	\E\Bigl[\| \bfA_m^{\nf 12} \mathbf{F}_z \|_{\underline{L}^2(z+\cu_n)}^2 \Bigr]
	} \qquad &
	\notag \\ &
	\leq
	2 \E\Bigl[ \| \bfAhom_{m-h}^{-\nf 12}\bfA_m  \bfAhom_{m-h}^{-\nf 12} \|_{\underline{L}^2(\cu_K)}^2 \Bigr]^{\nf 12}
	\E\biggl[ \avsum_{z \in 3^n \Zd \cap \cu_K}
	\Bigl( \bigl\| \g_{D,z}  \bigr\|_{\underline{L}^4(z+\cu_n)}^4 +    \bigl\| \g_{N,z}  \bigr\|_{\underline{L}^4(z+\cu_n)}^4 \Bigr)
	\biggr]^{\nf 12}
	\,.
\end{align*}
By the Poincar\'e inequality we get
\begin{align*}
	\lefteqn{
	\avsum_{z \in 3^n \Zd \cap \cu_K}
	\Bigl( \bigl\| \g_{D,z}  \bigr\|_{\underline{L}^4(z+\cu_n)}^4 +    \bigl\| \g_{N,z}  \bigr\|_{\underline{L}^4(z+\cu_n)}^4 \Bigr)
	} \qquad &
	\notag \\ &
	\leq C 3^{4n} \Bigl(
	\bigl\| \nabla^2 w_{D,e}^{(K)}  \bigr\|_{\underline{L}^4(\cu_K)}^4
	+
	\bigl\| \nabla^2 w_{N,e'}^{(K)}  \bigr\|_{\underline{L}^4(\cu_K)}^4
	+
	\shom_{m-h}^{-4} \bigl\| \nabla\hshell  \bigr\|_{\underline{L}^4(\cu_K)}^4
	\Bigr)
	\,,
\end{align*}
so that, by~\eqref{e.response.hess.l8} and~\ref{a.j.reg}, we deduce that
\begin{equation*}
	\E\biggl[ \avsum_{z \in 3^n \Zd \cap \cu_K}
	\Bigl( \bigl\| \g_{D,z}  \bigr\|_{\underline{L}^4(z+\cu_n)}^4 +    \bigl\| \g_{N,z}  \bigr\|_{\underline{L}^4(z+\cu_n)}^4 \Bigr)
	\biggr]^{\nf 12}
	\leq
	C\, \nu^{-4}\, (\nu^{-1} m)^{-200}
	\,,
\end{equation*}
where the Hessian estimate contributes~$3^{2n}\shom_{m-h}^{-2}3^{-2(m-h)}$ after taking the~$\nf12$ power. We used~\eqref{e.sstar.lower.bound} to bound~$\shom_{m-h}^{-2}$; after enlarging the threshold and, if necessary, reducing the exponent~$200$, all powers of~$\cstar$ and logarithms are absorbed into the large negative power of~$\nu^{-1}m$. Before taking the~$\nf12$ power,~\eqref{e.n.def.recurrence} gives~$3^{4n}3^{-4(m-h)} \leq C(\nu^{-1}m)^{-400}$, and hence the displayed factor~$(\nu^{-1}m)^{-200}$ after the~$\nf12$ power. Moreover, by~\eqref{e.kmn.Lp},
\begin{equation*}
	\E\Bigl[ \| \bfAhom_{m-h}^{-\nf 12}\bfA_m  \bfAhom_{m-h}^{-\nf 12} \|_{\underline{L}^2(\cu_K)}^2 \Bigr]^{\nf 12}
	\leq
	C \shom_{m-h}^{-1} \Bigl( \nu + \nu^{-1} \E\bigl[ \| \k_m \|_{\underline{L}^4(\cu_K)}^4 \bigr]^{\nf 12}
	\Bigr)
	+ C \shom_{m-h} \nu^{-1}
	\leq
	C\, \nu^{-2} m^2
	\,.
\end{equation*}
This is the same crude polynomial bound as in~\eqref{e.crude.Sz.bound}. The last two displays together yield~\eqref{e.crude.Fz.bound}.

\smallskip

\emph{Step 4.} Conclusion. Combining~\eqref{e.localization.error.step1} with the Cauchy-Schwarz inequality,~\eqref{e.crude.Sz.bound} and~\eqref{e.crude.Fz.bound}, we obtain
\begin{align*}
	\lefteqn{
	\avsum_{z \in 3^n \Zd \cap \cu_K}
	\E\biggl[
	\Bigl|
	2 P_z \cdot \fint_{z+\cu_n} \bfA_m \widetilde S_z
	\Bigr|
	+
	\fint_{z+\cu_n} \widetilde S_z \cdot \bfA_m \widetilde S_z
	\biggr]
	} \qquad &
	\notag \\ &
	\leq
	C\, \nu^{-4} m^4\, (\nu^{-1} m)^{-100}
	=
	C\, \nu^{96}\, m^{-96}
	\leq
	C\, m^{-50} \,,
\end{align*}
where the last inequality uses~$\nu \in (0,1]$. This proves~\eqref{e.localization.error}.
\end{proof}

We next estimate the principal term in~\eqref{e.setup.basic-split}. After localizing from cutoff~$m$ to cutoff~$m-h$, we take the expectation over the lower-scale field to obtain the finite-volume annealed matrix~$\bfAhom_{m-h}(z+\cu_n) = \E[\bfA_{m-h}(z+\cu_n)]$. We then compare~$\bfAhom_{m-h}(z+\cu_n)$ with~$\bfAhom_{m-h}$ using Proposition~\ref{p.homog.below}. This estimates the error in the expected quadratic form directly, without first bounding the fluctuations of the random matrix.

\begin{lemma}[Reduction of the principal term]
\label{lem.principal.term}
Under the setup above, there exists a constant~$C(d) < \infty$ such that
\begin{align}
	\lefteqn{
	\avsum_{z \in 3^n \Zd \cap \cu_K}
	\E\bigl[
	P_z \cdot \bfA_m(z+\cu_n) P_z
	\bigr]
	} \qquad &
	\notag \\ &
	\leq
	\bigl( 1 + C\, \shom_{m-h}^{-2}\log^2(\nu^{-1}m) \bigr)
	\avsum_{z \in 3^n \Zd \cap \cu_K}
	\E\bigl[
	\bigl| \bfAhom_{m-h}^{\nf12} \G_{-\bar{\hshell}_z}  P_z \bigr|^2
	\bigr]
	+
	C m^{-10} \,.
	\label{e.principal.reduction}
\end{align}
\end{lemma}

\begin{proof}
We work throughout with the notation on each cube introduced near the beginning of the proof of Lemma~\ref{lem.localization}. For each~$z \in 3^n \Zd \cap \cu_K$, recall the averaged quantities and oscillation terms
\begin{equation*}
	\g_{D,z}  \coloneqq  \nabla w_{D,e}^{(K)} - \bigl(\nabla w_{D,e}^{(K)}\bigr)_{z+\cu_n} \,,
	\qquad
	\g_{N,z}  \coloneqq  \nabla w_{N,e'}^{(K)} + \shom_{m-h}^{-1}\hshell e'
	- \bigl(\nabla w_{N,e'}^{(K)} + \shom_{m-h}^{-1}\hshell e'\bigr)_{z+\cu_n} \,,
\end{equation*}
\begin{equation*}
	e_{D,z}  \coloneqq  e' + \bigl(\nabla w_{D,e}^{(K)}\bigr)_{z+\cu_n} \,,
	\qquad
	e_{N,z}  \coloneqq  e + \bigl(\nabla w_{N,e'}^{(K)} + \shom_{m-h}^{-1}\hshell e'\bigr)_{z+\cu_n} \,,
\end{equation*}
together with~$\bar{\hshell}_z  \coloneqq  (\hshell)_{z+\cu_n}$, a constant antisymmetric matrix on~$z+\cu_n$. As in~\eqref{e.setup.Pz} we have
\begin{equation*}
	P_z
	=
	\bfAhom_{m-h}^{-\nf 12}
	\begin{pmatrix}
		e_{D,z}
		\\
		e_{N,z}
		\end{pmatrix}\,.
\end{equation*}

\smallskip

We first switch from the new cutoff~$m$ to the old cutoff~$m-h$ on the box~$z+\cu_n$. Since~$\a_m-\bar{\hshell}_z = \a_{m-h}+(\hshell-\bar{\hshell}_z)$ and
\begin{equation*}
	\|\hshell-\bar{\hshell}_z\|_{L^\infty(z+\cu_n)}
	\leq C3^n\|\nabla\hshell\|_{L^\infty(z+\cu_n)}\,,
\end{equation*}
Lemma~\ref{l.localization.A}, applied with~$\a = \a_{m-h}$, perturbation field~$\hshell-\bar{\hshell}_z$, and domain~$z+\cu_n$, and with~$\s=\nu\Id$ in the perturbative ellipticity parameter, together with~\eqref{e.commute.coarse.grained.k0}, gives
\begin{equation}
	\label{e.principal.localize}
	\bigl|
	\bfA_{m-h}^{-\nf 12}(z+\cu_n)
	\G_{\bar{\hshell}_z}^{\,t} \bfA_m(z+\cu_n) \G_{\bar{\hshell}_z}
	\bfA_{m-h}^{-\nf 12}(z+\cu_n)
	- \Itwod
	\bigr|
	\leq
	D_z \,,
\end{equation}
where
\begin{equation}
	\label{e.principal.Dz}
	D_z
	\coloneqq
	C(d)\nu^{-1} 3^n \| \nabla \hshell \|_{L^\infty(z+\cu_n)}
	+
	C(d)\nu^{-2} 3^{2n} \| \nabla \hshell \|_{L^\infty(z+\cu_n)}^2 \,.
\end{equation}
The factors~$\nu^{-1}$ and~$\nu^{-2}$ come from~$\|\s^{-\nf12}(\hshell-\bar{\hshell}_z)\s^{-\nf12}\|$ with~$\s=\nu\Id$. For each~$z \in 3^n \Zd \cap \cu_K$, define
\begin{equation}
	\label{e.principal.hbar}
	\widehat P_z  \coloneqq  \G_{-\bar{\hshell}_z} P_z \,.
\end{equation}
The matrix inequality~\eqref{e.principal.localize} implies, pointwise in the environment,
\begin{equation}
	\label{e.principal.switch}
	P_z \cdot \bfA_m(z+\cu_n) P_z
	\leq
	(1+D_z)\,
	\widehat P_z \cdot \bfA_{m-h}(z+\cu_n) \widehat P_z \,.
\end{equation}
Using the shell gradient bound~\eqref{e.nabla.kmn.Linfty} on translated cubes by stationarity, together with the definition of~$n$ in~\eqref{e.n.def.recurrence} and the fact~$\nu \in (0,1]$, we obtain
\begin{equation}
	\label{e.principal.Dz.moment}
	\Biggl(
	\avsum_{z \in 3^n \Zd \cap \cu_K}
	\E\bigl[ D_z^4 \bigr]
	\Biggr)^{\! \nf14}
	\leq
	C m^{-100} \,.
\end{equation}

\smallskip

We next average over the lower-scale field. By the definition of the finite-volume annealed matrices in~\eqref{e.homs.defs.U.0} and~\eqref{e.homs.defs.U} and stationarity,
\begin{equation}
	\label{e.principal.Ahom.fv}
	\bfAhom_{m-h}(z+\cu_n)
	\coloneqq
	\E\bigl[ \bfA_{m-h}(z+\cu_n) \bigr]
	=
	\begin{pmatrix}
		\shom_{m-h}(z+\cu_n)\Id & 0
		\\
		0 & \shom_{m-h,*}^{-1}(z+\cu_n)\Id
		\end{pmatrix}\,.
\end{equation}
Moreover,~$\bfA_{m-h}(z+\cu_n)$ depends only on the lower-scale field~$\kfield_{m-h}$, while~$\widehat P_z$ and~$D_z$ depend only on the fresh shell~$\hshell = \kfield_m-\kfield_{m-h}$ through~$w_{D,e}^{(K)}$,~$w_{N,e'}^{(K)}$, and the box average~$\bar{\hshell}_z$. Equivalently, if
\begin{equation*}
	\mathcal F_{\mathrm{old}}\coloneqq
	\sigma(\mathbf j_r\,:\,r\leq m-h)
	\qand
	\mathcal F_{\mathrm{new}}\coloneqq
	\sigma(\mathbf j_r\,:\,m-h<r\leq m)
	\,,
\end{equation*}
then~$\bfA_{m-h}(z+\cu_n)$ is~$\mathcal F_{\mathrm{old}}$-measurable, whereas~$\widehat P_z$ and~$D_z$ are~$\mathcal F_{\mathrm{new}}$-measurable. By~\ref{a.j.indy}, these sigma-fields are independent. Therefore, conditioning on~$\mathcal F_{\mathrm{new}}$ and using~\eqref{e.principal.Ahom.fv}, we obtain
\begin{equation}
	\label{e.principal.conditional}
	\E\bigl[
	(1+D_z)\widehat P_z \cdot \bfA_{m-h}(z+\cu_n) \widehat P_z
	\bigr]
	=
	\E\bigl[
	(1+D_z)\widehat P_z \cdot \bfAhom_{m-h}(z+\cu_n) \widehat P_z
	\bigr]\,.
\end{equation}
By stationarity,~\eqref{e.ass.valid.applied} and the definition~\eqref{e.n.def.recurrence}, since both matrices below are diagonal in the block decomposition, we obtain
\begin{equation}
	\label{e.principal.fv.to.inf}
	\bigl|
	\bfAhom_{m-h}^{-\nf 12}
	\bfAhom_{m-h}(z+\cu_n)
	\bfAhom_{m-h}^{-\nf 12}
	- \Itwod
	\bigr|
	\leq
	C\, \shom_{m-h}^{-2}\log^2(\nu^{-1}m) \,.
\end{equation}
The additive remainder~$Cn^{-3000}$ from Proposition~\ref{p.homog.below} is at most~$C(m-h)^{-999}$ because~$n\geq m/2$. Moreover,~\eqref{e.crude.shom.bnd} and~$m-h\geq m/2$ give~$\shom_{m-h}^{-2}\log^2(\nu^{-1}m) \geq c m^{-4}\log^2(\nu^{-1}m) \geq c(m-h)^{-999}$ after increasing the threshold. Thus the explicit term on the right of~\eqref{e.principal.fv.to.inf} absorbs the remainder. Combining~\eqref{e.principal.conditional} and~\eqref{e.principal.fv.to.inf} yields
\begin{equation}
	\label{e.principal.conditional.upper}
	\E\bigl[
	(1+D_z)\widehat P_z \cdot \bfA_{m-h}(z+\cu_n) \widehat P_z
	\bigr]
	\leq
	\bigl( 1 + C\, \shom_{m-h}^{-2}\log^2(\nu^{-1}m) \bigr)
	\E\bigl[
	(1+D_z)\bigl| \bfAhom_{m-h}^{\nf 12} \widehat P_z \bigr|^2
	\bigr]\,.
\end{equation}
Combining~\eqref{e.principal.switch} and~\eqref{e.principal.conditional.upper} gives
\begin{equation}
	\label{e.principal.pre.average}
	\E\bigl[
	P_z \cdot \bfA_m(z+\cu_n) P_z
	\bigr]
	\leq
	\bigl( 1 + C\, \shom_{m-h}^{-2}\log^2(\nu^{-1}m) \bigr)
	\E\bigl[
	(1+D_z)\,
	\bigl| \bfAhom_{m-h}^{\nf 12} \widehat P_z \bigr|^2
	\bigr]
	\,.
\end{equation}

\smallskip

It remains to bound the moment of the normalized vector appearing with~$D_z$. By~\ref{a.j.iso} and~\eqref{e.homs.defs}, together with~\eqref{e.setup.Pz} and the definition of~$\G_{-\bar{\hshell}_z}$,
\begin{equation}
	\label{e.principal.Phat}
	\bfAhom_{m-h}^{\nf 12} \widehat P_z
	=
	\begin{pmatrix}
		e_{D,z}
		\\
		e_{N,z} - \shom_{m-h}^{-1} \bar{\hshell}_z e_{D,z}
		\end{pmatrix}\,.
\end{equation}
By~\eqref{e.response.grad.l8},~\eqref{e.kmn.Lp} and Jensen's inequality, we obtain
\begin{align}
	\label{e.principal.uz}
	\lefteqn{
	\Biggl(
	\avsum_{z \in 3^n \Zd \cap \cu_K}
	\E\bigl[ |e_{D,z}|^8 + |e_{N,z}|^8 +  |\shom_{m-h}^{-1}\bar{\hshell}_z|^8 \bigr]
	\Biggr)^{\nf 18}
	} \qquad &
	\notag \\ &
	\leq
	C\Bigl( 1 + \E\bigl[ \| \nabla w_{D,e}^{(K)} \|_{\underline{L}^8(\cu_K)}^8 \bigr] + \E\bigl[ \| \nabla w_{N,e'}^{(K)} \|_{\underline{L}^8(\cu_K)}^8\bigr]  + \shom_{m-h}^{-8}\E\bigl[ \| \k_m - \k_{m-h} \|_{\underline{L}^8(\cu_K)}^8 \bigr] \Bigr)^{\nf 18}
	\notag \\ &
	\leq
	C (1 + \shom_{m-h}^{-1} h^{\nf 12}) \,.
\end{align}
The identity~\eqref{e.principal.Phat}, together with H\"older's inequality and~\eqref{e.principal.uz},  yields
\begin{equation}
	\label{e.principal.Phat.moment}
	\Biggl(
	\avsum_{z \in 3^n \Zd \cap \cu_K}
	\E\Bigl[
	\bigl| \bfAhom_{m-h}^{\nf 12} \widehat P_z \bigr|^4
	\Bigr]
	\Biggr)^{\! \nf14}
	\leq
	C \bigl( 1 + \shom_{m-h}^{-1} h^{\nf 12} \bigr)^2 \,.
\end{equation}
The square on the right is a harmless upper bound for the cross term involving~$\bar{\hshell}_z e_{D,z}$.

\smallskip

We now average~\eqref{e.principal.pre.average} over~$z$, apply Cauchy--Schwarz, use Jensen's inequality to pass from the fourth moment of~$D_z$ to the second moment, and invoke~\eqref{e.principal.Dz.moment} and~\eqref{e.principal.Phat.moment}:
\begin{align*}
	\lefteqn{
	\avsum_{z \in 3^n \Zd \cap \cu_K}
	\E\Bigl[
	(1+D_z)\,
	\bigl| \bfAhom_{m-h}^{\nf 12} \widehat P_z \bigr|^2
	\Bigr]
	}
	\qquad &
	\\
	&\leq
	\avsum_{z \in 3^n \Zd \cap \cu_K}
	\E\bigl[
	\bigl| \bfAhom_{m-h}^{\nf 12} \widehat P_z \bigr|^2
	\bigr]
	+
	\Biggl(
	\avsum_{z \in 3^n \Zd \cap \cu_K}
	\E\bigl[ D_z^2 \bigr]
	\Biggr)^{\nf 12}
	\Biggl(
	\avsum_{z \in 3^n \Zd \cap \cu_K}
	\E\Bigl[
	\bigl| \bfAhom_{m-h}^{\nf 12} \widehat P_z \bigr|^4
	\Bigr]
	\Biggr)^{\nf 12}
	\notag \\
	&\leq
	\avsum_{z \in 3^n \Zd \cap \cu_K}
	\E\bigl[
	\bigl| \bfAhom_{m-h}^{\nf 12} \widehat P_z \bigr|^2
	\bigr]
	+
	C\, m^{-10} \,,
\end{align*}
where in the last inequality we used~\eqref{e.principal.Dz.moment}, then absorbed the factor~$\bigl(1+\shom_{m-h}^{-1}h^{\nf 12}\bigr)^4$ into the~$m^{-100}$ prefactor by combining~\eqref{e.L.vs.Lnaught.again} and~\eqref{e.ass.valid.applied} with the upper bound~$h \leq m/400$ from~\eqref{e.r.def.recurrence} (which together give~$\shom_{m-h}^{-1} h^{\nf 12}\leq C\, m^{\nf 18}$, hence~$\bigl(1+\shom_{m-h}^{-1}h^{\nf12}\bigr)^4\leq C\, m^{\nf 12}$). Substituting this bound into~\eqref{e.principal.pre.average} proves~\eqref{e.principal.reduction}.
\end{proof}

\smallskip

We now evaluate the quadratic form with the Dirichlet component set to zero. The Neumann response energy enters with a positive sign, giving the upper bound on the ratio of diffusivities.

\begin{lemma}[Evaluation using the Neumann response]
\label{lem.neumann.input}
In the setup above, for every unit vector~$e' \in \Rd$, we have
\begin{multline}
	\label{e.neumann.estimate.explicit}
	\limsup_{K\to\infty}
	\avsum_{z \in 3^n \Zd \cap \cu_K}
	\E\Biggl[
	\biggl|
	\bfAhom_{m-h}^{\nf 12}
	\G_{-\bar{\hshell}_z}
	\bfAhom_{m-h}^{-\nf 12}
	\begin{pmatrix}
		e'
		\\
		\bigl( \nabla w_{N,e'}^{(K)} + \shom_{m-h}^{-1}\hshell e' \bigr)_{z+\cu_n}
	\end{pmatrix}
	\biggr|^2
	\Biggr]
	\\
	\leq
	1
	+
	\cstar(\log 3) h \shom_{m-h}^{-2}
	+
	\nondegconst \shom_{m-h}^{-2} \,.
\end{multline}
\end{lemma}

\begin{proof}
Fix a unit vector~$e' \in \Rd$. For each~$z \in 3^n \Zd \cap \cu_K$, the block form of~$\bfAhom_{m-h}$ from~\eqref{e.homs.defs} and the definition of~$\G_{-\bar{\hshell}_z}$ give
\begin{equation*}
	\bfAhom_{m-h}^{\nf 12}
	\G_{-\bar{\hshell}_z}
	\bfAhom_{m-h}^{-\nf 12}
	\begin{pmatrix}
		e'
		\\
		\bigl( \nabla w_{N,e'}^{(K)} + \shom_{m-h}^{-1}\hshell e' \bigr)_{z+\cu_n}
	\end{pmatrix}
	=
	\begin{pmatrix}
		e'
		\\
		(\nabla w_{N,e'}^{(K)})_{z+\cu_n}
		\end{pmatrix}\,.
\end{equation*}
Therefore
\begin{equation*}
	\Biggl|
	\bfAhom_{m-h}^{\nf 12}
	\G_{-\bar{\hshell}_z}
	\bfAhom_{m-h}^{-\nf 12}
	\begin{pmatrix}
		e'
		\\
		\bigl( \nabla w_{N,e'}^{(K)} + \shom_{m-h}^{-1}\hshell e' \bigr)_{z+\cu_n}
	\end{pmatrix}
	\Biggr|^2
	=
	1 + \bigl| (\nabla w_{N,e'}^{(K)})_{z+\cu_n} \bigr|^2 \,.
\end{equation*}
Averaging over~$z$, using the exact tiling of~$\cu_K$ by the subcubes and then Jensen's inequality, yields
\begin{equation*}
	\avsum_{z \in 3^n \Zd \cap \cu_K}
	\Biggl|
	\bfAhom_{m-h}^{\nf 12}
	\G_{-\bar{\hshell}_z}
	\bfAhom_{m-h}^{-\nf 12}
	\begin{pmatrix}
		e'
		\\
		\bigl( \nabla w_{N,e'}^{(K)} + \shom_{m-h}^{-1}\hshell e' \bigr)_{z+\cu_n}
	\end{pmatrix}
	\Biggr|^2
	\leq
	1 + \| \nabla w_{N,e'}^{(K)} \|_{\underline L^2(\cu_K)}^2 \,.
\end{equation*}
Taking expectation, passing to the limit superior~$K\to\infty$ and using~\eqref{e.response.energy}, whose explicit~$K$-dependent error vanishes in this limit, proves~\eqref{e.neumann.estimate.explicit}.
\end{proof}

We will use the following consequence of~\eqref{e.L.vs.Lnaught.again} and~\eqref{e.ass.valid.applied} in the two ratio estimates:
\begin{equation}
	\label{e.shom.log.bound}
	\shom_{m-h}^{-2}\log^2 m \leq C\, m^{-\nf 34} \,.
\end{equation}

\begin{corollary}[Upper bound on the ratio of diffusivities]
\label{cor.upper.ratio}
Under the setup above, there exists a constant~$C(d) < \infty$ such that
\begin{equation}
	\label{e.upper.ratio}
	\shom_m \shom_{m-h}^{-1}
	\leq
	1
	+
	\bigl( \cstar (\log 3) h + C(\log^2 m + \nondegconst) \bigr)\shom_{m-h}^{-2}
	+
	C  \shom_{m-h}^{-4} h^2
	\,.
\end{equation}
\end{corollary}

\begin{proof}
Fix a unit vector~$e' \in \Rd$ and set~$e  \coloneqq  0$ in the preceding setup. Then~$w_{D,0}^{(K)} \equiv 0$, and the block form~\eqref{e.homs.defs.U} and finite-volume convergence~\eqref{e.homs.defs} give
\begin{equation*}
	\lim_{K\to\infty}
	\E\bigl[ P \cdot \bfA_m(\cu_K) P \bigr]
	=
	\shom_m \shom_{m-h}^{-1} \,.
\end{equation*}
Here the off-diagonal homogenized skew part vanishes by~\ref{a.j.iso}, and~$\shom_m(\cu_K)\to\shom_m$ by the convergence in~\eqref{e.homs.defs}. By the basic split~\eqref{e.setup.basic-split} and the localization bound~\eqref{e.localization.error},
\begin{equation*}
	\shom_m \shom_{m-h}^{-1}
	\leq
	\limsup_{K\to\infty}
	\avsum_{z \in 3^n \Zd \cap \cu_K}
	\E\bigl[
	P_z \cdot \bfA_m(z+\cu_n) P_z
	\bigr]
	+ C\, m^{-10} \,.
\end{equation*}
Applying the principal-term estimate~\eqref{e.principal.reduction} followed by the estimate for the Neumann response~\eqref{e.neumann.estimate.explicit}, and using the choice of the threshold following~\eqref{e.r.def.recurrence} to absorb~$\log^2(\nu^{-1}m) \leq C \log^2 m$ into the constant, we obtain
\begin{equation}
	\label{e.upper.ratio.product}
	\shom_m \shom_{m-h}^{-1}
	\leq
	\bigl( 1 + C\, \shom_{m-h}^{-2}\log^2 m \bigr)
	\Bigl(
	1
	+
	\cstar (\log 3) h \shom_{m-h}^{-2}
	+
	\nondegconst \shom_{m-h}^{-2}
	\Bigr)
	+
	C\, m^{-10} \,.
\end{equation}
We now expand the product on the right of~\eqref{e.upper.ratio.product} into additive form. Set
\begin{equation*}
	A  \coloneqq  C\, \shom_{m-h}^{-2}\log^2 m \,,
	\qquad
	B  \coloneqq  \cstar(\log 3) h\, \shom_{m-h}^{-2} \,,
	\qquad
	N  \coloneqq  \nondegconst\, \shom_{m-h}^{-2} \,,
\end{equation*}
so that the right-hand side of~\eqref{e.upper.ratio.product} equals~$1 + B + N + A + AB + AN + C m^{-10}$. The cross term~$AB$ is bounded by Young's inequality~$h \log^2 m \leq \frac12 h^2 + \frac12 \log^4 m$ together with~\eqref{e.cstar.bound} and~\eqref{e.shom.log.bound} (used in the form~$\shom_{m-h}^{-4}\log^4 m = (\shom_{m-h}^{-2}\log^2 m)^2 \leq C m^{-\nf 34}\, \shom_{m-h}^{-2}\log^2 m \leq C\, \shom_{m-h}^{-2}\log^2 m$),
\begin{equation*}
	AB
	\leq
	C\, \shom_{m-h}^{-4} h^2
	+
	C\, \shom_{m-h}^{-2}\log^2 m \,.
\end{equation*}
The cross term~$AN$ is bounded directly by~\eqref{e.shom.log.bound},
\begin{equation*}
	AN
	=
	C\, \nondegconst\, \shom_{m-h}^{-2}
	\bigl( \shom_{m-h}^{-2}\log^2 m \bigr)
	\leq
	C\, \nondegconst\, m^{-\nf 34}\, \shom_{m-h}^{-2}
	\leq
	C\, \nondegconst\, \shom_{m-h}^{-2} \,.
\end{equation*}
Finally, by~\eqref{e.crude.shom.bnd},~$\shom_{m-h}^{-2}\log^2 m \geq c\, m^{-4}\log^2 m \geq c m^{-10}$ after increasing the threshold, so the residual~$C m^{-10}$ is absorbed into~$C\,\shom_{m-h}^{-2}\log^2 m$ at the price of enlarging the constant. Combining the above bounds yields~\eqref{e.upper.ratio}.
\end{proof}

We now set~$e'=0$ and use the Dirichlet response. Expanding the quadratic form and estimating the oscillation terms gives a negative contribution from the response energy, which we estimate using~\eqref{e.response.energy}.

\begin{lemma}[Evaluation using the Dirichlet response]
\label{lem.coupled.input}
In the setup above, for every unit vector~$e \in \Rd$, we have
\begin{multline}
	\label{e.coupled.estimate.explicit}
	\limsup_{K\to\infty}
	\avsum_{z \in 3^n \Zd \cap \cu_K}
	\E\Biggl[
	\biggl|
	\bfAhom_{m-h}^{\nf 12}
	\G_{-\bar{\hshell}_z}
	\bfAhom_{m-h}^{-\nf 12}
	\begin{pmatrix}
		( \nabla w_{D,e}^{(K)} )_{z+\cu_n} \\ e
	\end{pmatrix}
	\biggr|^2
	\Biggr]
	\\
	\leq
	1
	-
	\cstar(\log 3) h \shom_{m-h}^{-2}
	+
	\nondegconst \shom_{m-h}^{-2}
	+
	C(d)\, \shom_{m-h}^{-4} h^2 \,.
\end{multline}
\end{lemma}

\begin{proof}
Fix a unit vector~$e \in \Rd$ and specialize the notation on each cube from Lemma~\ref{lem.principal.term} to the present case~$e' = 0$. In this case~$w_{N,0}^{(K)} \equiv 0$, so the definitions reduce to
\begin{equation*}
	\g_{D,z}  \coloneqq  \nabla w_{D,e}^{(K)} - \bigl(\nabla w_{D,e}^{(K)}\bigr)_{z+\cu_n} \,,
	\qquad
	\g_{N,z} = 0 \,,
\end{equation*}
\begin{equation*}
	e_{D,z} = \bigl(\nabla w_{D,e}^{(K)}\bigr)_{z+\cu_n} \,,
	\qquad
	e_{N,z} = e \,,
\end{equation*}
together with the constant skew matrix~$\bar{\hshell}_z  \coloneqq  (\hshell)_{z+\cu_n}$ already used in the proof of Lemma~\ref{lem.principal.term}. In particular,~$\g_{D,z} = \nabla w_{D,e}^{(K)} - e_{D,z}$ in this specialization. We also need the shell-weighted averaged Dirichlet gradient
\begin{equation*}
	B_z  \coloneqq  \fint_{z+\cu_n} \hshell \nabla w_{D,e}^{(K)} \,.
\end{equation*}
The coupling correction enters through the residual
\begin{equation}
	\label{e.coupled.Rz.def}
	R_z
	\coloneqq
	B_z - \bar{\hshell}_z\, e_{D,z}
	=
	\fint_{z+\cu_n}
	\bigl(\hshell - \bar{\hshell}_z\bigr)\, \g_{D,z} \,.
\end{equation}

\smallskip

\emph{Step 1.} Auxiliary oscillation estimates. We claim that
\begin{align}
	\label{e.coupled.shell.osc}
	\Biggl(
	\avsum_{z \in 3^n \Zd \cap \cu_K}
	\E\Bigl[
	\bigl\|
	\hshell - \bar{\hshell}_z
	\bigr\|_{\underline L^4(z+\cu_n)}^4
	\Bigr]
	\Biggr)^{\! \nf14}
	&\leq
	C\, m^{-100} \,,
	\\
	\label{e.coupled.Rz}
	\Biggl(
	\avsum_{z \in 3^n \Zd \cap \cu_K}
	\E\bigl[ |R_z|^2 \bigr]
	\Biggr)^{\! \nf 12}
	&\leq
	C\, \shom_{m-h}^{-1}\, m^{-100} \,,
	\\
	\label{e.coupled.hgradw}
	\E\Bigl[
	\| \hshell \nabla w_{D,e}^{(K)} \|_{\underline L^2(\cu_K)}^2
	\Bigr]
	&\leq
	C\, \shom_{m-h}^{-2} h^2 \,.
\end{align}
The shell oscillation bound~\eqref{e.coupled.shell.osc} follows from the~$L^4$ Poincar\'e inequality on each cube~$z+\cu_n$ and stationarity:
\begin{equation*}
	\E\Bigl[
	\bigl\|
	\hshell - \bar{\hshell}_z
	\bigr\|_{\underline L^4(z+\cu_n)}^4
	\Bigr]
	\leq
	C3^{4n}3^{-4(m-h)}
	\leq
	C m^{-400}
	\,,
\end{equation*}
where the last inequality uses the shell gradient bound~\eqref{e.nabla.kmn.Linfty}, the definition~\eqref{e.n.def.recurrence},~$\nu\leq1$, and~\ref{a.j.reg}, after increasing the threshold. For~\eqref{e.coupled.Rz}, H\"older's inequality on each~$z+\cu_n$ gives
\begin{equation*}
	|R_z|
	\leq
	\bigl\|
	\hshell - \bar{\hshell}_z
	\bigr\|_{\underline L^4(z+\cu_n)}
	\| \g_{D,z} \|_{\underline L^4(z+\cu_n)} \,;
\end{equation*}
averaging over~$z$, applying Cauchy-Schwarz with~\eqref{e.coupled.shell.osc} and the~$L^4$ Poincar\'e inequality on~$z+\cu_n$ together with~\eqref{e.response.hess.l8}, and using~\eqref{e.n.def.recurrence} to turn the factor~$3^n3^{-(m-h)}$ into a large negative power of~$\nu^{-1}m$, yields~\eqref{e.coupled.Rz}. For~\eqref{e.coupled.hgradw}, H\"older's inequality gives
\begin{equation*}
	\| \hshell \nabla w_{D,e}^{(K)} \|_{\underline L^2(\cu_K)}^2
	\leq
	\| \hshell \|_{\underline L^4(\cu_K)}^2
	\| \nabla w_{D,e}^{(K)} \|_{\underline L^4(\cu_K)}^2
	\leq
	\| \hshell \|_{\underline L^8(\cu_K)}^2
	\| \nabla w_{D,e}^{(K)} \|_{\underline L^8(\cu_K)}^2 \,.
\end{equation*}
Taking expectations and using Cauchy-Schwarz with the shell estimate~\eqref{e.kmn.Lp} for~$p=8$ and the response-field bound~\eqref{e.response.grad.l8} yields
\begin{equation*}
	\E\Bigl[
	\| \hshell \nabla w_{D,e}^{(K)} \|_{\underline L^2(\cu_K)}^2
	\Bigr]
	\leq
	\E\Bigl[
	\| \hshell \|_{\underline L^8(\cu_K)}^8
	\Bigr]^{\nf14}
	\E\Bigl[
	\| \nabla w_{D,e}^{(K)} \|_{\underline L^8(\cu_K)}^8
	\Bigr]^{\nf14}
	\leq
	C(d)\,\shom_{m-h}^{-2} h^2 \,,
\end{equation*}
proving~\eqref{e.coupled.hgradw}.

\smallskip

\emph{Step 2.} Pointwise expansion of the coupled quadratic form. By~\eqref{e.principal.Phat} specialized to~$e'=0$,
\begin{equation*}
	\bfAhom_{m-h}^{\nf 12}
	\G_{-\bar{\hshell}_z}
	\bfAhom_{m-h}^{-\nf 12}
	\begin{pmatrix}
		e_{D,z} \\ e
	\end{pmatrix}
	=
	\begin{pmatrix}
		e_{D,z}
		\\
		e - \shom_{m-h}^{-1}\bar{\hshell}_z\, e_{D,z}
		\end{pmatrix}\,,
\end{equation*}
so
\begin{equation}
	\label{e.coupled.pointwise}
	\biggl|
	\bfAhom_{m-h}^{\nf 12}
	\G_{-\bar{\hshell}_z}
	\bfAhom_{m-h}^{-\nf 12}
	\begin{pmatrix}
		e_{D,z} \\ e
	\end{pmatrix}
	\biggr|^2
	=
	1
	+
	|e_{D,z}|^2
	+
	\shom_{m-h}^{-2} |\bar{\hshell}_z\, e_{D,z}|^2
	-2\shom_{m-h}^{-1} e \cdot \bar{\hshell}_z\, e_{D,z} \,.
\end{equation}
Substituting~$\bar{\hshell}_z\, e_{D,z} = B_z - R_z$ in~\eqref{e.coupled.pointwise}, expanding, and using the triangle inequality gives
\begin{align}
	\label{e.coupled.pointwise.split}
	\lefteqn{
	\biggl|
	\bfAhom_{m-h}^{\nf 12}
	\G_{-\bar{\hshell}_z}
	\bfAhom_{m-h}^{-\nf 12}
	\begin{pmatrix}
		e_{D,z} \\ e
	\end{pmatrix}
	\biggr|^2
	}
	\qquad &
	\notag \\
	&\leq
	1
	+
	|e_{D,z}|^2
	+
	\shom_{m-h}^{-2} |B_z|^2
	-2\shom_{m-h}^{-1} e\cdot B_z
	+
	\shom_{m-h}^{-2} |R_z|^2
	+
	2\shom_{m-h}^{-2} |B_z|\,|R_z|
	+
	2\shom_{m-h}^{-1} |R_z| \,.
\end{align}

\smallskip

\emph{Step 3.} Bulk terms. By Jensen's inequality,
\begin{equation*}
	|e_{D,z}|^2
	\leq
	\fint_{z+\cu_n} |\nabla w_{D,e}^{(K)}|^2
	\qquad\text{and}\qquad
	|B_z|^2
	\leq
	\fint_{z+\cu_n} |\hshell \nabla w_{D,e}^{(K)}|^2 \,.
\end{equation*}
Averaging over the cubes~$z+\cu_n$ and using the energy identity~\eqref{e.setup.w-energy}, we obtain
\begin{align}
	\label{e.coupled.bulk}
	\lefteqn{
	\avsum_{z \in 3^n \Zd \cap \cu_K}
	\E\Bigl[
	|e_{D,z}|^2
	+
	\shom_{m-h}^{-2} |B_z|^2
	-2\shom_{m-h}^{-1} e\cdot B_z
	\Bigr]
	}
	\qquad &
	\notag \\
	&\leq
	\E\Bigl[
	\| \nabla w_{D,e}^{(K)} \|_{\underline L^2(\cu_K)}^2
	\Bigr]
	+
	\shom_{m-h}^{-2}
	\E\Bigl[
	\| \hshell \nabla w_{D,e}^{(K)} \|_{\underline L^2(\cu_K)}^2
	\Bigr]
	-2\shom_{m-h}^{-1}
	\E\Bigl[
	\fint_{\cu_K} e\cdot \hshell \nabla w_{D,e}^{(K)}
	\Bigr]
	\notag \\
	&=
	-\E\Bigl[
	\| \nabla w_{D,e}^{(K)} \|_{\underline L^2(\cu_K)}^2
	\Bigr]
	+
	\shom_{m-h}^{-2}
	\E\Bigl[
	\| \hshell \nabla w_{D,e}^{(K)} \|_{\underline L^2(\cu_K)}^2
	\Bigr] \,.
\end{align}

\smallskip

\emph{Step 4.} Oscillation terms. Using Cauchy-Schwarz together with~\eqref{e.coupled.Rz} and~\eqref{e.coupled.hgradw},
\begin{equation}
	\label{e.coupled.Rz.cross}
	\avsum_{z \in 3^n \Zd \cap \cu_K} \E\bigl[ |B_z|\,|R_z| \bigr] \leq \Biggl( \avsum_{z \in 3^n \Zd \cap \cu_K} \E\bigl[ |B_z|^2 \bigr] \Biggr)^{\!\nf 12} \Biggl( \avsum_{z \in 3^n \Zd \cap \cu_K} \E\bigl[ |R_z|^2 \bigr] \Biggr)^{\!\nf 12} \leq C \shom_{m-h}^{-2} h  m^{-100} \,,
\end{equation}
and, by Jensen's inequality,
\begin{equation}
	\label{e.coupled.Rz.linear}
	\avsum_{z \in 3^n \Zd \cap \cu_K}
	\E\bigl[ |R_z| \bigr]
	\leq
	C\, \shom_{m-h}^{-1}\, m^{-100} \,.
\end{equation}

\smallskip

\emph{Step 5.} Assembling the bound. Combining~\eqref{e.coupled.pointwise.split},~\eqref{e.coupled.bulk},~\eqref{e.coupled.Rz},~\eqref{e.coupled.Rz.cross}, and~\eqref{e.coupled.Rz.linear}, and using~\eqref{e.L.vs.Lnaught.again} and~\eqref{e.ass.valid.applied} to absorb the~$\shom_{m-h}^{-k}$ prefactors into the factor~$m^{-100}$, we obtain
\begin{align*}
	\lefteqn{
	\avsum_{z \in 3^n \Zd \cap \cu_K}
	\E\Biggl[
	\biggl|
	\bfAhom_{m-h}^{\nf 12}
	\G_{-\bar{\hshell}_z}
	\bfAhom_{m-h}^{-\nf 12}
	\begin{pmatrix}
		e_{D,z} \\ e
	\end{pmatrix}
	\biggr|^2
	\Biggr]
	} \qquad &
	\notag \\ &
	\leq
	1
	-
	\E\Bigl[
	\| \nabla w_{D,e}^{(K)} \|_{\underline L^2(\cu_K)}^2
	\Bigr]
	+
	\shom_{m-h}^{-2}
	\E\Bigl[
	\| \hshell \nabla w_{D,e}^{(K)} \|_{\underline L^2(\cu_K)}^2
	\Bigr]
	+
	C\, m^{-100}
	\,.
\end{align*}
Passing to the limit superior~$K\to\infty$ and applying first the energy bound~\eqref{e.coupled.hgradw} to the~$\hshell$-weighted term and then the response-energy bound~\eqref{e.response.energy} to the~$\| \nabla w_{D,e}^{(K)} \|_{\underline L^2(\cu_K)}^2$ term, whose explicit~$K$-dependent error vanishes in the limit, yields
\begin{multline}
	\label{e.coupled.estimate}
	\limsup_{K\to\infty}
	\avsum_{z \in 3^n \Zd \cap \cu_K}
	\E\Biggl[
	\biggl|
	\bfAhom_{m-h}^{\nf 12}
	\G_{-\bar{\hshell}_z}
	\bfAhom_{m-h}^{-\nf 12}
	\begin{pmatrix}
		( \nabla w_{D,e}^{(K)} )_{z+\cu_n} \\ e
	\end{pmatrix}
	\biggr|^2
	\Biggr]
	\\
	\leq
	1
	-
	\cstar(\log 3) h \shom_{m-h}^{-2}
	+
	\nondegconst \shom_{m-h}^{-2}
	+
	C(d)\, \shom_{m-h}^{-4} h^2
	+
	C\, m^{-100} \,.
\end{multline}
By the standing condition~$h \geq 1$ in~\eqref{e.r.def.recurrence} and the crude bound~\eqref{e.crude.shom.bnd}, we have~$\shom_{m-h}^{-4} h^2 \geq \shom_{m-h}^{-4} \geq c\, m^{-8} \geq c m^{-100}$ after increasing the threshold, so the additive remainder~$C m^{-100}$ in~\eqref{e.coupled.estimate} can be absorbed into the~$C\,\shom_{m-h}^{-4} h^2$ term at the price of enlarging the constant. This proves~\eqref{e.coupled.estimate.explicit}.
\end{proof}

\smallskip

The negative contribution of the Dirichlet response energy gives the lower bound on the ratio of diffusivities needed for the recurrence.

\begin{corollary}[Lower bound on the ratio of diffusivities]
\label{cor.lower.ratio}
Under the setup above, there exists a constant~$C(d) < \infty$ such that
\begin{equation}
	\label{e.lower.ratio.explicit}
	\shom_{m-h}\shom_m^{-1}
	\leq
	1
	+
	\bigl( - \cstar (\log 3) h + C(\log^2 m + \nondegconst) \bigr)\shom_{m-h}^{-2}
	+
	C \shom_{m-h}^{-4} h^2
	\,.
\end{equation}
\end{corollary}

\begin{proof}
Fix a unit vector~$e \in \Rd$ and set~$e'  \coloneqq  0$ in the preceding setup. Then
\begin{equation*}
	P
	=
	\bfAhom_{m-h}^{-\nf 12}
	\begin{pmatrix}
		0 \\ e
		\end{pmatrix}\,,
\end{equation*}
so the block form~\eqref{e.homs.defs.U} and finite-volume convergence~\eqref{e.homs.defs} yield
\begin{equation*}
	\lim_{K\to\infty}
	\E\bigl[ P \cdot \bfA_m(\cu_K) P \bigr]
	=
	\shom_{m-h}\shom_m^{-1} \,.
\end{equation*}
Here again the off-diagonal homogenized skew part vanishes by~\ref{a.j.iso}, and~$\shom_{m,*}^{-1}(\cu_K)\to\shom_m^{-1}$ by~\eqref{e.homs.defs}. By the basic split~\eqref{e.setup.basic-split} and the localization bound~\eqref{e.localization.error},
\begin{equation*}
	\shom_{m-h}\shom_m^{-1}
	\leq
	\limsup_{K\to\infty}
	\avsum_{z \in 3^n \Zd \cap \cu_K}
	\E\bigl[
	P_z \cdot \bfA_m(z+\cu_n) P_z
	\bigr]
	+ C\, m^{-10} \,.
\end{equation*}
Applying the principal-term estimate~\eqref{e.principal.reduction} followed by the estimate for the Dirichlet response~\eqref{e.coupled.estimate.explicit}, and using the choice of the threshold following~\eqref{e.r.def.recurrence} to absorb~$\log^2(\nu^{-1}m) \leq C \log^2 m$ into the constant, we obtain
\begin{equation}
	\label{e.lower.ratio.product}
	\shom_{m-h}\shom_m^{-1}
	\leq
	\bigl( 1 + C\, \shom_{m-h}^{-2}\log^2 m \bigr)
	\Bigl(
	1
	-
	\cstar(\log 3) h \shom_{m-h}^{-2}
	+
	\nondegconst \shom_{m-h}^{-2}
	+
	C\, \shom_{m-h}^{-4} h^2
	\Bigr)
	+
	C\, m^{-10} \,.
\end{equation}
We expand the product on the right of~\eqref{e.lower.ratio.product} as in the proof of Corollary~\ref{cor.upper.ratio}, with the only change being the sign of the linear~$h\shom_{m-h}^{-2}$ term and the additional contribution~$C\, \shom_{m-h}^{-4} h^2$ inside the second factor. Setting
\begin{equation*}
	A  \coloneqq  C\, \shom_{m-h}^{-2}\log^2 m \,,
	\quad
	B  \coloneqq  \cstar(\log 3) h\, \shom_{m-h}^{-2} \,,
	\quad
	N  \coloneqq  \nondegconst\, \shom_{m-h}^{-2} \,,
	\quad
	H  \coloneqq  C\, \shom_{m-h}^{-4} h^2 \,,
\end{equation*}
the right-hand side of~\eqref{e.lower.ratio.product} equals~$1 - B + N + H + A - AB + AN + AH + C m^{-10}$. Since~$A,B\geq0$, the term~$-AB$ may be discarded in an upper bound, and the remaining positive terms are bounded by~\eqref{e.shom.log.bound} as in the proof of Corollary~\ref{cor.upper.ratio}: the term~$A$ already has the required form and~$AN \leq C\, \nondegconst\, \shom_{m-h}^{-2}$. The new cross term~$AH$ satisfies, again by~\eqref{e.shom.log.bound},
\begin{equation*}
	AH
	=
	C\, h^2\, \shom_{m-h}^{-4}
	\bigl( \shom_{m-h}^{-2}\log^2 m \bigr)
	\leq
	C\, m^{-\nf 34}\, \shom_{m-h}^{-4} h^2
	\leq
	C\, \shom_{m-h}^{-4} h^2 \,.
\end{equation*}
Finally, the residual~$C m^{-10}$ is absorbed into~$C\,\shom_{m-h}^{-2}\log^2 m$ as in the proof of Corollary~\ref{cor.upper.ratio}, using~\eqref{e.crude.shom.bnd}. Combining the above bounds yields~\eqref{e.lower.ratio.explicit}.
\end{proof}

\begin{proof}[Proof of Proposition~\ref{p.one.step.sharp}]
Let~$n \geq M$ and~$h \in \N \cap [1, \shom_n]$, where~$M$ will be chosen sufficiently large depending only on~$(\cstar,\nu,\nondegconst,d)$. By increasing~$M$ if necessary, we may assume that~$n$ is large enough that the bracket bound following from~\eqref{e.sstar.lower.bound} and~\eqref{e.sL.growth} holds; in particular,~$\shom_n\leq C n^{\nf12}\log^{\nf92}n$. The upper bound~$h \leq \shom_n$ together with this bracket bound gives~$h \leq C\, n^{\nf 12}\log^{\nf 92} n = o(n)$, so for all sufficiently large~$n$ we have~$n+h \leq 2n$ and~$h \leq \frac{1}{400}(n+h)$. After increasing~$M$ if necessary, this guarantees that the scale conditions needed in Corollaries~\ref{cor.upper.ratio} and~\ref{cor.lower.ratio} hold with~$m = n+h$ (the lower bound~$h \geq 1$ is automatic from~$h \in \N$).

\smallskip

Applying Corollaries~\ref{cor.upper.ratio} and~\ref{cor.lower.ratio} with~$m = n+h$, and using the upper bound~$h \leq \shom_n$ to absorb~$C h^2 \shom_n^{-4}$ into~$C \shom_n^{-2}$, we obtain
\begin{equation}
	\label{e.proposition.recurrence.applied.upper}
	\frac{\shom_{n+h}}{\shom_n}
	\leq
	1 + \cstar (\log 3) h \shom_n^{-2}
	+ C\bigl(\log^2(n+h) + \nondegconst \bigr) \shom_n^{-2}
	\,,
\end{equation}
and
\begin{equation}
	\label{e.proposition.recurrence.applied.lower}
	\frac{\shom_n}{\shom_{n+h}}
	\leq
	1 - \cstar (\log 3) h \shom_n^{-2}
	+ C\bigl(\log^2(n+h) + \nondegconst \bigr) \shom_n^{-2} \,.
\end{equation}
The upper bound~$h \leq \shom_n$ together with the lower bound from~\eqref{e.sstar.lower.bound} yields
\begin{equation*}
	\cstar (\log 3) h \shom_n^{-2}
	\leq
	\cstar (\log 3) \shom_n^{-1}
	\xrightarrow[n\to\infty]{} 0 \,,
\end{equation*}
where the rate depends only on~$(\cstar,\nu,d)$. After increasing~$M$ if necessary, with the same permitted parameter dependence, we may therefore assume~$\cstar (\log 3) h \shom_n^{-2} \leq \frac 12$. Since~$C\bigl(\log^2(n+h)+\nondegconst\bigr)\shom_n^{-2} \geq 0$, the quantity
\begin{equation*}
	\cstar (\log 3) h \shom_n^{-2}
	- C\bigl(\log^2(n+h)+\nondegconst\bigr)\shom_n^{-2}
\end{equation*}
is also bounded above by~$\frac 12$, and is therefore strictly less than~$1$ regardless of its sign. Rewriting~\eqref{e.proposition.recurrence.applied.lower} as
\begin{equation*}
	\frac{\shom_n}{\shom_{n+h}}
	\leq
	1
	- \Bigl( \cstar (\log 3) h \shom_n^{-2}
	- C\bigl(\log^2(n+h)+\nondegconst\bigr)\shom_n^{-2} \Bigr)
\end{equation*}
and applying the elementary inequality~$\frac{1}{1-x} \geq 1+x$ for~$x < 1$ to the right-hand side, we obtain
\begin{equation*}
	\frac{\shom_{n+h}}{\shom_n}
	\geq
	1 + \cstar (\log 3) h \shom_n^{-2}
	- C\bigl( \log^2(n+h) + \nondegconst \bigr) \shom_n^{-2} \,.
\end{equation*}
Combining this with~\eqref{e.proposition.recurrence.applied.upper} yields
\begin{equation*}
	\biggl|
	\frac{\shom_{n+h}}{\shom_n}
	- 1 - \cstar(\log 3) h \shom_n^{-2}
	\biggr|
	\leq
	C\bigl( \log^2(n+h) + \nondegconst \bigr) \shom_n^{-2} \,.
\end{equation*}
Multiplying through by~$\shom_n$ and using~$n+h \leq 2n$ (so that~$\log^2(n+h) \leq C \log^2 n$), we obtain~\eqref{e.approximate.recurrence}.
\end{proof}

\section{Deduction of the large-scale regularity theorem}
\label{s.sharp.C.one}

The proof of Theorem~\ref{t.C1beta} is based on the classical idea of \emph{regularity by harmonic approximation}: if a given function can be well-approximated by harmonic functions on a range of length scales, then it inherits some regularity properties on those scales. In the context of uniformly elliptic homogenization, this leads to a large-scale Lipschitz estimate on solutions, meaning that solutions possess~$L^2$ oscillation bounds which have a Lipschitz-type scaling but are valid only on scales above a (random) multiple of the correlation length scale (see for instance~\cite[Theorem~1.21]{AK.Book}). One can prove higher-order regularity statements which assert that a general solution can be well-approximated by \emph{corrected polynomials} (solutions which are close to harmonic polynomials) with approximation errors that scale like a Taylor remainder (see for instance~\cite[Section 3.3]{AKMBook}).

\smallskip

As discussed in the introduction, what is different in our context is that we cannot expect this Lipschitz-type estimate to hold across an infinite range of length scales, as it is inconsistent with superdiffusivity. Nor do we expect stationary sublinear correctors with a prescribed deterministic slope to exist in infinite volume. In particular, a solution which is \emph{flat}---close to an affine function---at a certain scale will typically deviate substantially from this affine function on other scales. The key observation which allows the regularity iteration to work is that solutions should nevertheless be \emph{flat at every scale}, even if the slopes of the affine approximations change across scales. This is the main step of the proof of Theorem~\ref{t.C1beta}, and the argument appears below (as part of an induction) in Steps~3--5 of the proof of Proposition~\ref{p.C1beta.sharp}.

\subsection{Estimates for the regularity argument}
\label{ss.sharp.scale.inputs}
\label{ss.Dirichlet.blackbox}

The next lemma collects the consequences of Proposition~\ref{p.minimal.scales} and Theorem~\ref{t.sstar.sharp.bounds} that are used in the regularity proof below and in the Dirichlet section. Proposition~\ref{p.minimal.scales} applies directly to the full centered field~$\a-(\k)_{\cu_m}$. Theorem~\ref{t.sstar.sharp.bounds} bounds the resulting error by~$\delta_m$ below. Given parameters~$\rho \in (0,1)$ and~$\varepsilon\in(0,1]$, we set, for every scale~$j\geq2$,
\begin{equation}
	\label{e.Dir.new.delta}
	\delta_j
	\coloneqq
	\varepsilon j^{-\frac12(1-\rho)} \log j\,.
\end{equation}
The error term supplied directly by Proposition~\ref{p.minimal.scales} is of size~$\shom_m^{-1}m^{\rho/2}\log m$; the deterministic lower threshold below is chosen so that this is dominated by~$\delta_m$. For a parameter~$M\geq1$, we set
\begin{equation}
	\label{e.Dir.new.hm}
	h_m\coloneqq \lceil M\log m\rceil\,.
\end{equation}

\begin{lemma}[Consequences of the homogenization estimates]
\label{l.sharp.scale.inputs}
There exists~$C(d)<\infty$ such that the following holds. For every~$\varepsilon\in(0,1]$,~$\rho \in (0,1)$ and~$M\geq C$, with~$\delta_j$ and~$h_j$ defined in~\eqref{e.Dir.new.delta} and~\eqref{e.Dir.new.hm}, there exist a constant~$\hat{L}(\ep,\rho,M,\cstar,\nu,\nondegconst,d)$ and a stationary family of random minimal scales~$\{\X_0(y):y\in\Rd\}$ such that, uniformly in~$y\in\Rd$,
\begin{equation}
	\label{e.Dir.new.minscale.tail}
	\log \X_0(y)
	=
	\O_{\Gamma_{\rho}}(\hat L)
	\,.
\end{equation}
We write~$\X_0\coloneqq\X_0(0)$. If~$m,n\in\N$ satisfy
\begin{equation}
	\label{e.Dir.new.window}
	3^m\geq \X_0\,,
	\qquad
	m\geq \hat L\,,
	\qquad
	m-h_m \leq n\leq m\,,
\end{equation}
then the following estimates hold.

\begin{itemize}
\item \emph{Full-field homogenization error.} We have
\begin{equation}
	\label{e.Dir.new.full.good}
	\max_{z\in3^{n-3}\Zd\,:\, z+\cu_n\subseteq\cu_m}
	\mathcal E_{\nf19,2}\bigl(
	z+\cu_n\,;\, \a-(\k)_{\cu_m}, \shom_m\Id
	\bigr)
	\leq
	\delta_m\,.
\end{equation}

\item \emph{Coarse-grained ellipticity.} We have
\begin{equation}
	\label{e.Dir.new.reg.ellipticity}
	\max_{z\in3^{n-3}\Zd\,:\, z+\cu_n\subseteq\cu_m}
	\Bigl(
	\shom_m^{-1} \Lambda_{\nf 14,1}(z+\cu_n \, ; \a-(\k)_{\cu_m})
	+
	\shom_m \lambda_{\nf 14,1}^{-1}(z+\cu_n\, ; \a-(\k)_{\cu_m})
	\Bigr)
	\leq C\,.
\end{equation}

\item \emph{Weak flux and harmonic approximation.} If~$z \in 3^{n-3}\Zd$ satisfies~$z+\cu_n\subseteq\cu_m$ and~$u\in\A(z+\cu_n;\a)$, then
\begin{align}
	\label{e.Dir.new.weak.flux}
	3^{-\frac{1}{4} n}
	\|(\a-(\k)_{\cu_m} -\shom_m\Id)\nabla u\|_{\underline{H}^{-\nf 14}(z+\cu_n)}
	&\leq
	C\shom_m^{\nf12}\delta_m
	\nu^{\nf12}\|\nabla u\|_{\underline{L}^2(z+\cu_n)}\,,
	\\
	\label{e.Dir.new.weak.grad}
	3^{-\frac{1}{4} n}
	\|\nabla u\|_{\underline{H}^{-\nf 14}(z+\cu_n)}
	&\leq
	C\shom_m^{-\nf12}
	\nu^{\nf12}\|\nabla u\|_{\underline{L}^2(z+\cu_n)}\,.
\end{align}
Moreover, there exists a harmonic function~$w$ in~$z+\cu_{n-1}$ such that
\begin{equation}
	\label{e.Dir.new.harmonic.approx}
	3^{-n}\|u-w\|_{\underline{L}^2(z+\cu_{n-1})}
	\leq
	C\delta_m
	\shom_m^{-\nf12}\nu^{\nf12}
	\|\nabla u\|_{\underline{L}^2(z+\cu_n)}\,.
\end{equation}

\item \emph{Local matrices and coefficient bounds.} We have
\begin{equation}
	\label{e.Dir.new.sstar.close}
	\max_{z\in3^{n-3}\Zd\,:\, z+\cu_n\subseteq\cu_m} \Bigl(
	\bigl|
	\shom_m^{-1}\s(z+\cu_n;\a)-\Id
	\bigr|
	+
	\bigl|
	\shom_m^{-1}\s_*(z+\cu_n;\a)-\Id
	\bigr|
	\Bigr)
	\leq
	C\delta_m\,,
\end{equation}
and
\begin{equation}
	\label{e.Dir.new.k.bounds}
	\begin{aligned}
		&m^{-1}\|\k-(\k)_{\cu_m}\|_{L^\infty(\cu_m)}
		+3^{-\frac{1}{4} m}
		[\k-(\k)_{\cu_m}]_{\Hminushat{-\nf14}(\cu_m)}
		\leq m^\rho\,,
		\\
		&3^{-\frac{1}{2d}m}
		[\k-(\k)_{\cu_m}]_{\Hminushat{-\nf1{2d}}(\cu_m)}
		\leq m^{\nf\rho2}\,.
	\end{aligned}
\end{equation}
\end{itemize}
For every~$y\in\Rd$ and~$m\geq\hat L$, on the event~$\{3^m\geq\X_0(y)\}$, all the conclusions above hold simultaneously for every~$z\in3^{m-h_m}\Zd\cap\cu_m$,~$k\in[m-h_m,m]\cap\N$ with~$k\geq\hat L$, and~$\ell\in[k-h_k,k]\cap\N$, after replacing the coefficient field~$x\mapsto\a(x)$ by~$x\mapsto\a(y+z+x)$ and~$(m,n)$ by~$(k,\ell)$. Every deterministic consequence of these conclusions holds on the same event. In particular, no additional condition~$3^k\geq\X_0(y+z)$ is needed.
\end{lemma}

\begin{proof}
Let~$\theta\coloneqq\rho/2$. Choose~$c_1=c_1(\cstar,d)\in(0,1]$ below. Apply Proposition~\ref{p.minimal.scales}, with exponent~$\theta$,~$s\coloneqq\nf19$, smallness parameter~$\delta \coloneqq c_1\varepsilon$, and with its~$M$-parameter replaced by a dimensional multiple of the present~$M$, chosen so that the logarithmic window in Proposition~\ref{p.minimal.scales} contains~$[m-h_m,m]$. Let~$\X_E$ be the resulting minimal scale, and let~$L_*$ dominate the deterministic threshold in Proposition~\ref{p.minimal.scales} with these parameters. Then, by Lemma~\ref{l.Gamma.sigma.triangle},
\begin{equation*}
	\log\X_E=\O_{\Gamma_\rho}(L_*)\,,
\end{equation*}
since~$2\theta=\rho$.

\smallskip

We next take two minimal scales from Lemma~\ref{l.ellip.k.scales.estimates}:~$\mathcal K_\rho$ with Sobolev exponent~$s=\nf14$ and~$\mathcal K_{\rho/2}$ with Sobolev exponent~$s=\nf1{2d}$, in both cases with~$\delta=\nf12$. We first set
\begin{equation*}
	\widetilde\X_0\coloneqq \X_E\vee\mathcal K_\rho\vee\mathcal K_{\rho/2}\,.
\end{equation*}
Applying the same measurable construction to the translated coefficient field~$x\mapsto\a(y+x)$ defines~$\widetilde\X_0(y)$. By stationarity,~\eqref{e.mathcal.K.int}, and Lemma~\ref{l.Gamma.sigma.triangle}, after increasing~$L_*$ to dominate the deterministic prefactors,
\begin{equation*}
	\log\widetilde\X_0(y)=\O_{\Gamma_\rho}(L_*)
	\qquad(y\in\Rd)\,.
\end{equation*}
We now fix the constant~$\hat L$ in the statement, with a sufficiently large constant~$C(d,M,\rho)$, so that
\begin{equation*}
	\hat L
	\geq
	C L_*
	\bigl(1+\log(2+L_*)\bigr)^{\nf1\rho}
	\,,
\end{equation*}
and so that~$h_m\leq m/2$ and every deterministic lower bound used below holds for~$m\geq\hat L$. The open-cube convention gives~$|3^{m-h_m}\Zd\cap\cu_m|=3^{dh_m}$. Since~$h_m=\lceil M\log m\rceil$, stationarity and a union bound give, for every~$m\geq\hat L$ and~$y\in\Rd$,
\begin{align*}
	\P\biggl[
	\max_{z\in3^{m-h_m}\Zd\cap\cu_m}
	\widetilde\X_0(y+z)>3^{m-h_m}
	\biggr]
	&\leq
	3^{dh_m}
	\exp\biggl(-c\biggl(\frac {m-h_m}{L_*}\biggr)^\rho\biggr)
	\leq
	\exp\biggl(-c\biggl(\frac m{L_*}\biggr)^\rho\biggr)
	\,.
\end{align*}
We define the minimal scale in the statement directly by
\begin{equation*}
	\X_0(y)\coloneqq\sup\biggl(\bigl\{3^{\lceil\hat L\rceil}\bigr\} \cup\biggl\{3^{m+1}:m\in\N,\ m\geq\lceil\hat L\rceil,\ \max_{z\in3^{m-h_m}\Zd\cap\cu_m}\widetilde\X_0(y+z)>3^{m-h_m}\biggr\}\biggr)\,.
\end{equation*}
The preceding summable bound shows that this stationary construction is finite almost surely. Another union bound gives, for every~$r\in\N$ with~$r\geq\hat L$,
\begin{align*}
	\P\bigl[\log_3\X_0(y)>r\bigr]
	&\leq
	\sum_{m\geq r}
	\P\biggl[
	\max_{z\in3^{m-h_m}\Zd\cap\cu_m}
	\widetilde\X_0(y+z)>3^{m-h_m}
	\biggr]
	\\
	&\leq
	\exp\biggl(-c\biggl(\frac r{L_*}\biggr)^\rho\biggr)
	\leq
	\exp\biggl(-\biggl(\frac r{C\hat L}\biggr)^\rho\biggr)\,.
\end{align*}
Increasing~$\hat L$ once more, without changing~$\X_0(y)$, proves~\eqref{e.Dir.new.minscale.tail}. If~$3^m\geq\X_0(y)$, then the definition gives
\begin{equation}
	\label{e.Dir.new.translated.scales}
	\max_{z\in3^{m-h_m}\Zd\cap\cu_m}
	\widetilde\X_0(y+z)
	\leq
	3^{m-h_m}\,.
\end{equation}
Consequently, for every~$z\in3^{m-h_m}\Zd\cap\cu_m$ and~$k\in[m-h_m,m]\cap\N$ with~$k\geq\hat L$,~\eqref{e.Dir.new.translated.scales} gives
\begin{equation*}
	\widetilde\X_0(y+z)\leq3^{m-h_m}\leq3^k\,.
\end{equation*}
The argument below therefore applies to the field~$x\mapsto\a(y+z+x)$ at scale~$k$, on the same event for all such~$z$ and~$k$.

\smallskip

Taking~$y=0$ in~\eqref{e.Dir.new.translated.scales} and using the grid point~$z=0$ shows that~$3^m\geq\X_0$ implies~$3^m\geq\widetilde\X_0$. The two applications of~\eqref{e.kbounds.minscale.form} therefore give
\begin{equation*}
	\begin{aligned}
		&m^{-1}\|\k-(\k)_{\cu_m}\|_{L^\infty(\cu_m)}
		+3^m\|\nabla\k-\nabla\k_m\|_{L^\infty(\cu_m)}
		+3^{-\frac{1}{4} m}
		[\k-(\k)_{\cu_m}]_{\Hminushat{-\nf14}(\cu_m)}
		\leq\frac12m^\rho\,,
		\\
		&3^{-\frac{1}{2d}m}
		[\k-(\k)_{\cu_m}]_{\Hminushat{-\nf1{2d}}(\cu_m)}
		\leq\frac12m^{\nf\rho2}\,.
	\end{aligned}
\end{equation*}
Dropping the nonnegative middle term in the first line proves~\eqref{e.Dir.new.k.bounds}.

\smallskip

We now fix the constant~$c_1=c_1(\cstar,d)>0$ introduced above. Using~\eqref{e.sm.sharp.bounds}, we choose~$c_1$ small enough and increase~$\hat L$ so that, for every~$m\geq\hat L$,
\begin{equation}
	\label{e.Dir.new.absorb}
	C c_1\varepsilon\shom_m^{-1}m^\theta\log m
	\leq
	\delta_m\,.
\end{equation}
This is possible because~$\theta-\nf12=-\frac12(1-\rho)$ and~$\delta_m=\varepsilon m^{-\frac12(1-\rho)}\log m$.

\smallskip

Fix~$m,n$ satisfying~\eqref{e.Dir.new.window}. Since~$n\geq m-h_m$ lies inside the window of Proposition~\ref{p.minimal.scales}, that proposition, applied to the full centered field~$\a-(\k)_{\cu_m}$, gives
\begin{equation*}
	\max_{z\in3^{n-3}\Zd\,:\, z+\cu_n\subseteq\cu_m}
	\mathcal E_{\nf19,2}\bigl(z+\cu_n\,;\,\a-(\k)_{\cu_m},\shom_m\Id\bigr)
	\leq
	C c_1\varepsilon\shom_m^{-1}m^\theta\log m
	\leq
	\delta_m\,,
\end{equation*}
where the last step is~\eqref{e.Dir.new.absorb}. This proves the estimate~\eqref{e.Dir.new.full.good}. By~\eqref{e.mathcal.E.monotone.s}, the same bound is available with~$\nf19$ replaced by any larger exponent, in particular by~$\nf18$. Using~\eqref{e.ellipticities.qone.by.qtwo} and~\eqref{e.bound.Lambdas.by.Es} with comparison matrix~$\shom_m\Id$ and the choice of~$\hat L$, which ensures that~$\delta_m\leq1$, we obtain
\begin{equation*}
	\max_{z\in3^{n-3}\Zd\,:\, z+\cu_n\subseteq\cu_m}
	\Bigl(
	\shom_m^{-1}\Lambda_{\nf14,1}(z+\cu_n\,;\a-(\k)_{\cu_m})
	+
	\shom_m\lambda_{\nf14,1}^{-1}(z+\cu_n\,;\a-(\k)_{\cu_m})
	\Bigr)
	\leq C\,.
\end{equation*}

\smallskip

Next, fix~$z\in3^{n-3}\Zd$ satisfying~$z+\cu_n\subseteq\cu_m$ and put~$Q\coloneqq z+\cu_n$. Define
\begin{equation*}
	\tilde\a_m
	\coloneqq
	\shom_m^{-1}\bigl(\a-\shom_m\Id-(\k)_{\cu_m}\bigr)
	\qand
	\a_m^{\rm reg}\coloneqq
	\shom_m^{-1}\bigl(\a-(\k)_{\cu_m}\bigr)
	=\Id+\tilde\a_m\,,
\end{equation*}
and let~$u\in\A(Q;\a)$. Since subtracting a constant skew matrix and multiplying by~$\shom_m^{-1}$ do not change the solution space,~$u$ is also~$\a_m^{\rm reg}$-harmonic. By~\eqref{e.mathcal.E.scalar.skew.normalize}, the estimate~\eqref{e.Dir.new.full.good} gives in particular
\begin{equation*}
	\mathcal E_{\nf19,2}\bigl(Q\,;\,\a_m^{\rm reg},\Id\bigr)\leq\delta_m\,.
\end{equation*}
By~\eqref{e.mathcal.E.monotone.s}, the same bound holds with~$\nf19$ replaced by~$\nf14$. If~$\|\nabla u\|_{\underline L^2(Q)}=\infty$, the estimate below is vacuous. Otherwise, Poincar\'e's inequality gives~$u$, modulo a constant, in~$H^1(Q)$; harmonicity and the normalized coefficient then put~$u$ in the finite-energy class required by Proposition~\ref{p.big.black.box}, with its~$H^1$ trace. Translate and rescale~$Q$ to~$\cu_0$, apply Proposition~\ref{p.big.black.box} in~$Q$ with comparison matrix~$\Id$, and let~$w$ be the harmonic function with the same boundary values as~$u$. Scaling back in~\eqref{e.bbb.homogenization} gives
\begin{equation*}
	3^{-\frac{1}{4} n}
	\Bigl(
	\|\nabla u-\nabla w\|_{\underline{H}^{-\nf 14}(Q)}
	+
	\|\a_m^{\rm reg}\nabla u-\nabla w\|_{\underline{H}^{-\nf 14}(Q)}
	\Bigr)
	\leq
	C\delta_m
	\shom_m^{-\nf12}\nu^{\nf12}
	\|\nabla u\|_{\underline{L}^2(Q)}\,.
\end{equation*}
Subtracting the two terms and multiplying by~$\shom_m$ gives~\eqref{e.Dir.new.weak.flux}. Since~$u-w\in H^1_0(Q)$, the elementary zero-boundary estimate
\begin{equation*}
	3^{-n}\|u-w\|_{\underline{L}^2(Q)}
	\leq
	C3^{-\frac{1}{4} n}
	\|\nabla u-\nabla w\|_{\underline{H}^{-\nf 14}(Q)}
	\,,
\end{equation*}
gives~\eqref{e.Dir.new.harmonic.approx}, after restricting the left side from~$Q=z+\cu_n$ to~$z+\cu_{n-1}$ and absorbing the dimensional volume factor.

\smallskip

To prove~\eqref{e.Dir.new.weak.grad}, we use the positive Besov form of the coarse-grained Poincar\'e inequality with exponent~$\nf79$. Since~$\frac12(1-\nf79)=\nf19$ and~$\delta_m\leq1$, the bound~$\mathcal E_{\nf19,2}(Q;\a_m^{\rm reg},\Id)\leq\delta_m$ allows us to apply~\eqref{e.Poinked} to~$\a_m^{\rm reg}$ with this exponent. The symmetric part of~$\a_m^{\rm reg}$ is exactly~$(\nu/\shom_m)\Id$, and hence
\begin{equation}
	\label{e.Dir.new.Besov.buffer}
	\|u-(u)_Q\|_{\underline{B}_{2,\infty}^{\nf79}(Q)}
	\leq
	C3^{\frac{2}{9} n}\shom_m^{-\nf12}\nu^{\nf12}
	\|\nabla u\|_{\underline{L}^2(Q)}\,.
\end{equation}
By duality in~$\underline H^{-\nf14}$, for every~$\Phi\in H^{\nf14}_0(Q;\Rd)$,
\begin{equation*}
	\biggl|\fint_Q \nabla u\cdot\Phi\biggr|=\biggl|\fint_Q (u-(u)_Q)\nabla\cdot\Phi\biggr|\leq C\|u-(u)_Q\|_{\underline H^{\nf34}(Q)}\|\Phi\|_{\underline H^{\nf14}(Q)}\,.
\end{equation*}
Thus, using the equivalence of~$\underline H^{\nf34}$ and~$\underline B_{2,2}^{\nf34}$ from~\cite[Appendix~A]{AK.HC}, and the rescaled Besov embedding~$\|\cdot\|_{\underline B_{2,2}^{\nf34}(Q)}\leq C3^{\frac{1}{36} n}\|\cdot\|_{\underline B_{2,\infty}^{\nf79}(Q)}$, we get
\begin{align*}
	3^{-\frac{1}{4} n}
	\|\nabla u\|_{\underline{H}^{-\nf 14}(Q)}
	&\leq
	C3^{-\frac{1}{4} n}3^{\frac{1}{36} n}
	\|u-(u)_Q\|_{\underline{B}_{2,\infty}^{\nf79}(Q)}
	\leq
	C \shom_m^{-\nf12}\nu^{\nf12}
	\|\nabla u\|_{\underline{L}^2(Q)}
	\,.
\end{align*}
This proves~\eqref{e.Dir.new.weak.grad}.

\smallskip

It remains to prove~\eqref{e.Dir.new.sstar.close}. Fix~$z\in3^{n-3}\Zd$ satisfying~$z+\cu_n\subseteq\cu_m$ and put~$Q\coloneqq z+\cu_n$. By~\eqref{e.sum.CE.FE}, the bound~\eqref{e.Dir.new.full.good} and the monotonicity~\eqref{e.mathcal.E.monotone.s} control the corresponding~$\CE_{\nf 14}$ and~$\FE_{\nf 14}$ quantities for~$\a_m^{\rm reg}$ by~$C\delta_m$. Looking at the top-scale term in the first sum in~\eqref{e.FE.def}, we get
\begin{equation*}
	\bigl|
	\s_*^{-\nf12}(Q;\a_m^{\rm reg})
	\bigl(\s_*(Q;\a_m^{\rm reg})-\Id\bigr)
	\bigr|
	\leq C\delta_m\,.
\end{equation*}
The scaling of~\eqref{e.Dir.new.reg.ellipticity} gives~$\Lambda_{\nf 14,1}(Q;\a_m^{\rm reg})\leq C$, and hence~$|\s_*(Q;\a_m^{\rm reg})|\leq C$ by~\eqref{e.CG.bounds.1} and~\eqref{e.ellipticities.monotone.ordered}. This converts the normalized estimate into
\begin{equation*}
	\bigl|\s_*(Q;\a_m^{\rm reg})-\Id\bigr|
	\leq C\delta_m\,.
\end{equation*}
Similarly, the top-scale term in~\eqref{e.CE.def} gives
\begin{equation*}
	\bigl|\s(Q;\a_m^{\rm reg})-\s_*(Q;\a_m^{\rm reg})\bigr|
	\leq
	C\delta_m^2
	\leq
	C\delta_m\,,
\end{equation*}
where the last inequality follows from the choice of~$\hat L$, which ensures that~$\delta_m\leq1$. Combining the last two displays gives
\begin{equation*}
	\bigl|\s(Q;\a_m^{\rm reg})-\Id\bigr|
	+
	\bigl|\s_*(Q;\a_m^{\rm reg})-\Id\bigr|
	\leq C\delta_m\,.
\end{equation*}
The relations
\begin{equation*}
	\s(Q;\a_m^{\rm reg})=\shom_m^{-1}\s(Q;\a)\,,
	\qquad
	\s_*(Q;\a_m^{\rm reg})=\shom_m^{-1}\s_*(Q;\a)\,,
\end{equation*}
which follow from scaling and invariance under constant skew perturbations, give~\eqref{e.Dir.new.sstar.close}. This completes the proof.
\end{proof}

\smallskip

Lemma~\ref{l.sharp.scale.inputs} is the probabilistic part of the regularity argument. We now turn it into the deterministic iteration which produces flatness at every scale, the Liouville theorem, and the large-scale~$C^{1,\gamma}$ estimate.

\begin{proposition}
\label{p.C1beta.sharp}
Let~$\gamma, \rho \in (0,1)$. There exist constants~$C_{\mathrm{tail}}(\gamma,\rho,\nu,\cstar,\nondegconst,d)<\infty$ and~$C(\gamma,d)<\infty$, and a random variable~$\X\geq1$ satisfying
\begin{equation*}
	\log \X \leq \O_{\Gamma_\rho} \bigl( C_{\mathrm{tail}} \bigr)
\end{equation*}
such that the following statements are valid.
\begin{enumerate}
\item \underline{Liouville theorem.} Almost surely with respect to~$\P$, the space~$\mathcal{A}^{1+\gamma} (\Rd)$ has dimension~$1+d$ and does not depend on~$\gamma$.

\item \underline{Flatness at every scale.} For every~$\phi\in\A^{1+\gamma}(\Rd)$ and~$r \geq \X$, we have
\begin{equation*}
	\inf_{e\in\Rd} \| \phi - \linear_e - (\phi)_{B_r} \|_{\underline{L}^2(B_r)}
	\leq
	C(\log r)^{-\frac12(1-\rho)}\log\log r\, \| \phi \|_{\underline{L}^2(B_r)} \,.
\end{equation*}

\item \underline{Large-scale~$C^{1,\gamma}$ estimate.} For each~$R\in [\X,\infty)$ and~$u \in \mathcal{A}(B_R)$, there exists~$\phi \in \mathcal{A}^{1+\gamma} (\Rd)$ such that
\begin{equation*}
	\| \nabla u - \nabla \phi\|_{\underline{L}^2(B_r)} \leq C \Bigl( \frac{r}{R} \Bigr)^{\!\gamma} \| \nabla u\|_{\underline{L}^2(B_R)}
	\,, \quad \forall r \in [\X,R)\,.
\end{equation*}
\end{enumerate}
\end{proposition}

\begin{proof}
Put
\begin{equation*}
	\beta\coloneqq \frac12(1-\rho)\,.
\end{equation*}
We also fix~$\eta\in(\gamma,1)$ with~$\eta\geq \frac12$, a number~$\kappa\in(0,\gamma\wedge\frac12(\eta-\gamma))$, and a number~$\varepsilon\in(0,1]$, to be chosen sufficiently small in terms of~$(\eta,\kappa,\rho,\nu,\cstar,d)$. All constants denoted by~$C$ below may depend on~$(\gamma,d)$. Dependence on the other displayed parameters is absorbed into the minimal scale and the tail constant~$C_{\mathrm{tail}}$.

\smallskip

\emph{Step 1.} Choice of the minimal scale. Let~$M$ be a deterministic constant with~$M\geq C(d)$, to be selected large enough depending on~$(\eta,d)$. The value of~$M$ is fixed in the excess-decay paragraph below; all objects from Lemma~\ref{l.sharp.scale.inputs} are understood for this final choice. Apply Lemma~\ref{l.sharp.scale.inputs} with the parameters~$(\varepsilon,\rho,M)$. We use the notation~$\delta_j$,~$h_j$,~$\hat L$ and~$\X_0$ from that lemma. The estimates in the lemma are already stated for the full centered field~$\a-(\k)_{\cu_j}$. The deterministic threshold~$\hat L$ may be enlarged below by constants depending only on the displayed parameters. Let~$m_0=m_0(\rho,\beta,\eta,\kappa,\varepsilon,d)\in\N$ be a deterministic offset, chosen sufficiently large below. Fix also a dimensional constant~$C_0(d)$ with~$C_0(d)\geq1$, to be selected large enough in the final passage from cubes to balls, and define
\begin{equation}
	\label{e.sharp.Cone.mstar}
	m_* \coloneqq m_0+\bigl(\lceil\hat L\rceil\vee \lceil \log_3\X_0\rceil\bigr)
	\qand
	\X\coloneqq C_0 3^{m_*}\,.
\end{equation}
By~\eqref{e.Dir.new.minscale.tail} and the deterministic definition of~$m_0$ and~$\hat L$, increasing~$C_{\mathrm{tail}}$ gives the asserted bound on~$\log\X$.

\smallskip

\emph{Step 2.} Basic deterministic estimates. We work on the event~$\{m_*<\infty\}$. The goal of this step is to record the deterministic estimates for one interval of scales: coarse-grained ellipticity, Caccioppoli, Poincar\'e, harmonic approximation, and the top-scale corrected-affine bounds. For each~$j$ with~$j\geq m_*$, write
\begin{equation*}
	\a_j^{\rm reg}
	\coloneqq
	\shom_j^{-1}\bigl(\a-(\k)_{\cu_j}\bigr)\,.
\end{equation*}
The equations for~$\a_j^{\rm reg}$ and for~$\a$ are the same in~$\cu_j$, since we have only subtracted a constant antisymmetric matrix and multiplied the operator by a positive scalar. By increasing~$m_0$ and decreasing~$\varepsilon$, we may assume throughout that~$\delta_j$ is smaller than the deterministic smallness thresholds in the estimates below. The deterministic asymptotic~\eqref{e.sm.sharp.bounds} makes the scalar normalizations comparable throughout logarithmic windows. We use this comparison to state the estimates in the~$\shom_j$-normalization.

\smallskip

\emph{Coarse-grained ellipticity and Caccioppoli.} The ellipticity estimate for the full coefficient field~\eqref{e.Dir.new.reg.ellipticity} gives
\begin{equation}
	\label{e.sharp.Cone.cg.ellipticity}
	\max_{ n \in \Z \cap [j-h_j,j]}
	\max_{z \in 3^{n-3}\Zd\,:\, z+\cu_n\subseteq\cu_j}
	\Bigl(
	\Lambda_{\nf 14,1}(z+\cu_n \, ;\a_j^{\rm reg})
	+
	\lambda_{\nf 14,1}^{-1}(z+\cu_n \, ;\a_j^{\rm reg})
	\Bigr)
	\leq C(d)\,.
\end{equation}
Applying the coarse-grained Caccioppoli estimate~\eqref{e.coarse.grained.Caccioppoli} to~$u-(u)_{\cu_k}$ for the field~$\a_j^{\rm reg}$, with~$s=t=\nf 25$, and then using~\eqref{e.sharp.Cone.cg.ellipticity}, gives a dimensional constant on the right side. This is legitimate with the~$q=1$ constants since~$\nf 14<\nf 25<1$ and the ordering~\eqref{e.ellipticities.monotone.ordered} gives~$\Lambda_{\nf 25,1}\leq\Lambda_{\nf 14,1}$ and~$\lambda_{\nf 25,1}^{-1}\leq\lambda_{\nf 14,1}^{-1}$ on the same cube. Since the symmetric part of~$\a_j^{\rm reg}$ is~$(\nu/\shom_j)\Id$, this yields, for every~$k\in\Z\cap[j-h_j+1,j]$ and every~$u\in\A(\cu_k)$,
\begin{equation}
	\label{e.sharp.Cone.caccioppoli}
	\shom_j^{-\nf12}\nu^{\nf12}
	\|\nabla u\|_{\underline{L}^2(\cu_{k-1})}
	\leq
	C3^{-k}\|u-(u)_{\cu_k}\|_{\underline{L}^2(\cu_k)}\,.
\end{equation}
For Poincar\'e, we use~\eqref{e.Poink} with exponent~$0$ and then use the same ordering to bound~$\lambda_{1,1}^{-1}$ by~$\lambda_{\nf 14,1}^{-1}$. Combining this with~\eqref{e.sharp.Cone.cg.ellipticity} gives, for every~$k\in\Z\cap[j-h_j,j]$ and every~$u\in\A(\cu_k)$,
\begin{equation}
	\label{e.sharp.Cone.poincare}
	\|u-(u)_{\cu_k}\|_{\underline{L}^2(\cu_k)}
	\leq
	C3^k\shom_j^{-\nf12}\nu^{\nf12}
	\|\nabla u\|_{\underline{L}^2(\cu_k)}\,.
\end{equation}

\smallskip

\emph{Harmonic approximation and top-scale correctors.} Let~$k\in\Z\cap[j-h_j+1,j]$ and let~$u\in\A(\cu_k)$. The harmonic approximation estimate~\eqref{e.Dir.new.harmonic.approx}, applied with top scale~$j$, gives a harmonic function~$w$ in~$\cu_{k-1}$ satisfying
\begin{equation}
	\label{e.sharp.Cone.harmonic}
	3^{-k}\|u-w\|_{\underline{L}^2(\cu_{k-1})}
	\leq
	C\delta_j\,
	\shom_j^{-\nf12}\nu^{\nf12}
	\|\nabla u\|_{\underline{L}^2(\cu_k)}\,.
\end{equation}
For the corrected affine function associated with the full coefficient field
\begin{equation*}
	v_{j,e}\coloneqq v(\cdot,\cu_j,e)
	=v(\cdot,\cu_j,0,\s_*(\cu_j)e\,;\a)\,,
\end{equation*}
we next prove the two top-scale estimates needed to start the corrected-affine iteration. The proof is deterministic, once the estimates of Lemma~\ref{l.sharp.scale.inputs} are available. First, the estimate for the full coefficient field~\eqref{e.Dir.new.full.good}, applied with~$m=n=j$ and~$z=0$, and Proposition~\ref{p.big.black.box}, applied to~$\a_j^{\rm reg}$ with~$\a_1=\Id$, give, for every~$e\in\Rd$, the standard corrected-affine corollary of the homogenization estimate:
\begin{equation}
	\label{e.sharp.Cone.corrector.flat}
	3^{-j}\inf_{c\in\R}
	\|v_{j,e}-\linear_e+c\|_{\underline{L}^2(\cu_j)}
	\leq
	C\delta_j|e|\,.
\end{equation}
Indeed, after rescaling~$\cu_j$ to~$\cu_0$, subtracting the constant skew matrix~$(\k)_{\cu_j}$ and dividing by~$\shom_j$ do not change the solution space. Let~$u_{j,e}$ be the Dirichlet solution with boundary data~$\linear_e$. The estimates~\eqref{e.bbb.homogenization} and~\eqref{e.attainability.estimate}, together with the zero-boundary estimate used above, show that~$u_{j,e}$ is within~$C\delta_j|e|$ of~$\linear_e$ in the normalized~$3^{-j}L^2$ norm. For the normalized field, the weak equation and~\eqref{e.J.def} bound the Dirichlet energy by~$2J(\cu_j,e,0;\a_j^{\rm reg})$; formula~\eqref{e.J.mat} and the top-scale terms in~\eqref{e.sum.CE.FE} show that its excess over~$e\cdot\s_*(\cu_j;\a_j^{\rm reg})e$ is at most~$C\delta_j^2|e|^2$. Therefore~\eqref{e.secondvar} and~\eqref{e.sharp.Cone.poincare} give
\begin{equation*}
	3^{-j}\inf_{c\in\R}
	\|v_{j,e}-u_{j,e}+c\|_{\underline{L}^2(\cu_j)}
	\leq C\delta_j|e|\,.
\end{equation*}
The triangle inequality proves~\eqref{e.sharp.Cone.corrector.flat}.

\smallskip

The nondegeneracy is an immediate consequence of the energy identity. By~\eqref{e.solution.p.q.avg.energy} and~\eqref{e.Dir.new.sstar.close}, applied with~$m=n=j$ and~$z=0$,
\begin{equation*}
	(1-C\delta_j)|e|^2
	\leq
	\shom_j^{-1}\nu
	\|\nabla v_{j,e}\|_{\underline{L}^2(\cu_j)}^2
	=
	e\cdot \shom_j^{-1}\s_*(\cu_j;\a)e
	\leq
	(1+C\delta_j)|e|^2\,.
\end{equation*}
After increasing~$m_0$ and decreasing~$\varepsilon$, this gives
\begin{equation}
	\label{e.sharp.Cone.corrector.nondeg}
	\frac12 |e|
	\leq
	\shom_j^{-\nf12}\nu^{\nf12}
	\|\nabla v_{j,e}\|_{\underline{L}^2(\cu_j)}
	\leq
	2|e|\,.
\end{equation}

\smallskip

\emph{Step 3.} One-block corrected-affine estimates. The goal of this step is to prove~\eqref{e.sharp.Cone.corrector.scale.flat} and~\eqref{e.sharp.Cone.corrector.scale.nondeg}. We run a deterministic iteration based on~\eqref{e.sharp.Cone.caccioppoli},~\eqref{e.sharp.Cone.poincare} and~\eqref{e.sharp.Cone.harmonic}. For~$k\in\Z\cap[j-h_j,j]$ and~$e\in\Rd$, let~$\ell_{k,j}[e]$ be the affine function minimizing
\begin{equation*}
	E_{k,j}[e]\coloneqq
	3^{-k}\|v_{j,e}-\ell_{k,j}[e]\|_{\underline{L}^2(\cu_k)}
\end{equation*}
and let~$P_{k,j}[e]\coloneqq\nabla\ell_{k,j}[e]$. Then the maps~$P_{k,j}$ are linear. We prove below that, for every~$j$ with~$j\geq m_*$ and every~$k\in\Z\cap[j-h_j,j]$, the map~$P_{k,j}$ is invertible and that, with~$Q_{k,j}\coloneqq P_{k,j}^{-1}$, for every~$p\in\Rd$,
\begin{equation}
	\label{e.sharp.Cone.corrector.scale.flat}
	3^{-k}\inf_{c\in\R}
	\|v_{j,Q_{k,j}[p]}-\linear_p+c\|_{\underline{L}^2(\cu_k)}
	\leq
	C\delta_j |p|
\end{equation}
and
\begin{equation}
	\label{e.sharp.Cone.corrector.scale.nondeg}
	C^{-1}|p|
	\leq
	\shom_j^{-\nf12}\nu^{\nf12}
	\|\nabla v_{j,Q_{k,j}[p]}\|_{\underline{L}^2(\cu_k)}
	\leq
	C|p|\,.
\end{equation}
We prove these estimates by a finite-dimensional induction on the scale. Put
\begin{equation*}
	D_{k,j}[e]\coloneqq
	\shom_j^{-\nf12}\nu^{\nf12}
	\|\nabla v_{j,e}\|_{\underline{L}^2(\cu_k)}\,.
\end{equation*}
Let~$N=N(d)$ be the integer selected in the harmonic approximation step below. Increasing~$m_0$, we assume that, for every~$j$ with~$j\geq m_*$,
\begin{equation*}
	N+2\leq h_j\leq \frac12 j\,.
\end{equation*}
The top-scale estimates~\eqref{e.sharp.Cone.corrector.flat} and~\eqref{e.sharp.Cone.corrector.nondeg} imply, after increasing a dimensional constant~$K$, that, for every~$k\in\Z\cap[j-N,j]$,
\begin{equation}
	\label{e.sharp.Cone.corrector.induction.top}
	E_{k,j}[e]\leq K\delta_j |P_{k,j}[e]|\,.
\end{equation}
Indeed, on these top scales the restriction from~$\cu_j$ to~$\cu_k$ only costs a dimensional factor depending on~$N$, and, for every~$k\in\Z\cap[j-N,j]$, the minimizing slopes satisfy
\begin{equation}
	\label{e.sharp.Cone.top.slope.close}
	|P_{k,j}[e]-e|\leq C_N\delta_j|e|\,.
\end{equation}

\smallskip

We now extend this estimate to all~$k\in\Z\cap[j-h_j,j]$. Assume that~$k_{\ast}\leq j-N-1$ and that the induction estimate
\begin{equation*}
	E_{k,j}[e]\leq K\delta_j |P_{k,j}[e]|
\end{equation*}
is already known for every scale~$k\in\Z\cap[k_{\ast}+1,j]$. Applying~\eqref{e.sharp.Cone.harmonic} at scale~$k_{\ast}+N+1$ and then approximating the harmonic comparison by its affine Taylor polynomial in~$\cu_{k_{\ast}}$, we obtain, after choosing~$N(d)$ so large that the harmonic regularity loss is at most~$\frac18$,
\begin{equation}
	\label{e.sharp.Cone.corrector.induction.step}
	E_{k_{\ast},j}[e]
	\leq
	\frac18 E_{k_{\ast}+N,j}[e]
	+
	C\delta_j D_{k_{\ast}+N+1,j}[e]\,.
\end{equation}
The elementary stability of best affine slopes gives, for consecutive scales,
\begin{equation}
	\label{e.sharp.Cone.slope.stability}
	|P_{r+1,j}[e]-P_{r,j}[e]|
	\leq
	C E_{r+1,j}[e]\,.
\end{equation}
Here the minimizer on~$\cu_{r+1}$ is admissible after restriction to~$\cu_r$, so~$E_{r,j}\leq C E_{r+1,j}$; this bounds the contribution from the smaller cube by the displayed term. Thus the induction hypothesis implies, for every~$r$ with~$k_{\ast}\leq r<j$,
\begin{equation*}
	(1+C K\delta_j)^{-1}|P_{r,j}[e]|
	\leq
	|P_{r+1,j}[e]|
	\leq
	(1+C K\delta_j)|P_{r,j}[e]|\,.
\end{equation*}
Since~$\delta_j h_j\leq c(d)$ after increasing~$m_0$ and decreasing~$\varepsilon$, iterating this inequality over at most~$h_j$ scales gives, for every~$r$ with~$k_{\ast}\leq r\leq j$,
\begin{equation}
	\label{e.sharp.Cone.slopes.comparable.in.block}
	|P_{r,j}[e]|\leq C |P_{k_{\ast},j}[e]|\,.
\end{equation}
Since~$N$ is fixed, after first increasing~$K$ enough for the closing inequality below and then decreasing~$\varepsilon$, the preceding consecutive-scale comparison also gives, for every~$r$ with~$k_{\ast}\leq r\leq j$ and~$r-k_{\ast}\leq N+2$,
\begin{equation*}
	|P_{r,j}[e]|\leq 2|P_{k_{\ast},j}[e]|\,.
\end{equation*}
Moreover, the Caccioppoli estimate~\eqref{e.sharp.Cone.caccioppoli} applied to~$v_{j,e}$ and the definition of~$E_{r,j}$ give, for every~$r$ with~$j-h_j+1\leq r\leq j$,
\begin{equation}
	\label{e.sharp.Cone.energy.by.slope.excess}
	D_{r-1,j}[e]\leq C\bigl(|P_{r,j}[e]|+E_{r,j}[e]\bigr)\,.
\end{equation}
Combining~\eqref{e.sharp.Cone.corrector.induction.step}, the preceding fixed-$N$ bound and~\eqref{e.sharp.Cone.energy.by.slope.excess}, with the crude upper bound in~\eqref{e.sharp.Cone.corrector.nondeg} used for the case~$k_{\ast}+N+1=j$, yields
\begin{equation*}
	E_{k_{\ast},j}[e]
	\leq
	\Bigl(\frac{K}{4}+C\Bigr)\delta_j |P_{k_{\ast},j}[e]|\,.
\end{equation*}
By the choice of~$K$, this gives, for every~$k\in\Z\cap[j-h_j,j]$,
\begin{equation}
	\label{e.sharp.Cone.corrector.induction}
	E_{k,j}[e]\leq K\delta_j |P_{k,j}[e]|\,.
\end{equation}
Fix any~$k\in\Z\cap[j-h_j,j]$. First observe that~\eqref{e.sharp.Cone.slope.stability} and~\eqref{e.sharp.Cone.corrector.induction} give, for every~$r\in\Z\cap[k,j-1]$,
\begin{equation*}
	|P_{r,j}[e]|
	\leq
	(1+CK\delta_j)|P_{r+1,j}[e]|\,.
\end{equation*}
For every~$r$ with~$k\leq r\leq j$, iterating this inequality from~$r$ up to~$j$ and using~$j-r\leq h_j$ and~$\delta_jh_j\leq c(d)$ yields
\begin{equation*}
	|P_{r,j}[e]|
	\leq
	(1+CK\delta_j)^{j-r}|P_{j,j}[e]|
	\leq
	C|P_{j,j}[e]|\,.
\end{equation*}
Summing~\eqref{e.sharp.Cone.slope.stability} from~$k$ to~$j$ then gives
\begin{align*}
	|P_{k,j}[e]-P_{j,j}[e]|
	&\leq
	\sum_{r=k}^{j-1}|P_{r+1,j}[e]-P_{r,j}[e]|
	\leq
	C\delta_j\sum_{r=k}^{j-1}|P_{r+1,j}[e]|
	\leq
	C\delta_j h_j |P_{j,j}[e]|\,.
\end{align*}
Since the top-scale bound gives~$|P_{j,j}[e]-e|\leq C\delta_j|e|$, we obtain, after increasing~$m_0$ and decreasing~$\varepsilon$ so that~$C\delta_jh_j\leq\frac12$,
\begin{equation*}
	|P_{k,j}[e]-e|\leq C\delta_j h_j |e|\leq \frac12 |e|\,.
\end{equation*}
Thus~$P_{k,j}$ is invertible, and~$|P_{k,j}[e]|\simeq |e|$ uniformly on the block. Taking~$e=Q_{k,j}[p]$ in~\eqref{e.sharp.Cone.corrector.induction} gives~\eqref{e.sharp.Cone.corrector.scale.flat}. Finally,~\eqref{e.sharp.Cone.energy.by.slope.excess} gives the upper bound in~\eqref{e.sharp.Cone.corrector.scale.nondeg}, while~\eqref{e.sharp.Cone.poincare} gives
\begin{equation*}
	|P_{k,j}[e]|
	\leq
	C D_{k,j}[e]+C E_{k,j}[e]
	\leq
	C D_{k,j}[e]+CK\delta_j |P_{k,j}[e]|\,.
\end{equation*}
After decreasing~$\varepsilon$, the last term is absorbed, giving the lower bound in~\eqref{e.sharp.Cone.corrector.scale.nondeg}. The energy norm~$D_{k,j}$ is legitimate throughout by the deterministic scalar comparability in the window.

\smallskip

\emph{Step 4.} Excess decay. Fix~$j$ with~$j\geq m_*$. If~$\ell,k\in\Z$ satisfy~$j-h_j\leq \ell\leq k\leq j$ and~$u\in\A(\cu_k)$, then
\begin{equation}
	\label{e.sharp.Cone.one.block}
	\inf_{e\in\Rd}
	\|\nabla u-\nabla v_{j,e}\|_{\underline{L}^2(\cu_\ell)}
	\leq
	C3^{-\eta(k-\ell)}
	\inf_{e\in\Rd}
	\|\nabla u-\nabla v_{j,e}\|_{\underline{L}^2(\cu_k)}\,.
\end{equation}
Fix~$k\in\Z\cap[j-h_j,j]$ and~$u\in\A(\cu_k)$, let~$\bar e_k$ attain the infimum on the right side of~\eqref{e.sharp.Cone.one.block}, and put~$w_k\coloneqq u-v_{j,\bar e_k}$. Applying~\eqref{e.sharp.Cone.harmonic} to~$w_k$, approximating the harmonic comparison by its affine Taylor polynomial in~$\cu_{k-h}$, and then using~\eqref{e.sharp.Cone.corrector.scale.flat} to replace this affine function by an element of the corrected-affine space~$\{v_{j,e}:e\in\Rd\}$, gives, for every integer~$h$ with~$h\geq2$ and~$k-h\geq j-h_j$,
\begin{equation*}
	3^{-k}
	\inf_{e\in\Rd}\inf_{c\in\R}
	\|w_k-v_{j,e}+c\|_{\underline{L}^2(\cu_{k-h})}
	\leq
	C\bigl(3^{-2h}+\delta_j3^{\frac{d}{2} h}\bigr)
	\shom_j^{-\nf12}\nu^{\nf12}
	\|\nabla w_k\|_{\underline{L}^2(\cu_k)}\,.
\end{equation*}
Here the factor~$3^{\frac{d}{2}h}$ only comes from comparing normalized~$L^2$ norms on~$\cu_k$ and~$\cu_{k-h}$. Applying~\eqref{e.sharp.Cone.caccioppoli} to the solution~$w_k-v_{j,e}$ on one smaller cube, then replacing~$h+1$ by~$h$ and absorbing the resulting dimensional constants, gives, for every integer~$h$ with~$h\geq3$ and~$k-h\geq j-h_j$,
\begin{equation*}
	\inf_{e\in\Rd}
	\|\nabla u-\nabla v_{j,e}\|_{\underline{L}^2(\cu_{k-h})}
	\leq
	C\bigl(3^{-h}+\delta_j3^{(\frac{d}{2}+1)h}\bigr)
	\inf_{e\in\Rd}
	\|\nabla u-\nabla v_{j,e}\|_{\underline{L}^2(\cu_k)}\,.
\end{equation*}
We now choose a fixed integer~$\hat h=\hat h(\eta,d)$ so large that~$C3^{-(1-\eta)\hat h}\leq\nf14$. The constant~$M$ in the definition of~$h_m$ is chosen large enough, and then~$m_0$ is increased if necessary, so that, for every~$m$ with~$m\geq m_*$, we have~$\hat h\leq h_m$.
This fixes the previously unspecified constant~$M$; the threshold~$\hat L$ in Step~1 is then understood to have been chosen for this value of~$M$. In summary,~$\hat h$ and~$M$ depend only on~$(\eta,d)$; the stochastic threshold continues to carry the dependence on~$(\varepsilon,\rho,M,\cstar,\nu,\nondegconst,d)$. We then choose~$\varepsilon$ small enough that~$C\delta_j3^{(\frac{d}{2}+1)\hat h}\leq\frac14\,3^{-\eta \hat h}$ for every~$j$ with~$j\geq m_*$. This yields the one-step improvement
\begin{equation}
	\label{e.sharp.Cone.excess.step}
	\inf_{e\in\Rd}
	\|\nabla u-\nabla v_{j,e}\|_{\underline{L}^2(\cu_{k-\hat h})}
	\leq
	\frac12 3^{-\eta \hat h}
	\inf_{e\in\Rd}
	\|\nabla u-\nabla v_{j,e}\|_{\underline{L}^2(\cu_k)}\,.
\end{equation}
Iterating~\eqref{e.sharp.Cone.excess.step} proves~\eqref{e.sharp.Cone.one.block} when~$k-\ell$ is a multiple of~$\hat h$. For a remainder of at least three scales, we use the estimate before the choice of~$\hat h$. If one or two scales remain, the required bound follows from the coarse-grained Caccioppoli and Poincar\'e estimates~\eqref{e.sharp.Cone.caccioppoli} and~\eqref{e.sharp.Cone.poincare}. Since~$\hat h$ depends only on~$(\eta,d)$, the constants arising over fewer than~$\hat h$ scales also depend only on~$(\eta,d)$; after increasing~$m_0$, the change from~$\eta$ to~$\eta-\kappa$ in the estimates below accounts for these constants.

\smallskip

\emph{Step 5.} Globalizing the corrected-affine estimates.
We first compare the corrected-affine spaces at scales~$s<r$ with~$r-s\leq H$. The resulting flatness and nondegeneracy estimates allow us to repeat Step~4 with~$\{v_{j,e}:e\in\Rd\}$ as the comparison space and to telescope the minimizing slopes.

\smallskip

\emph{Changing the reference cube.} Choose~$H=H(\eta,\kappa,d)$ so large that~$C3^{-\eta(H-1)}\leq\frac1{16}$, where~$C$ also includes the constant from restricting an estimate to the next smaller cube. Increase~$m_0$ and decrease~$\varepsilon$ so that~$H+2\leq h_r$,~$\delta_s\simeq\delta_r$,~$\shom_s\simeq\shom_r$, and~$C_H\delta_r$ is sufficiently small whenever~$1\leq r-s\leq H$. For such~$s<r$, the map~$Q_{s,s}P_{s,r}$ matches the best affine slopes. The estimates from Step~3 over the scales from~$s$ to~$r$ and the coarse-grained Caccioppoli estimate~\eqref{e.sharp.Cone.caccioppoli}, applied from~$\cu_s$ to~$\cu_{s-1}$, give
\begin{equation}
	\label{e.sharp.Cone.chart.match}
	\bigl|Q_{s,s}P_{s,r}-\Id\bigr|\leq C_H\delta_s\,,
	\qquad
	\shom_s^{-\nf12}\nu^{\nf12}
	\|\nabla v_{r,p}-\nabla v_{s,Q_{s,s}P_{s,r}[p]}\|_{\underline{L}^2(\cu_{s-1})}
	\leq C_H\delta_s|p|\,.
\end{equation}
Fix~$j$ with~$j\geq m_*$, and fix~$e\in\Rd$ and~$k\in[m_*,j-H)\cap\Z$. Starting from~$b_0=a_0=j$, set~$b_{i+1}=b_i-H$ and~$a_{i+1}=b_{i+1}-1$ while~$a_{i+1}\geq k$, and let~$a_N$ be the last such scale. Then~$0\leq a_N-k<H$. Write~$v_{j,e}=v_{b_i,p_i}+w_i$ in~$\A(\cu_{a_i})/\R$, with~$p_0=e$ and~$w_0=0$. Given~$(p_i,w_i)$ with~$i<N$, let~$q_i$ minimize
\begin{equation*}
	q\mapsto\|\nabla w_i-\nabla v_{b_i,q}\|_{\underline{L}^2(\cu_{b_{i+1}})}\,.
\end{equation*}
By~\eqref{e.sharp.Cone.corrector.scale.nondeg}, the minimizer is unique and linear in~$w_i$. Set
\begin{equation*}
	\begin{aligned}
		p_{i+1}&\coloneqq Q_{b_{i+1},b_{i+1}}P_{b_{i+1},b_i}[p_i+q_i]\,,\\
		w_{i+1}&\coloneqq w_i-v_{b_i,q_i}+v_{b_i,p_i+q_i}-v_{b_{i+1},p_{i+1}}\,.
	\end{aligned}
\end{equation*}
Then~$v_{j,e}=v_{b_{i+1},p_{i+1}}+w_{i+1}$ modulo constants. Set
\begin{equation*}
	R_i\coloneqq\shom_{b_i}^{-\nf12}\nu^{\nf12}\|\nabla w_i\|_{\underline{L}^2(\cu_{a_i})}\,.
\end{equation*}
By local nondegeneracy,~$|q_i|\leq C_HR_i$. Applying~\eqref{e.sharp.Cone.one.block} to~$w_i$ before changing the reference cube, and then using~\eqref{e.sharp.Cone.chart.match}, gives, for some fixed~$\theta<1/8$,
\begin{equation}
	\label{e.sharp.Cone.chart.recursion.bounds}
	R_{i+1}\leq\theta R_i+C_H\delta_{b_{i+1}}\bigl(|p_i|+R_i\bigr)\,,
	\qquad
	|p_{i+1}-p_i|\leq C_H\bigl(R_i+\delta_{b_{i+1}}|p_i|\bigr)\,.
\end{equation}
After enlarging~$C_H$ and then making~$C_H\delta_r$ sufficiently small,~\eqref{e.sharp.Cone.chart.recursion.bounds} and the comparability of~$\delta_{b_i}$ and~$\delta_{b_{i+1}}$ close the induction
\begin{equation}
	\label{e.sharp.Cone.chart.invariant}
	R_i\leq C_H\delta_{b_i}|p_i|\,,
	\qquad
	\bigl(1-C_H\delta_{b_{i+1}}\bigr)|p_i|\leq|p_{i+1}|\leq\bigl(1+C_H\delta_{b_{i+1}}\bigr)|p_i|\,.
\end{equation}

\smallskip

Let~$P_{k,j}[e]$ be the best affine slope of~$v_{j,e}$ on~$\cu_k$. If~$k<j-H$, restrict the final decomposition to~$\cu_k$ and apply~\eqref{e.sharp.Cone.one.block} over the remaining fewer than~$H$ scales, together with~\eqref{e.sharp.Cone.corrector.scale.flat}, the coarse-grained Poincar\'e estimate~\eqref{e.sharp.Cone.poincare} applied to~$w_N$, and~\eqref{e.sharp.Cone.chart.invariant}. If~$k\in[j-H,j]$, use~\eqref{e.sharp.Cone.corrector.scale.flat},~$\delta_k\simeq\delta_j$, and~$\shom_k\simeq\shom_j$. In both cases,
\begin{equation}
	\label{e.sharp.Cone.global.flat}
	3^{-k}\inf_{c\in\R}\|v_{j,e}-\linear_{P_{k,j}[e]}+c\|_{\underline{L}^2(\cu_k)}
	\leq
	C_H\delta_k|P_{k,j}[e]|\,.
\end{equation}
For~$k<j-H$, stability of best affine slopes and~\eqref{e.sharp.Cone.chart.invariant} also give
\begin{equation*}
	\bigl|P_{k,j}[e]-P_{k,b_N}[p_N]\bigr|\leq C_HR_N\,.
\end{equation*}
Combining this with~\eqref{e.sharp.Cone.corrector.scale.nondeg} gives
\begin{equation}
	\label{e.sharp.Cone.global.nondeg}
	C^{-1}|P_{k,j}[e]|
	\leq
	\shom_k^{-\nf12}\nu^{\nf12}\|\nabla v_{j,e}\|_{\underline{L}^2(\cu_k)}
	\leq
	C|P_{k,j}[e]|\,.
\end{equation}
For~$k\in[j-H,j]$, the same estimate follows from Step~3 and~$\shom_k\simeq\shom_j$. Iterating the second bound in~\eqref{e.sharp.Cone.chart.invariant}, using the comparison at the final scale above and Step~3 for~$k\in[j-H,j]$, yields, for every~$k$ and~$t$ with~$m_*\leq k\leq t\leq j$,
\begin{equation}
	\label{e.sharp.Cone.global.slope.transport}
	C^{-1}3^{-\frac18\kappa(t-k)}|P_{t,j}[e]|
	\leq
	|P_{k,j}[e]|
	\leq
	C3^{\frac18\kappa(t-k)}|P_{t,j}[e]|\,.
\end{equation}
Since~$P_{j,j}$ is close to the identity by Step~3,~\eqref{e.sharp.Cone.global.slope.transport} makes every~$P_{k,j}$ invertible. Set~$Q_{k,j}=P_{k,j}^{-1}$.

\smallskip

\emph{Excess decay relative to~$\{v_{j,e}:e\in\Rd\}$.} Fix~$R$ with~$m_*\leq R\leq j$, and fix~$u\in\A(\cu_R)$. Repeat Step~4 with~$\{v_{j,e}:e\in\Rd\}$ as the comparison space. More precisely, if the affine Taylor polynomial in the harmonic-approximation argument has slope~$p$ at scale~$s$, replace it by~$v_{j,Q_{s,j}[p]}$; since~$Q_{s,j}=P_{s,j}^{-1}$, this function has best affine slope~$p$ on~$\cu_s$. By~\eqref{e.sharp.Cone.global.flat} and~\eqref{e.sharp.Cone.global.nondeg}, the argument leading to~\eqref{e.sharp.Cone.excess.step} applies with~$\delta_k$ in place of~$\delta_j$. The normalization factors telescope to~$(\shom_R/\shom_k)^{\nf12}$, which is bounded by~$C(R/k)^{\nf14}\leq C3^{\frac18\kappa(R-k)}$ by~\eqref{e.sm.sharp.bounds}. Hence, for every~$k$ with~$m_*\leq k\leq R$,
\begin{equation}
	\label{e.sharp.Cone.fixed.excess}
	\shom_k^{-\nf12}\nu^{\nf12}
	\inf_{e\in\Rd}\|\nabla u-\nabla v_{j,e}\|_{\underline{L}^2(\cu_k)}
	\leq
	C3^{-(\eta-\frac18\kappa)(R-k)}
	\shom_R^{-\nf12}\nu^{\nf12}
	\inf_{e\in\Rd}\|\nabla u-\nabla v_{j,e}\|_{\underline{L}^2(\cu_R)}\,.
\end{equation}
For each~$k$ with~$m_*\leq k\leq R$, let~$e_k$ attain the infimum on the left side of~\eqref{e.sharp.Cone.fixed.excess}. By~\eqref{e.sharp.Cone.global.nondeg}, it is unique. The same nondegeneracy estimate and the triangle inequality give, for every~$k$ with~$m_*\leq k<R$,
\begin{equation}
	\label{e.sharp.Cone.fixed.slope.stability}
	|P_{k,j}[e_k-e_{k+1}]|
	\leq
	C\bigl(\shom_k^{-\nf12}\nu^{\nf12}\|\nabla u-\nabla v_{j,e_k}\|_{\underline{L}^2(\cu_k)}+\shom_{k+1}^{-\nf12}\nu^{\nf12}\|\nabla u-\nabla v_{j,e_{k+1}}\|_{\underline{L}^2(\cu_{k+1})}\bigr)\,.
\end{equation}
Combining~\eqref{e.sharp.Cone.fixed.slope.stability},~\eqref{e.sharp.Cone.global.slope.transport},~\eqref{e.sharp.Cone.fixed.excess}, and~\eqref{e.sharp.Cone.global.nondeg}, summing the geometric series, and returning to the unnormalized norm with~\eqref{e.sm.sharp.bounds} gives, for each~$k$ with~$k_0\leq k\leq R$,
\begin{equation*}
	\|\nabla u-\nabla v_{j,e_{k_0}}\|_{\underline{L}^2(\cu_k)}
	\leq
	C3^{-(\eta-\frac12\kappa)(R-k)}\|\nabla u\|_{\underline{L}^2(\cu_R)}\,,
\end{equation*}
and
\begin{equation*}
	\shom_R^{\nf12}\nu^{-\nf12}|P_{R,j}[e_{k_0}]|
	\leq
	C\|\nabla u\|_{\underline{L}^2(\cu_R)}\,.
\end{equation*}
The map~$u\mapsto e_{k_0}$ is linear.

\smallskip

Taking~$j=R=m$ gives, for every~$m$ with~$m\geq m_*$, every~$k\in[m_*,m]\cap\Z$, and every~$u\in\A(\cu_m)$,
\begin{equation}
	\label{e.sharp.Cone.finite.Cone}
	\inf_{e\in\Rd}\|\nabla u-\nabla v(\cdot,\cu_m,e)\|_{\underline{L}^2(\cu_k)}
	\leq
	C3^{-(\eta-\kappa)(m-k)}\|\nabla u\|_{\underline{L}^2(\cu_m)}\,.
\end{equation}
Moreover, the slope may be selected so that
\begin{equation}
	\label{e.sharp.Cone.finite.Cone.slope}
	\shom_m^{\nf12}\nu^{-\nf12}|e|
	\leq
	C\|\nabla u\|_{\underline{L}^2(\cu_m)}\,.
\end{equation}
For every~$k_0\in[m_*,m]\cap\Z$, the slope selected at the smallest scale satisfies, simultaneously for every~$k\in[k_0,m]\cap\Z$,
\begin{equation}
	\label{e.sharp.Cone.finite.Cone.simultaneous}
	\|\nabla u-\nabla v(\cdot,\cu_m,e)\|_{\underline{L}^2(\cu_k)}
	\leq
	C3^{-(\eta-2\kappa)(m-k)}\|\nabla u\|_{\underline{L}^2(\cu_m)}\,,
\end{equation}
and~\eqref{e.sharp.Cone.finite.Cone.slope} holds for this same~$e$.

\smallskip

\emph{Step 6.} Global corrected affines.
We first construct linear maps between consecutive corrected-affine spaces. Fix~$m$ with~$m\geq m_*$ and~$e\in\Rd$, and put
\begin{equation*}
	q_0\coloneqq Q_{m,m}P_{m,m+1}[e]
	\qand
	u\coloneqq v_{m+1,e}-v_{m,q_0}\,.
\end{equation*}
The two corrected affines have the same best affine slope on~$\cu_m$, so~\eqref{e.sharp.Cone.corrector.scale.flat} gives
\begin{equation*}
	3^{-m}\inf_{c\in\R}\|u+c\|_{\underline{L}^2(\cu_m)}
	\leq
	C\delta_{m+1}|e|\,.
\end{equation*}
If~$m\geq m_*+1$, coarse-grained Caccioppoli on one smaller cube yields
\begin{equation*}
	\shom_m^{-\nf12}\nu^{\nf12}\|\nabla u\|_{\underline{L}^2(\cu_{m-1})}
	\leq
	C\delta_{m+1}|e|\,.
\end{equation*}
If~$m\geq m_*+1$, let~$p(e)$ be the slope selected in Step~5 for~$u$, with comparison space~$\{v_{m,p}:p\in\Rd\}$ and scales ranging from~$m_*$ to~$m-1$; for~$m=m_*$, set~$p(e)=0$. Define~$T_m[e]\coloneqq q_0+p(e)$. This choice depends linearly on~$e$, and~\eqref{e.sharp.Cone.finite.Cone.slope},~\eqref{e.sharp.Cone.top.slope.close}, and~$\shom_{m-1}\simeq\shom_m\simeq\shom_{m+1}$ give~$|p(e)|\leq C\delta_{m+1}|e|$. The estimate from Step~5 gives, for every~$\ell\in[m_*,m-1]\cap\Z$,
\begin{equation}
	\label{e.sharp.Cone.consecutive}
	\|\nabla v(\cdot,\cu_{m+1},e)-\nabla v(\cdot,\cu_m,T_m[e])\|_{\underline{L}^2(\cu_\ell)}
	\leq
	C\varepsilon m^{-\beta}\log m\,\shom_{m+1}^{\nf12}\nu^{-\nf12}3^{-(\eta-\kappa)(m-\ell)}|e|\,.
\end{equation}
The definition of~$q_0$ and~\eqref{e.sharp.Cone.chart.match}, with~$(s,r)=(m,m+1)$, also give~$|q_0-e|\leq C\delta_{m+1}|e|$, and hence
\begin{equation}
	\label{e.sharp.Cone.T.close}
	|T_m[e]-e|
	\leq
	C\delta_{m+1}|e|
	\leq
	C\varepsilon m^{-\beta}\log m\,|e|\,.
\end{equation}
After decreasing~$\varepsilon$ and increasing~$m_0$, write~$T_m=\Id+A_m$ with~$|A_m|\leq C\varepsilon m^{-\beta}\log m\leq\frac12$. The Neumann series gives
\begin{equation}
	\label{e.sharp.Cone.T.inverse.close}
	|T_m^{-1}[e]-e|
	\leq
	C\varepsilon m^{-\beta}\log m\,|e|\,.
\end{equation}
Thus~$T_m$ is invertible and~$|T_m[e]|+|T_m^{-1}[e]|\leq C|e|$.

\smallskip

Fix~$n$ with~$n\geq m_*$, and fix~$e\in\Rd$. Set~$e_n=e$ and, for every~$k$ with~$k\geq n$, define
\begin{equation*}
	e_{k+1}\coloneqq T_k^{-1}[e_k]\,.
\end{equation*}
After increasing~$m_0$, the function~$r\mapsto r^{-\beta}\log r$ is decreasing for~$r\geq m_0$. Thus~\eqref{e.sharp.Cone.T.inverse.close} gives, for every~$k$ with~$k>n$,
\begin{equation*}
	|e_k|
	\leq
	|e|\prod_{r=n}^{k-1}\bigl(1+C\varepsilon r^{-\beta}\log r\bigr)
	\leq
	|e|\exp\biggl(C\varepsilon\sum_{r=n}^{k-1}r^{-\beta}\log r\biggr)
	\leq
	|e|\exp\bigl(C\varepsilon(k-n)n^{-\beta}\log n\bigr)\,.
\end{equation*}
Since~$n\geq m_*\geq m_0$, increasing~$m_0$ so that~$C\varepsilon m_0^{-\beta}\log m_0\leq\kappa\log3$ gives, for every~$k$ with~$k\geq n$,
\begin{equation}
	\label{e.sharp.Cone.slope.growth}
	|e_k|\leq 3^{\kappa(k-n)}|e|\,.
\end{equation}
Define
\begin{equation*}
	v_k[n,e]\coloneqq
	v(\cdot,\cu_k,e_k)
	-
	(v(\cdot,\cu_k,e_k))_{\cu_k}\,.
\end{equation*}
Since~$T_r[e_{r+1}]=e_r$, the consecutive-scale estimate~\eqref{e.sharp.Cone.consecutive}, applied with top slope~$e_{r+1}$, gives, for every~$r$ with~$r\geq n$ and every~$\ell\in\Z\cap[m_*,r-1]$,
\begin{equation*}
	\|\nabla v_{r+1}[n,e]-\nabla v_r[n,e]\|_{\underline{L}^2(\cu_\ell)}
	\leq
	C\varepsilon r^{-\beta}\log r\,\shom_{r+1}^{\nf12}\nu^{-\nf12}3^{-(\eta-\kappa)(r-\ell)}\,|e_{r+1}|\,.
\end{equation*}
Together with~\eqref{e.sharp.Cone.slope.growth} and~\eqref{e.sm.sharp.bounds}, this gives, for every~$r$ with~$r\geq n$ and every~$\ell\in\Z\cap[m_*,r-1]$,
\begin{equation*}
	\|\nabla v_{r+1}[n,e]-\nabla v_r[n,e]\|_{\underline{L}^2(\cu_\ell)}
	\leq
	C|e|\,r^{C(d)}3^{-(\eta-2\kappa)(r-\ell)}3^{\kappa(\ell-n)}\,.
\end{equation*}
Since~$\eta>2\kappa$, the right side is summable in~$r$ for every fixed~$\ell$, so the gradients~$\{\nabla v_k[n,e]\}_{k\geq n+1}$ are Cauchy locally in~$L^2$. After fixing zero average on~$\cu_{m_*}$, Poincar\'e's inequality gives a local~$H^1$ limit, denoted modulo constants by~$\phi_e[n]$. Summing~\eqref{e.sharp.Cone.consecutive} from~$r=k$ to infinity and using~\eqref{e.sharp.Cone.slope.growth} and~\eqref{e.sm.sharp.bounds} gives, for every~$k$ with~$k\geq n$ and every~$\ell\in[m_*,k-1]\cap\Z$,
\begin{equation}
	\label{e.sharp.Cone.limit.compare}
	\|\nabla \phi_e[n]-\nabla v_k[n,e]\|_{\underline{L}^2(\cu_\ell)}
	\leq
	C\varepsilon k^{-\beta}\log k
	\shom_k^{\nf12}\nu^{-\nf12}
	3^{-(\eta-\kappa)(k-\ell)}3^{\kappa(k-n)}|e|\,.
\end{equation}
Thus~$\phi_e[n]\in\A(\Rd)$.
Since~$v(\cdot,U,e)$ and every~$T_m^{\pm1}$ depend linearly on~$e$, so do~$v_k[n,e]$ and~$\phi_e[n]$.

Set~$p_n\coloneqq P_{n,n+1}[e_{n+1}]$. Applying~\eqref{e.sharp.Cone.top.slope.close} with~$(k,j,e)=(n,n+1,e_{n+1})$ and then~\eqref{e.sharp.Cone.T.inverse.close} gives
\begin{equation*}
	|p_n-e|\leq C\delta_{n+1}|e|
	\qquad\mbox{and hence}\qquad
	|p_n|\geq\frac12|e|\,.
\end{equation*}
Applying~\eqref{e.sharp.Cone.corrector.scale.nondeg} to~$(n,n+1)$ and~\eqref{e.sharp.Cone.limit.compare} with~$k=n+1$ and~$\ell=n$ shows that~$\phi_e[n]$ is nonconstant when~$e\neq0$. Thus~$e\mapsto\phi_e[n]$ is injective modulo constants, and its image together with the constants has dimension~$d+1$. After fixing the additive constant on one reference cube, all local~$H^1$ norms on this space are equivalent.

\smallskip

If~$n'\geq n$, the recursion defining~$e_k$, started from~$e_{n'}$ at scale~$n'$, gives the same slopes for every~$k$ with~$k\geq n'$. Hence~$\phi_e[n]$ and~$\phi_{e_{n'}}[n']$ agree modulo constants. Their common space is denoted by~$\A^1_{\sharp}$, with elements~$\phi_e$.

\smallskip

The growth estimate~\eqref{e.sharp.Cone.slope.growth}, with~$\kappa<\gamma$, and~\eqref{e.sharp.Cone.limit.compare} imply that
\begin{equation*}
	\A^1_{\sharp}\subseteq \A^{1+\gamma}(\Rd)\,.
\end{equation*}

\smallskip

\emph{Step 7.} Liouville theorem. We prove the reverse inclusion. Let~$u\in\A^{1+\gamma}(\Rd)$. Applying~\eqref{e.sharp.Cone.caccioppoli} with top scale~$j+1$ and using the deterministic scalar comparability of~$\shom_j$ and~$\shom_{j+1}$, we have, for every~$j$ with~$j\geq m_*$,
\begin{equation*}
	\shom_j^{-1}
	\|\nabla u\|_{\underline{L}^2(\cu_j)}^2
	\leq
	C\nu^{-1}3^{-2j}\|u\|_{\underline{L}^2(\cu_{j+1})}^2\,.
\end{equation*}
Since~$u\in\A^{1+\gamma}(\Rd)$, this yields
\begin{equation}
	\label{e.sharp.Cone.grad.decay}
	\shom_j^{-\nf12}
	\|\nabla u\|_{\underline{L}^2(\cu_j)}
	\leq C(u)3^{\gamma j}\,.
\end{equation}
Fix~$r_0$ with~$r_0\geq\X$ and choose~$\ell_0\in\Z$ such that~$B_{r_0}\subseteq\cu_{\ell_0}\subseteq B_{Cr_0}$. For every integer~$k$ with~$k\geq\ell_0$,~\eqref{e.sharp.Cone.finite.Cone} gives~$\bar e_k\in\Rd$ such that
\begin{equation*}
	\|\nabla u-\nabla v(\cdot,\cu_k,\bar e_k)\|_{\underline{L}^2(\cu_{\ell_0})}
	\leq
	C3^{-(\eta-\kappa)(k-\ell_0)}
	\|\nabla u\|_{\underline{L}^2(\cu_k)}\,.
\end{equation*}
By~\eqref{e.sharp.Cone.finite.Cone.slope} and~\eqref{e.sharp.Cone.grad.decay},
\begin{equation*}
	|\bar e_k|
	\leq
	C\shom_k^{-\nf12}
	\|\nabla u\|_{\underline{L}^2(\cu_k)}
	\leq C(u)3^{\gamma k}\,.
\end{equation*}
Let~$\phi_k\coloneqq\phi_{\bar e_k}[k]\in\A^1_{\sharp}$ be the corrected affine with initial scale~$k$ and slope~$\bar e_k$. The preceding estimates and~\eqref{e.sharp.Cone.limit.compare}, with initial and comparison scale~$k$, give
\begin{align*}
	\|\nabla u-\nabla\phi_k\|_{\underline{L}^2(B_{r_0})}
	&\leq
	C\|\nabla u-\nabla v(\cdot,\cu_k,\bar e_k)\|_{\underline{L}^2(\cu_{\ell_0})}
	{}+C\|\nabla v(\cdot,\cu_k,\bar e_k)-\nabla\phi_k\|_{\underline{L}^2(\cu_{\ell_0})}
	\\
	&\leq
	C3^{-(\eta-\kappa)(k-\ell_0)}
	\|\nabla u\|_{\underline{L}^2(\cu_k)}
	{}+C\varepsilon k^{-\beta}\log k\,
	\shom_k^{\nf12}\nu^{-\nf12}
	3^{-(\eta-\kappa)(k-\ell_0)}|\bar e_k|
	\\
	&\leq
	C(u,r_0)k^{\nf14}(1+\varepsilon k^{-\beta}\log k)3^{-(\eta-\kappa-\gamma)k}\,.
\end{align*}
Here we used~\eqref{e.sharp.Cone.grad.decay}, the upper bound for~$\shom_k$, the bound on~$|\bar e_k|$, and the fact that~$\ell_0$ is fixed. Since~$\eta-\kappa>\gamma$,
\begin{equation*}
	\inf_{\phi\in\A^1_{\sharp}}
	\|\nabla u-\nabla\phi\|_{\underline{L}^2(B_{r_0})}
	=0\,.
\end{equation*}
Choose~$r_i\to\infty$ with~$r_i\geq\X$ and~$\cu_{m_*}\subseteq B_{r_i}$. A diagonal choice in the preceding approximation gives~$\phi_i\in\A^1_{\sharp}$ such that, for every~$q$ with~$q\leq i$,
\begin{equation*}
	\|\nabla u-\nabla\phi_i\|_{\underline{L}^2(B_{r_q})}\leq i^{-1}\,.
\end{equation*}
Normalize by~$(\phi_i)_{\cu_{m_*}}=0$. The Step~6 nonconstancy estimate, obtained from~\eqref{e.sharp.Cone.corrector.scale.nondeg} and~\eqref{e.sharp.Cone.limit.compare}, bounds their slopes by~$\|\nabla\phi_i\|_{\underline{L}^2(\cu_{m_*})}$, which is bounded by the preceding convergence. Finite dimensionality gives a subsequence converging locally in~$H^1$ to some~$\phi\in\A^1_{\sharp}$, and the diagonal estimate gives~$\nabla\phi=\nabla u$ in~$\Rd$. Thus
\begin{equation*}
	\A^{1+\gamma}(\Rd)=\A^1_{\sharp}\,,
\end{equation*}
and the Liouville dimension statement follows. Intersecting the full-probability events over rational~$\gamma\in(0,1)$ makes this conclusion simultaneous. For arbitrary~$\gamma$, choose rational~$\gamma_-<\gamma<\gamma_+$; the inclusions~$\A^{1+\gamma_-}(\Rd)\subseteq\A^{1+\gamma}(\Rd)\subseteq\A^{1+\gamma_+}(\Rd)$ and equality of the outer spaces give the result.

\smallskip

\emph{Step 8.} Flatness and large-scale~$C^{1,\gamma}$. Let~$\phi\in\A^{1+\gamma}(\Rd)=\A^1_{\sharp}$ and fix~$r$ with~$r\geq\X$. Choose~$n\in\Z$ such that~$B_r\subseteq\cu_n\subseteq B_{Cr}$.
By~\eqref{e.sharp.Cone.limit.compare}, applied with~$k=n+1$ and~$\ell=n$, and the coarse-grained Poincar\'e inequality~\eqref{e.sharp.Cone.poincare}, there exists~$\bar e_n\in\Rd$ such that, writing~$\bar e_{n+1}\coloneqq T_n^{-1}[\bar e_n]$,
\begin{equation*}
	3^{-n}
	\inf_{c\in\R}
	\|\phi-v(\cdot,\cu_{n+1},\bar e_{n+1})+c\|_{\underline{L}^2(\cu_n)}
	\leq
	C\varepsilon n^{-\beta}\log n\,|\bar e_n|\,.
\end{equation*}
Set~$p_n\coloneqq P_{n,n+1}[\bar e_{n+1}]$. Applying~\eqref{e.sharp.Cone.top.slope.close} with~$(k,j,e)=(n,n+1,\bar e_{n+1})$ and then~\eqref{e.sharp.Cone.T.inverse.close} gives
\begin{equation*}
	|p_n-\bar e_n|\leq C\delta_{n+1}|\bar e_n|\,,
	\qquad
	|p_n|\simeq|\bar e_n|\,.
\end{equation*}
For comparable centered balls and cubes,
\begin{equation*}
	\inf_{c\in\R}
	\|\linear_{p_n}+c\|_{\underline{L}^2(B_r)}
	\simeq
	r|p_n|
	\simeq
	3^n|\bar e_n|\,.
\end{equation*}
The scale-flatness estimate~\eqref{e.sharp.Cone.corrector.scale.flat}, applied to the pair~$(n,n+1)$, gives
\begin{equation*}
	3^{-n}
	\inf_{c\in\R}
	\|v(\cdot,\cu_{n+1},\bar e_{n+1})-\linear_{p_n}+c\|_{\underline{L}^2(\cu_n)}
	\leq
	C\varepsilon n^{-\beta}\log n\,|\bar e_n|\,.
\end{equation*}
The same flatness and nondegeneracy estimates, after decreasing~$\varepsilon$, give
\begin{equation*}
	3^n |\bar e_n|
	\leq
	C\|\phi\|_{\underline{L}^2(B_r)}\,.
\end{equation*}
Indeed, modulo constants both functions are~$C\varepsilon n^{-\beta}\log n$-close to~$\linear_{p_n}$ on comparable centered cubes and balls. Restricting from~$\cu_n$ to~$B_r$ gives
\begin{equation*}
	\inf_{e\in\Rd}
	\|\phi-\linear_e-(\phi)_{B_r}\|_{\underline{L}^2(B_r)}
	\leq
	C\varepsilon n^{-\beta}\log n\,3^n|\bar e_n|
	\leq
	C\varepsilon n^{-\beta}\log n
	\|\phi\|_{\underline{L}^2(B_r)}\,.
\end{equation*}
Since~$n\simeq\log r$,
\begin{equation*}
	n^{-\beta}\log n
	\leq C(\log r)^{-\beta}\log\log r\,.
\end{equation*}
This proves the flatness estimate. Finally, fix~$R$ with~$R\geq\X$ and fix~$u\in\A(B_R)$. By the choice of the dimensional factor in~$\X$, select~$m$ with~$m\geq m_*$,~$\cu_m\subseteq B_R$, and~$3^m\simeq R$, and then~$c=c(d)\in(0,1)$ so that~$B_{cR}\subseteq\cu_{m-1}$. Applying~\eqref{e.sharp.Cone.finite.Cone.simultaneous} in~$\cu_m$, with bottom scale~$m_*$, and passing between comparable cubes and balls gives~$\bar e_m\in\Rd$ such that, for every~$r\in[\X,cR]$,
\begin{equation*}
	\|\nabla u-\nabla v(\cdot,\cu_m,\bar e_m)\|_{\underline{L}^2(B_r)}
	\leq
	C\Bigl(\frac rR\Bigr)^{\eta-2\kappa}
	\|\nabla u\|_{\underline{L}^2(B_R)}\,.
\end{equation*}
The selected-slope bound~\eqref{e.sharp.Cone.finite.Cone.slope} and comparability of the normalized norms on~$\cu_m$ and~$B_R$ give
\begin{equation*}
	\shom_m^{\nf12}\nu^{-\nf12}|\bar e_m|
	\leq
	C\|\nabla u\|_{\underline{L}^2(B_R)}\,.
\end{equation*}
Let~$\phi\in\A^1_{\sharp}$ be the corrected affine with initial scale~$m$ and slope~$\bar e_m$. By~\eqref{e.sharp.Cone.limit.compare}, the preceding bounds, and~$B_{cR}\subseteq\cu_{m-1}$, we have, for every~$r\in[\X,cR]$,
\begin{equation*}
	\|\nabla\phi-\nabla v(\cdot,\cu_m,\bar e_m)\|_{\underline{L}^2(B_r)}
	\leq
	C\Bigl(\frac rR\Bigr)^{\eta-\kappa}
	\|\nabla u\|_{\underline{L}^2(B_R)}\,.
\end{equation*}
For the global bound, choose~$m'=m+O_d(1)$ with~$B_R\subseteq\cu_{m'}$, and let~$e_{m'+1}$ be the recursively defined slope in the construction of~$\phi$. Since~$m'-m$ is bounded,~\eqref{e.sharp.Cone.slope.growth}, fixed-scale comparability of~$\shom_m$ and~$\shom_{m'+1}$, and~\eqref{e.sharp.Cone.finite.Cone.slope} imply
\begin{equation*}
	\shom_{m'+1}^{\nf12}\nu^{-\nf12}|e_{m'+1}|
	\leq
	C\|\nabla u\|_{\underline{L}^2(B_R)}\,.
\end{equation*}
Applying~\eqref{e.sharp.Cone.limit.compare} with~$k=m'+1$ and then~\eqref{e.sharp.Cone.corrector.nondeg} gives
\begin{equation*}
	\|\nabla\phi\|_{\underline L^2(B_R)}
	\leq
	C\|\nabla u\|_{\underline L^2(B_R)}\,.
\end{equation*}
For~$r\in[cR,R)$ we use this bound, the trivial estimate for~$u$, and the fact that~$(r/R)^\gamma$ is bounded below by a dimensional constant. The triangle inequality and~$\eta-2\kappa>\gamma$ yield the asserted large-scale~$C^{1,\gamma}$ estimate. The proof is complete.
\end{proof}

\begin{proof}[Proof of Theorem~\ref{t.C1beta}]
Fix~$\gamma,\sigma\in(0,1)$ and~$\alpha<\frac12(1-\sigma)$. Apply Proposition~\ref{p.C1beta.sharp} with~$\rho=\sigma$, and replace its minimal scale~$\X$ by~$\X\vee R_0\vee2$, where~$R_0=R_0(\alpha,\sigma)$ is deterministic and large enough that
\begin{equation*}
	(\log r)^{-\frac12(1-\sigma)}\log\log r
	\leq
	C(\alpha,\sigma)(\log r)^{-\alpha}
	\qquad \forall r\geq R_0\,.
\end{equation*}
The tail bound~$\log\X=\O_{\Gamma_\sigma}(C_{\mathrm{tail}})$, after increasing the constant to account for~$R_0\vee2$, gives~\eqref{e.Xgamma.Xi} by the definition of~$\O_{\Gamma_\sigma}$. The Liouville statement and the large-scale~$C^{1,\gamma}$ estimate are exactly the first and third conclusions of Proposition~\ref{p.C1beta.sharp}, while the preceding logarithmic absorption converts the second conclusion into~\eqref{e.flatness.at.every.scale}. This proves the theorem.
\end{proof}

\section{Homogenization of the Dirichlet problem}
\label{s.Dirichlet}

In this section we deduce the Dirichlet homogenization estimates from Lemma~\ref{l.sharp.scale.inputs}.

Throughout this section, unless explicitly stated otherwise, we fix
\begin{equation*}
	0<\alpha<\alpha_0<\nf12\,,
	\qquad
	\rho\coloneqq1-2\alpha_0\,,
	\qquad
	\varepsilon\in(0,1]\,,
	\qquad
	M\geq C(d)\,.
\end{equation*}
Only the exponent~$\rho$ enters the estimates below, while~$\alpha$ and~$\alpha_0$ keep the notation compatible with the preceding large-scale regularity section. Let~$\delta_j$,~$h_j$,~$\hat L$, and~$\{\X_0(y):y\in\Rd\}$ be supplied by Lemma~\ref{l.sharp.scale.inputs}, and write~$\X_0\coloneqq\X_0(0)$. We allow the lower bounds on~$M$ and~$\hat L$ to be enlarged below.

\smallskip

For localized estimates, we cover the domain by a bounded number of aligned cubes of comparable size. We restrict positive-norm estimates to the cover and, for negative norms, split the dual test functions by a smooth partition of unity and extend them to the covering cubes. The resulting normalized~$L^p$,~$W^{-1,p}$,~$H^{-s}$, and hatted negative-norm bounds have constants depending only on~$(U,d)$; the boundary domains used below have uniformly smooth rescalings.

\smallskip

\emph{Deterministic domain conventions.} The Dirichlet domain~$U$ is assumed to be a smooth bounded domain in~$\cu_0$, and we write~$U_K\coloneqq 3^K U$. Once the ambient scale~$m$ has been fixed, we set
\begin{equation}
	\label{e.Dir.new.tilde.a.def}
	\tilde\a_m
	\coloneqq
	\shom_m^{-1}\bigl(\a-\shom_m\Id-(\k)_{\cu_m}\bigr)\,.
\end{equation}
We use, without further comment, the following standard deterministic consequences of the smoothness of~$\partial U$: the measure of a boundary layer of thickness~$r$ in~$U_K$ is bounded by~$C(U,d)r3^{K(d-1)}$, and the corresponding grid cubes have the same total-volume bound up to a constant. The boundary-layer Poincar\'e and Sobolev--Poincar\'e inequalities are used globally on~$U_K$. Local Poincar\'e, Sobolev--Poincar\'e, and affine-Poincar\'e inequalities on boundary-intersected cubes are used only below a fixed single-chart cutoff; the bounded band of larger relative scales is controlled by the global estimates on~$U_K$.
We also use the Whitney extension theorem~\cite[Chapter~VI]{Stein}, in the form that functions on~$U_K$ may be extended to~$\Rd$ with the usual Sobolev or Lipschitz norms increased by at most~$C(U,d)$. We also set
\begin{equation}
	\label{e.Dir.new.Sobolev.conjugates}
	2^\ast\coloneqq
	\begin{cases}
		\frac{2d}{d-2} &\mbox{if } d>2\,,\\
		6 &\mbox{if } d=2\,,
	\end{cases}
	\qquad
	p_\ast\coloneqq \frac{2^\ast}{2^\ast-1}\,.
\end{equation}
Thus~$p_\ast=(2^\ast)'$ is always the H\"older conjugate of the Sobolev exponent~$2^\ast$. In dimension two this fixes a finite Sobolev exponent for the section, rather than using the endpoint value suggested by the formal expression~$2d/(d+2)$.

\subsection{Weak estimates with right-hand sides}
\label{ss.Dirichlet.rhs.blackbox}

The next lemma extends the preceding estimates to equations with a nonzero right-hand side.

\begin{lemma}[Adding a right-hand side]
\label{l.Dirichlet.rhs.blackbox}
There exists~$C(d)<\infty$ such that the following holds. Assume that~$m,n\in\N$ satisfy~\eqref{e.Dir.new.window}. If~$z\in3^{n-3}\Zd$ satisfies~$z+\cu_n\subseteq\cu_m$ and~$u\in H^1(z+\cu_n)$ solves
\begin{equation*}
	-\nabla\cdot\a\nabla u=f
	\qquad \mbox{in } z+\cu_n\,,
\end{equation*}
with~$f\in L^{p_\ast}(z+\cu_n)$. Then
\begin{align}
	\label{e.Dir.new.flux.with.f}
	3^{-\frac14 n}
	\|\tilde\a_m\nabla u\|_{\underline{H}^{-\nf 14}(z+\cu_n)}
	&\leq
	C\shom_m^{-\nf12}\delta_m
	\nu^{\nf12}\|\nabla u\|_{\underline{L}^2(z+\cu_n)}
	+C\shom_m^{-1}m^{1+\rho}3^n\nu^{-1}
	\|f\|_{\underline{L}^{p_\ast}(z+\cu_n)}\,,\\
	\label{e.Dir.new.grad.with.f}
	3^{-\frac14 n}\|\nabla u\|_{\underline{H}^{-\nf 14}(z+\cu_n)}
	&\leq
	C\shom_m^{-\nf12}
	\nu^{\nf12}\|\nabla u\|_{\underline{L}^2(z+\cu_n)}
	+C3^n\nu^{-1}
	\|f\|_{\underline{L}^{p_\ast}(z+\cu_n)}\,.
\end{align}
\end{lemma}

\begin{proof}
Write~$Q\coloneqq z+\cu_n$. Let~$u_z$ be the~$\a$-harmonic function in~$Q$ with boundary data~$u$. Testing the equation for~$u-u_z$ by~$u-u_z$ and using the Sobolev--Poincar\'e inequality gives
\begin{equation*}
	\|\nabla u-\nabla u_z\|_{\underline{L}^2(z+\cu_n)}
	\leq
	C3^n\nu^{-1}\|f\|_{\underline{L}^{p_\ast}(z+\cu_n)}\,.
\end{equation*}
In particular,
\begin{equation*}
	\|\nabla u_z\|_{\underline{L}^2(Q)}
	\leq
	\|\nabla u\|_{\underline{L}^2(Q)}
	+C3^n\nu^{-1}\|f\|_{\underline{L}^{p_\ast}(Q)}\,.
\end{equation*}
The elementary embedding~$L^2(Q)\hookrightarrow \underline{H}^{-\nf 14}(Q)$ gives, after rescaling,
\begin{equation}
	\label{e.Dir.new.Hminus.trivial}
	3^{-\frac14 n}\|F\|_{\underline{H}^{-\nf 14}(Q)}
	\leq
	C\|F\|_{\underline{L}^2(Q)}\,.
\end{equation}
By~\eqref{e.Dir.new.tilde.a.def},~$\tilde\a_m=\shom_m^{-1}\bigl((\nu-\shom_m)\Id+\k-(\k)_{\cu_m}\bigr)$. Using~\eqref{e.Dir.new.k.bounds}, the inclusion~$Q\subseteq\cu_m$, and the upper bound~$\shom_m\leq Cm$ for large~$m$, we have
\begin{equation}
	\label{e.Dir.new.tilde.a.Linfty}
	\|\tilde\a_m\|_{L^\infty(Q)}
	\leq
	C\shom_m^{-1}m^{1+\rho}\,.
\end{equation}
Here the constant part and the term~$\shom_m\Id$ are harmless after enlarging the deterministic threshold, since~$\nu\leq1$ and~$\shom_m\leq Cm$. Apply the homogeneous weak estimates~\eqref{e.Dir.new.weak.flux} and~\eqref{e.Dir.new.weak.grad} from Lemma~\ref{l.sharp.scale.inputs} to~$u_z$. Using~\eqref{e.Dir.new.Hminus.trivial}--\eqref{e.Dir.new.tilde.a.Linfty} and the preceding energy estimate,
\begin{align*}
	3^{-\frac14 n}\|\tilde\a_m\nabla u\|_{\underline{H}^{-\nf 14}(Q)}
	&\leq
	3^{-\frac14 n}\|\tilde\a_m\nabla u_z\|_{\underline{H}^{-\nf 14}(Q)}
	+C\shom_m^{-1}m^{1+\rho}\|\nabla u-\nabla u_z\|_{\underline{L}^2(Q)}
	\\
	&\leq
	C\shom_m^{-\nf12}\delta_m
	\nu^{\nf12}\|\nabla u\|_{\underline{L}^2(Q)}
	+C\shom_m^{-1}m^{1+\rho}3^n\nu^{-1}
	\|f\|_{\underline{L}^{p_\ast}(Q)}\,.
\end{align*}
This is~\eqref{e.Dir.new.flux.with.f}. The gradient estimate is identical, except that the coefficient bound~\eqref{e.Dir.new.tilde.a.Linfty} is not used:
\begin{align*}
	3^{-\frac14 n}\|\nabla u\|_{\underline{H}^{-\nf 14}(Q)}
	&\leq
	3^{-\frac14 n}\|\nabla u_z\|_{\underline{H}^{-\nf 14}(Q)}
	+C\|\nabla u-\nabla u_z\|_{\underline{L}^2(Q)}
	\\
	&\leq
	C\shom_m^{-\nf12}
	\nu^{\nf12}\|\nabla u\|_{\underline{L}^2(Q)}
	+C3^n\nu^{-1}
	\|f\|_{\underline{L}^{p_\ast}(Q)}\,.
\end{align*}
The proof is complete.
\end{proof}

\subsection{The Dirichlet problem in~\texorpdfstring{$L^2$}{L2}}
\label{ss.Dirichlet.L2.blackbox}

We now prove the~$L^2$ Dirichlet homogenization estimate by inserting Lemma~\ref{l.Dirichlet.rhs.blackbox} into a cutoff and mollification argument.

\begin{proposition}[Dirichlet homogenization in~$L^2$]
\label{p.Dirichlet.L2.blackbox}
Let~$U\subseteq\cu_0$ be a smooth bounded domain. There exists~$C(U,d)<\infty$ such that the following holds. Suppose
\begin{equation*}
	K\geq \hat L\vee C\,,
	\qquad
	3^K\geq \X_0\,.
\end{equation*}
Let~$g\in H^1(U_K)$,~$f\in L^{p_\ast}(U_K)$, and let~$u,\uhom\in H^1(U_K)$ solve
\begin{equation}
	\label{e.Dir.new.problems}
	\left\{
	\begin{aligned}
		-\nabla\cdot\a\nabla u&=f &&\mbox{in } U_K\,,\\
		u&=g &&\mbox{on } \partial U_K\,,
	\end{aligned}
	\right.
	\qquad
	\left\{
	\begin{aligned}
		-\shom_K\Delta\uhom&=f &&\mbox{in } U_K\,,\\
		\uhom&=g &&\mbox{on } \partial U_K\,.
	\end{aligned}
	\right.
\end{equation}
Then
\begin{equation}
	\label{e.Dir.new.L2.homog}
	3^{-K}\|u-\uhom\|_{\underline{L}^2(U_K)}
	\leq
	C\shom_K^{-\nf12}\delta_K
	\nu^{\nf12}\|\nabla u\|_{\underline{L}^2(U_K)}
	+CK^{-100}
	\bigl(
	\|\nabla g\|_{\underline{L}^2(U_K)}
	+3^K\|f\|_{\underline{L}^{p_\ast}(U_K)}
	\bigr)\,.
\end{equation}
\end{proposition}

\begin{proof}
We give the argument in detail. Put
\begin{equation*}
	h\coloneqq h_K
	\qand
	n\coloneqq K-h\,.
\end{equation*}
By increasing the lower bound on~$M$ and the deterministic lower threshold in~$\hat L$, we may assume throughout the proof that
\begin{equation}
	\label{e.Dir.new.h.choice}
	h\leq \frac K2
	\qand
	K^A\nu^{-A}\shom_K^A3^{\theta(n-K)}
	\leq
	K^{-300}\,,
\end{equation}
for every fixed dimensional power~$A<\infty$ and every fixed dimensional exponent~$\theta\in(0,1]$ which appears below. Let~$\eta=\eta_{3^n}$ be the even standard mollifier at scale~$3^n$. Choose~$\zeta\in C^\infty_c(U_K)$ such that~$0\leq\zeta\leq1$,
\begin{equation*}
	\|\nabla\zeta\|_{L^\infty(U_K)}
	\leq
	C3^{-n}\,,
\end{equation*}
\begin{equation*}
	\zeta=0
	\quad\mbox{if } \dist(x,\partial U_K)\leq 3^{n+d}
	\qand
	\zeta=1
	\quad\mbox{if } \dist(x,\partial U_K)\geq 3^{n+d+3}\,.
\end{equation*}
Set
\begin{equation*}
	U_K^\circ
	\coloneqq
	\bigl\{x\in U_K:\dist(x,\partial U_K)\geq 3^{n+d+3}\bigr\}
\end{equation*}
and let~$Z_{b,n}\subseteq3^n\Zd$ be a boundary grid set such that~$z+\cu_{n+1}\subseteq U_K$ for every~$z\in Z_{b,n}$,
\begin{equation*}
	\{\nabla\zeta\neq0\}
	\subseteq
	\bigcup_{z\in Z_{b,n}}(z+\cu_n)
	\qand
	\supp(1-\zeta)
	\subseteq
	\bigcup_{z\in Z_{b,n}}\bigl((z+\cu_{n+4+d})\cap U_K\bigr)\,.
\end{equation*}
The smoothness of~$\partial U$ gives
\begin{equation}
	\label{e.Dir.new.boundary.count}
	\sum_{z\in Z_{b,n}} |z+\cu_n|
	\leq
	C3^{n-K}|U_K|\,.
\end{equation}
We extend~$g$ to~$H^1(\Rd)$ with norm bounded by~$C(U)\|g\|_{H^1(U_K)}$, and extend~$u$ and~$\uhom$ outside~$U_K$ by setting them equal to~$g$.

\smallskip

Define
\begin{equation}
	\label{e.Dir.new.w.def}
	w
	\coloneqq
	\zeta(\eta\ast u)+(1-\zeta)g\,.
\end{equation}
Then~$w-\uhom\in H^1_0(U_K)$. We first record that~$w$ is close to~$u$. The mollification error away from the boundary satisfies
\begin{equation*}
	3^{-K}\|\zeta(\eta\ast u-u)\|_{\underline{L}^2(U_K)}
	\leq
	C3^{n-K}\|\nabla u\|_{\underline{L}^2(U_K)}\,.
\end{equation*}
On the boundary layer, Poincar\'e's inequality applied to~$u-g$, which vanishes on~$\partial U_K$, and~\eqref{e.Dir.new.boundary.count} give
\begin{equation*}
	3^{-K}\|(1-\zeta)(u-g)\|_{\underline{L}^2(U_K)}
	\leq
	C3^{n-K}
	\bigl(
	\|\nabla u\|_{\underline{L}^2(U_K)}
	+\|\nabla g\|_{\underline{L}^2(U_K)}
	\bigr)\,.
\end{equation*}
Consequently, by~\eqref{e.Dir.new.h.choice},
\begin{equation}
	\label{e.Dir.new.w.minus.u}
	3^{-K}\|w-u\|_{\underline{L}^2(U_K)}
	\leq
	CK^{-200}
	\bigl(
	\|\nabla u\|_{\underline{L}^2(U_K)}
	+\|\nabla g\|_{\underline{L}^2(U_K)}
	\bigr)\,.
\end{equation}
The first term on the right is absorbed into the leading term in~\eqref{e.Dir.new.L2.homog}. Indeed, after increasing~$\hat L$, we have~$K^{-200}\leq C\shom_K^{-\nf12}\delta_K\nu^{\nf12}$, since~$\hat L$ depends on~$\varepsilon$.

\smallskip

It remains to estimate~$w-\uhom$. By the Sobolev inequality in~$U_K$, it is enough to prove
\begin{align}
	\label{e.Dir.new.gradient.goal}
	\|\nabla w-\nabla\uhom\|_{\underline{L}^{p_\ast}(U_K)}
	&\leq
	C\shom_K^{-\nf12}\delta_K
	\nu^{\nf12}\|\nabla u\|_{\underline{L}^2(U_K)}
	+CK^{-100} \bigl( \|\nabla g\|_{\underline{L}^2(U_K)} {+} 3^K\|f\|_{\underline{L}^{p_\ast}(U_K)} \bigr)\,.
\end{align}
Indeed,~$w-\uhom=0$ on~$\partial U_K$ implies
\begin{equation*}
	3^{-K}\|w-\uhom\|_{\underline{L}^2(U_K)}
	\leq
	C\|\nabla w-\nabla\uhom\|_{\underline{L}^{p_\ast}(U_K)}\,.
\end{equation*}

\smallskip

We next compute the equation for~$w$. Since the divergence of a constant antisymmetric matrix times a gradient vanishes, we may subtract~$(\k)_{\cu_K}$ in the flux. Using the definition of~$\tilde\a_K$, we have
\begin{align}
	\label{e.Dir.new.convolution.identity}
	-\shom_K\Delta w-f
	&=
	\zeta\,\eta\ast f-f
	+\shom_K\nabla\cdot\bigl(\zeta\,\eta\ast(\tilde\a_K\nabla u)\bigr)
	\notag\\
	&\qquad
	-\shom_K\nabla\cdot\bigl(\nabla\zeta(\eta\ast u-g)+(1-\zeta)\nabla g\bigr)
	-\nabla\zeta\cdot\eta\ast\bigl((\a-(\k)_{\cu_K})\nabla u\bigr)\,.
\end{align}
The global Calder\'on--Zygmund estimate for the Laplacian in the smooth domain~$U_K$ (see, for instance,~\cite[Exercise~7.10 and Remark~7.11]{AKMBook}), applied to~$w-\uhom\in H^1_0(U_K)$ and used with~$p=p_\ast$, gives
\begin{align}
	\label{e.Dir.new.CZ}
	\shom_K
	\|\nabla w-\nabla\uhom\|_{\underline{L}^{p_\ast}(U_K)}
	&\leq
	C\shom_K
	\|\zeta\,\eta\ast(\tilde\a_K\nabla u)\|_{\underline{L}^{p_\ast}(U_K)}
	\notag\\
	&\qquad
	+C\shom_K
	\|\nabla\zeta(\eta\ast u-g)\|_{\underline{L}^{p_\ast}(U_K)}
	\notag\\
	&\qquad
	+C\shom_K
	\|(1-\zeta)\nabla g\|_{\underline{L}^{p_\ast}(U_K)}
	\notag\\
	&\qquad
	+C\|\zeta\,\eta\ast f-f\|_{\underline{W}^{-1,p_\ast}(U_K)}
	\notag\\
	&\qquad
	+C\|\nabla\zeta\cdot\eta\ast((\a-(\k)_{\cu_K})\nabla u)\|_{\underline{W}^{-1,p_\ast}(U_K)}\,.
\end{align}
We estimate the five terms on the right.

\smallskip

First consider the interior flux term. The elementary mollifier estimate
\begin{equation*}
	\|\eta\ast F\|_{L^\infty(z+\cu_n)}
	\leq
	C3^{-\frac14 n}
	\|F\|_{\underline{H}^{-\nf 14}(z+\cu_{n+1})}
\end{equation*}
together with Lemma~\ref{l.Dirichlet.rhs.blackbox}, used at the scale~$K$ on the bounded family of aligned scale~$n$ cubes meeting~$z+\cu_{n+1}$, yields for every~$z\in3^n\Zd$ with~$z+\cu_{n+1}\subseteq U_K$,
\begin{equation}
	\label{e.Dir.new.local.flux.mollified}
	\|\eta\ast(\tilde\a_K\nabla u)\|_{L^\infty(z+\cu_n)}
	\leq
	C\shom_K^{-\nf12}\delta_K
	\nu^{\nf12}
	\|\nabla u\|_{\underline{L}^2(z+\cu_{n+1})}
	{}+ CK^{-200}3^K
	\|f\|_{\underline{L}^{p_\ast}(z+\cu_{n+1})}\,.
\end{equation}
Here we used~\eqref{e.Dir.new.h.choice} to bound the right-hand-side term in~\eqref{e.Dir.new.flux.with.f}. To pass from the local estimate to the global one, we sum over the grid cubes on which~$\zeta$ may be nonzero. With the sums restricted to~$z\in3^n\Zd$ such that~$z+\cu_{n+1}\subseteq U_K$, finite overlap and~$p_\ast<2$ give
\begin{align*}
	\|\zeta\,\eta\ast(\tilde\a_K\nabla u)\|_{\underline{L}^{p_\ast}(U_K)}
	&\leq
	C\biggl(
	\avsum_z
	\|\eta\ast(\tilde\a_K\nabla u)\|_{L^\infty(z+\cu_n)}^{p_\ast}
	\biggr)^{\!\nf1{p_\ast}}
	\\
	&\leq
	C\shom_K^{-\nf12}\delta_K
	\nu^{\nf12}
	\biggl(
	\avsum_z
	\|\nabla u\|_{\underline{L}^2(z+\cu_{n+1})}^{p_\ast}
	\biggr)^{\!\nf1{p_\ast}}
	\\
	&\qquad
	{}+ CK^{-200}3^K
	\biggl(
	\avsum_z
	\|f\|_{\underline{L}^{p_\ast}(z+\cu_{n+1})}^{p_\ast}
	\biggr)^{\!\nf1{p_\ast}}\,.
\end{align*}
Using again finite overlap, this gives
\begin{align}
	\label{e.Dir.new.term.one}
	\|\zeta\,\eta\ast(\tilde\a_K\nabla u)\|_{\underline{L}^{p_\ast}(U_K)}
	&\leq
	C\shom_K^{-\nf12}\delta_K
	\nu^{\nf12}
	\|\nabla u\|_{\underline{L}^2(U_K)}
	\notag\\
	&\qquad
	+CK^{-200}3^K
	\|f\|_{\underline{L}^{p_\ast}(U_K)}\,.
\end{align}

\smallskip

We also need a crude bound for the full flux in the boundary layer. Since
\begin{equation*}
	(\a-(\k)_{\cu_K})\nabla u
	=
	\shom_K(\Id+\tilde\a_K)\nabla u\,,
\end{equation*}
the preceding estimate and the gradient estimate in Lemma~\ref{l.Dirichlet.rhs.blackbox}, again applied on the aligned scale~$n$ cubes meeting~$z+\cu_{n+1}$, imply, for~$z\in3^n\Zd$ with~$z+\cu_{n+1}\subseteq U_K$,
\begin{equation}
	\label{e.Dir.new.local.fullflux.mollified}
	\|\eta\ast((\a-(\k)_{\cu_K})\nabla u)\|_{L^\infty(z+\cu_n)}
	\leq
	C\shom_K^{\nf12}\nu^{\nf12}
	\|\nabla u\|_{\underline{L}^2(z+\cu_{n+1})}
	{}+ CK^{-200}\shom_K3^K
	\|f\|_{\underline{L}^{p_\ast}(z+\cu_{n+1})}\,.
\end{equation}
Let
\begin{equation*}
	\mathcal B_n
	\coloneqq
	\bigl\{x\in U_K:\dist(x,\partial U_K)\leq 3^{n+d+3}\bigr\}\,.
\end{equation*}
Write~$F\coloneqq(\a-(\k)_{\cu_K})\nabla u$ and~$\mathcal A_n\coloneqq\{\nabla\zeta\neq0\}$. Testing the last term in~\eqref{e.Dir.new.CZ} against~$\psi\in W^{1,2^\ast}_0(U_K)$ with~$\|\psi\|_{\underline{W}^{1,2^\ast}(U_K)}\leq1$, and using Poincar\'e on the boundary layer, gives
\begin{equation*}
	\biggl|
	\fint_{U_K}\psi\,\nabla\zeta\cdot\eta\ast F
	\biggr|
	\leq
	C|U_K|^{-\frac1{p_\ast}}\|\eta\ast F\|_{L^{p_\ast}(\mathcal A_n)}
	3^{-n}|U_K|^{-\frac1{2^\ast}}\|\psi\|_{L^{2^\ast}(\mathcal A_n)}
	\leq
	C|U_K|^{-\frac1{p_\ast}}\|\eta\ast F\|_{L^{p_\ast}(\mathcal A_n)}\,.
\end{equation*}
The powers of~$|U_K|$ are exactly those coming from normalized H\"older, since~$\frac1{p_\ast}+\frac1{2^\ast}=1$. Here the last inequality uses~$\mathcal A_n\subseteq\mathcal B_n$ and the boundary-layer Poincar\'e estimate
\begin{equation*}
	|U_K|^{-\frac1{2^\ast}}\|\psi\|_{L^{2^\ast}(\mathcal A_n)}
	\leq
	C3^n\|\nabla\psi\|_{\underline{L}^{2^\ast}(U_K)}
	\leq
	C3^n\,.
\end{equation*}
Taking the supremum over such~$\psi$, and then using~\eqref{e.Dir.new.boundary.count},~\eqref{e.Dir.new.local.fullflux.mollified}, and finite overlap, we obtain
\begin{align*}
	\shom_K^{-1}
	\|\nabla\zeta\cdot\eta\ast F\|_{\underline{W}^{-1,p_\ast}(U_K)}
	&\leq
	C\shom_K^{-\nf12}
	3^{(n-K)(\frac1{p_\ast}-\frac12)}
	\nu^{\nf12}
	\|\nabla u\|_{\underline{L}^2(U_K)}
	\\
	&\qquad
	{}+CK^{-200}3^K
	\|f\|_{\underline{L}^{p_\ast}(U_K)}\,.
\end{align*}
By~\eqref{e.Dir.new.h.choice}, the first term on the right is bounded by~$CK^{-200}\|\nabla u\|_{\underline{L}^2(U_K)}$. Thus
\begin{equation}
	\label{e.Dir.new.boundary.flux.term}
	\shom_K^{-1}
	\|\nabla\zeta\cdot\eta\ast F\|_{\underline{W}^{-1,p_\ast}(U_K)}
	\leq
	CK^{-200} \bigl(
	\|\nabla u\|_{\underline{L}^2(U_K)}
	+ 3^K
	\|f\|_{\underline{L}^{p_\ast}(U_K)} \bigr)\,.
\end{equation}
The first term on the right is again absorbed into the leading term of~\eqref{e.Dir.new.gradient.goal}.

\smallskip

The mollification error on the right-hand side is estimated by duality. We extend~$f$ and~$\psi$ by zero outside~$U_K$. If~$\psi\in W^{1,2^\ast}_0(U_K)$ and~$\|\psi\|_{\underline{W}^{1,2^\ast}(U_K)}\leq1$, then
\begin{align*}
	\biggl|\fint_{U_K}\psi(f-\zeta\,\eta\ast f)\biggr|
	&=
	\biggl|
	\fint_{U_K} f\bigl(\psi-\eta\ast(\zeta\psi)\bigr)
	\biggr|
	\\
	&\leq
	\|f\|_{\underline{L}^{p_\ast}(U_K)}
	\bigl(
	\|\psi-\eta\ast\psi\|_{\underline{L}^{2^\ast}(U_K)}
	+\|\eta\ast((1-\zeta)\psi)\|_{\underline{L}^{2^\ast}(U_K)}
	\bigr)\,.
\end{align*}
The first term is the standard mollification estimate,
\begin{equation*}
	\|\psi-\eta\ast\psi\|_{\underline{L}^{2^\ast}(U_K)}
	\leq
	C3^n\|\nabla\psi\|_{\underline{L}^{2^\ast}(U_K)}
	\leq
	C3^n\,.
\end{equation*}
For the boundary-layer term, the convolution is an~$L^{2^\ast}$ contraction and~$\supp(1-\zeta)\subseteq\mathcal B_n$. Since~$\psi$ has zero trace on~$\partial U_K$, the boundary-layer Poincar\'e inequality gives
\begin{align*}
	\|\eta\ast((1-\zeta)\psi)\|_{\underline{L}^{2^\ast}(U_K)}
	&\leq
	\|(1-\zeta)\psi\|_{\underline{L}^{2^\ast}(U_K)}
	\\
	&\leq
	C \biggl( \frac{|\cu_n|}{|U_K|} \sum_{z \in Z_{b,n}}
	\|\psi\|_{\underline{L}^{2^\ast}((z+\cu_{n+4+d})\cap U_K)}^{2^\ast} \biggr)^{\nf 1{2^\ast}}
	\leq
	C3^n\|\nabla\psi\|_{\underline{L}^{2^\ast}(U_K)}
	\leq
	C3^n\,.
\end{align*}
Therefore
\begin{equation*}
	\biggl|\fint_{U_K}\psi(f-\zeta\,\eta\ast f)\biggr|
	\leq
	C3^n\|f\|_{\underline{L}^{p_\ast}(U_K)}\,.
\end{equation*}
Thus, by~\eqref{e.Dir.new.h.choice},
\begin{equation}
	\label{e.Dir.new.f.mollification}
	\shom_K^{-1}
	\|\zeta\,\eta\ast f-f\|_{\underline{W}^{-1,p_\ast}(U_K)}
	\leq
	CK^{-200}3^K
	\|f\|_{\underline{L}^{p_\ast}(U_K)}\,.
\end{equation}

\smallskip

It remains to estimate the two boundary terms involving~$g$ and~$u-g$. By H\"older's inequality,~\eqref{e.Dir.new.boundary.count}, and~\eqref{e.Dir.new.h.choice},
\begin{equation}
	\label{e.Dir.new.boundary.g.term}
	\|(1-\zeta)\nabla g\|_{\underline{L}^{p_\ast}(U_K)}
	\leq
	CK^{-200}
	\|\nabla g\|_{\underline{L}^2(U_K)}\,.
\end{equation}
For the other term, split
\begin{equation*}
	\eta\ast u-g
	=
	(\eta\ast u-u)+(u-g)
\end{equation*}
on the support of~$\nabla\zeta$. The first part is controlled by the scale~$3^n$ mollification estimate:
\begin{equation*}
	\|\nabla\zeta(\eta\ast u-u)\|_{\underline{L}^{p_\ast}(U_K)}
	\leq
	C3^{-\frac{2-p_\ast}{2\cdot p_\ast}(K-n)}
	\|\nabla u\|_{\underline{L}^2(U_K)}\,.
\end{equation*}
For the second part we use Poincar\'e in the boundary layer, since~$u-g=0$ on~$\partial U_K$, and obtain
\begin{equation*}
	\|\nabla\zeta(u-g)\|_{\underline{L}^{p_\ast}(U_K)}
	\leq
	C3^{-\frac{2-p_\ast}{2\cdot p_\ast}(K-n)}
	\bigl(
	\|\nabla u\|_{\underline{L}^2(U_K)}
	+\|\nabla g\|_{\underline{L}^2(U_K)}
	\bigr)\,.
\end{equation*}
By~\eqref{e.Dir.new.h.choice}, these estimates imply
\begin{equation}
	\label{e.Dir.new.boundary.u.term}
	\|\nabla\zeta(\eta\ast u-g)\|_{\underline{L}^{p_\ast}(U_K)}
	\leq
	CK^{-200}
	\bigl(
	\|\nabla u\|_{\underline{L}^2(U_K)}
	+\|\nabla g\|_{\underline{L}^2(U_K)}
	\bigr)\,.
\end{equation}

\smallskip

Divide~\eqref{e.Dir.new.CZ} by~$\shom_K$ and insert~\eqref{e.Dir.new.term.one},~\eqref{e.Dir.new.boundary.flux.term},~\eqref{e.Dir.new.f.mollification},~\eqref{e.Dir.new.boundary.g.term}, and~\eqref{e.Dir.new.boundary.u.term}. The terms proportional to~$K^{-200}\|\nabla u\|_{\underline{L}^2(U_K)}$ are absorbed into the leading term by the deterministic lower bound in~$\hat L$, and the remaining errors are bounded by the last line of~\eqref{e.Dir.new.gradient.goal}. This proves~\eqref{e.Dir.new.gradient.goal}. Combining this with~\eqref{e.Dir.new.w.minus.u} and the Sobolev step used in~\eqref{e.Dir.new.gradient.goal} gives~\eqref{e.Dir.new.L2.homog}.
\end{proof}

\subsection{Caccioppoli and large-scale Lipschitz estimates}
\label{ss.Dirichlet.C01.blackbox}

The next estimates are deterministic consequences of the weak bounds above. They account for the right-hand side and the boundary geometry in the proof of the pointwise estimates.

\begin{lemma}[Superdiffusive Caccioppoli]
\label{l.Dirichlet.Cacc.blackbox}
Let~$U\subseteq\cu_0$ be smooth and bounded. There exist~$C(U,d)<\infty$ and~$n_U=n(U,d)\in\N$ such that the following holds. Let~$y\in U_K$ be fixed, and assume that~$K\in\N$,~$m\geq\hat L$, and~$3^m\geq\X_0(y)$. If~$y+\cu_m\subseteq U_K$,~$f\in L^2(y+\cu_m)$, and~$u$ solves
\begin{equation*}
	-\nabla\cdot\a\nabla u=f
\end{equation*}
in~$y+\cu_m$, then
\begin{equation}
	\label{e.Dir.new.Cacc.interior}
	\nu\|\nabla u\|_{\underline{L}^2(y+\cu_{m-1})}^2
	\leq
	C\shom_m3^{-2m}\|u-(u)_{y+\cu_m}\|_{\underline{L}^2(y+\cu_m)}^2
	+C\shom_m^{-1}3^{2m}\|f\|_{\underline{L}^2(y+\cu_m)}^2\,.
\end{equation}
Moreover, if~$K-m\geq n_U$,~$u\in H^1(U_K)$ solves
\begin{equation*}
	-\nabla\cdot\a\nabla u=f
	\quad\mbox{in } U_K\,,
	\qquad
	u=g
	\quad\mbox{on } \partial U_K\,,
\end{equation*}
with~$f\in L^2(U_K)$ and~$g\in H^2(U_K)$, then, writing~$\hat\cu_m(y)\coloneqq (y+\cu_m)\cap U_K$,
\begin{align}
	\nu\|\nabla u\|_{\underline{L}^2(\hat\cu_{m-1}(y))}^2
	&\leq
	C\shom_m3^{-2m}\|u-g\|_{\underline{L}^2(\hat\cu_m(y))}^2
	+C\shom_m^{-1}3^{2m}\|f\|_{\underline{L}^2(\hat\cu_m(y))}^2
	\notag\\
	&\qquad
	+C\shom_m\|\nabla g\|_{\underline{L}^2(\hat\cu_m(y))}^2
	+Cm^{-2000}3^{2m}\|\nabla^2g\|_{\underline{L}^2(\hat\cu_m(y))}^2\,.
	\label{e.Dir.new.Cacc.boundary}
\end{align}
For the same Dirichlet problem, if instead~$K\geq\hat L$ and~$3^K\geq\X_0$, then the corresponding global estimate is
\begin{equation}
	\label{e.Dir.new.Cacc.global}
	\nu\|\nabla u\|_{\underline{L}^2(U_K)}^2
	\leq
	C\shom_K^{-1}3^{2K}\|f\|_{\underline{L}^2(U_K)}^2
	+C\shom_K\|\nabla g\|_{\underline{L}^2(U_K)}^2
	+CK^{-2000}3^{2K}\|\nabla^2g\|_{\underline{L}^2(U_K)}^2\,.
\end{equation}
\end{lemma}

\begin{proof}
For the local estimates we give the proof for~$y=0$. For general~$y$, Lemma~\ref{l.sharp.scale.inputs} supplies the same estimates for the translated coefficient field, with~$\X_0(y)$ as the minimal scale, and the proof is unchanged. Put
\begin{equation*}
	h\coloneqq h_m
	\qand
	n\coloneqq m-h\,.
\end{equation*}
After increasing the deterministic lower bound in~$\hat L$ and the lower bound on~$M$, we may assume that~$h\leq m/10$ and that, for a fixed dimensional~$A_0<\infty$ large enough for every polynomial prefactor used below and every fixed~$\theta\in(0,1]$ used below,
\begin{equation}
	\label{e.Dir.new.Cacc.scale.separation}
	m^{A_0}\nu^{-A_0}\shom_m^{A_0}3^{\theta(n-m)}
	\leq m^{-3000}\,.
\end{equation}
Since~$p_\ast\leq2$, normalized~$L^2$ bounds control the~$\underline L^{p_\ast}$ norms required by Lemma~\ref{l.Dirichlet.rhs.blackbox}. We shall repeatedly use the following consequence of that lemma. If~$z\in3^{n-3}\Zd$ satisfies~$z+\cu_n\subseteq\cu_m$ and this cube is contained in the domain where~$u$ solves the equation, and if~$\xi\in H^1(z+\cu_n;\Rd)$, then
\begin{align}
	\lefteqn{
	\biggl|
	\fint_{z+\cu_n}
	\bigl(\a-(\k)_{\cu_m}\bigr)\nabla u\cdot\xi
	\biggr|
	} \qquad &
	\notag \\ &
	\leq
	C\shom_m^{\nf12}
	\bigl(
	\|\xi\|_{\underline{L}^2(z+\cu_n)}
	+3^n\|\nabla\xi\|_{\underline{L}^2(z+\cu_n)}
	\bigr)
	\bigl(
	\nu^{\nf12}\|\nabla u\|_{\underline{L}^2(z+\cu_n)}
	+m^2\nu^{-1}3^n\|f\|_{\underline{L}^{p_\ast}(z+\cu_n)}
	\bigr)\,.
	\label{e.Dir.new.Cacc.weak.test}
\end{align}
Indeed, write
\begin{equation*}
	\a-(\k)_{\cu_m}
	=\shom_m\Id+\shom_m\tilde\a_m\,.
\end{equation*}
The first term is estimated by the weak gradient bound~\eqref{e.Dir.new.grad.with.f}, the second by the weak flux bound~\eqref{e.Dir.new.flux.with.f}, and the factor~$m^2$ absorbs~$\shom_m$ and~$m^{1+\rho}$ using~\eqref{e.sm.sharp.bounds}.

\smallskip

We next record the crude estimate used to control the terms in which a full gradient appears at the small scale. Write~$\hat\cu_m\coloneqq\cu_m\cap U_K$ in the boundary case and~$\hat\cu_m\coloneqq\cu_m$ in the interior case. Let~$w=u-(u)_{\cu_m}$ in the interior case, and let~$w=u-g$ in the boundary case. The ordinary Caccioppoli inequality, using the pointwise bound following from~\eqref{e.Dir.new.k.bounds}, gives
\begin{align}
	\label{e.Dir.new.Cacc.crude}
	\nu^{\nf12}
	\|\nabla w\|_{\underline{L}^2(\frac12\cu_m\cap U_K)}
	&\leq
	C\nu^{-2}m^2
	\bigl(
	3^{-m}\|w\|_{\underline{L}^2(\hat\cu_m)}
	+\|\nabla g\|_{\underline{L}^2(\hat\cu_m)}
	+3^m\|f\|_{\underline{L}^2(\hat\cu_m)}
	\bigr)\,,
\end{align}
where the terms involving~$g$ are omitted in the interior case.

\smallskip

We first prove the local estimate, including the boundary case. Choose~$\varphi\in C^\infty_c(\frac12\cu_m)$ such that~$\indc_{\cu_{m-1}}\leq\varphi\leq1$,
\begin{equation*}
	3^m\|\nabla\varphi\|_{L^\infty(\Rd)}
	+3^{2m}\|\nabla^2\varphi\|_{L^\infty(\Rd)}
	\leq C\,,
\end{equation*}
and, for every~$z\in3^n\Zd$ with~$z+\cu_{n+1}\subseteq\{\varphi>0\}$,
\begin{equation}
	\label{e.Dir.new.Cacc.cutoff.harnack}
	\sup_{z+\cu_n}\varphi
	\leq C\inf_{z+\cu_n}\varphi\,.
\end{equation}
Set
\begin{equation*}
	\tilde\cu_m\coloneqq U_K\cap\{\varphi>0\}
	\qand
	\tilde\cu_m^\circ
	\coloneqq
	\{x\in\tilde\cu_m:x+\cu_{n+3}\subseteq\tilde\cu_m\}\,.
\end{equation*}
In the interior case the same notation is used with~$U_K$ replaced by~$\cu_m$. Since the divergence of a constant antisymmetric matrix times a gradient vanishes, we may subtract~$(\k)_{\cu_m}$ from the coefficient. Testing the equation with~$w\varphi^2$ gives
\begin{align*}
	\fint_{\tilde\cu_m} f w\varphi^2
	&=
	\fint_{\tilde\cu_m}
	\bigl(\a-(\k)_{\cu_m}\bigr)\nabla u\cdot\nabla(w\varphi^2)
	\\
	&=
	\nu\fint_{\tilde\cu_m}\varphi^2|\nabla u|^2
	+\fint_{\tilde\cu_m}
	\bigl(\a-(\k)_{\cu_m}\bigr)\nabla u\cdot w\nabla\varphi^2
	-\fint_{\tilde\cu_m}
	\bigl(\a-(\k)_{\cu_m}\bigr)\nabla u\cdot\varphi^2\nabla g\,,
\end{align*}
where the final term is omitted in the interior case. Here we used that the symmetric part of~$\a-(\k)_{\cu_m}$ is~$\nu\Id$ and that the antisymmetric part has zero quadratic form. Young's inequality gives
\begin{equation*}
	\biggl|\fint_{\tilde\cu_m} f w\varphi^2\biggr|
	\leq
	C\shom_m3^{-2m}\|w\|_{\underline{L}^2(\tilde\cu_m)}^2
	+C\shom_m^{-1}3^{2m}\|f\|_{\underline{L}^2(\tilde\cu_m)}^2\,.
\end{equation*}
Therefore,
\begin{align}
	\label{e.Dir.new.Cacc.basic.decomp}
	\nu\|\varphi\nabla u\|_{\underline{L}^2(\tilde\cu_m)}^2
	&\leq
	C\shom_m3^{-2m}\|w\|_{\underline{L}^2(\tilde\cu_m)}^2
	+C\shom_m^{-1}3^{2m}\|f\|_{\underline{L}^2(\tilde\cu_m)}^2
	+T_\varphi+T_g\,,
\end{align}
where
\begin{equation*}
	T_\varphi
	\coloneqq
	\biggl|
	\fint_{\tilde\cu_m}
	\bigl(\a-(\k)_{\cu_m}\bigr)\nabla u\cdot w\nabla\varphi^2
	\biggr|\,.
\end{equation*}
In the boundary case,
\begin{equation*}
	T_g
	\coloneqq
	\biggl|
	\fint_{\tilde\cu_m}
	\bigl(\a-(\k)_{\cu_m}\bigr)\nabla u\cdot \varphi^2\nabla g
	\biggr|\,.
\end{equation*}
In the interior case we set~$T_g=0$. We now make the split into interior and boundary cubes precise. Let
\begin{equation*}
	Z_{\rm int}
	\coloneqq
	3^n\Zd\cap\tilde\cu_m^\circ
	\qand
	Z_b
	\coloneqq
	\{z\in3^n\Zd:(z+\cu_n)\cap\tilde\cu_m\neq\emptyset,\ z\notin Z_{\rm int}\}\,.
\end{equation*}
Define the interior contribution
\begin{equation*}
	T_\varphi^{\rm int}
	\coloneqq
	\frac1{|\cu_m|}
	\sum_{z\in Z_{\rm int}}
	\biggl|
	\int_{z+\cu_n}
	\bigl(\a-(\k)_{\cu_m}\bigr)\nabla u\cdot w\nabla\varphi^2
	\biggr|\,.
\end{equation*}
Define also
\begin{align*}
	T_\varphi^{\rm bdry}
	&\coloneqq
	\frac1{|\cu_m|}
	\sum_{z\in Z_b}
	\int_{(z+\cu_n)\cap\hat\cu_m}
	\bigl|
	\bigl(\a-(\k)_{\cu_m}\bigr)\nabla u\cdot w\nabla\varphi^2
	\bigr|\,,
\end{align*}
and, in the boundary case,
\begin{align*}
	T_g^{\rm int}
	&\coloneqq
	\frac1{|\cu_m|}
	\sum_{z\in Z_{\rm int}}
	\biggl|
	\int_{z+\cu_n}
	\bigl(\a-(\k)_{\cu_m}\bigr)\nabla u\cdot\varphi^2\nabla g
	\biggr|\,,
	\\
	T_g^{\rm bdry}
	&\coloneqq
	\frac1{|\cu_m|}
	\sum_{z\in Z_b}
	\int_{(z+\cu_n)\cap\hat\cu_m}
	\bigl|
	\bigl(\a-(\k)_{\cu_m}\bigr)\nabla u\cdot\varphi^2\nabla g
	\bigr|\,.
\end{align*}
In the interior case we set~$T_g^{\rm int}=T_g^{\rm bdry}=0$. Since~$|\tilde\cu_m|\simeq_U|\cu_m|$ and the grid cubes have bounded overlap,
\begin{equation*}
	T_\varphi+T_g
	\leq
	C\bigl(
	T_\varphi^{\rm int}
	+T_\varphi^{\rm bdry}
	+T_g^{\rm int}
	+T_g^{\rm bdry}
	\bigr)\,.
\end{equation*}

\smallskip

We estimate~$T_\varphi^{\rm int}$ by applying~\eqref{e.Dir.new.Cacc.weak.test} on each cube~$z+\cu_n$ with~$z\in Z_{\rm int}$ and with~$\xi=w\nabla\varphi^2$. Using~\eqref{e.Dir.new.Cacc.cutoff.harnack},~$3^m|\nabla\varphi|\leq C$, and finite overlap,
\begin{align*}
	T_\varphi^{\rm int}
	&\leq
	C\shom_m^{\nf12}3^{-m}
	\bigl(
	\|w\|_{\underline{L}^2(\tilde\cu_m)}
	+3^n\|\nabla w\|_{\underline{L}^2(\frac12\cu_m\cap\hat\cu_m)}
	\bigr)
	\notag\\
	&\qquad
	\times
	\bigl(
	\nu^{\nf12}\|\varphi\nabla u\|_{\underline{L}^2(\tilde\cu_m)}
	+m^2\nu^{-1}3^n\|f\|_{\underline{L}^2(\tilde\cu_m)}
	\bigr)\,.
\end{align*}
Combining this with~\eqref{e.Dir.new.Cacc.crude},~\eqref{e.Dir.new.Cacc.scale.separation}, and Young's inequality gives
\begin{align}
	\label{e.Dir.new.Cacc.Tphi}
	T_\varphi^{\rm int}
	&\leq
	\frac1{100}\nu\|\varphi\nabla u\|_{\underline{L}^2(\tilde\cu_m)}^2
	+C\shom_m3^{-2m}\|w\|_{\underline{L}^2(\hat\cu_m)}^2
	\notag\\
	&\qquad
	+C\shom_m^{-1}3^{2m}\|f\|_{\underline{L}^2(\hat\cu_m)}^2
	+Cm^{-1000}\shom_m\|\nabla g\|_{\underline{L}^2(\hat\cu_m)}^2\,.
\end{align}
In the boundary case, for~$T_g^{\rm int}$ on the same interior cubes, we use~\eqref{e.Dir.new.Cacc.weak.test} with~$\xi=\varphi^2\nabla g$. This gives
\begin{align*}
	T_g^{\rm int}
	&\leq
	C\shom_m^{\nf12}
	\bigl(
	\|\nabla g\|_{\underline{L}^2(\tilde\cu_m)}
	+3^n\|\nabla^2g\|_{\underline{L}^2(\tilde\cu_m)}
	\bigr)
	\notag\\
	&\qquad
	\times
	\bigl(
	\nu^{\nf12}\|\varphi\nabla u\|_{\underline{L}^2(\tilde\cu_m)}
	+m^2\nu^{-1}3^n\|f\|_{\underline{L}^2(\tilde\cu_m)}
	\bigr)\,.
\end{align*}
Therefore, by~\eqref{e.Dir.new.Cacc.scale.separation} and Young's inequality,
\begin{align}
	\label{e.Dir.new.Cacc.Tg}
	T_g^{\rm int}
	&\leq
	\frac1{100}\nu\|\varphi\nabla u\|_{\underline{L}^2(\tilde\cu_m)}^2
	+C\shom_m\|\nabla g\|_{\underline{L}^2(\hat\cu_m)}^2
	\notag\\
	&\qquad
	+C\shom_m^{-1}3^{2m}\|f\|_{\underline{L}^2(\hat\cu_m)}^2
	+Cm^{-2000}3^{2m}\|\nabla^2g\|_{\underline{L}^2(\hat\cu_m)}^2\,.
\end{align}
Here the norms over~$\tilde\cu_m$ were replaced by norms over~$\hat\cu_m$, using~$\tilde\cu_m\subseteq\hat\cu_m$ and~$|\tilde\cu_m|\simeq|\hat\cu_m|$.

\smallskip

It remains to estimate the contribution of the cubes in~$Z_b$. The smoothness of~$\partial U$ and the construction of~$\varphi$ imply
\begin{equation*}
	\sum_{z\in Z_b}|z+\cu_n|
	\leq C3^{n-m}|\cu_m|\,.
\end{equation*}
Thus, writing
\begin{equation*}
	B\coloneqq
	\bigcup_{z\in Z_b}\bigl((z+\cu_n)\cap\hat\cu_m\bigr)\,,
\end{equation*}
we have~$|B|\leq C3^{n-m}|\cu_m|$. By H\"older's inequality and the coefficient bound~\eqref{e.Dir.new.k.bounds}, for every~$\xi\in H^1(\hat\cu_m;\Rd)$ with~$\supp\xi\subseteq\frac12\cu_m$,
\begin{align*}
	\frac1{|\cu_m|}
	\sum_{z\in Z_b}
	\int_{(z+\cu_n)\cap\hat\cu_m}
	\bigl|
	\bigl(\a-(\k)_{\cu_m}\bigr)\nabla u\cdot\xi
	\bigr|
	&\leq
	Cm^2
	\biggl(\frac{|B|}{|\cu_m|}\biggr)^{\frac12-\frac1{2^\ast}}
	\|\xi\|_{\underline{L}^{2^\ast}(\hat\cu_m)}
	\|\nabla u\|_{\underline{L}^2(\frac12\cu_m\cap\hat\cu_m)}
	\\
	&\leq
	Cm^2
	3^{-\frac{1}{d\vee6}(m-n)}
	\|\xi\|_{\underline{L}^{2^\ast}(\hat\cu_m)}
	\|\nabla u\|_{\underline{L}^2(\frac12\cu_m\cap\hat\cu_m)}\,.
\end{align*}
The restriction~$K-m\geq n_U$ places every fixed enlargement used below in one boundary chart, so the local Poincar\'e and Sobolev--Poincar\'e constants are bounded by~$C(U,d)$.
The Sobolev--Poincar\'e inequality on~$\hat\cu_m$ gives
\begin{equation*}
	\|\xi\|_{\underline{L}^{2^\ast}(\hat\cu_m)}
	\leq
	C\bigl(
	\|\xi\|_{\underline{L}^2(\hat\cu_m)}
	+3^m\|\nabla\xi\|_{\underline{L}^2(\hat\cu_m)}
	\bigr)\,.
\end{equation*}
Combining the previous two displays, we obtain
\begin{align}
	\lefteqn{
	\frac1{|\cu_m|}
	\sum_{z\in Z_b}
	\int_{(z+\cu_n)\cap\hat\cu_m}
	\bigl|
	\bigl(\a-(\k)_{\cu_m}\bigr)\nabla u\cdot\xi
	\bigr|
	} \qquad &
	\notag \\ &
	\leq
	Cm^2
	3^{-\frac{1}{d\vee6}(m-n)}
	\bigl(
	\|\xi\|_{\underline{L}^2(\hat\cu_m)}
	+3^m\|\nabla\xi\|_{\underline{L}^2(\hat\cu_m)}
	\bigr)
	\|\nabla u\|_{\underline{L}^2(\frac12\cu_m\cap\hat\cu_m)}\,.
	\label{e.Dir.new.Cacc.boundary.layer.raw}
\end{align}
The support condition is used only to localize the last factor to~$\frac12\cu_m\cap\hat\cu_m$. In the interior case this unweighted gradient is controlled directly by~\eqref{e.Dir.new.Cacc.crude}; in the boundary case we use~$\|\nabla u\|\leq\|\nabla(u-g)\|+\|\nabla g\|$ and then apply~\eqref{e.Dir.new.Cacc.crude}, with the additional~$\nabla g$ contribution absorbed into the boundary-data term below. The two choices of~$\xi$ below satisfy the support condition because~$\varphi$ is supported in~$\frac12\cu_m$. Applying~\eqref{e.Dir.new.Cacc.boundary.layer.raw} with~$\xi=w\nabla\varphi^2$ and, in the boundary case, with~$\xi=\varphi^2\nabla g$, and then using~\eqref{e.Dir.new.Cacc.crude},~\eqref{e.Dir.new.Cacc.scale.separation}, and Young's inequality, yields
\begin{align}
	\label{e.Dir.new.Cacc.boundary.layer}
	T_\varphi^{\rm bdry}+T_g^{\rm bdry}
	&\leq
	Cm^{-1000}\shom_m3^{-2m}\|w\|_{\underline{L}^2(\hat\cu_m)}^2
	+Cm^{-1000}\shom_m^{-1}3^{2m}\|f\|_{\underline{L}^2(\hat\cu_m)}^2
	\notag\\
	&\qquad
	+Cm^{-1000}\shom_m\|\nabla g\|_{\underline{L}^2(\hat\cu_m)}^2
	+Cm^{-2000}3^{2m}\|\nabla^2g\|_{\underline{L}^2(\hat\cu_m)}^2\,.
\end{align}

\smallskip

Combining~\eqref{e.Dir.new.Cacc.basic.decomp},~\eqref{e.Dir.new.Cacc.Tphi},~\eqref{e.Dir.new.Cacc.Tg}, and~\eqref{e.Dir.new.Cacc.boundary.layer}, and absorbing the small multiples of~$\nu\|\varphi\nabla u\|_{\underline{L}^2}^2$ into the left side, gives~\eqref{e.Dir.new.Cacc.boundary}. In the interior case the same argument, with~$g$ replaced by the constant~$(u)_{\cu_m}$ so that the terms involving~$g$ vanish, gives~\eqref{e.Dir.new.Cacc.interior}.

\smallskip

Finally, we prove the global estimate~\eqref{e.Dir.new.Cacc.global}. Put~$w=u-g$ and repeat the basic testing argument with~$m=K$ and~$\varphi\equiv1$. The terms involving~$T_\varphi$ disappear, while~$T_g$ is estimated by the same interior and boundary decompositions with~$K$ in place of~$m$. The boundary-layer estimate is used in its non-localized version, with the last factor bounded by~$\|\nabla u\|_{\underline{L}^2(U_K)}$, and this term is absorbed by Young's inequality. This gives
\begin{align*}
	\nu\|\nabla u\|_{\underline{L}^2(U_K)}^2
	&\leq
	C\shom_K3^{-2K}\|w\|_{\underline{L}^2(U_K)}^2
	+C\shom_K^{-1}3^{2K}\|f\|_{\underline{L}^2(U_K)}^2
	\notag\\
	&\qquad
	+\frac14\nu\|\nabla u\|_{\underline{L}^2(U_K)}^2
	+C\shom_K\|\nabla g\|_{\underline{L}^2(U_K)}^2
	+CK^{-2000}3^{2K}\|\nabla^2g\|_{\underline{L}^2(U_K)}^2\,.
\end{align*}
It remains to remove the term involving~$w$. Let~$\uhom$ solve the homogenized Dirichlet problem with the same right side and boundary data. Proposition~\ref{p.Dirichlet.L2.blackbox}, the constant-coefficient energy estimate for~$\uhom$, and Poincar\'e's inequality for~$\uhom-g\in H^1_0(U_K)$ give
\begin{align*}
	3^{-K}\|w\|_{\underline{L}^2(U_K)}
	&\leq
	3^{-K}\|u-\uhom\|_{\underline{L}^2(U_K)}
	+3^{-K}\|\uhom-g\|_{\underline{L}^2(U_K)}
	\\
	&\leq
	C\shom_K^{-\nf12}\delta_K
	\nu^{\nf12}\|\nabla u\|_{\underline{L}^2(U_K)}
	+C\shom_K^{-1}3^K\|f\|_{\underline{L}^2(U_K)}
	+C\|\nabla g\|_{\underline{L}^2(U_K)}\,.
\end{align*}
Here the~$K^{-100}$ errors in Proposition~\ref{p.Dirichlet.L2.blackbox} were absorbed into the last two terms by the deterministic lower bound on~$\hat L$. Squaring the last display and inserting it into the preceding estimate, the term proportional to~$\delta_K^2\nu\|\nabla u\|_{\underline{L}^2(U_K)}^2$ is absorbed into the left side after increasing~$\hat L$, and the remaining terms give~\eqref{e.Dir.new.Cacc.global}.
\end{proof}

\smallskip

We now combine the Caccioppoli estimate with the harmonic approximation estimate to obtain the large-scale Lipschitz estimate used in the pointwise estimates.

\begin{proposition}[Large-scale Lipschitz estimate]
\label{p.Dirichlet.C01.blackbox}
Let~$U\subseteq\cu_0$ be smooth and bounded. There exist~$C(U,d)<\infty$ and~$c(d)>0$ such that the following holds. If
\begin{equation*}
	n<m\,,
	\qquad
	n\geq \hat L\,,
	\qquad
	m\geq \hat L\,,
	\qquad
	m-n\leq c\delta_m^{-1}\,,
	\qquad
	3^n\geq \X_0\,,
\end{equation*}
and~$f\in L^\infty(\cu_m)$ and~$u$ solves~$-\nabla\cdot\a\nabla u=f$ in~$\cu_m$, then
\begin{align}
	\shom_m^{-\nf12}\nu^{\nf12}
	\|\nabla u\|_{\underline{L}^2(\cu_n)}
	+ 3^{-n}\|u-(u)_{\cu_n}\|_{\underline{L}^2(\cu_n)}
	&\leq
	C3^{-m}\|u-(u)_{\cu_m}\|_{\underline{L}^2(\cu_m)}
	+C\shom_m^{-1}3^m\|f\|_{L^\infty(\cu_m)}\,.
	\label{e.Dir.new.C01}
\end{align}
The following boundary versions also hold after taking~$M\geq C(U,d)$ sufficiently large and enlarging~$\hat L$, with the condition~$3^n\geq\X_0$ above replaced by
\begin{equation*}
	m-h_m\leq n
	\qand
	3^m\geq\X_0\,.
\end{equation*}
If~$g\in W^{2,\infty}(U_m)$,~$f\in L^\infty(U_m)$, and~$u\in H^1(U_m)$ solves
\begin{equation*}
	-\nabla\cdot\a\nabla u=f
	\quad\mbox{in }U_m\,,
	\qquad
	u=g
	\quad\mbox{on }\partial U_m\,,
\end{equation*}
then, writing~$\hat\cu_n(z)\coloneqq (z+\cu_n)\cap U_m$, for every~$z\in3^n\Zd\cap U_m$ we have
\begin{align}
	\shom_m^{-\nf12}\nu^{\nf12}
	\|\nabla u\|_{\underline{L}^2(\hat\cu_n(z))}
	&
	\leq
	C3^{-m}
	\|u-(u)_{U_m}\|_{\underline{L}^2(U_m)}
	+C\shom_m^{-1}3^m\|f\|_{L^\infty(U_m)}
	\notag \\ & \qquad
	+C(m-n)\|\nabla g\|_{L^\infty(U_m)}
	+Cm^{-1000}3^m\|\nabla^2g\|_{L^\infty(U_m)}\,.
	\label{e.Dir.new.C01.boundary}
\end{align}
\end{proposition}

\begin{proof}
We first prove~\eqref{e.Dir.new.C01}. For~$k\in[n,m-1]$, let~$V_k$ be a smooth domain with~$\cu_{k-1}\subseteq V_k\subseteq\cu_k$, and let~$\uhom^{(k)}\in u+H^1_0(V_k)$ solve~$-\shom_k\Delta\uhom^{(k)}=f$ in~$V_k$. Proposition~\ref{p.Dirichlet.L2.blackbox}, the deterministic comparison of~$\shom_k$ and~$\shom_m$, and Lemma~\ref{l.Dirichlet.Cacc.blackbox} give
\begin{equation*}
	3^{-k}\|u-\uhom^{(k)}\|_{\underline{L}^2(V_k)}
	\leq
	C\delta_k3^{-k}\|u-(u)_{\cu_{k+1}}\|_{\underline{L}^2(\cu_{k+1})}
	+C\shom_m^{-1}3^k\|f\|_{L^\infty(\cu_{k+1})}\,.
\end{equation*}
Comparing~$\uhom^{(k)}$ with the solution of the constant-coefficient harmonic equation in~$V_k$, and then using the interior affine approximation for harmonic functions, yields, for a dimensional~$k_0$ chosen large enough,
\begin{equation*}
	E_{k-k_0}
	\leq
	\frac12 E_{k+1}
	+C\shom_m^{-1}3^k\|f\|_{L^\infty(\cu_{k+1})}
	+C\delta_k|\nabla\ell_{k+1}|\,,
\end{equation*}
where
\begin{equation*}
	E_j\coloneqq
	\inf_{\ell\ {\rm affine}}3^{-j}\|u-\ell\|_{\underline{L}^2(\cu_j)}\,,
\end{equation*}
and~$\ell_j$ is a minimizing affine. The affine Poincar\'e inequality gives the top-scale bound
\begin{equation*}
	|\nabla\ell_m|
	\leq
	C3^{-m}\|u-(u)_{\cu_m}\|_{\underline{L}^2(\cu_m)}\,.
\end{equation*}
The slopes telescope as
\begin{equation*}
	|\nabla\ell_k|
	\leq
	C|\nabla\ell_m|
	+C\sum_{j=k}^{m-1} E_j\,.
\end{equation*}
The constraint~$m-n\leq c\delta_m^{-1}$ implies~$n\geq\frac12 m$ after increasing~$\hat L$, and therefore~$\delta_n\leq C\delta_m$. We use an iteration over a finite interval of scales for inequalities of the form~$E_{k-k_0}\leq\frac12E_{k+1}+a_k+b_k\sum_{j=k}^{m}E_j$, with~$\sum b_k$ small. Summing the inequality for~$E_{k-k_0}$ over~$k\in[n,m-1]$ and using~$\delta_k\leq C\delta_n$ and~$(m-n)\delta_n\leq C(m-n)\delta_m\leq c$ absorbs the terms~$C\delta_k\sum_{j=k+1}^{m-1}E_j$. The remaining contribution involving~$|\nabla\ell_m|$ is bounded by~$C3^{-m}\|u-(u)_{\cu_m}\|_{\underline{L}^2(\cu_m)}$. We obtain
\begin{equation*}
	\max_{k\in[n,m]}
	3^{-k}\|u-(u)_{\cu_k}\|_{\underline{L}^2(\cu_k)}
	\leq
	C3^{-m}\|u-(u)_{\cu_m}\|_{\underline{L}^2(\cu_m)}
	+C\shom_m^{-1}3^m\|f\|_{L^\infty(\cu_m)}\,.
\end{equation*}
The displayed oscillation estimate and the interior Caccioppoli estimate in Lemma~\ref{l.Dirichlet.Cacc.blackbox} give~\eqref{e.Dir.new.C01}.

\smallskip

We next prove the boundary estimate. Fix~$z\in3^n\Zd\cap U_m$ and write
\begin{equation*}
	\hat\cu_j\coloneqq(z+\cu_j)\cap U_m\,.
\end{equation*}
Since~$n\geq m-h_m$, the lattice nesting gives~$z\in3^{m-h_m}\Zd\cap\cu_m$. The final assertion of Lemma~\ref{l.sharp.scale.inputs} therefore supplies its estimates for the coefficient field~$x\mapsto\a(z+x)$ for every~$k\in[n,m]\cap\N$ and~$\ell\in[k-h_k,k]\cap\N$.

\smallskip

Fix~$k_0=k_0(U,d)$ so that the boundary and interior estimates for~$\bar u_k$ give the factor~$\frac12$ in the excess inequality. Increase the chart offset~$n_U$ if necessary so that~$z+\cu_{k+2}$ lies in one boundary chart whenever~$k\leq m-n_U$ and~$(z+\cu_{k+2})\cap\partial U_m\neq\emptyset$, and put
\begin{equation*}
	r\coloneqq m-n_U\,.
\end{equation*}
Increase~$M=M(U,d)$ and~$\hat L$ so that
\begin{equation*}
	h_m\geq n_U+2k_0
	\qquad\mbox{and}\qquad
	C(U,d)h_m\delta_{m-h_m}\leq\frac1{100}
\end{equation*}
for every admissible~$m$. If~$n>r-2k_0$, then~$m-n\leq n_U+2k_0$; the uniform volume-density estimate and the global Caccioppoli estimate~\eqref{e.Dir.new.Cacc.global} give~\eqref{e.Dir.new.C01.boundary} directly. We may therefore assume that~$n\leq r-2k_0$.

\smallskip

If~$(z+\cu_j)\cap\partial U_m=\emptyset$ for every~$j\in[n,r]\cap\N$, then~$z+\cu_r\subseteq U_m$. Applying~\eqref{e.Dir.new.C01} to~$x\mapsto u(z+x)$ with coefficient field~$x\mapsto\a(z+x)$ between scales~$n$ and~$r$, and then using~$m-r=n_U$ to compare the normalized oscillations on~$z+\cu_r$ and~$U_m$, proves~\eqref{e.Dir.new.C01.boundary}. Otherwise, let~$n_0$ be the least integer in~$[n,r]$ such that~$(z+\cu_{n_0})\cap\partial U_m\neq\emptyset$. For every~$k\in[n_0,r]\cap\N$, let~$\eta_k$ be the standard mollifier on scale~$3^k$, choose~$\zeta_k\in C^\infty_c(z+\cu_{k+2})$ with~$\indc_{z+\cu_{k+1}}\leq\zeta_k\leq\indc_{z+\cu_{k+2}}$ and~$\|\nabla\zeta_k\|_{L^\infty}\leq C3^{-k}$, and set
\begin{equation*}
	g_k\coloneqq \zeta_k(\eta_k\ast g)+(1-\zeta_k)g\,.
\end{equation*}
By the assumed smoothness of~$U$ and the choice of~$n_U$, the normalized single-chart configurations form a compact family, so a relative smoothing between~$\hat\cu_{k-1}$ and~$\hat\cu_k$ yields a smooth connected domain~$\hat V_k$ such that~$\hat\cu_{k-1}\subseteq\hat V_k\subseteq\hat\cu_k$. The rescaled domain~$3^{-k}(\hat V_k-z)$ has uniformly bounded~$C^{1,1}$ character,~$\partial\hat V_k$ agrees with~$\partial U_m$ near~$\hat\cu_{k-k_0}$, and~$\dist(\hat\cu_{k-k_0},\partial\hat V_k\setminus\partial U_m)\geq c(U,d,k_0)3^k$.

\smallskip

Let~$u_k$ solve the original equation in~$\hat\cu_{k+2}$ with boundary data~$g_k$ on~$\partial U_m\cap(z+\cu_{k+2})$ and boundary data~$u$ on~$\partial(z+\cu_{k+2})\cap U_m$, and let~$\bar u_k\in H^1(\hat V_k)$ solve
\begin{equation*}
	-\shom_k\Delta\bar u_k=f\quad\mbox{in }\hat V_k\,,
	\qquad
	\bar u_k=u_k\quad\mbox{on }\partial\hat V_k\,.
\end{equation*}
The weak maximum principle for bounded nonsymmetric divergence-form coefficients, applied to~$u-u_k$, and the definition of~$g_k$ give
\begin{equation*}
	3^{-k}\|u-u_k\|_{L^\infty(\hat\cu_{k+2})}
	\leq
	C\|\nabla g\|_{L^\infty(U_m)}\,.
\end{equation*}
Since~$k\leq r=m-n_U$, the cubes~$z+\cu_{k+1}$ and~$z+\cu_{k+2}$ remain within one boundary chart whenever they meet~$\partial U_m$. Applying the testing argument from Lemma~\ref{l.Dirichlet.Cacc.blackbox} on~$\hat\cu_{k+1}$ to~$u-u_k$, with data~$g-\eta_k\ast g$ on~$\partial U_m$, gives
\begin{equation*}
	\shom_m^{-\nf12}\nu^{\nf12}
	\|\nabla(u-u_k)\|_{\underline{L}^2(\hat\cu_k)}
	\leq
	C\|\nabla g\|_{L^\infty(U_m)}
	+Ck^{-1000}3^k\|\nabla^2g\|_{L^\infty(U_m)}\,.
\end{equation*}
Lemma~\ref{l.sharp.scale.inputs} supplies the hypotheses of Proposition~\ref{p.Dirichlet.L2.blackbox} on~$\hat V_k$. Its localized form therefore bounds~$3^{-k}\|u_k-\bar u_k\|_{\underline L^2(\hat V_k)}$ by~$C\shom_m^{-\nf12}\delta_k\nu^{\nf12}\|\nabla u_k\|_{\underline L^2(\hat V_k)}+Ck^{-100}3^k\|f\|_{L^\infty(U_m)}$; after increasing~$\hat L$, the term~$Ck^{-100}\|\nabla u_k\|_{\underline L^2(\hat V_k)}$ is absorbed into the term proportional to~$\delta_k\nu^{\nf12}\|\nabla u_k\|_{\underline L^2(\hat V_k)}$.
For a member~$Q$ of a fixed cover of~$\hat V_k$ that meets~$\partial U_m$, choose a smooth unit-mass density on a patch~$F\subseteq\partial Q\cap\partial U_m$ and a smooth test function vanishing on~$\partial Q$, whose normal derivative equals this density on~$F$ and vanishes elsewhere. The choices of~$n_U$ and~$M$ make the rescaled patches uniformly flat, so the density and the test function can be chosen with uniform rescaled bounds. Since~$\bar u_k=g_k$ on~$F$, Green's identity applied to~$\bar u_k-(\bar u_k)_Q$ gives
\begin{equation*}
	\bigl|(\bar u_k-g_k)_Q\bigr|\leq C\|\bar u_k-(\bar u_k)_Q\|_{\underline L^2(Q)}+C3^k\|\nabla g_k\|_{L^\infty(Q)}+C\shom_k^{-1}3^{2k}\|f\|_{L^\infty(Q)}\,.
\end{equation*}
Decomposing~$u_k-g_k$ through~$\bar u_k$ and using Poincar\'e's inequality for~$g_k$ therefore yields
\begin{equation*}
	\|u_k-g_k\|_{\underline L^2(Q)}\leq\|u_k-\bar u_k\|_{\underline L^2(Q)}+C\|\bar u_k-(\bar u_k)_Q\|_{\underline L^2(Q)}+C3^k\|\nabla g_k\|_{L^\infty(Q)}+C\shom_k^{-1}3^{2k}\|f\|_{L^\infty(Q)}\,.
\end{equation*}
The last display controls~$\|u_k-g_k\|_{\underline L^2(Q)}$, as required by the boundary Caccioppoli estimate on every~$Q$ meeting~$\partial U_m$. A member~$Q$ meeting only~$\partial\hat V_k\setminus\partial U_m$ lies in the interior of the domain of~$u_k$, since~$u_k$ is defined on~$\hat\cu_{k+2}$, so we use the centered Caccioppoli estimate on~$Q$. Summing the boundary and centered interior Caccioppoli estimates over the cover and inserting the resulting bound for~$\|\nabla u_k\|_{\underline L^2(\hat V_k)}$ into the estimate for~$u_k-\bar u_k$ leaves~$C\delta_k3^{-k}\|u_k-\bar u_k\|_{\underline L^2(\hat V_k)}$, which is absorbed after increasing~$\hat L$. The estimates for~$u-u_k$ and the fixed-scale mean comparisons then give
\begin{align*}
	3^{-k}\|u_k-\bar u_k\|_{\underline{L}^2(\hat V_k)}
	&\leq C\delta_k3^{-k}\|u-(u)_{\hat\cu_{k+1}}\|_{\underline L^2(\hat\cu_{k+1})} +C\shom_m^{-1}3^k\|f\|_{L^\infty(U_m)}
	\notag\\ &\qquad
	+C\|\nabla g\|_{L^\infty(U_m)}
	+Ck^{-1000}3^k\|\nabla^2g\|_{L^\infty(U_m)}
	\,.
\end{align*}
Near~$\hat\cu_{k-k_0}$, we have~$\partial\hat V_k=\partial U_m$ and~$\bar u_k=g_k$. Since~$\dist(\hat\cu_{k-k_0},\partial\hat V_k\setminus\partial U_m)\geq c(U,d,k_0)3^k$, the standard boundary and interior estimates for the Laplace equation give
\begin{align*}
	\inf_{\ell\ {\rm affine}}
	3^{-(k-k_0)}\|\bar u_k-\ell\|_{L^\infty(\hat\cu_{k-k_0})}
	&\leq
	C3^{-k_0}
	\inf_{\ell\ {\rm affine}}\bigl(3^{-k}\|\bar u_k-\ell\|_{\underline{L}^2(\hat\cu_{k-1})}+3^{k-m}|\nabla\ell|\bigr)
	\notag\\ &\qquad
	+C\shom_m^{-1}3^k\|f\|_{L^\infty(U_m)}
	+C\|\nabla g\|_{L^\infty(U_m)}\,,
\end{align*}
Here~$\partial U_m$ has curvature bounded by~$C(U)3^{-m}$, and~$g_k=\eta_k\ast g$ near~$\hat V_k$ gives~$\|\nabla^2g_k\|_{L^\infty(\hat V_k)}\leq C3^{-k}\|\nabla g\|_{L^\infty(U_m)}$. Define the boundary excess
\begin{equation*}
	E^z_j\coloneqq
	\inf_{\ell\ {\rm affine}}
	3^{-j}\|u-\ell\|_{\underline{L}^2(\hat\cu_j)}\,.
\end{equation*}
Choose a minimizing affine~$\ell^z_j$. Combining the estimates for~$u-u_k$,~$u_k-\bar u_k$, and the affine approximation of~$\bar u_k$, and then using Poincar\'e's inequality, gives, for~$k\in[n_0,r]\cap\N$,
\begin{align*}
	E^z_{k-k_0}
	&\leq
	\frac12 E^z_{k+1}
	+C\bigl(\delta_k+3^{k-m}\bigr)|\nabla\ell^z_{k+1}|
	+C\shom_m^{-1}3^k\|f\|_{L^\infty(U_m)}
	+C\|\nabla g\|_{L^\infty(U_m)}
	\notag\\ &\qquad
	+Ck^{-1000}3^k\|\nabla^2g\|_{L^\infty(U_m)}\,,
\end{align*}
after increasing~$\hat L$. The uniform volume-density and affine-Poincar\'e inequalities in the single boundary chart, together with~$r=m-n_U$, give
\begin{equation*}
	\max_{j\in[r-k_0,r+1]\cap\N}
	\bigl(E^z_j+|\nabla\ell^z_j|\bigr)
	\leq
	C3^{-m}
	\|u-(u)_{U_m}\|_{\underline{L}^2(U_m)}\,.
\end{equation*}
The slopes telescope in the form
\begin{equation*}
	|\nabla\ell^z_k|
	\leq
	C\bigl(|\nabla\ell^z_{r+1}|+E^z_{r+1}\bigr)
	+C\sum_{j=k}^{r}E^z_j\,.
\end{equation*}
Since~$n\geq m-h_m$, we have~$\delta_k\leq C\delta_{m-h_m}$ on the iteration range. Increasing~$n_U$ and then~$\hat L$ gives
\begin{equation*}
	C(U,d)\sum_{k=n_0}^{r}\bigl(\delta_k+3^{k-m}\bigr)
	\leq
	C(U,d)\bigl(h_m\delta_{m-h_m}+3^{-n_U}\bigr)
	\leq
	\frac1{100}\,.
\end{equation*}
Since~$k\simeq m$ on this range, we have~$\sum_{k=n_0}^r k^{-1000}3^k\leq Cm^{-1000}3^m$. The iteration over this finite interval of scales, the top-scale bound, and the telescoping inequality for~$\nabla\ell^z_k$ therefore give
\begin{align*}
	\max_{j\in[n_0,r+1]\cap\N}
	3^{-j}\|u-(u)_{\hat\cu_j}\|_{\underline{L}^2(\hat\cu_j)}
	&\leq
	C3^{-m}\|u-(u)_{U_m}\|_{\underline{L}^2(U_m)}
	+C\shom_m^{-1}3^m\|f\|_{L^\infty(U_m)}
	\notag\\
	&\qquad
	+C(m-n)\|\nabla g\|_{L^\infty(U_m)}
	+Cm^{-1000}3^m\|\nabla^2g\|_{L^\infty(U_m)}\,.
\end{align*}
If~$n\leq n_0-3$, applying~\eqref{e.Dir.new.C01} to~$x\mapsto u(z+x)$ with coefficient field~$x\mapsto\a(z+x)$ between scales~$n$ and~$n_0-2$ reduces~\eqref{e.Dir.new.C01.boundary} to the displayed bound for~$\max_{j\in[n_0,r+1]}3^{-j}\|u-(u)_{\hat\cu_j}\|_{\underline L^2(\hat\cu_j)}$; Lemma~\ref{l.Dirichlet.Cacc.blackbox} controls the at most two remaining scales. Apply its interior estimate to~$u$ on cubes contained in~$U_m$ and its boundary estimate to~$u_k$ on cover members~$Q$ meeting~$\partial U_m$, using the displayed bounds for~$\|u_k-g_k\|_{\underline L^2(Q)}$ and~$\|\nabla(u-u_k)\|_{\underline L^2(\hat\cu_k)}$. This gives the gradient term in~\eqref{e.Dir.new.C01.boundary}, and~$|\hat\cu_n|\geq c(U,d)|\cu_n|$ compares the normalized norms.
\end{proof}

\begin{proposition}[Localized large-scale Lipschitz estimate]
\label{p.Dirichlet.C01.boundary.local}
Let~$U\subseteq\cu_0$ be smooth and bounded, and let~$V\subseteq W\subseteq U$ be smooth bounded domains such that
\begin{equation*}
	\operatorname{dist}(\overline V\setminus\partial U,\partial W\setminus\partial U)>0\,.
\end{equation*}
Assume also that~$\partial W$ agrees with~$\partial U$ in a relative neighborhood of~$\overline V\cap\partial U$. There exists~$C(U,V,W,d)<\infty$ such that the following holds. After taking~$M\geq C(U,V,W,d)$ sufficiently large and enlarging~$\hat L$, let~$m,n\in\N$ satisfy
\begin{equation*}
	m\geq\hat L\,, \qquad n=m-h_m\geq\hat L \qand 3^m\geq\X_0\,.
\end{equation*}
If~$g\in W^{2,\infty}(3^mW)$,~$f\in L^\infty(3^mW)$, and~$u\in H^1(3^mW)$ solves
\begin{equation*}
	-\nabla\cdot\a\nabla u=f\quad\mbox{in }3^mW\,, \qquad u=g\quad\mbox{on }\partial U_m\cap\partial(3^mW)\,,
\end{equation*}
then, writing~$\hat\cu_n(z)\coloneqq(z+\cu_n)\cap U_m$, for every~$z\in3^n\Zd\cap3^mV$ we have
\begin{align}
	\lefteqn{
	\shom_m^{-\nf12}\nu^{\nf12}\|\nabla u\|_{\underline L^2(\hat\cu_n(z))}+3^{-n}\|u-(u)_{\hat\cu_n(z)}\|_{\underline L^2(\hat\cu_n(z))} } \qquad &
	\notag \\ &\leq C3^{-m}\|u-(u)_{3^mW}\|_{\underline L^2(3^mW)}
	+C\shom_m^{-1}3^m\|f\|_{L^\infty(3^mW)}
	\notag \\ & \qquad
	+C(m-n)\|\nabla g\|_{L^\infty(3^mW)}
	+Cm^{-1000}3^m\|\nabla^2g\|_{L^\infty(3^mW)}\,.
	\label{e.Dir.new.C01.boundary.local}
\end{align}
\end{proposition}

\begin{proof}
After increasing the chart offset~$n_U$ depending on~$(U,V,W,d)$, the geometric assumptions keep every comparison cube away from~$\partial(3^mW)\setminus\partial U_m$. The proof of Proposition~\ref{p.Dirichlet.C01.blackbox}, with the top-scale norms taken over~$3^mW$, then gives~\eqref{e.Dir.new.C01.boundary.local} without using the trace there.
\end{proof}

\subsection{Pointwise estimates}
\label{ss.Dirichlet.Linfty.blackbox}

To obtain pointwise estimates, we combine the large-scale bounds with local elliptic regularity. The large-scale estimate above controls oscillations down to the bottom of the logarithmic window; below this scale we use the standard inhomogeneous De Giorgi--Nash~$L^\infty$--$L^2$ estimate, with the Moser-iteration factor~$\Theta^{\nf d4}$ for ellipticity ratio~$\Theta$.

\begin{proposition}[Dirichlet homogenization in~$L^\infty$]
\label{p.Dirichlet.Linfty.blackbox}
Let~$U\subseteq\cu_0$ be smooth and bounded. There exists~$C(U,d)<\infty$ such that the following holds. Assume that
\begin{equation*}
	K\geq C\,,
	\qquad
	K-\lceil C\log K\rceil\geq \hat L
	\qquad\mbox{and}\qquad
	3^{K-\lceil C\log K\rceil}\geq\X_0\,,
\end{equation*}
if~$g\in W^{1,\infty}(U_K)$ and~$f\in L^\infty(U_K)$, and if~$u,\uhom$ solve~\eqref{e.Dir.new.problems}, then
\begin{multline}
	\label{e.Dir.new.Linfty.homog}
	\|u-\uhom\|_{L^\infty(U_K)}
	+ \|\nabla u-\nabla\uhom\|_{\underline{H}^{-1}(U_K)}
	+\bigl\|
	(\Id+\tilde\a_K)\nabla u-\nabla\uhom
	\bigr\|_{\underline{H}^{-1}(U_K)}
	\\
	\leq
	C\delta_K
	\bigl(
	\shom_K^{-1}3^{2K}\|f\|_{L^\infty(U_K)}
	+\log K\,3^K\|\nabla g\|_{L^\infty(U_K)}
	\bigr)\,.
\end{multline}
\end{proposition}

\begin{proof}
Set
\begin{equation*}
	Q_d\coloneqq2d\max\{201,d+103\}
	\qand
	N_d\coloneqq\lceil\frac{Q_d}{\log3}\rceil\,.
\end{equation*}
Put
\begin{equation*}
	h\coloneqq \lceil N_d\log K\rceil\,,
	\qquad
	n\coloneqq K-h\,,
	\qquad
	r_K\coloneqq3^{n-K}
	\qand
	E_K\coloneqq C(U,d)\nu^{-(d+2)}K^{d+2}\,.
\end{equation*}
Then~$r_K\leq K^{-Q_d}$. After increasing the standing lower bound on~$M$ and the deterministic threshold in~$\hat L$, the hypotheses in the statement give~$n\geq\hat L$,~$3^n\geq\X_0$, and~$h\leq h_K$, while
\begin{equation}
	\label{e.Dir.new.Linfty.scale.choice}
	E_Kr_K\leq K^{-300}\,,
	\qquad
	E_Kr_K^{\nf 1{2d}}\leq K^{-100}
	\qand
	r_K^{\nf 1{2d}}\leq K^{-200}\,.
\end{equation}
Indeed,~$Q_d/(2d)\geq201$,~$Q_d/(2d)\geq d+103$, and~$Q_d-(d+2)>300$. These are the only numerical scale-separation inequalities used below. We choose the constant~$C(U,d)$ in the statement large enough that~$N_d\leq C(U,d)$.
After enlarging~$\hat L$ once more, we also have
\begin{equation*}
	n+d+7\leq K-n_U\,.
\end{equation*}
This places every enlarged boundary cube used below inside the single-chart range.
We also write
\begin{equation*}
	\mathcal D_K
	\coloneqq
	\shom_K^{-1}3^{2K}\|f\|_{L^\infty(U_K)}
	\,+\,
	\log K\,3^K\|\nabla g\|_{L^\infty(U_K)}\,.
\end{equation*}
Let~$\eta=\eta_{3^n}$ be the standard mollifier at scale~$3^n$.

\smallskip

We first reduce to smooth boundary data on the separate scale
\begin{equation*}
	\ell_g\coloneqq3^K K^{-208}\,.
\end{equation*}
Extend~$g$ to a function in~$W^{1,\infty}(\Rd)$ with norm bounded by~$C(U)\|g\|_{W^{1,\infty}(U_K)}$ and set~$\tilde g\coloneqq \eta_{\ell_g}\ast g$. Then
\begin{equation}
	\label{e.Dir.new.Linfty.gsmooth}
	\|\nabla \tilde g\|_{L^\infty(U_K)}
	{}+
	\ell_g\|\nabla^2\tilde g\|_{L^\infty(U_K)}
	\leq
	C\|\nabla g\|_{L^\infty(U_K)}\,,
\end{equation}
and
\begin{equation*}
	\|\tilde g-g\|_{L^\infty(U_K)}
	\leq
	C\ell_g\|\nabla g\|_{L^\infty(U_K)}\,.
\end{equation*}
Let~$v,\vhom\in H^1(U_K)$ solve~\eqref{e.Dir.new.problems} with the boundary datum~$g$ replaced by~$\tilde g$. The maximum principle gives
\begin{equation*}
	\|u-v\|_{L^\infty(U_K)}
	+\|\uhom-\vhom\|_{L^\infty(U_K)}
	\leq
	C\ell_g\|\nabla g\|_{L^\infty(U_K)}\,.
\end{equation*}
To compare the gradients, choose a smooth cutoff~$\zeta_g$ which is equal to~$1$ at distance at least~$\ell_g$ from~$\partial U_K$, vanishes at distance at most~$\ell_g/2$ from~$\partial U_K$, and satisfies~$\|\nabla\zeta_g\|_{L^\infty}\leq C\ell_g^{-1}$. Testing the equation for~$u-v$ with~$u-v-(1-\zeta_g)(g-\tilde g)\in H^1_0(U_K)$, and using the coefficient bound~\eqref{e.Dir.new.k.bounds} after subtracting the constant skew matrix~$(\k)_{\cu_K}$, gives
\begin{align*}
	\nu\|\nabla(u-v)\|_{\underline{L}^2(U_K)}^2
	&\leq
	C\nu^{-1} K^{2+2\rho}
	\bigl(
	\|(1-\zeta_g)\nabla(g-\tilde g)\|_{\underline{L}^2(U_K)}^2
	+\|(g-\tilde g)\nabla\zeta_g\|_{\underline{L}^2(U_K)}^2
	\bigr)
	\\
	&\leq
	C\nu^{-1}K^{-206+2\rho}
	\|\nabla g\|_{L^\infty(U_K)}^2\,.
\end{align*}
In particular,
\begin{equation*}
	\|\nabla(u-v)\|_{\underline{L}^2(U_K)}
	\leq
	C\nu^{-1}K^{-103+\rho}\|\nabla g\|_{L^\infty(U_K)}\,.
\end{equation*}
The same argument for the constant-coefficient equation gives the analogous estimate for~$\uhom-\vhom$. By~\eqref{e.Dir.new.Linfty.scale.choice}, and using again~\eqref{e.Dir.new.k.bounds} for the flux term, these comparisons imply
\begin{multline}
	\label{e.Dir.new.Linfty.smooth.reduction}
	3^{-K}\bigl(
	\|u-v\|_{L^\infty(U_K)}
	+\|\uhom-\vhom\|_{L^\infty(U_K)}
	\bigr)
	+\|\nabla(u-v)\|_{\underline{L}^2(U_K)}
	+\|\nabla(\uhom-\vhom)\|_{\underline{L}^2(U_K)}
	\\
	+\|(\Id+\tilde\a_K)\nabla(u-v)\|_{\underline{L}^2(U_K)}
	\leq
	CK^{-100}\|\nabla g\|_{L^\infty(U_K)}\,.
\end{multline}
Put~$\phi\coloneqq u-v$ and~$\phi_{\rm hom}\coloneqq\uhom-\vhom$. The normalized Poincar\'e inequality gives
\begin{equation*}
	\|F\|_{\underline H^{-1}(U_K)}
	\leq
	C3^K\|F\|_{\underline L^2(U_K)}\,.
\end{equation*}
Using this inequality in~\eqref{e.Dir.new.Linfty.smooth.reduction}, and then increasing the deterministic threshold so that~$K^{-100}\leq\delta_K\log K$, gives
\begin{multline}
	\label{e.Dir.new.Linfty.smooth.weak}
	\|\phi\|_{L^\infty(U_K)}
	+\|\phi_{\rm hom}\|_{L^\infty(U_K)}
	+\|\nabla \phi-\nabla \phi_{\rm hom}\|_{\underline{H}^{-1}(U_K)}
	\\
	+\|(\Id+\tilde\a_K)\nabla \phi-\nabla \phi_{\rm hom}\|_{\underline{H}^{-1}(U_K)}
	\leq
	C\delta_K\mathcal D_K\,.
\end{multline}
Thus it is enough to prove the estimate for~$v,\vhom$ with boundary datum~$\tilde g$. We keep the boundary data explicit below. The quantity~$\mathcal D_K$ remains the one defined with the original boundary datum~$g$, and~\eqref{e.Dir.new.Linfty.gsmooth} ensures that it also controls all terms involving~$\tilde g$.

\smallskip

Choose a new cutoff, again denoted~$\zeta\in C^\infty_c(U_K)$, as in the proof of Proposition~\ref{p.Dirichlet.L2.blackbox}: it vanishes within distance~$3^{n+d}$ of~$\partial U_K$, is equal to~$1$ at distance at least~$3^{n+d+3}$ from~$\partial U_K$, and~$\|\nabla\zeta\|_{L^\infty}\leq C3^{-n}$. Define
\begin{equation*}
	w\coloneqq \zeta(\eta\ast v)+(1-\zeta)\tilde g\,.
\end{equation*}
We also set
\begin{equation}
	\label{e.Dir.new.Linfty.HK.def}
	\mathcal H_K
	\coloneqq
	\|v-(v)_{U_K}\|_{\underline{L}^2(U_K)}
	+\|v-\tilde g\|_{\underline{L}^2(U_K)}
	+h\,3^K\|\nabla g\|_{L^\infty(U_K)}
	+\shom_K^{-1}3^{2K}\|f\|_{L^\infty(U_K)}\,.
\end{equation}
The boundary version~\eqref{e.Dir.new.C01.boundary} of Proposition~\ref{p.Dirichlet.C01.blackbox}, together with~\eqref{e.Dir.new.Linfty.gsmooth}, yields the following local gradient bound: for every boundary or interior grid cube~$z+\cu_n$ meeting~$U_K$,
\begin{equation}
	\label{e.Dir.new.Linfty.local.grad}
	\nu^{\nf12}
	\|\nabla v\|_{\underline{L}^2((z+\cu_{n+d+5})\cap U_K)}
	\leq
	C\shom_K^{\nf12}3^{-K}\mathcal H_K\,.
\end{equation}
If the center of a cube in the boundary cover lies outside~$U_K$, replace it by a bounded number of neighboring grid cubes with centers in~$U_K$; the larger cube~$z+\cu_{n+d+5}$ in~\eqref{e.Dir.new.Linfty.local.grad} contains all of these cubes. Moreover,~\eqref{e.Dir.new.Linfty.gsmooth} gives
\begin{equation*}
	K^{-1000}3^K\|\nabla^2\tilde g\|_{L^\infty(U_K)}
	\leq
	CK^{-792}\|\nabla g\|_{L^\infty(U_K)}\,.
\end{equation*}

\smallskip

We extend~$v$ and~$\vhom$ outside~$U_K$ by setting them equal to~$\tilde g$. Set~$A_K\coloneqq\a-(\k)_{\cu_K}$. By~\eqref{e.Dir.new.k.bounds},
\begin{equation*}
	\lambda_0=\nu\,,
	\qquad
	\mu_0\leq\nu+\nu^{-1}K^4
	\qand
	\Theta_K\leq2\nu^{-2}K^4\,.
\end{equation*}
Thus~$\Theta_K^{\nf d4}\leq C(d)\nu^{-d/2}K^d$. For interior points~$y$ with~$y+\cu_n\subseteq U_K$, the standard inhomogeneous De Giorgi--Nash~$L^\infty$--$L^2$ estimate, with its Moser-iteration factor~$\Theta_K^{\nf d4}$, and Poincar\'e's inequality give
\begin{align*}
	|v(y)-\eta\ast v(y)|
	&\leq
	C\Theta_K^{\nf d4}
	\bigl( 3^n
	\|\nabla v\|_{\underline{L}^2(y+\cu_n)}
	+ \nu^{-1}3^{2n}\|f\|_{L^\infty(U_K)} \bigr)
	\leq
	C E_Kr_K\mathcal H_K\,.
\end{align*}
Here and below, the second inequality follows from~\eqref{e.Dir.new.Linfty.local.grad}, the deterministic bounds~\eqref{e.sm.sharp.bounds}, and the definition of~$E_K$. If~$y$ is within distance~$C3^n$ of~$\partial U_K$, apply the boundary version of the same estimate to~$v-\tilde g$, with divergence datum~$F=A_K\nabla\tilde g$. Using~$\|A_K\|_{L^\infty(U_K)}\leq2K^2$, we obtain
\begin{align*}
	|v(y)-\tilde g(y)|
	&\leq
	C\Theta_K^{\nf d4}
	\bigl(
	3^n
	\|\nabla v\|_{\underline{L}^2((y+\cu_{n+d+6})\cap U_K)}
	\bigr)
	\\
	&\qquad{}
	+C\Theta_K^{\nf d4}
	\bigl(
	\nu^{-1}3^n\|A_K\nabla\tilde g\|_{L^\infty(U_K)}
	+
	\nu^{-1}3^{2n}\|f\|_{L^\infty(U_K)}
	\bigr)
	\\
	&\leq
	C E_Kr_K\mathcal H_K\,.
\end{align*}
The same bound for~$v-\eta\ast v$ near the boundary follows by extending~$v-\tilde g$ by zero and using the mollifier estimate. Consequently,
\begin{equation}
	\label{e.Dir.new.Linfty.v.minus.w}
	\|v-\eta\ast v\|_{L^\infty(U_K)}
	+\|v-w\|_{L^\infty(U_K)}
	\leq
	C E_Kr_K\mathcal H_K
	\leq
	CK^{-200}\mathcal H_K\,.
\end{equation}

\smallskip

It remains to estimate~$w-\vhom$. By Morrey's inequality,
\begin{equation}
	\label{e.Dir.new.Linfty.Morrey}
	\|w-\vhom\|_{L^\infty(U_K)}
	\leq
	C3^K
	\|\nabla w-\nabla\vhom\|_{\underline{L}^{2d}(U_K)}\,.
\end{equation}
The identity~\eqref{e.Dir.new.convolution.identity}, with~$u$ replaced by~$v$, and the Calder\'on--Zygmund estimate in the smooth domain~$U_K$~\cite[Exercise~7.10 and Remark~7.11]{AKMBook}, used with exponent~$2d$, give
\begin{align}
	\shom_K\|\nabla w-\nabla\vhom\|_{\underline{L}^{2d}(U_K)}
	&\leq
	C\shom_K\|\zeta\,\eta\ast(\tilde\a_K\nabla v)\|_{\underline{L}^{2d}(U_K)}
	\notag\\
	&\qquad
	+C\shom_K\|\nabla\zeta(\eta\ast v-\tilde g)\|_{\underline{L}^{2d}(U_K)}
	\notag\\
	&\qquad
	+C\shom_K\|(1-\zeta)\nabla \tilde g\|_{\underline{L}^{2d}(U_K)}
	\notag\\
	&\qquad
	+C\|\zeta\,\eta\ast f-f\|_{\underline{W}^{-1,2d}(U_K)}
	\notag\\
	&\qquad
	+C\|\nabla\zeta\cdot\eta\ast((\a-(\k)_{\cu_K})\nabla v)\|_{\underline{W}^{-1,2d}(U_K)}\,.
	\label{e.Dir.new.Linfty.CZ}
\end{align}
We estimate these terms one by one. The local mollifier estimate and Lemma~\ref{l.Dirichlet.rhs.blackbox} imply, for every~$z\in3^n\Zd$ with~$z+\cu_{n+1}\subseteq U_K$,
\begin{equation*}
	\|\eta\ast(\tilde\a_K\nabla v)\|_{L^\infty(z+\cu_n)}
	\leq
	C\delta_K3^{-K}\mathcal H_K
	+CK^{-200}3^K\|f\|_{L^\infty(U_K)}\,.
\end{equation*}
Summing over the grid cubes gives
\begin{equation}
	\label{e.Dir.new.Linfty.term.one}
	\|\zeta\,\eta\ast(\tilde\a_K\nabla v)\|_{\underline{L}^{2d}(U_K)}
	\leq
	C\delta_K3^{-K}\mathcal H_K
	+CK^{-200}3^K\|f\|_{L^\infty(U_K)}\,.
\end{equation}
The boundary terms are small. Indeed, the support of~$\nabla\zeta$ and~$1-\zeta$ has relative measure at most~$Cr_K$. Therefore~\eqref{e.Dir.new.Linfty.v.minus.w},~\eqref{e.Dir.new.Linfty.gsmooth}, and~\eqref{e.Dir.new.Linfty.scale.choice} give
\begin{align}
	\label{e.Dir.new.Linfty.boundary.terms}
	\|\nabla\zeta(\eta\ast v-\tilde g)\|_{\underline{L}^{2d}(U_K)}
	&\leq
	C3^{-n}r_K^{\nf 1{2d}}E_Kr_K\mathcal H_K
	=C3^{-K}E_Kr_K^{\nf 1{2d}}\mathcal H_K
	\leq CK^{-100}3^{-K}\mathcal H_K\,,
	\notag\\
	\|(1-\zeta)\nabla \tilde g\|_{\underline{L}^{2d}(U_K)}
	&\leq
	Cr_K^{\nf 1{2d}}\|\nabla\tilde g\|_{L^\infty(U_K)}
	\leq CK^{-100}3^{-K}\mathcal H_K\,.
\end{align}
Similarly, by duality against~$W^{1,(2d)'}_0(U_K)$, the standard mollification estimate and the boundary-layer Poincar\'e inequality give
\begin{equation}
	\label{e.Dir.new.Linfty.f.mollification}
	\|\zeta\,\eta\ast f-f\|_{\underline{W}^{-1,2d}(U_K)}
	\leq
	CK^{-200}3^K\|f\|_{L^\infty(U_K)}\,.
\end{equation}
It remains to estimate the boundary flux term directly. Write
\begin{equation*}
	F\coloneqq(\a-(\k)_{\cu_K})\nabla v\,,
	\qquad
	\mathcal A_n\coloneqq\{\nabla\zeta\neq0\}\,,
	\qquad
	q\coloneqq(2d)'\,.
\end{equation*}
If~$\psi\in W^{1,q}_0(U_K)$ and~$\|\psi\|_{\underline{W}^{1,q}(U_K)}\leq1$, then H\"older's inequality and the boundary-layer Poincar\'e inequality give
\begin{equation*}
	\biggl|
	\fint_{U_K}\psi\,\nabla\zeta\cdot\eta\ast F
	\biggr|
	\leq
	C3^{-n}
	|U_K|^{-\frac1q}\|\psi\|_{L^q(\mathcal A_n)}
	|U_K|^{-\frac1{2d}}\|\eta\ast F\|_{L^{2d}(\mathcal A_n)}
	\leq
	C|U_K|^{-\frac1{2d}}\|\eta\ast F\|_{L^{2d}(\mathcal A_n)}\,.
\end{equation*}
Since~$|\mathcal A_n|\leq Cr_K|U_K|$, we have
\begin{equation*}
	|U_K|^{-\frac1{2d}}\|\eta\ast F\|_{L^{2d}(\mathcal A_n)}
	\leq
	Cr_K^{\nf 1{2d}}
	\sup_{z\in3^n\Zd\,:\, z+\cu_{n+1}\subseteq U_K}
	\|\eta\ast F\|_{L^\infty(z+\cu_n)}\,.
\end{equation*}
The local full-flux estimate~\eqref{e.Dir.new.local.fullflux.mollified}, applied to~$v$, together with~\eqref{e.Dir.new.Linfty.local.grad} and~\eqref{e.Dir.new.Linfty.scale.choice}, therefore gives
\begin{equation}
	\label{e.Dir.new.Linfty.boundary.flux}
	\shom_K^{-1}
	\|\nabla\zeta\cdot\eta\ast((\a-(\k)_{\cu_K})\nabla v)\|_{\underline{W}^{-1,2d}(U_K)}
	\leq
	CK^{-200}3^{-K}\mathcal H_K
	+CK^{-200}3^K\|f\|_{L^\infty(U_K)}\,.
\end{equation}
Dividing~\eqref{e.Dir.new.Linfty.CZ} by~$\shom_K$, inserting~\eqref{e.Dir.new.Linfty.term.one}--\eqref{e.Dir.new.Linfty.boundary.flux}, and using~\eqref{e.Dir.new.Linfty.Morrey}, yields
\begin{equation}
	\label{e.Dir.new.Linfty.w.vhom}
	\|w-\vhom\|_{L^\infty(U_K)}
	+3^K\|\nabla w-\nabla\vhom\|_{\underline{L}^{2d}(U_K)}
	\leq
	C\delta_K\mathcal H_K
	+CK^{-100}\mathcal D_K\,.
\end{equation}

\smallskip

We next estimate~$\mathcal H_K$. By Poincar\'e's inequality, the triangle inequality, and the constant-coefficient energy estimate for~$\vhom$,
\begin{align}
	\label{e.Dir.new.Linfty.HK.bound}
	\mathcal H_K
	&\leq
	2\|v-\vhom\|_{L^\infty(U_K)}
	+Ch\,3^K\|\nabla g\|_{L^\infty(U_K)}
	+C\shom_K^{-1}3^{2K}\|f\|_{L^\infty(U_K)}\,.
\end{align}
Combining~\eqref{e.Dir.new.Linfty.v.minus.w},~\eqref{e.Dir.new.Linfty.w.vhom}, and~\eqref{e.Dir.new.Linfty.HK.bound}, and then increasing~$\hat L$ so that~$C\delta_K\leq\frac12$, gives
\begin{equation}
	\label{e.Dir.new.Linfty.v.vhom}
	\|v-\vhom\|_{L^\infty(U_K)}
	+3^K\|\nabla w-\nabla\vhom\|_{\underline{L}^{2d}(U_K)}
	\leq
	C\delta_K\mathcal D_K\,.
\end{equation}
This proves the asserted~$L^\infty$ estimate, up to the harmless smoothing errors in~\eqref{e.Dir.new.Linfty.smooth.weak}.

\smallskip

It remains to prove the two weak-norm estimates. First we compare~$\nabla v$ and~$\nabla w$. Let~$\psi\in H^1_0(U_K;\Rd)$ satisfy~$\|\psi\|_{\underline{H}^1(U_K)}\leq1$. Then
\begin{align*}
	\fint_{U_K}(\nabla w-\nabla v)\cdot\psi
	&=
	\fint_{U_K}(\nabla \tilde g-\nabla v)\cdot(1-\zeta)\psi
	+\fint_{U_K}\nabla\zeta\cdot\psi\,(\eta\ast v-\tilde g)
	\notag\\
	&\qquad
	+\fint_{U_K}\zeta(\eta\ast\nabla v-\nabla v)\cdot\psi\,.
\end{align*}
The middle term is bounded by~\eqref{e.Dir.new.Linfty.boundary.terms}. The boundary-data Hessian term in the global Caccioppoli estimate satisfies
\begin{equation*}
	K^{-2000}3^{2K}\|\nabla^2\tilde g\|_{L^\infty(U_K)}^2
	\leq
	CK^{-1584}\|\nabla g\|_{L^\infty(U_K)}^2\,.
\end{equation*}
For the first term, the boundary-layer Poincar\'e inequality,~\eqref{e.Dir.new.Cacc.global}, and~\eqref{e.Dir.new.Linfty.scale.choice} therefore give
\begin{equation*}
	\biggl|
	\fint_{U_K}(\nabla \tilde g-\nabla v)\cdot(1-\zeta)\psi
	\biggr|
	\leq
	CK^{-100}\mathcal D_K\,.
\end{equation*}
For the last term, using the cancellation of the mollifier,
\begin{align*}
	\biggl|
	\fint_{U_K}\zeta(\eta\ast\nabla v-\nabla v)\cdot\psi
	\biggr|
	&=
	\biggl|
	\fint_{U_K}\nabla v\cdot\bigl(\eta\ast(\zeta\psi)-\zeta\psi\bigr)
	\biggr|
	\\
	&\leq
	C3^n\|\nabla v\|_{\underline{L}^2(U_K)}
	\|\nabla(\zeta\psi)\|_{\underline{L}^2(U_K)}
	\leq
	CK^{-100}\mathcal D_K\,.
\end{align*}
Here the term with~$\nabla\zeta$ is supported in the boundary layer and is controlled as above, while the term with~$\nabla\psi$ is small by~\eqref{e.Dir.new.Linfty.scale.choice}. Taking the supremum over~$\psi$ gives
\begin{equation}
	\label{e.Dir.new.Linfty.grad.vw}
	\|\nabla v-\nabla w\|_{\underline{H}^{-1}(U_K)}
	\leq
	CK^{-100}\mathcal D_K\,.
\end{equation}
Together with~\eqref{e.Dir.new.Linfty.v.vhom} and~\eqref{e.Dir.new.Linfty.smooth.weak}, this proves the weak gradient estimate in~\eqref{e.Dir.new.Linfty.homog}.

\smallskip

For the flux estimate, observe that~$(\Id+\tilde\a_K)\nabla v=\shom_K^{-1}(\a-(\k)_{\cu_K})\nabla v$. Testing against the same~$\psi$ as above, we decompose
\begin{align*}
	\fint_{U_K}
	\bigl((\Id+\tilde\a_K)\nabla v-\nabla w\bigr)\cdot\psi
	&=
	\shom_K^{-1}
	\fint_{U_K}\zeta\psi\cdot
	\bigl(\a-(\k)_{\cu_K}-\shom_K\Id\bigr)\nabla v
	\notag\\
	&\qquad
	+\fint_{U_K}\zeta(\nabla v-\eta\ast\nabla v)\cdot\psi
	\notag\\
	&\qquad
	+\fint_{U_K}(1-\zeta)\bigl((\Id+\tilde\a_K)\nabla v-\nabla \tilde g\bigr)\cdot\psi
	\notag\\
	&\qquad
	-\fint_{U_K}\nabla\zeta\cdot\psi\,(\eta\ast v-\tilde g)\,.
\end{align*}
The last three terms are bounded by the estimates used in~\eqref{e.Dir.new.Linfty.grad.vw}, with the pointwise coefficient bound~\eqref{e.Dir.new.k.bounds} inserted in the boundary-layer term. For the first term, split the support of~$\zeta$ into cubes~$z+\cu_n\subseteq U_K$ with~$z\in3^n\Zd$. On each such cube, Lemma~\ref{l.Dirichlet.rhs.blackbox} applied to~$v$ gives
\begin{align*}
	\biggl|
	\fint_{z+\cu_n}
	\zeta\psi\cdot(\tilde\a_K\nabla v)
	\biggr|
	&\leq
	C\shom_K^{-\nf12}\delta_K
	\nu^{\nf12}\|\nabla v\|_{\underline{L}^2(z+\cu_n)}
	3^{\frac n4}\|\zeta\psi\|_{\underline{H}^{\nf 14}(z+\cu_n)}
	\notag\\
	&\qquad
	+CK^{-200}3^K\|f\|_{L^\infty(U_K)}
	3^{\frac n4}\|\zeta\psi\|_{\underline{H}^{\nf 14}(z+\cu_n)}\,.
\end{align*}
The normalized interpolation inequality
\begin{equation*}
	3^{\frac n4}\|\Phi\|_{\underline H^{\nf14}(z+\cu_n)}
	\leq
	C\bigl(
	\|\Phi\|_{\underline L^2(z+\cu_n)}
	+3^n\|\nabla\Phi\|_{\underline L^2(z+\cu_n)}
	\bigr)\,,
\end{equation*}
shows that, after summing the cubes, finite overlap,~\eqref{e.Dir.new.Linfty.local.grad}, and~$\|\psi\|_{\underline{H}^1(U_K)}\leq1$ give
\begin{equation}
	\label{e.Dir.new.Linfty.flux.vw}
	\|(\Id+\tilde\a_K)\nabla v-\nabla w\|_{\underline{H}^{-1}(U_K)}
	\leq
	C\delta_K\mathcal H_K
	+CK^{-100}\mathcal D_K
	\leq
	C\delta_K\mathcal D_K\,.
\end{equation}
Combining~\eqref{e.Dir.new.Linfty.flux.vw} with~\eqref{e.Dir.new.Linfty.v.vhom} and the perturbative estimate~\eqref{e.Dir.new.Linfty.smooth.weak} proves the flux estimate in~\eqref{e.Dir.new.Linfty.homog}. The proof is complete.
\end{proof}

\subsection{The large-scale H\"older estimate}
\label{ss.Dirichlet.large.scale.Holder}

We next record the pointwise interior oscillation estimate which follows from the Dirichlet estimates above. This is the form in which the large-scale H\"older theorem is used later.

\begin{lemma}[Interior pointwise homogenization]
\label{l.Dirichlet.interior.Linfty.approx}
There exist constants~$C(d)<\infty$ and~$N(d)<\infty$ such that the following holds. Assume that
\begin{equation*}
	m-\lceil N\log m\rceil\geq\hat L
	\qquad\mbox{and}\qquad
	3^{m-\lceil N\log m\rceil}\geq\X_0\,.
\end{equation*}
Let~$f\in L^\infty(\cu_m)$ and let~$u$ solve
\begin{equation*}
	-\nabla\cdot\a\nabla u=f
	\qquad\mbox{in }\cu_m\,.
\end{equation*}
Let~$\eta$ be the standard mollifier on scale~$3^{m-\lceil N\log m\rceil}$, and choose~$V$ so that~$3^{-m}V$ is a fixed smooth domain depending only on~$d$ and
\begin{equation*}
	\cu_{m-1}\Subset V\subseteq \frac12\cu_m\,.
\end{equation*}
Let~$\uhom\in H^1(V)$ solve
\begin{equation*}
	-\shom_m\Delta\uhom=f
	\quad\mbox{in }V\,,
	\qquad
	\uhom=\eta\ast u
	\quad\mbox{on }\partial V\,.
\end{equation*}
Then
\begin{equation}
	\label{e.Dir.new.interior.Linfty.approx}
	\|u-\uhom\|_{L^\infty(\cu_{m-1})}
	\leq
	C\delta_m
	\bigl(
	\|u-(u)_{\cu_m}\|_{\underline{L}^2(\cu_m)}
	+
	\shom_m^{-1}3^{2m}\|f\|_{L^\infty(\cu_m)}
	\bigr)\,.
\end{equation}
\end{lemma}

\begin{proof}
Set
\begin{equation*}
	Q_d\coloneqq2d\max\{201,d+103\}
	\qand
	N\coloneqq\Bigl\lceil\frac{Q_d}{\log3}\Bigr\rceil\,.
\end{equation*}
Put
\begin{equation*}
	\ell\coloneqq m-\lceil N\log m\rceil\,,
	\qquad
	r_m\coloneqq3^{\ell-m}
	\qand
	E_m\coloneqq C(d)\nu^{-(d+2)}m^{d+2}\,.
\end{equation*}
Then~$r_m\leq m^{-Q_d}$ and, after increasing the deterministic threshold,~$E_mr_m\leq m^{-300}$. We enlarge the dimensional lower bound on~$M$ so that~$N\leq M$, and hence~$\ell\geq m-h_m$. Let~$\bar u\in H^1(V)$ solve
\begin{equation*}
	-\shom_m\Delta\bar u=\eta\ast f
	\quad\mbox{in }V\,,
	\qquad
	\bar u=\eta\ast u
	\quad\mbox{on }\partial V\,.
\end{equation*}
The standard mollification estimate, the Calder\'on--Zygmund estimate in~$V$~\cite[Exercise~7.10 and Remark~7.11]{AKMBook}, and Morrey's inequality give
\begin{equation}
	\label{e.Dir.new.interior.Linfty.f.moll}
	\|\bar u-\uhom\|_{L^\infty(V)}
	\leq
	Cm^{-100}
	\shom_m^{-1}3^{2m}\|f\|_{L^\infty(\cu_m)}\,.
\end{equation}
Next,
\begin{equation*}
	-\shom_m\Delta(\eta\ast u)
	=
	\nabla\cdot\eta\ast\bigl((\a-(\k)_{\cu_m}-\shom_m\Id)\nabla u\bigr)
	\,+\,\eta\ast f\,,
\end{equation*}
in~$V$. The Calder\'on--Zygmund estimate and Morrey's inequality, Lemma~\ref{l.Dirichlet.rhs.blackbox} on the grid cubes of scale~$\ell$, and Proposition~\ref{p.Dirichlet.C01.blackbox} applied between the scales~$\ell$ and~$m$, imply
\begin{equation}
	\label{e.Dir.new.interior.Linfty.ueta.bar}
	\|\eta\ast u-\bar u\|_{L^\infty(V)}
	\leq
	C\delta_m
	\bigl(
	\|u-(u)_{\cu_m}\|_{\underline{L}^2(\cu_m)}
	+
	\shom_m^{-1}3^{2m}\|f\|_{L^\infty(\cu_m)}
	\bigr)
	\,.
\end{equation}
Here the term produced by the right-hand side in Lemma~\ref{l.Dirichlet.rhs.blackbox} is smaller than the displayed right side after the choice of~$N$. Finally, for every~$y\in\cu_{m-1}$, the standard inhomogeneous De Giorgi--Nash~$L^\infty$--$L^2$ estimate, applied in~$y+\cu_\ell$ after subtracting the constant skew matrix~$(\k)_{\cu_m}$ from the coefficients, gives
\begin{align}
	\label{e.Dir.new.interior.Linfty.u.ueta}
	|u(y)-\eta\ast u(y)|
	&\leq
	C E_m
	\bigl(
	\|u-(u)_{y+\cu_\ell}\|_{\underline{L}^2(y+\cu_\ell)}
	+
	\shom_m^{-1}3^{2\ell}\|f\|_{L^\infty(\cu_m)}
	\bigr)
	\notag\\
	&\leq
	C E_mr_m
	\bigl(
	\|u-(u)_{\cu_m}\|_{\underline{L}^2(\cu_m)}
	+
	\shom_m^{-1}3^{2m}\|f\|_{L^\infty(\cu_m)}
	\bigr)\,.
\end{align}
The last line uses the localized form of Proposition~\ref{p.Dirichlet.C01.blackbox} and the logarithmic separation~$m-\ell=\lceil N\log m\rceil$; after increasing~$\hat L$, this separation is within the range~$m-\ell\leq c\delta_m^{-1}$. Combining~\eqref{e.Dir.new.interior.Linfty.f.moll}--\eqref{e.Dir.new.interior.Linfty.u.ueta}, using~$E_mr_m\leq m^{-300}$, and increasing~$\hat L$ so that~$m^{-100}\leq\delta_m$, proves~\eqref{e.Dir.new.interior.Linfty.approx}.
\end{proof}

\smallskip

The previous lemma compares the solution with a homogenized problem after mollifying at the bottom of the logarithmic window. The next lemma uses this comparison to prove an oscillation bound between two macroscopic scales, which is the form needed for the large-scale H\"older theorem.

\begin{lemma}[Interior pointwise oscillation]
\label{l.Dirichlet.interior.pointwise}
There exist constants~$C(d)<\infty$,~$c(d)>0$, and~$N(d)<\infty$ such that the following holds. If
\begin{equation*}
	n<m\,,
	\qquad
	m\geq \hat L\,,
	\qquad
	m-n\leq c\delta_m^{-1}\,,
	\qquad
	n-\lceil N\log n\rceil\geq \hat L\,,
	\qquad
	3^{n-\lceil N\log n\rceil}\geq \X_0\,,
\end{equation*}
if~$f\in L^\infty(\cu_m)$, and if~$u$ solves
\begin{equation*}
	-\nabla\cdot\a\nabla u=f
	\qquad\mbox{in }\cu_m\,,
\end{equation*}
then
\begin{equation}
	\label{e.Dir.new.interior.pointwise}
	\|u-(u)_{\cu_n}\|_{L^\infty(\cu_n)}
	\leq
	C3^{-(m-n)}
	\bigl(
	\|u-(u)_{\cu_m}\|_{\underline{L}^2(\cu_m)}
	+
	\shom_m^{-1}3^{2m}\|f\|_{L^\infty(\cu_m)}
	\bigr)\,.
\end{equation}
\end{lemma}

\begin{proof}
We use Lemma~\ref{l.Dirichlet.interior.Linfty.approx} only at logarithmically separated mollification scales. Fix a dimensional~$k_0\in\N$ large enough for the constant-coefficient interior estimate below, and set
\begin{equation*}
	k\coloneqq m\wedge(n+k_0+3)\,.
\end{equation*}
After enlarging~$\hat L$, the function~$j\mapsto j-\lceil N\log j\rceil$ is increasing on~$[\hat L,\infty)$, and hence
\begin{equation*}
	k-\lceil N\log k\rceil\geq\hat L
	\qquad\mbox{and}\qquad
	3^{k-\lceil N\log k\rceil}\geq\X_0\,.
\end{equation*}
Choose~$V_k$ with the fixed rescaled geometry in Lemma~\ref{l.Dirichlet.interior.Linfty.approx}, let~$\eta_k$ be the standard mollifier on scale~$3^{k-\lceil N\log k\rceil}$, and let~$\uhom^{(k)}\in H^1(V_k)$ solve
\begin{equation*}
	-\shom_k\Delta\uhom^{(k)}=f
	\quad\mbox{in }V_k\,,
	\qquad
	\uhom^{(k)}=\eta_k\ast u
	\quad\mbox{on }\partial V_k\,.
\end{equation*}
By Lemma~\ref{l.Dirichlet.interior.Linfty.approx}, Proposition~\ref{p.Dirichlet.C01.blackbox} applied between the scales~$k$ and~$m$, and the deterministic comparison of~$\shom_k$ and~$\shom_m$ on this range,
\begin{equation}
	\label{e.Dir.new.pointwise.uk.error}
	\|u-\uhom^{(k)}\|_{L^\infty(\cu_{k-1})}
	\leq
	C3^{-(m-k)}
	\bigl(
	\|u-(u)_{\cu_m}\|_{\underline{L}^2(\cu_m)}
	+
	\shom_m^{-1}3^{2m}\|f\|_{L^\infty(\cu_m)}
	\bigr)\,.
\end{equation}
The estimates~\eqref{e.Dir.new.interior.Linfty.f.moll}--\eqref{e.Dir.new.interior.Linfty.ueta.bar} and the~$L^2$ bound for mollification also control~$\|\uhom^{(k)}-(u)_{\cu_k}\|_{\underline L^2(V_k)}$ by the right side of~\eqref{e.Dir.new.pointwise.uk.error}. Since~$\cu_{k-1}\Subset V_k$ with a fixed rescaled separation, the constant-coefficient interior estimate for the Poisson equation gives
\begin{equation}
	\label{e.Dir.new.pointwise.uhom.osc}
	\|\uhom^{(k)}-(\uhom^{(k)})_{\cu_n}\|_{L^\infty(\cu_n)} \leq C3^{-(k-n)} \|\uhom^{(k)}-(\uhom^{(k)})_{V_k}\|_{\underline{L}^2(V_k)} + C\shom_k^{-1}3^{k+n}\|f\|_{L^\infty(\cu_m)}\,.
\end{equation}
The preceding bound and the fact that~$k-n$ is bounded by a dimensional constant show that the right side of~\eqref{e.Dir.new.pointwise.uhom.osc} is bounded by
\begin{equation*}
	C3^{-(m-n)}
	\bigl(
	\|u-(u)_{\cu_m}\|_{\underline{L}^2(\cu_m)}
	+
	\shom_m^{-1}3^{2m}\|f\|_{L^\infty(\cu_m)}
	\bigr)\,.
\end{equation*}
Combining this with~\eqref{e.Dir.new.pointwise.uk.error} and replacing the center constant by~$(u)_{\cu_n}$ proves~\eqref{e.Dir.new.interior.pointwise}.
\end{proof}

\begin{proof}[Proof of Theorem~\ref{t.large.scale.Holder}]
The oscillation estimate~\eqref{e.large.scale.Holder} implies the Liouville statement~\eqref{e.Liouville.Calpha} by sending the outer radius to infinity along the sequence in the definition of the growth condition. We therefore prove~\eqref{e.large.scale.Holder}.

\smallskip

Fix~$\gamma,\sigma\in(0,1)$ and choose~$\gamma'\in(\gamma,1)$ and~$\rho_H\in(\sigma,1)$. Apply Lemma~\ref{l.sharp.scale.inputs} with~$\varepsilon=1$, this value of~$\rho_H$, and with~$M$ larger than the dimensional lower bound in Lemma~\ref{l.Dirichlet.interior.pointwise}. Let~$N_0(d)$ be large enough for the logarithmic lower bound in Lemma~\ref{l.Dirichlet.interior.pointwise}, and let
\begin{equation*}
	m_\ast\coloneqq
	\min\bigl\{
	j\in\N\,:\,
	j-\lceil N_0\log j\rceil\geq\hat L
	\quad\mbox{and}\quad
	3^{j-\lceil N_0\log j\rceil}\geq\X_0
	\bigr\}
	\,,
\end{equation*}
after increasing~$\hat L$ by the deterministic thresholds below, and put~$\X\coloneqq C(d)3^{m_\ast}$. Since~$m_\ast\leq C(\hat L+\log_3\X_0+\log(2+\log_3\X_0))$ and~$\rho_H>\sigma$,~\eqref{e.Dir.new.minscale.tail} gives~\eqref{e.large.scale.Holder.X} after enlarging the constant.

\smallskip

We first prove the corresponding cube estimate. Let~$u$ solve~$-\nabla\cdot\a\nabla u=f$ in~$\cu_m$, with~$m\geq m_\ast$, and set
\begin{equation*}
	B_m(u,f)\coloneqq
	\|u-(u)_{\cu_m}\|_{\underline{L}^2(\cu_m)}
	+
	\shom_m^{-1}3^{2m}\|f\|_{L^\infty(\cu_m)}\,.
\end{equation*}
Let~$C_0$ denote the constant in~\eqref{e.Dir.new.interior.pointwise}. Choose~$N_\gamma(d)\in\N$ so large that
\begin{equation*}
	C_0 3^{-N_\gamma}
	\leq
	\frac12\,3^{-\gamma' N_\gamma}\,.
\end{equation*}
After enlarging~$\hat L$, the hypotheses of Lemma~\ref{l.Dirichlet.interior.pointwise} hold on every interval~$[k-N_\gamma,k]$ with~$k\geq m_\ast+N_\gamma$. Applying the lemma successively on these intervals gives, for every~$n\in[m_\ast,m-N_\gamma]\cap\Z$,
\begin{equation}
	\label{e.Dir.new.large.scale.Holder.cube}
	\|u-(u)_{\cu_n}\|_{L^\infty(\cu_n)}
	\leq
	C3^{-\gamma'(m-n)}B_m(u,f)\,.
\end{equation}
The source term is harmless in the iteration since~$\shom_k^{-1}3^{2k}\leq C3^{-\gamma'(m-k)}\shom_m^{-1}3^{2m}$ for~$k\leq m$, after enlarging~$\hat L$ and decreasing~$\gamma'$ if necessary.

\smallskip

We pass from cubes to balls. Given~$R\geq\X$ and~$r\in[\X,R/2]$, choose integers~$m,n\geq m_\ast$ such that
\begin{equation*}
	\cu_m\subseteq B_R\subseteq C\cu_m\,,
	\qquad
	B_r\subseteq \cu_n\subseteq B_{Cr}\,.
\end{equation*}
If~$m-n$ is bounded by a dimensional constant, the desired estimate follows from the trivial bound and the fact that~$(r/R)^\gamma$ is then bounded below. Otherwise~\eqref{e.Dir.new.large.scale.Holder.cube}, applied in a slightly smaller centered cube contained in~$B_R$, gives
\begin{align*}
	\|u-(u)_{B_r}\|_{L^\infty(B_r)}
	&\leq
	C3^{-\gamma'(m-n)}
	\bigl(
	\|u-(u)_{B_R}\|_{\underline{L}^2(B_R)}
	+
	\shom_m^{-1}3^{2m}\|f\|_{L^\infty(B_R)}
	\bigr)
	\\
	&\leq
	C\Bigl(\frac rR\Bigr)^\gamma
	\bigl(
	\|u-(u)_{B_R}\|_{\underline{L}^2(B_R)}
	+(\log R)^{-\nf12}R^2\|f\|_{L^\infty(B_R)}
	\bigr)\,.
\end{align*}
Here we used the deterministic sharp asymptotic~\eqref{e.sm.sharp.bounds} to replace~$\shom_m^{-1}$ by~$C(\log R)^{-\nf12}$, and absorbed the fixed cube-ball comparison constants, including the changes of averages between comparable cubes and balls, by the gap~$\gamma'-\gamma$. This proves~\eqref{e.large.scale.Holder}. The proof is complete.
\end{proof}

\begin{proof}[Proof of Theorem~\ref{t.superdiffusivity}]
We use the normalization displayed in~\eqref{e.BVPs}, namely
\begin{equation*}
	S_\ep\coloneqq
	\bigl(2\cstar |\log\ep|\bigr)^{\nf12}\,.
\end{equation*}
It is enough to prove the theorem under the normalization~$U\subseteq\cu_0$. Indeed, for a general bounded smooth~$U$, choose a deterministic~$T=T(U)\geq1$ such that~$T^{-1}U\subseteq\cu_0$ and change variables~$x=Ty$. This replaces~$\ep$ by~$\ep/T$, multiplies~$f$ by~$T^2$, and changes~$|\log\ep|$ by the bounded additive constant~$\log T$; these deterministic changes are absorbed into the constant and the random threshold. The restriction~$U\subseteq\cu_0$ in Proposition~\ref{p.Dirichlet.Linfty.blackbox} is only a normalization of the reference domain. Since~$\lambda U\subseteq\cu_0$ for every~$\lambda\in[\nf13,1]$, the constants in that proposition are uniform for the family of smooth domains~$\lambda U$ in this range, with constants depending on~$(U,d)$.

\smallskip

Fix~$\alpha,\beta\in(0,1]$ with~$\beta+2\alpha<1$. Choose
\begin{equation*}
	\alpha<\alpha_0<\nf{1-\beta}{2}\,,
	\qquad
	\rho\coloneqq1-2\alpha_0>\beta\,,
	\qquad
	\theta\coloneqq \nf1{2d}\,.
\end{equation*}
We apply Lemma~\ref{l.sharp.scale.inputs} with~$\varepsilon=1$, the exponent~$\rho$, and a sufficiently large deterministic~$M=M(U,d)$, chosen uniformly for the compact family~$\{\lambda U:\lambda\in[\nf13,1]\}$.
Let~$\mathcal Z$ be a sufficiently large deterministic multiple of~$\X_0(1+\log\X_0)^C$ and of~$3^L$ for each deterministic lower bound~$L$ imposed on~$K$ below. By~\eqref{e.Dir.new.minscale.tail}, Lemma~\ref{l.Gamma.sigma.triangle}, and~$\rho>\beta$, this~$\mathcal Z$ satisfies~\eqref{e.Z.integrability}.

\smallskip

Let~$\ep\in(0,\frac12]$ with~$\ep^{-1}\geq\mathcal Z$, and put
\begin{equation*}
	R\coloneqq \ep^{-1}\,,
	\qquad
	K\coloneqq \lceil\log_3 R\rceil\,,
	\qquad
	\lambda\coloneqq R3^{-K}\in(\nf13,1]\,.
\end{equation*}
The logarithmic enlargement in the definition of~$\mathcal Z$ ensures, after increasing the deterministic constant~$C$ there if necessary, that
\begin{equation*}
	K-\lceil C\log K\rceil\geq\hat L\,,
	\qquad
	3^{K-\lceil C\log K\rceil}\geq\X_0\,.
\end{equation*}
Then~$RU=3^K(\lambda U)$. For~$y\in RU$, define
\begin{equation*}
	u_R(y)\coloneqq u^\ep(\ep y)\,,
	\qquad
	g_R(y)\coloneqq g(\ep y)\,,
	\qquad
	F_R(y)\coloneqq S_\ep\ep^2 f(\ep y)\,.
\end{equation*}
Then~$u_R$ solves
\begin{equation*}
	-\nabla\cdot\a\nabla u_R=F_R
	\quad\mbox{in }RU\,,
	\qquad
	u_R=g_R
	\quad\mbox{on }\partial(RU)\,.
\end{equation*}
Let~$\bar u_R$ solve
\begin{equation*}
	-\shom_K\Delta \bar u_R=F_R
	\quad\mbox{in }RU\,,
	\qquad
	\bar u_R=g_R
	\quad\mbox{on }\partial(RU)\,.
\end{equation*}
Proposition~\ref{p.Dirichlet.Linfty.blackbox}, applied in the domain~$\lambda U$, gives
\begin{multline}
	\label{e.Dir.new.superdiff.large.scale}
	\|u_R-\bar u_R\|_{L^\infty(RU)}
	+\|\nabla u_R-\nabla\bar u_R\|_{\underline{H}^{-1}(RU)}
	+
	\bigl\|
	\shom_K^{-1}(\a-(\k)_{\cu_K})\nabla u_R-\nabla\bar u_R
	\bigr\|_{\underline{H}^{-1}(RU)}
	\\
	\leq
	C\delta_K
	\bigl(
	\shom_K^{-1}3^{2K}\|F_R\|_{L^\infty(RU)}
	+\log K\,3^K\|\nabla g_R\|_{L^\infty(RU)}
	\bigr)\,.
\end{multline}
Since~$3^K\ep=\lambda^{-1}$ and, by~\eqref{e.sm.sharp.bounds},~$S_\ep\shom_K^{-1}\leq C$, the right side of~\eqref{e.Dir.new.superdiff.large.scale} is bounded by
\begin{equation}
	\label{e.Dir.new.superdiff.basic.rhs}
	C\delta_K\log K
	\bigl(
	\|f\|_{L^\infty(U)}
	+\|\nabla g\|_{L^\infty(U)}
	\bigr)
	\leq
	CK^{-\alpha}
	\bigl(
	\|f\|_{L^\infty(U)}
	+\|\nabla g\|_{L^\infty(U)}
	\bigr)\,,
\end{equation}
after increasing the deterministic threshold in~$\mathcal Z$.

\smallskip

Next set~$u_{{\rm hom},R}(y)\coloneqq \uhom(\ep y)$. Then
\begin{equation*}
	-S_\ep\Delta u_{{\rm hom},R}=F_R
	\quad\mbox{in }RU\,,
	\qquad
	u_{{\rm hom},R}=g_R
	\quad\mbox{on }\partial(RU)\,.
\end{equation*}
The sharp asymptotic estimate~\eqref{e.sm.sharp.bounds} and the relation~$|\log\ep|=(K+\log_3\lambda)\log3$ imply
\begin{equation}
	\label{e.Dir.new.superdiff.scalar}
	\left|\frac{S_\ep}{\shom_K}-1\right|
	\leq
	CK^{-\alpha}\,.
\end{equation}
The maximum principle and the constant-coefficient energy estimate therefore yield
\begin{equation}
	\label{e.Dir.new.superdiff.scalar.error}
	\|\bar u_R-u_{{\rm hom},R}\|_{L^\infty(RU)}
	+\|\nabla\bar u_R-\nabla u_{{\rm hom},R}\|_{\underline{H}^{-1}(RU)}
	\leq
	CK^{-\alpha}
	\|f\|_{L^\infty(U)}\,.
\end{equation}

\smallskip

It remains to change the centering of the skew part from~$\cu_K$ to~$RU$. The deterministic estimate
\begin{equation*}
	\bigl[
	|\lambda U|^{-1}\indc_{\lambda U}-1
	\bigr]_{\underline{H}^{\theta}(\cu_0)}
	\leq
	C(U,d)\,,
\end{equation*}
where~$1=|\cu_0|^{-1}\indc_{\cu_0}$ and~$\theta<\frac12$, and scaling give
\begin{align}
	\bigl|(\k)_{RU}-(\k)_{\cu_K}\bigr|
	&\leq
	C(U,d)3^{-\theta K}
	\bigl[
	\k-(\k)_{\cu_K}
	\bigr]_{\Hminushat{-\theta}(\cu_K)}
	\leq
	C(U,d)K^{\nf\rho2}\,.
	\label{e.Dir.new.U.centering}
\end{align}
The last inequality uses the second line of~\eqref{e.Dir.new.k.bounds}. Combining~\eqref{e.Dir.new.superdiff.large.scale} with the constant-coefficient energy estimate for~$\bar u_R$ gives
\begin{equation*}
	\|\nabla u_R\|_{\underline{H}^{-1}(RU)}
	\leq
	C
	\bigl(
	\shom_K^{-1}3^{2K}\|F_R\|_{L^\infty(RU)}
	+3^K\|\nabla g_R\|_{L^\infty(RU)}
	\bigr)\,.
\end{equation*}
Combining this estimate with~\eqref{e.Dir.new.U.centering} and the inequality~$\rho/2<\frac12-\alpha$, which follows from~$\alpha<\alpha_0$, gives
\begin{equation}
	\label{e.Dir.new.superdiff.centering.error}
	\bigl\|
	\shom_K^{-1}\bigl((\k)_{RU}-(\k)_{\cu_K}\bigr)\nabla u_R
	\bigr\|_{\underline{H}^{-1}(RU)}
	\leq
	CK^{-\alpha}
	\bigl(
	\|f\|_{L^\infty(U)}
	+\|\nabla g\|_{L^\infty(U)}
	\bigr)\,.
\end{equation}
Combining~\eqref{e.Dir.new.superdiff.large.scale},~\eqref{e.Dir.new.superdiff.basic.rhs},~\eqref{e.Dir.new.superdiff.scalar.error}, and~\eqref{e.Dir.new.superdiff.centering.error} gives
\begin{equation}
	\label{e.Dir.new.superdiff.flux.before.normalization}
	\bigl\|
	\shom_K^{-1}(\a-(\k)_{RU})\nabla u_R-\nabla u_{{\rm hom},R}
	\bigr\|_{\underline{H}^{-1}(RU)}
	\leq
	CK^{-\alpha}
	\bigl(
	\|f\|_{L^\infty(U)}
	+\|\nabla g\|_{L^\infty(U)}
	\bigr)\,.
\end{equation}
Moreover, the constant-coefficient energy estimate for~$u_{{\rm hom},R}$ gives
\begin{equation*}
	\|\nabla u_{{\rm hom},R}\|_{\underline{H}^{-1}(RU)}
	\leq
	C
	\bigl(
	\|f\|_{L^\infty(U)}
	+\|\nabla g\|_{L^\infty(U)}
	\bigr)\,.
\end{equation*}
Together with~\eqref{e.Dir.new.superdiff.flux.before.normalization}, this yields
\begin{equation*}
	\bigl\|
	\shom_K^{-1}(\a-(\k)_{RU})\nabla u_R
	\bigr\|_{\underline{H}^{-1}(RU)}
	\leq
	C
	\bigl(
	\|f\|_{L^\infty(U)}
	+\|\nabla g\|_{L^\infty(U)}
	\bigr)\,.
\end{equation*}
Finally, by~\eqref{e.Dir.new.superdiff.scalar}, the additional replacement of~$\shom_K^{-1}$ by~$S_\ep^{-1}$ costs the same amount. Under the change of variables~$x=\ep y$, the large-scale gradients and fluxes above are~$\ep$ times the corresponding fields in~$U$, and the normalized dual norms satisfy
\begin{equation*}
	\bigl[F\bigr]_{H^{-1}(U)}
	\leq
	C(U)
	\bigl\|\ep F(\ep\,\cdot)\bigr\|_{\underline{H}^{-1}(RU)}\,.
\end{equation*}
Rescaling back from~$RU$ to~$U$ therefore proves~\eqref{e.homogenization.error}. The proof is complete.
\end{proof}

\subsection{A coarse-grained Poincar\'e inequality in general domains}
\label{ss.Dirichlet.Whitney.Poincare}

We record here a coarse-grained Poincar\'e inequality for general domains satisfying the cubical-core Poincar\'e assumption below. The statement applies in particular to balls. The constant below is allowed to depend on the usual scale-invariant Poincar\'e and Sobolev--Poincar\'e constants of the domain. The proof uses a fixed-scale mollification on the cubical interior and estimates the remaining thin boundary layer directly.

\begin{lemma}[Coarse-grained Poincar\'e in general domains]
\label{l.Dirichlet.Whitney.Poincare}
Let~$U \subset \cu_0$ be a Lipschitz domain. There exists~$C(U,d)<\infty$ such that the following holds. For every~$\rho_{\rm loc}\in(0,1)$ and every~$M_{\rm loc}\geq C$, apply Lemma~\ref{l.sharp.scale.inputs} with~$(\varepsilon,\rho,M)=(1,\rho_{\rm loc},M_{\rm loc})$, and denote its corresponding quantities by~$h_m$,~$L_\sharp$, and~$\X_0$. For every~$D_{\rm loc}\in[1,\infty)$, there exist a constant
\begin{equation*}
	\hat L_{\rm P}=\hat L_{\rm P}(U,D_{\rm loc},M_{\rm loc},\rho_{\rm loc},\cstar,\nu,\nondegconst,d)\geq L_\sharp
\end{equation*}
and a random variable~$\X$ satisfying
\begin{equation}
	\label{e.Dir.new.Whitney.minscale.tail}
	\log \X
	=
	\O_{\Gamma_{\rho_{\rm loc}}}(\hat L_{\rm P})
	\,,
\end{equation}
with the following property. Let~$m\in\N$ satisfy~$3^m\geq\X$ and define, for~$h_m$ as in~\eqref{e.Dir.new.hm},
\begin{equation*}
	U_{\mathrm{int}} \coloneqq
	\operatorname{int}
	\biggl(
	\bigcup
	\bigl\{
	z+\overline{\cu}_{-h_m}\,:\,
	z\in 3^{-h_m}\Zd\,,
	z+\cu_{-h_m+3} \subset U
	\bigr\}
	\biggr)\,.
\end{equation*}
Assume that
\begin{equation}
	\label{e.Dir.new.Whitney.Poincare.assumption}
	\| \varphi - (\varphi)_{U_{\mathrm{int}}} \|_{L^2(U_{\mathrm{int}})}
	\leq D_{\rm loc} \| \nabla \varphi  \|_{L^2(U_{\mathrm{int}})}\,,
	\quad \forall \varphi \in H^1(U_{\mathrm{int}})\,.
\end{equation}
Then, for~$U_m \coloneqq 3^m U$, and for every~$f \in L^2(U_m)$ and~$u\in H^1(U_m)$ solving
\begin{equation*}
	-\nabla\cdot\a\nabla u=f
	\qquad \mbox{in } U_m\,,
\end{equation*}
we have the estimate
\begin{equation}
	\label{e.Dir.new.Whitney.Poincare}
	\|u-(u)_{U_m}\|_{\underline L^2(U_m)}
	\leq
	CD_{\rm loc} 3^m\shom_m^{-\nf12}\nu^{\nf12} \|\nabla u\|_{\underline L^2(U_m)}
	+
	CD_{\rm loc} 3^{2m-h_m}\nu^{-1}
	\|f\|_{\underline L^2(U_m)}
	\,.
\end{equation}
\end{lemma}

\begin{proof}
First choose~$\hat L_{\rm P}$ large enough for the deterministic bounds below and set~$\X\coloneqq\X_0\vee3^{\lceil\hat L_{\rm P}\rceil}$. Thus~$3^m\geq\X$ supplies both these bounds and~$3^m\geq\X_0$. Increasing~$\hat L_{\rm P}$ once more, without changing~$\X$, gives~\eqref{e.Dir.new.Whitney.minscale.tail} by Lemma~\ref{l.sharp.scale.inputs}.

\smallskip

In the proof we suppress the scale~$m$ from the notation: we write again~$U$,~$U_{\mathrm{int}}$ and~$U_{\mathrm{bndr}}$ for~$U_m=3^mU$,~$3^mU_{\mathrm{int}}$ and~$3^m(U\setminus U_{\mathrm{int}})$, respectively. With this convention the Poincar\'e assumption scales to
\begin{equation}
	\label{e.Dir.new.Whitney.Poincare.assumption.scaled}
	\| \varphi - (\varphi)_{U_{\mathrm{int}}} \|_{L^2(U_{\mathrm{int}})}
	\leq D_{\rm loc} 3^m  \| \nabla \varphi  \|_{L^2(U_{\mathrm{int}})}\,,
	\quad \forall \varphi \in H^1(U_{\mathrm{int}})\,.
\end{equation}
Constants denoted by~$C$ may depend on the fixed Lipschitz domain~$U$ and on~$d$, but not on~$m$,~$D_{\rm loc}$,~$u$ or~$f$. Put~$\ell_0\coloneqq3^{m-h_m}$. By the definition of~$U_{\mathrm{int}}$, every point of~$U$ whose distance to~$\partial U$ is at least~$C\ell_0$ belongs to~$U_{\mathrm{int}}$. Since the fixed domain in the statement is Lipschitz, the standard boundary-layer estimate gives
\begin{equation}
	\label{e.Dir.new.Whitney.boundary.thickness.scaled}
	|U_{\mathrm{bndr}}|
	\leq
	C3^{-h_m}|U|\,.
\end{equation}
By the lower bound on~$M_{\rm loc}$ in the statement, the definition~$h_m=\lceil M_{\rm loc}\log m\rceil$, and~\eqref{e.theorem.shom.bracket}, the choice of~$\hat L_{\rm P}$ and~$\X$ in the statement ensures that
\begin{equation}
	\label{e.Dir.new.hm.Whitney.absorption}
	C3^{-h_m}\leq \nf 14
	\qand
	3^{-\frac{h_m}{d\vee3}}
	\leq
	\shom_m^{-\nf12}\nu^{\nf12}\,.
\end{equation}
In particular,~$|U_{\mathrm{int}}|\geq\frac12 |U|$. We also enlarge~$\hat L_{\rm P}$ so that~$h_m\geq3$.

\smallskip

\emph{Step 1.} The interior mollification. Put~$k\coloneqq m-h_m$, so that~$\ell_0=3^k$. Fix a nonnegative function~$\eta\in C_c^\infty(B_{\nf12})$ satisfying~$\int_{\Rd}\eta=1$, and set~$\eta_{\ell_0}(x)\coloneqq\ell_0^{-d}\eta(x/\ell_0)$. For~$x\in U_{\mathrm{int}}$, the fixed enlargement in the definition of~$U_{\mathrm{int}}$ ensures that~$(\eta_{\ell_0}\ast u)(x)$ depends only on the values of~$u$ in~$U$.

\smallskip

Let
\begin{equation*}
	\mathcal Q_{\mathrm{int}}
	\coloneqq
	\bigl\{
	Q=z+\cu_k\,:\,
	z\in3^k\Zd\,,\ z+\cu_{k+3}\subseteq U
	\bigr\}\,,
\end{equation*}
and, for~$Q=z+\cu_k\in\mathcal Q_{\mathrm{int}}$, set~$Q^+\coloneqq z+\cu_{k+1}$. The closures of the cubes~$Q$ cover~$U_{\mathrm{int}}$. Moreover,~$Q^+\subseteq z+\cu_{k+3}\subseteq U$, and the family~$\{Q^+:Q\in\mathcal Q_{\mathrm{int}}\}$ has bounded overlap because all its members have the same size.

\smallskip

\emph{Step 2.} The estimate. Write~$w\coloneqq u-(u)_U$. We first split
\begin{equation}
	\label{e.Dir.new.Whitney.decomposition}
	\|w\|_{L^2(U)}
	\leq
	\|w\|_{L^2(U_{\mathrm{int}})}
	+
	\|w\|_{L^2(U_{\mathrm{bndr}})}\,.
\end{equation}
For every~$Q\in\mathcal Q_{\mathrm{int}}$, the ordinary Poincar\'e inequality on~$Q^+$ gives
\begin{equation}
	\label{e.Dir.new.Whitney.local.Poincare}
	\|u-(u)_{Q^+}\|_{\underline L^2(Q^+)}
	\leq
	C\ell_0\|\nabla u\|_{\underline L^2(Q^+)}\,.
\end{equation}
For~$x\in Q\cap U_{\mathrm{int}}$, the support of~$\eta_{\ell_0}(x-\cdot)$ and every segment joining~$x$ to a point of this support are contained in~$Q^+$. The fundamental theorem of calculus and bounded overlap therefore give
\begin{equation}
	\label{e.Dir.new.Whitney.interpolation.L2}
	\|u-\eta_{\ell_0}\ast u\|_{L^2(U_{\mathrm{int}})}^2
	\leq
	C\ell_0^2
	\sum_{Q\in\mathcal Q_{\mathrm{int}}}
	\|\nabla u\|_{L^2(Q^+)}^2\,,
\end{equation}
and cancellation of~$\nabla\eta_{\ell_0}$ gives, for almost every~$x\in Q\cap U_{\mathrm{int}}$,
\begin{equation*}
	\nabla(\eta_{\ell_0}\ast u)(x)
	=
	\int_{Q^+}
	\nabla\eta_{\ell_0}(x-y)
	\bigl(u(y)-(u)_{Q^+}\bigr)\,dy\,.
\end{equation*}
Since~$\|\nabla\eta_{\ell_0}\|_{L^2(\Rd)}\leq C\ell_0^{-1-\frac d2}$, integration over~$Q\cap U_{\mathrm{int}}$ and summation give
\begin{equation}
	\label{e.Dir.new.Whitney.interpolation.grad}
	\|\nabla(\eta_{\ell_0}\ast u)\|_{L^2(U_{\mathrm{int}})}^2
	\leq
	C\ell_0^{-2}
	\sum_{Q\in\mathcal Q_{\mathrm{int}}}
	\|u-(u)_{Q^+}\|_{ L^2(Q^+)}^2
	\,.
\end{equation}
The bounded overlap of the cubes~$Q^+$ in~\eqref{e.Dir.new.Whitney.interpolation.L2} yields
\begin{equation}
	\label{e.Dir.new.Whitney.u.minus.mollified}
	\|u-\eta_{\ell_0}\ast u\|_{L^2(U_{\mathrm{int}})}
	\leq
	C\ell_0\|\nabla u\|_{L^2(U)}
	\,.
\end{equation}
On the other hand, if~$Q=z+\cu_k\in\mathcal Q_{\mathrm{int}}$ and~$n\coloneqq k+1=m-h_m+1$, then~$z\in3^k\Zd\subseteq3^{n-3}\Zd$,~$Q^+=z+\cu_n\subseteq U\subseteq\cu_m$, and~$m-h_m\leq n\leq m$. Thus the refined-grid hypothesis of Lemma~\ref{l.Dirichlet.rhs.blackbox} is satisfied on~$Q^+$. We use the dual form of the fractional Poincar\'e inequality on~$Q^+$ in the form
\begin{equation*}
	\|v-(v)_{Q^+}\|_{\underline L^2(Q^+)}
	\leq
	C3^n
	\bigl(
	3^{-\frac{1}{4} n}
	\|\nabla v\|_{\underline H^{-\nf14}(Q^+)}
	\bigr)
	\qquad \forall v\in H^1(Q^+)\,.
\end{equation*}
This, together with~\eqref{e.Dir.new.grad.with.f}, applied with~$\|f\|_{\underline L^{p_\ast}(Q^+)}\leq\|f\|_{\underline L^2(Q^+)}$, gives
\begin{equation*}
	\|u-(u)_{Q^+}\|_{\underline L^2(Q^+)}
	\leq
	C\ell_0\shom_m^{-\nf12}\nu^{\nf12}
	\|\nabla u\|_{\underline L^2(Q^+)}
	+
	C\ell_0^2\nu^{-1}
	\|f\|_{\underline L^2(Q^+)}
	\,.
\end{equation*}
Combining this with~\eqref{e.Dir.new.Whitney.interpolation.grad} and using the bounded overlap of the enlargements gives
\begin{equation}
	\label{e.Dir.new.Whitney.grad.mollified}
	\|\nabla(\eta_{\ell_0}\ast u)\|_{L^2(U_{\mathrm{int}})}
	\leq
	C\shom_m^{-\nf12}\nu^{\nf12}
	\|\nabla u\|_{L^2(U)}
	+
	C\ell_0\nu^{-1}\|f\|_{L^2(U)}
	\,.
\end{equation}

\smallskip

Since~$\int_Uw=0$ and~$|U_{\mathrm{int}}|\geq\frac12|U|$, we have
\begin{equation}
	\label{e.Dir.new.Whitney.average.transfer}
	\bigl|(u)_{U_{\mathrm{int}}}-(u)_U\bigr|
	|U_{\mathrm{int}}|^{\nf12}
	\leq
	\Bigl(\frac{|U_{\mathrm{bndr}}|}{|U_{\mathrm{int}}|}\Bigr)^{\!\nf12}
	\|w\|_{L^2(U_{\mathrm{bndr}})}
	\leq
	C\|w\|_{L^2(U_{\mathrm{bndr}})}\,.
\end{equation}
Applying the assumed Poincar\'e inequality on~$U_{\mathrm{int}}$ to~$\eta_{\ell_0}\ast u-(\eta_{\ell_0}\ast u)_{U_{\mathrm{int}}}$, and then comparing its average with that of~$u$, gives
\begin{equation}
	\label{e.Dir.new.Whitney.interior.mollified}
	\|w\|_{L^2(U_{\mathrm{int}})}
	\leq
	2\|u-\eta_{\ell_0}\ast u\|_{L^2(U_{\mathrm{int}})}
	+D_{\rm loc}3^m\|\nabla(\eta_{\ell_0}\ast u)\|_{L^2(U_{\mathrm{int}})}
	+
	C\|w\|_{L^2(U_{\mathrm{bndr}})}\,.
\end{equation}
By the Sobolev--Poincar\'e inequality on~$U$ with exponent~$2^\ast$ from~\eqref{e.Dir.new.Sobolev.conjugates}, H\"older's inequality, and~\eqref{e.Dir.new.Whitney.boundary.thickness.scaled},
\begin{equation}
	\label{e.Dir.new.Whitney.boundary.layer}
	\|w\|_{L^2(U_{\mathrm{bndr}})}
	\leq
	C3^m3^{-\frac{h_m}{d\vee3}}
	\|\nabla u\|_{L^2(U)}\,.
\end{equation}
Combining~\eqref{e.Dir.new.Whitney.decomposition},~\eqref{e.Dir.new.Whitney.u.minus.mollified},~\eqref{e.Dir.new.Whitney.grad.mollified},~\eqref{e.Dir.new.Whitney.interior.mollified}, and~\eqref{e.Dir.new.Whitney.boundary.layer}, and using~$D_{\rm loc}\geq1$ and~\eqref{e.Dir.new.hm.Whitney.absorption}, gives
\begin{equation}
	\label{e.Dir.new.Whitney.Poincare.conclusion}
	\|u-(u)_U\|_{L^2(U)}
	\leq
	CD_{\rm loc} 3^m\shom_m^{-\nf12}\nu^{\nf12}\|\nabla u\|_{L^2(U)}
	+
	CD_{\rm loc} 3^m\ell_0\nu^{-1}\|f\|_{L^2(U)}\,.
\end{equation}
Dividing by~$|U|^{\nf12}$, using~$\ell_0=3^{m-h_m}$, and undoing the suppressed notation proves~\eqref{e.Dir.new.Whitney.Poincare}.
\end{proof}

\section{The invariance principle}
\label{s.invariance.principle}

In this section we complete the proof of Theorem~\ref{t.A} by proving the invariance principle~\eqref{e.convinlaw}, the quenched estimate~\eqref{e.Dt.exp}, and the annealed estimate~\eqref{e.annealed.Dt} for the displacement of the process~$X_t$. We recall the definition of the rescaled process~$X_t^\ep$ given in~\eqref{e.Xep.t.def}:
\begin{equation}
	\label{e.Xep.t.def.recall}
	X^{\ep} _t  \coloneqq
	\ep X_{\frac{t}{\ep^2 ( 8 \cstar \left| \log \ep \right|  )^{1/2}}} \,.
\end{equation}
The convergence asserted in~\eqref{e.convinlaw} follows from the following formulation after a deterministic reparametrization of the small parameter and Brownian scaling. Indeed, if~$\delta=\delta(\ep)$ is chosen by
\begin{equation*}
	\delta^2\bigl(8\cstar|\log\delta|\bigr)^{\nf12}
	=\ep^2\,,
\end{equation*}
then~$\delta\to0$ as~$\ep\to0$ and
\begin{equation*}
	\frac{\ep|\log\ep^2|^{-\nf14}}{\delta}
	\to
	\sqrt{2\cstar^{\nf12}}\,.
\end{equation*}
Thus it suffices to prove that, for~$\P$--almost every realization of the environment,
\begin{equation}
	\label{e.convinlaw.again}
	X^\ep_t
	\Rightarrow
	W_t
	\quad \mbox{in law}
	\qquad \mbox{as} \ \ep \to 0\,,
\end{equation}
with respect to the topology of locally uniform convergence on~$C([0,\infty);\Rd)$, where~$\{ W_t \}_{t\geq0}$ is a standard Brownian motion.

\subsection{Quantitative homogenization of the resolvents}
\label{ss.be.it.resolved}

We show next that Theorem~\ref{t.superdiffusivity} implies a quenched estimate on the resolvent. This asserts roughly that, for~$0\leq \lambda\leq\left| \log \ep \right| ^{\alpha}$ with~$\alpha<\frac12$, the inverse of the operator~$\lambda - \Lop^\ep$, with respect to Dirichlet boundary conditions, is close to that of the operator~$\lambda - \tfrac12\Delta$ in the~$L^\infty$ norm. Recall that~$\Lop^\ep$ is defined in~\eqref{e.A.ep}. Whenever Theorem~\ref{t.superdiffusivity} is applied below, we apply it to the equation after multiplying by~$2$, since the theorem is stated for~$-2\Lop^\ep$; this only changes the right side by a deterministic factor which is absorbed into the constants.

\begin{lemma}[Resolvent estimates]
\label{l.resolvent.estimate}
Let~$U \subseteq \Rd$ be a smooth bounded domain and~$\alpha,\beta \in (0,1]$ with~$\beta+2\alpha<1$. Let~$\mathcal{Z}$ be the random variable in the statement of Theorem~\ref{t.superdiffusivity}. There exists a constant~$C(U,\alpha,\beta, \cstar, \nu,\nondegconst, d) < \infty$ such that, for every~$\lambda\in [0,\infty)$,~$\ep\in (0,\nf12]$ with~$\ep^{-1} \geq \mathcal{Z}$, and functions~$f \in L^{\infty}(U)$ and~$g \in W^{1,\infty}(U)$, if we let~$u^\ep, \uhom \in H^1(U)$ denote the solutions of the boundary value problems
\begin{equation}
	\label{e.BVPs.resolvent}
	\left\{
	\begin{aligned}
		&\lambda u^\ep  -\Lop^\ep u^\ep =  f & \mbox{in} & \ U \,, \\
		& u^\ep = g & \mbox{on} & \ \partial U \,,
	\end{aligned}
	\right.
	\qand
	\left\{
	\begin{aligned}
		& \lambda \uhom - \tfrac12 \Delta \uhom   = f & \mbox{in} & \ U \,, \\
		& \uhom = g & \mbox{on} & \ \partial U \,,
	\end{aligned}
	\right.
\end{equation}
then we have the estimate
\begin{equation}
	\label{e.homogenization.error.resolvent}
	\bigl\| u^\ep - \uhom\bigr\|_{L^{\infty}(U)} \leq C\left| \log \ep \right| ^{-\alpha} \bigl(\| f -\lambda u^\ep \|_{L^\infty(U)} + \| \nabla g \|_{L^\infty(U)} \bigr)
	\,.
\end{equation}
\end{lemma}
\begin{proof}
We decompose~$\uhom = u_1 + u_2$ where~$u_1$ and~$u_2$ are the solutions of the Dirichlet problems
\begin{equation*}
	\left\{
	\begin{aligned}
		& - \tfrac12 \Delta u_1   = f - \lambda u^\ep & \mbox{in} & \ U \,, \\
		& u_1  = g & \mbox{on} & \ \partial U \,,
	\end{aligned}
	\right.
	\qand
	\left\{
	\begin{aligned}
		& \lambda u_2 - \tfrac12\Delta u_2 = \lambda ( u^\ep - u_1) & \mbox{in} & \ U \,, \\
		& u_2  = 0 & \mbox{on} & \ \partial U \,.
	\end{aligned}
	\right.
\end{equation*}
The maximum principle gives~$\| u_2 \|_{L^\infty(U) } \leq  \| u^\ep - u_1 \|_{L^\infty(U) }$. An application of Theorem~\ref{t.superdiffusivity} yields
\begin{align*}
	\| u^\ep - u_1 \|_{L^\infty(U)}
	&
	\leq
	C\left| \log \ep \right| ^{-\alpha}
	\bigl( \| f - \lambda u^\ep \|_{L^\infty(U)} + \| \nabla g \|_{L^\infty(U)} \bigr)
	\,.
\end{align*}
The triangle inequality and these estimates give us~\eqref{e.homogenization.error.resolvent}.
\end{proof}

We next deduce a resolvent estimate suited to the parabolic problems below.

\begin{lemma}
\label{l.resolvent.estimate.implicit.Euler}
Let~$U \subseteq \Rd$ be a smooth bounded domain and~$\alpha,\beta \in (0,1]$ with~$\beta+2\alpha<1$. Let~$\mathcal{Z}$ be the random variable in the statement of Theorem~\ref{t.superdiffusivity}. There exists a constant~$C(U,\alpha,\beta, \cstar, \nu,\nondegconst, d) < \infty$ such that, for every~$f \in L^\infty(U)$,~$g\in W^{1,\infty}(U)$ with~$\Delta g\in L^\infty(U)$,~$\ep\in (0,\nf12]$ with~$\ep^{-1} \geq \mathcal{Z}$ and~$\lambda\in [0, \left| \log \ep \right| ^\alpha]$, if we let~$u^\ep, \uhom \in H^1(U)$ denote the solutions of the boundary value problems
\begin{equation}
	\label{e.BVPs.resolvent.pp}
	\left\{
	\begin{aligned}
		&\lambda u^\ep  -\Lop^\ep u^\ep =  \lambda f & \mbox{in} & \ U \,, \\
		& u^\ep = g & \mbox{on} & \ \partial U \,,
	\end{aligned}
	\right.
	\qand
	\left\{
	\begin{aligned}
		& \lambda \uhom - \tfrac12\Delta \uhom  = \lambda f & \mbox{in} & \ U \,, \\
		& \uhom = g & \mbox{on} & \ \partial U \,,
	\end{aligned}
	\right.
\end{equation}
then we have the estimate
\begin{equation}
	\label{e.homogenization.error.resolvent.pp}
	\bigl\| u^\ep - \uhom\bigr\|_{L^{\infty}(U)}
	\leq
	C\left| \log \ep \right| ^{-\alpha} \bigl( \| \nabla g\|_{L^\infty(U)} +  \| \Delta g  \|_{L^\infty(U)}  + \lambda \| f- g\|_{L^\infty(U)} \bigr)
	\,.
\end{equation}
\end{lemma}
\begin{proof}
Let~$w^\ep$ and~$\tilde{w}$ be respectively the solutions of
\begin{equation*}
	\left\{
	\begin{aligned}
		& -\Lop^\ep w^\ep =  -\tfrac12\Delta g & \mbox{in} & \ U \,, \\
		& w^\ep = g & \mbox{on} & \ \partial U \,,
	\end{aligned}
	\right.
	\qand
	\left\{
	\begin{aligned}
		&\lambda \tilde{w}  -\tfrac12\Delta \tilde{w} =  \lambda (g  - w^\ep)  & \mbox{in} & \ U \,, \\
		& \tilde{w} = 0 & \mbox{on} & \ \partial U \,.
	\end{aligned}
	\right.
\end{equation*}
Then we can compare~$w^\ep$ to~$g$ using Theorem~\ref{t.superdiffusivity}, and then compare~$u^\ep - w^\ep$ to~$\uhom - g +\tilde{w}$ using Lemma~\ref{l.resolvent.estimate}, to obtain, respectively,
\begin{equation*}
	\| w^\ep - g \|_{L^\infty(U)}
	\leq
	C\left| \log \ep \right| ^{-\alpha} \bigl( \| \nabla g \|_{L^\infty(U)} + \| \Delta g \|_{L^\infty(U)} \bigr)
	\,.
\end{equation*}
Applying Lemma~\ref{l.resolvent.estimate} to~$u^\ep-w^\ep$ first gives
\begin{equation*}
	\| (u^\ep - w^\ep) - (\uhom - g +\tilde{w}) \|_{L^\infty(U)}
	\leq
	C\left| \log \ep \right| ^{-\alpha} \| \lambda (f  - u^\ep) + \tfrac12\Delta g \|_{L^\infty(U)}
	\,.
\end{equation*}
The function~$u^\ep-w^\ep$ has zero boundary data and solves
\begin{equation*}
	\lambda (u^\ep-w^\ep)-\Lop^\ep(u^\ep-w^\ep)
	=
	\lambda(f-w^\ep)+\tfrac12\Delta g
	\quad \mbox{in} \ U\,.
\end{equation*}
Thus, by the maximum principle,
\begin{equation*}
	\lambda\|u^\ep-w^\ep\|_{L^\infty(U)}
	\leq
	\|\lambda(f-w^\ep)+\tfrac12\Delta g\|_{L^\infty(U)}\,.
\end{equation*}
The previous displays, the triangle inequality and the restriction~$\lambda \left| \log \ep \right| ^{-\alpha} \leq 1$ yield
\begin{equation*}
	\| (u^\ep - w^\ep) - (\uhom - g +\tilde{w}) \|_{L^\infty(U)}
	\leq
	C\left| \log \ep \right| ^{-\alpha} \bigl( \| \nabla g\|_{L^\infty(U)} +  \| \Delta g  \|_{L^\infty(U)}  + \lambda \| f- g\|_{L^\infty(U)} \bigr)\,.
\end{equation*}
By the maximum principle, the function~$\tilde{w}$ satisfies
\begin{equation*}
	\| \tilde{w} \|_{L^\infty(U)}
	\leq \| w^\ep - g\|_{L^\infty(U)}
	\leq C\left| \log \ep \right| ^{-\alpha} \bigl(\| \nabla g \|_{L^\infty(U)} +  \| \Delta g \|_{L^\infty(U)} \bigr) \,.
\end{equation*}
Combining the previous displays, using the triangle inequality, we obtain
\begin{align*}
	\| u^\ep - \uhom \|_{L^\infty(U)}
	&
	\leq
	\| (u^\ep - w^\ep) - (\uhom - g +\tilde{w}) \|_{L^\infty(U)} + \| w^\ep -  g \|_{L^\infty(U)} + \| \tilde{w} \|_{L^\infty(U)}
	\notag \\ &
	\leq
	C\left| \log \ep \right| ^{-\alpha} \bigl( \| \nabla g\|_{L^\infty(U)} +  \| \Delta g  \|_{L^\infty(U)}  + \lambda \| f- g\|_{L^\infty(U)} \bigr)
	\,.
\end{align*}
The proof is complete.
\end{proof}

The implicit Euler numerical scheme for a parabolic equation is equivalent to iterating the inverse of the resolvent operator. Therefore, we can expect to obtain homogenization estimates for parabolic boundary value problems as a consequence of Lemma~\ref{l.resolvent.estimate.implicit.Euler}. This idea leads to the following statement. In what follows, we denote the parabolic boundary of a cylinder~$(0,T) \times U \subseteq\R^{1+d}$ by
\begin{equation*}
	\partial_{\mathrm{par}} ((0,T)\times U)  \coloneqq  \bigl( \{ 0 \} \times U \bigr) \cup \bigl( (0,T) \times \partial U \bigr) \,.
\end{equation*}

\begin{lemma}[Homogenization estimates for compatible parabolic problems]
\label{l.homogenization.parabolic.estimate}
Let~$U \subseteq \Rd$ be a smooth bounded domain and~$\alpha,\beta \in (0,1]$ with~$\beta+2\alpha<1$. Let~$\mathcal{Z}$ be the random variable in the statement of Theorem~\ref{t.superdiffusivity}. There exists a constant~$C(U,\alpha,\beta, \cstar, \nu,\nondegconst, d) < \infty$ such that, for every~$T\in (0,\infty)$ and~$\ep\in (0,\nf12]$ with~$\ep^{-1} \geq \mathcal{Z}$, and~$h \in C_c^\infty(U)$, if we let~$u^\ep$ and~$\uhom$ denote the solutions of the compatible Cauchy-Dirichlet problems
\begin{equation}
	\label{e.BVPs.parabolic}
	\left\{
	\begin{aligned}
		& \partial_t u^\ep  -\Lop^\ep u^\ep =  0 & \mbox{in} & \ (0,\infty) \times U \,, \\
		& u^\ep = 0 & \mbox{on} & \ (0,\infty) \times \partial U \,,
		\\
		& u^\ep = h & \mbox{on} & \ \{0\} \times U \,,
	\end{aligned}
	\right.
	\qand
	\left\{
	\begin{aligned}
		& \partial_t \uhom  - \tfrac12\Delta \uhom =  0 & \mbox{in} & \ (0,\infty) \times U \,, \\
		& \uhom = 0 & \mbox{on} & \ (0,\infty) \times \partial U \,,
		\\
		& \uhom = h & \mbox{on} & \ \{0\} \times U \,,
	\end{aligned}
	\right.
\end{equation}
then we have the estimate
\begin{equation}
	\label{e.homogenization.parabolic.estimate}
	\bigl\| u^\ep - \uhom  \bigr\|_{L^\infty((0,T]\times U)}
	\leq
	C(1+T) \left| \log \ep \right| ^{-\nf{\alpha}{3}}
	\| \nabla^2h  \|_{L^\infty(U)}
	\,.
\end{equation}
\end{lemma}
\begin{proof}
Since~$h$ is compactly supported in~$U$, the compatibility conditions at the lateral boundary are automatic. We first record the second-time-derivative estimate in the form used below. If~$g\in D((\Lop^\ep)^2)$ has the compatible zero boundary traces and~$u_g^\ep$ is its Dirichlet semigroup solution, then~$\partial_t^2u_g^\ep$ is the solution of
\begin{equation}
	\label{e.BVPs.parabolic.second}
	\left\{
	\begin{aligned}
		& \partial_t (\partial_t^2 u_g^\ep)  -\Lop^\ep (\partial_t^2 u_g^\ep) =  0 & \mbox{in} & \ (0,\infty) \times U \,, \\
		& (\partial_t^2 u_g^\ep)  = 0 & \mbox{on} & \ (0,\infty) \times \partial U \,, \\
		& (\partial_t^2 u_g^\ep)  = (\Lop^\ep)^2 g & \mbox{on} & \ \{ 0 \} \times U \,.
	\end{aligned}
	\right.
\end{equation}
By the maximum principle, we obtain the bound
\begin{equation}
	\label{e.parabolic.partial.t.two.bound}
	\| \partial_t^2 u_g^\ep \|_{L^\infty( (0,\infty) \times U)}
	\leq
	\| (\Lop^\ep)^2 g \|_{L^\infty(U)}
	\,.
\end{equation}
Similarly, we have the bound
\begin{equation}
	\label{e.parabolic.partial.t.two.bound.hom}
	\| \partial_t^2 \uhom \|_{L^\infty( (0,\infty) \times U)}
	\leq
	\frac14 \| \Delta^2 h \|_{L^\infty(U)}
	\,.
\end{equation}
The standing coefficient regularity does not imply that an arbitrary smooth~$h$ belongs to~$D((\Lop^\ep)^2)$. Therefore our first goal is to replace~$h$ by ``well-prepared'' initial data~$h^\ep$ which is close to~$h$ in~$L^\infty$, belongs to this domain, and has good effective bounds on~$(\Lop^\ep)^2 h^\ep$.

\smallskip

Throughout we let~$\lambda > 0$ be a parameter which will eventually be taken to be a power of~$\left| \log \ep \right|$ below, and we use~$\mu>0$ as an auxiliary resolvent scale. At the end of the proof we choose~$\mu=\lambda^{\nf12}$.

\smallskip

\emph{Step 1.} We modify the initial data so that it is well-prepared for the operator~$\Lop^\ep$. The claim is that, for~$0<\mu\leq c|\log\ep|^\alpha$, with~$c>0$ sufficiently small, there exists~$h^\ep\in L^\infty(U)$, belonging to the domain of~$(\Lop^\ep)^2$, satisfying
\begin{equation}
	\label{e.well.prepared.data}
	\mu \| h^\ep - h \|_{{L}^{\infty}(U)}
	+
	\bigl\| \Lop^\ep h^\ep \bigr\|_{L^\infty(U)}
	+
	\mu^{-1} \bigl\| (\Lop^\ep)^2 h^\ep \bigr\|_{L^\infty(U)}
	\leq
	C \| \nabla^2 h \|_{{L}^{\infty}(U)}
	\,.
\end{equation}
We consider the functions~$h^\ep_1$ and~$h^\ep_2$ defined recursively as the solutions of
\begin{equation}
	\label{e.BVPs.modify.h}
	\left\{
	\begin{aligned}
		&\mu h^\ep_1  -\Lop^\ep h^\ep_1 =  \mu \, h & \mbox{in} & \ U \,, \\
		& h^\ep_1 = 0 & \mbox{on} & \ \partial U \,,
	\end{aligned}
	\right.
	\qand
	\left\{
	\begin{aligned}
		& \mu h^\ep_2 - \Lop^\ep h^\ep_2   =  \mu (h^\ep_1 - h ) & \mbox{in} & \ U \,, \\
		& h^\ep_2 = 0 & \mbox{on} & \ \partial U \,.
	\end{aligned}
	\right.
\end{equation}
We compare~$h^\ep_1$ to the solution~$\overline{h}_1$ of the problem
\begin{equation}
	\label{e.BVPs.modify.h.hom}
	\left\{
	\begin{aligned}
		&\mu \overline{h}_1  -\tfrac12\Delta \overline{h}_1 =  \mu \, h & \mbox{in} & \ U \,, \\
		& \overline{h}_1 = 0 & \mbox{on} & \ \partial U \,.
	\end{aligned}
	\right.
\end{equation}
An application of Lemma~\ref{l.resolvent.estimate} yields
\begin{equation*}
	\| h^\ep_1 - \overline{h}_1 \|_{{L}^{\infty}(U)}
	\leq
	C\left| \log \ep \right| ^{-\alpha} \mu \| h - h_1^\ep \|_{{L}^\infty(U)}
	\,.
\end{equation*}
Observe that~$h - \overline{h}_1$ satisfies
\begin{equation*}
	\left\{
	\begin{aligned}
		&\mu (h-\overline{h}_1)  -\tfrac12\Delta (h-\overline{h}_1) =  -\tfrac12\Delta h & \mbox{in} & \ U \,, \\
		& h- \overline{h}_1 = 0 & \mbox{on} & \ \partial U \,.
	\end{aligned}
	\right.
\end{equation*}
By the maximum principle,
\begin{equation*}
	\| h - \overline{h}_1 \|_{{L}^{\infty}(U)} \leq C\mu^{-1} \| \Delta h \|_{{L}^{\infty}(U)}\leq C\mu^{-1} \| \nabla^2 h \|_{{L}^{\infty}(U)}
	\,.
\end{equation*}
Thus, by the triangle inequality,
\begin{equation*}
	\| h - h^\ep_1  \|_{{L}^{\infty}(U)}
	\leq
	C\left| \log \ep \right| ^{-\alpha}
	\mu \| h - h_1^\ep \|_{{L}^\infty(U)}
	+ C\mu^{-1} \| \nabla^2 h \|_{{L}^{\infty}(U)}
	\,.
\end{equation*}
For~$\mu \leq c \left| \log \ep \right| ^{\alpha}$, we can reabsorb the first term on the right side to obtain
\begin{equation*}
	\| h - h^\ep_1  \|_{{L}^{\infty}(U)}
	\leq
	C\mu^{-1} \| \nabla^2 h\|_{{L}^{\infty}(U)}
	\,.
\end{equation*}
By the maximum principle, we have
\begin{equation*}
	\| h^\ep_2 \|_{{L}^{\infty}(U)}
	\leq
	\| h - h^\ep_1  \|_{{L}^{\infty}(U)}
	\leq
	C\mu^{-1} \| \nabla^2 h\|_{{L}^{\infty}(U)}
	\,.
\end{equation*}
Define~$h^\ep  \coloneqq  h_1^\ep + h_2^\ep$ and observe that the above estimates and the triangle inequality yield the estimate for the first term on the left side of~\eqref{e.well.prepared.data}. Since~$- \Lop^\ep h^{\ep}  = - \mu h^\ep_2~$ and~$-\Lop^\ep \bigl( \Lop^\ep h^\ep \bigr) = \mu^2 ( h_1^\ep - h - h_2^\ep )$, the above estimates and the triangle inequality also give us the estimate for the remaining terms in~\eqref{e.well.prepared.data}. The identities are justified in the elliptic operator domain: the recursive construction gives~$\Lop^\ep h^\ep=\mu h_2^\ep$ and~$h_2^\ep$ lies in the domain of~$\Lop^\ep$. Equivalently, one may first regularize the coefficients, prove the estimates with constants independent of the regularization, and pass to the limit.

\smallskip

\emph{Step 2.} We may also modify~$h$ so that it is well-prepared for the Laplace operator. The claim is that there exists~$h' \in C^\infty(U)\cap L^\infty(U)$ satisfying
\begin{equation}
	\label{e.well.prepared.data.hom}
	\mu \| h' - h \|_{{L}^{\infty}(U)}
	+
	\bigl\| \Delta h' \bigr\|_{L^\infty(U)}
	+
	\mu^{-1} \bigl\| \Delta^2 h' \bigr\|_{L^\infty(U)}
	\leq
	C \| \nabla^2 h \|_{{L}^{\infty}(U)}
	\,.
\end{equation}
We omit the argument, as it follows from the one given in Step~1, just replacing~$\Lop^\ep$ with~$\Delta$. Since all auxiliary Dirichlet problems in Steps~1 and~2 have zero boundary data, we have~$h^\ep=h'=0$ on~$\partial U$; moreover~$\Lop^\ep h^\ep=\mu h_2^\ep$ and the analogous identity~$\Delta h'=\mu h'_2$, where~$h'_2$ denotes the second auxiliary function in the Laplacian construction, vanish on~$\partial U$. These are precisely the boundary compatibility conditions needed for the second-time-derivative bounds below.

\smallskip

\emph{Step 3.} Let~$v^\ep$ solve the first problem in~\eqref{e.BVPs.parabolic} with initial datum~$h^\ep$, and let~$\vhom$ solve the second problem with initial datum~$h'$. Applying~\eqref{e.parabolic.partial.t.two.bound} with~$g=h^\ep$ and using~\eqref{e.parabolic.partial.t.two.bound.hom},~\eqref{e.well.prepared.data}, and~\eqref{e.well.prepared.data.hom}, we obtain
\begin{equation}
	\label{e.well.prepared.data.yes}
	\| \partial_t^2 v^\ep \|_{L^\infty( (0,T] \times U)}
	+
	\| \partial_t^2 \vhom \|_{L^\infty( (0,T] \times U)}
	\leq
	C \mu \| \nabla^2 h \|_{L^\infty(U)}
	\,.
\end{equation}
Using also the maximum principle, we obtain
\begin{equation}
	\label{e.perturb.uep.parab}
	\| v^\ep - u^\ep \|_{L^\infty( (0,\infty) \times U)}
	+
	\| \vhom - \uhom \|_{L^\infty( (0,\infty) \times U)}
	\leq
	C\mu^{-1} \| \nabla^2 h \|_{L^\infty(U)}
	\,.
\end{equation}
Moreover, by Taylor's theorem,~\eqref{e.well.prepared.data.yes} and the triangle inequality, we have, for every~$t,t_0\in [0,\infty)$,
\begin{multline}
	\label{e.well.prepared.data.yes.indeed}
	\bigl\| v^\ep (t+t_0,\cdot) - v^\ep(t_0,\cdot) - t \partial_t v^\ep (t_0,\cdot) \bigr\|_{L^\infty(U)}
	+
	\bigl\| \vhom (t+t_0,\cdot) - \vhom(t_0,\cdot) - t \partial_t \vhom(t_0,\cdot) \bigr\|_{L^\infty(U)}
	\\
	\leq
	Ct^2 \mu \| \nabla^2 h \|_{L^\infty(U)}\,.
\end{multline}
Define now, for every~$t_0\in [0,\infty)$,
\begin{equation*}
	w^\ep_{\lambda,t_0} (x) \coloneqq  \int_0^\infty \lambda \exp(-\lambda t) v^\ep(t+t_0,x) \,dt \,,
	\quad
	w_{\mathrm{hom},\lambda,t_0} (x) \coloneqq  \int_0^\infty \lambda \exp(-\lambda t) \vhom(t+t_0,x) \,dt
	\,.
\end{equation*}
Notice that, by integration by parts,
\begin{align*}
	w^\ep_{\lambda,t_0} (x)
	&
	=
	v^\ep(t_0,x) + \lambda^{-1} \partial_t v^\ep(t_0,x)
	+ \lambda^{-1} \int_0^\infty \exp(-\lambda t) \partial_{t}^2 v^\ep(t_0+t,x) \, dt\,, \quad \mbox{and} \\
	w_{\mathrm{hom},\lambda,t_0}
	&=
	\vhom(t_0,x) + \lambda^{-1} \partial_t \vhom(t_0,x)
	+ \lambda^{-1} \int_0^\infty \exp(-\lambda t) \partial_{t}^2 \vhom(t_0+t,x) \, dt
	\,.
\end{align*}
  Therefore, we deduce by~\eqref{e.well.prepared.data.yes.indeed} that
\begin{multline}
	\label{e.time.increment}
	\bigl\| v^\ep (t_0+\lambda^{-1},\cdot) - \vhom(t_0+\lambda^{-1},\cdot) - (w^\ep_{\lambda,t_0} - w_{\mathrm{hom},\lambda,t_0} ) \bigr\|_{L^\infty(U)}
	\\
	\leq
	C \lambda^{-2}\mu \| \nabla^2 h \|_{L^\infty(U)}  \,.
\end{multline}
Thus our goal is to estimate the term~$w^\ep_{\lambda,t_0} - w_{\mathrm{hom},\lambda,t_0}$.

\smallskip

Observe next that~$w^\ep_{\lambda,t_0}$ and~$w_{\mathrm{hom},\lambda,t_0}$ are respectively the solutions of
\begin{equation*}
	\left\{
	\begin{aligned}
		&\lambda w^\ep_{\lambda,t_0} - \Lop^\ep w^\ep_{\lambda,t_0} = \lambda v^\ep(t_0,\cdot) & \mbox{in} & \ U \,, \\
		& w^\ep_{\lambda,t_0} = 0
		& \mbox{on} & \ \partial U \,,
	\end{aligned}
	\right.
	\qand
	\left\{
	\begin{aligned}
		& \lambda w_{\mathrm{hom},\lambda,t_0} - \tfrac12\Delta w_{\mathrm{hom},\lambda,t_0}  =  \lambda \vhom(t_0,\cdot) & \mbox{in} & \ U \,, \\
		& w_{\mathrm{hom},\lambda,t_0} = 0 & \mbox{on} & \ \partial U \,.
	\end{aligned}
	\right.
\end{equation*}
In order to apply Lemma~\ref{l.resolvent.estimate.implicit.Euler}, we also consider the auxiliary solution
\begin{equation*}
	\left\{
	\begin{aligned}
		&\lambda \tilde w^\ep_{\lambda,t_0} - \Lop^\ep \tilde w^\ep_{\lambda,t_0} = \lambda \vhom(t_0,\cdot)  & \mbox{in} & \ U \,, \\
		& \tilde w^\ep_{\lambda,t_0}  = 0 & \mbox{on} & \ \partial U \,.
	\end{aligned}
	\right.
\end{equation*}
By Lemma~\ref{l.resolvent.estimate.implicit.Euler} and the equation of~$w_{\mathrm{hom},\lambda,t_0}$, we obtain that
\begin{align*}
	&\bigl\| \tilde w^\ep_{\lambda,t_0} - w_{\mathrm{hom},\lambda,t_0} \bigr\|_{{L}^{\infty}(U)} \\
	&\qquad \leq
	C\left| \log \ep \right| ^{-\alpha} \bigl(\| \nabla w_{\mathrm{hom},\lambda,t_0}  \|_{L^{\infty}(U)} + \lambda \| w_{\mathrm{hom},\lambda,t_0}  - \vhom(t_0,\cdot)  \|_{L^{\infty}(U)}  \bigr) \,.
\end{align*}
Here we used~$w_{\mathrm{hom},\lambda,t_0}$ as the interior extension of the zero boundary data; the Laplacian term in~\eqref{e.homogenization.error.resolvent.pp} is then controlled by the equation of~$w_{\mathrm{hom},\lambda,t_0}$. In other words, the boundary datum in Lemma~\ref{l.resolvent.estimate.implicit.Euler} is the homogenized solution itself. By the maximum principle,
\begin{equation*}
	\bigl\| w^\ep_{\lambda,t_0} - \tilde{w}_{\lambda,t_0}^\ep \bigr\|_{L^\infty(U)}
	\leq
	\| (v^\ep -  \vhom)(t_0,\cdot) \|_{L^\infty(U)}  \,.
\end{equation*}
Furthermore, using the formula for~$w_{\mathrm{hom},\lambda,t_0}$, we get by integration by parts, the maximum principle, and~\eqref{e.well.prepared.data.hom}
\begin{align*}
	\bigl\|  w_{\mathrm{hom},\lambda,t_0}  - \vhom(t_0,\cdot) \bigr\|_{L^\infty(U)}
	& \leq
	\biggl\| \int_{0}^\infty \exp(-\lambda t) \partial_{t} \vhom(t_0 + t,\cdot) \, dt \biggr\|_{L^\infty(U)}
	\notag \\ &
	\leq
	\frac1{\lambda}
	\| \partial_t \vhom\|_{L^\infty((t_0,\infty)   \times U)}
	\notag \\ &
	\leq
	\frac{C}{\lambda} \| \nabla^2h  \|_{L^\infty(U)}
	\,,
\end{align*}
and, by the equation of~$w_{\mathrm{hom},\lambda,t_0}$ and the global elliptic gradient estimate for the zero-boundary Poisson problem in the smooth domain~$U$,
\begin{align*}
	\| \nabla w_{\mathrm{hom},\lambda,t_0}  \|_{L^{\infty}(U)}
	&
	\leq
	C \lambda \bigl\|  w_{\mathrm{hom},\lambda,t_0}  - \vhom(t_0,\cdot) \bigr\|_{L^\infty(U)}
	\notag \\ &
	\leq
	C \| \nabla^2h  \|_{L^\infty(U)}
	\,.
\end{align*}
The above four displays and the triangle inequality yield that
\begin{align*}
	\bigl\| w^\ep_{\lambda,t_0} - w_{\mathrm{hom},\lambda,t_0}  \bigr\|_{L^\infty(U)}
	&
	\leq
	\| (v^\ep -  \vhom)(t_0,\cdot) \|_{L^\infty(U)}
	\notag \\ & \qquad  +
	C\left| \log \ep \right| ^{-\alpha}
	\| \nabla^2h  \|_{L^\infty(U)}
	\,.
\end{align*}
The previous display and~\eqref{e.time.increment} together yield
\begin{align*}
	\lefteqn{
	\bigl\| ( v^\ep - \vhom ) (t_0+\lambda^{-1} ,\cdot) \bigr\|_{L^\infty(U)}
	-
	\| (v^\ep -  \vhom)(t_0,\cdot) \|_{L^\infty(U)}
	}
	\qquad &
	\notag \\ &
	\leq
	C\bigl(\lambda^{-2} \mu+ \left| \log \ep \right| ^{-\alpha}  \bigr)
	\| \nabla^2h  \|_{L^\infty(U)}
	\,.
\end{align*}
Iterating this estimate and using~\eqref{e.well.prepared.data} and~\eqref{e.well.prepared.data.hom} yields
\begin{align*}
	\lefteqn{
	\sup_{k \in\{0,\ldots,N \} }
	\bigl\| ( v^\ep - \vhom ) (k \lambda^{-1} ,\cdot) \bigr\|_{L^\infty(U)}
	} \qquad &
	\notag \\ &
	\leq
	\| h^\ep - h' \|_{{L}^{\infty}(U)}
	+
	CN\bigl(\lambda^{-2} \mu + \left| \log \ep \right| ^{-\alpha}  \bigr)
	\| \nabla^2h  \|_{L^\infty(U)}
	\notag \\ &
	\leq
	C\bigl( \mu^{-1} + N \lambda^{-2} \mu + N \left| \log \ep \right| ^{-\alpha} \bigr)
	\| \nabla^2h  \|_{L^\infty(U)}
	\,.
\end{align*}
By the maximum principle,
\begin{equation*}
	\| \partial_t v^\ep\|_{L^\infty((0,\infty) \times U)}
	+
	\| \partial_t \vhom\|_{L^\infty((0,\infty) \times U)}
	\leq
	C \| \nabla^2h  \|_{L^\infty(U)}
	\,.
\end{equation*}
The preceding estimate is at the grid times~$k\lambda^{-1}$. For an arbitrary~$t\in[0,T]$, we choose a neighboring grid time and use the previous display to bound the interpolation error by the~$\lambda^{-1}$ term below. From the above two displays with the change of variables~$T = N \lambda^{-1}$ we obtain
\begin{align*}
	\bigl\| v^\ep - \vhom  \bigr\|_{L^\infty((0,T]\times U)}
	&
	\leq
	C\bigl( \mu^{-1} + \lambda^{-1} + T \lambda^{-1} \mu + T \lambda \left| \log \ep \right| ^{-\alpha} \bigr)
	\| \nabla^2h  \|_{L^\infty(U)}
	\,.
\end{align*}
The previous inequality,~\eqref{e.perturb.uep.parab}, the choices~$\mu = \lambda^{\nf12}$ and~$\lambda \coloneqq   \left| \log \ep \right| ^{\nf{2\alpha}{3}}$ and the triangle inequality imply~\eqref{e.homogenization.parabolic.estimate}. The proof is complete.
\end{proof}

\subsection{Estimates on the first exit time of the process}
\label{ss.exit.times}
Given a domain~$U\subseteq \Rd$, we define
\begin{equation*}
	T_U \coloneqq  \inf\{ s>0 \,:\, X_s \not \in U\}
\end{equation*}
to be the first exit time from~$U$ of the diffusion process~$X_t$ given in~\eqref{e.sde.intro}. Similarly we let~$T_U^{\ep}$ be the first exit time from~$U$ of the rescaled process~$\{ X_t^{\ep}\}$ defined in~\eqref{e.Xep.t.def.recall}. For convenience, we define the time scale~$\tau_\ep$ by
\begin{equation}
	\label{e.tau.ep.def}
	\tau_\ep  \coloneqq  \ep^{-2} ( 8 \cstar \left| \log \ep \right|  )^{-\nf12}\,.
\end{equation}
The factor~$8$ is the deterministic normalization corresponding to~\eqref{e.A.ep.recall} and the diffusive space-time rescaling in~\eqref{e.Xep.t.def.recall}. It is immediate from the definitions that
\begin{equation}
	\label{e.tau.ep.relations}
	X^\ep_t = \ep X_{t\tau_\ep}
	\quad \mbox{and} \quad
	T_{U}^{\ep} = \tau_\ep^{-1} T_{\ep^{-1}U}\,.
\end{equation}
The probability that these processes starting at~$x\in U$ have exited~$U$ before~$t>0$ is denoted by
\begin{equation*}
	p_U (t,x)  \coloneqq  \mathbf{P}^{x} \bigl[ T_{U} \leq t \bigr]
	\quad \mbox{and} \quad
	p^\ep_U (t,x)  \coloneqq  \mathbf{P}^{x/\ep} \bigl[ T_{U}^\ep \leq t \bigr]
	\,.
\end{equation*}
The function~$(t,x) \mapsto p^\ep_U (t,x)$ is the unique solution of the parabolic Cauchy-Dirichlet problem
\begin{equation}
	\label{e.BVPs.stopping.time}
	\left\{
	\begin{aligned}
		&\partial_t p^\ep_U  - \Lop^\ep p^\ep_U = 0 & \mbox{in} & \ (0,\infty) \times U \,, \\
		& p^\ep_U = 1 & \mbox{on} & \ (0,\infty) \times  \partial U \,, \\
		& p^\ep_U = 0 & \mbox{on} & \ \{0 \}  \times  U \,.
	\end{aligned}
	\right.
\end{equation}
In the next lemma, we obtain an upper bound on~$p^\ep_{B_1}(t,\cdot)$ for small times~$t$.

\begin{lemma}
\label{l.exit.time.estimate}
Let~$\alpha,\beta \in (0,1)$ with~$\beta+2\alpha<1$ and let~$\mathcal{Z}$ be the random variable in the statement of Theorem~\ref{t.superdiffusivity}, enlarged if necessary so that the parabolic estimate~\eqref{e.homogenization.parabolic.estimate} applies uniformly in~$B_\rho$ for~$\rho\in[\nf12,1]$ in the proof below. There exists a constant~$C(\alpha,\beta, \cstar, \nu,\nondegconst, d) < \infty$ such that, for every~$\ep\in (0,\nf12]$ with~$\ep^{-1} \geq \mathcal{Z}$ and every~$t_0\in (0,1]$,
\begin{equation}
	\label{e.exit.time.crushed}
	\bigl\| p^\ep_{B_1} \bigr\|_{L^\infty([0,t_0]\times B_{\nicefrac12})}
	\leq
	2\exp \bigl( -c \min \bigl\{ t_0^{-\nf12} , \left| \log \ep \right| ^{\nf{\alpha}{6}} \bigr\}  \bigr) \,.
\end{equation}
\end{lemma}
\begin{proof}
The uniformity in~$\rho\in[\nf12,1]$ follows by rescaling~$B_\rho$ to~$B_1$ before applying Lemma~\ref{l.homogenization.parabolic.estimate}. This replaces~$\ep$ by~$\ep/\rho$ and gives
\begin{equation*}
	\rho^2\Lop_x^\ep
	=
	q_{\ep,\rho}\Lop_y^{\ep/\rho}\,,
	\qquad
	q_{\ep,\rho}
	\coloneqq
	\biggl(\frac{|\log(\ep/\rho)|}{|\log\ep|}\biggr)^{\!\nf12}
	=1+O(|\log\ep|^{-1})\,.
\end{equation*}
The homogenized comparison therefore has generator~$q_{\ep,\rho}\Delta/2$. A deterministic time change, or equivalently the heat-semigroup estimate over~$t\in[0,1]$, contributes at most~$C|q_{\ep,\rho}-1|\|\nabla^2g\|_{L^\infty}$ for the comparison datum~$g$; this is absorbed by the~$|\log\ep|^{-\alpha/3}$ error. Thus, after enlarging~$\mathcal Z$ by a deterministic factor, the parabolic estimate applies uniformly on every~$B_\rho$ with~$\rho\in[\nf12,1]$. The parameter~$\beta$ enters only through this invocation of Lemma~\ref{l.homogenization.parabolic.estimate}. Define~$h \coloneqq  C ( t_0^{\nf12} + \left| \log \ep \right| ^{-\nf{\alpha}{6}} )$, where~$C$ is a large constant, depending only on~$(\alpha,\beta, \cstar, \nu,\nondegconst, d)$. Taking~$C$ larger if necessary gives
\begin{equation*}
	t_0 \leq c h^2
	\qand
	h^{-2} \left| \log \ep \right| ^{-\nf{\alpha}{3}} \leq c
	\,,
\end{equation*}
where~$c>0$ is the small constant needed below. By restricting both~$\ep$ and~$t_0$ to be sufficiently small, without loss of generality, we may assume that~$h\in (0,\nicefrac18)$. The complementary regime~$h\geq\nicefrac18$ is trivial, after decreasing the constant~$c$ in~\eqref{e.exit.time.crushed}, since the right side of~\eqref{e.exit.time.crushed} is then at least one.

\smallskip

For each~$r \in (\nicefrac12, 1-2h)$ we select a smooth function~$\zeta_r \in C^\infty(\overline{B}_{r+h})$ satisfying
\begin{equation*}
	\begin{aligned}
		&0 \leq \zeta_r \leq 1 \,,\quad
		\zeta_r \equiv 1 \ \mbox{on} \ B_{r+h}\setminus B_{r+\frac34 h} \,,\quad
		\zeta_r \equiv 0 \ \mbox{on} \ B_{r+h/2} \,,
		\\
		&h\| \nabla  \zeta_r \|_{L^\infty(B_{r+h})}
		+h^2 \| \nabla^2  \zeta_r \|_{L^\infty(B_{r+h})} \leq 10\,.
	\end{aligned}
\end{equation*}
Set
\begin{equation*}
	M_r\coloneqq
	\| p^\ep_{B_1} \|_{L^\infty([0,t_0]\times B_{r+h})}\,.
\end{equation*}
For each such~$r$, we let~$u^\ep_r$ and~$u_r$ be the solutions of the problems
\begin{equation}
	\label{e.exit.time.approx.ep}
	\left\{
	\begin{aligned}
		& \partial_t u^\ep_r  -\Lop^\ep u^\ep_r =  0
		\quad \mbox{and} \quad
		\partial_t u_r  - \tfrac12\Delta u_r =  0
		& \mbox{in} & \ (0,t_0) \times  B_{r+h} \,, \\
		& u^\ep_r = u_r = \zeta_r M_r & \mbox{on} & \ \partial_{\mathrm{par}} ( (0,t_0) \times  B_{r+h})  \,.
	\end{aligned}
	\right.
\end{equation}
The functions~$u^\ep_r-M_r$ and~$u_r-M_r$ have zero lateral boundary data and common initial datum~$(\zeta_r-1)M_r$, which is compactly supported in~$B_{r+h}$. By Lemma~\ref{l.homogenization.parabolic.estimate}, we have
\begin{align*}
	\bigl\| u^\ep_r - u_r  \bigr\|_{L^\infty((0,t_0]\times B_{r+h})}
	&
	\leq
	C(1+t_0) \left| \log \ep \right| ^{-\nf{\alpha}{3}}
	M_r
	\| \nabla^2 \zeta_r \|_{L^\infty(B_{r+h})}
	\notag \\ &
	\leq
	Ch^{-2} \left| \log \ep \right| ^{-\nf{\alpha}{3}}
	M_r
	\notag \\ &
	\leq
	\frac14 M_r
	\,,
\end{align*}
where the last line uses the second inequality in the display above. By the standard heat-kernel propagation estimate, the initial and boundary data of~$u_r$ are supported at distance at least~$h/4$ from~$B_r$, and therefore, by the first inequality in the display above, we have
\begin{equation*}
	\| u_r \|_{L^\infty((0,t_0] \times B_{r})}
	\leq
	\frac14
	\| u_r \|_{L^\infty((0,t_0] \times B_{r+h})}
	\leq
	\frac14 M_r
	\,.
\end{equation*}
The maximum principle implies that~$u^\ep_r \geq p_{B_1}^\ep$ in~$B_{r+h}$, since the parabolic boundary data of~$u^\ep_r$ dominate those of~$p^\ep_{B_1}$ restricted to~$B_{r+h}$. We therefore deduce, using the previous displays and the triangle inequality, that
\begin{align*}
	\| p^\ep_{B_1} \|_{L^\infty([0,t_0]\times B_{r})}
	\leq
	\| u^\ep_r \|_{L^\infty((0,t_0] \times B_{r})}
	&
	\leq
	\bigl\| u^\ep_r - u_r  \bigr\|_{L^\infty((0,t_0]\times B_{r+h})}
	+
	\| u_r \|_{L^\infty((0,t_0] \times B_{r})}
	\notag \\ &
	\leq
	\frac 12 \| p^\ep_{B_1} \|_{L^\infty([0,t_0]\times B_{r+h})}\,.
\end{align*}
Iterating over radii separated by~$h$ from~$\nicefrac12$ to~$1-2h$ gives at least~$ch^{-1}$ steps and yields
\begin{equation*}
	\| p^\ep_{B_1} \|_{L^\infty([0,t_0]\times B_{\nicefrac12})}
	\leq
	\exp( - ch^{-1} )
	\leq
	\exp \bigl( -c \min \bigl\{ t_0^{-\nf12} , \left| \log \ep \right| ^{\nf{\alpha}{6}} \bigr\}  \bigr)
	\,.
\end{equation*}
This completes the proof.
\end{proof}

\subsection{Asymptotics for the diffusivity}
\label{ss.Dt}
We next use Theorem~\ref{t.superdiffusivity} and Lemma~\ref{l.exit.time.estimate} to obtain the quenched estimate~\eqref{e.Dt.exp} on the large-time behavior of the first and second moments of the process~$\{ X_t\}$ stated in Theorem~\ref{t.A}. The annealed estimate~\eqref{e.annealed.Dt} will then be obtained from both quenched moment estimates, using the crude Nash--Aronson-type upper bounds proved in Appendix~\ref{app.Feller}.

\begin{proof}[Proof of~\eqref{e.Dt.exp}]
Let~$\alpha,\beta \in (0,1)$ with~$\beta +2\alpha<1$ and~$\mathcal{Z}  \coloneqq  \X \vee \mathcal{K}_2$, where~$\X$ is the maximum of the minimal scales needed in Theorem~\ref{t.superdiffusivity} with~$U=B_1$ and in Lemma~\ref{l.exit.time.estimate}, and~$\mathcal{K}_2$ is the minimal scale in the statement of Lemma~\ref{l.ellip.k.scales.estimates}. By~\eqref{e.Z.integrability} for the homogenization and exit-time scales, by~\eqref{e.mathcal.K.int} for~$\mathcal K_2$, and by Lemma~\ref{l.Gamma.sigma.triangle}, this maximum has the same stretched-log tail, up to constants. Fix a time~$t_0>1$ with~$\sqrt{t_0} \geq \mathcal{Z}$ and define a length scale~$r[t_0]$ by
\begin{equation}
	\label{e.rtnaught.def}
	r[t_0] \coloneqq  \bigl( t_0 ( \log t_0)^{\nf{1+\alpha}{2}} \bigr)^{\frac12}\,.
\end{equation}
We will apply Theorem~\ref{t.superdiffusivity} and Lemma~\ref{l.exit.time.estimate} with~$\ep = r[t_0]^{-1}$. We note that, with these choices, we have that~$\ep^{-1} \geq\mathcal{Z}$ and~$\tau_\ep$ defined in~\eqref{e.tau.ep.def} satisfies
\begin{equation*}
	\tau_\ep = \ep^{-2} ( 8 \cstar \left| \log \ep \right|  )^{-\nf12}
	\geq c t_0 ( \log t_0)^{\nf{\alpha}{2}} \,.
\end{equation*}
Indeed,~$\ep^{-2}=t_0(\log t_0)^{\nf{1+\alpha}{2}}$ and~$|\log\ep|\simeq\log t_0$. After increasing the deterministic threshold, we may assume that~$C(\log t_0)^{-\nf{\alpha}{2}}\leq1$ in the application below; the bounded range of smaller~$t_0$ is absorbed into the constants. An application of Lemma~\ref{l.exit.time.estimate} implies that~$T_{B_{r[t_0]}} > t_0$ with high probability; we have
\begin{equation}
	\label{e.sorry.exit.closed}
	\mathbf{P}^0 \bigl[ T_{B_{r[t_0]}} \leq t_0\bigr]
	=
	\mathbf{P}^0 \bigl[ T^\ep_{B_{1}} \leq \tau_{\ep}^{-1} t_0 \bigr]
	\leq
	\mathbf{P}^0 \bigl[ T^\ep_{B_{1}} \leq C (\log t_0)^{-\nf{\alpha}{2}} \bigr]
	\leq
	2 \exp \bigl( -c (\log t_0)^{\nf{\alpha}{6}} \bigr)
	\,.
\end{equation}
We begin with the first moment, whose proof is the affine version of the argument below for the second moment. For~$e \in \partial B_1$, set~$\ell_e(x) \coloneqq e\cdot x$ and let~$v_e^\ep$ be the solution of
\begin{equation}
	\label{e.DP.for.affine.uep.sec9}
	\left\{
	\begin{aligned}
		& -\Lop^\ep v_e^\ep  =  0 & \mbox{in} & \ B_1 \,, \\
		& v_e^\ep = \ell_e & \mbox{on} & \ \partial B_1 \,.
	\end{aligned}
	\right.
\end{equation}
Since~$\ell_e$ is harmonic, it is the solution of the corresponding homogenized problem. Since~$\ep^{-1} \geq \mathcal{Z}$, the estimate~\eqref{e.homogenization.error} yields, uniformly in~$e \in \partial B_1$,
\begin{equation}
	\label{e.error.affine.est.appl}
	\bigl\| v_e^\ep - \ell_e\bigr\|_{{L}^{\infty}(B_1)}
	\leq
	C\left| \log \ep \right| ^{-\alpha}
	\,.
\end{equation}
The generator computation with zero right side gives
\begin{equation*}
	\partial_t \mathbf{E}^0 \bigl[ v_e^\ep (X_{t\wedge T^\ep_{B_1}}^\ep) \bigr]
	= \mathbf{E}^0 \bigl[ \Lop^\ep v_e^\ep (X_{t\wedge T^\ep_{B_1}}^\ep) \indc_{\{T^\ep_{B_1} >t\}}\bigr]
	=0\,.
\end{equation*}
Therefore, using~$\ell_e(0)=0$ and~\eqref{e.error.affine.est.appl},
\begin{align}
	\label{e.mean.with.stopping.rescaled}
	\bigl| \mathbf{E}^0 \bigl[ e\cdot X_{t\wedge T^\ep_{B_1}}^\ep \bigr] \bigr|
	&\leq
	\bigl| \mathbf{E}^0 \bigl[ v_e^\ep(X_{t\wedge T^\ep_{B_1}}^\ep) \bigr] \bigr|
	+ \bigl\| v_e^\ep-\ell_e\bigr\|_{{L}^{\infty}(B_1)}
	\notag \\ &
	=
	|v_e^\ep(0)|
	+ \bigl\| v_e^\ep-\ell_e\bigr\|_{{L}^{\infty}(B_1)}
	\leq
	C\left| \log \ep \right| ^{-\alpha}
	\,.
\end{align}
Taking the supremum over~$e\in\partial B_1$ and rescaling~\eqref{e.mean.with.stopping.rescaled}, with~$t \coloneqq \tau_\ep^{-1}t_0$, using~\eqref{e.tau.ep.relations},~\eqref{e.rtnaught.def} and~$\ep=r[t_0]^{-1}$, we get
\begin{equation}
	\label{e.mean.with.stopping}
	\bigl| \mathbf{E}^0 \bigl[ X_{t_0 \wedge T_{B_{r[t_0]}}} \bigr] \bigr|^2
	\leq
	C\ep^{-2}\left| \log \ep \right| ^{-2\alpha}
	\leq
	C t_0 \left| \log t_0 \right| ^{\nf12-\nf{3\alpha}{2}}
	\,.
\end{equation}
We will show next that
\begin{equation}
	\label{e.diffusivity.with.stopping}
	\Bigl| \frac1d\mathbf{E}^0 \Bigl[ \bigl| X_{t_0 \wedge T_{B_{r[t_0]}} } \bigr|^2 \Bigr]- 2\cstar^{\nf12} (\log t_0 )^{\nf12} t_0 \Bigr|
	\leq
	C (\log t_0) ^{\nf12-\nf{\alpha}{2}} t_0
	\,.
\end{equation}
Consider the solution~$u^\ep$ of the problem
\begin{equation}
	\label{e.DP.for.uep.sec9}
	\left\{
	\begin{aligned}
		& -\Lop^\ep u^\ep  =  -1 & \mbox{in} & \ B_1 \,, \\
		& u^\ep = Q & \mbox{on} & \ \partial B_1 \,,
	\end{aligned}
	\right.
\end{equation}
where the boundary data is the quadratic function~$Q(x) = \frac1{d} |x|^2$\,. Since~$-\tfrac12\Delta Q = -1$, the solution of the corresponding homogenized problem is~$\uhom = Q$. Equivalently, the heterogeneous equation is~$-2\Lop^\ep u^\ep=-2$, matching the normalization of Theorem~\ref{t.superdiffusivity}. Since~$\ep^{-1} \geq \mathcal{Z}$, the estimate~\eqref{e.homogenization.error} from Theorem~\ref{t.superdiffusivity} yields
\begin{equation}
	\label{e.error.est.appl}
	\bigl\| u^\ep - Q\bigr\|_{{L}^{\infty}(B_1)}
	\leq
	C\left| \log \ep \right| ^{-\alpha}
	\,.
\end{equation}
We next compute that
\begin{equation*}
	\partial_t \mathbf{E}^0 \bigl[ u^\ep (X_{t\wedge T^\ep_{B_1}}^\ep) \bigr]
	= \mathbf{E}^0 \bigl[ \Lop^\ep u^\ep (X_{t\wedge T^\ep_{B_1}}^\ep) \indc_{\{T^\ep_{B_1} >t\}}\bigr]
	= \mathbf{P}^0  \bigl[ T^\ep_{B_1} >t\bigr]\,.
\end{equation*}
Integrating this in~$t$, we deduce that
\begin{equation}
	\label{e.apply.generator}
	\mathbf{E}^0 \bigl[ u^\ep (X_{t\wedge T^\ep_{B_1}}^\ep) \bigr]
	=
	u^\ep (0) + \int_0^{t}  \mathbf{P}^0  \bigl[ T^\ep_{B_1} >s\bigr] \,ds
	=
	u^\ep (0) + t - \int_0^{t}  \mathbf{P}^0  \bigl[ T^\ep_{B_1} \leq s\bigr] \,ds
	\,.
\end{equation}
Using the triangle inequality and~\eqref{e.error.est.appl}, we obtain
\begin{align}
	\label{e.shield.generator}
	\Bigl| \frac1{d}\mathbf{E}^0 \bigl[  \bigl| X_{t\wedge T^\ep_{B_1}}^\ep\bigr|^2 \bigr]
	-t \Bigr|
	& =
	\Bigl| \mathbf{E}^0 \bigl[  Q( X_{t\wedge T^\ep_{B_1}}^\ep ) \bigr]
	-t \Bigr|
	\notag \\ &
	\leq
	\Bigl| \mathbf{E}^0 \bigl[  u^\ep ( X_{t\wedge T^\ep_{B_1}}^\ep ) \bigr]
	-t \Bigr|
	+ \bigl\| u^\ep -Q \bigr\|_{{L}^{\infty}(B_1)}
	\notag \\ &
	\leq
	2 \bigl\| u^\ep -Q \bigr\|_{{L}^{\infty}(B_1)}+ t \mathbf{P}^0  \bigl[ T^\ep_{B_1} \leq t\bigr]
	\leq
	C\left| \log \ep \right| ^{-\alpha}
	+ t \mathbf{P}^0  \bigl[ T^\ep_{B_1} \leq t\bigr]
	\,.
\end{align}
Rescale the estimate and set~$t\coloneqq\tau_\ep^{-1}t_0$. Substituting~\eqref{e.sorry.exit.closed},~\eqref{e.tau.ep.relations},~\eqref{e.rtnaught.def}, and~$\ep=r[t_0]^{-1}$ gives
\begin{equation*}
	\Bigl| \frac1d\mathbf{E}^0 \Bigl[ \bigl| X_{t_0 \wedge T_{B_{r[t_0]}} } \bigr|^2 \Bigr]- (8\cstar\left| \log \ep \right|  )^{\nf12} t_0 \Bigr|
	\leq
	C t_0 \left| \log t_0 \right| ^{\nf{1-\alpha}{2}}
	+C t_0 \left|\log t_0 \right|^{-1000}
	\leq
	C t_0 \left| \log t_0 \right| ^{\nf{1-\alpha}{2}}
	\,.
\end{equation*}
Using that~$2\left| \log \ep \right|  = \left|\log \ep^2\right| = \log t_0 + O(\log \log t_0)$, we obtain~\eqref{e.diffusivity.with.stopping}.

\smallskip

To conclude, we need to remove the stopping time from~\eqref{e.diffusivity.with.stopping}. According to~\eqref{e.a.ellipticity.pointwise.applied}, Proposition~\ref{p.Nash.Aronson.with.log} and~\eqref{e.nth.moment}, there exists~$M(d)<\infty$ such that, for each~$t$, we have the crude estimate
\begin{equation*}
	\mathbf{E}^0 \bigl[ \bigl| X_t \bigr|^4 \bigr]^{\frac12}
	\leq
	C t \bigl(\log (t \vee \mathcal{K}_2^2 ) \bigr)^{M(d)}\,.
\end{equation*}
Therefore, using~\eqref{e.sorry.exit.closed} and~$\sqrt{t_0} \geq \mathcal{Z} \geq \mathcal{K}_2$, we find that
\begin{equation*}
	\mathbf{E}^0 \bigl[ \bigl| X_{t_0} \bigr|^2 \indc_{T_{B_{r[t_0]}} \leq t_0 } \bigr]
	\leq
	C t_0 (\log t_0)^{-1000}\,.
\end{equation*}
The stopped term on the exit event is even smaller: by~\eqref{e.rtnaught.def} and~\eqref{e.sorry.exit.closed},
\begin{equation*}
	\mathbf{E}^0 \bigl[ \bigl| X_{T_{B_{r[t_0]}}} \bigr|^2 \indc_{T_{B_{r[t_0]}} \leq t_0 } \bigr]
	\leq
	r[t_0]^2\mathbf{P}^0 \bigl[ T_{B_{r[t_0]}} \leq t_0\bigr]
	\leq
	C t_0(\log t_0)^{-1000}\,.
\end{equation*}
The same estimates also remove the stopping time from the first-moment estimate. Indeed, using~\eqref{e.sorry.exit.closed},~\eqref{e.rtnaught.def} and the crude moment bound above, we have
\begin{align}
	\label{e.remove.stop.mean.diffusivity}
	\bigl| \mathbf{E}^0[X_{t_0}]
	-\mathbf{E}^0[X_{t_0 \wedge T_{B_{r[t_0]}}}] \bigr|
	&\leq
	\mathbf{E}^0 \bigl[ |X_{t_0}| \indc_{T_{B_{r[t_0]}} \leq t_0 } \bigr]
	+
	\mathbf{E}^0 \bigl[ |X_{T_{B_{r[t_0]}}}| \indc_{T_{B_{r[t_0]}} \leq t_0 } \bigr]
	\notag \\ &
	\leq
	C t_0^{\nf12}(\log t_0)^{-1000}\,.
\end{align}
Combining~\eqref{e.remove.stop.mean.diffusivity} with~\eqref{e.mean.with.stopping} gives
\begin{equation}
	\label{e.mean.final.good.event}
	\bigl| \mathbf{E}^0[X_{t_0}] \bigr|^2
	\leq
	C t_0 \left| \log t_0 \right| ^{\nf12-\nf{3\alpha}{2}}
	\,.
\end{equation}
Combining the second-moment estimates with~\eqref{e.diffusivity.with.stopping} and using again that~$2\left| \log \ep \right|=\log t_0+O(\log\log t_0)$, we obtain
\begin{equation*}
	\Bigl| \frac1d\mathbf{E}^0 \Bigl[ \bigl| X_{t_0} \bigr|^2 \Bigr]- 2\cstar^{\nf12}(\log t_0)^{\nf12} t_0 \Bigr|
	\leq
	C t_0 \left| \log t_0 \right| ^{\nf{1-\alpha}{2}}
	\,.
\end{equation*}
The previous inequalities are valid on the event~$\{ \mathcal{Z} \leq \sqrt{t_0} \}$, and therefore by~\eqref{e.Z.integrability} we obtain
\begin{equation*}
	\P \biggl[ \Bigl| \frac1d\mathbf{E}^0 \Bigl[ \bigl| X_{t_0} \bigr|^2 \Bigr]- 2\cstar^{\nf12}(\log t_0)^{\nf12} t_0 \Bigr|
	>
	C t_0 \left| \log t_0 \right| ^{\nf{1-\alpha}{2}}  \biggr]
	\leq
	\P \bigl[  \mathcal{Z} > \sqrt{t_0}\, \bigr]
	\leq C \exp \bigl( - c( \log t_0)^{\beta} \bigr)
	\,.
\end{equation*}
The same argument, applied to~\eqref{e.mean.final.good.event} gives the corresponding probability bounds for the first moment. This proves~\eqref{e.Dt.exp} for~$\delta = \frac14(1-2\alpha)$. The proof is now complete.
\end{proof}

\begin{proof}[{Proof of~\eqref{e.annealed.Dt}}]
Fix~$p<\infty$. Let~$B(t)$ be the event defined by
\begin{equation*}
	B(t) \coloneqq  \Bigl\{
	\bigl| t^{-1}
	\mathbf{E}^0 \bigl[ \bigl| X_t \bigr|^2 \bigr] -
	2d\cstar^{\nf12} (\log t)^{\nf12}
	\bigr|
	+
	t^{-1}\bigl|\mathbf E^0[X_t]\bigr|^2
	>
	C \bigl( \log t \bigr)^{\nf14+\delta}
	\Bigr\} \,.
\end{equation*}
According to~\eqref{e.Dt.exp}, for every finite exponent~$N<\infty$ we have that~$\P [ B(t) ] \leq C_N (\log t)^{-N}$. According to Proposition~\ref{p.Nash.Aronson.with.log}, the pointwise ellipticity-scale estimate~\eqref{e.a.ellipticity.pointwise.applied}, the moment bound~\eqref{e.nth.moment}, and~\eqref{e.mathcal.K.int} for~$\mathcal K_2$, we have that, for some~$M(d)<\infty$,
\begin{equation*}
	\E \Bigl[ \mathbf{E}^0 \bigl[ \bigl| X_t \bigr|^2 \bigr]^{2p} \Bigr] ^{\nicefrac1{2p}}
	\leq
	C_p t (\log t)^M \,.
\end{equation*}
We deduce that
\begin{align*}
	\E \biggl[ \Bigl| \frac1t
	\mathbf{E}^0 \bigl[ \bigl| X_t \bigr|^2 \bigr] -
	2d\cstar^{\nf12} (\log t)^{\nf12}
	\Bigr|^p  \biggr]^{\nicefrac1p}
	&
	\leq
	C \bigl( \log t \bigr)^{\nf14+\delta}
	+
	\E \Bigl[ \Bigl| \frac1t
	\mathbf{E}^0 \bigl[ \bigl| X_t \bigr|^2 \bigr] -
	2d\cstar^{\nf12} (\log t)^{\nf12}
	\Bigr|^p \indc_{B(t)}   \Bigr]^{\nicefrac1p}
	\notag \\ &
	\leq
	C \bigl( \log t \bigr)^{\nf14+\delta}
	+
	C_{p} (\log t)^M \P \bigl[ B(t) \bigr]^{\nicefrac1{2p}}
	\notag \\ &
	\leq
	C \bigl( \log t \bigr)^{\nf14+\delta}
	+ C_{p,N} ( \log t)^{M - \frac{N}{2p}}\,.
\end{align*}
It remains to estimate the quenched mean. Put~$Y_t\coloneqq t^{-1}|\mathbf E^0[X_t]|^2$. On~$B(t)^c$, we have~$Y_t\leq C(\log t)^{\nf14+\delta}$. Jensen's inequality and the preceding crude moment bound give
\begin{equation*}
	\E[Y_t^{2p}]^{\nf1{2p}}
	\leq
	t^{-1}
	\E\bigl[\mathbf E^0[|X_t|^2]^{2p}\bigr]^{\nf1{2p}}
	\leq
	C_p(\log t)^M\,.
\end{equation*}
Consequently, H\"older's inequality yields
\begin{equation*}
	\E\bigl[Y_t^p\indc_{B(t)}\bigr]^{\nf1p}
	\leq
	C_p(\log t)^M\P[B(t)]^{\nf1{2p}}\,.
\end{equation*}
Taking~$N$ sufficiently large in both estimates completes the proof of~\eqref{e.annealed.Dt}.
\end{proof}

\subsection{Proof of the invariance principle}
We use the generator-convergence criterion for Feller processes to prove~\eqref{e.convinlaw.again}. Recall from~\eqref{e.A.ep} that the infinitesimal generator~$\Lop^\ep$ of~$X^\ep_t$ is given by
\begin{equation}
	\label{e.A.ep.recall}
	\Lop^\ep =
	\frac12
	\bigl( 2 \cstar \left| \log \ep \right|  \bigr)^{-\nf12}
	\nabla
	\cdot \bigl( \nu\Id + \k \bigl( \nicefrac \cdot \ep \bigr) \bigr) \nabla
	\,.
\end{equation}
The rescaled processes are Feller by Corollary~\ref{c.yes.Feller} and deterministic space-time scaling. Moreover,~$C_c^\infty(\Rd)$ is a standard core for the Brownian generator~$\frac12\Delta$ on~$C_0(\Rd)$. Therefore, by~\cite[Theorem 19.25]{Kallenberg}, it is enough to prove the following generator approximation: for every~$u\in C_c^\infty(\Rd)$, construct~$u^\ep$ in the domain of~$\Lop^\ep$ such that~$u^\ep\to u$ and~$\Lop^\ep u^\ep\to\frac12\Delta u$ uniformly.

\begin{proposition}[Convergence of the generators]
\label{p.generators}
For~$\P$-almost every realization of the field~$\k(\cdot)$ and every~$u \in C^\infty_c(\Rd)$, there exists a sequence~$\{ u^\ep \} \subseteq C^2(\Rd) \cap C_0(\Rd)$ with~$\Lop^\ep u^\ep\in C_0(\Rd)$ such that
\begin{equation}
	\label{e.homogenization.giveth}
	\lim_{\ep \to 0}
	\| u^\ep - u \|_{L^\infty(\Rd)} = 0
\end{equation}
and
\begin{equation}
	\label{e.RHS.converge}
	\lim_{\ep \to 0}
	\| \Lop^\ep u^\ep - \tfrac12 \Delta u \|_{L^\infty(\Rd)} = 0 \,.
\end{equation}
\end{proposition}

Proposition~\ref{p.generators} is a qualitative whole-space homogenization statement for the generator. For a fixed~$u\in C^\infty_c(\Rd)$, the desired approximation is obtained by solving
\begin{equation}
	\label{e.wholespace}
	-\Lop^\ep u^\ep = -\tfrac12\Delta u
	\quad \text{in} \ \Rd\,.
\end{equation}
For each~$\ep>0$, the coefficients of~$\Lop^\ep$ belong to~$C^{1,1}_{\mathrm{loc}}(\Rd)$ by assumption~\ref{a.j.reg}, and therefore the solution of~\eqref{e.wholespace}, once constructed, belongs locally to~$C^2$ and satisfies~\eqref{e.RHS.converge}. It remains to construct the whole-space solution as the limit of Dirichlet problems in large balls and then prove~\eqref{e.homogenization.giveth}.

\smallskip

The following lemma provides the passage from finite to infinite volume. Since we use this lemma qualitatively we did not attempt to optimize the stochastic prefactor appearing on the right in~\eqref{e.decay.estimate.Linfty.no.avgs}.

\begin{lemma}[Decay estimate]
\label{l.decay.estimate.Linfty}
For every~$\gamma, \sigma \in (0,1)$, there exists a constant~$C(\gamma, \sigma, \cstar,\nu,\nondegconst,d)<\infty$ and a random variable~$\X$ satisfying
\begin{equation*}
	\log \X = \O_{\Gamma_\sigma} (C)
\end{equation*}
such that, after enlarging~$\X$ by a deterministic constant depending only on the displayed parameters, for every~$\X \leq r \leq R$ and~$\mathbf{f} \in C^\infty_c(B_1;\Rd)$, every parameter~$\ep^{-1} > \X$ and solution~$u\in H^1_0(B_R)$ of the Dirichlet problem
	\begin{equation}
		\label{e.decay.estimate.DP.instatement}
		\left\{
		\begin{aligned}
			& -\Lop^\ep u  = \nabla \cdot \f & \mbox{in} & \ B_R\,, \\
			& u = 0 & \mbox{on} & \ \partial B_R\,.
		\end{aligned}
		\right.
	\end{equation}
		we have the estimate
\begin{equation}
	\label{e.decay.estimate.Linfty.no.avgs}
	\| u \|_{{L}^\infty(B_R \setminus B_{r})}
	\leq
	C \X^{d+1+\gamma}
	r^{2-\gamma-d}  \| \nabla\cdot\f \|_{L^2(B_1)}
	\,.
\end{equation}
\end{lemma}
The estimate is stated in terms of~$\| \nabla\cdot\f \|_{L^2(B_1)}$ because in the application below~$\f=-\frac12\nabla u$ with~$u\in C_c^\infty(B_1)$.

\smallskip

The proof of Lemma~\ref{l.decay.estimate.Linfty} requires the following consequences of the analysis in Section~\ref{s.Dirichlet}. The scale~$r$ Poincar\'e estimate stated informally in the introduction is obtained from the next lemma by applying it to~$U_0=B_1$ after rescaling~$B_r$ to the unit ball.

\begin{lemma}[Superdiffusive Poincar\'e with RHS]
\label{l.superdiffusive.Poincare.with.rhs}
Let~$U_0\subseteq\Rd$ be a ball centered at the origin. For every~$\sigma \in (0,1)$, there exist constants~$C(U_0,d)<\infty$,~$\hat L(U_0,\sigma,\cstar,\nu,\nondegconst,d)<\infty$, and a random variable~$\X$ satisfying
\begin{equation*}
	\log \X = \O_{\Gamma_\sigma} (\hat L)\,,
\end{equation*}
such that, for every~$\ep^{-1} > \X$ and~$u\in H^1(U_0)$ satisfying
\begin{equation*}
	-\Lop^\ep u = f \quad \mbox{in} \ U_0 \,,
\end{equation*}
with~$f\in L^2(U_0)$, we have the estimate
\begin{equation}
	\label{e.superdiff.poincare.with.rhs}
	\| u - (u)_{U_0} \|_{L^2(U_0)}
	\leq
	C\cstar^{-\nf14}\nu^{\nf12}
	\left| \log \ep \right| ^{-\nf14} \| \nabla u \|_{{L}^2(U_0)}
	+
	C \left| \log \ep \right|^{-100}\| f \|_{{L}^{2}(U_0)}
	\,.
\end{equation}
\end{lemma}
\begin{proof}
Choose~$\rho_{\rm loc}\in(\sigma,1)$ and~$a=a(U_0)\geq1$ such that~$U_0\subseteq a\cu_0$. The constants in Lemma~\ref{l.Dirichlet.Whitney.Poincare} are uniform over the compact family~$\{\lambda U_0:\lambda\in[\frac1{3a},\frac1a]\}$. Choose~$M_{\rm loc}=M_{\rm loc}(U_0,d)\geq C(d)$ sufficiently large that~$M_{\rm loc}\log3>101$ and that the fixed-size cubical cores of these balls satisfy the Poincar\'e assumption~\eqref{e.Dir.new.Whitney.Poincare.assumption} at every admissible scale, with a uniform constant~$D_{\rm loc}=D_{\rm loc}(U_0,d)$. Apply Lemma~\ref{l.sharp.scale.inputs} with~$(\varepsilon,\rho,M)=(1,\rho_{\rm loc},M_{\rm loc})$, and let~$L_\sharp$ be its deterministic threshold.

\smallskip

Let~$\hat L_{\rm P}\geq L_\sharp$ denote the enlarged threshold in Lemma~\ref{l.Dirichlet.Whitney.Poincare} for~$D_{\rm loc}$. Its random variable is independent of~$\lambda$ once these constants are fixed; applying the lemma with this uniform data gives~$\X$ satisfying~$\log\X=\O_{\Gamma_{\rho_{\rm loc}}}(\hat L_{\rm P})$. Since~$\rho_{\rm loc}>\sigma$, enlarging the constant gives the asserted~$\Gamma_\sigma$ bound, with the threshold~$\hat L$ in the statement chosen to dominate~$\hat L_{\rm P}$.

Fix~$\ep^{-1}>\X$. Put~$R\coloneqq\ep^{-1}$,~$m\coloneqq\lceil\log_3(aR)\rceil$ and~$\lambda\coloneqq R3^{-m}\in[\frac1{3a},\frac1a]$. Then~$RU_0=3^m(\lambda U_0)$. For~$u_R(y)\coloneqq u(\ep y)$, the equation becomes
\begin{equation*}
	-\nabla\cdot\a\nabla u_R=F_R
	\quad\mbox{in }RU_0\,,
\end{equation*}
with
\begin{equation*}
	\|F_R\|_{\underline{L}^{2}(RU_0)}
	\leq
	C\cstar^{\nf12}\ep^2|\log\ep|^{\nf12}
	\|f\|_{\underline L^2(U_0)}\,.
\end{equation*}
After enlarging~$\X$ deterministically,~$3^m\geq\X$ and~$m\geq\hat L_{\rm P}$, and therefore Lemma~\ref{l.Dirichlet.Whitney.Poincare} gives
\begin{equation*}
	\|u_R-(u_R)_{RU_0}\|_{\underline L^2(RU_0)}
	\leq
	C3^m\shom_m^{-\nf12}\nu^{\nf12}
	\|\nabla u_R\|_{\underline L^2(RU_0)}
	+
	C3^{2m-h_m}\nu^{-1}\|F_R\|_{\underline L^2(RU_0)}
	\,.
\end{equation*}
Rescaling back and using~$3^m\simeq_{U_0}\ep^{-1}$ gives
\begin{equation*}
	\|u-(u)_{U_0}\|_{\underline L^2(U_0)}
	\leq
	C\nu^{\nf12}\shom_m^{-\nf12}\|\nabla u\|_{\underline L^2(U_0)}
	+C\cstar^{\nf12}\nu^{-1}3^{-h_m}
	|\log\ep|^{\nf12}\|f\|_{\underline L^2(U_0)}\,.
\end{equation*}
The deterministic sharp bound~\eqref{e.sm.sharp.bounds} and the choice of~$M_{\rm loc}$ give
\begin{equation*}
	\nu^{\nf12}\shom_m^{-\nf12}
	\leq
	C\cstar^{-\nf14}\nu^{\nf12}
	|\log\ep|^{-\nf14}
	\qand
	\cstar^{\nf12}\nu^{-1}3^{-h_m}|\log\ep|^{\nf12}
	\leq
	C|\log\ep|^{-100}\,.
\end{equation*}
Since~$U_0$ is fixed, normalized and unnormalized~$L^2$ norms differ only by a constant depending on~$U_0$, and the desired estimate follows.
\end{proof}

\begin{lemma}[Global~$L^\infty$-to-$L^2$ estimate]
\label{l.inproof.Linfty.L2}
Let~$V\subseteq W\subseteq U\subseteq\Rd$ be smooth, bounded domains satisfying
\begin{equation*}
	\operatorname{dist}(\overline V\setminus\partial U,\partial W\setminus\partial U)>0\,.
\end{equation*}
Assume also that~$\partial W$ agrees with~$\partial U$ in a relative neighborhood of~$\overline V\cap\partial U$.
Let~$\sigma \in (0,1)$. There exists a constant~$C(U,V,W, \sigma, \cstar,\nu,\nondegconst,d)<\infty$ and a random variable~$\X$ satisfying
\begin{equation*}
	\log \X = \O_{\Gamma_\sigma} (C)\,,
\end{equation*}
such that, for every~$\ep^{-1} > \X$ and~$u\in H^1(U)$ satisfying
\begin{equation*}
	\left\{
	\begin{aligned}
		& -\Lop^\ep u = 0 & \mbox{in} & \ U\,, \\
		& u = 0 & \mbox{on} & \ (\partial U) \cap (\partial W)\,,
	\end{aligned}
	\right.
\end{equation*}
we have the estimate
\begin{equation}
	\label{e.Linfty.L2.estimate}
	\| u - (u)_{V} \|_{L^\infty(V)}
	\leq
	C \| u - (u)_W \|_{\underline{L}^2(W)}
	\,.
\end{equation}
\end{lemma}
\begin{proof}
Choose~$\rho\in(\sigma,1)$ and apply Lemma~\ref{l.sharp.scale.inputs} with~$\varepsilon=1$ and~$M$ large enough depending on the fixed geometry of the triple~$(U,V,W)$. Fix a finite atlas and a deterministic constant~$C_0=C_0(U,V,W,d)$, and set~$m=\lceil\log_3(C_0\ep^{-1})\rceil$. We apply Proposition~\ref{p.Dirichlet.C01.boundary.local} with~$f=g=0$ and the normalized domains~$3^{-m}\ep^{-1}U$,~$3^{-m}\ep^{-1}V$, and~$3^{-m}\ep^{-1}W$. After increasing~$C_0$, these domains lie in~$\cu_0$ and their geometries range over a compact family, so the constants are uniform in~$\ep$. At the top scale~$m$, put~$n=m-h_m$ and replace each rescaled chart center by a nearest point of~$3^n\Zd\cap\cu_m$. The displacement is bounded by~$C3^n$. Increasing~$M$ by a constant depending only on the fixed atlas makes~$3^n$ small relative to every rescaled chart radius, so bounded neighboring enlargements of the recentered charts cover the original charts. After increasing the deterministic threshold so that~$n\geq\hat L$, Lemma~\ref{l.sharp.scale.inputs} ensures, on one event, that Lemma~\ref{l.Dirichlet.interior.pointwise} and Proposition~\ref{p.Dirichlet.C01.boundary.local} are available on every translated chart and scale needed below whenever~$3^m\geq\X_0$. We may therefore choose~$\X$ with~$\log\X=\O_{\Gamma_\sigma}(C)$ such that these estimates hold for every~$\ep^{-1}>\X$. This one random scale is independent of~$\ep$.

The two geometric assumptions give a finite cover of~$\overline V$ by interior balls whose fixed enlargements lie in~$W$ and boundary half-balls whose flat boundary lies in~$(\partial U)\cap(\partial W)$. On an interior ball, rescale to the large cube of side comparable to~$\ep^{-1}$. Lemma~\ref{l.Dirichlet.interior.pointwise}, used with~$f=0$, gives the desired local~$L^\infty$ bound by the normalized~$L^2$ oscillation on a slightly larger ball. On a boundary chart, use Proposition~\ref{p.Dirichlet.C01.boundary.local}, followed on the bottom scale by the De Giorgi--Nash~$L^\infty$--$L^2$ estimate; the zero trace of~$u$ on~$(\partial U)\cap(\partial W)$ removes the boundary additive constant. This is the only place where~$\nondegconst$ enters the constant. Since there are only finitely many charts and all enlarged balls are contained in~$W$, the local estimates combine to give
\begin{equation*}
	\|u-(u)_V\|_{L^\infty(V)}
	\leq
	C\|u-(u)_W\|_{\underline{L}^2(W)}\,.
\end{equation*}
\end{proof}

\begin{proof}[{Proof of Lemma~\ref{l.decay.estimate.Linfty}}]
		Define the adjoint~$(\Lop^\ep)^*$ of~$\Lop^\ep$ by
	\begin{equation}
		(\Lop^\ep)^*  \coloneqq
		\frac12
		\bigl( 2 \cstar \left| \log \ep \right|  \bigr)^{-\nf12}
		\nabla
		\cdot \bigl( \nu\Id - \k \bigl( \nicefrac \cdot \ep \bigr) \bigr) \nabla
	\end{equation}
		Observe that~$(\Lop^\ep)^*$ has the same law as~$\Lop^\ep$ by the assumption of negation symmetry in~\ref{a.j.iso}. We let~$\X$ be the maximum of the random variables appearing in Lemmas~\ref{l.superdiffusive.Poincare.with.rhs},~\ref{l.inproof.Linfty.L2} and Theorem~\ref{t.large.scale.Holder}, together with the coefficient-bound scale from Lemma~\ref{l.ellip.k.scales.estimates}, all applied with~$-\k$ in place of~$\k$ when the adjoint operator is involved. By the symmetry assumption, this has the same stochastic integrability. We also denote~$\theta \coloneqq 9\sqrt{d}$. We retain the harmless prefactor~$\X^{d+1+\gamma}$ in the conclusion to absorb the finite-cover and chaining losses below. Whenever Lemma~\ref{l.superdiffusive.Poincare.with.rhs} is applied on a ball of radius~$s\geq\X$, we rescale the ball to unit size; the effective parameter is~$\ep/s$, so~$(\ep/s)^{-1}=s\ep^{-1}>\X$ after the deterministic enlargement of~$\X$, and the resulting logarithmic factor is bounded by the displayed~$|\log\ep|^{-\nf14}$. For the source contribution we use only the weaker bound~$Cs^2\|f\|_{L^2(B_s)}$, discarding the additional logarithmic gain.

\smallskip

Denote~$A_r  \coloneqq  B_{r} \setminus B_{r/2}$. Fix~$g \in L^2(A_r)$ with~$\int_{A_r}g =0$ and extend~$g$ to be zero outside of~$A_r$. Let~$v \in H^1_0(B_R)$ be the solution of the Dirichlet problem
	\begin{equation}
		\left\{
		\begin{aligned}
			& - ( \Lop^\ep)^* v = g & \mbox{in} & \ B_R\,, \\
			& v = 0 & \mbox{on} & \ \partial B_R\,.
		\end{aligned}
		\right.
	\end{equation}
Testing the equation for~$v$ with itself, and using~$\int g=0$, we obtain
	\begin{align*}
		\frac12
		\bigl( 2 \cstar \left| \log \ep \right|  \bigr)^{-\nf12} \nu  \| \nabla v \|_{L^2(B_R)}^2
		=\int_{B_R} g(v-c)
		\leq
		\| v-c \|_{L^2(\supp g)}
		\| g \|_{L^2(A_r)}
		\,,
	\end{align*}
with~$c$ to be chosen below. If~$\theta r\leq R$, then we take~$c=(v)_{B_r}$. Lemma~\ref{l.superdiffusive.Poincare.with.rhs}, applied in~$B_r$, gives
\begin{align*}
	\| v- (v)_{B_r} \|_{L^2(B_r)}
	&
	\leq
	Cr
	\bigl( 2 \cstar \left| \log \ep \right|  \bigr)^{-\nf14}   \| \nabla v \|_{L^2(B_r)}
	+
	C r^2 \| g \|_{L^2(B_r)}
	\notag \\ &
	\leq
	Cr
	\bigl( 2 \cstar \left| \log \ep \right|  \bigr)^{-\nf14}   \| \nabla v \|_{L^2(B_{R})}
	+
	C r^2 \| g \|_{L^2(A_r)}
	\,.
	\end{align*}
If~$\theta r>R$, we instead take~$c=(v)_{B_R}$, apply the same estimate in~$B_R$, and use~$R\leq\theta r$ and the zero boundary data of~$v$ to get the same last line. We deduce from the previous two displays and Young's inequality that
\begin{equation}
	\label{e.L2.bound.rg}
	\bigl( 2 \cstar \left| \log \ep \right|  \bigr)^{-\nf14}  \| \nabla v \|_{L^2(B_R)}
	\leq C r \| g \|_{L^2(A_r)} \,.
\end{equation}
		By testing the equation for~$u$ with~$v$ and the equation for~$v$ with~$u$, we obtain
	\begin{align*}
		\left|\int_{A_r}
		u g\right|
		& =
		\left| \int_{B_R} v\, \nabla\cdot\f \right|
		=
		\left| \int_{B_1} \bigl(v-(v)_{B_1}\bigr)\, \nabla\cdot\f \right|
		\\
		&\leq
		\| v-(v)_{B_1} \|_{L^2(B_1)} \| \nabla\cdot\f \|_{L^2(B_1)}
		\,.
	\end{align*}
Using the large-scale H\"older estimate~\eqref{e.large.scale.Holder}, after rescaling by~$\ep$, and observing that~$g$ vanishes in the ball~$B_{\nicefrac r2}$, we get
\begin{equation*}
	\|  v  - (v)_{B_1}\|_{L^2(B_1)}
	\leq
	C r^{-\gamma}
	\|  v  - (v)_{B_{r/(3\theta)}}\|_{\underline{L}^2(B_{r/(3\theta)})}\,.
\end{equation*}
By applying Lemma~\ref{l.superdiffusive.Poincare.with.rhs} in~$B_{r/(3\theta)}$, where~$g$ vanishes, we have
\begin{equation*}
	\|  v  - (v)_{B_{r/(3\theta)}} \|_{\underline{L}^2(B_{r/(3\theta)})}
	\leq
	Cr\bigl( 2 \cstar \left| \log \ep \right|  \bigr)^{-\nf14} \nu^{\nf12}
	\|  \nabla v \|_{\underline{L}^2(B_{r/(3\theta)})}
	\leq
	Cr\bigl( 2 \cstar \left| \log \ep \right|  \bigr)^{-\nf14} \nu^{\nf12}
	\|  \nabla v \|_{\underline{L}^2(B_{\nicefrac r3})}\,.
\end{equation*}
Here the last inequality uses~$B_{r/(3\theta)}\subseteq B_{\nicefrac r3}$ and absorbs the fixed volume-normalization factor into~$C$. Combining the previous two displays, we obtain
\begin{equation*}
	\|  v  - (v)_{B_1}\|_{L^2(B_1)}
	\leq
	C r^{1-\gamma-\nicefrac d2}
	\bigl( 2 \cstar \left| \log \ep \right|  \bigr)^{-\nf14} \nu^{\nf12}
	\|  \nabla v \|_{L^2(B_{\nicefrac r3})}
	\,.
\end{equation*}
By~\eqref{e.L2.bound.rg} and the previous display, we obtain
\begin{align*}
	\|  v  - (v)_{B_1}\|_{L^2(B_1)}
	&
	\leq
	C r^{1-\nicefrac d2-\gamma}
	\bigl( 2 \cstar \left| \log \ep \right|  \bigr)^{-\nf14}
	\|  \nabla v \|_{{L}^2(B_{R})}
	\leq
	C r^{2-\gamma-\nicefrac d2}
	\| g \|_{L^2(B_R)}
	\,.
\end{align*}
Combining the above inequalities, we get
		\begin{equation*}
			\left| \int_{A_r} ug \right| \leq
			C r^{2-\gamma-\nicefrac d2}  \| g \|_{L^2(B_R)}
			\| \nabla\cdot\f \|_{L^2(B_1)}\,.
		\end{equation*}
		Taking~$g = u - (u)_{A_r}$ yields
	\begin{equation*}
		\| u - (u)_{A_r} \|_{L^2(A_r)}
		\leq
		C r^{2-\gamma-\nicefrac{d}{2}}  \| \nabla\cdot\f \|_{L^2(B_1)}\,.
		\end{equation*}
Dividing by~$|A_r|^{\nf12} =C r^{\nicefrac d2}$ and using~$|A_r|^{-\nf12}  \| u - (u)_{A_r} \|_{L^2(A_r)} = \| u - (u)_{A_r} \|_{\underline{L}^2(A_r)}$, we get
	\begin{equation}
		\| u - (u)_{B_{r} \setminus B_{r/2}} \|_{\underline{L}^2(B_{r} \setminus B_{r/2})}
		\leq
		C r^{2-\gamma-d}  \| \nabla\cdot\f \|_{L^2(B_1)}\,.
		\end{equation}
In the case that~$r < \frac23 R$, we use the interior~$L^\infty$-$L^2$ estimate in Lemma~\ref{l.inproof.Linfty.L2} on the annular domains~$B_{\nicefrac{3r}2}\setminus B_{\nicefrac r4}$,~$B_r\setminus B_{\nicefrac r2}$ and~$B_{\nicefrac{4r}3}\setminus B_{\nicefrac r3}$ to get
\begin{equation}
	\label{e.decay.estimate.Linfty.r}
	\| u - (u)_{B_{r} \setminus B_{r/2}} \|_{{L}^\infty(B_{r} \setminus B_{r/2})}
	\leq
	C \| u - (u)_{B_{4r/3} \setminus B_{r/3}} \|_{\underline{L}^2(B_{4r/3} \setminus B_{r/3})}
	\leq
	C \X^{d+1+\gamma}
	r^{2-\gamma-d}  \| \nabla\cdot\f \|_{L^2(B_1)}
	\,.
\end{equation}
In the case~$r \geq \frac23 R$, the preceding dual argument applies verbatim to the comparable outer annulus~$B_R\setminus B_{R/4}$. We then use the full (global) version of Lemma~\ref{l.inproof.Linfty.L2}, with~$U= B_1\setminus B_{\nicefrac18}$,~$V=B_{1} \setminus B_{\nicefrac13}$ and~$W=B_{1} \setminus B_{1/4}$, and with~$\ep/R$ in place of~$\ep$, to obtain, after a rescaling,
\begin{equation*}
	\| u - (u)_{B_{R} \setminus B_{R/3}} \|_{{L}^\infty(B_{R} \setminus B_{R/3})}
	\leq
	C \| u - (u)_{B_{R} \setminus B_{R/4}} \|_{\underline{L}^2(B_{R} \setminus B_{R/4})}
	\leq
	C \X^{d+1+\gamma}
	R^{2-\gamma-d}  \| \nabla\cdot\f \|_{L^2(B_1)}
	\,.
\end{equation*}
Since~$u=0$ on~$\partial B_R$, we obtain by the triangle inequality and the previous display that
\begin{equation*}
	\| u \|_{{L}^\infty(B_{R} \setminus B_{R/3})}
	\leq
	C \X^{d+1+\gamma}
	R^{2-\gamma-d}  \| \nabla\cdot\f \|_{L^2(B_1)}
	\,.
\end{equation*}
Combining this with~\eqref{e.decay.estimate.Linfty.r} and using the triangle inequality and chaining together overlapping annuli, we obtain~\eqref{e.decay.estimate.Linfty.no.avgs}.
\end{proof}

\begin{proof}[{Proof of Proposition~\ref{p.generators}}]
Fix~$u \in C^\infty_c(\Rd)$. Let~$R_0 \geq 1$ be so large that~$u \in C^\infty_c(B_{R_0})$. By a deterministic change of variables, replacing~$\ep$ by~$\ep/R_0$ and absorbing the resulting bounded change in logarithmic factors into constants depending on~$u$, it is enough to prove the case~$R_0=1$. In what follows, we will allow the constants~$C$ to depend on~$u$ in addition to~$(\cstar,\nu,\nondegconst,d)$.

\smallskip

For each~$\ep >0$, let~$u^\ep_R$ be the solution of the Dirichlet problem
\begin{equation}
	\left\{
	\begin{aligned}
		& -\Lop^\ep u^\ep_R  = -\tfrac12\Delta u & \mbox{in} & \ B_R\,, \\
		& u^\ep_R = 0 & \mbox{on} & \ \partial B_R\,.
	\end{aligned}
	\right.
\end{equation}
This is the normalization of Theorem~\ref{t.superdiffusivity} with source~$f=-\Delta u$: equivalently,~$-2\Lop^\ep u^\ep_R=-\Delta u$, and the corresponding homogenized Dirichlet problem is solved by~$u$ since~$u=0$ on~$\partial B_R$ for~$R\geq1$. Set~$\mathbf f\coloneqq-\frac12\nabla u$, so that~$\nabla\cdot\mathbf f=-\frac12\Delta u$. We extend the domain of~$u^\ep_R$ to~$\Rd$ by defining~$u^\ep_R$ to be zero in~$\Rd \setminus B_R$. According to Lemma~\ref{l.decay.estimate.Linfty}, we have that, for every~$R\geq r \geq \X$,
\begin{equation*}
	\| u^\ep_R \|_{{L}^\infty(\Rd \setminus B_{r})}
	\leq
	C \X^{d+1+\gamma}
	r^{2-\gamma-d}  \| \nabla\cdot\mathbf f \|_{L^2(B_1)}
	\leq
	C \X^{d+1+\gamma}
	r^{2-\gamma-d}\,.
\end{equation*}
By the maximum principle, for each~$S\geq R\geq \X$, the difference~$u^\ep_R - u^\ep_{S}$ satisfies
\begin{equation*}
	\| u^\ep_R - u^\ep_{S} \|_{L^\infty(B_R)}
	\leq
	\| u^\ep_{S} \|_{L^\infty(\partial B_R)}
	\leq
	C \X^{d+1+\gamma}
	R^{2-\gamma-d}
	\,.
\end{equation*}
It follows that
\begin{equation*}
	\| u^\ep_R - u^\ep_{S} \|_{L^\infty(\Rd)}
	\leq  C \X^{d+1+\gamma}
	R^{2-\gamma-d}
	\,.
\end{equation*}
Since~$d\geq2$ and~$2-\gamma-d \leq -\gamma <0$, we deduce that~$u^\ep_R$ converges as~$R\to \infty$ to a function~$u^\ep$ which solves
\begin{equation}
	\label{e.yes.eq.holds}
	-\Lop^\ep u^\ep  = -\tfrac12\Delta u \quad \mbox{in}  \ \Rd
\end{equation}
and satisfies
\begin{equation}
	\label{e.uep.to.uepR}
	\| u^\ep_R - u^\ep  \|_{L^\infty(B_R)}
	\leq
	C \X^{d+1+\gamma}
	R^{2-\gamma-d}
\end{equation}
and the decay estimate
\begin{equation*}
	\| u^\ep \|_{L^\infty(\Rd \setminus B_R )}
	\leq  C \X^{d+1+\gamma}
	R^{2-\gamma-d}
	\,.
\end{equation*}
In particular,~$u^\ep \in C_0(\Rd)$. The equation~\eqref{e.yes.eq.holds} follows in the distributional sense on each fixed ball by passing to the limit in the weak formulation, using the local Caccioppoli estimate for~$u^\ep_R$. Since the operator~$\Lop^\ep$ has coefficients which are~$\P$--almost surely~$C^{1,1}$ by assumption~\ref{a.j.reg}, standard Schauder estimates imply that~$u^\ep\in C^2(\Rd)$. The limit~\eqref{e.RHS.converge} is obvious from~\eqref{e.yes.eq.holds}. To check~\eqref{e.homogenization.giveth}, we apply Theorem~\ref{t.superdiffusivity} in a fixed ball~$B_R$, which yields, for all sufficiently small~$\ep$ (depending on~$R$),
\begin{equation*}
	\| u^\ep_R - u \|_{L^\infty(B_R)}
	\leq
	C\left| \log \ep \right| ^{-\alpha}  \| \Delta u\|_{L^\infty(B_R)}
	\leq
	C \left| \log \ep \right| ^{-\alpha} \,.
\end{equation*}
Since~$u$ vanishes outside of~$B_1$, we also have that
\begin{equation*}
	\| u^\ep - u \|_{L^\infty(\Rd \setminus B_R)}
	=
	\| u^\ep  \|_{L^\infty(\Rd \setminus B_R)}
	\leq
	C \X^{d+1+\gamma}
	R^{2-\gamma-d}\,.
\end{equation*}
Combining the previous displays with~\eqref{e.uep.to.uepR} and using the triangle inequality, we find that
\begin{equation*}
	\limsup_{\ep \to 0}
	\| u^\ep - u \|_{L^\infty (\Rd)}
	\leq
	C \X^{d+1+\gamma}
	R^{2-\gamma-d}
	\,.
\end{equation*}
Sending~$R \to \infty$ yields~\eqref{e.homogenization.giveth}.
\end{proof}

On the full-probability event on which Proposition~\ref{p.generators} holds, we now apply~\cite[Theorem 19.25]{Kallenberg}. The proposition gives, for every~$u\in C_c^\infty(\Rd)$, functions~$u^\ep\in C^2(\Rd)\cap C_0(\Rd)$ satisfying~$u^\ep\to u$ and~$\Lop^\ep u^\ep\to\frac12\Delta u$ uniformly. Since~$\Lop^\ep u^\ep=\frac12\Delta u\in C_0(\Rd)$ by construction, these~$u^\ep$ belong to the domains of the generators of the Feller processes~$X^\ep$. The theorem therefore gives~\eqref{e.convinlaw.again}. The deterministic reparametrization discussed after~\eqref{e.Xep.t.def.recall}, together with Brownian scaling, then gives~\eqref{e.convinlaw}.

\smallskip

\appendix

\smallskip

\clearpage
\section{The log-correlated Gaussian field}
\label{s.LGF}

\subsection{Construction of the log-correlated Gaussian field}
In this appendix, we give an explicit construction of the log-correlated Gaussian field (LGF) in general dimension~$d\geq 2$ based on an annular decomposition and a standard Gaussian white noise. This physical-space definition is equivalent to the usual definitions in Fourier space or by the Bochner--Minlos theorem. Its advantage here is that it makes the finite-range decomposition transparent.

\smallskip

Throughout, we let~$W$ denote a standard Gaussian white noise field on~$\R^{1+d}$. That is,~$W$ is a random distribution on~$\R^{1+d}$ such that~$W(\psi)$ is a Gaussian random variable for every test function~$\psi\in C^\infty_c(\R^{1+d})$,  and which satisfies the covariance formula
\begin{equation}
	\label{e.cov.W}
	\cov[ W(\psi_1) , W(\psi_2) ] = \int_{\R^{1+d}} \psi_1(x)\psi_2(x)\,dx \,, \quad \forall \psi_1,\psi_2 \in C^\infty_c(\R^{1+d})\,.
\end{equation}
The distribution~$W$ is not a function. In fact, it (almost surely) belongs to~$H^{-\frac12 (1+d)-\ep}_{\mathrm{loc}}(\R^{1+d})$, for every~$\ep>0$, but not~$H^{-\frac12 (1+d)}_{\mathrm{loc}}(\R^{1+d})$. See~\cite[Chapter 5]{AKMBook} for a proof of this fact, as well as an explicit construction of the white noise field~$W$. Throughout, we abuse notation by informally writing~$\int_{\R^{1+d}} \psi(x) W(x)\,dx$ in place of~$W(\psi)$.

\smallskip

Fix a smooth, radial function~$\zeta:\R^d \to \R$ and a constant~$C_\zeta<\infty$ satisfying
\begin{equation}
	\label{e.nice.zeta}
	\indc_{B_{\nicefrac12}} \leq \zeta \leq \indc_{B_1}
	\quad \mbox{and} \quad
	\| \nabla^k \zeta \|_{L^\infty(B_1)}
	\leq C_\zeta \,, \quad \forall k\in\{1,2,3,4\}\,.
\end{equation}
For each~$r>0$, define
\begin{equation}
	\zeta_r (x) \coloneqq  \zeta(\nicefrac xr )\,.
\end{equation}
\begin{lemma}
For every~$\alpha>0$ and~$x\in\Rd\setminus\{0\}$, we have that
\begin{equation}
	\label{e.power.decomp}
	|x|^{-\alpha}
	=
	M_{\alpha} \int_0^\infty r^{-(1+d+\alpha)} \bigl( \zeta_r \ast \zeta_r \bigr) (x) \, dr\,,
	\quad \mbox{where} \quad
	M_{\alpha} \coloneqq  \biggl( \int_{0}^\infty r^{-(1+\alpha)} (\zeta \ast \zeta)(\nicefrac 1r) \,dr  \biggr)^{-1 } \,,
\end{equation}
and
\begin{equation}
	\label{e.malpha.constant}
	\lim_{\alpha \to 0} \alpha M_{\alpha}^{-1} = (\zeta \ast \zeta)(0) = \int_{\Rd} |\zeta(x)|^2\,dx =: M_0 \,.
\end{equation}
\end{lemma}
\begin{proof}
Using that~$\zeta_r \ast \zeta_r = r^d (\zeta \ast \zeta)(\nicefrac \cdot r)$ and that~$\zeta$ is radial, we find that, for every~$\beta>0$,
\begin{align*}
	\int_0^\infty r^{-\beta} \bigl( \zeta_r \ast \zeta_r \bigr) (x) \, dr
	&
	=
	\int_0^\infty r^{d-\beta} (\zeta \ast \zeta)(\nicefrac x  r)\, dr
	=
	|x|^{1+d-\beta} \int_0^\infty r^{d-\beta} (\zeta \ast \zeta)(\nicefrac 1  r)\, dr\,.
\end{align*}
Taking~$\beta = 1+d+\alpha$ yields~\eqref{e.power.decomp}. The limit~\eqref{e.malpha.constant} follows from the change of variables~$s=r^{-1}$:
\begin{equation*}
	M_\alpha^{-1}
	=
	\int_0^2 s^{\alpha-1}(\zeta*\zeta)(s)\,ds
	\quad\mbox{and hence}\quad
	\alpha M_\alpha^{-1}\to (\zeta*\zeta)(0)
\end{equation*}
by the continuity of~$\zeta*\zeta$ at the origin.
\end{proof}

We define, for each~$\psi \in C^\infty_c(\Rd)$ with~$\int_{\Rd} \psi = 0$,
\begin{equation}
	\label{e.Hpsi.def.now.with.finite.range}
	H(\psi)  \coloneqq
	\biggl(
	\frac{d |B_1|}{M_0(2 \pi)^{d}} \biggr)^{\!\nicefrac12}
	\int_0^\infty r^{-\frac12(d+1)} \int_{\Rd}  (\zeta_r \ast \psi) (y)  W(r,y) \,dy \,dr \,.
\end{equation}
This stochastic integral is well-defined in~$L^2$: since~$\int_{\Rd}\psi=0$ and~$\psi$ is compactly supported,
\begin{equation*}
	\|\zeta_r*\psi\|_{L^2(\Rd)}
	\leq C_\psi r^d \quad (0<r\leq1)\,,
	\qquad
	\|\zeta_r*\psi\|_{L^2(\Rd)}
	\leq C_\psi r^{\frac d2-1} \quad (r\geq1)\,.
\end{equation*}
These two estimates also give the dominated convergence used below. The following covariance formula identifies~$H$ with other definitions of the LGF, such as those in~\cite{LSSW,DRSV}.

\begin{lemma}
\label{l.cov.LGF}
Let~$\psi_1,\psi_2\in C^\infty_c(\Rd)$ with~$\int_{\Rd} \psi _i= 0$ for~$i\in\{1,2\}$. Then
	\begin{align}
		\label{e.cov.LGF}
		\cov[ H(\psi_1), H(\psi_2) ]
		&
		=
		\frac{d |B_1|}{(2 \pi)^{d}}
		\int_{\Rd} \int_{\Rd}
		- \log |x-y| \psi_1(x)\psi_2(y)\,dx\,dy
		\,.
	\end{align}

\end{lemma}

The expression~\eqref{e.cov.LGF} matches the covariance formula (including the multiplicative constant) for the LGF defined in~\cite{LSSW}: see Theorem 3.3(iv) of that paper. In particular, in two dimensions we have the formula
\begin{equation}
	\label{e.cov.LGF.2d}
	\cov[ H(\psi_1), H(\psi_2) ]
	=
	\int_{\R^2} \int_{\R^2}
	- \frac1{2\pi} \log |x-y| \psi_1(x)\psi_2(y)\,dx\,dy\,, \quad \mbox{if} \ d=2\,,
\end{equation}
which matches the usual definition of the Gaussian free field, as the function~$G(x,y) = - \frac1{2\pi} \log |x-y|$ is the Green function for the Laplacian in two dimensions.

\smallskip

\begin{proof}[Proof of Lemma~\ref{l.cov.LGF}]
Because~$H(\psi)$ is Gaussian, its law is determined by its covariance. It therefore suffices by polarization to compute the variance for a single test function~$\psi$, that is, we may assume that~$\psi_1=\psi_2=\psi$. Using the definition of white noise~\eqref{e.cov.W}, we have
\begin{align}
	\label{e.var.Hpsi}
	\lefteqn{
	\frac{M_0 (2 \pi)^d }{d |B_1|}   \var \bigl[ H(\psi) \bigr]
	} \qquad &
	\notag \\ &
	=
	\E \biggl[
	\int_0^\infty\!
	\int_0^\infty\!
	\int_{\Rd}\!
	\int_{\Rd}\!
	(st)^{-\frac12(d+1) }
	(\zeta_s \ast \psi) (x)
	(\zeta_t \ast \psi) (y)
	W(s,x)W(t,y)
	\, dx \, dy \, dt \, ds
	\biggr]
	\notag \\ &
	=
	\int_0^\infty\!
	r^{-(d+1)}
	\int_{\Rd}\!
	|(\zeta_r \ast \psi) (x)|^2
	\, dx \, dr \,.
\end{align}
The above display implies by dominated convergence,
\begin{equation}
	\label{e.dominating.an.expr}
	\frac{M_0 (2 \pi)^d }{d |B_1|}   \var \bigl[ H(\psi) \bigr]
	= \lim_{\alpha \to 0+}
	\int_0^\infty\!
	r^{-(d+1+\alpha)}
	\int_{\Rd}\!
	|(\zeta_r \ast \psi) (x)|^2
	\, dx \, dr \,,
\end{equation}
and for each~$\alpha\in(0,1)$ we have, by~\eqref{e.power.decomp} that
\begin{align*}
	\lefteqn{
	\int_0^\infty\!
	r^{-(d+1+\alpha)}
	\int_{\Rd}\!
	|(\zeta_r \ast \psi) (x)|^2
	\, dx \, dr
	} \qquad &
	\notag \\  &
	=
	\int_0^\infty\!
	r^{-(d+1+\alpha)}
	\int_{\Rd}\!
	\int_{\Rd}\!
	\int_{\Rd}\!
	\zeta_r(x-z) \zeta_r(x-y)  \psi(z) \psi(y) \,dx \,dy \,dz \,dr \\
	&=
	\int_0^\infty\!
	r^{-(d+1+\alpha)}
	\int_{\Rd}\!
	\int_{\Rd}\!
	(\zeta_r \ast \zeta_r)(z-y)  \psi(z) \psi(y) \, dy \, dz \, dr \\
	&=
	M_{\alpha}^{-1}\int_{\Rd}\!
	\int_{\Rd}\!
	|z-y|^{-\alpha}  \psi(z) \psi(y) \,dy \,dz  \,.
\end{align*}
Using the limit
\begin{equation*}
	-\log |z| = \lim_{\alpha \to 0+}
	\frac{1}{\alpha}\bigl( |z|^{-\alpha} -1 \bigr) \,,\quad \forall z\in \Rd\setminus \{ 0\}\,,
\end{equation*}
and~\eqref{e.malpha.constant} together with the assumption that~$\int_{\Rd} \psi=0$ we get
\begin{align*}
	\frac{M_0 (2 \pi)^d }{d |B_1|}   \var \bigl[ H(\psi) \bigr]  &= 	\lim_{\alpha \to 0+} \alpha M_{\alpha}^{-1} \int_{\Rd}\!
	\int_{\Rd}\!
	\alpha^{-1} |z-y|^{-\alpha}  \psi(z) \psi(y) dy dz   \\
	&
	=
	\lim_{\alpha\to 0+} \alpha M_{\alpha}^{-1}
	\int_{\Rd} \!
	\int_{\Rd} \!
	\alpha^{-1}
	\bigl( |z-y|^{-\alpha} - 1 \bigr)
	\psi (z) \psi (y) \,dy \,dz
	\notag \\ &
	=
	\int_{\Rd} \!
	\int_{\Rd} \!
	-M_0 \log |z-y|
	\psi (z) \psi (y) \,dy \,dz
	\,.
\end{align*}
The last limit follows by dominated convergence on~$\supp \psi\times\supp\psi$, using the integrable bound~$C(1+|z-y|^{-1})|\log|z-y||$. Combining the above display with~\eqref{e.dominating.an.expr} we obtain
\begin{equation*}
	\var \bigl[ H(\psi) \bigr]
	=
	\frac{d |B_1|}{ (2 \pi)^d }
	\int_{\Rd} \!
	\int_{\Rd} \!
	-\log |x-y|
	\psi (x) \psi (y) \,dx \,dy
	\,,
\end{equation*}
which matches the right side of~\eqref{e.cov.LGF}. The proof is complete.
\end{proof}

\subsection{A decomposition of the LGF into finite range fields}
In this subsection we construct a finite-range decomposition of the LGF, mollified on the unit scale, and verify Assumptions~\ref{a.j.frd}--\ref{a.j.nondeg} after the deterministic normalization recorded at the end of the subsection. We define, for each~$n\in\Z$,
\begin{equation*}
	H_n(x)  \coloneqq
	\biggl(
	\frac{d |B_1|}{M_0(2 \pi)^{d}} \biggr)^{\!\nicefrac12}
	\int_{3^{n-1}}^{3^{n}} r^{-\frac12(d+1)} \int_{\Rd}  \zeta_r (x-y)  W(r,y) \,dy \,dr \,.
\end{equation*}
It is clear that~$\{ H_n \}_{n\in\Z}$ is a sequence of independent,~$\Rd$--stationary Gaussian random fields with zero mean. The range of dependence of~$H_n$ is at most~$C3^n$ by definition. Since the function~$\zeta$ is radial, the joint law of the sequence~$\{ H_n \}_{n\in\Z}$ is isotropic and invariant under negation. In view of the definition~\eqref{e.Hpsi.def.now.with.finite.range}, we have that, for every~$\psi \in C^\infty_c(\Rd)$ with~$\int_{\Rd} \psi(x)\,dx=0$,
\begin{equation}
	\label{e.H.is.sum.of.Hn}
	H(\psi) = \sum_{n\in\Z}
	\int_{\Rd} H_n(x) \psi(x)\,dx\,.
\end{equation}
We have that
\begin{align}
	\label{e.var.Hn}
	\var \bigl[ H_n(0) \bigr]
	&
	=
	\frac{d |B_1|}{M_0(2 \pi)^{d}}
	\int_{3^{n-1}}^{3^{n}} r^{-(d+1)} \int_{\Rd} \bigl| \zeta_r (x)\bigr|^2  \,dx \,dr \,.
	\notag \\ &
	=
	\frac{d |B_1|}{M_0(2 \pi)^{d}}
	\int_{3^{n-1}}^{3^{n}} r^{-1} \int_{\Rd} \bigl| \zeta (x)\bigr|^2  \,dx \,dr
	\notag \\ &
	=
	\frac{d |B_1|}{(2 \pi)^{d}}
	\int_{3^{n-1}}^{3^{n}} r^{-1} \,dr
	=
	\frac{d |B_1|}{(2 \pi)^{d}} (\log 3)
	=
	\frac{2^{1-d}\pi^{-\nicefrac d2}}{\gammafun(\nicefrac d2) } (\log 3) \,.
\end{align}
In particular, there exists~$C(d)<\infty$ such that
\begin{equation*}
	H_n(x)
	=
	\O_{\Gamma_2}(C)\,.
\end{equation*}
By a similar computation, using~\eqref{e.nice.zeta}, we find that
\begin{equation*}
	3^n | \nabla H_n(0) |+3^{2n} | \nabla^2 H_n (0) |+3^{3n} | \nabla^3 H_n (0) |\leq\O_{\Gamma_2}(C)\,.
\end{equation*}
After rescaling the fields to unit scale and applying standard Gaussian supremum estimates on fixed balls, we deduce that, for every~$R<\infty$,
\begin{equation}
	\label{e.Hn.control.bounds}
	\| H_n \|_{L^\infty(B_{R3^n})}
	+
	3^n \| \nabla H_n \|_{L^\infty(B_{R3^n})}
	+
	3^{2n} \| \nabla^2 H_n \|_{L^\infty(B_{R3^n})}
	\leq
	\O_{\Gamma_2}(C(R,d))\,.
\end{equation}
Fix a nonnegative radial mollifier~$\eta\in C^\infty_c(B_1)$ satisfying~$\int_{\Rd}\eta=1$. Define the mollified scalar shells
\begin{equation}
	\label{e.Htilde.def}
	\widetilde H_0(x)
	\coloneqq
	\biggl(
	\frac{d |B_1|}{M_0(2 \pi)^{d}} \biggr)^{\!\nicefrac12}
	\int_{0}^{1} r^{-\frac12(d+1)}
	\int_{\Rd}  (\zeta_r*\eta) (x-y)  W(r,y) \,dy \,dr
\end{equation}
and, for~$n\geq1$,
\begin{equation*}
	\widetilde H_n\coloneqq H_n*\eta\,.
\end{equation*}
Then~$H*\eta$ is represented, modulo additive constants, by~$\sum_{n=0}^\infty \widetilde H_n$. The fields~$\{ \widetilde H_n\}_{n\in\N}$ are independent, stationary, dihedrally symmetric and have range of dependence~$C3^n$; moreover~\eqref{e.Hn.control.bounds} gives the corresponding~$C^2$ bound for~$\widetilde H_n$ when~$n\geq1$, and the case~$n=0$ follows from the same fixed-scale Gaussian-process estimate. Choose~$n_\eta\in\N$ so large that these ranges are bounded by~$\sqrt d\,3^{n+n_\eta}$. Set
\begin{equation*}
	G_n\coloneqq 0 \quad (0\leq n<n_\eta),\qquad G_{n_\eta}\coloneqq \widetilde H_0,\qquad G_n\coloneqq \widetilde H_{n-n_\eta} \quad (n>n_\eta)\,.
\end{equation*}
This fixed relabeling gives the exact range convention in~\ref{a.j.frd}, and preserves independence, symmetry and the local~$C^2$ bounds; it changes the non-degeneracy estimate below only by an additive constant. We next record an integration by parts identity in probability. By this we mean that, for every locally smooth, jointly~$\Rd$--stationary pair of random fields~$F$ and~$G$ such that
\begin{equation*}
	F(0),\ \partial_{x_i}F(0),\ G(0),\ \partial_{x_i}G(0)\in L^2(\Omega)\,,
\end{equation*}
we have that
\begin{equation}
	\label{e.ibp.probab}
	\E \bigl[\partial_{x_i} F(0) G(0) \bigr]
	=
	-\E \bigl[ F(0) \partial_{x_i} G(0) \bigr]
	\,.
\end{equation}
To see this, we reduce this identity to an integration by parts in space. Fix a cutoff function~$\zeta_R$ such that~$\indc_{B_R} \leq \zeta_R \leq \indc_{B_{R+1}}$ and~$\|\nabla \zeta_R\|_{L^\infty(\Rd)} \leq C$ and compute
\begin{align*}
	\lefteqn{
	\E \bigl[\partial_{x_i} F(0) G(0) \bigr]
	} \qquad &
	\notag \\ &
	=
	\lim_{R\to \infty}
	\E \biggl[ |B_R|^{-1}
	\int_{\Rd}
	\zeta_R (x)
	\partial_{x_i} F(x) G(x)
	\,dx
	\biggr]
	\notag \\ &
	=
	\lim_{R\to \infty}
	\E \biggl[ - |B_R|^{-1}
	\int_{\Rd}
	F(x)\partial_{x_i}  G(x)
	\zeta_R (x)
	\,dx
	\biggr]
	+
	\underbrace{\lim_{R\to \infty}
	\E \biggl[ - |B_R|^{-1}
	\int_{\Rd}
	\partial_{x_i}\zeta_RF(x)G(x)
	\,dx
	\biggr] \,}_{=0}
	\notag \\ &
	=
	-\E \bigl[ F(0)\partial_{x_i}  G(0) \bigr]
	\,.
\end{align*}
The boundary term vanishes since its absolute value is bounded by
\begin{equation*}
	C\frac{|B_{R+1}\setminus B_R|}{|B_R|}\E\bigl[|F(0)G(0)|\bigr]\,.
\end{equation*}
More generally, if~$F$ is a centered stationary~$L^2$ field with spectral measure~$\rho_F$ and~$\int_{\Rd}|\xi|^4\,\rho_F(d\xi)<\infty$, then
\begin{align}
	\label{e.D2.to.lap.probab}
	\sum_{i,j=1}^d
	\E \Bigl[ \bigl| \partial_{x_i} \partial_{x_j} F(0) \bigr|^2 \Bigr]
	&
	=
	\int_{\Rd}\sum_{i,j=1}^d \xi_i^2\xi_j^2\,\rho_F(d\xi)
	=
	\int_{\Rd}|\xi|^4\,\rho_F(d\xi)
	=
	\E \Bigl[ \bigl| \Delta F(0) \bigr|^2 \Bigr]
	\,.
\end{align}
We use~$\nabla\Delta^{-1}\partial_{x_i}$ below as the stationary~$L^2$ Fourier multiplier with symbol~$\xi_i\xi/|\xi|^2$. Thus, if~$F$ is one of the centered shell fields used below, or more generally a centered stationary~$L^2$ field with dihedral symmetry and no atom at zero, with spectral measure~$\rho_F$, then
\begin{equation}
	\label{e.Riesz.shell.identity}
	\E \Bigl[
	\bigl| \nabla \Delta^{-1}
	\bigl(\partial_{x_i} F
	\bigr) (0)
	\bigr|^2
	\Bigr]
	=
	\int_{\Rd} \frac{\xi_i^2}{|\xi|^2}\,\rho_F(d\xi)
	=
	\frac1d \E \bigl[ | F(0) |^2 \bigr]\,.
\end{equation}
Using independence, we obtain that
\begin{align*}
	\E \Biggl[
	\biggl| \nabla \Delta^{-1}
	\biggl(\partial_{x_i}
	\sum_{n=l+1}^m
	H_n
	\biggr) (0)
	\biggr|^2
	\Biggr]
	=
	\frac1d \E \Biggl[ \biggl| \sum_{n=l+1}^m H_n(0) \biggr|^2 \Biggr]
	&
	=
	\frac1d \sum_{n=l+1}^m \var\bigl[ H_n(0) \bigr]
	\notag \\ &
	=
	\frac{2^{1-d}\pi^{-\nicefrac d2}(\log 3)(m-l)}{d \gammafun(\nicefrac d2) }
	\,.
\end{align*}
For~$n\geq1$,~\eqref{e.Hn.control.bounds} implies
\begin{equation*}
	\bigl|(H_n*\eta)(0)-H_n(0)\bigr|
	\leq
	C\|\nabla H_n\|_{L^\infty(B_1)}
	\leq
	\O_{\Gamma_2}(C3^{-n})\,.
\end{equation*}
Using the~$L^2$ contraction of~$\nabla\Delta^{-1}\partial_{x_i}$, this gives the summable shell-by-shell error
\begin{equation*}
	\biggl|
	\E\biggl[
	\bigl|\nabla\Delta^{-1}\bigl(\partial_{x_i}(H_n*\eta)\bigr)(0)\bigr|^2
	\biggr]
	-
	\E\biggl[
	\bigl|\nabla\Delta^{-1}\bigl(\partial_{x_i}H_n\bigr)(0)\bigr|^2
	\biggr]
	\biggr|
	\leq C3^{-n}\,.
\end{equation*}
Therefore, for the mollified fields~$H_n\ast\eta$,
\begin{equation*}
	\Biggl| \E \Biggl[ \biggl| \nabla \Delta^{-1} \biggl(\partial_{x_i} \sum_{n=l+1}^m  ( H_n \ast \eta ) \biggr) (0) \biggr|^2 \Biggr]
	-
	\frac{2^{1-d}\pi^{-\nicefrac d2}(\log 3)(m-l)}{d \gammafun(\nicefrac d2) } \Biggr|
	\leq C \,.
\end{equation*}
Suppose now that~$\k$ is a~$d$-by-$d$ anti-symmetric matrix such that the entries~$\k_{ij}$ for~$i<j$ are independent copies of~$H\ast \eta$. For each~$n\in\N$, let~$\mathbf{j}_n$ be the anti-symmetric matrix whose upper-triangular entries are independent copies of~$G_n$. Then~$\k=\sum_{n=0}^\infty \mathbf{j}_n$ entrywise, modulo additive constants. For a scalar field~$F$ as in~\eqref{e.Riesz.shell.identity}, the same Fourier-multiplier computation and sign symmetry give
\begin{equation*}
	\E\bigl[
	\nabla\Delta^{-1}(\partial_{x_i}F)(0)
	\cdot
	\nabla\Delta^{-1}(\partial_{x_j}F)(0)
	\bigr]
	=
	\frac{\delta_{ij}}{d}\E[|F(0)|^2]\,.
\end{equation*}
Therefore each independent upper-triangular entry contributes~$d^{-1}(e_i^2+e_j^2)\var(F(0))$ to the projected quantity, and~$\sum_{i<j}(e_i^2+e_j^2)=d-1$ for~$e\in\partial B_1$. The fixed relabeling from~$\widetilde H_n$ to~$G_n$ changes only the additive constant, and we obtain, for every~$l,m\in\N$ with~$l<m$ and~$e\in\partial B_1$,
\begin{equation*}
	\Biggl| \E \Biggl[ \biggl| \nabla \Delta^{-1} \biggl( \nabla \cdot  \sum_{n=l+1}^m  \mathbf{j}_n e \biggr) (0) \biggr|^2 \Biggr]
	-
	\frac{(d-1) 2^{1-d}(\log 3)}{d\,\pi^{\nicefrac d2}\gammafun(\nicefrac d2) }(m-l) \Biggr|
	\leq C \,.
\end{equation*}

The preceding display identifies the coefficient~$\cstar$ of the original mollified LGF with the value announced in~\eqref{e.c.star.for.d}. By~\eqref{e.Hn.control.bounds}, the analogous fixed-scale estimate for~$\widetilde H_0$, the fixed relabeling above, the finite number of matrix entries and Lemma~\ref{l.Gamma.sigma.triangle}, there exists a deterministic constant~$K_{\rm reg}=K_{\rm reg}(d,\zeta,\eta,n_\eta)\in[1,\infty)$ such that~\eqref{e.k(n).reg} holds with the threshold~$t$ inside its probability replaced by~$K_{\rm reg}t$, uniformly in~$n\in\N$.

\smallskip

The normalized shell family~$\{K_{\rm reg}^{-1}\mathbf j_n\}_{n\in\N}$ therefore satisfies Assumptions~\ref{a.j.frd}--\ref{a.j.nondeg} with coefficient field~$K_{\rm reg}^{-1}\k$, diffusivity~$K_{\rm reg}^{-1}\nu$, and parameters~$K_{\rm reg}^{-2}\cstar$ and~$K_{\rm reg}^{-2}\nondegconst$. Homogeneity gives the corresponding generator~$K_{\rm reg}^{-1}\Lop$ and homogenized diffusivity~$K_{\rm reg}^{-1}\shom_m$; moreover, the rescaled operators in~\eqref{e.A.ep} are identical for the normalized and original coefficients, since both the numerator and its square-root normalization are multiplied by~$K_{\rm reg}^{-1}$.

\smallskip

The process associated with the normalized coefficients, run for time~$K_{\rm reg}t$, has the same law as~$X_t$. Brownian scaling therefore restores the prefactor~$\sqrt{2\cstar^{\nf12}}$ exactly, while~$\sqrt{\log(K_{\rm reg}t)}-\sqrt{\log t}=O_{K_{\rm reg}}((\log t)^{-\nf12})$ is absorbed by the stated quantitative errors. Consequently, all conclusions transfer to the original mollified LGF with~$\cstar$ given by~\eqref{e.c.star.for.d}. Constants and deterministic or random minimal scales in the transferred quantitative conclusions may additionally depend on~$K_{\rm reg}$; principal asymptotic coefficients and rate exponents remain unchanged.

\clearpage
\section{Nash-Aronson estimates}
\label{app.Feller}

Following~\cite[Chapter 19]{Kallenberg}, a Markov process~$Y$ is Feller if its transition operators
\begin{equation*}
	T_t f(x) \coloneqq \E_x[f(Y_t)]
\end{equation*}
form a Feller semigroup on~$C_0(\Rd)$. Equivalently, the operators~$\{T_t\}_{t\geq0}$ satisfy:
\begin{itemize}
	\item $\{T_t\}_{t\geq0}$ is a semigroup of positive contractions;
	\item $T_t C_0(\Rd)\subseteq C_0(\Rd)$ for every~$t\geq0$;
	\item for every~$f\in C_0(\Rd)$ and~$x\in\Rd$,~$T_t f(x)\to f(x)$ as~$t\downarrow0$.
\end{itemize}
Thus, for the SDE~\eqref{e.sde.intro}, it is enough to verify these three criteria for its transition semigroup.

\smallskip

The deterministic input is Proposition~\ref{p.Nash.Aronson.with.log} below, which gives off-diagonal decay for conservative Green kernels under a logarithmic ellipticity hypothesis. We first record how this hypothesis is verified for the coefficient field from the main text. For the field~$\a$ defined by~\eqref{e.a.superdiff} and~\eqref{e.k.sum.def}, the condition~\eqref{e.a.ellipticity.pointwise} yields, for every~$e \in \Rd$,
\begin{equation}
	\label{e.a.ellipticity.pointwise.applied}
	e \cdot (\a(x) - \k(0))^{-1} e \geq \Bigl(C \nu^{-1} \bigl( \log(\mathcal{K}_\sigma^2 + |x|^2)  \bigr)^{2(1+\sigma)} \Bigr)^{-1} |e|^2
\end{equation}
with the random variable~$\mathcal{K}_\sigma$ satisfying~\eqref{e.mathcal.K.int}. We apply the proposition below to~$\tilde \a\coloneqq \a-\k(0)$. Since~$\k(0)$ is constant and anti-symmetric,~$\nabla\cdot\tilde \a\nabla=\nabla\cdot\a\nabla$; moreover~$e\cdot\tilde\a(x)e=\nu|e|^2$, and~\eqref{e.a.ellipticity.pointwise.applied} gives the inverse bound required in~\eqref{e.basic.ue}.

\begin{proposition}[Nash-Aronson type bounds]
\label{p.Nash.Aronson.with.log}
Suppose that~$\Lambda_0,\theta,\nu\in (0,\infty)$,~$\kappa \in [2,\infty)$, and let~$\Lambda:\Rd \to [\nu,\infty)$ be given by
\begin{equation}
	\label{e.Lambda.growth}
	\Lambda(x) \leq \frac{\Lambda_0}{1+\theta^2} \bigl( \log (\kappa^2+|x|^2 )  \bigr)^\theta\,.
\end{equation}
Suppose that~$\a(\cdot)$ is an~$\R^{d\times d}$-valued coefficient field on~$\Rd$ which satisfies the local uniform ellipticity condition
\begin{equation}
	\label{e.basic.ue}
	e\cdot \a(x) e
	\geq \nu |e|^2
	\qquad \mbox{and} \qquad
	e \cdot \a^{-1} (x) e
	\geq \bigl( \Lambda(x) \bigr)^{-1}  |e|^2
	\,,
	\qquad
	\forall x,e\in\R^d\,.
\end{equation}
Assume moreover that~$\partial_t-\nabla\cdot\a\nabla$ admits a nonnegative parabolic Green kernel~$P(t,x,y)$ with the following properties. For every~$u_0\in C_c^\infty(\Rd)$,
\begin{equation*}
	u(t,x)\coloneqq \int_{\Rd}P(t,x,z)u_0(z)\,dz
\end{equation*}
is the Cauchy problem solution; the kernel is conservative, in the sense that~$\int_{\Rd}P(t,x,y)\,dy=1$; and~$P$ satisfies the semigroup property. Then there exists a constant~$C(d,\theta)<\infty$ such that~$P$ satisfies the pointwise upper bound
\begin{equation}
	\label{e.Nash.Aronson.with.log}
	\bigl| P(t,x,y) \bigr| \leq
	C(\nu t)^{-\nicefrac d2}
	\exp\biggl( - \frac{|x-y|^2}{C\Lambda_0 t ( \log (\kappa^2+|y|^2 +|x-y|^2) )^\theta} \biggr)
	\,.
\end{equation}
\end{proposition}
\begin{proof}
Fix~$y \in \Rd$, and let~$u$ be a solution of the parabolic equation
\begin{equation*}
	\partial_t u - \nabla \cdot \a\nabla u = 0 \quad \mbox{in} \ (0,\infty)\times\Rd
\end{equation*}
satisfying an initial condition~$u(0,\cdot) = u_0 \in C^\infty_c(B_1(y))$ with~$u_0 \geq 0$ and~$\int_{\Rd} u_0 = 1$. By the representation formula and conservativeness of~$P$, we have~$u(t,\cdot)\geq0$ and~$\|u(t,\cdot)\|_{L^1(\Rd)}=1$ for every~$t>0$.

\smallskip

\emph{Step 1.} The diagonal estimate. There exists~$C(d)<\infty$ such that
\begin{equation}
	\label{e.diagonal.estimate}
	\int_{\Rd} |u(t,x)|^2\,dx \leq C (\nu t)^{-\nicefrac d2} \,.
\end{equation}
This estimate uses only the lower bound for~$\a(x)$ and is identical to the usual uniformly elliptic argument. Using Nash's inequality, we compute
\begin{align*}
	\partial_t \int_{\Rd} \frac12 |u(t,\cdot)|^2
	&
	= - \int_{\Rd} (\nabla u\cdot \a\nabla u)(t,\cdot)
	\\ &
	\leq
	- \nu \int_{\Rd} |\nabla u(t,\cdot)|^2
	\\ &
	\leq
	-c\nu
	\biggl( \int_{\Rd} |u(t,\cdot)|^2 \biggr)^{\!\!1+\nicefrac 2d}
	\biggl( \int_{\Rd} |u(t,\cdot)| \biggr)^{\!-\nicefrac 4d}
	=
	-c \nu \,
	\biggl( \int_{\Rd} |u(t,\cdot)|^2 \biggr)^{\!\!1+\nicefrac 2d}
	\,.
	\end{align*}
		Integrating this inequality yields~\eqref{e.diagonal.estimate}. We will also use the following shifted version for approximate identities. When~$\tau_n\coloneqq \nu^{-1}n^{-2}$,~$\supp u_0^{(n)}\subseteq B_{\nicefrac1{2n}}(y)$ and~$\|u_0^{(n)}\|_{L^2(\Rd)}^2\leq Cn^d$, the same differential inequality gives
	\begin{equation}
		\label{e.diagonal.estimate.shifted}
		\|u^{(n)}(t,\cdot)\|_{L^2(\Rd)}^2
		\leq
		C(\nu(t+\tau_n))^{-\nicefrac d2}\,.
	\end{equation}

	\smallskip

	\emph{Step 2.} The off-diagonal estimate for approximate identities. Let~$\{u_0^{(n)}\}_n$ be nonnegative smooth functions with~$\int_{\Rd}u_0^{(n)}=1$,~$\supp u_0^{(n)}\subseteq B_{\nicefrac1{2n}}(y)$ and~$\|u_0^{(n)}\|_{L^2(\Rd)}^2\leq Cn^d$, and let~$u^{(n)}$ be the corresponding solutions. We will show that, with
	\begin{equation*}
		L_y(x)\coloneqq \log(\kappa^2+|y|^2+|x-y|^2)\,,
		\qquad
		\tau_n\coloneqq \nu^{-1}n^{-2}\,,
	\end{equation*}
	\begin{equation}
		\label{e.off.diagonal.estimate}
		\int_{\Rd} | u^{(n)}(t,x)|^2
		\exp\biggl( \frac{|x-y|^2}{40 \Lambda_0 (t+\tau_n) (2L_y(x))^\theta } \biggr) \,dx
		\leq C(\nu(t+\tau_n))^{-\nicefrac d2}
		\,.
	\end{equation}
		The weighted energy computation below is justified by first working in~$B_R$, with zero Dirichlet boundary condition and initial datum~$u_0^{(n)}$, where the weight is bounded and the boundary term vanishes. The estimates are independent of~$R$, and the standard bounded-domain exhaustion of the Cauchy problem gives~\eqref{e.off.diagonal.estimate} after sending~$R\to\infty$. We write the computation directly in whole-space notation. For any smooth bounded test function~$\zeta$, the energy computation gives
	\begin{equation}
		\label{e.testfunction.insert}
		\partial_t \int_{\Rd} \frac 12 (\zeta u^{(n)} ) ^2 \,dx
		\leq
		- \frac12 \int_{\Rd}
		\zeta^2 \nabla u^{(n)}\cdot\a\nabla u^{(n)}\,dx
		+
		\int_{\Rd} (u^{(n)})^2 \bigl(  \zeta \partial_t \zeta + 8 \Lambda(x) |\nabla \zeta|^2 \bigr)\,dx
		\,.
	\end{equation}
		Take~$A\coloneqq20\Lambda_0$ and
	\begin{equation*}
		\zeta_n(t,x)  \coloneqq
		\exp\biggl( \frac{|x-y|^2}{4A(t+\tau_n)(2L_y(x))^\theta} \biggr)\,.
	\end{equation*}
		Since~$\kappa\geq2$,
	\begin{equation*}
		\log(\kappa^2+|x|^2)
		\leq
		2L_y(x)\,,
	\end{equation*}
		and hence~\eqref{e.Lambda.growth} gives~$\Lambda(x)\leq \Lambda_0(2L_y(x))^\theta/(1+\theta^2)$. The derivative calculation for~$\zeta_n$, inserted in~\eqref{e.testfunction.insert}, gives
	\begin{align}
		\label{e.testfunction.explogstuff}
		\lefteqn{
		\partial_t \int_{\Rd} \frac 12 (\zeta_n u^{(n)} ) ^2 \,dx
		+ \frac12 \int_{\Rd}
		\zeta_n^2 \nabla u^{(n)}\cdot\a\nabla u^{(n)}
		\,dx
		} \qquad &
		\notag \\ &
		\leq
		-
		\frac1{t+\tau_n}\biggl( 1 - \frac{16\Lambda_0}{A} \biggr)
		\int_{\Rd} (u^{(n)})^2 \zeta_n^2
		\biggl( \frac{|x-y|^2}{4A(t+\tau_n)(2L_y(x))^\theta} \biggr)\,dx
		\,.
	\end{align}
		Let
	\begin{equation*}
		E_n(t)\coloneqq
		\bigl\{ x\in\Rd\,:\, |x-y|^2<K(4A)(t+\tau_n)(2L_y(x))^\theta\bigr\}\,.
	\end{equation*}
		Discarding the positive energy term in~\eqref{e.testfunction.explogstuff}, using~\eqref{e.diagonal.estimate.shifted} to estimate the contribution of~$E_n(t)$, and taking
	\begin{equation*}
		K  \coloneqq  \frac52 + \frac{5d}{4}\,,
		\qquad
		2K\biggl(1-\frac{16\Lambda_0}{A}\biggr)=1+\frac d2\,,
	\end{equation*}
		we get
	\begin{equation*}
		\partial_t\biggl( (t+\tau_n)^{1+\nicefrac d2}
		\int_{\Rd}\frac12(\zeta_n u^{(n)})^2\,dx\biggr)
		\leq
		C\nu^{-\nicefrac d2}\,.
	\end{equation*}
		On~$\supp u_0^{(n)}$, the function~$\zeta_n(0,\cdot)$ is bounded by a constant depending only on~$\theta$, and therefore
	\begin{equation*}
		\tau_n^{1+\nicefrac d2}
		\int_{\Rd}(\zeta_n(0,x)u_0^{(n)}(x))^2\,dx
		\leq
		C\nu^{-1-\nicefrac d2}n^{-2}\,.
	\end{equation*}
		Integrating the previous differential inequality and using~$\tau_n=\nu^{-1}n^{-2}$ gives~\eqref{e.off.diagonal.estimate}.

\smallskip

\emph{Step 3. Passage to the Green function.} Choose the approximate identities in Step~2 and represent the corresponding solutions by the Green kernel:
\begin{equation*}
	u^{(n)}(t,x)\coloneqq \int_{\Rd} P(t,x,z)u_0^{(n)}(z)\,dz\,.
\end{equation*}
Then~$u^{(n)}$ is the solution with initial datum~$u_0^{(n)}$, and therefore~\eqref{e.off.diagonal.estimate} holds. Since~$\Lambda$ is locally bounded, the operator is uniformly elliptic on bounded space-time cylinders; applying local parabolic regularity to the adjoint Green function~$P^*(t,z,x)=P(t,x,z)$ shows that, for every fixed~$t>0$ and~$x\in\Rd$, the map~$z\mapsto P(t,x,z)$ is continuous. Since~$u_0^{(n)}\rightharpoonup\delta_y$ and~$\supp u_0^{(n)}\subseteq B_{\nicefrac1{2n}}(y)$, we have~$u^{(n)}(t,x)\to P(t,x,y)$. Passing to the limit in~\eqref{e.off.diagonal.estimate} by Fatou's lemma gives
\begin{equation}
	\label{e.off.diagonal.estimate.for.P}
	\int_{\Rd} | P(t,x,y)|^2
	\exp\biggl( \frac{|x-y|^2}{40 \Lambda_0  t (2L_y(x))^\theta } \biggr) \,dx
	\leq
	C(\nu t)^{-\nicefrac d2}
	\,.
\end{equation}

\smallskip

\emph{Step 4. Pointwise bound.} We combine~\eqref{e.off.diagonal.estimate.for.P} with the semigroup property. We also use that~$(t,x,y)\mapsto P(t,y,x)$ is the parabolic Green function for the adjoint operator, which satisfies the same bounds. We claim that
\begin{equation}
	\label{e.triangle.silly}
	\frac{|x-y|^2}{C_\theta (\log (\kappa^2+|y|^2 + |x-y|^2) )^\theta}
	\leq
	\frac{|x-z|^2}{(2\log (\kappa^2+|x|^2+|x-z|^2) )^\theta}
	+
	\frac{|z-y|^2}{(2\log (\kappa^2+|y|^2+|z-y|^2) )^\theta}
	\,.
\end{equation}
Indeed, if~$|x-y|\leq2|z-y|$, this follows from the comparison between the logarithms centered at~$y$; otherwise~$|x-y|\leq2|x-z|$ and~$\log(\kappa^2+|y|^2+|x-y|^2)\leq C\log(\kappa^2+|x|^2+|x-z|^2)$. Therefore, by H\"older's inequality and~\eqref{e.off.diagonal.estimate.for.P}, we obtain
\begin{align*}
	P(2t,x,y)
	&
	=
	\int_{\Rd} P(t,x,z) P(t,z,y)\,dz
	\\ &
	\leq
	\exp\biggl( - \frac{|x-y|^2}{C\Lambda_0 t ( \log(\kappa^2+|y|^2 +|x-y|^2) )^\theta } \biggr)
	\notag \\ & \qquad
	\times
	\biggl(
	\int_{\Rd} |P(t,x,z)|^2
	\exp\biggl( \frac{|x-z|^2}{40 \Lambda_0  t (2\log(\kappa^2+|x|^2 +|x-z|^2) )^\theta } \biggr) \,dz\biggr)^{\nf12}
	\notag \\ & \qquad
	\times
	\biggl( \int_{\Rd} |P(t,z,y)|^2
	\exp\biggl( \frac{|z-y|^2}{40 \Lambda_0  t (2\log(\kappa^2+|y|^2 +|z-y|^2) )^\theta } \biggr) \,dz\biggr)^{\nf12}
	\notag \\ &
	\leq
	C(\nu t)^{-\nicefrac d2}
	\exp\biggl( - \frac{|x-y|^2}{C \Lambda_0 (2t) ( \log(\kappa^2+|y|^2 +|x-y|^2) )^\theta } \biggr)
	\,.
\end{align*}
This completes the proof.
\end{proof}

As a consequence of~\eqref{e.Nash.Aronson.with.log}, the Green kernel has the following moment bounds.
\begin{corollary}
Assume the hypotheses of Proposition~\ref{p.Nash.Aronson.with.log}. Then for every~$n \in \N$ there exists a constant~$C_n(\Lambda_0,\theta,\nu,d)<\infty$ such that, for every~$t>0$,
		\begin{equation}
			\label{e.nth.moment}
			\int_{\Rd} |y|^n \bigl| P(t,0,y) \big| \, dy
			\leq
			C_n t^{\nicefrac{n}{2}} (\log (t \vee \kappa^2)  )^{\frac \theta2 (n+ d)}
			\,.
		\end{equation}
\end{corollary}
\begin{proof}
Set~$\bar\kappa\coloneqq\max\{\kappa,\kappa_\theta\}$, where~$\kappa_\theta\geq2$ is deterministic and chosen so that~$\bar\kappa^2>e^\theta$ and~$\bar\kappa^2(\log\bar\kappa^2)^{-\theta}\geq2$. Replacing~$\kappa$ by~$\bar\kappa$ preserves the hypothesis, and~$\log(t\vee\bar\kappa^2)\leq C(\theta)\log(t\vee\kappa^2)$. We therefore write~$\kappa$ for~$\bar\kappa$ below. For given~$t$, we consider separately the cases~$\kappa^2 (\log \kappa^2)^{-\theta} > t$ and~$\kappa^2 (\log \kappa^2)^{-\theta} \leq t$. In the first case, set~$R^2\coloneqq t(\log\kappa^2)^\theta$. We have that
\begin{equation*}
	\log(\kappa^2 + |y|^2) \leq \log \kappa^2 +  \log (1 + \kappa^{-2} |y|^2)
	\leq
	C\log\kappa^2\bigl(1+\log (1+R^{-2}|y|^2)\bigr)
	\,,
\end{equation*}
and thus~\eqref{e.Nash.Aronson.with.log} implies, for every~$n \in \N$ and~$\kappa^2 (\log \kappa^2)^{-\theta} > t$, that
\begin{align}
	\label{e.nth.moment.pre.one}
	\int_{\Rd} |y|^n \bigl| P(t,0,y) \big|\,dy
	&
	\leq
	C t^{- \nicefrac d2}
	\int_{\Rd} |y|^n
	\exp\biggl( - \frac{R^{-2} |y|^2  }{C\bigl(1+\log (1 + R^{-2}|y|^2) \bigr)^\theta } \biggr) \, dy
	\notag  \\ &
	=
	C_n t^{\nicefrac{n}{2}} (\log \kappa )^{\frac \theta2 (n+ d)}
	\,.
\end{align}
On the other hand, assume~$\kappa^2 (\log \kappa^2)^{-\theta} \leq t$ and let~$\tau\geq\kappa^2$ be the unique solution on the increasing branch of~$r\mapsto r(\log r)^{-\theta}$ to~$\tau (\log \tau)^{-\theta} = t$. Then, again by~\eqref{e.Nash.Aronson.with.log},
\begin{align}
	\label{e.nth.moment.pre.two}
	\int_{\Rd} |y|^n \bigl| P(t,0,y) \big|\,dy
	&
	\leq
	C t^{- \nicefrac d2}
	\int_{\Rd} |y|^n
	\exp\biggl( - \frac{|y|^2  }{C  t(\log \tau)^\theta \bigl(1+\log (1 + \tau^{-1} |y|^2)\bigr)^{\theta}} \biggr) \, dy
	\notag  \\ &
	=
	C t^{- \nicefrac d2}
	\int_{\Rd} |y|^n
	\exp\biggl( - \frac{\tau^{-1} |y|^2  }{C \bigl(1+\log(1 + \tau^{-1} |y|^2)\bigr)^\theta} \biggr) \, dy
	\notag  \\ &
	=
	C_n t^{-\nicefrac{d}{2}} \tau^{\frac 12 (n+d)}
	\,.
\end{align}
Since we assume that~$\kappa$ is larger than two and~$\tau \geq \kappa^2$ in the latter case, we have~$\tau \leq C t \log^\theta t$. Combining the above two displays yields~\eqref{e.nth.moment}.
\end{proof}

We now verify the Feller criteria stated at the beginning of the appendix.

\begin{corollary}
\label{c.yes.Feller}
Assume the hypotheses of Proposition~\ref{p.Nash.Aronson.with.log}. Define~$T_0f=f$ and, for~$t>0$,
\begin{equation*}
	T_t f(x)\coloneqq \int_{\Rd} f(y)P(t,x,y)\,dy\,.
\end{equation*}
Then~$\{T_t\}_{t\geq0}$ is a Feller semigroup on~$C_0(\Rd)$. In particular, the solution of the SDE~\eqref{e.sde.intro} is a Feller process.
\end{corollary}
\begin{proof}
We first check that~$T_t$ maps~$C_0(\Rd)$ into~$C_0(\Rd)$, that is,
\begin{equation}
	\label{e.yes.Feller}
	x\mapsto \int_{\Rd} f(y) P(t,x,y) \,dy \in C_0(\Rd)\,, \quad \forall t>0\,, f \in C_0(\Rd)\,.
\end{equation}
Denote~$R_f(\ep) \coloneqq  \sup \{ |x| \,:\, |f(x)|>\ep \}$ and observe that
\begin{align*}
	\biggl| \int_{\Rd} f(y) P(t,x,y) \,dy \biggr|
	&
	\leq
	\int_{B_{R_f(\ep)}} ( |f(y)|-\ep)_+ P(t,x,y) \,dy
	+
	\ep \int_{\Rd} P(t,x,y) \,dy
	\notag \\ &
	\leq
	\| f \|_{L^\infty(\Rd)}
	\int_{B_{R_f(\ep)}}  P(t,x,y) \,dy
	+
	\ep \,.
\end{align*}
Applying~\eqref{e.Nash.Aronson.with.log} yields that, for every~$R \in [1,\infty)$,
\begin{equation*}
	\lim_{|x|\to \infty}
	\int_{B_{R}} P(t,x,y) \,dy
	=0\,.
\end{equation*}
If~$f$ is compactly supported, then~$x\mapsto \int_{\Rd}f(y)P(t,x,y)\,dy$ is continuous by local parabolic regularity on positive-time cylinders; the general case~$f\in C_0(\Rd)$ follows by approximation in~$L^\infty(\Rd)$ and the contraction property. This completes the proof of~\eqref{e.yes.Feller}.

\smallskip

We next observe that~\eqref{e.Nash.Aronson.with.log} implies that
\begin{equation}
	\lim_{t\to 0}
	\sup_{x\in\Rd}
	\biggl|
	\int_{\Rd} f(y) P(t,x,y) \,dy
	-
	f(x)
	\biggr|
	=0\,, \quad \forall f \in C_0(\Rd)\,.
\end{equation}
Indeed,~$P(t,x,\cdot)$ has unit mass. For compactly supported~$f$, uniform continuity controls the contribution of a small ball near~$x$, while~\eqref{e.Nash.Aronson.with.log} controls the complementary mass uniformly for~$x$ in a fixed large ball; outside this ball, use~$f\in C_0(\Rd)$ and the same tail estimate for the mass assigned to the support of~$f$. The general case follows by approximation in~$L^\infty(\Rd)$ and the contraction property.

\smallskip

Since~$P$ is nonnegative, has unit mass and satisfies the semigroup property, the operators~$T_t$ form a Feller semigroup. For the solution of~\eqref{e.sde.intro}, the transition semigroup is~$T_t$ by the standard identification of the diffusion transition kernel with the Green function for the backward Kolmogorov equation. This completes the proof.
\end{proof}

\subsubsection*{Acknowledgments}
S.A. was supported by NSF grants DMS-1954357 and DMS-2000200 and by the Simons Programme at IHES during a sabbatical visit. A.B. was partially supported by NSF grant DMS-2202940. T.K. was supported by the Academy of Finland and the European Research Council (ERC) under the European Union's Horizon 2020 research and innovation programme (grant agreement No 818437). We thank Pierre Le Doussal for making us aware of~\cite{BCGLD}.

\smallskip

\clearpage {\small \bibliographystyle{alpha} \bibliography{highcontrast} }

\newcommand{\etalchar}[1]{$^{#1}$}
\newcommand{\noop}[1]{} \def\cprime{$'$}
\begin{thebibliography}{BCGLD88}

\bibitem[AK24]{AK.Book}
S.~Armstrong and T.~Kuusi.
\newblock Elliptic homogenization from qualitative to quantitative, 2024.
\newblock arXiv:2210.06488v2.

\bibitem[AK25]{AK.HC}
S.~Armstrong and T.~Kuusi.
\newblock Renormalization group and elliptic homogenization in high contrast.
\newblock {\em Invent. Math.}, 242(3):895--1086, 2025.

\bibitem[AKM19]{AKMBook}
S.~Armstrong, T.~Kuusi, and J.-C. Mourrat.
\newblock {\em Quantitative stochastic homogenization and large-scale
  regularity}, volume 352 of {\em Grundlehren der mathematischen Wissenschaften
  [Fundamental Principles of Mathematical Sciences]}.
\newblock Springer, Cham, 2019.

\bibitem[AN84]{AN}
J.~A. Aronovitz and D.~R. Nelson.
\newblock Anomalous diffusion in steady fluid flow through a porous medium.
\newblock {\em Phys. Rev. A}, 30(4):1948--1954, Oct 1984.

\bibitem[AV25]{AV}
S.~Armstrong and V.~Vicol.
\newblock Anomalous diffusion by fractal homogenization.
\newblock {\em Ann. PDE}, 11(1):Paper No. 2, 145, 2025.

\bibitem[BCGLD87]{BCGLD}
J.P. Bouchaud, A.~Comtet, A.~Georges, and P.~Le~Doussal.
\newblock Anomalous diffusion in random media of any dimensionality.
\newblock {\em J. Phys. France}, 48(9):1445--1450, 1987.

\bibitem[BCGLD88]{BCGLD.erratum}
J.P. Bouchaud, A.~Comtet, A.~Georges, and P.~Le~Doussal.
\newblock Erratum: Anomalous diffusion in random media of any dimensionality.
\newblock {\em J. Phys. France}, 49(2):369--369, 1988.

\bibitem[BDG20]{BDG}
M.~Biskup, J.~Ding, and S.~Goswami.
\newblock Return probability and recurrence for the random walk driven by
  two-dimensional {G}aussian free field.
\newblock {\em Comm. Math. Phys.}, 373(1):45--106, 2020.

\bibitem[BG90]{BG}
J.-P. Bouchaud and A.~Georges.
\newblock Anomalous diffusion in disordered media: statistical mechanisms,
  models and physical applications.
\newblock {\em Phys. Rep.}, 195(4-5):127--293, 1990.

\bibitem[BK91]{BK}
J.~Bricmont and A.~Kupiainen.
\newblock Random walks in asymmetric random environments.
\newblock {\em Comm. Math. Phys.}, 142(2):345--420, 1991.

\bibitem[BS24]{BellaSchaff.Lipschitz}
P.~Bella and M.~Sch\"{a}ffner.
\newblock Lipschitz bounds for integral functionals with {$(p,q)$}-growth
  conditions.
\newblock {\em Adv. Calc. Var.}, 17(2):373--390, 2024.

\bibitem[BSW23]{BSW}
J.~Burczak, L.~Sz\'ekelyhidi, and B.~Wu.
\newblock Anomalous dissipation and {E}uler flows, 2023.
\newblock arXiv:2310.02934.

\bibitem[CHST22]{CHT}
G.~Cannizzaro, L.~Haunschmid-Sibitz, and F.~Toninelli.
\newblock {$\sqrt{\log t}$}-superdiffusivity for a {B}rownian particle in the
  curl of the 2{D} {GFF}.
\newblock {\em Ann. Probab.}, 50(6):2475--2498, 2022.

\bibitem[CMOW25]{CMOW}
G.~Chatzigeorgiou, P.~Morfe, F.~Otto, and L.~Wang.
\newblock The {G}aussian free-field as a stream function: asymptotics of
  effective diffusivity in infra-red cut-off.
\newblock {\em Ann. Probab.}, 53(4):1510--1536, 2025.

\bibitem[DRSV17]{DRSV}
B.~Duplantier, R.~Rhodes, S.~Sheffield, and V.~Vargas.
\newblock Log-correlated {G}aussian fields: an overview.
\newblock In {\em Geometry, analysis and probability}, volume 310 of {\em
  Progr. Math.}, pages 191--216. Birkh\"{a}user/Springer, Cham, 2017.

\bibitem[Fan98]{Fann}
A.~Fannjiang.
\newblock Anomalous diffusion in random flows.
\newblock In {\em Mathematics of multiscale materials ({M}inneapolis, {MN},
  1995--1996)}, volume~99 of {\em IMA Vol. Math. Appl.}, pages 81--99.
  Springer, New York, 1998.

\bibitem[FFQ{\etalchar{+}}85]{FFQSS}
D.~S. Fisher, D.~Friedan, Z.~Qiu, S.~J. Shenker, and S.~H. Shenker.
\newblock Random walks in two-dimensional random environments with constrained
  drift forces.
\newblock {\em Phys. Rev. A}, 31(6):3841--3845, 1985.

\bibitem[Fis84]{Fisher}
D.~S. Fisher.
\newblock Random walks in random environments.
\newblock {\em Phys. Rev. A}, 30(2):960--964, Aug 1984.

\bibitem[FNS77]{FNS}
D.~Forster, D.~R. Nelson, and M.~J. Stephen.
\newblock Large-distance and long-time properties of a randomly stirred fluid.
\newblock {\em Phys. Rev. A (3)}, 16(2):732--749, 1977.

\bibitem[GNO20]{GNO.reg}
A.~Gloria, S.~Neukamm, and F.~Otto.
\newblock A regularity theory for random elliptic operators.
\newblock {\em Milan J. Math.}, 88(1):99--170, 2020.

\bibitem[Kal02]{Kallenberg}
Olav Kallenberg.
\newblock {\em Foundations of modern probability}.
\newblock Probability and its Applications (New York). Springer-Verlag, New
  York, second edition, 2002.

\bibitem[KLY85]{KLY}
V.~E. Kravtsov, I.~V. Lerner, and V.~I. Yudson.
\newblock Random walks in media with constrained disorder.
\newblock {\em J. Phys. A: Math. Gen.}, 18(12):L703--L707, 1985.

\bibitem[KO02]{KO}
T.~Komorowski and S.~Olla.
\newblock On the superdiffusive behavior of passive tracer with a {G}aussian
  drift.
\newblock {\em J. Statist. Phys.}, 108(3-4):647--668, 2002.

\bibitem[LSSW16]{LSSW}
A.~Lodhia, S.~Sheffield, X.~Sun, and S.~S. Watson.
\newblock {Fractional Gaussian fields: A survey}.
\newblock {\em Probab. Surv.}, 13:1 -- 56, 2016.

\bibitem[MT16]{MT}
J.~Marklof and B.~T\'{o}th.
\newblock Superdiffusion in the periodic {L}orentz gas.
\newblock {\em Comm. Math. Phys.}, 347(3):933--981, 2016.

\bibitem[Ste70]{Stein}
E.~M. Stein.
\newblock {\em Singular integrals and differentiability properties of
  functions}, volume No. 30 of {\em Princeton Mathematical Series}.
\newblock Princeton University Press, Princeton, NJ, 1970.

\bibitem[SV07]{SV}
D.~Sz\'{a}sz and T.~Varj\'{u}.
\newblock Limit laws and recurrence for the planar {L}orentz process with
  infinite horizon.
\newblock {\em J. Stat. Phys.}, 129(1):59--80, 2007.

\bibitem[SZ06]{SZ}
A.-S. Sznitman and O.~Zeitouni.
\newblock An invariance principle for isotropic diffusions in random
  environment.
\newblock {\em Invent. Math.}, 164(3):455--567, 2006.

\bibitem[T{\'o}t18]{Toth}
B.~T{\'o}th.
\newblock Diffusive and super-diffusive limits for random walks and diffusions
  with long memory.
\newblock In {\em Proceedings of the {I}nternational {C}ongress of
  {M}athematicians---{R}io de {J}aneiro 2018. {V}ol. {IV}. {I}nvited lectures},
  pages 3039--3058. World Sci. Publ., Hackensack, NJ, 2018.

\bibitem[TV12]{TV}
B.~T\'{o}th and B.~Valk\'{o}.
\newblock Superdiffusive bounds on self-repellent {B}rownian polymers and
  diffusion in the curl of the {G}aussian free field in {$d=2$}.
\newblock {\em J. Stat. Phys.}, 147(1):113--131, 2012.

\end{thebibliography}

\end{document}